\documentclass[12pt]{article}
\usepackage[centertags]{amsmath}
\usepackage{amsfonts}
\usepackage{amssymb}
\usepackage{amsthm}
\usepackage{newlfont}
\usepackage{xcolor}
\usepackage{makeidx}
\usepackage[columns=1]{idxlayout}
\makeindex
\def\N{\mathbb{N}}
\def\Z{\mathbb{Z}}
\def\Q{\mathbb{Q}}
\def\R{\mathbb{R}}
\def\C{\mathbb{C}}

\def\A{\mathbb{A}}
\def\E{\mathbb{E}}
\def\F{\mathbb{F}}

\def\T{\mathbb{T}}
\def\PP{\mathbb{P}}
\theoremstyle{plain}
\newtheorem{thm}{Theorem}[section]
\newtheorem{cor}[thm]{Corollary}
\newtheorem{lem}[thm]{Lemma}
\newtheorem{prop}[thm]{Proposition}
\newtheorem{propdef}[thm]{Proposition-Definition}
\newtheorem{defn}[thm]{Definition}
\newtheorem{nota}[thm]{Notation}
\newtheorem{rem}[thm]{Remark}
\newtheorem{exam}[thm]{Example}
\newtheorem{quest}[thm]{Question}

\numberwithin{equation}{section}

\def\clin{\mathrm {in}}

\def\div{\mathrm{ div}}

\def\inf{\mathop{\mathrm {inf}}}
\def\mod{\mathrm{ mod}}

\def\Vdir{\mathop{\mathrm{ Vdir}}}
\def\V{\mathop{\mathbf{ V}}}
\date{}

\begin{document}


\title{Resolution of  Singularities of Arithmetical
Threefolds.}

\author{Vincent Cossart \footnotemark[1] \hbox{   } Olivier Piltant \footnotemark[1]}
\footnotetext[1]{Laboratoire de Math\'{e}matiques LMV UMR 8100, UVSQ, CNRS, Universit\'e Paris-Saclay,
45, avenue des \'{E}tats-Unis, 78035 VERSAILLES Cedex, France\\
{\tt cossart@math.uvsq.fr, piltant@math.uvsq.fr}}

\maketitle

{\it \hskip 25mm Dedicated to Shreeram Shankar Abhyankar, in memoriam.}

\vskip 5mm


\bibliographystyle{amsplain}

\begin{abstract}
We prove Grothendieck's  { conjecture} on Resolution of Singularities  for quasi-excellent
schemes ${\cal X}$ of dimension three and of arbitrary characteristic.
This applies in particular to ${\cal X}=\mathrm{Spec}A$,
$A$ a reduced complete Noetherian local ring  of dimension three and to algebraic
or arithmetical varieties of dimension three.
Similarly, if $F$ is a number field, a complete discretely valued field or
more generally the quotient field of any excellent Dedekind domain ${\cal O}$,
any regular projective surface $X/F$ has a proper and flat model ${\cal X}$
over ${\cal O}$ which is everywhere regular.
\end{abstract}


\tableofcontents

\section{Introduction.}

The Resolution of Singularities conjecture \index{Resolution of Singularities conjecture @ Resolution of Singularities conjecture} has been, and still is a long standing  {open problem}
since it was formulated by A. Grothendieck in the 1960's \cite{EGA2}(7.9.6). Grothendieck emphasized
its importance for studying homo\-lo\-gical and homotopical properties of schemes.
Even since H. Hironaka's celebrated Theorem \cite{H1} proved fifty years ago,
some new results have bettered our understanding of the problem in equal characteristic zero
\cite{BM}\cite{Vi1}\cite{W}. These results focus on the constructivity and functoriality
of their algorithms for Resolution in contrast with Hironaka's.

\smallskip

In arbitrary characteristic, a major advance towards Grothendieck's conjecture is due to
A.J. de Jong \cite{dJ} Theorem 4.1 and Theorem 6.5. He proved a weaker form of the above conjecture for
varieties $X$ over a field or a complete discrete valuation ring.
A significant difference with Grothendieck's formulation is that de Jong's alterations allow
a finite extension of the function field. Furthermore, de Jong's result does not in general provide
a regular compactification $\overline{X}$ of some \'etale covering $U$ of the regular locus
$\mathrm{Reg}X$.

\smallskip

Resolution of Singularities in its full birational form was to this date restricted to surfaces
\cite{Ab1}\cite{Ab6}\cite{H2}\cite{L3}\cite{Cu9}\cite{Cu8}\cite{CoJS}, only to mention some contributions.
In dimension three, some partial results do exist for algebraic
varieties over an algebraically closed field $k$ of positive characteristic $p\geq 7$ \cite{Ab5}\cite{Cu5}.
These results extend to all characteristics $p>0$ and fields $k$ with $[k:k^p]<+\infty$ \cite{CoP1}\cite{CoP2}
Theorem on p. 1839. For arithmetical schemes (unequal residue characteristic), birational Resolution of Singularities
was  {so far} restricted to surfaces. The first and main purpose of this article
is to prove:
\begin{thm}\label{mainthm}
Let ${\cal X}$ be a reduced and separated Noetherian scheme which is quasi-excellent and
of dimension at most three.
There exists a proper birational morphism $\pi : \ {\cal X}' \rightarrow {\cal X}$
with the following properties:
\begin{itemize}
    \item [(i)] ${\cal X}'$ is everywhere regular;
    \item [(ii)] $\pi$ induces an isomorphism $\pi^{-1}(\mathrm{Reg}{\cal X}) \simeq \mathrm{Reg}{\cal X}$;
    \item [(iii)] $\pi^{-1}(\mathrm{Sing}{\cal X})$ is a strict normal crossings divisor on ${\cal X}'$.
\end{itemize}
If furthermore a finite affine covering ${\cal X}={\cal U}_1\cup {\cal U}_2 \cup \cdots \cup {\cal U}_n$
is specified, one may take $\pi^{-1}({\cal U}_i)\rightarrow {\cal U}_i$ projective, $1\leq i \leq n$.
\end{thm}

We emphasize that no assumption is made on the characteristic of ${\cal X}$ in this theorem.
A proper birational morphism $\pi$ with property (i) was called a
resolution of singularities by Grothendieck \cite{EGA2}(7.9.1), though more recent terminology (this
article included) tends to require property (ii) as well. When property (iii) also holds, one says that $\pi$ is a
good resolution or a log-resolution. In dimension three, the hard part is to prove (i).
The following gives a strong basis for the local study of three dimensional singularities
via Resolution of Singularities:
\begin{cor}\label{completeresolution}
Let $A$ be a reduced complete Noetherian local ring of dimension three. Then
${\cal X}:=\mathrm{Spec}A$ has a good resolution of singularities which is projective.
\end{cor}

 {Since the class of quasi-excellent schemes is stable by morphisms of finite type, Theorem \ref{mainthm}
applies in particular to algebraic varieties and to arithmetical varieties over excellent Dedekind rings.
Another application of Theorem \ref{mainthm} concerns formal geometry.  Indeed, Theorem \ref{mainthm} applies to reduced completions of affine  Noetherian schemes along
quasi-excellent subschemes (O.~Gabber \cite{ILO} p.~17, Theorem~9.2,
see also C.~Rotthaus \cite{R} Theorem~3 in the semilocal case).}

\begin{cor}\label{integralmodel}
Let ${\cal O}$ be an excellent Dedekind domain with quotient field $F$ and
$\Sigma /F$ be a regular projective surface. There exists a
proper and flat ${\cal O}$-scheme ${\cal X}$ with generic fiber ${\cal X}_F=\Sigma$
which is everywhere regular.
\end{cor}

\begin{rem}\label{remHironaka}
The morphism $\pi$ provided by Theorem \ref{mainthm} is {\it not} constructed
as a composition of {\it Hironaka-permissible} blowing ups\index{Hironaka-permissible blowing ups, see also Definition~\ref{Hironakapermis}}, i.e. with  regular centers
along which the successive strict transforms of ${\cal X}$ are normally flat ( {called Hironaka Resolution\index{Hironaka Resolution @ Hironaka Resolution}
for short). Taking global sections of an appropriate exceptional divisor with exceptional support, a
Hironaka Resolution provides an ideal sheaf ${\cal I}\subseteq {\cal O}_{\cal X}$ whose blowing up is regular,
with zero locus $V({\cal I})=\mathrm{Sing}{\cal X}$. When ${\cal X}$ is affine,
our theorem states that there exists  an ideal sheaf ${\cal I}'\subseteq {\cal O}_{\cal X}$ whose blowing up is regular.
In contrast, ${\cal I}'{\cal O}_{\mathrm{Reg}{\cal X}}$  is locally principal but not necessarily trivial.}

\smallskip

On the other hand,
a certain local version of Theorem \ref{mainthm} is proved using only local Hironaka-permissible blowing ups
in Theorem \ref{luthm} below. This fact appears to be a piece of evidence that Hironaka Resolution
could be true for threefolds of nonzero residue characteristic, {\it vid.} also \cite{Co5}\cite{Moh1} in
positive characteristic. It is however restricted to certain hypersurface threefolds of multiplicity not bigger
than the residue characteristic and the problem remains widely open even in dimension three.
\end{rem}

\smallskip

In higher dimensions $n\geq 4$, the Resolution of Singularities conjecture for algebraic varieties
over a field is considered in several recent papers
\cite{BeV1}\cite{BeV2}\cite{BrV}\cite{H7}\cite{H8}\cite{Ka}\cite{KaM}\cite{Moh2}
but remains open to this date. Its local variant for valuations is also considered in
\cite{ILO}\cite{KnKu}\cite{NSp}\cite{T1}\cite{T2}\cite{Te} but remains equally unsolved.
The case of arithmetical schemes has apparently attracted less attention.\\

The second purpose of this article is to explore the Resolution of Singularities Conjecture
as formulated by A. Grothendieck \cite{EGA2}(7.9.6). The text includes numerous examples and prospective
remarks aimed at preparing the ground for further research in higher dimension.
For this purpose, we consider finite morphisms $\eta : \ {\cal X} \rightarrow \mathrm{Spec}S$, where
$S$ is an arbitrary excellent regular local ring. A test case for Resolution if $S$ has positive
characteristic $p>0$ is when $\eta$ is purely inseparable; this was already recognized by O.Zariski \cite{Z6} p.88
and S. Abhyankar \cite{Ab5} and recently confirmed by M. Temkin's purely inseparable Local Uniformization Theorem
\cite{Te} Theorem 1.3.2, {\it vid.} Remark 1.3.5 (iii). In residue characteristic $p>0$,
we also include Galois coverings of degree $p$ to this test case. The main step in proving Theorem \ref{mainthm}
consists in proving the following result.  {Assumption (i) below is the purely
inseparable case for $\mathrm{char}S=p$. Assumption (ii) below is the cyclic Galois case. For $\mathrm{char}S=p$,
Artin-Schreier polynomials  $h=X^p -g^{p-1}X+f$, $f,g\in S$, $g\neq 0$ satisfy assumption (ii). For
$\mathrm{char}S=0$, $S$ containing the group $\mathbf{\mu}_p$ of $p^{th}$-roots of unity, cyclic
polynomials $h=X^p -f$, $f\in S$, $f\neq 0$ satisfy assumption (ii).}  {The total quotient ring  $L=\mathrm{Tot}(S[X]/(h))$ is a direct product of fields. By a valuation of $L$, we mean a valuation of one of these fields.}

\begin{thm}\label{luthm}
Let $(S,m_S,k)$ be an excellent regular local ring of dimension $n=3$,
quotient field $K:=QF(S)$ and residue characteristic $\mathrm{char}k=p>0$. Let
\begin{equation}\label{eq101}
    h:=X^p+f_1X^{p-1}+ \cdots +f_p \in S[X],  \ f_1, \ldots , f_p \in S
\end{equation}
be a reduced polynomial, ${\cal X} :=\mathrm{Spec}(S[X]/(h))$ and $L:=\mathrm{Tot}(S[X]/(h))$
be its total quotient ring. Assume that $h$ satisfies one of the following assumptions:
\begin{itemize}
    \item [(i)] $\mathrm{char}K=p$ and $f_1= \cdots =f_{p-1}=0$, or
    \item [(ii)] ${\cal X}$ is $G$-invariant, where $G:=\mathrm{Aut}_K(L)=\Z/p $.
\end{itemize}

Let $\mu$ be a valuation of $L$  which is centered in $m_S$.
There exists a composition of local Hironaka-permissible blowing ups:
\begin{equation}\label{eq102}
    ({\cal X}=:{\cal X}_0,x_0) \leftarrow ({\cal X}_1,x_1) \leftarrow \cdots \leftarrow ({\cal X}_r,x_r),
\end{equation}
where $x_i \in {\cal X}_i$ is the center of $\mu$, such that $({\cal X}_r,x_r)$ is regular.
\end{thm}

We develop an approach to the Resolution of Singularities Conjecture for hypersurface singularities defined
by an equation (\ref{eq101}) such that (i) or (ii) holds (condition {\bf (G)} in the text)
{\it in any dimension} $n:=\mathrm{dim}S\geq 1$. No other assumption on $S$ is required here
than excellence of $S$; we do not even assume that $[k:k^p]<+\infty$ as suggested by
A. Grothendieck {\it loc.cit.} An extra condition {\bf (E)} on $\eta$ (Definition \ref{conditionE}) is also assumed:

\smallskip

\noindent  {(i) purely inseparable case: the image in $\mathrm{Spec}S$ of the locus
$\mathrm{Sing}_p{\cal X}$ of multiplicity $p$, is contained in a normal crossings
divisor $E$;}

\smallskip

\noindent  {(ii) cyclic Galois case: the discriminant locus of ${\cal X}\rightarrow \mathrm{Spec}S$ is contained
in a normal crossings divisor $E$. If $\mathrm{char}S=0$} ( {so \ref{luthm}(ii) holds}),
 {$E$ has characteristic $p$.}

\smallskip

This extra condition {\bf (E)} can be achieved by preparatory blowing ups in dimension three (Corollary
\ref {EEfait}), applying known Resolution theorems for two-dimensional schemes.

\smallskip

The basic structure we work with is the triple $(S,h,E)$ thus defined.
The main combinatorial data
attached with the singularity ${\cal X}$ is a {\it characteristic polyhedron} \cite{H3}\cite{CoP3}:
\begin{equation}\label{eq1021}
    \Delta_S(h;u_1, \ldots ,u_n;Z)\subseteq \R^n_{\geq 0},
\end{equation}
where $Z:=X-\phi$, $\phi \in S$, is a linear coordinate change minimizing this polyhedron (beginning of chapter 2).

\smallskip

Resolution for hypersurface singularities in residue characteristic zero uses two primary
invariants: the multiplicity function $x \mapsto m(x)$ and the (normalized) slope function $x \mapsto \epsilon (x)$.
The latter is not well-behaved in residue characteristic $p>0$: it is in general not a constructible
function on ${\cal X}$; the pair $(m(x), \epsilon (x))$ in general increases after
performing Hironaka-permissible blowing ups. This pair is denoted $(\nu , \tilde{\epsilon})$ for
surfaces in \cite{H3} p.253.

\smallskip

In contrast, we construct a numerical function (Definition \ref{defomega})
\begin{equation}\label{eq103}
\iota : {\cal X} \rightarrow \{1, \ldots ,p\}\times \N \times \{1, \geq 2\}: \ x \mapsto (m(x),\omega (x), \kappa (x))
\end{equation}
which refines the multiplicity function at those points $x \in {\cal X}$ such that
$m(x)=p$. This function is differential in nature and has ``expected'' properties: $\iota$ is invariant
by regular base change $S \subset \tilde{S}$, $\tilde{S}$ excellent (Theorem \ref{omegageomreg})
and is constructible on ${\cal X}$ (Corollary \ref{constructible}).

\begin{rem}
The differential multiplicity $\omega (x)$ sprouts from Hironaka's $\epsilon (x)$
if one requires invariance by smooth base change, {\it vid.} Theorem \ref{omegageomreg}.
A difference takes place between (i) the purely inseparable case, and (ii) the Galois case
considered in Theorem \ref{luthm}: eventually $\iota$ is uppersemicontinuous in case (i) but only
constructible in general in case (ii), {\it vid.} Corollary \ref{constructible} and following
Example \ref{Exampleconstructible}.

 { The proof of Theorem 1.4 relies mainly on the properties of our function $\omega$: constructiblity and behavior under a family of permissible blowing ups. It is defined using equation~\eqref{eq101} and both  (i) and (ii).
We know no analogue of $\omega$ with these properties in the apparently similar case
$$h=X^{p^e}+f_{p^e}\in S[X],\  e\geq 2.$$
This equation is studied by Moh
in  [56] with a related open problem ``On the bound of $d_2$'' and by H.~Hauser and S.~Perlega in \cite{HP}.}
\end{rem}

We develop a notion of permissible blowing up for $\iota$ refining that of H. Hironaka.
Permissible centers ${\cal Y} \subset {\cal X}$ are of two different kinds (Definitions
\ref{deffirstkind} and \ref{defsecondkind}), first kind being ``$\epsilon$-constant''.
They also extend to permissible centers under regular base change (Theorem \ref{geomregpermis}).
The function $\iota$ is nonincreasing with respect to permissible blowing ups (Theorem \ref{bupthm}).
Differential multiplicities and permissible centers have a similar behavior to adapted multiplicities
and permissible blowing ups considered in Resolution of Singularities
for differential forms and vector fields \cite{Se}\cite{Ca1}\cite{Ca2}\cite{CaRSp}\cite{MQPa}\cite{Pa}
and for toroidalization of morphisms \cite{Cu6}\cite{Cu7}.

\begin{rem}
Our notion of permissible blowing up also sprouts from Hironaka's $\epsilon$-constant
blowing ups if one requires invariance by smooth base change, {\it vid.} Theorem  \ref{geomregpermis}.
Permissibility at a point $y\in {\cal X}$ implies permissibility on a nonempty
Zariski open subset ${\cal U}\subseteq {\cal Y}:=\overline{ \{y\}}$ (Theorem \ref{Zariskiopen}).
Example \ref{examsecondkind} shows the relevance of permissible blowing ups of
the second kind whenever ${\cal X}$ has dimension $n\geq 3$. Section 3.3 includes further results
intended to serve as a guideline for $n\geq 4$.
\end{rem}

Beginning from chapter 4, dimension $n=3$ is assumed and we focus on the proof of Theorem \ref{mainthm}.
Chapter 4 reduces the proof of Theorem \ref{mainthm} to that of Theorem \ref{luthm}
and is adapted from \cite{CoP1} to our arbitrary characteristic context. The main issue for proving
Theorem \ref{luthm} is to achieve $m(x_{r_1})<p$ for some $r_1\geq 0$; achieving $({\cal X}_r,x_r)$ regular,
i.e. $m(x_r)=1$, is then relatively easy and has been proved in \cite{CoP4}.

\smallskip

The last four chapters contain the technical bulk of this article. In chapter 5,
the function $\kappa$ in (\ref{eq103}) is refined with values in $\{1,2,3,4\}$.
For fixed $\iota (x)$, we attach a generic projection from $\mathrm{Spec}S$ to dimension two.
In contrast with residue characteristic zero, there is no
obvious way to attach a projected two-dimensional structure similar to $(S,h,E)$. This
difficulty (no reasonable notion of ``maximal contact'') seems to be inherent
to residue characteristic $p>0$ and has proved to be quite a match. Our method consists
in projecting only the combinatorial structure provided by the characteristic polyhedron given
in (\ref{eq1021}), say:
\begin{equation}\label{eq106}
\mathbf{p}_2: \ \lbrack \Delta_S(h;u_1, u_2 , v;Z)\subseteq \R^3_{\geq 0}\rbrack \mapsto
\lbrack \Delta_2(h;u_1,u_2;v;Z)\subseteq \R^2_{\geq 0}\rbrack .
\end{equation}
Here, $\mathbf{p}_2$ is a linear projection and $v:=u_3-\phi_2$, $\phi_2 \in S$, is a linear coordinate change
minimizing the image polygon. New combinatorial invariants are associated to the right-hand side polygon; their
control under permissible blowing ups eventually leads to a smaller value $\iota (x')<\iota (x)$.
This is the content of the Projection Theorem \ref{projthm} from which Theorem \ref{luthm} follows
easily by induction on $\iota (x)$ (Corollary \ref{projthmcor}). The strategy follows that of \cite{CoP2}
but also contains very substantial improvements:
\begin{itemize}
  \item  the sequence (\ref{eq102}) which is constructed involves Hironaka-permissible blowing ups only,
  in contrast with  \cite{CoP2}. It does {\it not} depend on the given
valuation $\mu$ and can be considered as a version of Hironaka's Local Control (Hironaka's A/B Game, in residue
characteristic zero) for equations (\ref{eq101}). Precise statements use the notion of independent sequence
(Definition \ref{indepseq}) and Theorem \ref{projthm} is stated in these terms.
The authors hope that Theorem \ref{luthm} could be extended
to a Resolution of Singularities $\pi : {\cal X}' \rightarrow {\cal X}$,
$\pi$ a composition of Hironaka-permissible (global) blowing ups (and with $G$-invariant
centers under assumption (ii)).
  \item all resolution invariants used in this text are defined in terms of initial form
polynomials $\mathrm{in}_\sigma h$  w.r.t. certain faces $\sigma$ of the characteristic polyhedron
attached to $h$. Furthermore, these initial form polynomials
provide control for the invariants under blowing up. These facts are the main reason why our
proof is characteristic free: $\mathrm{in}_\sigma h$ is a polynomial with coefficients
in the residue field $k(x)$. They are also the reason why the extra assumption
$[k(x):k(x)^p]<+\infty$ is not required in the proof.
  \item the role played by small residue characteristics is very minor (essentially the extra twist
in Lemma \ref{kappa2fin25} for $p=2$). Difficulties caused by nonperfect residue fields $k(x)$
appear mostly technical in nature, because one is led to carry along (absolute)
$p$-bases $(\lambda_l)_{l\in \Lambda}$ in the construction (section 2.4). Nontrivial
issues are related to regular base change (Proposition \ref{Deltageomreg}, Theorem \ref{omegageomreg}
and Theorem \ref{geomregpermis}), the Hilbert-Samuel stratum (Proposition \ref{conedirectrix}) and
Zariski closure of formal centers (Proposition \ref{permisarc}) in arbitrary dimension $n\geq 1$.
For $n=3$, {\it vid.} Remark \ref{ridgedimthree}, Proposition \ref{tausup2}  and section 7.5;
real difficulties come from Lemma \ref{gamma2*12}(3)(3') for inseparable extensions of degree $d=p=2$.
\end{itemize}

The proof of Theorem \ref{projthm} is spread along chapters 6 ($\kappa (x)=1$), 7 ($\kappa (x)=2$),
8 and 9 ($\kappa (x)=3,4$). Chapter 9 uses blowing ups along Hironaka-permissible curves which are
not necessarily of the first or second kind. The authors do not know if such blowing ups are required
in general in order to achieve Resolution (in contrast with permissible blowing ups of the
second kind, {\it vid.} Example \ref{examsecondkind}). They do not appear in \cite{Co5}.\\

Quoting H. Hironaka's euphemism from \cite{H3} p.254: ``in the case of dimension 3 or more, the behavior of
[the characteristic polyhedron] appears to be far more complicated and has not yet been fully investigated
[...] a little experiments lead us to an aphorism: Reduction of singularities is
sharpening of polyhedra.''

When the hypersurface singularity ${\cal X}$ has dimension 3 and satisfies the assumptions of Theorem \ref{luthm},
our results give a precise content to this aphorism:
\begin{itemize}
  \item [(1)] the numerical character $\iota (x)=(m(x), \omega (x),\kappa (x))$ is attached to the initial form polynomial
  $\mathrm{in}_{m_S}h$ w.r.t. the initial face of the characteristic polyhedron;
  \item [(2)] permissible blowing ups produce a smaller value $\iota (x')$, or a monic form for
  the new initial $\mathrm{in}_{m_{S'}}h'$, with $(m(x'),\omega (x'))=(m(x),\omega (x))$. This monic form
  corresponds to a certain vertex $\mathbf{v}'$ of the characteristic polyhedron;
  \item [(3)] projecting from $\mathbf{v}'$ produces a characteristic {\it polygon} with
  numerical character $\gamma (x')\in \N$;
  \item [(4)] further Hironaka-permissible blowing ups either produce a smaller value $\iota (x'')<\iota (x)$, or
  achieve
  $$
  \iota (x'')=\iota (x'), \ \mathrm{in}_{m_{S''}}h'' \ \mathrm{in} \  \mathrm{monic} \
  \mathrm{form} \  \mathrm{with}  \ \gamma (x'')<\gamma (x').
  $$
\end{itemize}

\noindent {\it Acknowledgement:} the authors acknowledge many stimulating discussions held during the ``Fall School
on Resolution of Threefolds in Positive Characteristic'', University of Regensburg, October 1-11/2013.
They hereby thank H. Kawanoue, S. Perlega, S. Saito, M. Spivakovsky, A. Voitovitch, A. Weber and J. W{\l}odarczyk for
numerous questions and suggestions, with very special thanks to the organizers U. Jannsen and B. Schober.

Very special thanks also to H.~Mourtada, D.~Rydh and B.~Schober who organized the workshop ``Two approaches towards local uniformization and resolution of singularities in characteristic $p > 0$'', at the Mittag-Leffler Institute, Djursholm,  May 23-27/2016.

Extending the proof of Corollary \ref{integralmodel} to that of Theorem \ref{mainthm}
was suggested to the authors by D. Cutkosky who is thanked warmly about it.

Both authors thank the referees for their accurate readings, for the lot of time and energy they have devoted to this paper: 36 pages of reports. Their comments have been very helpful for amending the text,  clarifying its content and making it more accessible to readers.

\subsection{Overview of the content and proof of Theorem \ref{mainthm}.}

This article is organized as follows: in chapter 2, we introduce our main tool which is
the Hironaka Characteristic Polyhedron \cite{H3}  {(Definition \ref{defminimal})}. This is performed
for any polynomial equation
$$
    h:=X^m+f_{1,X}X^{m-1}+ \cdots +f_{m,X} \in S[X], \ f_{1,X}, \ldots , f_{m,X} \in S
$$
where $S$ is an excellent regular local ring of dimension $n\geq 1$.

Our notation $\Delta_S (h;\{u_j\}_{j\in J};X)$ for polyhedra (Definition \ref{defDelta}) slightly differs
from Hironaka's because we focus our attention on the {\it variation}
of the characteristic polyhedron along regular subschemes
$$
W:=(\{u_j\}_{j\in J})\subseteq \mathrm{Spec}S, \ J\subseteq \{1,\ldots ,n\}.
$$
A basic algebraic object attached to $W$ is the graded algebra:
\begin{equation}\label{eq1071}
G(W):=\bigoplus_{i\geq 0}I(W)^i/I(W)^{i+1}\simeq S/(\{u_j\}_{j\in J})[\{U_j\}_{j\in J}].
\end{equation}
To a given face $\sigma = \sigma_\alpha$ defined by a weight vector $\alpha \in \R^n_{\geq 0}$,
an initial form polynomial  $\mathrm{in}_\alpha h$ is attached (Definition \ref{definh}).
Proposition \ref{Deltaalg} is imported from \cite{CoP3} and is an essential
tool for studying these variations along $W$. It states that $\Delta_S(h;u_1, \ldots ,u_n;X)\subseteq \R^n_{\geq 0}$
can be made minimal by a suitable linear coordinate change $Z:=X-\phi$, $\phi \in S$. Denote
$$
{\cal X}:=\mathrm{Spec}(S[Z]/(h)), \ \eta : {\cal X} \longrightarrow \mathrm{Spec}S.
$$

If $x\in \eta^{-1}(m_S)$ is a point of multiplicity $m(x)=m$, then
$$
\eta^{-1}(m_S)=\{x\}, \ k(x)=S/m_S.
$$
Hironaka's slope for $\Delta_S (h;u_1, \ldots ,u_n;Z)$ is denoted by $\delta (x)\geq 1$
when this polyhedron is minimal (Proposition \ref{deltainv} and Definition \ref{defdelta}).

\smallskip

Assume that a reduced normal crossings divisor
\begin{equation}\label{eq107}
E=\mathrm{div}(u_1 \cdots u_e)\subseteq \mathrm{Spec}S
\end{equation}
is specified. Well adapted coordinates $(u_1, \ldots ,u_n;Z)$ are coordinates such that (\ref{eq107}) holds and
$\Delta_S (h;u_1, \ldots ,u_n;Z)$ is minimal (Definition \ref{defwelladapted}). Relevant numerical data
are defined for well adapted coordinates only. For such coordinates, $h$ has weights
$$
d_j:=\min\{x_j : (x_1,\ldots ,x_n)\in \Delta_S (h;u_1, \ldots ,u_n;Z)\}, \ 1 \leq j \leq e.
$$
When $m=p$, assumptions (i) or (ii) of Theorem \ref{luthm} (condition {\bf (G)} in the text) and {\bf (E)}
(Definition \ref{conditionE}) imply that
\begin{equation}\label{eq108}
p\delta (x), \ H_j:=pd_j \in \N \ (\mathrm{Corollary} \ \ref{cordeltaint})
\end{equation}
and provide the structure Theorem \ref{initform} for the
initial form polynomials $\mathrm{in}_\alpha h$ with respect to its compact faces (Definition \ref{definh}).
This fact allows us to reproduce part of the equicharacteristic $p>0$
constructions used in \cite{CoP2}. Note that $E$ is always assumed to be equicharacteristic $p>0$
(Definition \ref{conditionE}).

For example when $\alpha =\mathbf{1}:=(1, \ldots ,1)$,
$\sigma_\mathbf{1}$ is the {\it initial face}\index{initial face of the polyhedron} of the polyhedron $\Delta_S (h;u_1, \ldots ,u_n;Z)$;
the corresponding homogeneous polynomial
$$
\mathrm{in}_{\mathbf{1}} h \in G(m_S)[Z], \ G(m_S):=\mathrm{gr}_{m_S}S \simeq k(x)[U_1,\ldots ,U_n]
$$
(denoted by $\mathrm{in}_{m_S} h$ in the text) has degree $p\delta (x)$, setting $\mathrm{deg}Z:=\delta (x)$.
Theorem \ref{initform} can be stated as follows: assume that $\Delta_S (h;u_1, \ldots ,u_n;Z)$ is {\it not}
an orthant with vertex in $\R^e$ ($\epsilon (x)\neq 0$ in the text); then
\begin{equation}\label{eq104}
    \mathrm{in}_{m_S} h=Z^p -G^{p-1}Z +F_{p,Z}\in G(m_S)[Z].
\end{equation}
Let $H:=\prod_{j=1}^eU_j^{H_j}\in G(m_S)$ with notations as in (\ref{eq108}).
We denote (Definition \ref{defepsilon}):
$$
\epsilon (x):=\mathrm{deg}(\mathrm{in}_{m_S} h)- \mathrm{deg}H= p\delta (x)- \sum_{j=1}^eH_j \in \N.
$$

This leads us to define the function $\iota$ in (\ref{eq103}) (Definition \ref{defomega}).
The function $\omega$ is a differential version of Hironaka's $\epsilon$-function \cite{H3}
and requires introducing a differential structure $(S,h,E)$ adapted to the normal crossings divisor
$E \subset \mathrm{Spec}S$ (section 2.5). This is done by considering the $G(m_S)$-module
$\Omega_{G(m_S)} (\log U_1 \cdots U_e)$ of absolute logarithmic differentials and its
dual space of derivatives ${\cal D}(m_S)$. The derivatives
\begin{equation}\label{eq109}
H^{-1}{\partial \hfill{} \over \partial Z}, \ \{H^{-1}D\}_{D \in {\cal D}(m_S)}
\end{equation}
act on $\mathrm{in}_{m_S} h$. If $G=0$, we simply let $\kappa (x)\geq 2$, {\it vid.} (\ref{eq103}),
and
\begin{equation}\label{eq110}
\omega (x) :=\left\{
  \begin{array}{cc}
    \epsilon (x) \hfill{}& \mathrm{if} \ {\partial F_{p,Z} \over \partial U_j}=0, \ e+1 \leq j \leq n \\
     &  \\
    \epsilon (x)-1 &  \mathrm{otherwise} \hfill{} \\
  \end{array}
\right.
.
\end{equation}
If $G\neq 0$, the definition is more delicate but only relies on elementary linear algebra. We then have
\begin{equation}\label{eq111}
(\omega (x)=\epsilon (x), \ \kappa (x)=1) \ \mathrm{or} \
(\omega (x)=\epsilon (x)-1, \ \kappa (x)\geq 2)  .
\end{equation}

In order to deal with blowing ups along Hironaka-permissible subschemes ${\cal Y}\subset {\cal X}$, the
above construction is performed in a more general setup; we introduce logarithmic Nagata derivatives
${\cal D}(W)$ on the graded algebras $\widehat{G(W)}=G(\hat{W})$ for $W\subset E$ having normal crossings with $E$
(note that $\mathrm{char}W=p$ since $\mathrm{char}E=p$). The main definitions are given in (\ref{eq244}): homogeneous
submodules
$$
V(F,E,W)\subset G(W)_{d-d_W-1}, \ J(F,E,W)\subset \widehat{G(W)}_{d-d_W}
$$
are attached to a homogeneous element $F\in G(W)_d$ and  {a} monomial ideal $H_W\subset G(W)_{d_W}$. This
construction plays a fundamental role in this article and is used {\it passim.}

Another important notion is that of the affine cone $\mathrm{Max}(x)$ and  {the} affine space $\mathrm{Dir}(x)$
(Definition \ref{deftauprime}).
These are respectively the stratum and  {the} directrix of the space of forms of degree $\omega (x)$ obtained
by applying those derivatives in (\ref{eq109}). Once again, the definition is more delicate when $G\neq 0$
but elementary in nature. For applications to dimension three, we always have
$\mathrm{Max}(x)=\mathrm{Dir}(x)$, {\it vid.} Remark \ref{ridgedimthree}.

When $\omega (x)=0$ in (\ref{eq103}), a simple combinatorial blowing up
algorithm (similar to residue characteristic zero) makes the value of the multiplicity function
smaller than $p$ at all points of the blown up space mapping to $x$ (Theorem \ref{omegazero}).
 {It remains} to deal with points $x \in {\cal X}$ such that $m(x)=p$, $\omega (x)>0$.

\smallskip

Chapter 3 develops a notion of permissible blowing up
$\pi : {\cal X}'\rightarrow {\cal X}$ which refines that of H. Hironaka.
Roughly speaking, a Hironaka permissible center ${\cal Y}\subset {\cal X}$
is permissible in our sense if ${\cal X}$ is ``differentially equimultiple" along ${\cal Y}$ (Definition
\ref{deffirstkind} and Definition \ref{defsecondkind}). The notion is somewhat subtle but has good properties,
the main result being Theorem \ref{bupthm}: $\iota$ is nonincreasing along permissible blowing ups.
Furthermore, $\iota$ decreases except possibly at
exceptional points $x' \in \pi^{-1}(x)$ belonging to some embedded projective cone
$$
PC(x,{\cal Y})\subset \pi^{-1}(x)
$$
given in Definition \ref{defcone}. The cone $PC(x,{\cal Y})$ is the projectivization
of a certain cone containing  $\mathrm{Max}(x)$ and coincides with it when $\omega (x)=\epsilon (x)$.
We also mention:
\begin{itemize}
  \item  persistence of permissibility under regular base change
(Theorem \ref{geomregpermis});
  \item   the strict transform ${\cal Z}'\subset {\cal X}'$  of a permissible center
${\cal Z}\subset {\cal X}$  under a permissible blowing up $\pi$ with center ${\cal Y} \subset {\cal Z}$
is permissible (Theorem \ref{transfstricte});
  \item  the support of a formal arc can be made permissible at its special point
  by performing permissible blowing ups (Proposition \ref{permisarc});
  \item  Hironaka permissible centers are permissible in a dense open subset of their support
(Theorem \ref{Zariskiopen}).
\end{itemize}

\begin{rem}
Example \ref{exampermisarc} points out a substantial difference between permissibility for $\iota$
and Hironaka-permissibility when $n\geq 4$. It states that the support ${\cal Z}\subseteq {\cal X}$
of a formal arc cannot in general be made permissible for $\iota$ at its special point $x$
by iterated quadratic transforms. This phenomenon also occurs for $n=3$ but only for $\omega (x)=1$;
it is then easily dealt with.
\end{rem}

The section concludes with the constructibility on ${\cal X}$ of the function $\iota$
(Corollary \ref{constructible}). Dimension $n=3$ is assumed in the next chapters.

\smallskip

Chapter 4 contains what can be deduced from known Embedded Resolution results
in excellent regular threefolds. We also adapt some of the equal characteristic $p>0$ material from \cite{CoP1}
to our arbitrary characteristic context and prove:
\begin{itemize}
    \item [(4.1)] reduction of Theorem \ref{mainthm} to its Local
    Uniformization form along valuations;
    \item [(4.2)] reduction of Local Uniformization to Theorem \ref{luthm};
    \item [(4.3)] the normal crossings condition {\bf (E)} can be achieved (Corollary \ref{EEfait}).
\end{itemize}

Chapter 5 collects together all previous results. A projection
number $\kappa (x)\in \{1,2,3,4\}$  (Definition \ref{defkappa})
is associated to a singular point $x \in {\cal X}$ such that $m(x)=p$, $\omega (x)>0$.
This function basically expresses the transverseness or tangency of
the initial form (\ref{eq104})
of the characteristic polyhedron with respect to the initial face.
For convenience of the reader, we give {\it a sample} of the main types of initial form polynomials
occurring when $E=\mathrm{div}(u_1)$; we take $\omega (x)>0$, $\lambda  \in k(x)$ and all exponents are integers
in these formul{\ae}. Furthermore, we have $\lambda \neq 0$,  $\lambda \not \in k(x)^p$ if
$$
(d_1,\omega (x)/p) \in \N^2 \  (\mathrm{resp.}\ \mathrm{if} \ d_1 +\omega (x)/p \in \N)
$$
in the second (resp. fifth) formula:
$$
\mathrm{in}_{m_S} h=
\left\{
  \begin{array}{ccc}
    Z^p -\left (\lambda U_1^{d_1+{\omega (x) \over p}}\right )^{p-1}Z  & & \kappa (x)=1 \\
     & & \\
    Z^p + \lambda U_1^{pd_1}U_3^{\omega (x)} \hfill{}& \omega (x)\equiv 0 \mathrm{mod}p & \kappa (x)=2 \\
     & & \\
    Z^p + \lambda U_1^{pd_1}U_2U_3^{\omega (x)} \hfill{}& \omega (x)\equiv 0 \mathrm{mod}p &  \kappa (x)=2 \\
     & & \\
    Z^p + \lambda U_1^{pd_1}U_3^{1+\omega (x)} \hfill{}& 1+\omega (x)\not \equiv 0 \mathrm{mod}p & \kappa (x)=3\\
     & & \\
    Z^p + \lambda U_1^{pd_1+\omega (x)} \hfill{}& & \kappa (x)=4\\
     & &  \\
    Z^p + \lambda U_1^{pd_1+\omega (x)}U_2 \hfill{}& & \kappa (x)=4 \\
  \end{array}
\right.
$$

The complete definition of $\kappa (x)$ takes into account all possible $\mathrm{in}_{m_S} h$
and $E$ which may occur. The simpler forms listed above are  ``monic forms'' in the sense
that a certain monomial computing  $\omega (x)$ occurs in $\mathrm{in}_{m_S} h$. We now explain
these definitions and the hierarchy between them: for fixed $\omega (x)$, the singularity is
considered as milder as $\kappa (x)$ decreases. To begin with, $\omega (x)$ is computed
from $\mathrm{in}_{m_S} h$ by applying certain derivatives (\ref{eq109})-(\ref{eq111}).

\smallskip

\noindent $\bullet$ when this derivative is transverse to the
base $\mathrm{Spec}S$, i.e. applying $H^{-1}{\partial \hfill{} \over \partial Z}$ in
(\ref{eq109}), we set $\kappa (x)=1$; otherwise $\kappa (x)\geq 2$.

\smallskip

\noindent $\bullet$ when $\kappa (x)\geq 2$, we set $\kappa (x)=4$ if the {\it directrix} affine space
$\mathrm{Dir}(x)$ has equations in  $U_1, \ldots ,U_e$, i.e. in those coordinates
corresponding to $E$. Otherwise, $\mathrm{Dir}(x)$ has an equation which is  transverse
to $E$, say $U_3=0$ with $e=1$ or $e=2$. The very transverse case $\kappa (x)=2$
means that a derivative transverse to $U_3$ is involved in (\ref{eq109}), i.e.
a derivative w.r.t. another variable $U_1,U_2$ or to a constant in $k(x)$:
$$
D=H^{-1}U_1{\partial \hfill{} \over \partial U_1},
\ D=H^{-1}{\partial \hfill{} \over \partial U_2} \ (e=1),
\ \mathrm{or} \ D=H^{-1}{\partial \hfill{} \over \partial \lambda}.
$$

 {We set $\kappa(x)=3$ if  none of the cases before holds.}

Theorem \ref{projthm} states that $\iota (x)$ can be made smaller
by performing local Hironaka permissible blowing ups. Theorem \ref{luthm} then follows easily
by descending induction on $\iota (x)$. \\

The proof of Theorem \ref{projthm} is very long and intricate. For $\kappa (x)=1$ (resp. $\kappa (x)=2,3,4$),
the proof is given in Corollary \ref{projthmkappa1} (resp. Theorem \ref{proofkkappa2},
Theorem \ref{proofkappa34}, {\it ibid.}). Three main phenomena are responsible for these intricacies:

\begin{itemize}
  \item [(i)] no obvious way shows up for reducing Theorem \ref{projthm} for $(S,h,E)$ to
  some statement on the {\it coefficients} of the polynomial $h$. When this is possible
  (for $\kappa (x)=1$ and in part for $\kappa (x)=3,4$), the proofs are notably simplified.
  This is done in section 6 where some weak form of maximal contact with a component of $E$ is assumed for $\iota$.
  \item [(ii)] reducing Theorem \ref{projthm} to the  ``monic forms'' corresponding to $\kappa (x)$ is achieved by a
  casuistic analysis which seems for the moment out of reach in higher dimensions. Sections 7.2, 8.3
  and part of 8.1, 8.2 are concerned with this problem.
  \item [(iii)] blowing up a monic form along a permissible center (e.g. a closed point)
  may lead to a bigger value $\iota (x')=(p,\omega (x),4)>\iota (x)$ when $\kappa (x)=2,3$.
  These situations are also dealt with by a casuistic analysis whose extension to higher
  dimensions seems out of reach. Section 7.1 and part of 8.1, 8.2 are concerned with this problem.
\end{itemize}

Chapter 6 proves Theorem \ref{projthm} for sequences of permissible blowing ups with centers lying
inside the successive strict transforms of a fixed irreducible component of $E$. This proves Theorem \ref{projthm} in the case
$\kappa (x)=1$ and prepares the ground in the cases $\kappa (x)=3,4$. The proof is similar to
that of Resolution for excellent surfaces \cite{H3}\cite{Co2}\cite{Co3}, but does not follow from it.

\smallskip

Chapter 7 proves Theorem \ref{projthm} when $\kappa (x)=2$. The above phenomenon (iii) is studied
in section 7.1. The proofs are essentially the same as in \cite{CoP2} chapter 2.{\bf II} except that
all statements and proofs are phrased only in terms of initial form polynomials $\mathrm{in}_\alpha h$
w.r.t. certain faces $\sigma_\alpha$ of $\Delta_S(h;u_1, u_2 ,u_3;Z)$. Section 7.2 defines the
``monic forms'' (Definition \ref{*kappadeux}) and deals with the above phenomenon (ii) in
Proposition \ref{redto*}.

No obvious reduction to Resolution for surfaces is available (phenomenon (i)). The proof then follows
our strategy as indicated at the end of the previous section (3) and (4).
Section 7.3 builds up the projected polygon $\Delta_2(h;u_1,u_2;v;Z)$
of (\ref{eq106}) (Theorem \ref{well2prepared}) and defines secondary numerical invariants
(Definition \ref{definvariants2}). The main invariant is denoted by $\gamma (x)\in \N$.
Two main difficulties arise here: rationality over $S$ (i.e. $v$
can be chosen in $S$ and not only in $\hat{S}$), and independence of choices of coordinates.
Section 7.4 studies the behavior of the invariants under blowing up a closed point. Finally,
section 7.5 proves that permissible blowing ups produce some point $x'$
with $\iota (x')\leq (p,\omega (x),1)$ (Theorem \ref{proofkkappa2}). The algorithm blows up
permissible curves only when $\gamma (x)=0,1$.

\smallskip

Chapters 8 and 9 prove Theorem \ref{projthm} for $\kappa (x)=3,4$. Since only Hironaka-permissible
centers are used, this chapter contains many new features in comparison with the corresponding \cite{CoP2}
chapter 3.{\bf II}. Definition \ref{**} states what is required of the ``monic forms'', called
respectively (**) ($\kappa (x)=3,4$) and (T**) ($\kappa (x)=4$). Phenomenon (iii) seems to be untractable
here and is the reason for these stronger conditions imposed on $h$. Reduction to these monic
forms is harder than in chapter 7 and is spread along sections 8.1, 8.2 and 8.3 (Propositions \ref{redto**3}
and \ref{redto**4}).

Section 9.2 reduces a monic form (T**) to (**) or to $\kappa (x)\leq 2$
(Proposition \ref{redto**casT**}). The proof is an application of Theorem  \ref{contactmaxFIN}
since a weak form of maximal contact with a component of $E$ holds for this reduction.
Section 9.3  finally proves that monic forms (**) can be reduced to $\kappa (x)\leq 2$
(Proposition \ref{END}). When $\omega (x)\geq p$, this reduction is achieved by blowing up along
Hironaka-permissible curves, not necessarily permissible of the first or second kind,
but contained in the locus
$$
\Omega_+({\cal X}):=\{y\in {\cal X} : \omega (y)>0\}.
$$
In order to ensure Hironaka-permissibility, the condition $E=\eta (\mathrm{Sing}_p{\cal X})$ is required
(section 9.2.1, condition ({\bf E'}) in the text). Section 9.2.2  builds up
the projected polygon $\Delta_2(h;u_1,u_2;v;Z)$ (Definition \ref{kappa3preparation}
and Proposition \ref{kappa3prepatot}) and defines
secondary numerical invariants (Definition \ref{kappa3invariants}). Said blowing ups
along Hironaka-permissible curves are performed mostly in Propositions \ref{redto3**casii}
and \ref{**gamma}.

\section{Adapted structure and primary invariants.}

All along this article, we will denote by $S$ a regular
local ring of arbitrary dimension $n\geq 1$, and by $(u_1, \ldots,
u_n)$ a regular system of parameters (r.s.p. for short) of $S$.
Its maximal ideal is denoted by $m_S:=(u_1, \ldots,u_n)$ and its formal completion
w.r.t. $m_S$ by $\hat{S}$. The order function $\mathrm{ord}_{m_S}$ on $S$ is defined by:
$$
\mathrm{ord}_{m_S}f:=\sup \{n \in \N : f\in m_S^n\}\in \N \cup \{+ \infty\}, \ f \in S.
$$
This order function extends to a discrete valuation on the quotient field $K:=QF (S)$ of $S$.

\smallskip

We will assume that $\mathrm{char}(S/m_S)>0$ except for the next three sections.
We also assume that $S$ is {\it excellent} beginning from Proposition \ref{Deltaalg} on.
The basic reference for excellent rings is \cite{EGA2} 7.8 and 7.9. A useful {\it compendium} is
\cite{Ma} pp. 255-260; some extensions and examples of non excellent regular local rings
can be found in \cite{ILO} pp. 7-22.
Let
\begin{equation}\label{eq201}
    h:=X^m+f_{1,X}X^{m-1}+ \cdots +f_{m,X} \in S[X], \ f_{1,X}, \ldots ,
f_{m,X} \in S
\end{equation}
be a  {monic} polynomial of degree $m\geq 2$. We denote by
\begin{equation}\label{eq202}
{\cal X}:= \mathrm{Spec}(S[X]/(h)) \  \mathrm{and} \  \eta: \ {\cal X} \longrightarrow \mathrm{Spec}S
\end{equation}
respectively the corresponding hypersurface and induced projection.

The total ring of fractions ${\cal X}$ is denoted by $L:=\mathrm{Tot}(S[X]/(h))$.
Given a point $y \in {\cal X}$, its ideal, residue field and multiplicity are respectively
denoted by $m_y$, $k(y)$ and $m(y)$.

\smallskip

For convenience of the reader, we make the definition of $m(y)$ explicit.
Let $s:=\eta (y)\in \mathrm{Spec}S$, $k(s)$ be the residue field of $S_s$ and
$\overline{h}\in k(s)[X]$ be the reduction of $h$. The point $y$ corresponds to a certain irreducible factor
$T$ of $\overline{h}$ with $k(y)=k(s)[X]/(T)$. One defines $m(y)$ by:
$$
m(y):=\mathrm{ord}_{m_{S[X]_y}}h\geq 1.
$$
By definition of regular local rings, we thus have:
$$
{\cal O}_{{\cal X},y} \ \mathrm{is} \ \mathrm{a} \ \mathrm{regular} \ \mathrm{local} \ \mathrm{ring}
\Leftrightarrow m(y)=1.
$$
The singular (i.e. not regular) locus (resp. {\it locus of multiplicity} $m$) of ${\cal X}$ is denoted by :
$$
\mathrm{Sing}{\cal X}=\{y \in {\cal X} : m(y)\geq 2\} \ (\mathrm{resp.}
\ \mathrm{Sing}_m {\cal X}:=\{ y \in {\cal X} : m(y)=m\}).
$$
Both are  viewed as reduced embedded subschemes of ${\cal X}$. Nontrivial material concerning regularity
and the multiplicity function is normally accompanied with a reference to \cite{EGA2} or \cite{Ma}.
A basic, but especially important property is Proposition \ref{deltainv}.\\

Given a ``linear change of" (one also says ``translation on") the
$X$-coor\-di\-nate, say $X':=X-\phi$, $\phi \in \hat{S}$, we still denote by
$$
h={X'}^m+f_{1,X'}{X'}^{m-1}+ \cdots +f_{m,X'}\in \hat{S}[X']
$$
the corresponding expansion of $h(X'+\phi)$, $f_{1,X'}, \ldots , f_{m,X'} \in \hat{S}$.
The explicit formula for this change of coordinate is :
\begin{equation}\label{eq2011}
f_{i,X'}= \begin{pmatrix}
  m \\
  i \\
\end{pmatrix}
\phi^i +\sum_{j=1}^i{ \begin{pmatrix}
  m -j\\
  i-j \\
\end{pmatrix}
f_{j,X}\phi^{i-j}}, \ 1 \leq i \leq m.
\end{equation}

Given $\phi \in S$ and a rational number $d \leq
\mathrm{ord}_{m_S}\phi$, we denote by $\mathrm{cl}_d\phi$ the {\it
initial form} of $\phi$ in $\mathrm{gr}_{m_S}S \simeq S/m_S
[U_1,\ldots,U_n]$ (resp. the null form) if $d =
\mathrm{ord}_{m_S}\phi$ (resp. otherwise). Similarly, if $I
\subseteq S$ and $d \leq \mathrm{ord}_{m_S}I$, we denote
$$
\mathrm{cl}_d I:= \mathrm{Vect}(\{\mathrm{cl}_d\phi \}_{\phi \in
I})\subseteq S/m_S [U_1,\ldots,U_n]_d.
$$

Suppose that a weight vector $\alpha =(\alpha_1 , \ldots , \alpha_n)\in {\R}^n_{\geq 0}$
is given. Let $\Gamma_\alpha: =\Z \alpha_1 + \cdots + \Z \alpha_n \subset \R$. For
$\mathbf{x}=(x_1, \ldots , x_n) \in {\R}^n_{\geq 0}$, denote
$$
\mid \mathbf{x} \mid_\alpha :=\alpha_1 x_1+ \cdots + \alpha_n x_n \in (\Gamma_\alpha )_{\geq 0} .
$$
An associated valuation
$\mu_\alpha$ of $K$ is defined by setting for $f \in S$, $f \neq 0$:
$$
\mu_\alpha (f):= \mathrm{max}\{ a \in \Gamma_\alpha : f \in
I_\alpha (a):=(\{u_1^{x_1} \cdots u_n^{x_n} : \mid \mathbf{x} \mid_\alpha  \geq a\})\}.
$$
It easily follows from the Noetherianity of $S$ that $\mu_\alpha (f)$ is well defined. One sets
$$
\mu_\alpha (f/g):= \mu_\alpha (f)- \mu_\alpha (g) \ \mathrm{for} \  f, g \in S, fg\neq 0.
$$
Note that $\mathrm{ord}_{m_S}=\mu_\mathbf{1}$, where $\mathbf{1}= (1, 1 , \ldots ,1)\in {\R}^n_{> 0}$. We
will systematically use the graded ring $\mathrm{gr}_\alpha S$ of $S$ w.r.t. $\mu_\alpha$: \index{$\mathrm{gr}_\alpha S$! the graded ring $\mathrm{gr}_\alpha S$ of $S$ w.r.t. $\mu_\alpha$}
\begin{equation}\label{eq:gradue}
    \mathrm{gr}_\alpha S \simeq S /(\{u_i : \alpha_i >0\}) [\{U_i : \alpha_i>0\}].
\end{equation}
If $a \in \Gamma_\alpha$ and $\phi \in S$
is given with $a \leq \mu_\alpha (\phi )$, its initial form
$\mathrm{cl}_{\alpha , a}\phi \in \mathrm{gr}_\alpha S$ is defined as before. Similarly, if $I
\subset S$ and $a \leq \mu_\alpha (I)$, we associate a $(\mathrm{gr}_\alpha S)_0$-module
denoted by
$$
\mathrm{cl}_{\alpha , a} I:=\mathrm{Span}(\{\mathrm{cl}_{\alpha , a}\phi \}_{\phi \in
I})\subseteq (\mathrm{gr}_\alpha S)_a .
$$

\subsection{Characteristic polyhedron and first invariants.}

\smallskip

 {Polygons have been used since Newton  to compute the leading terms in Puiseux parametrizations of plane branches.}
  {Nowadays, Newton polyhedra are classical tools for the study of singularities. They encode certain numerical data attached with a given singular germ and provide a rough approximation of the geometry of the singularity. }
 { For toric singularities  \cite{Mu}, resolution of singularities is recovered from their polyhedron.}
 {In general, it only provides invariants for bettering singularities by using a stepwise blowing up process \cite{Sc}, \cite{L4} for quasi-ordinary singularities. }


 {Hironaka showed how Newton polyhedra can be used to construct resolution of singularities for surfaces \cite{H3}.
For singularities of any dimension, one first projects the Newton polyhedron from a special face related to a transversal projection, then minimizes the image by suitable changes of coordinates. Hironaka's characteristic polyhedron is defined to be the closure of the image. 
}

 {Each face of the characteristic polyhedron is defined by some  monomial valuation which leads to associated graded rings and initial forms. From these data we will define our main invariants. We also study the behavior by base change and blowing ups.}

\medskip
Let $S$ and $(u_1, \ldots ,u_n)$ be fixed as above. Given a subset $J \subseteq \{1, \ldots ,n\}$,
we denote by
$$
I_J:=(\{u_j\}_{j\in J})\subset S \  \mathrm{and} \ \overline{S}^J:=S/I_J.
$$
We also use the notation $s^J \in  \mathrm{Spec}S$ to denote the point $s^J=I_J$, reserving the idealistic
notation $I_J$ to commutative algebraic formul{\ae}. The next proposition will be applied to each coefficient of $h$ in \eqref{eq201}.

\begin{prop}\label{monomexp}
Let $f \in S$. There exists a unique finite set $\mathbf{S}^J(f)\subset \N^J$
such that the following holds:
\begin{itemize}
    \item [(i)] the set of monomials $\{\prod_{j\in J}u_j^{a_j} : \mathbf{a}=(\{a_j\}_{j\in J})
    \in \mathbf{S}^J (f) \}$ forms a minimal system of generators of the ideal
    $$
    I(f):= \left (\left \{\prod_{j\in J}u_j^{a_j} : \mathbf{a}=(\{a_j\}_{j\in J}) \in \mathbf{S}^J (f)
    \right \}\right );
    $$
    \item [(ii)] there is an expansion
    \begin{equation}\label{eq2036}
    f=\sum_{\mathbf{a}\in \mathbf{S}^J (f)}\gamma(f,\mathbf{a}) \prod_{j\in J}u_j^{a_j}\in S,
    \ \gamma(f,\mathbf{a})\in S
    \end{equation}
    such that $\gamma(f,\mathbf{a})\not \in I_J$  {(i.e. $\gamma(f,\mathbf{a})$ is a unit in $S_{I_J}$)}
for every  $\mathbf{a}\in \mathbf{S}^J(f)$.
\end{itemize}
\end{prop}

\begin{proof}
Let $\widehat{S}^J$ be the formal completion of $S$ {\it along} $I_J$.
Since $I_J \subseteq m_S$,  $\widehat{S}^J$ is faithfully flat over $S$ \cite{Ma} Theorem 8.14(3).
Thus $I\widehat{S}^J \cap S=I$ for any ideal $I \subseteq S$, in particular for any monomial ideal in
$\{u_j\}_{j\in J}$. One deduces that property (i) and existence of an expansion (\ref{eq2036})
descend from $\widehat{S}^J$ to $S$.

\smallskip

Suppose that an expansion (\ref{eq2036}) exists for a given $\mathbf{S}^J (f)$ satisfying (i).
Each $S/I_J^{n+1}$, $n\geq 0$ has a structure of free $\overline{S}^J$-module with basis
$$
\left \{\prod_{j\in J}u_j^{a_j} : \mathbf{a}=(\{a_j\}_{j\in J})  \ \mathrm{and} \ \sum_{j\in J}a_j \leq n \right \}.
$$

Therefore the class $\gamma(f,\mathbf{a})+I_J$ is independent
of the chosen expansion (\ref{eq2036}) by the minimality property in (i).
This proves that the property $\gamma(f,\mathbf{a})\not \in I_J$ in (ii)
also descends from $\widehat{S}^J$ to $S$.
In other terms, we may assume that $S$ is $I_J$-adically complete.

\smallskip

Independent monomial generators in $S/I_J^n$ lift to independent monomial
generators in $S/I_J^{n+1}$ for every $n \geq 1$. One easily deduces the existence
of an expansion (ii) satisfying (i) for some finite subset $\mathbf{S}^J(f)\subset \N^J$,
since $S$ is $I_J$-adically complete and Noetherian.

\smallskip

Uniqueness of $\mathbf{S}^J(f)$ is also checked by taking images in  $S/I_J^{n+1}$ for
some $n>>0$.
\end{proof}

 {Given an equation $h \in S[X]$ (\ref{eq201}) and a r.s.p. $(u_1,\ldots ,u_n)$ of $S$, let us write a finite expansion:
\begin{equation}\label{eq000}
  h:=\sum_{i,A}c_{i,A}X^i u_1^{a_1} \cdots u_n^{a_n},\ A=(a_1,\cdots,a_n)\subset\N^n,\ c_{i,A}\in S
\end{equation}
$c_{i,A}$ invertible in $S$.
The Newton polyhedron $NP (h;u_1,\ldots,u_n;X)$ w.r.t. the variables $(u_1, \ldots ,u_n,X)$  associated to $h$ is defined as:
$$
NP (h;u_1,\ldots,u_n;X):=\hbox{convex hull of}\ \bigcup_{c_{i,A}} (A,i)+\R^{n+1}_{\geq 0} \subseteq \R^{n+1}_{\geq 0}.
$$
}
Let $P:=(0, \ldots ,0,1) \in \R^{n+1}_{\geq 0}$, so $ P \in {1 \over m}NP (h;u_1,\ldots,u_n;X)$, and
$$
\mathbf{p}: \R^{n+1} \ \backslash \{P\} \longrightarrow \R^{n}
$$
be the projection on the $(u_1,\ldots ,u_n)$-space  {from the point $P$}. We define a polyhedron by:
$$
\Delta_S (h;u_1,\ldots,u_n;X):=\mathbf{p}\left ({1 \over m}NP (h;u_1,\ldots,u_n;X) \cap \{x_{n+1}<1\}\right )
\subseteq {\R}^n_{\geq 0}.
$$
The {\it characteristic polyhedron} is introduced in a more general context in \cite{H3}. In
our setting, it consists in minimizing $\Delta_S (h;u_1,\ldots,u_n;X')$ over all linear
changes of coordinates $X'=X-\phi$, $\phi \in \hat{S}$ (\ref{eq2011}):  {see Definition~\ref{defminimal} below.}
In this section, we review and adapt notations to fit our purposes.
A fundamental algebraicity result is borrowed from \cite{CoP3}
in  Proposition \ref{Deltaalg} below. 

\begin{defn}\label{defDelta}\textbf{(Associated Polyhedron).}
\index{Associated Polyhedron, Definition~\ref{defDelta}}Given an equation $h\in S[X]$  (\ref{eq201}) and $J \subseteq \{1, \ldots ,n\}$,
we define a rational polyhedron:
$$
\Delta_S (h;\{u_j\}_{j\in J};X):=\mathrm{Conv} \left ( \bigcup_{i=1}^m
\bigcup_{\mathbf{a}\in \mathbf{S}^J (f_{i,X})}
\left \{{\mathbf{a} \over i} + {\R}^J_{\geq 0}\right \} \right )\subseteq \R^J_{\geq 0}.
$$\index{$\Delta_S$,  Definition~\ref{defDelta}}
\end{defn}

\begin{defn}\label{definh}\textbf{(Initial forms).}
\index{Initial forms, Definition~\ref{definh}! Initial forms defined by a weight vector $\alpha$}Let $\alpha =(\{\alpha_j\}_{j \in J})\in {\R}^J_{> 0}$ be a weight vector.
We define
$$
\delta_\alpha (h;\{u_j\}_{j\in J};X):=
\mathrm{min}\{ \mid  {\mathbf{y}} \mid_\alpha  :  {\mathbf{y}} \in \Delta_S (h;\{u_j\}_{j\in J};X)\}.
$$ \index{$\delta_\alpha $, Definition~\ref{definh} ! $\delta_\alpha (h;\{u_j\}_{j\in J};X)$ where $\alpha$ is  a weight vector}

The weight vector defines a {\it compact face} $\sigma_\alpha$ of $\Delta_S (h;\{u_j\}_{j\in J};X)$ \index{$\sigma_\alpha $,  Definition~\ref{definh} !  compact face defined by a weight vector $\alpha$} compact face
by:
$$
\sigma_\alpha := \{  {\mathbf{y}} \in \Delta_S (h;\{u_j\}_{j\in J};X) : \
\mid  {\mathbf{y}} \mid_\alpha   =\delta_\alpha (h;\{u_j\}_{j\in J};X)\}.
$$

Given $h$ and $\alpha$, the grading
of $\mathrm{gr}_\alpha S$ \eqref{eq:gradue} can be extended to $\mathrm{gr}_\alpha (S [X])=(\mathrm{gr}_\alpha S) [X]$ by setting:
$$
\mathrm{deg}X:= \delta_\alpha (h;\{u_j\}_{j\in J};X).
$$
Then the {\it initial form} $\mathrm{in}_\alpha h$ of $h$ w.r.t. $\alpha$ is the polynomial
\begin{equation}\label{eq2035}
    \mathrm{in}_\alpha h:=X^m + \sum_{i=1}^mF_{i,X,\alpha} X^{m-i} \in (\mathrm{gr}_\alpha S) [X],
\end{equation}
where
$$
F_{i,X,\alpha}:= \sum_{ {\mathbf{y}}\in \sigma_\alpha}{\overline{\gamma} (f_{i,X},i {\mathbf{y}}) U^{i {\mathbf{y}}}},
$$
and bars  {denote} images in $(\mathrm{gr}_\alpha S)_0=\overline{S}^J$, i.e.
$$
\overline{\gamma} (f_{i,X},i {\mathbf{y}}):= \mathrm{cl}_{\alpha , 0}\gamma (f_{i,X},i {\mathbf{y}})\in
(\mathrm{gr}_\alpha S)_0=\overline{S}^J .
$$
By convention, we take $\overline{\gamma} (f_{i,X},i {\mathbf{y}})=0$ in these formul{\ae} whenever
$i {\mathbf{y}} \not \in \mathbf{S}^J (f_{i,X})$. Note that the  polynomial $\mathrm{in}_\alpha h$ is {\it homogeneous}
for this grading of degree $m\delta_\alpha (h;\{u_j\}_{j\in J};X)$.
\end{defn}

\begin{rem}
Any vertex of $\Delta_S (h;\{u_j\}_{j\in J};X)$ has coordinates in ${1 \over m!}\N$. We have:
$$
\Delta_S (h;\{u_j\}_{j\in J};X) = \emptyset \Leftrightarrow h = X^m.
$$
\end{rem}
\begin{rem}\label{rem:initsommet}
It is worth emphasizing that the polynomial $\mathrm{in}_\alpha h$ only depends on the face
$\sigma_\alpha $ and not on the specific weight vector $\alpha$ defining it.
 {Given a vertex $\mathbf{y}\in \Delta_S (h;\{u_j\}_{j\in J};X)$,  $ \mathrm{in}_\alpha h$ is the \emph{$\mathbf{y}$-initial}
$ \mathrm{in}_\mathbf{y} h$ defined in \cite{H3}~Definition~(3.7) for any $\alpha \in {\R}^J_{> 0}$ such that $\sigma_\alpha =\mathbf{y}$.
This motivates the need to consider weights $\alpha\not=\mathbf{1}$.}

\end{rem}

We now briefly review the behaviour of polyhedra and initial forms under basic operations
such as formal completion, localization and projection onto a regular subscheme. The case of
regular local morphisms $S \subset \tilde{S}$, $\tilde{S}$ excellent will be considered further on. \\

With notations as above, let $\alpha \in {\R}^J_{> 0}$ be a weight vector and
$$
\sigma_\alpha \subset \Delta_S (h;\{u_j\}_{j\in J};X), \ \mathrm{in}_\alpha h \in (\mathrm{gr}_\alpha S) [X].
$$

\noindent {\it Formal Completion:} $\hat{S}$ is excellent \cite{EGA2} Theorem 7.8.3(iii).
Proposition \ref{monomexp} and Definition \ref{defDelta} give an identification
\begin{equation}\label{eq2034}
\Delta_S (h;\{u_j\}_{j\in J};X)=\Delta_{\hat{S}} (h;\{u_j\}_{j\in J};X).
\end{equation}
This identification preserves the initial form $\mathrm{in}_\alpha h $ for each weight vector $\alpha$ via the inclusion
$ \mathrm{gr}_\alpha S \subseteq \mathrm{gr}_\alpha \hat{S} \simeq \mathrm{gr}_\alpha S \otimes_S \hat{S}$.

\smallskip

\noindent {\it Localization:} the regular local ring $S_{s^J}$ is excellent if $S$ is excellent \cite{EGA2} Theorem 7.4.4.
Similarly, the identifications
\begin{equation}\label{eq2031}
    \Delta_S (h;\{u_j\}_{j\in J};X) = \Delta_{S_{s^J}}(h;\{u_j\}_{j\in J};X)
\end{equation}
also preserve the initial form $\mathrm{in}_\alpha h $ (\ref{eq2035}) via the inclusion
$$
\mathrm{gr}_\alpha S\subseteq \mathrm{gr}_\alpha S_{s^J} \simeq
(\mathrm{gr}_\alpha S) \otimes_S QF(\overline{S}^J).
$$

\noindent {\it Projection:} let $J \subseteq \{1, \ldots ,n\}$ and denote by $J':= \{1, \ldots ,n\} \backslash J$ its
complement. The regular local ring $\overline{S}^J$ is excellent if $S$ is excellent. A r.s.p.
of $\overline{S}^J$ is $(\{\overline{u}_{j'}\}_{j'\in J'})$, where bars denote images in $\overline{S}^J$.
With notations as above, we have:
\begin{equation}\label{eq2032}
\Delta_S (h;\{u_j\}_{j\in J};X)=\mathrm{pr}^J \Delta_S (h;u_1,\ldots,u_n;X),
\end{equation}
where $\mathrm{pr}^J: \R^n \rightarrow \R^J, \ \mathbf{x} \mapsto \mathbf{y}=(\{x_j\}_{j\in J})$
denotes the projection. Let
$$
f_{i,X}=\sum_{\mathbf{a}\in \mathbf{S}(f_{i,X})}\gamma(f_{i,X},\mathbf{a}){u_1^{a_1} \cdots u_n^{a_n}}\in S,
$$
be an expansion (\ref{eq2036}) (for the subset $\{1, \ldots ,n\}$, where $\mathbf{S}(f_{i,X})$ here stands for
$\mathbf{S}^{\{1, \ldots ,n\}}(f_{i,X})$), $1 \leq i \leq m$. Then (\ref{eq2035}) is given by
\begin{equation}\label{eq2033}
F_{i,X,\alpha}:= \sum_{\mathbf{y}\in \sigma_\alpha}
\left ( \sum_{\mathrm{pr}^J(\mathbf{x})=\mathbf{y}}
{\overline{\gamma} (f_{i,X},i\mathbf{x}) \prod_{j'\in J'}\overline{u}_{j'}^{ix_{j'}}} \right )\prod_{j\in J}U_j^{iy_j},
\end{equation}
where bars   {denote} images in $(\mathrm{gr}_\alpha S)_0=\overline{S}^J$ as before (recall that by convention, we take
$\overline{\gamma} (f_{i,X},i\mathbf{x}):=0$ in this formula if $i\mathbf{x} \not \in \mathbf{S}(f_{i,X})$).  {In other terms: }
 {$$  \overline{\gamma} (f_{i,X},i\mathbf{y})=  \sum_{\mathrm{pr}^J(\mathbf{x})=\mathbf{y}}
{\overline{\gamma} (f_{i,X},i\mathbf{x}) \prod_{j'\in J'}\overline{u}_{j'}^{ix_{j'}}} . $$}


\begin{defn}\label{defsolvable}\textbf{(Solvable vertices).}
\index{solvable vertices, Definition~\ref{defsolvable}}Let $ {\mathbf{y}}\in \R^J$ be a vertex of the polyhedron $\Delta_S (h;\{u_j\}_{j\in J};X)$, that is, a 0-dimensional
face $\sigma=\{ {\mathbf{y}}\}$.  {Following Hironaka (cf. Remark~\ref{rem:initsommet}), we denote by}
$$
\mathrm{in}_{ {\mathbf{y}}} h =X^m + \sum_{i=1}^mF_{i,X, {\mathbf{y}}} X^{m-i} \in (\mathrm{gr}_\alpha S) [X]
$$
the initial form polynomial (\ref{eq2035}) w.r.t. any defining weight vector $\alpha $.
We will say that $ {\mathbf{y}}$ is solvable if $ {\mathbf{y}}\in \N^J$ and there exists
$\overline{\lambda} \in \overline{S}^J$ such that
$$
\mathrm{in}_{ {\mathbf{y}}} h=(X - \overline{\lambda} U^{ {\mathbf{y}}})^m.
$$
\end{defn}

Explicitly, with notations as in (\ref{eq2035}) {\it sqq.}, the latter equality means that
$$
\overline{\gamma} (f_{i,X},i {\mathbf{y}}) = (-1)^i \begin{pmatrix}
  m \\
  i \\
\end{pmatrix}
\overline{\lambda}^i \in \overline{S}^J , \ 1 \leq i \leq m .
$$
Note that $\begin{pmatrix}
  m \\
  i \\
\end{pmatrix}
\in \overline{S}^J$ is not a unit in general when $\mathrm{char}(S/m_S)>0$.\\

The following result is a rewriting of \cite{H3} in this hypersurface situation.

\smallskip

\begin{prop}\label{Deltamin}\textbf{(Hironaka).}
There exists a linear change of the
$X$-coordinate $Z:=X-\theta$, with $\theta \in \hat {S}$, such
that
\begin{equation}\label{eq204}
    \Delta_{\hat{S}} (h;\{u_j\}_{j\in J};Z)=
    \min_{X'}\Delta_{\hat{S}} (h;\{u_j\}_{j\in J};X'),
\end{equation}
where the minimum is taken w.r.t. inclusions and over all possible
linear changes of coordinates $X':=X-\phi$, $\phi \in \hat{S}$.

\smallskip

Given $X':=X-\phi$, $\phi \in \hat{S}$, $\Delta_{\hat{S}} (h;\{u_j\}_{j\in J};X')$
achieves equality in (\ref{eq204}) if and only if it has no solvable vertex.

\smallskip

If $S$ is excellent, there is an equivalence
$$
\Delta_{\hat{S}} (h;\{u_j\}_{j\in J};Z)=\emptyset \Leftrightarrow  \exists g \in S
 : h = (X - g)^m.
$$
\end{prop}

\begin{proof}
This is respectively \cite{H3} Hironaka's Vertex Preparation Lemma (3.10) and Theorem (4.8),
and \cite{CoP3} Lemma II.1.
\end{proof}

\begin{defn}\label{defminimal} \textbf{(Characteristic Polyhedron).}
\index{Characteristic Polyhedron, Definition~\ref{defminimal}} For $X':=X-\phi$, $\phi \in \hat{S}$, we will say that the polyhedron
$\Delta_{\hat{S}} (h;\{u_j\}_{j\in J};X')$ is minimal if it has no solvable vertex.
\end{defn}

\begin{exam}
Let $p$ be a prime number and $n\in \Z$ not divisible by $p$. We take:
$$
S:=\Z_{(p)} \ \mathrm{and} \ h:=X^p -np^a \in S[X], \ a\geq 0.
$$
The following holds:
\begin{itemize}
  \item [(1)] if $a \not \in p\Z$, then $\Delta_{\Z_p}(h; p ;X)=\lbrack a/p, +\infty \lbrack$ is minimal;
  \item [(2)]if $a \in p\Z$, then $\Delta_{\Z_p}(h; p ;Z)$ is minimal, where $Z:=X-np^{a/p}$ and we have:
  $$
\Delta_{\Z_p}(h; p ;Z)=
\left\{
\begin{array}{ccc}
  \lbrack {a+1 \over p}, +\infty \lbrack  \hfill{} & \mathrm{if} &  n^p-n \not \in p^2\Z \\
   & & \\
  \lbrack {a \over p}+{1\over p-1}, +\infty \lbrack & \mathrm{if} &  n^p-n \in p^2\Z\\
\end{array}
\right .
.
$$
\end{itemize}
\end{exam}

\begin{prop} \label{deltainv}
 {With notations and conventions as in (\ref{eq201}) and (\ref{eq202}), assume that  $J=\{1, \ldots ,n\}$ and $\alpha =\mathbf{1}$ (so $\mu_\mathbf{1}=\mathrm{ord}_{m_S}$)
\cite{H3} \cite{Co6}.} Then
the rational number $\delta_\mathbf{1}(h;u_1,\ldots,u_n;Z)$ is
independent of the r.s.p. $(u_1, \ldots, u_n)$ and $Z=X-\theta$,
$\theta \in \hat{S}$ such that $\Delta_{\hat{S}}(h; u_1,\ldots,u_n;Z)$ is
minimal.

If $\Delta_{\hat{S}}(h; u_1,\ldots,u_n;Z)$ is minimal,
the following characterizations hold:
\begin{itemize}
    \item [(i)] $\delta_\mathbf{1}(h;u_1,\ldots,u_n;Z)>0 \Leftrightarrow
    (\eta^{-1}(m_S)=\{x\} \ \mathrm{and} \ k(x)=S/m_S)$;
    \item [(ii)] $\delta_\mathbf{1}(h;u_1,\ldots,u_n;Z)\geq 1 \Leftrightarrow
    \eta^{-1}(m_S)\cap  \mathrm{Sing}_m{\cal X}\neq \emptyset$.
\end{itemize}
\end{prop}

\begin{proof}
Let $(Z',u'_1,\ldots,u'_n)$ and $(Z,u_1,\ldots,u_n)$ be two systems of coordinates such that both
polyhedra $\Delta_{\hat{S}} (h;u'_1,\ldots,u'_n;Z')$ and $\Delta_{\hat{S}}
(h;u_1,\ldots,u_n;Z)$ are minimal. Suppose that
$\delta_\mathbf{1}(h;u'_1,\ldots,u'_n;Z')>\delta_\mathbf{1}(h;u_1,\ldots,u_n;Z)$. Then
 {$$
f_{i,Z'}^{m!\over i} \in m_S^{{m! }\delta_\mathbf{1}(h;u'_1,\ldots,u'_n;Z')}
$$}
for each $i$, $1 \leq i \leq m$, hence
$$
\delta_\mathbf{1}(h;u_1,\ldots,u_n;Z')\geq \delta_\mathbf{1}(h;u'_1,\ldots,u'_n;Z')>
\delta_\mathbf{1}(h;u_1,\ldots,u_n;Z).
$$
This contradicts the assumption $\Delta_{\hat{S}} (h;u_1,\ldots,u_n;Z)$
minimal. The first assertion follows by symmetry.

\smallskip

Let $\overline{h}\in S/m_S[Z]$ be the reduction of $h$ modulo $m_S$.
Since
$$
\eta^{-1}(m_S)=\mathrm{Spec}(S/m_S[Z]/(\overline{h})),
$$
(i) and the ``only if'' part in (ii) are immediate from the definitions. We have
$$
\mathrm{ord}_{x}h(Z) \leq \mathrm{ord}_{x}\overline{h}(Z)\leq m ,
$$
hence $x \in \mathrm{Sing}_m{\cal X}$ implies $\overline{h}(Z)=(Z- \lambda)^m$
for some $\lambda \in S/m_S$. Since $\Delta_{\hat{S}} (h;u_1,\ldots,u_n;Z)$
is minimal,  $\mathbf{0} \in \R^n$ is not a solvable vertex and therefore we have $\lambda =0$.
This proves that (i) holds, the ``if'' part in (ii) being then obvious.
\end{proof}

\begin{defn}\label{defdelta}
Let $s \in \mathrm{Spec}S$, $(v_1,\ldots ,v_{n(s)})$ be a r.s.p. of $S_s$ and $y
\in \eta^{-1}(s)$. Let $Z:=X-\theta$, $\theta \in \widehat{S_s}$ be such that
$\Delta_{\widehat{S_s}} (h;v_1,\ldots,v_{n(s)};Z)$ is minimal, where $\widehat{S_s}$ denotes
the formal completion of $S_s$ w.r.t. its maximal ideal. We let:
$$
\delta (y):=\delta_\mathbf{1}(h;v_1,\ldots,v_{n(s)};Z)=\min_{1 \leq i \leq
m}\left \{{\mathrm{ord}_{m_{\widehat{S_s}}}f_{i,Z} \over i}\right \}\in {1\over m!}\N.\\
$$
\end{defn}

This invariant is classical and appears in e.g. \cite{Co1}, \cite{Co2} and
\cite{BeV1} Definition 4.2 and Proposition 4.8 in an equal characteristic context.
Our main resolution invariants will be defined in terms of
coordinates $(u_1, \ldots ,u_n)$ and $Z=X -\theta$, $\theta \in
\hat{S}$ such that $\Delta_{\hat{S}} (h;u_1,\ldots,u_n;Z)$ is minimal.
Since minimizing polyhedra involves in principle choosing {\it formal}
coordinates, an {\it algebraic} version will be useful for proving
the constructibility of our invariants. The following proposition is
fundamental for this purpose. When $\mathrm{char}S/m_S=0$,
the first statement in the proposition easily
follows from Proposition \ref{Deltamin} by applying the Tschirnhausen
transformation (take $\theta =-{1 \over m}f_{1,X}$ below). \\

{\it We assume from this point on that $S$ is excellent.}\\

\begin{prop}\label{Deltaalg}\cite{CoP3}
Given $h\in S[X]$ (\ref{eq201}) and a r.s.p. $(u_1,\ldots ,u_n)$ of $S$, there exists $Z:=X-\theta$,
$\theta \in S$ such that $\Delta_{\hat{S}}(h;u_1,\ldots, u_n;Z)$ is minimal.

\smallskip

For any such $Z$, the following holds: for every subset $J \subseteq \{1, \ldots ,n\}$, the polyhedron
$\Delta_{\widehat{S_{s^J}}}(h;\{u_j\}_{j\in J};Z)$ is also minimal and is computed by:
\begin{equation}\label{eq2044}
\Delta_{\widehat{S_{s^J}}} (h;\{u_j\}_{j\in J};Z)=\mathrm{pr}^J \Delta_{\hat{S}} (h;u_1,\ldots,u_n;Z),
\end{equation}
where $\mathrm{pr}^J: \R^n \rightarrow \R^J, \ \mathbf{x} \mapsto \mathbf{y}=(\{x_j\}_{j\in J})$
denotes the projection. In particular, we have
$$
\delta (y)=\min \left \{{1 \over i}\sum_{j\in J}a_j ,
\ \mathbf{a} \in \mathbf{S}^{\{1, \ldots ,n\}}(f_{i,Z}), \ 1 \leq i \leq m\right \}, \ y \in \eta^{-1}(s^J).
$$
\end{prop}

\begin{proof}
The proposition is trivial if $\mathbf{0} \in \R^n$ is a nonsolvable vertex of
the polyhedron $\Delta_{\hat{S}}(h;u_1,\ldots, u_n;Z)$, taking $Z:=X$. Otherwise it can be assumed that
$f_{i,X}\in m_S$, $1 \leq i \leq m$. The first statement is \cite{CoP3} Corollary II.4.

\smallskip

Formula (\ref{eq2044}) follows from (\ref{eq2034}) (\ref{eq2031}) (\ref{eq2032}). To prove minimality,
suppose that $\mathbf{y} \in \N^J$ is a
solvable vertex of $\Delta_{\widehat{S_{s^J}}} (h;\{u_j\}_{j\in J};Z)$ defined by some $\alpha \in \R^J_{>0}$.
By definition,
\begin{equation}\label{eq2045}
\exists \overline{\lambda} \in QF(\overline{S}^J) : \mathrm{in}_\mathbf{y} h=(Z - \overline{\lambda} U^\mathbf{y})^m.
\end{equation}
By (\ref{eq2033}), we have $\overline{\lambda}^m =(-1)^mU^{-m\mathbf{y}}F_{m,Z,\alpha} \in \overline{S}^J$. Hence
$\overline{\lambda} \in \overline{S}^J$, since the regular ring $\overline{S}^J$ is integrally closed.
By (\ref{eq2044}), there exists a  {\it vertex} $\mathbf{x} \in \Delta_{\hat{S}}(h;u_1,\ldots, u_n;Z)$
such that $\mathbf{y}=\mathrm{pr}^J(\mathbf{x})$. Lifting up,  there exists $\beta \in \R^n_{>0}$,
$\alpha = \mathrm{pr}^J(\beta)$ defining $\mathbf{x}$, and we let $\alpha ':= \mathrm{pr}^{J'}(\beta)$. There
is an induced valuation $\mu_{\alpha '}$ on $\overline{S}^J$. The initial form of $\overline{\lambda}$ in
$\mathrm{gr}_{\alpha '}\overline{S}^J$  has the form
$$
\lambda \prod_{j' \in J'}\overline{U}_{j'}^{x_{j'}},
\ \lambda \in S/m_S, \ \lambda \neq 0, \ \{x_{j'}\}_{j'\in J'}\in \N^{J'}.
$$
Collecting together (\ref{eq2033}) and (\ref{eq2045}), we get
$\mathrm{in}_\mathbf{x} h = (Z - \lambda  U^\mathbf{x})^m$,
i.e. $\mathbf{x}$ is a solvable vertex: a contradiction. Therefore
$\Delta_{\widehat{S_{s^J}}} (h;\{u_j\}_{j\in J};Z)$ has no solvable vertex, hence
is minimal by the second statement in Proposition \ref{Deltamin}.
The last statement is a rewriting of Definition \ref{defdelta}.
\end{proof}

\begin{rem}
This proposition allows us to skip the reference to formal completion when stating that
a certain polyhedron is minimal, i.e. given $Z:=X-\phi$, $\phi \in S$,
the statement ``$\Delta_{S}(h;u_1,\ldots, u_n;Z)$ is minimal" stands for
``$\Delta_{\hat{S}}(h;u_1,\ldots, u_n;Z)$ is minimal". On the other hand, we will keep the
reference to the regular local ring $S$ since we are also interested in base change.
\end{rem}

Let $S \subseteq \tilde{S}$ be a {\it local} base change which is {\it regular}, i.e. flat with
geometrically regular fibers \cite{EGA2} Definition 6.8.1(iv). In particular $\tilde{S}$ is
regular \cite{EGA2} Proposition 6.5.1(ii) and faithfully flat over $S$. The ring $\tilde{S}$ is not
excellent in general, but this certainly holds in the following cases:

\begin{itemize}
  \item [(i)] $\tilde{S}=\hat{S}$ \cite{EGA2} 7.8.3(iii);
  \item [(ii)] $\tilde{S}$ is ind-\'etale over $S$ \cite{ILO} Theorem I.8.1(iv), or
  \item [(iii)] $\tilde{S}$ is essentially of finite type over $S$, i.e. smooth over $S$ \cite{EGA2}
Proposition 7.8.6(i).
\end{itemize}

An important special case of (ii) is when $\tilde{S}$ is the Henselization
or strict Henselization of $S$. When regular base changes are concerned, we always assume that
$\tilde{S}$ is excellent. These conditions are preserved by localizing, i.e. replacing
$S \subseteq \tilde{S}$ by $S_s \subseteq \tilde{S}_{\tilde{s}}$, $\tilde{s}\in \mathrm{Spec}\tilde{S}$
and $s \in \mathrm{Spec}S$ its image.\\

\begin{nota}\label{notageomreg1}
Let $S \subseteq \tilde{S}$ be a local base change which is regular, $\tilde{S}$ excellent,
$\tilde{s}\in \mathrm{Spec}\tilde{S}$ with image $m_S \in \mathrm{Spec}S$. Any r.s.p. $(u_1,\ldots,u_n)$
of $S$ can be extended to a r.s.p.  $(u_1,\ldots ,u_{\tilde{n}})$ of $\tilde{S}$. We let
$\tilde{h}\in \tilde{S}[X]$ be the image of $h$ and
$$
\tilde{\eta}: \ \tilde{{\cal X}}={\cal X}\times_S \mathrm{Spec}\tilde{S} \rightarrow \mathrm{Spec}\tilde{S}.
$$
\end{nota}

It follows from Definition \ref{defsolvable} that, if $\mathbf{x}\in \R^n_{\geq 0}$ is a
nonsolvable vertex of $\Delta_S(h; u_1,\ldots,u_n;Z)$, the vertex
$$
(\mathbf{x},\underbrace{0, \ldots ,0}_{\tilde{n} -n}) \in
\Delta_{\tilde{S}}(h; u_1,\ldots,u_{\tilde{n}};Z)\subseteq \R^{\tilde{n}}_{\geq 0}
$$
is nonsolvable provided that    {$ (S/m_S) \cap (\tilde{S}/m_{\tilde{S}})^p=(S/m_S)^p$}. This is of course
always satisfied when $S/m_S$ is perfect (e.g. $\mathrm{char}S/m_S=0$). An obvious
consequence of the second statement in Proposition \ref{Deltamin} is:

\begin{prop}\label{Deltageomreg} {\textbf{(Behavior under regular base change).}}
Let $S \subseteq \tilde{S}$ be a local base change which is regular, $\tilde{S}$ excellent.
Assume that
 {$$ (S/m_S) \cap (\tilde{S}/m_{\tilde{S}})^p=(S/m_S)^p.$$}
Let $Z=X-\theta$, $\theta \in S$, be such that $\Delta_{S}(h; u_1,\ldots,u_n;Z)$ is
minimal. Then
$$
\Delta_{\tilde{S}}(h; u_1,\ldots,u_{\tilde{n}};Z)=
\Delta_{S}(h; u_1,\ldots,u_n;Z)\times \R^{\tilde{n}-n}_{\geq 0}\subseteq \R^{\tilde{n}}_{\geq 0}
$$
and this polyhedron is minimal.
\end{prop}

Note that the assumptions of the proposition are satisfied in the above situation (ii):
$\tilde{S}$ is ind-\'etale over $S$. In situation (iii), i.e. $\tilde{S}$ smooth over $S$, the
following example will make the situation clear:

\begin{exam}\label{ex:kangourou}
Let $(S,m_S,k)$ be an excellent DVR, $\mathrm{char}k=p>0$, and $\gamma \in S$ be a unit.
Let  $\lambda \in k$ be the residue of $\gamma$ and assume furthermore that
$$
h:=X^p -\gamma u_1^{pa} \in S[X], \ a \geq 1, \ \lambda \in k \backslash k^p.
$$
Then $\Delta_{S}(h; u_1 ;X)=\lbrack a , +\infty \lbrack$ and is minimal. Take
$\tilde{S}=S[t]_{(u_1,P(t))}$, where $P$ is a monic polynomial with irreducible residue
$\overline{P}(t) \in k[t]$ (resp. $P=0$). Let $u_2:=P(t)$, so $(u_1,u_2)$ (resp. $(u_1)$)
is a r.s.p. of $\tilde{S}$. Let
$$
\tilde{k}:=\tilde{S}/m_{\tilde{S}}=k[t]/(\overline{P}(t)) \ (\mathrm{resp.} \ \tilde{k}=k(t))
$$
be the residue field of $\tilde{S}$. Setting $\{\tilde{x} \}=\tilde{\eta}^{-1}(m_{\tilde{S}})$, we have
$$
    \left\{
\begin{array}{ccc}
  \delta (\tilde{x})=a \hfill{}    & \mathrm{if} &  \lambda \not \in \tilde{k}^p \\
  \delta (\tilde{x})=a+{1 \over p} & \mathrm{if} &  \lambda \in \tilde{k}^p\\
\end{array}
\right .
.
$$
This is obvious if $\lambda \not \in \tilde{k}^p$  {(in particular when $P=0$)}; if $\lambda \in \tilde{k}^p$, take
$$
Z:= X - \tilde{\gamma}u_1^a, \  {\mathrm{with}} \ \tilde{v}:=\tilde{\gamma}^p -\gamma \in m_{\tilde{S}}.
$$
 {We claim that $(u_1,\tilde{v})$ is a r.s.p. of $\tilde{S}$. Indeed,  $\tilde{S}$ is smooth over $S$ and $S[t]/ (t^p-\gamma)$ is regular at its closed point. Hence
$$ \tilde{S} \otimes_S S[t]/ (t^p-\gamma)=\tilde{S}[t]/ (t^p-\gamma)\simeq \tilde{S}[t']/({t'}^p+\tilde{v})$$ is also regular. So $(u_1,\tilde{v},t')$ is a r.s.p. at $(\tilde{x},0)$.}
We have:
$$
\Delta_{\tilde{S}}(\tilde{h};u_1,\tilde{v};Z)=(a,1 / p)+\R^2_{\geq 0}.
$$

In particular, the function
$$
\A^1_k =\{x\}\times \A^1_k \subset {\cal X}\times_k\A^1_k \rightarrow {1 \over p}\N, \ \tilde{x}\mapsto \delta (\tilde{x})
$$
is not a constructible function.
\end{exam}

Proposition \ref{Deltaalg} and Proposition \ref{Deltageomreg} suggest the following question. An
affirmative answer would be very useful in order  to build geometrical invariants from characteristic
polyhedra. Proposition \ref{Deltageomreg} answers in the affirmative  {only} when $S/m_S$ is perfect,   {taking
$\tilde{S}:=S$ in the answer to the following question.}

\begin{quest}
Let $S$ be an excellent regular local ring with r.s.p. $(u_1, \ldots ,u_n)$ and $h \in S[X]$ (\ref{eq201}).
Does there exist a smooth local base change $S\subseteq \tilde{S}$, a r.s.p. $(u_1, \ldots ,u_{\tilde{n}})$
of $\tilde{S}$ extending $(u_1, \ldots ,u_n)$ and $Z=X -\tilde{\phi}$, $\tilde{\phi} \in \tilde{S}$,
such that the following holds:

\smallskip

``for every smooth local base change $\tilde{S}\subseteq S'$ and r.s.p. $(u_1, \ldots ,u_{n'})$ of $S'$
extending $(u_1, \ldots ,u_{\tilde{n}})$, the polyhedron $\Delta_{S'}(h;u_1, \ldots ,u_{n'};Z)$
is minimal"?
\end{quest}

Uncovering transformation rules for the characteristic polyhedron under blowing up is a
major problem, {\it vid.} \cite{H3} p.254. A good behavior is
known in the  special case of a blowing up along a Hironaka permissible subscheme  {(cf.
Definition~\ref{Hironakapermis})} and an exceptional point at the origin of some standard chart.

\begin{prop}\label{originchart} {\textbf{(Behavior under blowing up).}}
With notations as before, let $J\subseteq \{1, \ldots ,n\}$, $y \in \eta^{-1}(s^J)$ and assume that $\delta (y)\geq 1$.
Fix $j_0 \in J$ and let $S':= S[\{u'_j\}_{j \in J}]_{(u'_1, \ldots , u'_n)}$, where
$$
    \left\{
\begin{array}{ccccc}
  u'_j & := & u_j/u_{j_0} & \mathrm{if} &  j \in J \backslash \{j_0\}; \\
  u'_j & := & u_j          & \mathrm{if} & j \in J' \cup \{j_0\}.\\
\end{array}
\right .
$$
Let $Z=X-\theta$, $\theta \in S$, with $\Delta_{S}(h; u_1,\ldots,u_n;Z)$ minimal and define:
\begin{equation}\label{eq2043}
h'(Z'):=u_{j_0}^{-m}h(Z)=
{Z'}^m +u_{j_0}^{-1}f_{1,Z}{Z'}^{m-1}+ \cdots +u_{j_0}^{-m}f_{m,Z} \in S'[Z'],
\end{equation}
where $Z':= Z/u_{j_0}$. Define a map
$l: \ \R^n \longrightarrow \R^n$ by
\begin{equation}\label{eq2042}
\mathbf{x}=(x_1, \ldots ,x_n)\mapsto \mathbf{x}'=
(x_1, \ldots ,x_{j_0-1}, \sum_{j\in J}x_j -1, x_{j_0+1}, \ldots ,x_n).
\end{equation}
Then $l(\Delta_{S}(h; u_1,\ldots,u_n;Z))=\Delta_{S'}(h'; u'_1,\ldots,u'_n;Z')$ and
this polyhedron is minimal.
\end{prop}

\begin{proof}
The assumption $\delta (y)\geq 1$ forces $f_{i,Z} \in I_J^i$ by the last
statement in Proposition \ref{Deltaalg}. Therefore (\ref{eq2043}) makes sense, i.e.
$h'(Z') \in S'[Z']$. Since $l$ is one-to-one, we have
$$
{1 \over i}\mathbf{S}^{\{1, \ldots ,n\}} (f_{i,Z'})\subseteq
l\left ({1 \over i}\mathbf{S}^{\{1, \ldots ,n\}} (f_{i,Z})\right ), \ 1 \leq i \leq m,
$$
with notations as in Proposition \ref{monomexp}. By  Definition \ref{defDelta}, we get:
$$
l(\Delta_{S}(h; u_1,\ldots,u_n;Z))=\Delta_{S'}(h'; u'_1,\ldots,u'_n;Z').
$$

\smallskip

Let $\mathbf{x}'=l(\mathbf{x})$ be a vertex of $\Delta_{S'}(h'; u'_1,\ldots,u'_n;Z')$. Denote
$$
\mathrm{in}_\mathbf{x}h=Z^m + \lambda_1U^{\mathbf{x}}Z^{m-1}+ \cdots + \lambda_mU^{m\mathbf{x}},
\ \lambda_1, \ldots ,\lambda_m \in S/m_S,
$$
with the convention as before that $\lambda_i=0$ if $i\mathbf{x} \not \in \N^n$, $1 \leq i \leq m$.
Applying $l$ (\ref{eq2042}), we get
$$
\mathrm{in}_{\mathbf{x}'}h={Z'}^m + \lambda_1{U'}^{\mathbf{x}'}{Z'}^{m-1}+ \cdots + \lambda_m{U'}^{m\mathbf{x}'}.
$$
Since $S'/m_{S'}=S/m_S$, Definition \ref{defsolvable} then shows that $\mathbf{x}'$ is
solvable if and only if $\mathbf{x}$ is solvable. Since $\Delta_{S}(h; u_1,\ldots,u_n;Z)$ is
minimal, the polyhedron
$\Delta_{S'}(h'; u'_1,\ldots,u'_n;Z')$ is also minimal by Proposition \ref{Deltamin}.
\end{proof}

\subsection{Normal crossings divisors.}

We now introduce a normal crossings divisor $E \subseteq \mathrm{Spec}S$. This section fixes
the terminology and notations for blowing ups and base changes with respect to $E$, then
introduces the Hironaka $\epsilon$ function on ${\cal X}$.
 {Hironaka-permissible centers are classically defined as regular subschemes
along which a given Noetherian scheme is normally flat. Since we are dealing with hypersurface singularities,
the latter condition can be stated in terms of the multiplicity function $m$, {\it viz.} (\ref{eq201}) {\it sqq.}
Introducing a normal crossings divisor $E$ leads to an additional transverseness requirement for the center.
This leads to Definitions \ref{Hironakapermis} and \ref{HironakapermisE} below.}

\begin{defn}\label{defadapted}
A r.s.p. $(u_1,\ldots ,u_n)$ of $S$ is said to be adapted to $E$
if $E=\mathrm{div}(u_1\cdots u_e)$ for some $e$, $0 \leq e \leq n$.
\end{defn}

We emphasize that we allow $e=0$, i.e. $E=\emptyset$ in this definition.

\smallskip

In this context, we use the following notion of Hironaka permissible center:

\begin{defn}\label{Hironakapermis}
Let ${\cal Y} \subset {\cal X}$ be an integral closed subscheme with generic point $y$.
We say that ${\cal Y}$ is  {Hironaka-permissible} at $x \in {\cal Y}$ if
$$m(y)=m(x) \ \mathrm{and} \  {\cal Y} \ \mathrm{is} \ \mathrm{regular} \ \mathrm{at} \  x.$$
\end{defn}

 {\begin{defn}\label{HironakapermisE}
Let ${\cal Y} \subset {\cal X}$ be an integral closed subscheme with generic point $y$.
We say that ${\cal Y}$ is  {Hironaka-permissible with respect to $E$} \index{Hironaka-permissible with respect to $E$, Definition~\ref{HironakapermisE}}at $x \in {\cal Y}$ if
${\cal Y}\subseteq \mathrm{Sing}_m{\cal X}$, i.e. $m(y)=m(x)=m$, and
$$W:=\eta ({\cal Y})  \ \mathrm{has} \ \mathrm{normal} \ \mathrm{crossings} \ \mathrm{with} \  E \
\mathrm{at} \  s:=\eta (x).$$
\end{defn}}

We remind the reader that an integral closed subscheme $W \subseteq \mathrm{Spec}S$ has
normal crossings with $E=\mathrm{div}(u_1\cdots u_e)$ if the family $(u_1,\ldots ,u_e)$ can be extended
to a r.s.p. $(u_1,\ldots ,u_n)$ of $S$ such that the ideal $I(W)$ of $W$ is
of the form $I_J=(\{u_j\}_{j \in J})\subseteq S$, for some $J \subseteq \{1,\ldots ,n\}$.

Note that a Hironaka-permissible center w.r.t. any $E$ (e.g. $E=\emptyset$) is Hironaka-permissible:
we have $m(y)=m(x)=m$ and $y \in \eta^{-1}(w)\cap \mathrm{Sing}_m{\cal X}$, where $w$ is the generic point of $W$;
by Proposition \ref{deltainv} applied to $S_w$, the map
${\cal Y} \rightarrow W$ is birational, hence an isomorphism since $W$ is regular.

Since the notion is local on ${\cal X}$, a Hironaka-permissible blowing up (w.r.t. $E$) is
simply the blowing up along a center ${\cal Y}\subset {\cal X}$ which is
Hironaka-permissible (w.r.t. $E$) at each point of its support.
By a {\it local} Hironaka-permissible blowing up, we simply
mean the localization at some point of the exceptional divisor $\pi^{-1}({\cal Y})$ of the blowing up
$\pi$ along a Hironaka-permissible center. The important fact is that Hironaka-permissible blowing
ups w.r.t. $E$ preserve our structure:

\begin{prop}\label{Hironakastable}
Let $S$, $h\in S[X]$ (\ref{eq201}), ${\cal X}$ and $E=\mathrm{div}(u_1\cdots u_e)$ be as above. Let
$\pi : {\cal X}'\rightarrow {\cal X}$ be a Hironaka-permissible blowing up w.r.t. $E$ at $x \in {\cal X}$.
There exists a commutative diagram
\begin{equation}\label{eq210}
\begin{array}{ccc}
  {\cal X}  & {\buildrel \pi  \over \longleftarrow} & {\cal X}' \\
  \downarrow &   & \downarrow \\
  \mathrm{Spec}S &  {\buildrel \sigma \over \longleftarrow} &  {\cal S}' \\
\end{array}
\end{equation}
where $\sigma: {\cal S}' \rightarrow \mathrm{Spec}S$ is the blowing up along $W$.

\smallskip

For every $s' \in \sigma^{-1}(s)$, $S':={\cal O}_{{\cal S}' ,s'}$, there exists $h'\in S'[X']$  {monic} of degree $m$
such that ${\cal X}'_{s'}=\mathrm{Spec}(S'[X']/(h'))$.

\smallskip

Furthermore, there exists a r.s.p. $(u'_1,\ldots ,u'_n)$ of $S'$ adapted to
the stalk $E'_{s'}$, $E':=\sigma^{-1}(E\cup W)_{\mathrm{red}}$.
\end{prop}

\begin{proof}
By the above remarks, there exists $J \subseteq \{1,\ldots ,n\}$ such that
$I(W)=I_J=(\{u_j\}_{j \in J})$. By Proposition \ref{Deltaalg}, there exists $Z:=X- \theta $, $\theta \in S$,
such that $\Delta_{S} (h;u_1,\ldots ,u_n;Z)$ is minimal. Since $x, y \in \mathrm{Sing}_m{\cal X}$, we have
$$
\eta^{-1}(s)=\{x\}, \ \eta^{-1}(W)={\cal Y} \ \mathrm{and} \ \delta (x)\geq 1, \ \delta (y)\geq 1
$$
by Proposition \ref{deltainv}. In particular, the ideal of ${\cal Y}$ at $x$ is
$$
I({\cal Y})=(Z, \{u_j\}_{j \in J}).
$$
Since $\delta (y)\geq 1$, the point at infinity $(1:0: \cdots :0)$ does not belong to ${\cal X}'$ so
$(\{u_j\}_{j \in J}){\cal O}_{{\cal X}'} $ is invertible. By
the universal property of blowing up, there is a commutative diagram (\ref{eq210}).

\smallskip

Let $s' \in \sigma^{-1}(s)$ and $j_0 \in J$ be such that $u_{j_0}$ is a local equation of $\pi_0^{-1}(W)$.
We take $X':= Z/u_{j_0}$ and
\begin{equation}\label{eq212}
h':=u_{j_0}^{-m}h(Z)={X'}^m +u_{j_0}^{-1}f_{1,Z}{X'}^{m-1}+ \cdots +u_{j_0}^{-m}f_{m,Z}.
\end{equation}
Note that $h'\in S'[X']$ follows from the last statement in Proposition \ref{Deltaalg}. The last statement is obvious
because $E'=\sigma^{-1}(E \cup W)_{\mathrm{red}}$ is a normal crossings divisor on ${\cal S}'$.
\end{proof}

We will stick to these notations when local Hironaka-permissible blowing ups are concerned,
or compositions of such local blowing ups. We always refer to the reduced total
transform of $E$ on the blown up base $\mathrm{Spec}S$.

Suppose a base change is given as considered
in the previous section, i.e. formal completion $S \subseteq \hat{S}$, localization at a prime $S\subseteq S_s$ or
regular local base change $S\subseteq \tilde{S}$, $\tilde{S}$ excellent.

\begin{nota}\label{notaprime}
Given $S\subseteq S'$ such a base change, we denote
$$
E':=E\times_S\mathrm{Spec}S', \ \eta ': \ {\cal X}'={\cal X}\times_S \mathrm{Spec}S' \rightarrow \mathrm{Spec}S'.
$$
The image of $h$ in $S'[X]$ is denoted $h' \in S'[X]$. This notation is used consistently with Notation \ref{notageomreg1}.
\end{nota}

For instance if  $s \in \mathrm{Spec}S$, there exists a r.s.p. $(v_1,\ldots ,v_{n(s)})$ of $S_s$
which is adapted to $E_s$, where $E_s$ is the stalk of $E$ at $s$.
We then have $E_s=\mathrm{div}(v_1\cdots v_{e(s)})$ and may choose
$v_j=u_{\varphi (j)}$ for some injective map $\varphi : \ \{1, \ldots ,e(s)\} \rightarrow  \{1, \ldots ,e\}$.
It is of course not possible in general to extend a given $(v_1,\ldots ,v_{n(s)})$ to a r.s.p.
$(u_1,\ldots ,u_n)$ of $S$. We let $h_s\in S_s[X]$ be the image of $h$.

\begin{defn}\label{defwelladapted}
Let $s \in \mathrm{Spec}S$ and $(v_1,\ldots ,v_{n(s)})$ be   {a} r.s.p. of $S_s$ which is
adapted to $E_s$, $E_s=\mathrm{div}(v_1\cdots v_{e(s)})$.
We say that coordinates
$$
(v_1,\ldots ,v_{n(s)};Z_s),  \ Z_s:=X-\phi_s, \ \phi_s\in S_s,
$$
are well adapted at $y \in \eta^{-1}(s)$ if $\Delta_{S_s} (h;v_1,\ldots ,v_{n(s)};Z_s)$ is minimal.
\end{defn}

 {As remarked after Definition \ref{defdelta}, the invariant $\epsilon (y)$ introduced below
is classically used  in Resolution of Singularities. Although the situation is subtle in positive residue characteristic,
a general purpose is performing Hironaka-permissible blowing ups to get smaller values of the function $\epsilon$ at
singular points.}

\begin{defn}\label{defepsilon}
Let $(u_1,\ldots ,u_n)$ be a r.s.p. of $S$ which is adapted to $E$.
Let $j$, $1 \leq j \leq e$, and let ${\cal Y}_j\subset {\cal X}$ be an
irreducible component of $\eta^{-1} (\mathrm{div}(u_j))$ with
generic point $y_j\in {\cal X}$. We let
$$
d_j:=\delta (y_j)\in {1 \over m!}\N .
$$ \index{$d_j$, Definition~\ref{defepsilon}}

For any $s \in \mathrm{Spec}S$ and $y \in \eta^{-1}(s)$, we let
$$
\epsilon(y):=m \left (\delta (y)- \sum_{\mathrm{div}(u_j)\subseteq
E_s}{d_j} \right )\in {1 \over (m-1)!}\Z .
$$  \index{$\epsilon(y)$, Definition~\ref{defepsilon}}
\end{defn}

Summing up results from the previous section, we have:

\begin{prop}\label{epsiloninv}
Let $(u_1, \ldots ,u_n;Z)$ be well adapted coordinates at $x \in \eta^{-1}(m_S)$. With notations as above, we have
$$
d_j=\min \left \{{a_j \over i},
\ \mathbf{a} \in \mathbf{S}^{\{1, \ldots ,n\}}(f_{i,Z}), \ 1 \leq i \leq m \right \}, \ 1 \leq j \leq e.
$$

For $s\in \mathrm{Spec}S$ and $y \in \eta^{-1}(s)$, we have $\epsilon (y)\geq 0$.
\end{prop}

\begin{proof}
The first (resp. second) statement follows from the last one in Proposition \ref{Deltaalg}
applied to $S$ and $J:=\{j\}$ (resp. to $S_s$ and each $J:=\{j\}$ with $\mathrm{div}(u_j)\subseteq E_s$).
\end{proof}

\subsection{The Galois or purely inseparable assumption.}\label{subsection:G}

In this section, we introduce the assumptions of Theorem \ref{luthm}:   {the polynomial $h$ is either purely inseparable (char$(S)=p>0$) or Galois  (char$(S)=p>0$ or  char$(S)=0$).}
 {This is phrased as condition {\bf (G)} below. This condition  {\bf (G)} plays an important role in this article for two main reasons:}

  {Firstly, {\bf (G)} is stable under Hironaka permissible blowing ups (Definitions \ref{Hironakapermis} and \ref{HironakapermisE}, Proposition \ref{SingX} below).}

   {Secondly, the initial form polynomials  $\mathrm{in}_\alpha h$ from Definition \ref{definh} satisfy again  {\bf (G)} (Proposition \ref{deltaint}(a)). This implies that   $\mathrm{in}_\alpha h$ is either an Artin-Schreier polynomial or a purely inseparable polynomial (Theorem \ref{initform}).}

  \medskip

  {We recall the notations:} 
  { $$ \begin{array}{c}
   h:=X^p+f_1X^{p-1}+ \cdots +f_p \in S[X], \   {\cal X} :=\mathrm{Spec}(S[X]/(h)), \\
  \\
 K:=QF(S)  \ \mathrm{and}\ L:=\mathrm{Tot}(S[X]/(h)). \hfill{}\\
  \end{array}
 $$}

  \medskip


From now on, we assume furthermore that the following property holds:\\

\noindent {\bf (G)}  $m=p$ is a prime number, $h$ is reduced, the ring extension $L|K$ is normal and
${\cal X}$ is $G$-invariant, where $G:=\mathrm{Aut}_K(L)$. \index{{\bf (G)}, condition~{\bf (G)} or assumption~{\bf (G)}}\\

{\it Assumption {\bf (G)} is maintained up to the end of this chapter.}\\

Since $[L:K]=p$ is a prime number, we have either $G=\Z /p$ ($L|K$ separable, cases (a) and (b) below)
or $G= (1)$ ($L|K$ inseparable, case (c) below). Case (a) is included here for the sake of completeness
and because residue actions in case (b) may lead to case (a). The three cases to be considered are:
\begin{itemize}
  \item [(a)] $h$ is totally split (product of $p$ pairwise distinct linear factors) over $K$;
  \item [(b)] $h$ is irreducible and Galois over $K$ with group $G=\Z /p$;
  \item [(c)] $h$ is irreducible, $\mathrm{char}S=p$, $f_{i,X}= 0$, $1 \leq i \leq p-1$.
\end{itemize}

Assumption {\bf (G)} is also preserved by those base changes considered in the previous sections, i.e.
formal completion $S \subseteq \hat{S}$, localization at a prime $S\subseteq S_s$
or regular local base change $S\subseteq \tilde{S}$, $\tilde{S}$ excellent. Note that in any case, $h$ reduced
implies respectively $h_s$, $\hat{h}$ (since $S$ is excellent) and $\tilde{h}$ reduced  (Notation \ref{notaprime}).
Recall notations and definitions of initial forms from Definition \ref{definh}.

\begin{propdef}\label{izero}
Assume that $\mathrm{char}S/m_S=p$. Let $(u_1,\ldots,u_n)$ be a given r.s.p. of $S$ and
$\alpha \in \R^n_{>0}$ be a weight vector. The integer
$$
i_0 (\alpha ):=\mathrm{min}\{i\in \{1,\ldots p\} : F_{i,Z,\alpha} \neq 0\}
$$ \index{$i_0 (\alpha ) $, where $\alpha \in \R^n_{>0}$ is a weight vector}
does not depend on  $Z=X-\theta$, $\theta \in \hat{S}$ such that
$\Delta_{\hat{S}} (h;u_1,\ldots ,u_n;Z)$ is minimal. If  $i_0 (\alpha) <p$, the
form $F_{i_0(\alpha),Z,\alpha}$ is also independent of the choice of $Z=X-\theta$ as above.

In case $\alpha =\mathbf{1}$, the integer $i_0 (\mathbf{1})$ (also denoted by $i_0(x)$ for $x \in \eta^{-1}(m_S)$)
and   {the forms} $F_{i_0(\mathbf{1}),Z}=F_{i_0(\mathbf{1}),Z,\mathbf{1}}$ (if $i_0(\mathbf{1})<p$) are also independent
of the choice of the r.s.p. $(u_1,\ldots,u_n)$ of $S$ and $Z=X-\theta$, $\theta \in \hat{S}$
such that $\Delta_{\hat{S}} (h;u_1,\ldots ,u_n;Z)$ is minimal.
\end{propdef}

\begin{proof}
Take $Z'= Z-\phi$ such that both polyhedra  $\Delta_{\hat{S}} (h;u_1,\ldots ,u_n;Z)$ and
$\Delta_{\hat{S}} (h;u_1,\ldots ,u_n;Z')$ are minimal. By minimality, we have
$$
\mu_\alpha (\phi)\geq a:=\delta_\alpha (h;u_1,\ldots ,u_n;Z).
$$
The initial forms $\mathrm{in}_{\alpha}h (Z)\in (\mathrm{gr}_\alpha S)[Z]$
and $\mathrm{in}_{\alpha}h(Z') \in (\mathrm{gr}_\alpha S)[Z']$ are related by
$$
\mathrm{in}_{\alpha}h(Z')=\mathrm{in}_{\alpha}h(Z - \mathrm{cl}_{\alpha , a}\phi).
$$

The first statement follows from the elementary fact that
$\mu_\alpha \left(
\begin{array}{c}
  p \\
  i \\
\end{array}
\right) >0$ for $1\leq i \leq p-1$, since $p \in m_S$. The second statement then
follows from Proposition \ref{deltainv}.
\end{proof}

\begin{prop}\label{SingX}
 {Let  $x \in \mathrm{Sing}{\cal X}$, $s:=\eta (x)$. Then we have:}
\begin{equation}\label{eq211}
\eta^{-1}(s)=\{x\}, \ k(x)=k(s) \ \mathrm{and} \ \delta (x)>0.
\end{equation}

 {Assume furthermore that a normal crossings divisor $E=\mathrm{div}(u_1 \cdots u_e)\subset \mathrm{Spec}S$ is specified
and let $\pi : {\cal X}'\rightarrow {\cal X}$ be a Hironaka-permissible blowing up w.r.t. $E$ at $x$.}
Then, with notations as in Proposition \ref{Hironakastable}, for every $s' \in \sigma^{-1}(s)$,
${\cal X}'_{s'}$ satisfies again {\bf (G)}.
\end{prop}

\begin{proof}
It can be assumed that $s=m_S$. Let $(u_1,\ldots ,u_n;Z)$ be well adapted
coordinates at $x$ and $\overline{h}(Z)\in S/m_S [Z]$ be the
reduction of $h$ modulo $m_S$. By {\bf (G)}, $G$ acts transitively on the
fiber $\eta^{-1}(s)$. Then $\overline{h}(Z)$ is either a $p^{th}$-power
or satisfies again {\bf (G)} w.r.t. the zero-dimensional regular local ring $S/m_S$.

If $\overline{h}(Z)$ satisfies {\bf (G)}, then $(h(Z),u_1, \ldots , u_n)$ is a r.s.p. of the local ring
$S[Z]_{m_x}$, so $x$ is a regular point of ${\cal X}$.

Assume now that $\overline{h}(Z)=(Z-\overline{\lambda})^p$ for some $\overline{\lambda}\in S/m_S$.
Now $(0, \ldots ,0)$ is a solvable vertex of $\Delta_{S}(h;u_1,\ldots,u_n;Z)$ unless
$\overline{\lambda}= 0$. Since $(u_1,\ldots ,u_n;Z)$ are well adapted coordinates at $x$, we have
$\overline{\lambda}= 0$.

 {To prove the last statement, let
us first note that $x$ is $G$-invariant by (\ref{eq211}). Let ${\cal Y}\subset {\cal X}$ be
Hironaka-permissible w.r.t. $E$ and $y$ be its generic point. Applying again (\ref{eq211}), $y$  is also $G$-invariant (i.e.  $g(\mathrm{I}({\cal Y}))=\mathrm{I}({\cal Y})\subset \mathcal{O}_{\cal X}$ for every $g\in G$): with notations as in Proposition \ref{Hironakastable},  the blow up $ {\cal X}' $ of ${\cal X}$ along ${\cal Y} $ is then $G$-invariant.}
\end{proof}

\medskip
 {The following proposition prepares the proof of Theorem~\ref{initform}. Statement (i) is the main ingredient to get this structure theorem about  the polynomial $\mathrm{in}_{\alpha} h$.}

\medskip

\begin{prop}\label{deltaint}
Let $x \in \eta^{-1}(m_S)$ and $(u_1,\ldots ,u_n;Z)$ be well adapted
coordinates at $x$. For $\alpha \in \R^n_{>0}$ a weight vector, the following   {statements} hold:
\begin{itemize}
    \item [(i)] the polynomial $\mathrm{in}_{\alpha} h \in (\mathrm{gr}_\alpha S)[Z]$ satisfies again {\bf (G)} w.r.t. the local ring $(\mathrm{gr}_\alpha S)_{(U_1, \ldots ,U_n)}$;
    \item [(ii)] if ($\mathrm{char}S/m_S=p$ and $i_0 (\alpha)<p$), then
$$
\delta_\alpha (h;u_1,\ldots,u_n;Z) \in \Gamma_\alpha =\Z\alpha_1 + \cdots + \Z\alpha_n ;
$$
    \item [(iii)] if $\mathrm{char}S/m_S=0$ or if ($\mathrm{char}S/m_S=p$ and $i_0 (\alpha)=p$),
    then
    $$
    \delta_\alpha (h;u_1,\ldots,u_n;Z) \in {1 \over p}\Gamma_\alpha.
    $$
\end{itemize}
\end{prop}

\begin{proof}
If $\delta (x)=0$, we have
$\delta_\alpha (h;u_1,\ldots,u_n;Z)=0$ and  $\mathrm{in}_{\alpha} h =\overline{h}(Z)$
with notations as in the previous proof, so the proposition is trivial. Assume that $\delta (x)>0$.

\smallskip

By Proposition \ref{Deltamin}, we have $\Delta_{S}(h;u_1,\ldots,u_n;Z)
\neq \emptyset$ and this polyhedron has no solvable vertex.
Therefore $\mathrm{in}_{\alpha} h$ is not a $p^{th}$-power.
Let $z \in L$ be the image of $Z$ and $\nu_\alpha$ be any extension of $\mu_\alpha$ to $L$. Then
$\nu_\alpha$ is centered at $x$, since ${\cal X}$ is $G$-invariant and $\eta^{-1}(m_S)=\{x\}$ by
Proposition \ref{deltainv}(i). We have:
\begin{equation}\label{eq230}
    \nu_\alpha(z)=\mu_\alpha(f_{i,Z})/i=\delta_\alpha (h;u_1,\ldots,u_n;Z) \in \Gamma_\alpha \otimes_{\Z}\Q
\end{equation}
for each $i$, $1 \leq i \leq p$ such that $F_{i,Z,\alpha}\neq 0$. Since $L|K$ is normal of degree $p$,
the reduced ramification index $e_0$ of $\nu_\alpha | \mu_\alpha$ is $e_0 =1$ or $e_0=p$.

\smallskip

Assume that ($\mathrm{char}S/m_S=p$ and $i_0 (\alpha )=p$). Then $\mathrm{in}_{\alpha} h$ is in case (c)
of {\bf (G)} and we get (iii) from (\ref{eq230}).

Assume that $\mathrm{char}S/m_S=0$ or ($\mathrm{char}S/m_S=p$ and $i_0 (\alpha)<p$). Then $h$ is in case (a) or (b).
Since $G=\Z /p$ in these cases and ${\cal X}$ is $G$-invariant,
$G$ acts transitively on the roots of $\mathrm{in}_{\alpha} h$. We have:
$$
    \left\{
\begin{array}{cccc}
  \mathrm{Tot}((\mathrm{gr}_\alpha S)[Z]/(\mathrm{in}_{\alpha} h)) & =
  & \prod_{\nu_\alpha}QF(\mathrm{gr}_\alpha S) & \mathrm{if} \ \mu_\alpha \ \mathrm{splits}; \\
   & & & \\
  QF((\mathrm{gr}_\alpha S)[Z]/(\mathrm{in}_{\alpha} h))  & =
  & QF(\mathrm{gr}_{\nu_\alpha} S)         & \mathrm{otherwise}, \hfill{}\\
\end{array}
\right .
$$
and this proves (i). Statement (iii) follows from (\ref{eq230}) if $\mathrm{char}S/m_S=0$.

Assume finally that ($\mathrm{char}S/m_S=p$ and $i_0 (\alpha)<p$). By (\ref{eq230}), we have
$$
p\nu_\alpha(z)=p\mu_\alpha(f_{i_0 (\alpha),Z})/i_0 (\alpha ) \in \Gamma_\alpha .
$$
Since $\Gamma_\alpha \simeq \Z^r$ for some $r \geq 1$, this implies
$$
\delta_\alpha (h;u_1,\ldots,u_n;Z)=\mu_\alpha(f_{i,Z})/i_0 (\alpha ) \in \Gamma_\alpha
$$
which completes the proof of (ii).
\end{proof}

\begin{cor}\label{cordeltaint}
Assume that a normal crossings divisor
$$
E=\mathrm{div}(u_1 \cdots u_e)\subset \mathrm{Spec}S
$$
is specified. We have $pd_j\in \N$, $1 \leq j \leq e$,  and $\epsilon (y)\in \N$ for every $y \in {\cal X}$.
\end{cor}

\begin{proof}
In view of Definition \ref{defepsilon} and Proposition \ref{epsiloninv},
this follows from Proposition \ref{deltaint} (ii)(iii) applied to the
local rings $S_{(u_j)}$ and $S_s$, $s:=\eta (y)$.
\end{proof}

This corollary allows us to define the following invariant:

\begin{defn}\label{defH}
Let $(u_1,\ldots ,u_n)$ be a r.s.p. of $S$ which is adapted to the normal crossings divisor $E=\mathrm{div}(u_1
\cdots u_e)$. For $y \in {\cal X}$, $s:=\eta (y)$, we define a principal ideal:
$$
H(y):=\left (\prod_{\mathrm{div}(u_j)\subseteq
E_s}{u_j^{H_j}}\right ) \subseteq S,
$$
where $H_j:=pd_j\in \N$. \index{$H(y)$}
\end{defn}

\subsection{The discriminant assumption.}

 {Discriminant theory has been used since Jung \cite{J} in order to simplify singularities. Namely the fundamental group $\pi_1(\C^n \setminus \{x_1\cdots x_e=0\})\simeq \Z^e$ classifies unramified coverings away from the normal crossing divisor $D:=\{x_1\cdots x_e=0\}$. Jung's observation that any such covering can be described by a monomial mapping $\C^n\rightarrow \C^n$  allowed  to resolve the singularities of surfaces $n=2$ \cite{J}\cite{W}\cite{L2}.}

\smallskip

 {This method extends to positive characteristics provided no wild ramification occurs and it is the content of Abhyankar's Lemma \cite{SGA1} Appendice I. Even when wild ramification occurs, this method induces some simplification from the general case and is the starting point of several approaches \cite{BrV}\cite{T2}.}

 {In any characteristic, getting Jung's situation of a discriminant with normal crossings in dimension~$n$ is a consequence of embedded resolution for the discriminant subscheme which has dimension~$n-1$. For $n=3$, this is possible as   embedded resolution of surfaces is known \cite{CoJS}. In our problem, this reduction  is stated  as corollary~\ref{EEfait} below.}

 {The main result in this section is Theorem~\ref{initform} below which plays an important role in the proof of Theorem~\ref{luthm}. Indeed, Theorem~\ref{initform}    basically  reduces the proof of   Theorem~\ref{luthm} to computations on purely inseparable or Artin-Schreier polynomials of degree $p$ over fields of  characteristic $p>0$.}

\medskip

We now introduce the critical locus of the map $\eta: \ {\cal X}
\rightarrow \mathrm{Spec}S$ together with its scheme structure
given by the discriminant $D:=\mathrm{Disc}_X h \in S$.
We are interested in the case where $D$ is a normal crossings divisor. \\

Note that $D$ is by definition independent of the choice of
regular parameters of $S$ and invariant by those translations
$X':=X-\phi$, $\phi \in \hat{S}$ used in minimizing polyhedra. If
$(S,h,E)$ is in case (c) of {\bf (G)}, then $D=0$.

\begin{defn}\label{conditionE}
Let $S$, $h\in S[X]$ (\ref{eq201}), ${\cal X}$ and $E=\mathrm{div}(u_1\cdots u_e)$ be specified.
We say that $(S,h,E)$ satisfies assumption {\bf (E)} \index{{\bf (E)}, condition {\bf (E)} or assumption {\bf (E)}, Definition~\ref{conditionE}} if $\mathrm{char}(S/m_S)=p > 0$ and
one of the following properties hold:
\begin{equation}\label{eq221}
    \left\{
\begin{array}{cccc}
  (i)   & D=0       & \mathrm{and} & \eta (\mathrm{Sing}_p{\cal X}) \subseteq E, \hfill{}\\
        &           &              &    \\
  (ii)  & D \neq 0  & \mathrm{and} & \mathrm{div}(D)_\mathrm{red}\subseteq E\subseteq \mathrm{div}(p)_\mathrm{red}.     \\
\end{array}
\right .
\end{equation}
\end{defn}

{\it Assumption {\bf (E)} is maintained up to the end of this chapter.}\\

This assumption implies that $\mathrm{Sing}_p{\cal X} \subseteq \eta^{-1}(E)\subset {\cal X}$:
(i) by definition; (ii) because $\eta^{-1}(\mathrm{Spec}S \backslash E)$ is regular since
$\mathrm{Spec}S \backslash E$ is. In particular $E\neq \emptyset$ if $\mathrm{Sing}_p{\cal X}\neq \emptyset$.

\smallskip

\begin{exam}

 {Let us illustrate cases (i)(ii) by examples. As $h$ is reduced,  case (i) of Definition \ref{conditionE} cannot occur when char$(S)=0$. When char$(S)=p>0$, the following example fits into condition~(i):
$$h=Z^p+u_1^au_2^bf,\ f\in S=k[[u_1,u_2,v]],\ a+b\geq p,\ \mathrm{char}(k)=p>0,$$
with V$(f)\subset$Spec$(S)$  regular outside $E=$div$(u_1u_2)$.}

\smallskip
 {The following is an example of condition~(ii) with  char$(S)=0$. Let
$$A:={\Z_p[\pi] \over \pi^{n(p-1)}},\ n\in \N-\{0\},\ S:=A[[u_2,u_3]],\ { E}=\mathrm{div}(\pi).$$}

 {\noindent Let $\mu_p$ be the group of $p$-th roots of unity. Note that $\mu_p\subset \Z_p[\pi^n]\subset A$.}

 {\noindent Let $h:=X^p-\pi^{ap}(1+f)$, $f\in m_S$, $a\in \N-\{0\}$.}

 {Note that $(S,h,E)$ satisfies assumption \textbf{(G)} (section \ref{subsection:G})  since $\mu_p$ acts on $S[X]/(h)$ by $x\mapsto \zeta x$.   We have:
$$\mathrm{Disc}_X(h)=\prod_{\xi, h(\xi)=0}h'(\xi)=p^{p-1}(\prod_{\xi}\xi)^{p-1}=p^{p-1}\pi^{ap(p-1)}(1+f)^{p-1}.$$ Therefore assumption \textbf{(E)} is satisfied. Note that the coordinates $(\pi,u_2,u_3;X)$ are adapted  but not well adapted (Definition \ref{defwelladapted}).  To minimize the  polyhedron $\Delta(h; \pi,u_2,u_3;X)$, we first make the translation: $Z:=X-\pi^a$. This leads to:
$$  h(Z)=Z^p+\sum_{1\leq i \leq p-1}  \begin{pmatrix}
  p\\
  i \\
\end{pmatrix} \pi^{ai}Z^{p-i}-   \pi^{ap}f.$$
The monomial $p \pi^{a(p-1)} Z$ leads to the vertex $(a+n,0,0)$ whenever
$$\mathrm{max}_{\lambda \in A} \{\mathrm{ ord}_{\pi}(f(0,0)-\lambda^p)\}\geq n.$$
Other vertices depend on the expansion of $f$.}

\end{exam}

Assumption  {\bf (E)} is also preserved by those base changes considered in the previous section:
formal completion $S \subseteq \hat{S}$, localization at a prime $S\subseteq S_s$ or
regular local base change $S\subseteq \tilde{S}$, $\tilde{S}$ excellent. For Hironaka-permissible
blowing ups, we have:

\begin{prop}\label{Estable}
Let $\pi : {\cal X}'\rightarrow {\cal X}$ be a Hironaka-permissible blowing up w.r.t. $E$ at $x \in {\cal X}$.
Then, with notations as in Proposition \ref{Hironakastable}, for every $s' \in \sigma^{-1}(s)$,
$(S',h',E')$ satisfies again {\bf (E)}.
\end{prop}

\begin{proof}
Any Hironaka-permissible center ${\cal Y} \subset {\cal X}$ w.r.t. $E$ at $x$
is contained in $E$ by the above remarks. Therefore the proposition is obvious in case (i) of
Definition \ref{conditionE}.

Let $(u_1, \ldots ,u_n;Z)$ be well adapted coordinates at $x$
and $h(Z)\in S[Z]$ be the corresponding expansion.
With notations as in Proposition \ref{Hironakastable} and (\ref{eq212}), we have
$h'(X')=u_{j_0}^{-p}h(X'u_{j_0})$ for some $u_{j_0} \in I(W)$. We deduce that
$$
D':=\mathrm{Disc}_{X'}h'=u_{j_0}^{-p(p-1)}\mathrm{Disc}_{Z}h=u_{j_0}^{-p(p-1)}D,
$$
hence $\mathrm{div}(D')_\mathrm{red}\subseteq E'\subseteq \mathrm{div}(p)_\mathrm{red}$ as required.
\end{proof}

 {\begin{rem}\label{ThmName}
We call the next Theorem -Reduction to characteristic $p>0$- to emphasize the fact that once all the statements and proofs are phrased purely in terms of initial forms with respect to certain faces of the Newton polyhedron, for the computations of the invariants after a blowing up, there is no difference between the equal and the mixed characteristic cases and they  will be treated uniformly. This allows us to adapt the techniques developed in
 \cite{CoP1} \cite{CoP2}. Cases (1) and (2) of Theorem~\ref{initform} are called respectively \textit{purely inseparable} \index{purely inseparable, Remark~\ref{ThmName}}and \textit{Artin-Schreier}.\index{Artin-Schreier, Remark~\ref{ThmName}}
\end{rem}}

\begin{thm}\label{initform}\textbf{(Reduction to characteristic $p$).}
 {Assume  that $(S,h,E)$ satisfies assumptions {\bf (G)} and {\bf (E)}.}
Let $x \in \eta^{-1}(m_S)$ be such that $\epsilon (x)>0$.
Then $({\cal X},x)$ is analytically irreducible.

\smallskip

Let $(u_1,\ldots ,u_n;Z)$ be well adapted coordinates at $x$ and
$\alpha \in \R^n_{>0}$ be a weight vector. Exactly one of the following properties holds.
\begin{itemize}
    \item [(1)] $i_0 (\alpha)=p$, i.e. $\mathrm{in}_{\alpha}h =Z^p +F_{p,Z, \alpha}$;
    \item [(2)] $i_0 (\alpha)=p-1$ i.e. $\mathrm{in}_{\alpha}h =Z^p +F_{p-1,Z, \alpha}Z +F_{p,Z, \alpha}$,
    $F_{p-1,Z, \alpha}\neq 0$. Furthermore, we have
    \begin{equation}\label{eq2313}
        -f_{p-1,Z}=\gamma_{p-1,Z}\prod_{j=1}^eu_j^{A_{p-1,j}}
    \end{equation}
    with $A_{p-1,j}\in (p-1)\N$, $1 \leq j \leq e$, and $\gamma_{p-1,Z}\in S$ a unit with residue
    $\overline{\gamma}_{p-1,Z} \in (S/m_S)^{p-1}$. In particular, $-F_{p-1,Z, \alpha}=G^{p-1}$ for some
    nonzero $ G \in \mathrm{gr}_\alpha S$, and we have
$$
\mathrm{cl}_{p(p-1)\delta_\alpha (h;u_1,\ldots ,u_n;Z)}(\mathrm{Disc}_Z(h))=<F_{p-1,Z, \alpha}^p>.
$$
\end{itemize}
\end{thm}

\begin{proof}
 {We start with some comments about discriminants. Let
$$P:=Z^d+a_1Z^{d-1}+\cdots+a_d\in S[a_1,\cdots,a_d][Z],\ d\geq 1,$$
 be the generic polynomial defined over the domain $S$,  $a_1,\cdots,a_d$ indeterminates.  Let
 $$D_P=\mathrm{Disc}_Z(P):=\prod _{i<j}(\varphi_i-\varphi_j)^2,$$
  be the discriminant of $P$ where $\varphi_1,\cdots,\varphi_d$ are the roots of $P$ in a suitable extension of $S[a_1,\cdots,a_d]$. As a polynomial in the $\varphi_i$, $D_P$ is \textit{homogeneous} of degree $d(d-1)$. By the theorem on symmetric functions, $D_P$ can be expressed as a \textit{homogeneous} polynomial in $a_1,\ldots,a_d$ (the elementary symmetric functions) where $a_i$ has degree $i$.}

 {Suppose that a specialization $ a_i \leadsto \bar{a}_i\in S$ is given, so $P\leadsto \bar{P}$, then
$$ D_{\bar{P}}=\mathrm{Disc}_Z(\bar{P})=D_P(\bar{a}_1,\cdots,\bar{a}_d).$$}
 {We apply this to $\bar{P}=h$. Then, denoting $D:=D_{\bar{P}}$,}
 we have
$$
    \mu_\alpha (D)\geq  p(p-1)\delta_\alpha (h;u_1,\ldots ,u_n;Z),
$$
since $\mu_\alpha (f_{i,Z})/i \geq \delta_\alpha (h;u_1,\ldots ,u_n;Z)$ for $1 \leq i \leq p$. We deduce the formula
\begin{equation}\label{eq231}
    \mathrm{cl}_{\alpha , p(p-1)\delta_\alpha (h;u_1,\ldots ,u_n;Z)}D=\mathrm{Disc}_Z(\mathrm{in}_\alpha h).
\end{equation}
On the other hand, $\mathrm{in}_\alpha h $ has a multiple root
over an algebraic closure of $QF(\mathrm{gr}_\alpha S)$ if and only
if $i_0(\alpha)=p$ by Proposition \ref{deltaint} (i). When this holds,
we are in case (1) of this theorem. \\

Suppose that $h$ is analytically reducible. By Proposition \ref{epsiloninv} and Definition \ref{defdelta},
$\epsilon (x)=\delta (x)- \sum_{i=1}^e d_j$ is determined by $\Delta_{\hat{S}}(h;u_1, \ldots ,u_n;Z)$, thus
invariant by base change $S \subseteq \hat{S}$. Therefore it can be assumed w.l.o.g. that $S=\hat{S}$ in
order to prove the first statement, i.e. that $h$ is in case (a) of property {\bf (G)}.
Since $h$ splits, there is a factorization
$$
h=\prod_{i=1}^{p}(Z -\varphi_j) \in S[Z], \ \varphi_1, \ldots ,\varphi_{p}\in S.
$$
Let $z \in {\cal O}_{\cal X}$ be the image of $Z$ and $g \in G=\Z/p$, $g \neq 0$. By property {\bf (G)},
we have $g(z)\in {\cal O}_{\cal X}$ and $g(z)$ is a root of $h(Z)$. Up to reindexing, it can
therefore be assumed that
$$
g^i (z)=z-\varphi_{i+1}+\varphi_1\in S, \ 1 \leq i \leq p-1.
$$
In particular, we have $g(z)-z =\varphi_1 -\varphi_2 \in S$ and we deduce that
$$
g^i (z) -z =\sum_{k=0}^{i-1} g^k(g(z)-z)=i(g(z)-z), \ 1 \leq i \leq p-1.
$$
Since $(p-1)!$ is a unit in $S$, we get a formula
$$
D=\mathrm{Disc}_Z(h)=\gamma_0(\varphi_1 -\varphi_2)^{p(p-1)}, \ \gamma_0 \in S, \ \gamma_0 \ \mathrm{a} \ \mathrm{unit}.
$$
By assumption, $(u_1, \ldots ,u_n)$ is adapted to $E$. Then Definition \ref{conditionE}(ii) implies that
$$
\varphi_1 -\varphi_2 =\gamma u^\mathbf{a},
$$
with $\gamma \in S$ a unit, and $a_j=0$, $e+1 \leq j \leq n$. Take an expansion (\ref{eq2036}):
$$
\varphi_1 =\sum_{\mathbf{x}\in \mathbf{S}(\varphi_1)}\gamma_\mathbf{x}u^\mathbf{x} ,
\ \gamma_\mathbf{x}\in S, \ \gamma_\mathbf{x} \ \mathrm{unit}
$$
with  $\mathbf{S}(\varphi_1)\subset \N^n$  finite. If $x_j < a_j$ for some
$\mathbf{x} \in \mathbf{S}(\varphi_1)$ and some $j$, $1 \leq j \leq e$, then $\mathbf{x}$ is a vertex
of $\Delta_{S} (h;u_1,\ldots ,u_n;Z)$ with initial form
$$
\mathrm{in}_\mathbf{x} h = (Z - \lambda U^\mathbf{x})^p, \ \lambda \in S/m_S, \ \lambda \neq 0.
$$
This is a solvable vertex: a contradiction, since $\Delta_{S} (h;u_1,\ldots ,u_n;Z)$ is minimal. Therefore
$\varphi_1 \in (u^\mathbf{a})$ and we get $\epsilon (x)=0$: a contradiction. Hence $({\cal X},x)$ is
analytically irreducible as stated. It can be assumed that $h$ is in case (b)
of property {\bf (G)} from now on.\\

Assume now that $\mathrm{in}_\alpha h $ is in cases (a) or (b) of property {\bf (G)}, i.e. $i_0(\alpha)<p$ and
\begin{equation}\label{eq2311}
    \mathrm{Disc}_Z(\mathrm{in}_\alpha h)\neq 0 .
\end{equation}
We now compute $\mathrm{ord}_{(u_j)}D$ for $1 \leq j \leq e$. Let
$$
s_j:=(u_j) \in \mathrm{Spec}S,  \ S_j:=S_{s_j} \ \mathrm{and}  \ y_j \in \eta^{-1}(s_j).
$$
To begin with, $\Delta_{S_j} (h;u_j,Z)$ is minimal by Proposition \ref{Deltaalg}.
We denote by  $G(s_j)=k(s_j)[U_j]$ the graded ring of $S_j$ w.r.t. its
valuation $\mu_j:=\mathrm{ord}_{(u_j)}$ and by $\mathrm{in}_{j}$ the initial form map w.r.t. $\mu_j$. Let:
\begin{equation}\label{eq232}
    \gamma_{i,j}U_j^{A_{i,j}}:=\mathrm{in}_{j}f_{i,Z} \in
    G(s_j), \ 1 \leq i \leq p.
\end{equation}
By Definition \ref{conditionE}(ii), we have $\mathrm{char}S/(u_j)=p$.
Therefore Proposition \ref{izero} and (\ref{eq231}) apply to $S_j$ with $\alpha=1\in \R$.
The corresponding integer $i_0(1)$ is denoted by
$i_0(s_j)$ in order to avoid confusion and we have
\begin{equation}\label{eq2321}
\mu_j(D) \geq p(p-1)\delta (y_j)=(p-1)H_j.
\end{equation}

\noindent {\it Case 1:} $i_0(s_j)<p$. Then equality holds in the former formula as remarked right after
(\ref{eq231}).\\

\noindent {\it Case 2:} $i_0(s_j)=p$. Then inequality is strict in the former formula.
Since $\Delta_{S_j} (h;u_j,Z)$ is minimal, we have $\gamma_{p,j}U_j^{A_{p,j}} \not \in G(s_j)^p$ and $A_{p,j}=H_j$.
Let $z \in L$ be the image of $Z$. The discrete valuation
$\mu_j$ of $K$ has a unique extension to $L$, still denoted by $\mu_j$.
There is an embedding $G(s_j)\subset G_j$, where $G_j$ is the graded ring of the valuation ring
${\cal O}_j:=\{f \in L : \mu_j (f) \geq 0\}$. \\

\noindent {\it Case 2a:} $H_j \in p\N$. We have
\begin{equation}\label{eq233}
    G_j=k(s_j)(\gamma_{p,j}^{1 \over p})[U_j], \ \mathrm{in}_{j}z=
    -\gamma_{p,j}^{1 \over p} U_j^{H_j\over p};
\end{equation}

\noindent {\it Case 2b:} $H_j \not \in p\N$. We have
\begin{equation}\label{eq234}
    G_j=k(s_j)[\gamma_{p,j}^{l_j \over p}U_j^{1 \over p}], \ \mathrm{in}_{j}z=
    -\gamma_{p,j}^{1 \over p} U_j^{H_j\over p},
\end{equation}
where $l_j$ satisfies $l_jH_j \equiv 1 \ \mathrm{mod}p$, since the element
$t:=z^{l_j}u_j^{-{l_jH_j-1 \over p}}$ is a regular parameter of ${\cal O}_j$
with $(\mathrm{in}_{j}t)^p=-\gamma_{p,j}^{l_j}U_j$.\\

Let $g \in G=\mathrm{Gal}(L|K)$ be nontrivial. We have
\begin{equation}\label{eq235}
    g (z)^p-z^p +\sum_{i=1}^{p-1}f_{i,Z} (g (z)^{p-i}-z^{p-i})=0.
\end{equation}
Since $\mu_j(g (z)-z)>\mu_j(z)$ and $\mu_j((p-1)!)=0$, we deduce from (\ref{eq232}) and
(\ref{eq233})-(\ref{eq234}) that
\begin{equation}\label{eq236}
\mathrm{in}_{j}(f_{i,Z} (g
(z)^{p-i}-z^{p-i}))=(-1)^{p-i}iT_j\gamma_{i,j}\gamma_{p,j}^{(p-i-1)/p}U_j^{(p-i-1){H_j
\over p}+A_{i,j}}
\end{equation}
for $1 \leq i \leq p-1$, where $T_j:=\mathrm{in}_{j}(g (z)-z)$. On the other hand, we have
\begin{equation}\label{eq2361}
g (z)^p-z^p= (g (z)-z)^p + \sum_{i=1}^{p-1}
\left(
\begin{array}{c}
  p \\
  i \\
\end{array}
\right)
(g (z)-z)^{p-i}z^i.
\end{equation}
Computing $\mu_j(D)$ by the Hilbert formula \cite{ZS1} V.11.(8) gives
\begin{equation}\label{eq2363}
\mu_j(D)=p(p-1)\mu_j(g (z)-z).
\end{equation}
Since equality is strict in (\ref{eq2321}), we have $\mu_j(H(x)^{-(p-1)}D)>0$ and we deduce
that $\mu_j (g (z)-z)> H_j/p$. Computing initial forms for each term
on the right hand side of (\ref{eq2361}), we get for $1 \leq i \leq p-1$:
$$
\mathrm{in}_{j}((g (z)-z)^{p-i}z^i)=(-1)^{i}T_j^{p-i}\gamma_{p,j}^{i \over p}U_j^{i{H_j \over p}}.
$$
Since $\mu_j (g (z)-z)> H_j/p$ and
$\mu_j(\left(
\begin{array}{c}
  p \\
  i \\
\end{array}
\right)
) =\mu_j(p)$, $1 \leq i \leq p-1$, the unique minimal value term in (\ref{eq2361}) inside the summation
symbol is obtained with $i=p-1$. This shows
\begin{equation}\label{eq2362}
\mathrm{in}_{j}\left ( \sum_{i=1}^{p-1}
\left (
\begin{array}{c}
  p \\
  i \\
\end{array}
\right)
(g (z)-z)^{p-i}z^i \right ) =\mathrm{in}_{j}(p)T_j\gamma_{p,j}^{p-1 \over p}U_j^{(p-1){H_j \over p}}.
\end{equation}

\noindent {\it Case 2a.} By (\ref{eq233}), all terms $\gamma_{p,j}^{(p-i-1)/p}$
for $1\leq i \leq p-1$ appearing in (\ref{eq236}) are
linearly independent over $k(s_j)$. Since $p \in S_j$, $pu_j^{-\mu_j (p)}$ is a unit
in $S_j$. Let $\gamma \in k(s_j)$ be its residue, so the family
$(\gamma \gamma_{p,j}^{p-1 \over p}, \{\gamma_{p,j}^{{p-i-1 \over p}}\}_{1 \leq i \leq p-1} )$
is a {\it basis} of the $k(s_j)$-vector space $k(s_j)(\gamma_{p,j}^{1/p})$. Tracing back to
(\ref{eq235}) an (\ref{eq2361}), the value of $(g (z)-z)^p$ is the value of a
sum of terms with linearly independent initial forms in $G_j$. We deduce the formula
\begin{equation}\label{eq237}
\mu_j(g (z)-z)^{p-1} =\min\{\mu_j (p) + (p-1){H_j \over p},
\min_{1 \leq i \leq p-1}\{(p-i-1){H_j \over p}+A_{i,j}\}\}.
\end{equation}

\noindent {\it Case 2b.} By (\ref{eq234}), all values $(p-i-1)H_j/p$ for $1\leq i \leq p-1$
appearing in (\ref{eq236}) are pairwise distinct modulo $\Z$. Since $p \in S_j$, the family
$$
(\mu_j(p)+(p-1){H_j \over  p},\{(p-i-1){H_j\over p} +A_{i,j}\}_{1 \leq i \leq p-1} )
$$
represent all cosets of $(1/p)\Z $ modulo $\Z$. The argument is now similar to  case 2a above
and (\ref{eq237}) holds as well.
Note that the minimum in the right hand side of (\ref{eq237}) is achieved exactly once in this case 2b.
\\

By (\ref{eq2363}) and (\ref{eq237}), we conclude in all three cases 1, 2a and 2b that
\begin{equation}\label{eq238}
\mu_j(H(x)^{-(p-1)}D)=\min\{p\mu_j (p), \min_{1 \leq i \leq p-1}\{pA_{i,j}-iH_j\}\}.
\end{equation}
By (\ref{eq232}) and definition of $i_0(\alpha)$, we have
\begin{equation}\label{eq239}
\sum_{j=1}^e {A_{i_0(\alpha),j}\alpha_j}\leq \mu_\alpha (f_{i_0(\alpha),Z}) =i_0(\alpha)\delta_\alpha (h;u_1,\ldots ,u_n;Z) .
\end{equation}

Collecting together, since it was assumed in (\ref{eq2311}) that $\mathrm{Disc}_Z (\mathrm{in}_\alpha h) \neq 0$,
we have
$$
\sum_{j=1}^e
{\mu_j(H(x)^{-(p-1)}D)\alpha_j}=(p-1)\left (p\delta_\alpha (h;u_1,\ldots ,u_n;Z)- \sum_{j=1}^e{H_j}\alpha_j \right )
$$
by (\ref{eq231}). By (\ref{eq238})-(\ref{eq239}), we deduce
\begin{equation}\label{eq2391}
(p-1 -i_0(\alpha))(p\delta_\alpha (h;u_1,\ldots ,u_n;Z)- \sum_{j=1}^e{H_j}\alpha_j) \leq 0.
\end{equation}

Suppose that $p\delta_\alpha (h;u_1,\ldots ,u_n;Z)- \sum_{j=1}^e{H_j}\alpha_j=0$. Definition \ref{defH}
implies that $f_{i,Z}^p \in H(x)^i$ for $1 \leq i \leq p$. Definition \ref{defDelta} yields the equality
$$
\Delta_{S} (h;u_1,\ldots ,u_n;Z)=({H_1 \over p}, \ldots , {H_e \over p}, 0 , \ldots , 0)+\R^n_{ \geq 0}.
$$
This is a contradiction, since it is assumed that $\epsilon (x)>0$.

We thus have $p\delta_\alpha (h;u_1,\ldots ,u_n;Z)- \sum_{j=1}^e{H_j}\alpha_j>0$.  By (\ref{eq2391}),
this implies $i_0(\alpha)=p-1$, since $i_0(\alpha)\leq p-1$ was assumed in (\ref{eq2311}).

We may now sharpen (\ref{eq2391}) as follows, since it is an equality: equality holds in (\ref{eq239})
{\it and} the minimum on the right hand side of (\ref{eq238}) is achieved with $i=i_0(\alpha)=p-1$
for each $j$, $1\leq j \leq e$. These two properties are equivalent to the existence of
an expansion (\ref{eq2313}) with $\gamma_{p-1,Z}\in S$ a unit.

By Proposition \ref{deltaint}(i), $G=\Z/p$ acts on the roots of $\mathrm{in}_\alpha h$. Let
$$
z_\alpha \in (\mathrm{gr}_\alpha S)[Z]/(\mathrm{in}_\alpha h)
$$
be the image of $Z$. Then
$(g (z_\alpha )-z_\alpha )^{p-1}+F_{p-1,Z,\alpha}=0$ for $g \in G$ nontrivial, so
the polynomial $X^{p-1} +F_{p-1,Z,\alpha}$ is totally split over $\mathrm{gr}_\alpha S$, i.e.
$-F_{p-1,Z,\alpha}$ is a $(p-1)^{\mathrm{th}}$ in $\mathrm{gr}_\alpha S$ as required.
The last formula in the theorem is obvious.
\end{proof}

\subsection{Adapted differential structure.}

 {The use of differentials in the local study of singularities has a long history. We include a short summary
of this exciting topic where the case of local rings of positive characteristic is put forward.
The Jacobian criterion for regularity was formulated by O. Zariski in
localizations of polynomial rings \cite{Z4} and by M. Nagata in localizations of formal power series rings \cite{Na}.
Differential operators are used for computing the HS-stratum in the works of B. Bennett \cite{BB}, H. Hironaka \cite{H4}, J. Giraud \cite{Gi2}
and M.J. Pomerol \cite{Po}. }

\smallskip

 {Hironaka's theory of maximal contact is differential in nature and its positive characteristic version was developed
by J. Giraud \cite{Gi2}. In a formal power series rings $R=k[[X_0 , \ldots , X_n]]$, $\mathrm{char} k=p>0$, let
$J\subset R$ define a subscheme $X \subset Z=\mathrm{Spec}R$, $x$ denote the closed point. There exists an associated scheme
$W \subset Z$ with ``maximal contact'' in the following sense  \cite{Gi2} Proposition 3.3, Theorem 5.2 and Corollary 5.4:}

\smallskip

\noindent  {1- every permissible center $Y\subset X$ is also permissible for $W$;}

\noindent  {2- this property is stable at every $x'\in X'$, $x'$ near to $x$, where $X'\rightarrow X$ is the blowing up along $Y$.}

\smallskip

 {Roughly speaking, the space $W$ is constructed by taking
a projection transverse to the tangent cone of $J$ and applying transverse differential operators of appropriate order. The scheme $W$ has
a ``simpler'' singularity in the sense that its tangent cone $C_xW$ coincides with the ridge of $C_xX$, {\it viz.} \cite{Gi2} 1.5.
It is worth noting however that, unlike for $\mathrm{char} k=0$, $W$ is not regular at $x$ in general. Furthermore, $W=X$ when
the tangent cone $C_xX$ coincides with its ridge, i.e. its defining equations are additive polynomials:
$$
\lambda_0 X_0^{p^\alpha} + \cdots +\lambda_nX_n^{p^\alpha}, \ \lambda_1, \ldots, \lambda_n \in k, \ \alpha \geq 0.
$$
}

\medskip

 {New ideas were introduced with H. Hironaka's characteristic algebras \cite{H7} \cite{H8}. Hironaka and Giraud's ideas
have been influential in the last decade. Introducing independent new ideas of their own, H. Kawanoue and K. Matsuki
defined and studied idealistic filtrations \cite{Ka} \cite{KaM}.
Giraud's result was extended to algebraic varieties over perfect fields by A. Benito, A. Bravo and O. Villamayor \cite{BeV1} \cite{BeV2} \cite{BrV}.
Furthermore, they introduced a differential Rees algebra encoding this extended Giraud space $W$ and whose behavior by blowing up is traced by techniques of elimination theory
developed by the authors. Both approaches have produced
new proofs of resolution of singularities for surfaces \cite{KaM2} \cite{BeV3}.}

\medskip

 {C.~Abad gave a relative version  of  Zariski's Jacobian criterion over regular rings of characteristic~$p>0$ with an absolute $p$-basis. He also used absolute differential operators for  computing the singular locus of differential Rees algebras for varieties over a possibly non perfect field \cite{Abad} Proposition~5.1 and Theorem~7.5.} \\







In this section, we introduce the differential structure on the graded algebras
$\mathrm{gr}_\alpha S$. We will only consider here the case $\alpha =\mathbf{1} \in \R^J_{>0}$ with
notations as in Definition \ref{definh}. These algebras appear naturally as blow up
algebras of $S$ along regular primes. Our construction uses formal coordinates and Nagata derivatives.
For the reader's convenience, we remind the main concepts and classical results used here,
refering to \cite{Ma} pp.201-205 and pp. 235-245 when necessary.

\smallskip

 {The final part of the section is devoted to practical computations. These can be
performed without using formal coordinates when the exceptional divisor $E$ is locally
of finite type over some field, {\it vid.} Remark \ref{Nagataftype} below and
following propositions.}\\

 {To state the main goal of this section, assume for simplicity that some triple $(S,h,E)$
is specified as in Definition \ref{conditionE}. Assume furthermore that a permissible center ${\cal Y}\subset {\cal X}$
at $x$ w.r.t. $E$ is specified, where $x\in {\cal X}$ is the closed point. Let $W:=\eta ({\cal Y})\subset \mathrm{Spec}S$.
We will construct a certain ${\cal O}_W$-module ${\cal V}(F,E,W)$ (Definition \ref{calV}) and a
certain $\widehat{{\cal O}_W}$-module ${\cal J}(F,E,W)$ (Definition \ref{calJ}), where $\widehat{{\cal O}_W}$ is the completion of  ${\cal O}_W$. Factoring out a monomial part
from these modules, we obtain $V(F,E,W)$ and $J(F,E,W)$ in (\ref{eq244}). }

\smallskip

 {The origin of these modules sits in Cossart's thesis \cite{Co1} where resolution of singularities is proved
for hypersurfaces with equation
$$
h=y^p -f(u_1,u_2,u_3)=0
$$
over fields of characteristic $p>0$, see also \cite{Gi3} for the case of surfaces. Starting with a point of multiplicity
$p$ of the hypersurface $h=0$, making it drop by permissible blowing ups is very close to resolving the singularities of the
form $df$. With notations as above, the ideal ${\cal V}(f,E)$ generated by the coefficients of $df\in \Omega_S(\log E)$ is a reasonable
invariant for resolution and the goal is to get ${\cal V}(f,E)$ locally principal by blowing up regular centers $W\subset \mathrm{Spec}S$.
The transformation law for $df\in \Omega_S(\log E)$ involves a certain Jacobian ideal ${\cal J}(f,E,W)$. }

\smallskip

 {In our present -not necessarily equicharacteristic- setup, some initial form modules ${\cal V}(F,E,W)$ and ${\cal J}(F,E,W)$
can  be defined from the algebra gr$_WS$ and initial form polynomial in$_W h$, see Notation~\ref{notacenter} below. Taking $W=\{m_S\}$, we will define in forthcoming sections a numerical invariant $\omega (x)\in \N$ (Definition \ref{defomega})
and a stratum $\mathrm{Max}(x)$ in the tangent cone. The corresponding transformation law is the forthcoming Blowup Formula Proposition \ref{bupformula}(v).
}

\bigskip

 {We now proceed with formal definitions and constructions. As usual, $(S,m_S)$ is an excellent regular local ring with residue field
$$
k:={S\over m_S}, \ \mathrm{char}k=p>0.
$$
A r.s.p. $(u_1, \ldots ,u_n)$ of $S$ and a normal crossings divisor $E=\mathrm{div}(u_1 \cdots u_e)$ are specified, $1 \leq e \leq n$.
We assume that
$$
\mathrm{char}{S\over(u_j)}=p, \ 1\leq j \leq e,
$$
which is implied by assumption \textbf{(E)}, Definition~\ref{conditionE}.} We first adapt and simplify notations
as much as possible in order to fit with the forthcoming computations.

\begin{nota}\label{notacenter}
Let $W \subseteq E$ be a regular closed subset of
$\mathrm{Spec}S$ having normal crossings with $E$. For some suitable r.s.p. $(u_1, \ldots ,u_n)$ adapted to $E$ as above,
we may write
$$
I(W):=I_J=(\{u_j\}_{j \in J})\subset S \ \mathrm{for} \  \mathrm{some} \ J \subseteq \{1,\ldots ,n\}.
$$
Let $J_E:=J \cap \{1,\ldots ,e\}$, $J':=\{1,\ldots ,n\}\backslash J$, so $(J')_E=\{1,\ldots ,e\} \backslash J_E$.

\smallskip

 {Let $\cal{O}_W:=S/I(W)$ and $\overline{u}_{j'} \in \cal{O}_W$ be the image of $u_{j'}$, $j' \in J'$,
so}
 {$$
\overline{m}_S:=m_{\cal{O}_W}= (\{\overline{u}_{j'}\}_{j' \in J'}).
$$}

 {The $\overline{m}_S$-adic completion of $\cal{O}_W$ is  denoted by $\widehat{\cal{O}_W}$.}
 {The algebra $\mathrm{gr}_{\mathbf{1}}S$ of Definition \ref{definh} is denoted by:
$$
G(W):=\mathrm{gr}_{I(W)}S \simeq \cal{O}_W[\{U_j\}_{j \in J}].
$$}
Since $W \subseteq E$, we have:
$$
\mathrm{char}\ G(W)=\mathrm{char}k =p>0.
$$
The initial form $\mathrm{in}_{\mathbf{1}}h$ w.r.t. the weight vector
$\mathbf{1} \in \R^J_{>0}$ is now denoted
$$
\mathrm{in}_{W}h =X^p + \sum_{i=1}^p{F_{i,X,W}}X^{p-i} \in G(W)[X],
$$
with $F_{i,X,W} \in G(W)_{i\delta_{\mathbf{1}} (h;u_1,\ldots,u_n;X)}$, $1 \leq i \leq p$. \\
\end{nota}

Any local equation of $E$ has an initial form in $G(W)$, and we denote by $E(W)$ the associated divisor. Explicitly:
\begin{equation}\label{eq2414}
E(W):=\mathrm{div}\left (\prod_{j\in J_E}{U_j}\prod_{j'\in
(J')_E}{\overline{u}_{j'}}\right )\subset
\mathrm{Spec}G(W).
\end{equation}
We include in these definitions the case where $W=\mathrm{div}(u_j)$ is an irreducible component of $E$.
This corresponds to  $(J')_E=\{1, \ldots ,e\} \backslash \{j\}$ and
$$
G(W)=S / (u_j)[U_j], \  E(W)=\mathrm{div}\left (U_j\prod_{j'\in (J')_E}{\overline{u}_{j'}} \right ).
$$

 {We now recall the notion of $p$-basis and its connection with differentials and derivatives.}

 {\begin{defn}\label{pbasis}
Let $(\lambda_l)_{l \in \Lambda}$ be a family of elements of $k$. A $p$-monomial \index{p-monomial @ $p$-monomial, Definition~\ref{pbasis}}on $(\lambda_l)_{l \in \Lambda}$
is any element of the form:
$$
\prod_{l\in \Lambda}\lambda_l^{i_l}, \ 0 \leq i_l \leq p-1, \ i_l=0 \ \text{for almost all} \ l\in \Lambda.
$$
The family $(\lambda_l)_{l \in \Lambda}$ is called an absolute $p$-basis of $k$ \index{absolute  $p$-basis, Definition~\ref{pbasis}} if the family of all
$p$-monomials on $(\lambda_l)_{l \in \Lambda}$ is a basis of the $k^p$-vector space $k$.
\end{defn}}

 {This condition can be restated in terms of absolute differentials \cite{Ma} Theorem 26.5:}

 {\begin{prop}\label{absolutediff}
Let $(\lambda_l)_{l \in \Lambda}$ be a family of elements of $k$. The following properties are equivalent:
\begin{itemize}
  \item [(1)] $(\lambda_l)_{l \in \Lambda}$ is an absolute $p$-basis of $k$;
  \item [(2)] $(d\lambda_l)_{l \in \Lambda}$ is a basis of
the $k$-vector space of absolute differentials $\Omega_k$.
\end{itemize}
\end{prop}}

 {In particular, this proves that absolute $p$-bases of $k$ do exist. The corresponding family of derivations
is denoted by $({\partial \hfill{}\over \partial \lambda_l})_{l \in \Lambda}$.
They are defined by
$$
{\partial \lambda_{l'}\over \partial \lambda_l}=\delta_{l,l'}, \ l,l' \in \Lambda
$$
where $\delta_{l,l'}$ is the Kronecker symbol.}

\smallskip

 {For $A$ a ring and $M$ an $A$-module, we denote by $\mathrm{Der}(A,M)$ the $A$-module of derivations
of $A$ with values in $M$. The module $\mathrm{Der}(A,A)$ is simply denoted by $\mathrm{Der}(A)$. For every $k$-vector space $M$, we  have:
\begin{equation}\label{eq2450}
\mathrm{Der}(k,M)=\mathrm{Hom}_k(\Omega_k,M)\simeq M^\Lambda
\end{equation}
as $k$-vector spaces. We allow $\Lambda$ infinite in this construction. Note the strict inclusion
$$
\mathrm{Vect}_k(({\partial \hfill{}\over \partial \lambda_l})_{l \in \Lambda})\varsubsetneq \mathrm{Der}(k)
$$
when $\Lambda$ is infinite.}\\

 {The following is Cohen's Structure Theorem stated in a constructive way \cite{Ma} Theorem 28.3 and Lemma 1 on p. 216.
\begin{prop}\label{Cohen}\textbf{(I.S. Cohen)}
Let $(\gamma_l)_{l \in \Lambda}$ be a family of units in $\cal{O}_W$ whose residue $(\lambda_l)_{l \in \Lambda}$
is an absolute $p$-basis of $k$. There exists a unique ring isomorphism
$$
\phi : \ \widehat{\cal{O}_W} {\buildrel \sim \over \longrightarrow} k[[\{X_{j'}\}_{j'\in J'}]]
$$
such that $\phi (\gamma_l)=\lambda_l$ for $l\in \Lambda$, $\Phi (\overline{u}_{j'})=X_{j'}$ for $j'\in J'$,
and $\phi$ induces the identity map $k=\widehat{\cal{O}_W}/\overline{m}_S \longrightarrow k$ on residue classes.
\end{prop}
A slight abuse of notations allows us to write
$$
{\partial \hfill{}\over \partial \overline{u}_{j'}}:=\phi^{-1} \circ {\partial \hfill{}\over \partial X_{j'}} \circ \phi \in \mathrm{Der}(\widehat{\cal{O}_W}).
$$
Let $D \in \mathrm{Der}(k,k[[\{X_{j'}\}_{j'\in J'}]])$ act coefficientwise on $k[[\{X_{j'}\}_{j'\in J'}]]$, i.e.
$$
D \cdot \left (\sum_{\mathbf{a}\in \N^{J'}}\mu_\mathbf{a}X^\mathbf{a}\right )=  \sum_{\mathbf{a}\in \N^{J'}}(D \cdot \mu_\mathbf{a})X^\mathbf{a}.
$$
The isomorphism $\phi$ then provides an inclusion
$$
\mathrm{Der}(k,k[[\{X_{j'}\}_{j'\in J'}]])\subseteq \mathrm{Der}(\widehat{\cal{O}_W}), \ D \mapsto \phi^{-1} \circ D \circ \phi
$$
and its image will be simply denoted by $\mathrm{Der}(k,\widehat{\cal{O}_W})$ (called ``derivations w.r.t. to constants'').
Collecting together, we have a decomposition
\begin{equation}\label{eq2451}
\mathrm{Der}(\widehat{\cal{O}_W})=\mathrm{Der}(k,\widehat{\cal{O}_W}) \oplus
\left ( \bigoplus_{j'\in J'} \widehat{\cal{O}_W} {\partial \hfill{}\over \partial \overline{u}_{j'}}\right ).
\end{equation}
This is because derivations of $\widehat{\cal{O}_W}$ are continuous for the $\overline{m}_S$-adic topology,
so they are determined by their action on coefficients and variables.} Let:
\begin{equation}\label{eq2452}
\widehat{G(W)}:=\widehat{\cal{O}_W}\otimes_{\cal{O}_W}G(W) {\simeq \widehat{\cal{O}_W}[\{U_j\}_{j\in J}]}.
\end{equation}
\index{$\widehat{G(W)} $, equation~\eqref{eq2452}}



 {We now introduce the $\widehat{G(W)}$-module of absolute derivations of $\widehat{G(W)}$ which respect
the logarithmic structure given by $E(W)$, {\it viz.} (\ref{eq2414}).
\begin{defn}\label{D(W)}
With notations as above, let:
$$
{\cal D}(W):=\{D \in \mathrm{Der}(\widehat{G(W)}) : \ D \cdot I(E(W))\subseteq I(E(W))\}.
$$ \index{${\cal D}(W) $, Definition~\ref{D(W)}}
\end{defn}}

 {Once an isomorphism
$\phi : \ \widehat{\cal{O}_W} {\buildrel \sim \over \longrightarrow} k[[\{X_{j'}\}_{j'\in J'}]]$
has been chosen (Proposition \ref{Cohen}),} ${\cal D}(W)$ is generated as a
$\widehat{G(W)}$-module by those derivations w.r.t. constants $\mathrm{Der}(k,\widehat{G(W)})\simeq (\widehat{G(W)})^\Lambda$,
{\it viz.} (\ref{eq2450}), together with the finite family

\begin{equation}\label{eq2412}
{\cal B}(W):=\left (
\begin{array}{c}
\{U_j{\partial \hfill{}\over \partial U_j} \}_{j \in J_E},
\{{\partial \hfill{}\over \partial U_j}\}_{j \in J \backslash J_E} ,
\{\overline{u}_{j'}{\partial \hfill{}\over \partial
\overline{u}_{j'}}\}_{j' \in (J')_E}, \hfill{} \\
\{{\partial \hfill{}\over \partial
\overline{u}_{j'}}\}_{j' \in J' \backslash (J')_E} \hfill{}
\end{array}
\right ).
\end{equation}

Since $S_W$ is excellent and integrally closed, we have
$(\widehat{\cal{O}_W})^p \cap \cal{O}_W=(\cal{O}_W)^p$. In particular, we get
$\widehat{G(W)}^p \cap G(W)= G(W)^p$.
Therefore for $F \in G(W)$, there is an equivalence:
\begin{equation}\label{eq241}
    \forall D \in {\cal D}(W), \ D \cdot F=0 \Leftrightarrow F \in G(W)^p.
\end{equation}

\begin{defn}\label{calV}
Let $F \in G(W)_d$ be  {\it homogeneous} of degree $d\geq 1$. We define a homogeneous $\cal{O}_W$-submodule of $G(W)_{d-1}$ as follows:
\begin{equation}\label{eq2411}
{\cal V}(F,E,W):=\sum_{j \in J \backslash J_E}\cal{O}_W{\partial F \over \partial U_j} \subseteq G(W)_{d-1}.
\end{equation} \index{${\cal V}(F,E,W)$, Definition~\ref{calV}}
\end{defn}

 {\begin{prop}\label{calVindep}
Let $F \in G(W)_d$. The $\cal{O}_W$-module ${\cal V}(F,E,W)$ is independent of the choice of an isomorphism $\phi$
as in Proposition \ref{Cohen} and of an adapted r.s.p. $(u_1, \ldots ,u_n)$ of $S$ such that $I(W)=(\{u_j\}_{j \in J})$. Furthermore, we have:
$$
\sum_{j \in J \backslash J_E}\widehat{\cal{O}_W}{\partial F \over \partial U_j} = \widehat{\cal{O}_W}\otimes_{\cal{O}_W}{\cal V}(F,E,W),
$$
where the left hand side module is computed in $\widehat{G(W)}_d$.
\end{prop}}

 {\begin{proof}
Obvious from the definitions.
\end{proof}}

\begin{defn}\label{D_W}
With notations as above, let:  \index{${\cal D}_W $, Definition~\ref{D_W}}
$$
{\cal D}_W := \{ D \in {\cal D}(W) : D \cdot \left ({I(W)\over I(W)^2}\right )\subseteq
\left ({I(W)\over I(W)^2}\right )\} \subseteq {\cal D}(W).
$$
\end{defn}

 {Once an isomorphism
$\phi : \ \widehat{\cal{O}_W} {\buildrel \sim \over \longrightarrow} k[[\{X_{j'}\}_{j'\in J'}]]$
has been chosen, ${\cal D}_W$ is generated as a
$\widehat{G(W)}$-module by $\mathrm{Der}(k,\widehat{G(W)})$ together with the finite family}

 {\begin{equation}\label{eq242}
{\cal B}_W:=\left (
\begin{array}{c}
\{U_j{\partial \hfill{}\over \partial U_j} \}_{j \in J_E},
\{U_{j_1}{\partial \hfill{}\over \partial U_j}\}_{j \in J \backslash J_E, j_1\in J},
 \{\overline{u}_{j'}{\partial \hfill{}\over \partial
\overline{u}_{j'}}\}_{j' \in (J')_E}, \hfill{} \\
\{{\partial \hfill{}\over \partial
\overline{u}_{j'}}\}_{j' \in J' \backslash (J')_E} \hfill{}
\end{array}
\right ) .
\end{equation}}

Note that there is an equivalence
\begin{equation}\label{eq243}
    {\cal D}_W = {\cal D}(W) \Leftrightarrow W \ \mathrm{is} \
\mathrm{an} \  \mathrm{intersection} \ \mathrm{of} \
\mathrm{components} \  \mathrm{of} \ E.
\end{equation}

\begin{defn}\label{calJ}
Let $F \in G(W)_d$ be  {\it homogeneous} of degree $d\geq 1$. We define a homogeneous $\widehat{\cal{O}_W}$-submodule of
$\widehat{G(W)}_d=\widehat{\cal{O}_W}\otimes_{\cal{O}_W}G(W)_d$ as follows: \index{${\cal J}(F,E,W)$, Definition~\ref{calJ}}
\begin{equation}\label{eq2431}
{\cal J}(F,E,W):=\mathrm{cl}_{d}({\cal D}_W \cdot F ) \subseteq \widehat{G(W)}_d.
\end{equation}
\end{defn}

 {The ${\cal J}(F,E,W)$-version of  Proposition \ref{calVindep} goes as follows:}
\begin{prop}\label{calJindep}
 {Let $F \in G(W)_d$. The $\widehat{\cal{O}_W}$-module ${\cal J}(F,E,W)$ is independent of the choice of an isomorphism $\phi$
as in Proposition \ref{Cohen} and of an adapted r.s.p. $(u_1, \ldots ,u_n)$ of $S$ such that $I(W)=(\{u_j\}_{j \in J})$.}

\smallskip

 {For any such choice of $\phi$ and $(u_1, \ldots ,u_n)$, there exists a finite subset $\Lambda_F\subseteq \Lambda$ such that
$$
{\cal J}(F,E,W) = \sum_{l\in \Lambda_F}\widehat{\cal{O}_W}{\partial F \over \partial \lambda_l} +
\sum_{D \in {\cal B}_W}\widehat{\cal{O}_W} (D \cdot F).
$$}
\end{prop}

 {\begin{proof}
The first statement is trivial since ${\cal D}(W)$ and ${\cal D}_W$ do not depend on any choice of $\phi$ and $(u_1, \ldots ,u_n)$.
To prove the second part of the proposition, we let: \index{${\cal J}_0(F,E,W) $, proof of Proposition~\ref{calJindep}}
$$
{\cal J}_0(F,E,W) = \sum_{l\in \Lambda}\widehat{\cal{O}_W}{\partial F \over \partial \lambda_l} +
\sum_{D \in {\cal B}_W}\widehat{\cal{O}_W} (D \cdot F) \subseteq {\cal J}(F,E,W) .
$$
Since $\widehat{G(W)}_d$ is a finite module over the Noetherian ring $\widehat{\cal{O}_W}$, it is sufficient to prove that
${\cal J}_0(F,E,W) = {\cal J}(F,E,W)$. Let  $n\in \N$ be fixed. We expand
$$
F=\sum_{\mid \mathbf{i}\mid =d} \sum_{\mid \mathbf{a}\mid \leq n}
{\lambda_{\mathbf{i}, \mathbf{a}}\overline{u}^{\mathbf{a}}  U^{\mathbf{i}}}+F_n,
$$
with $F_n \in \overline{m}_S^{n+1}\widehat{G(W)}_d$ and $\lambda_{\mathbf{i}, \mathbf{a}}\in k$ (via the isomorphism $\phi$).
Since $(\lambda_l)_{l\in \Lambda}$ is an absolute $p$-basis of $k$, there exists a finite subset $\Lambda_0\subset \Lambda$
such that
$$
\lambda_{\mathbf{i}, \mathbf{a}}=\sum_{M\in {\cal M}_0}(\lambda_{\mathbf{i}, \mathbf{a},M})^pM,
\ \lambda_{\mathbf{i}, \mathbf{a},M} \in k,
$$
where ${\cal M}_0$ is the finite family of all $p$-monomials on $(\lambda_{l_0})_{l_0\in \Lambda_0}$. Let $D\in \mathrm{Der}(k,\widehat{\cal{O}_W})$. We deduce that
$$
D \cdot F \in {\cal J}_0(F,E,W) + \overline{m}_S^{n+1}\widehat{G(W)}_d .
$$
Since this holds for arbitrary $n\geq 0$, $D \cdot F$ belongs to the topological closure of ${\cal J}_0(F,E,W)$
in $\widehat{G(W)}_d$ for the $\overline{m}_S$-adic topology of $\widehat{G(W)}_d$. Therefore $D \cdot F \in {\cal J}_0(F,E,W)$
as required \cite{Ma} Theorem 8.6.
\end{proof}}

Let $H_W$ be the initial form in $G(W)$ of the monomial ideal $H(x)\subseteq S$ (Definition \ref{defH}),
where $x \in \eta^{-1}(m_S)$, i.e.  \index{$H_W$, equation~\eqref{eq2441}}
\begin{equation}\label{eq2441}
H_W:=\left (\prod_{j\in J_E}{U_j^{H_j}}\prod_{j'\in (J')_E}{\overline{u}_{j'}^{H_{j'}}}\right )\subseteq G(W)_{d_W},
\end{equation}
where $d_W:=\sum_{j\in J_E}H_j$. For $F \in H_WG(W)_{d-d_W}$, it follows from the above definitions that
$$
{\cal V}(F,E,W)\subseteq H_W G(W)_{d-d_W-1} \ \mathrm{and}
\ {\cal J}(F,E,W)\subseteq H_W\widehat{G(W)}_{d-d_W}.
$$
For such $F \in H_WG(W)_{d-d_W}$, we denote: \index{$V(F,E,W) $, equation~\eqref{eq244}} \index{$J(F,E,W) $, equation~\eqref{eq244}}
\begin{equation}\label{eq244}
\left\{
  \begin{array}{ccccc}
    V(F,E,W) & := & H_W^{-1}{\cal V}(F,E,W) & \subseteq  & G(W)_{d-d_W-1}, \\
     & & & & \\
    J(F,E,W) & := & H_W^{-1}{\cal J}(F,E,W) & \subseteq  & \widehat{G(W)}_{d-d_W}. \\
  \end{array}
\right.
\end{equation}

For $F=F_{p,X,W}\in H_W G(W)_{d-d_W}$, this defines the submodules
$$
V(F_{p,X,W},E,W)\subseteq G(W)_{d-d_W-1} \ \mathrm{and} \ J(F_{p,X,W},E,W)\subseteq \widehat{G(W)}_{d-d_W}.
$$
We will continually apply this definition when the following
properties (i) and (ii) hold:
\begin{itemize}
  \item [(i)] $(u_1,\ldots,u_n;X)$ are well adapted coordinates at $x \in \eta^{-1}(m_S)$
(Definition \ref{defwelladapted}), and
  \item [(ii)] $d-d_W=\epsilon (y)$ with $ \eta^{-1}(s)=\{y\}$, $s$ the generic point of $W$.
\end{itemize}
Note that $F_{p,X,W}\in H_W G(W)_{d-d_W}$ is then a consequence of Definition \ref{defepsilon}
and Proposition \ref{epsiloninv}.

\smallskip

Some considerations will require localizing $S$ at some point $s \in W $. We then
denote by $W_s$ the stalk of $W$ at $s$. This notation is used jointly
with Notation \ref{notaprime} {\it sqq.} about the stalk $E_s$. The restriction of
$s$ is denoted by $\overline{s} \in \mathrm{Spec}\cal{O}_W=G(W)_0$. We have
$$
G(W_s)=\mathrm{gr}_{I(W_s)}S_s\simeq (\cal{O}_W)_{\overline{s}}[\{U_j\}_{j \in J}].
$$
\begin{exam}
 {$S:=k[u_1,u_2,u_3]_{(u_1,u_2,u_3)}$, char$(k)=p>0$, $E=$div$(u_1u_2)$, $W=\{m_S\}$, $x\in$ Spec$S[Z]$ is the point of parameters $(Z,u_1,u_2,u_3)$. Let us study two different equations:}

\noindent  {$h_1:=Z^p+u_1^a u_2^b(u_3^{p+1}+\phi)\in S[Z],\ a,b \in \N_{>0},\ \phi\in m_S^{p+2},$}

\noindent  {$h_2:=Z^p+u_1^a u_2^b( u_3^{p}+\phi)\in S[Z],\ a,b \in \N_{>0},\ a\ \mathrm{or}\ b\not=0 \mathrm{mod}\ p,\ \phi\in m_S^{p+1}$,}

\noindent  {in both cases, $H_W:=<U_1^a U_2^b> $, $d_W=a+b$, $G(W)=\widehat{G(W)}=k[U_1,U_2,U_3]$.
}

\bigskip
\noindent  {In the first case,
$$\mathrm{in}_{W}h_1 =X^p +U_1^{a} U_2^{b}U_3^{p+1},\
d=a+b+p+1,\ F=U_1^{a} U_2^{b}U_3^{p+1}\in k[U_1,U_2,U_3]_d,$$
$${\cal V}(F,E,W)=<U_1^{a} U_2^{b}U_3^p>\subseteq k[U_1,U_2,U_3]_{d-1}, \  V(F,E,W)=<U_3^{p}>,$$
$${\cal J}(F,E,W)=<U_1^{a} U_2^{b}U_3^{p+1}>\subseteq k[U_1,U_2,U_3]_{d},\ { J}(F,E,W)=<U_3^{p+1}>.$$}

\noindent  {In the second case,
$$ \ \mathrm{in}_{W}h_2 =X^p +U_1^{a} U_2^{b}U_3^{p},\
d=a+b+p,\ F=U_1^{a} U_2^{b}U_3^{p}\in k[U_1,U_2,U_3]_d,$$
$${\cal V}(F,E,W)=<0>\subseteq k[U_1,U_2,U_3]_{d-1}, \ { V}(F,E,W)=<0>,$$
$${\cal J}(F,E,W)=<U_1^{a}U_2^{b}U_3^{p}>\subseteq k[U_1,U_2,U_3]_{d},\ {J}(F,E,W)=<U_3^{p}>.$$
}
\end{exam}

\bigskip

\begin{rem}\label{Nagataftype}
Formal coordinates and Nagata derivatives can be avoided if one assumes that
\begin{equation}\label{eq2415}
E=\mathrm{Spec}(S/(u_1 \cdots u_e))\subset \mathrm{Spec}S
\end{equation}
is essentially of finite type over some field. We explain below how
Zariski's Jacobian criterion \cite{Ma} Theorem 30.5 (2) can be used to avoid
introducing formal coefficients in defining ${\cal J}(F,E,W)$. We do not know any such alternative
description for arbitrary excellent regular local rings of characteristic $p>0$.   {We point out recent developments due to C.~Abad \cite{Abad}  who extends the existence of $p$-basis and validity of the Jacobian criterion to affine neighborhoods of a regular point (instead of its local ring).}

\smallskip

The extra assumption (\ref{eq2415}) is satisfied for example when $E$ is contained
in the closed fiber of some previously performed blowing ups.
In dimension three, this extra property is easily achieved from embedded resolution theorems
in smaller dimensions, {\it vid.} Lemma \ref{imagepoints}.
\end{rem}

 {For the remainder of this section, we consider a field $k$, $\mathrm{char}k=p>0$, and an absolute $p$-basis $(\lambda_l)_{l\in \Lambda}$.
Let $S$ be a regular local ring  which is essentially of finite type over $k$. This means that
for some polynomial ring $R:=k[X_1, \ldots ,X_{N}]$, we have
$$
S:=\left ({R \over IR}\right )_P, I \subset R \hbox{ an ideal}, \ P \in \mathrm{Spec}R, \ P \in V(I).
$$
Let  $x_i \in S$ denote the image of $X_i$, $1 \leq i \leq N$. Let:
$$
k(P):=S/P, \ t:=\mathrm{tr.deg}_kk(P), \ r:= \mathrm{ht}(IR_P), \ n:=\mathrm{dim}S=N-r-t .
$$
The following is the necessary condition in Zariski's Jacobian criterion for regularity applied to $S$ \cite{Ma} Theorem 30.5.}

 {\begin{prop}\textbf{(Zariski)}
With notations as above, there exists $F_1, \ldots ,F_r \in I$ and a finite subset $\Phi \subseteq \Lambda$
such that the Jacobian matrix
$$
J(F_1, \ldots ,F_r; \{{\partial \hfill{} \over \partial \lambda_l}\}_{l\in \Phi}, {\partial \hfill{} \over \partial X_1},
\ldots ,{\partial \hfill{} \over \partial X_{n+r}})
$$
has a $r \times r$ minor with nonzero residue in $k(P)$.
\end{prop}}

 {The following proposition is merely a rewriting of Zariski's Jacobian criterion for regularity
from the point of view of explicit computations. Its proof is elementary linear algebra.}

 {\begin{prop}\label{explicitZariski}
With notations as above, there exists finite subsets
$$
\Phi \subseteq \Lambda, \ \Psi \subseteq  \{1, \ldots ,N\}, \ \mid \Psi \mid =t + \mid \Phi \mid
$$
 with the following properties:
\begin{itemize}
  \item [(1)] let $\Lambda ':=(\Lambda \backslash \Phi) \coprod \Psi$ and define
  $$
  \lambda'_{l'}:=
  \left\{
  \begin{array}{ccc}
    \lambda_l & \mathrm{if} & l' \in  \Lambda \backslash \Phi \\
     & &  \\
    x_{l'} & \mathrm{if} & l'\in \Psi \\
  \end{array}
\right.
;
$$
then the residue family $(\overline{\lambda'_{l'}})_{l'\in \Lambda '}$ of  $(\lambda'_{l'})_{l'\in \Lambda '}$ in $k(P)$ is an absolute $p$-basis of $k(P)$.
  \item [(2)] for every r.s.p. $(u_1, \ldots ,u_n)$ of $S$, the family $((d\lambda'_{l'})_{l'\in \Lambda '}, du_1, \ldots ,du_n)$ is a basis of the free module $\Omega_S$ of absolute differentials;
  \item [(3)] the family of all $p$-monomials on $((\lambda'_{l'})_{l'\in \Lambda '}, u_1, \ldots ,u_n)$ is a basis
  of the free $S^p$-module $S$.
\end{itemize}
\end{prop}}

\begin{proof}
 {First choose $x_{i_1}, \ldots , x_{i_t}$ whose residues in $k(P)$ are a transcendence basis of $k(P)$ over $k$.
We may replace $k$ with $=k(x_{i_1}, \ldots , x_{i_t})$, $\Lambda$ with $\Lambda  \coprod \{i_1, \ldots ,i_t\}$ and
$\{1, \ldots ,N\}$ with $\{1, \ldots ,N\} \backslash \{i_1, \ldots ,i_t\}$ and thus assume that $P$ is a maximal ideal.}

\smallskip

 {We first prove the proposition when $I=(0)$, so $S=R_P$. Since (1) only refers to the residue field $k(P)$,
we will only have to prove (2) and (3) for arbitrary $I$.}

\smallskip

 {By elementary field theory, e.g. \cite{Ma} Theorem 5.1, $P=(G_1, \ldots , G_N)$, where
$$
G_j =X_j^{m_j} +\sum_{i=1}^{m_j}G_{j,i}(X_1, \ldots , X_{j-1})X_j^{m_j-i}, \ m_j\geq 1, \ G_{j,i} \in k[X_1, \ldots , X_{j-1}]
$$
for $1 \leq j \leq N$. We have
\begin{equation}\label{eq246}
\Omega_{R_P}=\left (\bigoplus_{l\in \Lambda}R_Pd \lambda_l \right ) \oplus \left ( \bigoplus_{j=1}^N R_PdX_j\right ).
\end{equation}}

 {We use induction on $j$, $1 \leq j \leq N$, to construct finite subsets
$\Phi_{j}\subseteq \Lambda $ and $\Psi_j\subseteq \{1, \ldots ,j\}$, $\mid \Psi_{j}\mid=\mid \Phi_{j}\mid$
such that
\begin{equation}\label{eq2461}
\{d \lambda_l\}_{l\in \Lambda \backslash \Phi_j}, \ \{dX_{i}\}_{i\in \Psi_j}, \ dG_1, \ldots ,dG_j, \ dX_{j+1}, \ldots ,dX_N
\end{equation}
is a basis of $\Omega_{R_P}$, and the residue family
\begin{equation}\label{eq2462}
(\overline{\lambda_{l}})_{l\in \Lambda  \backslash \Phi_j}, \ \{d\overline{X}_{i}\}_{i\in \Psi_j}
\end{equation}
form an absolute  $p$-basis of $k_j:=k[\overline{X}_{1}, \ldots , \overline{X}_{j}]$. Take $\Phi_0=\emptyset$ to begin with
and assume that $\Phi_{j-1}$ and $\Psi_{j-1}$ have been constructed. Apply the following algorithm:}\\

 {\noindent (A1) if ${\partial G_j\over \partial X_j}\neq 0$, take $\Phi_j=\Phi_{j-1}$, $\Psi_j=\Psi_{j-1}$; otherwise
go to (A2);}

\smallskip

 {\noindent (A2) pick $i$, $1 \leq i \leq m_j$ such that $\mu_{j,i}:=G_{j,i}(\overline{X}_1, \ldots , \overline{X}_{j-1})\not \in k_{j-1}^p$
and go to (A3);}

\smallskip

 {\noindent (A3) choose any $l'_j\in (\Lambda  \backslash \Phi_{j-1}) \coprod \Psi_{j-1}$ such that
${\partial \mu_{j,i} \over \partial \lambda_{l'_j}}\neq 0$ or ${\partial \mu_{j,i} \over \partial \overline{X}_{l'_j}}\neq 0$;
take  ($\Phi_j=\Phi_{j-1}\cup \{l'_j\}$, $\Psi_j=\Psi_{j-1}\cup \{j\}$), or ($\Phi_j=\Phi_{j-1}$, $\Psi_j=(\Psi_{j-1}\backslash \{l'_j\})\cup \{j\}$)
accordingly.} \\

 {The natural map $k_j\otimes_{k_{j-1}}\Omega_{k_{j-1}}\longrightarrow \Omega_{k_j}$
is an isomorphism when step (A1) applies. When step (A2) applies, there is an exact sequence
$$
0\longrightarrow k_j\longrightarrow
k_j\otimes_{k_{j-1}}\Omega_{k_{j-1}} \longrightarrow \Omega_{k_j} \longrightarrow k_j \longrightarrow 0.
$$
Step (A3) then chooses a splitting $1 \mapsto d\overline{X}_j$ of the cokernel  and a nonzero coefficient w.r.t
to the basis $(1\otimes\overline{\lambda_{l}})_{l\in \Lambda  \backslash \Phi_{j-1}}, \ \{1\otimes d\overline{X}_{i}\}_{i\in \Psi_{j-1}}$
for the generator  of the kernel $\sum_{i=1}^{m_j}\overline{X}_j^{m_j-i}\otimes dG_{j,i}(\overline{X}_1, \ldots , \overline{X}_{j-1})$. Applying (\ref{eq2462}) for $j=N$, this completes the proof of (1).} \\

 {Applying (\ref{eq246}) together with (\ref{eq2461}) for $j=N$, we get (2) for $I=(0)$. For $I$ arbitrary,
there is an exact sequence
$$
0 \longrightarrow \left ({IR_P \over (IR_P)^2}\right )\otimes_Sk(P)\longrightarrow \Omega_{R_P}\otimes_{R_P}k(P)
\longrightarrow \Omega_S\otimes_Sk(P)\longrightarrow 0,
$$
where exactness on the left holds because $S$ is regular. Taking preimages of $u_1, \ldots ,u_n$ in $R_P$
defines a splitting on the right and we get (2) for arbitrary $I$.}\\

 {Finally, we consider the derivations
\begin{equation}\label{eq2463}
({\partial \hfill{} \over \partial \lambda'_{l'}})_{l'\in \Lambda '},
{\partial \hfill{} \over \partial u_1}, \ldots ,{\partial \hfill{} \over \partial u_n} \in \mathrm{Der}(S)
\end{equation}
corresponding to the basis of the free module $\Omega_S$ given by (2). Let $K$ be the quotient field of $S$.
By Proposition \ref{absolutediff}, the family of all $p$-monomials on
$((\lambda'_{l'})_{l'\in \Lambda '}, u_1, \ldots ,u_n)$ is a basis of the $K^p$-vector space $K$.
Let $f \in S$ and
$$
f= \sum_{\mathbf{a},\mathbf{b}}(f_{\mathbf{a},\mathbf{b}})^p {\lambda '}^{\mathbf{a}}u^{\mathbf{b}}
$$
be the corresponding expansion, where $f_{\mathbf{a},\mathbf{b}} \in K$. Applying the derivations
in (\ref{eq2463}) and arguing by induction w.r.t. the graded lexicographical ordering, we get
$$
f_{\mathbf{a},\mathbf{b}} \in S \cap K^p=S^p.
$$
This concludes the proof of (3).}
\end{proof}

 {\begin{exam}
Let $k_0$ be a perfect field of characteristic $p>0$, $k:=k_0(\lambda_1,\lambda_2)$, $\lambda_1,\lambda_2$ indeterminates. Take $R=k[X_1,X_2]$, $I=(F)$, with
$$
F:= X_1^p +\lambda_1X_2^p +\lambda_2, \ P:=(X_1^p -\lambda_1\lambda_2, X_2^p+\lambda_2 +{\lambda_2 \over \lambda_1}).
$$
The above algorithm leads to: $\lambda'_1=x_1$, $\lambda'_2=x_2$, $u_1=x_1^p -\lambda_1\lambda_2$. The reader may check that
$$
k(P)=k_0(\overline{x}_1, \overline{x}_2)[\overline{\lambda}_1],
\ \overline{\lambda}_1^2 +\left ({\overline{x}_1 \over \overline{x}_2}\right )^p \overline{\lambda}_1
+  \left ({\overline{x}_1 \over \overline{x}_2}\right )^p =0,
$$
and that the residue class map gives an isomorphism $k(P)\simeq k^{{1 \over p}}$.\\
\end{exam}}

 {We now go back to the framework of the beginning of this section, see Notation \ref{notacenter},
Definition \ref{calJ} and Proposition \ref{calJindep}.}

\begin{prop}\label{defJftype}
 {Assume that $\cal{O}_W$ is essentially of finite type over some field.
Let $(b_a)_{a\in A} $ be a family of elements of $\cal{O}_W$ containing
$\{\overline{u}_{j'}\}_{j' \in (J')_E}$, and such that the family $(db_a)_{a\in A} $
forms a basis of the free $\cal{O}_W$-module
$\Omega_{\cal{O}_W}$. Write $b_{a_{j'}}=\overline{u}_{j'}$ for  $j'\in (J')_E$.
Let $F \in G(W)_d$ and define:}

 {\begin{equation}\label{eq2464}
{\cal J}'(F,E,W):= \left (
\begin{array}{c}
\{U_j{\partial F\over \partial U_j} \}_{j \in J_E},
\{U_{j_1}{\partial F \over \partial U_j}\}_{j \in J \backslash J_E, j_1\in J},
  \hfill{} \\
\{\overline{u}_{j'}{\partial F\over \partial
\overline{u}_{j'}}\}_{j' \in (J')_E}, \{{\partial F\over \partial
\overline{b}_a}\}_{a \in A \backslash \{a_{j'}\}_{j'\in (J')_E} }\hfill{}
\end{array}
\right ) \subseteq G(W)_d.
\end{equation}}

 {Then  ${\cal J}(F,E,W)=\widehat{\cal{O}_W}\otimes_{\cal{O}_W}{\cal J}'(F,E,W)\subseteq \widehat{G(W)}_d$.}\index{${\cal J}'(F,E,W) $, equation~\eqref{eq2464}}
\end{prop}

\begin{proof}
 {Applying Proposition \ref{explicitZariski}(1)(2), we may assume that
$$
A=\Lambda ' \coprod J', \ (b_a)_{a\in A} =((\lambda'_{l'})_{l'\in \Lambda '}, (\overline{u}_{j'})_{j'\in J'}).
$$}

 {Proposition \ref{Cohen} provides an associated  isomorphism
$$
\phi : \ \widehat{\cal{O}_W} {\buildrel \sim \over \longrightarrow} k[[\{X_{j'}\}_{j'\in J'}]].
$$
Let $f \in \cal{O}_W$. By Proposition  \ref{explicitZariski}(3), there is a finite expansion
$$
f= \sum_{\mathbf{a},\mathbf{b}}(f_{\mathbf{a},\mathbf{b}})^p {\lambda '}^{\mathbf{a}}\overline{u}^{\mathbf{b}}
$$
in terms of $p$-monomials, $f_{\mathbf{a},\mathbf{b}} \in \cal{O}_W$. Applying $\phi$
to this equation, the current proposition follows directly from
Definition \ref{calJ} and Proposition \ref{calJindep}.}
\end{proof}

\subsection{Cones, ridge and directrix.}

In this section, we recollect some facts about the directrix and Hilbert-Samuel stratum of
a homogeneous ideal. These facts are then applied to extract numerical invariants from the vector spaces
$$
V(F_{p,Z},E,m_S)\subseteq G(m_S)_{\epsilon (x)-1} \ \mathrm{and} \ J(F_{p,Z},E,m_S)\subseteq G(m_S)_{\epsilon (x)}
$$
defined in the previous section (\ref{eq244}) when $(u_1,\ldots,u_n;Z)$ are well adapted coordinates
at $x \in \eta^{-1}(m_S)$. These
considerations are based on elementary linear algebra.

 {Theorem~\ref{initform} distinguishes between two different cases for in$_{m_S}h$: (1) purely inseparable, (2) Artin-Schreier. Both vector spaces $V(F_{p,Z},E,m_S)$ and $J(F_{p,Z},E,m_S)$ are easily seen to be independent of the well adapted coordinates $(u_1,\ldots,u_n;Z)$ in case (1) (Proposition~\ref{indiff}(iii)). However, in case~(2), they do depend on $(u_1,\ldots,u_n;Z)$: see Example~\ref{ex:T}. To extract relevant information, we use a truncation map $T$ (Definition~\ref{defT}) which kills all monomials in the expansion of $ F_{p,Z}$ which may vary with $(u_1,\ldots,u_n;Z)$. The relevant information is provided in Proposition~\ref{Tinvariant}. }

\smallskip

Most difficulties in this section appear only for $n \geq 4$, which will eventually
lead us to define our main invariant $\omega (x)$ in a different way than in \cite{CoP2} chapter 1
(for equicharacteristic $S$ of dimension $n=3$) in the next section.\\

Let $k$ be a field, $R_1$ be a $k$-vector space of finite
dimension $n \geq 1$ and $R:=k[R_1]$ be the symmetric algebra. Let
${\mathbf V}:=\mathrm{Spec}R$ and $I$ be a homogeneous ideal of
$R$ which defines a cone $C=C(I):=\mathrm{Spec}(R/I)$. With these notations, we define:

\begin{defn}\label{defdirectrix}
The directrix \index{directrix Vdir$(I)$ of $C=C(I)$, Definition~\label{defdirectrix}} $\mathrm{Vdir}(I)$ of $C=C(I)$
is the smallest $k$-vector subspace $W$ of $R_1$ such that
$I=(I \cap k[W])R$. We denote \index{$\tau (I)$, Definition~\ref{defdirectrix}} \index{$\mathrm{Vdir}(I)$, Definition~\ref{defdirectrix}}
$$
\tau (I):=\mathrm{dim}_k\mathrm{Vdir}(I), \ \mathrm{Dir}(I):=\mathrm{Spec}(R/(\mathrm{Vdir}(I))).
$$
\end{defn}

\begin{defn}\label{defHilbertSamuel}
Let $C=C(F)$ be a hypersurface cone, i.e.  {$F$ is an homogeneous polynomial} and $I=(F)$ is a nonzero
principal ideal. We define a reduced subcone \index{$\mathrm{Max}(F)$, Definition~\ref{defHilbertSamuel}}
$$
\mathrm{Max}(F):= \{ x \in {\mathbf V} :
\mathrm{ord}_xF=\mathrm{ord}_0F\}\subseteq C(F),
$$
where $0$ is the origin (so $\mathrm{ord}_0F=\mathrm{deg}F$).

\smallskip

Given a {\it fixed} degree $d \geq 1$ and an ideal
$I=(F_1, \ldots ,F_m) \subset R$ defined by homogeneous polynomials $F_1, \ldots ,F_m \in R$,
$\mathrm{deg}F_i=d$ for $1 \leq i \leq m$, we let
$$
\mathrm{Max}(I):= \{ x \in {\mathbf V} :
\mathrm{ord}_xF_i=d, 1 \leq i \leq m\}\subseteq C(I).
$$
\end{defn}

The cone $\mathrm{Max}(I)$ is the closed Hilbert-Samuel stratum of $C(I)$.
These two objects and the ridge are considered and connected by H. Hironaka
in a more general context. See also \cite{Gi1} \cite{Gi2} \cite{Po} for definition and
computation of the ridge and Hilbert-Samuel stratum.

\begin{prop}\label{conedirectrix} \textbf{(Hironaka)}\cite{H4}
Let $C=C(F)$ be a hypersurface cone. There are inclusions
$$
 \mathrm{Dir}(F) \subseteq \mathrm{Max}(F) \subseteq C(F).
$$

If $k$ is perfect or if $\mathrm{dim}R \leq p+1$, the left hand side inclusion is an equality.
\end{prop}

\begin{rem}\label{ridgedimthree}
Counterexamples to the last statement exist for  {non-perfect} $k$ and $\mathrm{dim}R > p+1$.
For $\mathrm{dim}R  \leq 4$, such counterexamples exist only if $\mathrm{dim}R = 4$ and $p=2$.
For applications to the proof of Theorem \ref{luthm}, we only have to deal with this difficulty
for the initial form polynomial ($\mathrm{dim}R  = 4$) which is of the form
$$
\mathrm{in}{ {_{m_S}}} h=Z^2 -\lambda U_1Z +F_{2,Z}, \ F_{2,Z} \in S/m_S[U_1,U_2,U_3]_2, \ \lambda \in S/m_S.
$$
By \cite{H4}, the polynomial $\mathrm{in}{ {_{m_S}}} h$ is a counterexample to the last statement in
Proposition \ref{conedirectrix} if and only if $\lambda =0$ and, up to a linear change of variables,
\begin{equation}\label{eq2609}
\mathrm{in}_{m_S}h =Z^2 + \lambda_2 U_1^2 +\lambda_1U_2^2+\lambda_1 \lambda_2U_3^2
\end{equation}
with $\lambda_1,\lambda_2$ 2-independent, i.e.
$[(S/m_S)^2(\lambda_1,\lambda_2):(S/m_S)^2]=4$.
This very special case is dealt with in Proposition \ref{tausup2}.\\
\end{rem}

Let $(u_1,\ldots,u_n;Z)$ be well adapted coordinates at $x \in \eta^{-1}(m_S)$
(Definition \ref{defwelladapted}). In case
$\epsilon (x)>0$, we have $\eta^{-1}(m_S)=\{x\}$, $k(x)=S/m_S$ (Proposition \ref{deltainv})
and the initial form polynomial has the form
\begin{equation}\label{eq2551}
\mathrm{in}_{m_S} h =Z^p - G^{p-1}Z +F_{p,Z} \in G(m_S)[Z]= S/m_S[U_1, \ldots ,U_n][Z]
\end{equation}
by Theorem \ref{initform} applied to $\alpha =\mathbf{1} \in \R^n_{>0}$. There is an
associated integer $i_0(x)=p-1$ (resp. $i_0(x)=p$) if $G\neq 0$ (resp. if $G=0$). We denote by
$H \subseteq G(m_S)_d$ the initial form vector space of the ideal $H(x)$, $d=\sum_{j=1}^e H_j$
(Definition \ref{defH}).  If $i_0(x)=p-1$, we have
\begin{equation}\label{eq2552}
H^{-1}G^p =< \prod_{j=1}^e{U_j^{pB_j}}>, \ B_j\in {1 \over p}\N \ \mathrm{and} \  \sum_{j=1}^e{pB_j}=\epsilon (x).
\end{equation}
We can restate previous material as follows:

\begin{prop}\label{indiff}
Let $(u_1,\ldots,u_n;Z)$ be well adapted coordinates at $x \in \eta^{-1}(m_S)$ and
assume that $\epsilon (x)>0$. The following  {statements hold}:
\begin{itemize}
    \item [(i)] the vector space $V(F_{p,Z},E,m_S)\subseteq G(m_S)_{\epsilon (x)-1}$ satisfies
$$
V(F_{p,Z},E,m_S) = 0 \Leftrightarrow F_{p,Z}\in S/m_S [U_1, \ldots ,
U_e][U_{e+1}^p, \ldots U_n^p];
$$
    \item [(ii)] the vector space $J(F_{p,Z},E,m_S)\subseteq G(m_S)_{\epsilon (x)}$ satisfies
$$
J(F_{p,Z},E,m_S)= 0 \Leftrightarrow F_{p,Z}\in \left (S/m_S[U_1,\ldots ,U_n]\right )^p;
$$
    \item [(iii)] if $i_0(x) =p$, the vector space $V(F_{p,Z},E,m_S)$ is independent of the
    well adapted coordinates $(u_1,\ldots ,u_n;Z)$; if $i_0 (x)=p$ and $V(F_{p,Z},E,m_S) = 0$,
    the vector space $J(F_{p,Z},E,m_S)_{\epsilon (x)}$ is independent of the well adapted
    coordinates $(u_1,\ldots ,u_n;Z)$.
\end{itemize}
\end{prop}

\begin{proof}
The first statement follows from (\ref{eq2411}) and (\ref{eq244}),
while (ii) follows from (\ref{eq241}). Assume now that $i_0(x)=p$, i.e. $G=0$.

\smallskip

To begin with, the situation in (ii) does not occur because the polyhedron
$\Delta_{\hat{S}}(h;u_1,\ldots,u_n;Z)$ is minimal. If $Z'=Z-\theta$, $\theta \in \hat{S}$
with $\mathrm{ord}_{m_S}\theta \geq \delta (x)/p$, we have $F_{p,Z'}=F_{p,Z}+\Theta^p$ for some $\Theta \in
S/m_S[U_1,\ldots ,U_n]_{\delta (x)/p} $ (so $\Theta =0$ if $\delta (x)\not \in \N$).
Hence $D \cdot F_{p,Z'}=D \cdot F_{p,Z}$ for every $D \in \mathrm{Der}(G(m_S))$.

\smallskip

By elementary calculus, the vector space
$$
V(F_{p,Z},E,m_S)=H^{-1}<\left \{{\partial F_{p,Z} \over \partial U_j}\right \}_{e+1 \leq j \leq n}>
$$
is unchanged by adapted coordinate change (more generally by changes stabilizing the vector space
$<U_1, \ldots ,U_e>$) and this proves the first statement in (iii). If $V(F_{p,Z},E,m_S)=0$, the vector space
$$
J(F_{p,Z},E,m_S)=H^{-1}<\left \{U_j{\partial F_{p,Z} \over \partial U_j}\right \}_{1 \leq j \leq e},
\left \{{\partial F_{p,Z} \over \partial \lambda_l}\right \}_{l \in \Lambda}>.
$$
is not affected either by changes of coordinates fixing each $<U_j>$, $j\leq e$.
\end{proof}

We now turn to the version of Proposition \ref{indiff}(iii) for $i_0(x)=p-1$.
The problem is elementary, though more technical,
and the remaining part of this section is devoted to it. \\

Let $(\mathbf{e}_j)_{1 \leq j \leq n}$ be the standard basis of $\R^n$ and let
$$
\E:=\{\mathbf{x}\in \R^n : x_{e+1}= \cdots   =x_n =0\}\simeq \R^e.
$$
Given $d \in {1 \over p}\N$ and $\mathbf{H}\in \N^n \cap \E$, we denote
$$
\Delta_\mathbf{H} (d):=\{\mathbf{x}=(x_1, \ldots ,x_n) \in \R_{\geq 0}^n : \mid \mathbf{x} \mid = d \
\mathrm{and} \  x_j \geq {H_j\over p}, 1 \leq j \leq e\}
$$
and
\begin{equation}\label{eq2610}
{\cal V}_\mathbf{H}(pd):=(U^\mathbf{H})\cap G(m_S)_{pd}\subseteq G(m_S)_{pd}.
\end{equation}

We fix once and for all
\begin{equation}\label{eq:b1}
\mathbf{b}\in  (\N^n\cap \Delta_\mathbf{H} (d))\cap  \E.
\end{equation}
Note that ${\cal V}_\mathbf{H}(pd) \neq (0)$ only if $H_1 + \cdots + H_e \leq pd$ and that
such $\mathbf{b}$ as above exists only if $d \in \N$. By convention, we take $\{\mathbf{b}\}=\emptyset$
if $d \not \in \N$ in the following formul{\ae}. For applications, we will take $d=\delta (x_0)$, $\mathbf{H}$ as
in Definition \ref{defH} and $\mathbf{b}$ will be defined by
\begin{equation}\label{eq:b2}
<G>= :<U_1^{b_1} \cdots U_{e}^{b_e}>.
\end{equation}

\begin{nota} Any homogeneous polynomial $F \in {\cal V}_\mathbf{H}(pd)$ has a unique  expansion of the form
$$
F:=\sum_{\mathbf{x} \in {1 \over p}\N^n \cap \Delta_\mathbf{H} (d)}{\lambda (\mathbf{x})U^{p\mathbf{x}}},
\lambda (\mathbf{x}) \in S/m_S.
$$
We denote
$$
\Delta (F):=\mathrm{Conv}(\{\mathbf{x} \in {1 \over p}\N^n \cap \Delta_\mathbf{H} (d) :
\lambda (\mathbf{x})\neq 0\} \cup \{\mathbf{b}\}) \subseteq \Delta_\mathbf{H} (d) .
$$
According to   {these} conventions, we have $\Delta (0)=\{\mathbf{b}\}$.
\end{nota}

\begin{defn}\label{defT}
With notations as above, let $T: \ {\cal V}_\mathbf{H}(pd)\rightarrow {\cal V}_\mathbf{H}(pd)$ be the $S/m_S$-linear
truncation operator defined as follows: let \index{truncation operator! $S/m_S$-linear
truncation operator $T$, Definition~\ref{defT}}
\begin{equation}\label{eq2613}
A:=\{ \mathbf{x} \in  {1 \over p}\N^n \cap \Delta_\mathbf{H} (d) : \mathbf{b} +p(\mathbf{x}-\mathbf{b})\in \Delta_\mathbf{H} (d) \}.
\end{equation}
and
\begin{equation}\label{eq2611}
T F:=\sum_{\mathbf{x} \not \in  A}{\lambda (\mathbf{x})U^{p\mathbf{x}}}
\in {\cal V}_\mathbf{H}(pd).
\end{equation}
For $d \not \in \N$, we have $A=\emptyset$ and $T$ is the identity map.
\end{defn}

The construction of the previous section associates  {to $x \in \eta^{-1}(m_S)$} two vector spaces
$V(T F , E,m_S)$ and $J(T F , E,m_S)$. Explicitly, we have: \index{$V(T F , E,m_S) $, equation~\eqref{eq2611bis}}
\begin{equation}\label{eq2611bis}
V(T F , E,m_S)=U^{-\mathbf{H}}<{\partial TF \over \partial U_j}, e+1 \leq j \leq n>
\subseteq G(m_S)_{pd -1 - \mid \mathbf{H}\mid}
\end{equation}

for the former one. If $V(T F , E,m_S)=0$ (and only in this case), we will use the latter
one, given explicitly by  \index{$J(T F , E,m_S) $, equation~\eqref{eq2611ter}}
\begin{equation}\label{eq2611ter}
J(T F , E,m_S)=U^{-\mathbf{H}}<\{U_j{\partial TF \over \partial U_j}\}_{1 \leq j \leq e},
\{{\partial TF \over \partial \lambda_l}\}_{l \in \Lambda} >
\subseteq G(m_S)_{pd - \mid \mathbf{H}\mid},
\end{equation}
with notations as in the previous section.

 {The $S/m_S$-linear
truncation operator $T$ defined above is useful to give a definition of  $\omega(x)$ (the adapted order defined below see Definition \ref{defomega}) independent of all possible choices of well adapted coordinates. The possible vanishing of $V(T F , E,m_S)$ is essential in this definition. Let us point out the problem. For simplicity, we take $E=$div$(u_1u_2)$. The vector space $V( F , E,m_S)$ depends on the choice of the pair $(v,Z)$ where $v=u_3$ is  a free variable and  $(u_1,u_2,v;Z)$ are well adapted coordinates.  The
truncation operator $T$ is devised to suppress this dependence.}

 {The following example shows the bad behavior of   $V( F , E,m_S)$ without truncating.}

\begin{exam}\label{ex:T}
 {char$(k)=2, \ k\not=\F_2,\ S=k[[u_1,u_2,v]],\  E=div(u_1u_2)$,
$ h = X^2 +u_1^3u_2^2X +u_1u_2 [ v^8 +u_2^4v^4 +u_1^3u_2^4v]\in S[X]$. We get:
$$\mathrm{Discr}(h)=u_1^6u_2^4,\ \epsilon (x)=8,\ \delta(x)=5.$$
We have $V( F_{Z,U_1,U_2,V} , E,m_S)\not=<0>$ for any choice of $Z$ such that $(u_1,u_2,v;Z)$ are well  adapted coordinates.}

 \smallskip
 { Let
$w:=v+\lambda  u_2, \ \lambda \in k$: $(X,u_1,u_2,w)$ is a regular system of parameters at $x$,
$$\clin_{m_S} h = X^2 +U_1^3U_2^2X +U_1U_2 [ W^8+U_2^4W^4+U_1^3U_2^4W+ \lambda U_1^3U_2^5+[\lambda(\lambda +1)]^4U_2^8 ] .$$
As $k\not=\F_2$,  we can choose $\lambda$ such that
 $\lambda(\lambda +1)\not=0.$
 Then $\Delta_{S}(h; u_1,u_2,w;X)$ has three not solvable vertices of same modules $\delta_x$:
$M:=(3,2,0)$ given by $U_1^3U_2^2X $,
$N:=({1 \over 2},{9 \over 2},0)$ given by $\lambda(\lambda +1)U_1U_2 U_2^8$,   $P:=({1 \over 2},{1 \over 2},4)$ given by $U_1U_2 W^8$. So $(u_1,u_2,w;X)$ are well adapted parameters.}

 {The monomial $\lambda U_1U_2  \times U_1^3U_2^4W$ defines the point $(2,5/2,1/2)$ inside the first face and this monomial gives
\begin{equation}\label{eqtronc1}
V( F_{X,U_1,U_2,W} , E,m_S)\not=<0>.
\end{equation}}

 {We make the change of variable: $Z=X+u_1 u_2^3 w$, as $(1,3,1)$ is in the interior of the triangle  $MNP$,  $\Delta_{S}(h; u_1,u_2,w;Z)=\Delta_{S}(h; u_1,u_2,w;X)$: the coordinates $(u_1,u_2,w;Z)$ are well adapted. This gives:
$$\clin_{m_S} h = Z^2 +U_1^3U_2^2Z+U_1U_2 [  W^8+U_2^4W^4+U_1U_2^5W^2+ \lambda U_1^3U_2^5+[\lambda(\lambda +1)]^4U_2^8].$$
 With natural notations, we get
\begin{equation}\label{eqtronc2}
V( F_{Z,U_1,U_2,W}  , E,m_S)=<0>.
\end{equation}
By Lemma \ref{kerT} below,
$
 V( TF , E,m_S)=<0>
 $
 for all the well adapted coordinates we used above.}

\end{exam}

We can now state:

\begin{lem}\label{kerT}
Assume that $d \in \N$. With notations as above, we have
$$
\mathrm{Ker}T =U^{(p-1)\mathbf{b}}{\cal V}_{\lceil {\mathbf{H}\over p}\rceil}(d),
$$
where $\lceil {\mathbf{H}\over p}\rceil :=(\lceil {H_1\over p}\rceil , \ldots , \lceil {H_e\over p}\rceil, 0 ,\ldots ,0)$.

Let $G:= \mu U^{\mathbf{b}}$, $\mu \in S/m_S$, $\Phi \in {\cal V}_{\lceil {\mathbf{H}\over p}\rceil}(d)$
and $F\in {\cal V}_\mathbf{H}(pd)$.
Then
$$
V( T(F+ \Phi^p -G^{ {p-1}}\Phi ), E,m_S)=V( TF, E,m_S).
$$
If $V( TF, E,m_S)=0$, then
$$
J( T(F+ \Phi^p -G^{ {p-1}}\Phi ), E,m_S)=J( TF, E,m_S),
$$
\end{lem}

\begin{proof}
We analyze the definition of $T$ in (\ref{eq2611}). The kernel of $T$ is generated by
those monomials $U^{p\mathbf{x}} \in {\cal V}_\mathbf{H}(pd)$ such that
$$
\mathbf{y}:= p\mathbf{x}-(p-1)\mathbf{b} \in \Delta_\mathbf{H} (d).
$$
Since $\mathbf{x}\in {1 \over p}\N^n$, $\mathbf{b}\in \N^n$, we have $\mathbf{y}\in \N^n$ for such
$\mathbf{y}$. Therefore $\mathrm{Ker}T$ is generated by
$$
\mathrm{Ker}T=<\{U^{(p-1)\mathbf{b}}U^{\mathbf{y}} : \mathbf{y}\in \N^n,
\ \mid \mathbf{y}\mid =d \ \mathrm{and}  \ y_j\geq {H_j \over p}, 1 \leq j \leq e \}>.
$$
This proves the first statement. For the second part, we have proved that
$$
T(F+ \Phi^p -G^{ {p-1}}\Phi)=TF + T\Phi^p.
$$
Hence $D \cdot T(F+ \Phi^p -G^{ {p-1}}\Phi)=D \cdot TF$ for every
$D \in \mathrm{Der}(G(m_S))$.
\end{proof}

We now study invariance properties of $V( F, E,m_S)$ and $J( F, E,m_S)$ under changes
of adapted coordinates. Given two r.s.p.'s $\mathbf{u}=(u_1,\ldots ,u_n)$ and
$\mathbf{u'}=(u'_1,\ldots ,u'_n)$  adapted to $E$,
there exists a matrix $M \in {\cal M}(S)$,
$$
{\cal M}(S):=\{(m_{ij}) \in \mathrm{GL}(n,S): m_{jj'}=0,
(j,j') \in \{1, \ldots, e\}\times\{1, \ldots ,n\}, j\neq j'\}
$$
such that $\mathbf{u}=M\mathbf{u'}$. The set ${\cal M}(S)$ is the set of $S$-points of an
affine $S$-scheme ${\cal M} \subset \mathrm{GL}(n,S)$. Denote by
$$
\mathrm{GL}(n,S) \rightarrow \mathrm{GL}(n,S/m_S), \ M \mapsto \overline{M}
$$
the canonical surjection. Each such $\overline{M}$ induces a graded $S/m_S$-automorphism of
$\mathrm{gr}_{m_S}(S)\simeq S/m_S[U_1, \ldots , U_n]$. By (\ref{eq2610}),
this automorphism restricts to an automorphism of ${\cal V}_\mathbf{H}(pd)$ for each $d \in {1 \over p}\N$
still denoted by $\overline{M}$.

Given a homogeneous polynomial $F \in {\cal V}_\mathbf{H}(pd)$ as above and a matrix
$\overline{M} \in {\cal M}(S/m_S)$, we denote for simplicity the transformed equation
$U \mapsto \overline{M}U'$ by
\begin{equation}\label{eq2711}
F'=: \sum_{\mathbf{x'} \in {1 \over p}\N^n \cap \Delta_\mathbf{H} (d)}{\lambda '(\mathbf{x'}){U'}^{p\mathbf{x'}}}.
\end{equation}
Let $\Delta (F'):=\mathrm{Conv}(\{\mathbf{x'} \in {1 \over p}\N^n \cap \Delta_\mathbf{H} (d) :
\lambda '(\mathbf{x'})\neq 0\} \cup \{\mathbf{b}\})\subseteq \Delta_\mathbf{H} (d)$
be the corresponding polytope and $T'$ be the corresponding operator on ${\cal V}_\mathbf{H}(pd)$
with variable $U'$. The linear operator $T$ obviously does not commute with $\overline{M}$ in
general (i.e. $(TF)'\neq T'F'$ in general), but the lemma below extracts the relevant invariant data.
We refer to Definition \ref{defHilbertSamuel} for the Notation $\mathrm{Max}(I)$, $I \subset G(m_S)$
generated by  {one homogeneous polynomial or homogeneous polynomials} of  the same degree. \\

\begin{nota}\label{defB}
 {Recall \eqref{eq:b1} and \eqref{eq:b2}}. We denote by
\begin{equation}\label{eq2612}
B:=\{ j , \ 1 \leq j \leq e: pb_j -H_j>0\}\ \mathrm{and} \ U_B:=\{U_j, j \in B\}.
\end{equation}
We denote $U_{B'}:=\{U_j, j \not \in B\}$ and stick to our former conventions, i.e.
$$
B'=\{1, \ldots ,n\} \backslash B, \ (B')_E =\{1, \ldots ,e\} \backslash B.
$$
\end{nota}

 {\begin{rem}The vector space $U_B$ is sometimes used to get a notion of maximal contact  for our main invariant $\omega$ with the components div$(u_j)$, $j\in B$: see chapter~6,  $\kappa(x)=1$.
\end{rem}}

\begin{lem}\label{Cmaxinv}
With notations as above, there is an equality of sets
\begin{equation}\label{eq2713}
\mathrm{Max}(V( TF, E,m_S))\cap \{U_B=0\}=\mathrm{Max}(V( T'F', E,m_S))\cap \{U'_B=0\}.
\end{equation}
If $V( TF, E,m_S)=0$, then $V( T'F', E,m_S)=0$ and there is an equality of sets
\begin{equation}\label{eq2714}
\mathrm{Max}(J( TF, E,m_S))\cap \{U_B=0\}=\mathrm{Max}(J( T'F', E,m_S))\cap \{U'_B=0\}.
\end{equation}
\end{lem}

\begin{proof}
The operator $T$ commutes with $\overline{M}$ when $\overline{M}$
stabilizes the vector space $<U_{e+1}, \ldots ,U_n>$. In these cases, we have
$$
V( T'F', E,m_S)=V( (TF)', E,m_S).
$$
If $V( TF, E,m_S)=0$, then
$$
V( T'F', E,m_S)=0 \ \mathrm{and} \ J( T'F', E,m_S)=J( (TF)', E,m_S).
$$
So the lemma is trivial in this case and we may therefore assume that
$$
m_{jj'}=0,(j,j') \in \{e+1, \ldots, n\}\times\{e+1, \ldots ,n\}, j\neq j' \
\mathrm{and} \ m_{jj}=1, 1 \leq j \leq n.
$$
By elementary calculus, this new assumption implies for every $\Phi \in G(m_S)$:
\begin{equation}\label{eq2712}
{\partial \Phi '\over \partial U'_j}=\left ({\partial \Phi \over \partial U_j} \right )', \ e+1 \leq j \leq n.
\end{equation}

Let $\mathbf{x}\in {1 \over p}\N^n \cap \Delta_\mathbf{H} (d)$. Since $pb_j=H_j$ for $j\in (B')_E$, we have by (\ref{eq2613}):
$$
\mathbf{x}\in A \Leftrightarrow \forall j \in B, px_j \geq (p-1)b_j.
$$
Expand $TF=\sum_\mathbf{y}U_B^\mathbf{y}F_\mathbf{y}(U_{B'})$, so we have:
$$
V(TF,E,m_S)=U^{-\mathbf{H}}<\{\sum_\mathbf{y}U_B^\mathbf{y}{\partial F_\mathbf{y}(U_{B'})\over \partial U_j} \}_{e+1 \leq j \leq n}>.
$$
For $P \in \mathrm{Spec}G(m_S)$ such that $(U_B) \subseteq P$, we get:
\begin{equation}\label{eq2715}
P\in \mathrm{Max}(V( TF, E,m_S)) \Leftrightarrow P \in \bigcap_{\mathbf{y}}
\bigcap_{j=e+1}^n{\mathrm{Max}( {G_{\mathbf{y},j}})},
\end{equation}
where $ {G_{\mathbf{y},j}}:= U_{B'}^{-\mathbf{H}'}{\partial F_\mathbf{y}(U_{B'}) \over \partial U_j}$,
$\mathbf{H}':=(H_{j'})_{j ' \in (B')_E}$.

Suppose furthermore that $\overline{M}$ stabilizes the vector space $<U_{B'}>$. Then
$T$ also commutes with $\overline{M}$ and each term
$ {G_{\mathbf{y},j}}$ in (\ref{eq2715}) is transformed into
$$
( {G_{\mathbf{y},j}})'=U_{B'}^{-\mathbf{H}_{B'}}{\partial F'_\mathbf{y}(U_{B'}') \over \partial U'_j}
$$
by (\ref{eq2712}) and (\ref{eq2713}) follows. Suppose furthermore that $V(TF,E,m_S)=0$; then $ {G_{\mathbf{y},j}}=0$
for each $\mathbf{y},j$ in (\ref{eq2713}) and we get $V(T'F',E,m_S)=0$.
For $1 \leq j \leq e$ and $l \in \Lambda$, we have
\begin{equation}\label{eq2716}
\left (U_j{\partial TF \over \partial U_j}\right )'=U'_j{\partial T'F' \over \partial U'_j}, \
\left ({\partial TF \over \partial \lambda_l}\right )'={\partial T'F' \over \partial \lambda_l},
\end{equation}
and (\ref{eq2714}) also follows. Hence we may furthermore assume that
$$
m_{jj'}=0,(j,j') \in \{e+1, \ldots, n\}\times (B')_E.
$$

In this situation, $T$ does not commute any longer with $\overline{M}$. However, for each term $ {G_{\mathbf{y},j}}$
as above, we have
\begin{equation}\label{eq2717}
\mathrm{ord}_P (D \cdot  {G_{\mathbf{y},j}})\geq \mathrm{deg} {G_{\mathbf{y},j}} -a
\end{equation}
for any differential operator $D$ on $S/m_S[U_{B'}]$ of order not greater than $a$. Let
$$
( {G_{\mathbf{y},j}})'=\sum_{\mid \alpha \mid\leq \mathrm{deg} {G_{\mathbf{y},j}}}(U_{B}')^\alpha (D^{(\alpha )}\cdot  {G_{\mathbf{y},j}}),
\ D^{(\alpha )}\cdot  {G_{\mathbf{y},j}}\in S/m_S[U'_{B'}]_{\mathrm{deg} {G_{\mathbf{y},j}}-\mid \alpha \mid}
$$
be the (characteristic free) Taylor expansion, where $D^{(\alpha )}$ is a differential operator
of order $\mid \alpha \mid$. Take again $P \in \mathrm{Spec}G(m_S)$
such that $(U_B) \subseteq P$. By (\ref{eq2717}), we have
$$
P \in \mathrm{Max}( {G_{\mathbf{y},j}}) \Rightarrow P \in \bigcap_{\alpha} \mathrm{Max}(D^{(\alpha )}\cdot { G_{\mathbf{y},j}})
\Rightarrow P \in \mathrm{Max}(( {G_{\mathbf{y},j}})').
$$
We deduce from (\ref{eq2715}) that
$$
P\in \mathrm{Max}(V( TF, E,m_S)) \Rightarrow P \in  \mathrm{Max}(V( (TF)', E,m_S)).
$$
This proves (\ref{eq2713}). If $V( TF, E,m_S)=0$, (\ref{eq2714}) follows from (\ref{eq2716}) as above.
\end{proof}

This lemma is the key to our version of Proposition \ref{indiff}(iii) for $i_0(x)=p-1$:

\begin{prop}\label{Tinvariant}
Let $(u_1,\ldots,u_n;Z)$ be well adapted coordinates at $x \in \eta^{-1}(m_S)$ and
assume that $\epsilon (x)>0$ and $i_0(x)=p-1$. Let
$$
d:=\delta (x), \ \mathbf{H}:=(H_1, \ldots , H_e, 0, \ldots ,0) \ \mathrm{and} \
<U_1^{b_1} \cdots U_{e}^{b_e}>:=<G>
$$
be defined respectively by Definition \ref{defdelta}, Definition \ref{defH} and (2) of
Theorem \ref{initform}. With  {notation as above, the following statements hold}:
\begin{itemize}
  \item [(i)] the set
$$
\mathrm{Max}(V( TF_{p,Z}, E,m_S))\cap \{U_B=0\} \subseteq \mathrm{Spec}G(m_S)
$$
is independent of the well adapted coordinates $(u_1,\ldots,u_n;Z)$;
  \item [(ii)] the property $V( TF_{p,Z}, E,m_S)=0$ is independent of the well adapted coordinates
$(u_1,\ldots,u_n;Z)$; when it holds, the set
$$
\mathrm{Max}(J( TF_{p,Z}, E,m_S))\cap \{U_B=0\} \subseteq \mathrm{Spec}G(m_S)
$$
is also independent of the well adapted coordinates $(u_1,\ldots,u_n;Z)$.
\end{itemize}
\end{prop}

\begin{proof}
For such $(u_1, \ldots ,u_n;Z)$, the corresponding initial form is
$$
\mathrm{in}_{m_S} h =Z^p - G^{p-1}Z +F_{p,Z} \in G(m_S)[Z].
$$
Since $G\neq 0$, we have $d=\delta (x)=\mathrm{deg}G \in \N$.
If $(u'_1, \ldots ,u'_n)$ is an adapted r.s.p. of $S$, there exists $M \in {\cal M}(S)$ such
that  $\mathbf{u}=M\mathbf{u'}$. Let $(u'_1, \ldots ,u'_n;Z')$
be well adapted coordinates at $x$. We have $Z'=Z - \phi$ for some $\phi \in S$, with
$\mathrm{ord}_{m_S}\phi \geq d $. We deduce that
$$
\mathrm{in}_{m_S} h ={Z'}^p - G^{p-1}Z' +\Phi ^p - G^{p-1}\Phi +F_{p,Z}\in G(m_S)[Z']
$$
for  $\Phi :=\mathrm{cl}_d \phi \in G(m_S)_d$. We deduce the formula
$$
   F_{p,Z'}=F_{p,Z}+\Phi^p - G^{p-1}\Phi .
$$
By Lemma \ref{kerT}, we have $V( T F_{p,Z'}, E,m_S)=V( TF_{p,Z}, E,m_S)$; if moreover $V( TF_{p,Z}, E,m_S)=0$,
then $J( T F_{p,Z'}, E,m_S)=J( TF_{p,Z}, E,m_S)$. By Lemma \ref{Cmaxinv}, we have
an equality of sets
$$
\mathrm{Max}(V( TF_{p,Z'}, E,m_S))\cap \{U_B=0\}=\mathrm{Max}(V( T'F'_{p,Z'}, E,m_S))\cap \{U'_B=0\}
$$
and this proves (i). If $V( TF_{p,Z'}, E,m_S)=0$, then $V( T'F'_{p,Z'}, E,m_S)=0$ by Lemma \ref{Cmaxinv} and
there is an equality of sets
$$
\mathrm{Max}(J( TF_{p,Z'}, E,m_S))\cap \{U_B=0\}=\mathrm{Max}(J( T'F'_{p,Z'}, E,m_S))\cap \{U'_B=0\}.
$$
This concludes the proof.
\end{proof}

\begin{rem}\label{Bempty}
We consider Proposition \ref{indiff}(iii) as the special case $B=\emptyset$, $T=\mathrm{id}$ of
Proposition \ref{Tinvariant}.
\end{rem}

\subsection{Main invariants.}

Let $s \in \mathrm{Spec}S$ and $y \in \eta^{-1}(s)$. The purpose of this section is
to attach to $y$ a resolution complexity \index{$\iota (y)$, equation\eqref{eq251}}
\begin{equation}\label{eq251}
\iota (y)=(m(y), \omega (y), \kappa (y))\in  \{1,\ldots ,p\}\times  \N \times  \{1,\geq 2\}
\end{equation}
with certain invariance properties. Auxiliary numbers
\begin{equation}\label{eq2511}
   (\tau (y),\tau '(y)) \in \{1,\ldots ,n+1\} \times \{1,\ldots ,n\}
\end{equation}
are similarly attached to $y$.

\smallskip

The pair $(m(y),\tau (y))$ are the standard multiplicity and
Hironaka $\tau$-num\-ber of ${\cal X}$ at $y$  (Definition \ref{defmult}).
The pair $(\omega(y),\tau '(y))$  {plays} the role of a differential multiplicity and
differential $\tau$-number attached to $\eta : {\cal X} \rightarrow \mathrm{Spec}S$ at $y$.
The behavior of the function $\iota $ under blowing up is studied in Theorem \ref{bupthm} below.

\smallskip

In all definitions that follow it can be assumed without loss of generality that $s=m_S$
by localizing $S$ at $s$, since our assumptions {\bf (G)} and {\bf (E)} are stable
when changing $(S,h,E)$ to $(S_s,h_s,E_s)$ (Notation \ref{notaprime}).

\begin{defn}\label{defmult}(Multiplicity).
Let $x \in \eta^{-1}(m_S)$. We have already defined \index{multiplicity $m(x)$, Definition~\ref{defmult}}
$$
m(x)=\mathrm{ord}_{m_{S[X]_x}} h(X)\leq p.
$$
Let $M_x \subset S[X]$ be the ideal of $x$,
$G_x:=\mathrm{Spec}(\mathrm{gr}_{M_x}S[X]_{M_x})$ and  {in$_xh(X)$} be the initial form
of $h$ in $(G_x)_{m(x)}$. From Definition \ref{defdirectrix}, we let
$$
\tau (x):=\tau (   {\mathrm{in}_xh(X)} ).
$$

If $m(x)<p$, we let $\iota (x):=(m(x), {\omega(x):=0},1)$.
\end{defn}

\begin{defn}\label{defomega}(Adapted order).
Assume that $m(x)=p$, where $\{x\}=\eta^{-1}(m_S)$. Let  $(u_1, \ldots ,u_n;Z)$
be well adapted coordinates at $x$. We let \index{adapted order $\omega$, Definition~\ref{defomega}}
$$
    \omega (x)=\left\{
\begin{array}{ccc}
  \epsilon (x)-1 & \mathrm{if} &  V(TF_{p,Z},E,m_{S})\neq 0 \\
  \epsilon (x)  \hfill{} & \mathrm{if} &  V(TF_{p,Z},E,m_{S})= 0\\
\end{array}
\right .
.
$$
We define: \index{$\kappa (x)= 1$, Definition~\ref{defomega}}
$$
    \kappa (x):= 1 \ \mathrm{if} \ \mathrm{(} \omega (x)=\epsilon (x) \ \mathrm{and} \ i_0(x)=p-1\mathrm{)}\  {\mathrm{or\ if}\ \omega (x)=0}.
$$
Otherwise, we simply let $\kappa (x) \geq 2$.
\end{defn}

\begin{rem} {
Note that $m (y) <p$ whenever $s=\eta (y) \not \in E$ (Definition
\ref{conditionE} and following comments). If $m(y)= p$, we have
$$
s =\eta (y) \in E, \ \eta^{-1}(s)=\{y\} \ \mathrm{and} \ k(y)=k(s)
$$
by Proposition \ref{SingX}.}

\smallskip

 {Applying Proposition \ref{indiff}(iii) (resp. Proposition  \ref{Tinvariant}(ii)) to $S$
if $i_0(x)=p$ (resp. if $i_0(x)=p-1$) proves that $(\omega (x), \kappa (x))$ is
well-defined. We recall that $TF_{p,Z}=F_{p,Z}$ whenever $i_0(x)=p$ (see Remark \ref{Bempty}).}
\end{rem}

\begin{rem}
It is obvious from this definition that $\omega (x)$ is not determined by the characteristic polyhedra
$\Delta_S (h;u_1,\ldots,u_n;Z)$, even for unspecified well adapted coordinates $(u_1,\ldots,u_n;Z)$.

For example, take  $n=3$, $p\geq 3$ for simplicity and $k(x)$ algebraically closed of characteristic $p>0$.
Suppose:
$$
\mathrm{in}_{m_S}h=Z^p +U_1U_2U_3^p +U_1^{p+2}+U_2^{p+2} +c U_3U_2U_1^{p} , \ E=\mathrm{div}(u_1u_2),
$$
where $c\in k(x)$. Let $(u'_1,u'_2,u'_3;Z')$ be well adapted coordinates such that
$\mathrm{div}(u_j)= \mathrm{div}(u'_j)$ for $j=1,2$.  {Then the corresponding initial face of $\Delta_S (h;u'_1,u'_2,u'_3;Z') $ is:
$$
\mathrm{Conv}(\{\mathbf{v}_1 , \mathbf{v}_2, \mathbf{v}_3\})
\subset \{x_1+x_2+x_3=\delta (x)=1+2/p\}
$$
and is independent of $c$}, where
$$
\mathbf{v}_1:=((p+2)/p,0,0), \ \mathbf{v}_2:=(0,(p+2)/p,0), \ \mathbf{v}_3:=(1/p,1/p,1).
$$
But $\omega (x)=p+2$ (resp. $\omega (x)=p+1$) for $c =0$ (resp. for $c\neq 0$).

\end{rem}

\begin{rem} This definition is different from the one used in \cite{CoP2} chapter 1,
Definition {\bf II.4} when $G\neq 0$. Let $(u_1,\ldots,u_n;Z)$ be well adapted coordinates at $x$.
There is an obvious implication
$$
\omega (x)=\epsilon (x)-1 \Longrightarrow V(F_{p,Z},E,m_{S})\neq 0.
$$
The converse is however false, even if it is assumed that $V(F_{p,Z},E,m_{S})\neq 0$ for every
possible choice of well adapted coordinates $(u_1,\ldots,u_n;Z)$ at $x$ and this is the reason
for this difference. For $n\leq 3$, this phenomenon is easily dealt with, {\it vid.}
\cite{CoP2} chapter 1 {\bf II.3.3.1} and {\bf II.3.3.2}; proof of {\bf II.5.4.2}(iv); Theorem
{\bf II.5.6}.

\smallskip

In chapter  {5}, we define the projection number $\kappa (x)\in \{2,3,4\}$ when $n=3$ and state that
$\iota (x)=(m(x),\omega (x), \kappa (x))$ can be decreased by Hironaka permissible blowing ups w.r.t. $E$
(Projection Theorem \ref{projthm} below).
\end{rem}

We now turn to the definition of the adapted cone and directrix and the attached invariant $\tau '(x)$.

\begin{defn}\label{deftauprime} (Adapted cone and directrix).
Assume that $m(x)=p$ and $\omega (x)>0$, where $\{x\}=\eta^{-1}(m_S)$. Let  $(u_1, \ldots ,u_n;Z)$
be well adapted coordinates at $x$.
We define a reduced subcone $\mathrm{Max}(x)\subseteq \mathrm{Spec}G(m_S)$ by:
$$
    \mathrm{Max}(x):=\left\{
\begin{array}{ccc}
  \mathrm{Max}(V( TF_{p,Z}, E,m_S))\cap \{U_B=0\} & \mathrm{if} &  \omega (x)=\epsilon (x)-1 \\
  \mathrm{Max}(J( TF_{p,Z}, E,m_S))\cap \{U_B=0\} & \mathrm{if} &  \omega (x)=\epsilon (x)  {.} \hfill{}\\
\end{array}
\right
.
$$
We define an affine subspace $\mathrm{Dir}(x)\subseteq \mathrm{Spec}G(m_S)$ by \index{$\mathrm{Dir}(x)$, Definition~\ref{deftauprime}}
$$
    \mathrm{Dir}(x):=\left\{
\begin{array}{ccc}
  \mathrm{Dir}(V( TF_{p,Z}, E,m_S),U_B) & \mathrm{if} &  \omega (x)=\epsilon (x)-1 \\
  \mathrm{Dir}(J( TF_{p,Z}, E,m_S),U_B) & \mathrm{if} &  \omega (x)=\epsilon (x)  {.} \hfill{}\\
\end{array}
\right
.
$$
We let $\mathrm{Vdir}(x)$ to be the underlying vector space of $\mathrm{Dir}(x)$ and \index{$\tau '(x)$, Definition~\ref{deftauprime}}
$$
\tau '(x):=\mathrm{dim}_{k(x)}\mathrm{Vdir}(x).
$$
\end{defn}

\begin{rem}  {Applying Proposition \ref{indiff}(iii) (resp. Proposition  \ref{Tinvariant})
if $i_0(x)=p$ (resp. if $i_0 (x)=p-1$) proves that $\mathrm{Max}(x)$, $\mathrm{Dir}(x)$
and $\tau '(x)$ are well defined.}
We will use the invariants $\mathrm{Dir}(x)$ and $\tau '(x)$
 {only when $n=3$. In this case, we have $\mathrm{Dir}(x)=\mathrm{Max}(x)$} (last statement in Proposition \ref{conedirectrix}
and remark  following).\\
\end{rem}

Let $S \subseteq \tilde{S}$ be a regular local base change, $\tilde{S}$ excellent. Recall
Notation \ref{notageomreg1} and Notation \ref{notaprime}. It has been explained when defining
conditions {\bf (G)} and {\bf (E)} that they are stable by such base changes and by
localization at a prime.
Let  $\tilde{s}\in \mathrm{Spec}\tilde{S}$ and $\tilde{y} \in \tilde{\eta}^{-1}(\tilde{s})$.
In order to relate $\iota (\tilde{y})$ and $\iota (y)$ (\ref{eq251}), where $y \in {\cal X}$ is
the image of $\tilde{y}$, we may thus assume that $s=m_S$, $\tilde{s}=m_{\tilde{S}}$.

\smallskip

Let $(u_1, \ldots ,u_n;Z)$ be well adapted coordinates at $x \in \eta^{-1}(m_S)$. Then
$(u_1, \ldots ,u_n)$ can be completed to a r.s.p. $(u_1, \ldots ,u_{\tilde{n}})$ of $\tilde{S}$
which is adapted to $\tilde{E}$. There is an inclusion
\begin{equation}\label{eq2513}
G(m_S)=k(x)[U_1, \ldots ,U_{n}] \subseteq
G(m_{\tilde{S}})=G(m_S) \otimes_{k(x)} {\tilde{S} \over m_{\tilde{S}}}[ {{U}_{n+1}, \ldots , {U}_{\tilde{n}}}].
\end{equation}

\begin{thm}\label{omegageomreg}
Let $S \subseteq \tilde{S}$ be a local base change which is regular, $\tilde{S}$ excellent. Let
$\tilde{x} \in \tilde{\eta}^{-1}(m_{\tilde{S}})$ and $x \in \eta^{-1}(m_S)$ be its image.
The following holds:
\begin{itemize}
  \item [(1)] we have $(m (\tilde{x}), \omega (\tilde{x}))= (m(x), \omega (x))$;
  \item [(2)] if $m(x)=p$, then
  \begin{itemize}
   \item [(i)] $H(\tilde{x})=H(x)\tilde{S}$, $i_0(\tilde{x})=i_0(x)$,
   and ($\kappa (\tilde{x})=1 \Leftrightarrow \kappa (x)=1$);
  \item [(ii)] we have $\epsilon (\tilde{x})\geq \epsilon (x)$,
  and $\epsilon (\tilde{x})> \epsilon (x)$ if and only if
  $$
  \mathrm{in}_{m_S}h=Z^p + F_{p,Z}, \ F_{p,Z} \in  (k(\tilde{x})[U_1, \ldots ,U_n])^p
  $$
  where $(u_1, \ldots ,u_n;Z)$ are well prepared coordinates at $x$. When this holds,
  we have $\tilde{n}>n$, $\epsilon (\tilde{x})=\epsilon (x)+1$ and
  $$
  \mathrm{in}_{m_{\tilde{S}}}\tilde{h} =\tilde{Z}^p + \sum_{j=n+1}^{\tilde{n}}U_j\Phi_j(U_1, \ldots , U_n)
   + \Psi(U_1, \ldots , U_n) \in G(m_{\tilde{S}})[\tilde{Z}],
  $$
   with $\Phi_j\neq 0$ for some  $j \geq n+1$ and $\Phi_j\in k(\tilde{x})[U^p_1, \ldots ,U^p_n]$
   for every $j \geq n+1$, where $(u_1, \ldots ,u_{\tilde{n}};\tilde{Z})$ are well prepared coordinates at $\tilde{x}$.
  \end{itemize}
\end{itemize}
\end{thm}

 {
\begin{exam}
Let us note that case (2)(ii) with $\epsilon(\tilde{x})>\epsilon(x)$ occurs in Example \ref{ex:kangourou}. We give another example involving a formal fiber.
Let  $k$ be a field of characteristic $p>0$,
$$R:=k[u_1,u_2,u_3,u_4]_{(u_1,u_2,u_3,u_4)}, \ E=\mathrm{div}(u_1u_2), \ \mathfrak{P}:=(u_1,u_2),\ S:=R_\mathfrak{P}.$$
Let $\phi\in k[[u_4]]$ be transcendental over  $k(u_4)$ with $\phi(0)=0$. Let $\widehat{\mathfrak{P}}:=(u_1,u_2,\hat{v}:=u_3-\phi^p)$ and $\tilde{S}:=\hat{R}_{\widehat{\mathfrak{P}}}$. The local base change $S\subset \tilde{S}$ is regular as $R$ is excellent. Let
$$h:=Z^p+u_1^pu_3+u_2^{p+1} \in S[Z].$$
We denote by $x$ and $\tilde x$ the closed points of Spec$(S)$ and Spec$( \tilde{S})$. The coordinates $(u_1,u_2;Z)$ are well adapted at $x$. We have
$$\epsilon(x)=p, \ \mathrm{in}_\mathfrak{P}(h)=Z^p+\overline{u_3}U_1^p\in k(\overline{u_3},\overline{u_4})[U_1,U_2].$$
Let $\tilde{Z}=Z+u_1 \phi$, the coordinates $(u_1,u_2,\hat{v};\tilde{Z})$ are well adapted at $\tilde x$.
$$h=\tilde{Z}^p+u_1^p \hat{v}+u_2^{p+1},$$
$$\epsilon(\tilde{x})=p+1,\ \mathrm{in}_{\widehat{\mathfrak{P}}}(h)=\tilde{Z}^p+U_1^p\widehat{V}+U_2^{p+1}\in k((u_4))[U_1,U_2,\widehat{V}].$$
\end{exam}}

\begin{proof}
The theorem is trivial if $m(x)=1$: then $m (\tilde{x})=1$ because
$S \subseteq \tilde{S}$ is regular.

Assume that $m(x)\geq 2$ and pick well prepared coordinates $(u_1,\ldots ,u_n;Z)$ at $x$,
then complete $(u_1,\ldots ,u_n)$ to a r.s.p. $(u_1, \ldots ,u_{\tilde{n}})$ of $\tilde{S}$
which is adapted to $\tilde{E}$. We have $\delta (x)>0$, so
$h \in (Z,u_1, \ldots ,u_n)$, and $k(x)=S/m_S$ by Proposition \ref{SingX}.
Applying (\ref{eq2513}) to
the local base change $S[Z]_{(m_S,Z)}\subseteq  {\tilde{S}[Z]_{(m_{\tilde{S}},Z)}}$ which is also
regular gives
$$
m(x)=\mathrm{ord}_x h(Z)=\mathrm{ord}_{\tilde{x}} \tilde{h}(Z)=m(\tilde{x}).
$$

This concludes the proof when $m(x)<p$ ( {$\omega(x)=\omega(\tilde{x})=0$ in this case}) and we assume from now on that $m(x)=p$.
In  particular we have $\{\tilde{x}\}=\tilde{\eta}^{-1}(m_{\tilde{S}})$,
$k(\tilde{x})=\tilde{S}/m_{\tilde{S}}$. Let
$$
\mathrm{in}_{m_S}h =Z^p + \sum_{i=1}^p{F_{i,Z}}Z^{p-i} \in
G(m_S)[Z],
$$
be the corresponding initial form polynomial. Let $\mathbf{x} \in \R^{n}_{\geq 0}$
be a vertex of the polyhedron $\Delta_{S} {(h;u_1, \ldots ,u_n;Z)}$. We denote by
$$
\mathbf{\tilde{x}}:=(\mathbf{x}, \underbrace{0, \ldots ,0}_{\tilde{n}-n}) \in
\Delta_{\tilde{S}}(u_1, \ldots ,u_{\tilde{n}};Z)
$$
the corresponding vertex in $\Delta_{\tilde{S}} {(h;u_1, \ldots ,u_{\tilde{n}};Z)}$. Note that $\mathbf{\tilde{x}}$
{\it may be a solvable vertex} of the latter polyhedron. We have:
$$
\mathbf{\tilde{x}} \ \mathrm{solvable} \Leftrightarrow \mathrm{in}_{\mathbf{\tilde{x}}}\tilde{h}
\in ((\mathrm{gr}_\alpha \tilde{S})[Z])^p
$$
with notations as in Definition \ref{defsolvable}. Therefore we have
$$
\mathbf{\tilde{x}} \ \mathrm{solvable} \Leftrightarrow (\mathrm{in}_{\mathbf{x}}h =Z^p + F_{p,Z,\mathbf{x}},
\mathbf{x}  \in \N^n , \ F_{p,Z,\mathbf{x}}=\lambda U^{p \mathbf{x}},
\lambda \in k(\tilde{x})^p).
$$
We deduce for the initial form polynomial that
\begin{equation}\label{eq2518}
\delta (\tilde{x})>\delta (x) \Leftrightarrow (i_0(x)=p \ \mathrm{and} \ F_{p,Z}\in  (k(\tilde{x})[U_1, \ldots ,U_{n}])^p).
\end{equation}

Since the fiber ring $\tilde{S}/m_S \tilde{S}$ is geometrically regular over $k(x)$,
the ring $\tilde{S}[Y]/(Y^p - l)$ is regular for every unit
$l \in S$ with residue $\overline{l} \not \in k(x)^p$. Therefore if
$\overline{l}  \in k(\tilde{x})^p$, we have
$$
\forall \tilde{l} \in \tilde{S},  \tilde{v}:=\tilde{l}^p - l \in m_{\tilde{S}} \Longrightarrow
\tilde{v} \ \mathrm{is} \  \mathrm{a} \ \mathrm{regular} \ \mathrm{parameter} \ \mathrm{in} \ \tilde{S}.
$$
Such $\tilde{v}$ restricts to a regular parameter of $\tilde{S}/m_S \tilde{S}$, so
the previous formula is refined to:
\begin{equation}\label{eq2516}
\tilde{v} \ \mathrm{is} \   \mathrm{a} \ \mathrm{regular} \ \mathrm{parameter} \
\mathrm{transverse} \ \mathrm{to} \ \mathrm{div}(u_1 \cdots u_n)\subset \mathrm{Spec} \tilde{S}.
\end{equation}

This equation implies in particular that $\tilde{n}>n$.
Let $\xi \in \mathrm{Spec}(\tilde{S}/m_S \tilde{S})$ be the generic point. Applying the above
remarks to the regular local base change $S\subset \tilde{S}_{\xi}$ shows that $k(\xi)^p \cap k(x)= k(x)^p$.

Let $s_j:=(u_j) \in \mathrm{Spec}S$, $1 \leq j \leq e$, and apply this remark to the regular local base change
$S_{(u_j)}\subseteq \tilde{S}_{(u_j)}$. This proves that the field inclusion
$QF(S/(u_j))\subseteq QF(\tilde{S}/(u_j))$ is inseparably closed.

The polynomial $\mathrm{in}_{(s_j)}h_{s_j} \in QF(S/(u_j))[U_j][Z]$ is not a $p^\mathrm{th}$-power
by Proposition \ref{Deltaalg}. Therefore $\mathrm{in}_{(s_j)}h_{s_j}$ is not a $p^\mathrm{th}$-power
in $QF(\tilde{S}/(u_j))[U_j][Z]$. Turning back to Definition \ref{defepsilon}, we get
\begin{equation}\label{eq2515}
H(\tilde{x})=H(x)\tilde{S}.
\end{equation}
Definition \ref{defepsilon} now shows that $\epsilon (\tilde{x}) \geq \epsilon (x)$ and that
\begin{equation}\label{eq2519}
\epsilon (\tilde{x}) > \epsilon (x) \Leftrightarrow
(i_0(x)=p \ \mathrm{and} \ F_{p,Z}\in  (k(\tilde{x})[U_1, \ldots ,U_n])^p).
\end{equation}
This proves the first part of (2.ii). To go on with the proof, we consider two cases.\\

{\it Case 1:} assume that $i_0(x)<p$. By (\ref{eq2519}), we have $\epsilon (\tilde{x}) = \epsilon (x)$,
so the proof of (2.ii) is already complete. Let $\tilde{\phi} \in \tilde{S}$ be
such that $\Delta_{\tilde{S}}(u_1, \ldots ,u_{\tilde{n}};\tilde{Z})$ is minimal,
with $\tilde{Z}:=Z-\tilde{\phi}$ and $\mathrm{ord}_{m_{\tilde{S}}}\tilde{\phi} \geq \delta (x)$. We have
$$
\mathrm{in}_{m_{\tilde{S}}}\tilde{h} =\tilde{Z}^p + \sum_{i=i_0}^{p}F_{i,\tilde{Z}}\tilde{Z}^{p-i}  \in
G(m_{\tilde{S}})[\tilde{Z}],
$$
with $F_{i_0,\tilde{Z}}=F_{i_0,Z}$ by Proposition \ref{izero}. Therefore $i_0(\tilde{x})=i_0(x)$
and it is sufficient to prove that $\omega (\tilde{x})=\omega (x)$ in order to complete the proof of
(1) and (2.i) in the theorem (still under the assumption $i_0(x)<p$). This is obvious if $\epsilon (x)=0$, since
$$
0 \leq \omega (\tilde{x})\leq \epsilon (\tilde{x})=\omega (x)=0.
$$
Assume that $\epsilon (x)>0$. We have $i_0(x)=p-1$ and $-F_{p-1,Z}=G^{p-1}$, with $<G>=<U^{\mathbf{b}}>$
for some $\mathbf{b} \in \N^n \cap \E$ by Theorem \ref{initform}(2) (in particular $\delta (x) \in \N$).
We have
$$
V(TF_{p,Z},E,m_S)=<\left \{H^{-1}{\partial TF_{p,Z} \over \partial U_j}\right \}_{e+1 \leq j \leq n}>.
$$
Note that the truncation maps $T$ and $\tilde{T}$ associated with the local rings $S$ and $\tilde{S}$
(Definition \ref{defT}) commute with the inclusion $G(m_S)\subseteq G(m_{\tilde{S}})$
by (\ref{eq2515}). Since $F_{p,Z}\in G(m_S)=k(x)[U_1, \ldots ,U_{n}]$, we have
$$
V(\tilde{T}F_{p,Z},\tilde{E},m_{\tilde{S}})=<\left \{H^{-1}{\partial \tilde{T}F_{p,Z} \over \partial U_j}
\right \}_{j=e+1}^{\tilde{n}}> =V(TF_{p,Z},E,m_S)\otimes_{k(x)}k(\tilde{x})
$$
with obvious notations, taking (\ref{eq2515}) into account.
There exists $\tilde{\Theta}\in G(m_{\tilde{S}})$ such that
$$
F_{p,\tilde{Z}}=F_{p,Z} + \tilde{\Theta}^p -G^{p-1}\tilde{\Theta}.
$$
By Lemma \ref{kerT} applied to $F_{p,\tilde{Z}}\in G(m_{\tilde{S}})$, we deduce that
\begin{equation}\label{eq2524}
V(\tilde{T}F_{p,\tilde{Z}},\tilde{E},m_{\tilde{S}})=V(TF_{p,Z},E,m_S)\otimes_{k(x)}k(\tilde{x}).
\end{equation}
This completes the proof of the theorem when $\omega (x)=\epsilon (x)-1$, applying Definition \ref{defomega}.
If $\omega (x)=\epsilon (x)$, (1) and the last statement of (2.i) in the theorem also
follow from (\ref{eq2524}) and the proof is complete. \\

{\it Case 2:} assume that $i_0 (x)=p$. The proof runs parallel to that of case 1 (with $B=\emptyset$,
$\tilde{T}=\mathrm{id}$, {\it cf.} Remark \ref{Bempty}) {\it provided that}
$\epsilon (\tilde{x})=\epsilon (x)$. Assume now that $\epsilon (\tilde{x})>\epsilon (x)$.
To complete the proof, we have to show
that
$$
(i_0(\tilde{x}), \omega (\tilde{x}))= (p, \omega (x)),
$$
as well as the last statement in (2.ii). By (\ref{eq2519}),
we have $\omega (x)=\epsilon (x)$, $\delta (x)\in \N$ and there is an expansion
$$
F_{p,Z}=\sum_{\mid \mathbf{x}\mid =\delta (x)}\lambda (\mathbf{x})U^{p\mathbf{x}}
\in (k(\tilde{x})[U_1, \ldots ,U_n]_{\delta (x)})^p, \ \lambda (\mathbf{x}) \in k(x).
$$
Note that this situation possibly occurs only if $k(x)$ is {\it not} inseparably closed in $k(\tilde{x})$
(in particular $\tilde{n}>n$). We have $\mathbf{x} \in \N^{n}$ for every $\mathbf{x}$ such that
$\lambda (\mathbf{x}) \neq 0$.
Without loss of generality, it can be assumed that $\lambda (\mathbf{x}) \not \in k(x)^p$ for
every $\mathbf{x}$ such that $\lambda (\mathbf{x}) \neq 0$. Let $l(\mathbf{x})\in S$ be a preimage of
$\lambda (\mathbf{x})$. By (\ref{eq2516}), we may pick for every such $\mathbf{x}$
a unit $\tilde{l}(\mathbf{x})\in T$ such that $\tilde{v}(\mathbf{x}):=\tilde{l}(\mathbf{x})^p - l(\mathbf{x})$ is
a regular parameter of $\tilde{S}$ transverse to $\mathrm{div}(u_1 \cdots u_n)$. Expand
$$
h= Z^p + \sum_{i=1}^p{f_{i,Z}}Z^{p-i} \in S[Z], \ \mathrm{ord}_{m_{S}}f_{i,Z}\geq i\delta (x).
$$
For $1 \leq i \leq p-1$, the above inequality is strict, since $i_0(x)=p$. On
the other hand, we have $\delta (x) \in \N$, so we deduce that
\begin{equation}\label{eq2520}
{\mathrm{ord}_{m_{S}}f_{i,Z} \over i}\geq \delta (x)+{1 \over i}>\delta (x)+{1 \over p},
\ 1 \leq i \leq p-1.
\end{equation}

Let
$$
\tilde{Z}:=Z + \sum_{\mid \mathbf{x}\mid =\delta (x)}\tilde{l}(\mathbf{x})u^{\mathbf{x}}.
$$
By (\ref{eq2520}), there is an expansion
\begin{equation}\label{eq2521}
f_{p,\tilde{Z}}=-\sum_{\mid \mathbf{x}\mid =\delta (x)}\tilde{v}(\mathbf{x}) u^{p\mathbf{x}} +
g  + \tilde{g},
\end{equation}
with $g \in S$, $\mathrm{ord}_{m_{S}}g \geq p\delta (x)+1$ and $\tilde{g} \in \tilde{S}$,
$\mathrm{ord}_{m_{\tilde{S}}}\tilde{g} > p\delta (x)+1$ . We deduce that
$$
\delta (h; u_1, \ldots ,u_{\tilde{n}}; \tilde{Z})= \delta (x)+{1 \over p}.
$$
Since $\delta (x)+{1 \over p}\not \in \N$, $\Delta_{\tilde{S}} (h;u_1, \ldots ,u_{\tilde{n}};\tilde{Z})$
has no solvable vertex within its initial face
$\{ \mathbf{\tilde{x}}\in \R^{\tilde{n}}_{\geq 0} : \mid \mathbf{\tilde{x}} \mid =\delta (x)+{1 \over p}\}$.

Let $(u_1, \ldots , u_{\tilde{n}};\tilde{Z}_1)$ be well adapted coordinates at $\tilde{x}$. Without loss of generality,
it can be assumed that $\tilde{Z}_1=\tilde{Z} -\tilde{\theta}_1$ with
$\mathrm{ord}_{m_{\tilde{S}}}\tilde{\theta}_1 \geq \delta (x)+1$. By (\ref{eq2521}),  we get
\begin{equation}\label{eq2522}
\mathrm{in}_{m_{\tilde{S}}}\tilde{h} =\tilde{Z}_1^p - \sum_{\mid \mathbf{x}\mid =\delta (x)}\tilde{V}(\mathbf{x}) U^{p\mathbf{x}} +
G(U_1, \ldots , U_n) \in G(m_{\tilde{S}})[\tilde{Z}_1]
\end{equation}
and (2.ii) is proved. We have $i_0(\tilde{x})=p$, $\delta (\tilde{x})=\delta (x)+{1 \over p}$ and
$\epsilon (\tilde{x})=\epsilon (x)+1$. Finally, we have
$$
{\partial F_{p,\tilde{Z}_1} \over \partial U_j} =
 \sum_{\mid \mathbf{x}\mid =\delta (x)}{\partial \tilde{V}(\mathbf{x}) \over \partial \tilde{V}_j}U^{p\mathbf{x}}
 \in k(\tilde{x})[U_1, \ldots , U_n], \ n+ 1 \leq j \leq \tilde{n},
$$
so $V(F_{p,\tilde{Z}_1}, \tilde{E},m_{\tilde{S}})\neq 0$ and
$\omega (\tilde{x})=\epsilon (\tilde{x})-1 =\omega (x)$.
\end{proof}

\begin{rem} \label{stricthenselian}
Theorem \ref{omegageomreg} reduces computations of $\omega (x)$ to the case
where $S$ is strict Henselian, i.e. Henselian with separably algebraically closed residue field $S/m_S$
by changing $S$ to its strict Henselianization $\tilde{S}$, $\mathrm{dim}\tilde{S}=n=\mathrm{dim}S$.

Applying the theorem to a tower $\tilde{S}$ of smooth local base changes of the form
$S \subseteq S[Y]_{(m_S,Y^p -l)}$ with $l \in S$ a unit
with residue $\overline{l} \not \in (S/m_S)^p$ also reduces computations of $\omega (x)$ to the case
of an algebraically closed residue field for some $\tilde{S}$ with $\mathrm{dim}\tilde{S}>n=\mathrm{dim}S$.

The cone $\mathrm{Max}(x)$ and directrix $\mathrm{Dir}(x)$ have no such good behavior w.r.t.
regular local base changes.
\end{rem}

\subsection{Resolution when $\omega (x)=0$.}

In this section, we prove that the multiplicity of ${\cal X}$ can be reduced at any point $x$ such that
$(m(x),\omega (x))=(p,0)$. This is achieved by combinatorial blowing ups in a way which is similar
to the equal characteristic zero situation.

\smallskip  {Up to the end of this article, we will define a
 resolution algorithm which picks up local blowing ups centers in a way which is independent on the
choice of a valuation.  The word independent is defined below.}

\smallskip  {The total quotient ring  $L=\mathrm{Tot}(S[X]/(h))$ is a direct product of fields. By a valuation of $L$, we mean a valuation of one of these fields.}

\begin{defn}\label{indepseq}
Let $(S,h,E)$ be as before, $x \in {\cal X}$ and $L=\mathrm{Tot}(S[X]/(h))$. Suppose that  {to}
every valuation $\mu$ of $L$ centered at $x$, a composition of
local Hironaka-permissible blowing ups (Definition \ref{Hironakapermis})
\begin{equation}\label{eq2915}
       ({\cal X},x)=:({\cal X}_0,x_0) \leftarrow ({\cal X}_1,x_1) \leftarrow \cdots \leftarrow ({\cal X}_r,x_r)
\end{equation}
is associated, where $x_i \in {\cal X}_i$ is the center of $\mu$, $0 \leq i \leq r$.
The sequence (\ref{eq2915}) is said to be {\it independent} if the blowing up center
${\cal Y}_i \subset ({\cal X}_i,x_i)$ does not depend on the chosen valuation $\mu$
{\it having center  $x_i$ in ${\cal X}_I $}, $0 \leq i \leq r-1$.
\end{defn}

Let $(u_1,\ldots,u_n;Z)$ be well adapted coordinates at $x \in \eta^{-1}(m_S)$. If $\epsilon (x)>0$,
recall that $\eta^{-1}(m_S)=\{x\}$, $k(x)=S/m_S$, and that
$$
\mathrm{in}_{m_S} h =Z^p - G^{p-1}Z +F_{p,Z} \in G(m_S)[Z]= k(x)[U_1, \ldots ,U_n][Z]
$$
by (\ref{eq2551}). The initial form of $H(x)$ in $G(m_S)$ is denoted $H$ as before.

\begin{lem}\label{lemomegazero}
Assume that $m(x)=p$ and $\epsilon (x)=1$, where $\{x\}=\eta^{-1}(m_S)$. Let $(u_1,\ldots,u_n;Z)$ be
well adapted coordinates at $x \in \eta^{-1}(m_S)$. If
$$
H^{-1}F_{p,Z} \nsubseteq <U_1, \ldots ,U_e>,
$$
then $\omega (x)=0$.
\end{lem}

\begin{proof}
According to Definition \ref{defomega}, we must show that
$V(TF_{p,Z}, E,m_S)\neq 0$. Expand
$$
H^{-1}F_{p,Z} =<\sum_{j=1}^n \alpha_j U_j> \subseteq G(m_S)_1 , \ \alpha_j \in k(x).
$$
By assumption, we have $\alpha_{j_0} \neq 0$ for some $j_0$, $e+1 \leq j_0 \leq n$, so
\begin{equation}\label{eq2917}
0 \neq H^{-1}{\partial F_{p,Z} \over \partial U_{j_0}} \subseteq V(F_{p,Z}, E,m_S).
\end{equation}
If $i_0(x)=p$, we have $TF_{p,Z}=F_{p,Z}$. If $i_0(x)=p-1$, then $H^{-1}G^p=<U_{j_1}>$
for some $j_1$, $1 \leq j_1 \leq e$, by Theorem \ref{initform}(2). Comparing
with Definition \ref{defT}, we have $\mathbf{x}\in A \Longrightarrow p x_{j_1}>H_{j_1}$,
therefore $F_{p,Z} - TF_{p,Z}\in H U_{j_1}$.
So (\ref{eq2917}) implies that $V(TF_{p,Z}, E,m_S)\neq 0$.
\end{proof}

\begin{prop}\label{bupomegazero}
Assume that $(m(x),\omega (x))=(p,0)$, $\{x\}:=\eta^{-1}(m_S)$. Let ${\cal Y} \subset ({\cal X},x)$
be a {\it Hironaka-permissible} center w.r.t. $E$, $\pi : {\cal X}' \rightarrow ({\cal X},x)$
be the blowing up along ${\cal Y}$ and $x' \in \pi^{-1}(x)$.

\smallskip

If $W:=\eta({\cal Y})$ is an intersection of components of $E$ or if $\epsilon (y)=\epsilon (x)$,
then $(m(x'),\omega (x')  {)} \leq (p,0)$.
\end{prop}

\begin{proof}
According to Definition \ref{defomega}, there are two different cases to
consider:
\begin{itemize}
\item[(1)] $\epsilon (x)=0$;
\item[(2)] $\epsilon (x)=1$, $V(TF_{p,Z},E,m_S)\neq (0)$.
\end{itemize}

To begin with, we have $\delta (x)\geq 1$ by Proposition \ref{deltainv}(ii). Let $(u_1,\ldots,u_n;Z)$
be well adapted coordinates at $x$ with $I(W)=(\{u_j\}_{j\in J})$ for some subset $J \subseteq \{1, \ldots ,n\}$.
By Definition \ref{defepsilon}, we have:
\begin{equation}\label{eq2901}
\epsilon (x)=\min_{1 \leq i \leq p} \left \{{\mathrm{ord}_{m_S}(H(x)^{-i}f^p_{i,Z}) \over i}\right \}.
\end{equation}

\noindent {\it Case 1:} $\epsilon (x)=0$. By (\ref{eq2901}), we have
\begin{equation}\label{eq2904}
\left\{
  \begin{array}{ccc}
    H(x)^{-i}f^p_{i,Z} \hfill{} & \subseteq m_S,       & 1 \leq i < i_0(x) \\
    H(x)^{-i_0(x)}f^p_{i_0(x),Z}      & =S,            & \\
    H(x)^{-i}f^p_{i,Z}\hfill{}  & \subseteq S,   & i_0 (x) < i \leq p .\\
  \end{array}
\right.
\end{equation}

By Proposition \ref{Hironakastable}, there exists a commutative diagram
$$
\begin{array}{ccc}
  {\cal X}  & {\buildrel \pi  \over \longleftarrow} & {\cal X}' \\
  \downarrow &   & \downarrow \\
  \mathrm{Spec}S &  {\buildrel \sigma \over \longleftarrow} &  {\cal S}'\\
\end{array}
$$
where $\sigma: {\cal S}' \rightarrow \mathrm{Spec}S$ is the blowing up along $W$. Let
$$
\eta ': {\cal X}' \rightarrow {\cal S}', \ s':=\eta ' (x'), \ S':={\cal O}_{{\cal S}' ,s'},
\ E':=(\sigma^{-1}(E)_{\mathrm{red}})_{s'}.
$$

Since $W \subseteq E$, it can be assumed after possibly reordering coordinates that
$$
(J')_E:=\{2 , \ldots ,e_0\}, \ J=\{1 ,e_0+1, \ldots , n_0\}, \ 1 \leq e_0\leq e \leq n_0.
$$
Furthermore, it can be assumed that
$s' \in \mathrm{Spec}(S[u_{e_0+1}/u_1, \ldots ,u_{n_0}/u_1])$  or that
$s' \in \mathrm{Spec}(S[u_{1}/u_{n_0},u_{e_0+1}/u_{n_0}, \ldots ,u_{n_0-1}/u_{n_0}])$ with
$n_0 >e_0$.

\smallskip

We first prove the proposition when $s' \in \mathrm{Spec}(S[u_{e_0+1}/u_1, \ldots ,u_{n_0}/u_1])$.
Let
$$
h':=u_1^{-p}h={Z'}^p+f_{1,Z'}{Z'}^{p-1}+ \cdots +f_{p,Z'}\in S'[Z'],
$$
where $Z':=Z/u_1, \ f_{i,Z'}:=u_1^{-i}f_{i,Z}\in S'$ for $1 \leq i \leq p$. We have
\begin{equation}\label{eq2902}
E'=\mathrm{div}(u_1 \cdots u_{e_0}{u_{e_0+1} \over u_1} \cdots {u_e \over u_1})
\end{equation}
and $(S',h',E')$ satisfies both conditions {\bf (G)} and {\bf (E)} by Propositions \ref{SingX}
and \ref{Estable}. There exists an adapted r.s.p. of $S'$ of the form
$$
(u'_1:=u_1,  \ldots ,u'_{e_0}:=u_{e_0} , u'_{e_0+1}, \ldots , u'_{n'_0}, u'_{n_0+1}:=u_{n_0+1}, \ldots , u'_n :=u_n).
$$
Since we do not assume that $x'$ is a closed point, we have $e_0 \leq n'_0\leq n_0$ in general,
with
$$
n':=\mathrm{dim}S'= n-(n_0-n'_0).
$$
We emphasize that the number of irreducible
components $e'$ of $E'$ satisfies $e_0 \leq e' \leq e$ and that $e'\neq e$ in general because some of the
$u_j/u_1$ in (\ref{eq2902}) may be units. After reordering coordinates, we may also assume that
$$
E'=\mathrm{div}(u'_1 \cdots u'_{e'}) \ \mathrm{and} \ u'_j:=u_j/u_1, \ e_0+1 \leq e' \leq e.
$$

Since ${\cal Y}$ is Hironaka-permissible at $x$, we have (see Definition \ref{defH}):
$$
\mathrm{ord}_WH(x) =p \sum_{j \in J}d_j  \geq p.
$$
Therefore $I':=u_1^{-p}H(x) \subseteq S'$ and this ideal is monomial
in $(u'_1 , \ldots ,u'_{e'})$, i.e. $I'=:({u'_1}^{H'_1}\cdots {u'_{e'}}^{H'_{e'}})$.
We let:
$$
\mathbf{x}':=(H'_1/p, \ldots , H'_{e'}/p, 0 , \ldots ,0)\in {1 \over p}\N^{n'},
$$
where
\begin{equation}\label{eq2906}
H'_1=p(\sum_{j \in J}d_j -1) \ \mathrm{and} \ H'_j=H_j=pd_j,  \ 2 \leq j \leq e'.
\end{equation}
Then (\ref{eq2904}) gives:
\begin{equation}\label{eq2916}
\left\{
  \begin{array}{ccc}
    {I'}^{-i}f^p_{i,Z'} \hfill{} & \subseteq m_SS'      & 1 \leq i < i_0 (x)\\
    {I'}^{-i_0(x)}f^p_{i_0(x),Z'}      & =S '           &  \\
    {I'}^{-i}f^p_{i,Z'}\hfill{}  & \subseteq S'   & i_0 (x)< i \leq p .\\
  \end{array}
\right.
\end{equation}

This shows that
\begin{equation}\label{eq2905}
\Delta_{\hat{S'}}  (h';u'_1,\ldots ,u'_{n};Z')=\mathbf{x}'+ {\R}^{n'}_{\geq 0}.
\end{equation}

If $i_0 (x)<p$, or if $\sum_{j \in J_E}d_j \not \in \N$ or if
$d_{j'} \not \in \N$ for some $j'$, $2\leq j' \leq e'$, then $\mathbf{x}'$ is not solvable (Definition
\ref{defsolvable}) by (\ref{eq2905}),
hence $\Delta_{\hat{S'}}  (h';u'_1,\ldots ,u'_{n};Z')$ is minimal. Therefore we may compute
$\epsilon (x')$ from (\ref{eq2905}) and get $\epsilon (x')=0$, so the proposition is proved in this case.

\smallskip

If ($i_0 (x)=p$, $\sum_{j \in J_E}d_j  \in \N$ and
$d_{j'} \in \N$ for all $j'$, $2\leq j' \leq e'$), write $f_{p,Z}=\gamma u^{p\mathbf{x}}$,
$\gamma \in S$ a unit and $\mathbf{x}:=(d_1, \ldots ,d_e, 0 , \ldots ,0)\in {1 \over p}\N^n$. We have
\begin{equation}\label{eq2907}
\mathrm{in}_{\mathbf{x}'}h'={Z'}^p + \lambda (\prod_{j=e'+1}^e\lambda_j^{H_j}){U'}^{p\mathbf{x}'},
\end{equation}
where $\lambda \in k(x)$ (resp. $\lambda_j \in k(x')$) is the residue of $\gamma$ (resp. of
$u_j/u_1$). We let:
$$
\lambda ':=\lambda \prod_{j=e'+1}^e\lambda_j^{H_j}\in k(x'), \ \lambda ' \neq 0.
$$

If $\lambda ' \not \in k(x')^p$, then $\mathbf{x}'$ is not solvable and we also have $\epsilon (x')=0$.

If $\lambda ' \in k(x')^p$, let
$$
C':=\mathrm{Spec}\left ( {k(x)[Z,U_1, U_{e_0+1}, \ldots , U_e]\over (\overline{H})}\right ),
\ \overline{H}:=\mathrm{in}_{m_S}h=Z^p + \lambda \prod_{j=e'+1}^eU_j^{H_j}.
$$
We claim that the affine cone $C'$ is regular away from the torus
$$
\T := \A^{e-e_0+2}_{k(x)} \backslash V(Z\prod_{j \in J_E}U_j ).
$$
To see this, let $(\lambda_l)_{l \in \Lambda}$ be an absolute $p$-basis of $k(x)$. By Zariski's
Jacobian criterion \cite{Ma} Theorem 30.5, the ideal of the singular locus of $C'$ is:
$$
I(\mathrm{Sing}C')=\left (\overline{H},\{{\partial \overline{H} \over \partial \lambda_l}\}_{l\in \Lambda},
\{{\partial \overline{H} \over \partial U_j}\}_{e'+1\leq j \leq e}\right ).
$$
If $d_j \not \in \N$ for some $j$, $e'+1\leq j \leq e$, then ${\partial \overline{H} \over \partial U_j}$
does not vanish on $\T$. Otherwise, we have $\lambda \not \in k(x)^p$ because $\mathbf{x}$ is a
vertex of $\Delta_{S}(u_1,\ldots,u_n;Z)$ and is not solvable.
Therefore ${\partial \overline{H} \over \partial \lambda_l}$ does not vanish on $\T $
for any $l\in \Lambda$ such that ${\partial \lambda \over \partial \lambda_l}\neq 0$
and the claim is proved. We deduce that there exists a unit $l'\in S'$ such that
$$
v':={l'}^p +\gamma \prod_{j=e'+1}^e\left ({u_j \over u_1}\right )^{H_j}
$$
is a regular parameter of $S'$ transverse to
$$
E'_1:=\mathrm{div}(u'_1 \cdots u'_{e'}u'_{n_0+1} \cdots u'_{n'}), \ E'_1 \supseteq E'.
$$
We may thus take $u'_{e'+1}:=v'$ in our r.s.p. of $S'$ adapted to $E'$. Let
$Z'_1:=Z'-l'{u'}^{p\mathbf{x}'}$, so the polyhedron $\Delta_{S'}  (h';u'_1,\ldots ,u'_{n};Z'_1)$
has a vertex
\begin{equation}\label{eq2908}
\mathbf{x}'_1:=(H'_1/p, \ldots , H'_{e'}/p, 1/p,0 , \ldots ,0)\in {1 \over p}\N^{n'}
\end{equation}
which is not solvable, since $\mathbf{x}'_1 \not \in \N^{n'}$. Let $Z'_2:=Z'_1 - \theta '$,
$\theta ' \in S'$, be such that $\Delta_{S'}  (h';u'_1,\ldots ,u'_{n};Z'_2)$
is minimal. We deduce from (\ref{eq2916}) and (\ref{eq2908}) that
$$
H(x')=({u'}^{p\mathbf{x}'}), \ \epsilon (x')=1 \ \mathrm{and} \ {H'}^{-1}F_{p,Z'_2}
\nsubseteq <U'_1, \ldots ,U'_{e'}>.
$$
We get $m(x')=1$ if $\mathbf{x}'=\mathbf{0}$, and  $(m(x'),\omega (x'))=(p,0)$ otherwise by Lemma  \ref{lemomegazero}
as required.

\smallskip

If $s' \in \mathrm{Spec}(S[u_{1}/u_{n_0},u_{e_0+1}/u_{n_0}, \ldots ,u_{n_0-1}/u_{n_0}])$, it
can be furthermore assumed that $s' \not \in \mathrm{Spec}(S[u_{e_0+1}/u_1, \ldots ,u_{n_0}/u_1])$, i.e.
$u_{j}/u_{n_0}$ is {\it not} a unit in $S'$ for $j\in J_E$. The proof is now a simpler
variation of the above one:  (\ref{eq2902}) is replaced by
$$
E'=\mathrm{div}({u_{1} \over u_{n_0}}u_2 \cdots u_{e_0}{u_{e_0+1} \over u_{n_0}} \cdots {u_e \over u_{n_0}}u_{n_0}).
$$
The polyhedron $\Delta_{S'}  (h';u'_1,\ldots ,u'_{n};Z')$ in (\ref{eq2905}) is minimal except if
($d_j \in \N$ for each $j$, $1 \leq j \leq e$, and $\lambda \in k(x')^p$) with notations as above.
We have $\epsilon (x')=0$ (resp. $\epsilon (x')=1$) in the former (resp. in the latter) situation.
This concludes the proof in case 1.\\

\noindent {\it Case 2:} $\epsilon (x)=1$. The proof runs parallel to
that in case 1 and we only indicate the necessary changes. By assumption,
$W$ is an intersection of components of $E$ (case 2a) or $\epsilon (y)=\epsilon (x)=1$
(case 2b).\\

To begin with, let $v \in S$ be such that $H(x)^{-1}f_{p,Z}=(v)$. By assumption, we
have $V(TF_{p,Z},E,m_S)\neq (0)$, so $v$ is transverse to $E$.

In case 2a, we may assume that $(u_1, \ldots ,u_e,v,u_{e+2}, \ldots ,u_n)$
is an adapted r.s.p. of $S$ after renumbering variables. Since
$\mathbf{x}_0:=(d_1, \ldots ,d_e, {1 \over p} , \ldots ,0)\not \in \N^n$ is the unique
vertex of $\Delta_{S} (h; u_1, \ldots ,u_e,v,u_{e+2}, \ldots ,u_n ;Z)$ induced
by $f_{p,Z}$, this polyhedron has no solvable vertex. In other terms, it can
be assumed that $v=u_{e+1}$.

In case 2b, Proposition \ref{Deltaalg} implies that $v \in I(W)$, so
$(u_1, \ldots ,u_e,v)$ can be completed to an adapted  r.s.p. of $S$  such that
$I(W)=(\{u_j\}_{j\in J})$ for some subset $J \subseteq \{1, \ldots ,n\}$. The polyhedron
$\Delta_{S} (h; u_1, \ldots ,u_e,v,u_{e+2}, \ldots ,u_n ;Z)$ has no solvable vertex
either and it can also be assumed that $v=u_{e+1}$.

We remark in both cases 2a and 2b that, if $\Delta_{S} (h; u_1, \ldots  ,u_n ;Z)$ has
a vertex distinct from $\mathbf{x}_0$, then it has exactly two vertices: this follows
from Theorem \ref{initform}(2), the other vertex being then given by
\begin{equation}\label{eq2903}
\mathbf{x}_1:=({D_1 \over p(p-1)}, \ldots ,{D_e \over p(p-1)}, 0 , \ldots ,0),
\ (\mathrm{Disc}_Z(h))=:(u_1^{D_1} \cdots u_e^{D_e}).
\end{equation}

After blowing up, we obtain a $(S',h',E')$ again satisfying conditions {\bf (G)} and {\bf (E)}.

In case 2a, there exists an adapted r.s.p. of $S'$ of the form
$$
(u'_1:=u_1,  \ldots ,u'_{e_0}:=u_{e_0} , u'_{e_0+1}, \ldots , u'_{e_1}, u'_{e+1}:=u_{e+1}, \ldots , u'_n :=u_n),
$$
with $J=\{1, e_0+1, \ldots ,e\}$ and $E'=\mathrm{div}(u'_1 \cdots u'_{e'})$ after reordering variables,
$1 \leq e_0 \leq e' \leq e_1 \leq e$. Then $\Delta_{S'} (h'; u'_1, \ldots ,u'_n ;Z')$ has again a vertex
$$
\mathbf{x}':=(H'_1/p, \ldots , H'_{e'}/p, 0 , \ldots ,0, 1/p, 0 , \ldots ,0)\not \in \N^{n-(e-e_1)},
$$
thus  $\mathbf{x}'$ is not solvable.  We deduce that $\epsilon (x')\leq 1$ and $\omega (x')=0$
follows from Lemma \ref{lemomegazero} if $(m(x'),\epsilon (x'))=(p,1)$.

\smallskip

In case 2b,  it can be assumed after reordering variables that
$$
(J')_E:=\{2 , \ldots ,e_0\}, \ J=\{1 ,e_0+1, \ldots , n_0\}, \ 1 \leq e_0\leq e, \ e+1 \leq n_0.
$$
We let $u'_{j'}:=u_{j'}$ for $j'\in J'$ and consider three distinct situations depending on $x'$,
up to reordering coordinates:

\begin{itemize}
  \item [(1)] $s' \in \mathrm{Spec}(S[u_{e_0+1}/u_1, \ldots ,u_{n_0}/u_1])$ and $u_{e+1}/u_1 \in m_{S'}$. We
  may complete the family $(\{u_{j'}\}_{j'\in J'})$ to an adapted  r.s.p. of $S'$ by adding
  $$
  (u'_1:=u_1, u'_{e_0+1}, \ldots ,u'_{e_1}, u'_{e_1+1}:=u_{e+1}/u_1), \ n':=\mathrm{dim}S'=n-(n_0-e_1).
  $$
  Then $\Delta_{\hat{S'}} (h'; u'_1, \ldots ,u'_{n} ;Z')$ has a vertex
$$
\mathbf{x}':=(H'_1/p, \ldots , H'_{e'}/p,  1/p, 0 , \ldots ,0)\not \in \N^{n'},
$$
thus  $\mathbf{x}'$ is not solvable. We conclude that $\epsilon (x')\leq 1$ and that $\omega (x')=0$
if $(m(x'),\epsilon (x'))=(p,1)$ by Lemma \ref{lemomegazero}.
  \item [(2)] $s' \in \mathrm{Spec}(S[u_{1}/u_{n_0},u_{e_0+1}/u_{n_0}, \ldots ,u_{n_0-1}/u_{n_0}])$
  and $u_{e+1}/u_{n_0} \in m_{S'}$, where $n_0>e+1$. After dealing with (1), we may assume furthermore that
  $u_j/u_{n_0}\in m_{S'}$, $j\in J_E$. We complete the family $(\{u_{j'}\}_{j'\in J'})$
  to an adapted  r.s.p. of $S'$ by adding
  $$
  (u'_{e_0+1}:=u_{e_0+1}/u_{n_0}, \ldots ,u'_{e+1}:=u_{e+1}/u_{n_0}, u'_{n_1}, \ldots ,u'_{n_0-1},u'_{n_0}:=u_{n_0}),
  $$
  with $n':=\mathrm{dim}S'=n-(n_1-e -2)$. We conclude as in (1).
  \item [(3)] $I(W)S' = (u_{e+1})$. We complete the family $(\{u_{j'}\}_{j'\in J'})$
  to an adapted  r.s.p. of $S'$ by adding
  $$
  (u'_1:=u_{e+1}, u'_{e_0+1}, \ldots ,u'_{n_1}), \ n':=\mathrm{dim}S'=n-(n_0-n_1).
  $$
  Let $E'=:\mathrm{div}(u'_1 \cdots u'_{e'})$ and consider two situations as in case 1:

  If ${1 \over p}+\sum_{j \in J_E}d_j \not \in \N$ or if $d_{j'} \not \in \N$ for some $j'$, $2\leq j' \leq e'$, then
the polyhedron $\Delta_{\hat{S'}}  (h';u'_1,\ldots ,u'_{n};Z')$ is minimal and we have $\epsilon (x')=0$.

  If (${1 \over p}+\sum_{j \in J_E}d_j  \in \N$ and $d_{j'} \in \N$ for every $j'$, $2\leq j' \leq e'$), the
initial form polynomial $\mathrm{in}_{\mathbf{x}'}h'$ has the form
$$
\mathrm{in}_{\mathbf{x}'}h'={Z'}^p -\mu^{p-1} {U'}^{(p-1)\mathbf{x}'}Z'+
\lambda (\prod_{j=e'+1}^e\lambda_j^{H_j}){U'}^{p\mathbf{x}'},
$$
where $\lambda \in k(x)$ (resp. $\lambda_j \in k(x')$) is the residue of $\gamma$ (resp. of
$u_j/u_{e+1}$), {\it vid.} (\ref{eq2907}). We have $\mu \neq 0$ in the above formula precisely if
$$
U^{p(\mathbf{x}_1 -\mathbf{x}_0)}=U_{j_0}/U_{e+1}, \ u_{j_0}/u_{e+1} \in S' \ \mathrm{a} \ \mathrm{unit}
$$
for some $j_0$, $e_0+1 \leq j_0 \leq e$ with notations as in (\ref{eq2903}).
Then $\mu^{p-1}$ is the residue in $k(x')$ of
$$
\gamma_{p-1,Z}\prod_{j=e'+1}^e\left ({u_j \over u_{e+1}}\right )^{A_{p-1},j}
$$
with notations as in Theorem \ref{initform}(2). The end of the
proof goes along as in case 1.
\end{itemize}
This completes the proof of (3), hence the proof
of the proposition in case 2.
\end{proof}

\begin{rem}
This proposition is a lighter version of Theorem \ref{bupthm} where it is assumed
that $\omega (x)>0$ and that the blowing up centers are permissible of the first
or second kind (Definitions \ref{deffirstkind} and \ref{defsecondkind} below).
\end{rem}

\begin{thm}\label{omegazero}
Assume that $(m(x),\omega (x))=(p,0)$, where $\{x\}=\eta^{-1}(m_S)$. For every valuation $\mu$
of $L=\mathrm{Tot}(S[X]/(h))$ centered at $x$, there exists a finite and
independent composition of local Hironaka-permissible blowing ups
(\ref{eq2915}) such that $m(x_r)<p$.
\end{thm}

\begin{proof}
We will produce a Hironaka-permissible center ${\cal Y} \subset ({\cal X},x)$
w.r.t. $E$ satisfying the assumptions of Proposition \ref{bupomegazero} and such that the following
holds:\\

\noindent (*) let $\pi : \ {\cal X}' \rightarrow ({\cal X},x)$ be the blowing up along ${\cal Y}$ and
$x' \in \pi^{-1}(x)$. Then
$$
 \delta (x') < \delta (x).
$$

Applying Proposition \ref{bupomegazero}, the center $x_1 \in {\cal X}'$ of a given valuation $\mu$
again satisfies the assumptions of the theorem if $m(x_1)=p$. Iterating, any finite sequence
(\ref{eq2915}) induces a sequence
$$
\delta (x_r) < \delta (x_{r-1}) < \cdots <\delta (x)
$$
provided that $m(x_i)=p$, $ 1 \leq i \leq r-1$. Since $\delta (x_i)\in {1 \over p}\N$, we have
$\delta (x_r)<1$ for some $r \geq 1$, hence $m(x_r)<p$ by Proposition \ref{deltainv}(2), so the
theorem follows from claim (*). In order to construct ${\cal Y}$ with the required properties,
we consider two cases as in the proof of Proposition \ref{bupomegazero}. \\

\noindent {\it Case 1:} $\epsilon (x)=0$. We have $\delta (x)=\sum_{j =1}^ed_j \geq 1$.
Therefore there exists a  {(not necessarily unique)} subset
$$
J \subseteq \{1, \ldots ,e\}, \ \sum_{j \in J}d_j \geq 1,
$$
with smaller possible number of elements among all subsets of $\{1, \ldots ,e\}$ with this property. Let
$W:=V(\{u_j\}_{j \in J})\subset \mathrm{Spec}S$ and remark that
$$
\mathrm{ord}_WH(x) =p \sum_{j \in J}d_j  \geq p.
$$

Hence ${\cal Y}:=\eta^{-1}(W)=V(Z,\{u_j\}_{j \in J})$ is Hironaka-permissible w.r.t. $E$ and $W$ is an
intersection of components of $E$. By (\ref{eq2906}), we have
\begin{equation}\label{eq2910}
\mathrm{ord}_{m_{S'}}H(x')\leq p(\delta (x)+\sum_{j\in J \backslash \{j_0\}}d_j-1),
\end{equation}
where $I(W)S'=(u_{j_0})$. The minimality property required of $J$ implies that
\begin{equation}\label{eq2911}
\sum_{j \in J \backslash \{j_1\}}d_{j}<1 \ \mathrm{for} \ \mathrm{every} \  j_1\in J \ (\mathrm{so}
\ \sum_{j \in J}d_j< 2 \ \mathrm{if} \ \mid J \mid  \geq 2).
\end{equation}

If $\epsilon (x')=0$, we deduce from (\ref{eq2910}) that
$$
p\delta (x')=\mathrm{ord}_{m_{S'}}H(x')< p\delta (x)
$$
as required in (*). Note that if $\mid J \mid =1$, we have $\lambda =\lambda '$ in (\ref{eq2907})
and $S=S'$, hence $\lambda '\not \in k(x')^p=k(x)^p$. Since $\epsilon (x')=0$ in
this situation, we may now assume that $\mid J \mid  \geq 2$.

\smallskip

If $\epsilon (x')= 1$, we are in the situation discussed in (\ref{eq2908}).
We may then take $j_0=1$, $E'=\mathrm{div}(u'_1 \cdots u'_{e'})$ and have
$$
\sum_{j \in J}d_j\in \N, \ d_j \in \N \ \mathrm{for} \ 2 \leq j \leq e'.
$$
By (\ref{eq2911}), we have $\sum_{j \in J}d_j=1$, $d_j=0$ for $2 \leq j \leq e'$, so
$H(x')=(1)$ and $m(x')=1$. This concludes the proof in case 1.\\

\noindent {\it Case 2:} $\epsilon (x)=1$. We have $\delta (x)={1 \over p} +\sum_{j =1}^ed_j \geq 1$.

If $\delta (x)>1$, there exists a subset
$$
J \subseteq \{1, \ldots ,e\}, \ \sum_{j \in J}d_j \geq 1,
$$
with smaller possible number of elements among all subsets of $\{1, \ldots ,e\}$ with this property
as in case 1 and we also let $W:=V(\{u_j\}_{j \in J})\subset \mathrm{Spec}S$. The proof goes along as in case 1,
with
$$
p\delta (x')-p\delta (x)\leq \mathrm{ord}_{m_{S'}}H(x') -\mathrm{ord}_{m_{S}}H(x)<0.
$$

If $\delta (x)=1$, we may assume that $H(x)^{-1}f_{p,Z}=(u_{e+1})$ and that (\ref{eq2903}) holds
if $\Delta_{S} (h; u_1, \ldots  ,u_n ;Z)$ has more than one vertex. In this case, this
polyhedron has exactly two vertices and we have
$$
H(x)^{-(p-1)}f_{p-1,Z}^p=(u_{j_0})^{p-1} \ \mathrm{for} \ \mathrm{some} \ j_0, \  1 \leq j_0 \leq e
$$
by Theorem \ref{initform}(2). We deduce that
\begin{equation}\label{eq2912}
H(x)^{-i}f_{i,Z}^p \subseteq (u_{j_0},u_{e+1})^i, \ 1 \leq i \leq p
\end{equation}
by definition of $\Delta_{S} (h; u_1, \ldots  ,u_n ;Z)$. We let
$J:=\{j : d_j>0\}\cup \{e+1\}$ and
$$
W:=V(\{u_j\}_{j \in J})\subset \mathrm{Spec}S, \ {\cal Y}:=\eta^{-1}(W)=V(Z,\{u_j\}_{j \in J}).
$$
We have $\mathrm{ord}_WH(x) =p$, so ${\cal Y}$ is Hironaka-permissible w.r.t. $E$. Since
$H(x)^{-1}f_{p,Z}=(u_{e+1})$, we have $\epsilon (y)=\epsilon (x)=1$ by (\ref{eq2912}), where
$y \in {\cal X}$ is the generic point of ${\cal Y}$. Thus Proposition \ref{bupomegazero}
applies and gives $m(x')\leq p-1$ under either assumption (1)(2) or (3) in the proof of
Proposition \ref{bupomegazero}.
\end{proof}

\section{Permissible blowing ups.}

\subsection{Blowing ups of the first and second kind.}

In this section, we introduce a notion of permissible blowing up
which is well behaved w.r.t. our main resolution invariant $y \mapsto \iota (y)$
on ${\cal X}$. {\it We assume that}
$$
m (x)=p, \ \{x\}=\eta^{-1}(m_S) \ \mathrm{and} \  \omega (x)>0
$$
{\it in what follows} since Theorem \ref{omegazero}  {takes care of} the case $\omega (x)=0$.
 {Two different kinds of permissible blowing ups are required. Permissibility behaves well with respect to regular base change (Theorem \ref{geomregpermis}). A permissible center is permissible on a nonempty Zariski open set   (Theorem \ref{Zariskiopen}). None of  these is  true  for permissible centers of a fixed  kind. Furhermore, by  Example \ref{examsecondkind}
we need both kinds of permissible blowing ups.}

\begin{defn}\label{deffirstkind}
Let ${\cal Y} \subset {\cal X}$ be an integral closed subscheme with generic point $y$.
We say that ${\cal Y}$ is {\it permissible of the first kind} \index{permissible of the first kind, Definition~\ref{deffirstkind}}at $x $
if $m(y)=m(x)=p$ and the following conditions hold:
\begin{itemize}
    \item [(i)] ${\cal Y}$ is Hironaka-permissible w.r.t. $E$ at $x$ (Definition \ref{HironakapermisE});
    \item [(ii)] $\epsilon(y)=\epsilon(x)$.
\end{itemize}
\end{defn}

If $y \in {\cal X}$ satisfies $m(y)=p$, it follows from
the definition that ${\cal Y}:=\overline{\{y\}}$ is permissible of the first kind at $y$.
It also follows from (ii) that a permissible center of the first kind
has codimension at least two in ${\cal X}$. \\

The main result of this chapter (Theorem \ref{bupthm} below) will require comparing the
initial form polynomials $\mathrm{in}_{W}h $ and $\mathrm{in}_{m_S} h$. We keep notations as in section 2.4:
given well adapted coordinates $(u_1,\ldots ,u_n;Z)$ at $x$, we let
\begin{equation}\label{eq71}
    W:=\eta ({\cal Y}), \  I(W)=( \{u_j\}_{j\in J}).
\end{equation}
We denote:
$$
\mathrm{in}_{W}h =Z^p + \sum_{i=1}^p{F_{i,Z,W}}Z^{p-i} \in G(W)[Z]
$$
and  {(Theorem \ref{initform}} since $\epsilon (x)>0$)
$$
\mathrm{in}_{m_S} h =Z^p - G^{p-1}Z +F_{p,Z} \in G(m_S)[Z].
$$
There are associated homogeneous submodules
\begin{equation}\label{eq:defH}
H_W \subseteq G(W)_{d_W} \ \mathrm{(resp.} \  H:=H_{m_S} \subseteq G(W)_{d}\mathrm{)}
\end{equation}
by (\ref{eq2441}), with
$$
d_W:=\sum_{j \in J_E}H_j , \ d=\sum_{j=1}^eH_j,
$$
 {where $J_E:=J \cap \{1,\ldots ,e\}$.}

\smallskip

A word of caution is required at this point: formula (\ref{eq2441}) {\it defines} the monomial ideal $H_W$
which is the {\it initial form} of $H(x)$ in $G(W)$ and is  different in general from
the ideal $H(\Xi)$ associated to the triple
$$
(G(W)_\Xi ,\mathrm{in}_{W}h ,E_W), \ \Xi :=(\{U_j\}_{j\in J}) + m_{\cal{O}_W}.
$$
 {For an example, let $h:=Z^p+u_1^a(u_2^b+u_1^{b+1})$, $ab>0$, $W=m_S$, $E=$div$(u_1u_2)$, we have $H_W=U_1^a$, $H(\Xi)=U_1^aU_2^b$.}

\smallskip

Corresponding to the above choice for $H_W$ (resp. to $H$), there are associated $\cal{O}_W$-submodules
$$
V(F_{p,Z,W},E,W)\subseteq G(W)_{\epsilon (y)-1}, \ J(F_{p,Z,W},E,W)\subseteq \widehat{G(W)}_{\epsilon (y)}
$$
(resp. $k(x)$-vector subspaces
$$
V(F_{p,Z},E,m_S)\subseteq G(m_S)_{\epsilon (x)-1}, \ J(F_{p,Z},E,m_S)\subseteq G(m_S)_{\epsilon (x)})
$$
given by (\ref{eq244}). \\

\begin{nota}\label{Fbar}
We first recall notations and definitions from section 2.4. We denote
$$
J_E:=J \cap \{1,\ldots ,e\}, \ J':=\{1,\ldots ,n\} \backslash J \ \mathrm{and} \
(J')_E:=\{1,\ldots ,e\} \backslash J_E.
$$
The image $\overline{m}_S$ of $m_S$ in $\cal{O}_W$ has regular parameters $(\overline{u}_{j'})_{j' \in J'}$,
the respective residues of the corresponding parameters of $S$.

\smallskip

Let now $d \in \N$ be fixed and
$$
F =\sum_{\mid \mathbf{a} \mid =d}\hat{f}_\mathbf{a}U^{\mathbf{a}}\in \widehat{G(W)}_d=\widehat{\cal{O}_W}[\{U_j\}_{j\in J}]_d.
$$
Note that $\mathrm{gr}_{\overline{m}_S}\widehat{G(W)}_d\simeq \mathrm{gr}_{\overline{m}_S}G(W)_d$ and
that it has a structure of graded $\mathrm{gr}_{\overline{m}_S}\cal{O}_W$-module.
For any $d_0 \leq \min_\mathbf{a}\{\mathrm{ord}_{\overline{m}_S}\hat{f}_\mathbf{a}\}$,
$F$ has an initial form in $\mathrm{gr}_{\overline{m}_S}G(W)_d$ by taking
\begin{equation}\label{eq2627}
\overline{F}:=\sum_{\mid \mathbf{a} \mid =d}(\mathrm{cl}_{d_0}\hat{f}_\mathbf{a})U^{\mathbf{a}}
\in (\mathrm{gr}_{\overline{m}_S}G(W)_d)_{d_0}.
\end{equation}
This notation requires specifying $d_0$ to avoid ambiguity. We extend the  notation to
homogeneous submodules $M \subseteq \widehat{G(W)}_d$ as follows:
$$
\overline{M}:=< \overline{F}, \ F \in M> \subseteq (\mathrm{gr}_{\overline{m}_S}G(W)_d)_{d_0}
$$
for fixed $d_0 \leq \min\{d_0 (F), F \in M\}$ with obvious notations. For fixed $d,d_0$,
there is an inclusion of $S/m_S$-vector spaces:
\begin{equation}\label{eq2626}
(\mathrm{gr}_{\overline{m}_S}G(W)_d)_{d_0} \subset {G(m_S)_{d+d_0} \over <(\{U_j\}_{j\in J})^{d+1}\cap G(m_S)_{d+d_0}>}.
\end{equation}
\end{nota}

\begin{prop}\label{firstkind}
Let ${\cal Y}$ be permissible of the first kind at $x\in {\cal Y}$. Then
for any well adapted coordinates $(u_1,\ldots ,u_n;Z)$ at $x$ such that
$I(W)=(\{u_j\}_{j\in J})$, the initial form $\mathrm{in}_{m_S}h\in G(m_S)[Z]$ satisfies
$$
H^{-1}<G^p, F_{p,Z}>\subseteq k(x)[\{U_j\}_{j\in J}]_{\epsilon (x)},
$$
 {with notation as in \eqref{eq:defH}.}
\end{prop}

\begin{proof}The existence of well adapted coordinates
$(u_1,\ldots ,u_n;Z)$ such that $I(W)=( \{u_j\}_{j\in J})$ follows from  Proposition \ref{Deltaalg}.
This theorem furthermore implies that the polyhedron
\begin{equation}\label{eq262}
    \Delta_{\hat{S}}(h;\{u_j\}_{j \in J};Z)=
\mathrm{pr}_J(\Delta_{S}(h;u_1,\ldots,u_n;Z))\ \mathrm{is} \
\mathrm{minimal},
\end{equation}
where $\mathrm{pr}_J: \ \R^n \rightarrow \R^J$ denotes the projection on the $(u_j)_{j \in
J}$-space.

By (ii) of Definition \ref{deffirstkind}, we have $\epsilon (x)=\epsilon (y)$. Therefore
$$
H^{-i}F^p_{i,Z} =\mathrm{cl}_{0}(H_W^{-i}F^p_{i,Z,W}) \subseteq G(m_S)_{i\epsilon (x)}=k(x) [U_1, \ldots ,U_n]_{i\epsilon (x)}
$$
is simply the reduction of $H_W^{-i}F^p_{i,Z,W}$ modulo $\overline{m}_S$ for  $1 \leq i \leq p$, i.e. taking
$d_0=0$ in Notation \ref{Fbar}, via the inclusion (\ref{eq2626})
$$
k(x)[\{U_j\}_{j\in J}]_{i\epsilon (y)}\simeq (\mathrm{gr}_{\overline{m}_S}G(W)_{i\epsilon (y)})_{0}
\subset G(m_S)_{i\epsilon (y)} \simeq k(x) [U_1, \ldots ,U_n]_{i\epsilon (x)}.
$$
We get respectively $(H^{-1}G^p)^{p-1}$, $(H^{-1}F_{p,Z})^p$ for $i=p-1,p$ and
this completes the proof.
\end{proof}

The following corollary will be required in the proof of the blowing up theorem below. The adapted
cone $\mathrm{Max}(x) \subseteq G(m_S)$ is defined in Definition \ref{deftauprime}.

\begin{cor}\label{Cmaxfibre}
With notations as above, let ${\cal Y}$ be permissible of the first kind at $x$.
The defining ideal $\mathrm{IMax}(x)\subseteq G(m_S)$ of $\mathrm{Max}(x)$ satisfies
$$
\mathrm{IMax}(x)= (\mathrm{IMax}(x)\cap k(x)[\{U_j\}_{j\in J}])G(m_S).
$$
\end{cor}

\begin{proof}
This follows from Proposition \ref{firstkind}
and Definition \ref{deftauprime}. Note that the truncation
operator $T$ used in the Definition of $\mathrm{Max}(x)$ does not affect the
conclusion of the corollary since it is obvious from the definitions that:
$$
V(F_{p,Z},E,m_S) \subseteq k(x)[\{U_j\}_{j\in J}]_{\epsilon (x)-1} \Rightarrow
V(TF_{p,Z},E,m_S) \subseteq k(x)[\{U_j\}_{j\in J}]_{\epsilon (x)-1}.
$$
The same implication holds for $J(F_{p,Z},E,m_S)$ and $J(TF_{p,Z},E,m_S)$.
\end{proof}

We now  define a second kind of permissible blowing up.

\begin{defn}\label{defsecondkind}
Let ${\cal Y} \subset {\cal X}$ be an integral closed subscheme with generic point $y$.
We say that ${\cal Y}$ is {\it permissible of the second kind} \index{permissible of the second kind, Definition~\ref{defsecondkind}} at $x$
if $m(y)=m(x)=p$ and the following conditions hold:
\begin{itemize}
    \item [(i)] ${\cal Y}$ is Hironaka-permissible w.r.t. $E$ at $x$ (Definition \ref{HironakapermisE});
    \item [(ii)] $\epsilon(y) =\epsilon (x)-1$ and $i_0(y) \leq i_0(x)$;
    \item [(iii)] $\overline{J}(F_{p,Z,W},E,W):=\mathrm{cl}_{0}J(F_{p,Z,W},E,W)\neq 0$.
\end{itemize}
\end{defn}

The following important example constructs a threefold ${\cal X}$ such that every resolution of singularities
$\tilde{{\cal X}}\rightarrow {\cal X}$ which is a composition of Hironaka-permissible blowing ups does actually
involve blowing up a permissible curve of the second kind.

\smallskip

\begin{exam}\label{examsecondkind}
Let $k$ be a perfect field of characteristic $p>0$, $A:=k[u_1,u_2,u_3]$,  {$P\in k[T] \backslash k[T^p]$} and take
$$
h:=Z^p + P(u_3)u_2^p +u_1^{p+1}\in A[Z], \ E:=\mathrm{div}(u_1).
$$
Let ${\cal Y}:=V(Z,u_1,u_2)\subseteq \mathrm{Sing}_p{\cal X}$ with generic point $y$. Let
$\pi : \ \tilde{{\cal X}}\rightarrow {\cal X}$ be any composition of Hironaka-permissible blowing ups
with $\tilde{{\cal X}}$ regular. Since $y$ is an isolated point of $\mathrm{Sing}_p{\cal X}$,
the map $\pi$ factors through the blowing up $\pi_0$ along ${\cal Y}$ above $y$. Define a nonempty
Zariski open subset ${\cal U}\subseteq {\cal Y}$ by:
$$
x\in {\cal U} \Leftrightarrow
\left\{
  \begin{array}{c}
    \pi \ \mathrm{factors} \ \mathrm{through} \ \pi_0 \ \mathrm{above} \ x   \\
      \\
    \mathrm{ord}_x  {{\partial P \over \partial T}(\overline{u}_3)=0} \hfill{}\\
  \end{array}
\right.
.
$$
For $x\in {\cal U}$, there exist well adapted coordinates $(u_1,u_2,v_x;Z_x:=Z - \gamma_xu_2)$ at $x$,
$\gamma_x \in A_{\eta(x)}$ a unit such that
$$
h=Z_x^p + v_xu_2^p+u_1^{p+1}\in A_{\eta(x)}[Z_x].
$$
Then ${\cal Y}$ is permissible of the second kind at every $x\in {\cal U}$ since
$$
\overline{J}(F_{p,Z_x,W},E,W)={\partial F_{p,Z_x,W} \over \partial \overline{v}_x}=U_2^p\neq 0,
\ F_{p,Z_x,W}=\overline{v}_xU_2^p \in G(W)_p
$$
with notations as in Definition \ref{defsecondkind}(iii). This is dealt with in the course of the proof of Theorem
\ref{luthm} in Proposition \ref{kappa2gamma0}  when applying Lemma \ref{kappa2bupcurve}
($\kappa (x)=2$ in this example,  {\it cf.} Definition \ref{defkappa}).

\smallskip

When $n=3$, permissible blowing ups of the second kind only occur in Propositions \ref{kappa2gamma0}
and \ref{kappa2fin10} ($\kappa (x)=2$).
\end{exam}

\begin{prop}\label{secondkind}
Let ${\cal Y}$ be permissible of the second kind at $x$. For any well adapted coordinates
$(u_1,\ldots ,u_n;Z)$ at $x$ such that $I(W)=(\{u_j\}_{j\in J})$,
the initial form $\mathrm{in}_{m_S}h  \in G(m_S)[Z]$ satisfies
\begin{equation}\label{eq2628}
\left\{
  \begin{array}{ccc}
    H^{-1}G^p & \subseteq & U_{j_0}k(x)[\{U_j\}_{j\in J}]_{\epsilon (y)}  \ \mathrm{for} \ \mathrm{some} \ j_0 \in (J')_E \hfill{}  \\
     & & \\
    H^{-1}F_{p,Z} & = & <\sum_{j\in J'}U_{j'}\Phi_{j'}(\{U_j\}_{j\in J})+ \Psi(\{U_j\}_{j\in J})>  \subseteq G(m_S)_{\epsilon (x)} \\
  \end{array}
\right.
.
\end{equation}
with $\Phi_{j'} \neq 0$ for some $j' \in J' \backslash (J')_E$. In particular
$\epsilon (y)=\omega (x)$.
\end{prop}

\begin{proof}
We argue as in the proof of Proposition \ref{firstkind} and build up from
(\ref{eq262}). By (ii) of Definition \ref{defsecondkind}, we have $\epsilon (x)=\epsilon (y)+1$.
Therefore
$$
\mathrm{cl}_{0}(H_W^{-i}F^p_{i,Z,W}) =0, \ 1 \leq i \leq p.
$$
This shows that $H_W^{-i}F^p_{i,Z,W} \subseteq \overline{m}_S \cal{O}_W[\{U_j\}_{j\in J_E}]_{i\epsilon (y)}$.
We have $\epsilon (y)>0$, so $F_{i,Z,W}=0$, $1 \leq i \leq p-2$ by Theorem \ref{initform}.
For $i=p-1$, we have $-F_{p-1,Z,W}=G_W^{p-1}$ for some $G_W \in G(W)_{\delta (y)}$ (so $G_W=0$ if
$\delta (y) \not \in \N$). We deduce that
\begin{equation}\label{eq2621}
H_W^{-1}(G_W^p,F_{p,Z,W})\subseteq \overline{m}_S \cal{O}_W[\{U_j\}_{j\in J_E}]_{\epsilon (y)}.
\end{equation}
If $i_0(x)=p$, we have $H^{-1}G^p=0$ so the first part of (\ref{eq2628}) is trivial. If
$i_0(x)=p-1$, we have $i_0(y)=p-1 $ by Definition \ref{defsecondkind}(ii), so $G_W\neq 0$.
The first part of (\ref{eq2628}) then follows from (\ref{eq2621}), i.e.
$$
H^{-1}G^p =\mathrm{cl}_1(H_W^{-1}G_W^p) \subseteq  U_{j_0}k(x)[\{U_j\}_{j\in J}]_{\epsilon (y)},
$$
for some $j_0 \in (J')_E$. \\

We deduce from (\ref{eq2621}) that
$$
\overline{J}(F_{p,Z,W},E,W)=<\mathrm{cl}_0(H_W^{-1}{\partial F_{p,Z,W} \over \partial \overline{u}_{j'}}), j'\in J' \backslash (J')_E>
\subseteq k(x)[\{U_j\}_{j\in J}]_{\epsilon (y)}.
$$
Taking classes as in (\ref{eq2627}) with $d_0=1$, we get
$$
\mathrm{cl}_1(H_W^{-1}F_{p,Z,W})\subseteq \sum_{j' \in J'}U_{j'}k(x)[\{U_j\}_{j\in J}]_{\epsilon (y)}.
$$
Since $\mathrm{cl}_1(H_W^{-1}F_{p,Z,W})$ is a homomorphic image of
$H^{-1}F_{p,Z}\in G(m_S)_{\epsilon (x)}$ as described in (\ref{eq2626}), there exists
an expansion (\ref{eq2628}). For $j' \in J'\backslash (J')_E$, we have
$$
H^{-1}{\partial F_{p,Z} \over \partial U_{j'}}=\mathrm{cl}_0(H_W^{-1}{\partial F_{p,Z,W} \over \partial \overline{u}_{j'}}).
$$
Collecting together for all $j' \in J' \backslash (J')_E$, we get
$$
\overline{J}(F_{p,Z,W},E,W)=<H^{-1}{\partial F_{p,Z} \over \partial U_{j'}}, j'\in J' \backslash (J')_E>
\subseteq k(x)[\{U_j\}_{j\in J}]_{\epsilon (y)}
$$
and the second part of (\ref{eq2628}) follows from Definition \ref{defsecondkind}(iii). \\

Note that $\epsilon (y)=\omega (x)$  is an immediate consequence of Definition \ref{defomega} if $i_0(m_S)=p$.
If $i_0(m_S)=p-1$, we must introduce a truncation operator $T: G(m_S)_{\delta (x)} \rightarrow G(m_S)_{\delta (x)}$
in order to compute $\omega (x)$. The first part of (\ref{eq2628}) now shows that there exists $j_0 \in (J')_E$
such that
$$
H^{-1}(F_{p,Z}-TF_{p,Z}) \in U_{j_0}k(x)[\{U_j\}_{j\in J}]_{\epsilon (y)}.
$$
Since $\overline{J}(F_{p,Z,W},E,W) \subseteq k(x)[\{U_j\}_{j\in J}]_{\epsilon (y)}$, we thus have:
$$
H^{-1}{\partial F_{p,Z} \over \partial U_{j'}}= H^{-1}{\partial TF_{p,Z} \over \partial U_{j'}}
$$
for every $j' \in J' \backslash (J')_E$. This proves that $\omega (x)=\epsilon (y)$.
\end{proof}

Note that it follows from the above proposition that a permissible center of the second kind
has codimension at least two in ${\cal X}$, since $\epsilon(y)>0$.

 {Permissible blowing ups of
the second kind appear naturally from permissible blowing ups of
the first kind if one requires stability by regular base change:}

\begin{thm}\label{geomregpermis}
Let $S \subseteq \tilde{S}$ be a local base change which is regular, $\tilde{S}$ excellent. Let
$\tilde{x} \in \tilde{\eta}^{-1}(m_{\tilde{S}})$ and $x \in \eta^{-1}(m_S)$ be its image.

\smallskip

If ${\cal Y}\subset {\cal X}$ is a permissible center (of the first or second kind) at $x$, then
$$
\tilde{{\cal Y}}:={\cal Y}\times_S\mathrm{Spec}\tilde{S}\subseteq \tilde{{\cal X}}={\cal X}\times_S\mathrm{Spec}\tilde{S}
$$
is permissible (of the first or second kind) at $\tilde{x}$.
\end{thm}

\begin{proof}
We denote $(\tilde{S},\tilde{h},\tilde{E})$ and $(u_1, \ldots , u_{\tilde{n}})$
as in Notations \ref{notageomreg1} and \ref{notaprime}. Since $W$ has normal crossings with $E$ at $x$,
$\tilde{W}:=\tilde{\eta}(\tilde{{\cal Y}})$ has normal crossings with $\tilde{E}$ at $\tilde{x}$.
Since ${\cal Y}$ is permissible at $x$, we have $m(y)=p$. Any generic point $\tilde{y}$ of $\tilde{{\cal Y}}$ has
$m(\tilde{y})=p$ by Theorem \ref{omegageomreg}(1), and $\tilde{{\cal Y}}$ itself is irreducible by
Proposition \ref{SingX}. Theorem \ref{omegageomreg}(2) applies to $\tilde{y}$ (with $n(y)=\tilde{n}(y)$)
and to $\tilde{x}$ and states that
$$
\epsilon (\tilde{y}) = \epsilon (y), \ \epsilon (\tilde{x})\geq  \epsilon (x), \ i_0(\tilde{y})=i_0(y),
\ i_0(\tilde{x})=i_0(x)
$$
Cases of inequality $\epsilon (\tilde{x})>  \epsilon (x)$ are classified in {\it ibid.}(2.ii). \\

Suppose that $\epsilon (\tilde{x}) > \epsilon (x)$. Then, 
$$
F_{p,Z}\in k(x)[U^p_1, \ldots ,U^p_n]  \ \mathrm{and} \  i_0(m_S)=i_0(m_{\tilde{S}})=p.
$$
Then ${\cal Y}$ is permissible of the first kind since $F_{p,Z}\in k(x)[U^p_1, \ldots ,U^p_n]$
is incompatible with the conclusion of Proposition \ref{secondkind}. Note that
$$
\epsilon (y)=\epsilon (x)=\epsilon (\tilde{x})-1=  {\epsilon({\tilde{y})}}.
$$
We claim that $\tilde{{\cal Y}}$ is permissible of the second kind at $\tilde{x}$.

\smallskip

To prove the claim, note that Definition \ref{defsecondkind}(i)  {and (ii)}
 are already checked. We have
\begin{equation}\label{eq2629}
H^{-1}{\partial F_{p,\tilde{Z}} \over \partial U_{j'}}=
H^{-1}\Phi_{j'}(U_1, \ldots , U_n) \neq 0,
\end{equation}
with notations as in Theorem \ref{omegageomreg}(2.ii) for some $j'$, $n+1 \leq j'\leq \tilde{n}$.
Since $H(\tilde{x})=H(x)\tilde{S}$ by Theorem \ref{omegageomreg}(2.i), and
$H^{-1}F_{p,Z} \subseteq k(x)[\{U_j\}_{j\in J}]_{\epsilon (x)}$ by Proposition \ref{firstkind}, we have
$$
H^{-1}F_{p,\tilde{Z}} \subseteq \sum_{j=1}^{\tilde{n}}U_j k(\tilde{x})[\{U_j\}_{j\in J}]_{\epsilon (x)}.
$$
This proves that Definition \ref{defsecondkind}(iii)  holds  and $\tilde{{\cal Y}}$ is permissible of the second kind at $\tilde{x}$.\\

Assume now that $\epsilon (\tilde{x}) = \epsilon (x)$. If ${\cal Y}$ is permissible of the first kind at $x$,
we have $\epsilon (\tilde{y})= \epsilon (\tilde{x})$, so
$\tilde{{\cal Y}}$ is also permissible of the first kind at $\tilde{x}$.

If ${\cal Y}$ is permissible of the second kind at $x$, Definition \ref{defsecondkind}(ii) is checked.
Finally by Proposition \ref{secondkind},
the polyhedron $\Delta_{S}(h;u_1, \ldots ,u_n;Z)$ has a vertex $\mathbf{x}$ such that
$x_{j'}\not \in \N $ for some $j' \in J' \backslash  {(J')}_E$. The corresponding vertex
$$
\mathbf{\tilde{x}}:=(\mathbf{x}, \underbrace{0, \ldots ,0}_{\tilde{n}-n}) \in
\Delta_{\tilde{S}} {(h;u_1, \ldots ,u_{\tilde{n}};Z)}
$$
is thus not solvable. We hence get $\mathbf{\tilde{x}} \in \Delta_{\tilde{S}} {(h;u_1, \ldots ,u_{\tilde{n}};\tilde{Z})}$
and Definition \ref{defsecondkind}(iii) is checked. Hence
$\tilde{{\cal Y}}$ is permissible of the second kind at $\tilde{x}$ as required, since $H(\tilde{x})=H(x) {\tilde{S}}$.
\end{proof}

\subsection{Blowing up Theorem.}

Let $ \pi : \ {\cal X}' \rightarrow {\cal X}$ be the blowing up along a
permissible center ${\cal Y}$ (of the first or second kind) at $x \in {\cal Y}$, $\{x\} = \eta^{-1}(m_S)$.
Our objective is to relate $\omega (x')$  to $\omega (x)$ for points $x' \in \pi^{-1}(x)$.  {The main result is Theorem \ref{bupthm} which states that the pair $(m(x),\omega(x))$ does not increase and studies the equality case. In contrast, we may have $\epsilon(x')>\epsilon(x)$, see \ref{bupthm}(2) about this jumping phenomenon.}\\

We keep notations as in Proposition \ref{Hironakastable} and Proposition \ref{SingX}. Then
$\sigma: {\cal S}' \rightarrow \mathrm{Spec}S$ denotes the blowing up along $W$ and
there is a commutative diagram (\ref{eq210}). Let
$$
\eta '  : \ {\cal X}' \rightarrow {\cal S}', \ s' :=\eta '(x')\in \sigma^{-1}(m_S), \ S':={\cal O}_{{\cal S}' ,s'}.
$$
We denote by $W':=\sigma^{-1}(W)$ and $E':=\sigma^{-1}(E)_\mathrm{red}$. We do not change notations
to denote stalks at $s'$, i.e. we will write $\eta ': \ {\cal X}_{s'} \rightarrow \mathrm{Spec}S'$ for
the stalk at $s'$ of the above map $\eta '$, and $W',E'$ for the stalks at $s'$ of the corresponding divisors.
By Proposition \ref{SingX}, we have ${\eta '}^{-1}(s')=\{x'\}$ if $x'$ is not a regular point of $X'$.

\smallskip

For the purpose of computations, we shall pick well adapted coordinates $(u_1,\ldots ,u_n;Z)$
such that
$$
I(W)=(\{u_j\}_{j\in J}), \ {\cal Y}=V(Z, \{u_j\}_{j\in J}).
$$
with notations as in (\ref{eq71}).
We denote by $u \in S'$ a local equation for $W'$, which can be taken to be some $u_{j_1}$, where $j_1\in J$
depends on $s'$. We have ${\cal X}'=\mathrm{Spec}(S'[X']/(h'))$, where
\begin{equation}\label{eq363}
    h':=u^{-p}h={X'}^p+f_{1,X'}{X'}^{p-1}+ \cdots +f_{p,X'}\in S'[X'],
\end{equation}
and
\begin{equation}\label{eq3631}
X':=Z/u, \ f_{i,X'}:=u^{-i}f_{i,Z}\in S' \  \mathrm{for} \ 1 \leq i \leq p.
\end{equation}

Since ${\cal Y}$ is permissible, we have $\epsilon (y)>0$ so the initial form $\mathrm{in}_{W}h$
reduces to :
\begin{equation}\label{eq3633}
\mathrm{in}_{W}h=Z^p - G_W^{p-1}Z +F_{p,Z,W} \in G(W)[Z],
\end{equation}
with $G_W \in G(W)_{\delta (y)}$ and $F_{p,Z,W} \in G(W)_{p\delta (y)}$ (in particular $G_W=0$
if $\delta (y) \not \in \N$). Since $\sigma^{-1}(W)= \mathbf{Proj}G(W)$, the restriction
map
$$
G(W)_d =\Gamma (W',{\cal O}_{W'}(d)) \rightarrow \Gamma (W'\backslash V(U),{\cal O}_{W'}(d))
$$
gives an inclusion
\begin{equation}\label{eq3632}
U^{-d}G(W)_d = \cal{O}_W[\{U_j/U\}_{j\in J}]_{\leq d} \subset {\cal O}_{W',s'}=S'/(u)
\end{equation}
for each $d \geq 0$. There is an identification:
\begin{equation}\label{eq3634}
U^{-d}G(W')_d = \left (\cal{O}_W[\{U_j/U\}_{j\in J}]\right )_{s'}=S'/(u).
\end{equation}
Finally, we note that ${\cal D}_{W'}={\cal D}(W')$ by (\ref{eq243}) since $W'$ is a component of $E'$.
These remarks are essential for stating the blow up formula in Proposition \ref{bupformula}(v) below.

\begin{prop}\label{bupformula} (Blow up formula) \index{Blow up Formula, Proposition~\ref{bupformula}}
Let $\pi : {\cal X}' \rightarrow {\cal X}$ be the blowing up along a permissible center ${\cal Y}$ at $x$,
$\{x\} = \eta^{-1}(m_S)$ and $x' \in \pi^{-1}(x)$  {be a closed point}.  With notations as above, the following holds:
\begin{itemize}
    \item [(i)] there exists a r.s.p. $(u'_1,\ldots , {u'_{n}})$ of $S'$ which is adapted to $(S',h',E')$;
    \item [(ii)] $\mathrm{in}_{W'}h'={X'}^p-G_{W'}^{p-1}X'+F_{p,X',W'} \in G(W')[X']$ and is
    given by
    $$
    G_{W'}=U^{-1}G_W\in G(W')_{\delta (y)-1}, \ F_{p,X',W'}=U^{-p}F_{p,Z,W} \in G(W')_{p(\delta (y)-1)};
    $$
    \item [(iii)] the polyhedron $\Delta_{S'}(h';u;X')$ is minimal;
    \item [(iv)] we have $H(x')=u^{\epsilon (y)-p}H(x)\subseteq S'$;
    \item [(v)] there is an equality of ideals of $\widehat{{\cal O}_{W',s'}}$:
    $$
    \left\{
  \begin{array}{ccc}
    H_{W'}^{-1}G_{W'}^p & = & (U^{-\epsilon (y)}H_{W}^{-1}G_{W}^p )_{s'} \hfill{},  \\
     & & \\
    J(F_{p,X',W'},E',W') & = & (U^{-\epsilon (y)}J(F_{p,Z,W},E,W) )\widehat{{\cal O}_{W',s'}}. \\
  \end{array}
\right.
    $$
\end{itemize}
\end{prop}

\begin{proof}
Statement (i) is proved in Proposition \ref{Hironakastable}. The formula in (ii) is
obvious from (\ref{eq363}), (\ref{eq3631}) and (\ref{eq3633}).

\smallskip

If $i_0(W)=p-1$, i.e. $G_W \neq 0$ in (\ref{eq3633}), we have $G_{W'}\neq 0$ by (ii),
so $\Delta_{\widehat{S'}}(h';u;X')\subseteq \R_{\geq 0}$ is minimal.

If $i_0(W)=p$, then $F_{p,Z,W}\not \in G(W)^p$, i.e.
$$
\delta (y)\not \in p\N \ \mathrm{or} \ U^{-\delta (y)}F_{p,Z,W}\not \in k(W')^p.
$$
Note that $G(W)^p = (k(W')[U,U^{-1}])^p \cap G(W)$ since $G(W)$ is integrally closed.
By (ii), $F_{p,X',W'}=U^{-p}F_{p,Z,W}$ so  $F_{p,X',W'}\not \in G(W')^p$ and
this proves (iii).

\smallskip

To prove (iv), first consider those irreducible components
$W_j=\mathrm{div}(u_j)$ of $E$, $1 \leq j \leq e$, whose strict
transform $W'_j$ passes through $s'$. We may pick a r.s.p. $(u'_1,\ldots
, {u'_{n}})$ of $S'$ which is adapted to $(S',h',E')$, containing
 {$u$,  $u'_j:=u_j/u$ for $j \in J_E$  and
$u'_j:=u_j$ for $j \in J'$}. Let
$$
\mathrm{in}_{W_j}h(Z)=Z^p +F_{1,Z,W_j}Z^{p-1}+ \cdots
+F_{p,Z,W_j}\in S/(u_j)[U_j][Z].
$$
We have $\mathrm{in}_{W'_j}h'= \mathrm{in}_{W_j}
u^{-p}h(uX')\in S'/(u'_j)[U'_j][X']$, since $u$ is a unit in
$S'_{(u'_j)}=S_{(u_j)}$. Since $\Delta_{S}(h;u_1,\ldots,u_n;Z)$ is
minimal, we have
$$
\Delta_{S_{(u_j)}}(h;u_j;Z)=\Delta_{S'_{(u'_j)}}(h';u'_j;X')
$$
minimal as well by Proposition \ref{Deltaalg}, hence
$\mathrm{ord}_{(u'_j)}H(x')= \mathrm{ord}_{(u_j)}H(x)$.

By (ii) and (iii), we have $\mathrm{ord}_{(u)}H(x')=p(\delta(y)-1)$. Therefore
$$
\mathrm{ord}_{(u)}H(x')-\mathrm{ord}_{(u)}H(x)=p(\delta(y)-1)-\mathrm{ord}_{W}H(x)=\epsilon (y)-p
$$
and the conclusion follows.

\smallskip





We now prove (v). The first part of the statement follows immediately
from (ii) and (iv). With notations as in (\ref{eq2431}), we have
$$
 \left\{
  \begin{array}{ccccc}
    J(F_{p,Z,W},E,W) & = & H_W^{-1}{\cal J}(F_{p,Z,W},E,W) & \subseteq &  \widehat{G(W)}_{\epsilon (y)}\hfill{},  \\
     & & & & \\
    J(F_{p,X',W'},E',W') & = & H_{W'}^{-1}{\cal J}(F_{p,X',W'},E',W')& \subseteq &  \widehat{G(W')}_{0}. \\
  \end{array}
\right.
$$

We now define and make explicit the required inclusion
$$
\widehat{G(W)}_{\epsilon(y)}\subset \widehat{G(W')}_0.
$$
Let us first complete
the $\widehat{{\cal O}_W}$- algebra $\widehat{G(W)}$ for  the $\overline{m}_S$-adic topology. There is an induced
inclusion
$$
\widehat{G(W)}\subset \lim_{\leftarrow} {G(W) \over \overline{m}_S^n}.
$$
The inclusions $G(W)\subset G(W') \subset \widehat{G(W')}$ lead to the following inclusions
\begin{equation}\label{eq:fixbup}
\widehat{G(W)}\subset \lim_{\leftarrow} {G(W) \over \overline{m}_S^n}
\subset \lim_{\leftarrow} {G(W') \over \overline{m}_S^n}\subset \lim_{\leftarrow} {\widehat{G(W')} \over \overline{m}_S^n}.
\end{equation}

Explicitly, we pick an isomorphism $\widehat{\cal{O}_W}\simeq k(x)[[\{\overline{u}_{j'}\}_{j'\in J'}]]$
given by Proposition \ref{Cohen}. The last three terms in (\ref{eq:fixbup}) are formal power
series rings in variables $\{\overline{u}_{j'}\}_{j'\in J'}$ with respective coefficient rings:
$$
A:=k(x)[\{U_j\}_{j\in J}] \subset A'[U]:=k(x)[\{V_j\}_{j\in J\backslash \{j_1\}}]_{s'}[U]
={\cal O}_{\sigma^{-1}(m_S),s'}[U]\subseteq \hat{A'}[U],
$$
where $V_j:=U_j/U \in G(W')_0$,  $j\in J\backslash \{j_1\}$. Finally, we have
\begin{equation}\label{eq:fixbup1}
\lim_{\leftarrow} {\widehat{G(W')} \over \overline{m}_S^n}\simeq \hat{A'}[U][[\{\overline{u}_{j'}\}_{j'\in J'}]],
\ \widehat{G(W')}\simeq \hat{A'}[[\{\overline{u}_{j'}\}_{j'\in J'}]][U].
\end{equation}
The required map $\widehat{G(W)}_{\epsilon(y)}\subset \widehat{G(W')}_0$ in (v) is given by
$F \mapsto U^{-\epsilon (y)}F$.

\smallskip

Applying (ii) and (iv), we get:
$$
F_{p,X',W'}=U^{-p}F_{p,Z,W}, \ H_{W'}=H_W U^{\epsilon (y)-p}G(W').
$$
Since $D \cdot U^p =0$ for every $D \in {\cal D}_{W'}$, (v) can be written in the following form:
\begin{equation}\label{eq3637}
U^{-\mathrm{deg}F_{p,Z,W}}{\cal J}(F_{p,Z,W},E',W')=(U^{-\mathrm{deg}F_{p,Z,W}}{\cal J}(F_{p,Z,W},E,W))\widehat{S'/(u)}.
\end{equation}

Any $D \in \{\overline{u}_{j'} {\partial \hfill{} \over \partial \overline{u}_{j'}}\}_{j' \in (J')_E} \cup
\{{\partial \hfill{} \over \partial \overline{u}_{j'}}\}_{j' \in J' \backslash (J')_E}$ extends in the obvious
way to the right hand side of this diagram, so it commutes with the inclusion $A \subset \hat{A'}[U]$. In
other terms, we are reduced to a statement on the coefficients of the power series (\ref{eq:fixbup1})
expliciting (\ref{eq:fixbup}).
This reduces the proof of (v) to the special case where $W=\{m_S\}$ is the closed point
($J=\{1, \ldots ,n\}$).

\smallskip

By (\ref{eq242}), the $G(m_S)$-module ${\cal D}_{m_S}$ is generated by the family
$$
\left ( \{U_j {\partial \hfill{} \over \partial U_j}\}_{1\leq j \leq e},
\{U_{j_1}{\partial \hfill{} \over \partial U_j}\}_{1\leq j_1 \leq n, e+1\leq j},
\{{\partial \hfill{} \over \partial \lambda_l}\}_{l \in \Lambda} \right ).
$$
The $A'$-module of absolute differentials
$$
\Omega^1_{A'}\left (\mathrm{log}(U \prod_{j=1}^eV_j)\right )
$$
has a basis obtained by collecting together $(d\lambda_l\otimes 1)_{l\in \Lambda}$, $dU/ U$ and the
$\{dV_j / V_j\}_{1\leq j \leq e}$, $\{dV_j\}_{e+1\leq j \leq n}$ with $j\neq j_1$.
For $F \in A$, we deduce the following standard formul{\ae} in $A'$:
\begin{equation}\label{eq3635}
\left\{
\begin{array}{cccc}
    U{\partial F \over \partial U}& = & \sum_{j =1}^n U_j{\partial F \over \partial U_j} & \\
     & & & \\
    V_j{\partial F \over \partial V_j} &  = & U_j{\partial F \over \partial U_j} &  1\leq j \leq e, j\neq j_1 \hfill{}\\
     & & & \\
    {\partial F \over \partial V_j}&  = &  U{\partial F \over \partial U_j} &  e+1\leq j \leq n, j\neq j_1\\
\end{array}
\right.
.
\end{equation}
Taking $F\in A_d$, $d \in \N$, we have for $j \geq e+1$:
$$
(U^{-d}\{U_{j_1}{\partial F \over \partial U_j}\}_{1\leq j_1\leq n})A'=(U^{-d}U{\partial F \over \partial U_j})A'.
$$
Collecting together this equation with  (\ref{eq3635}), we get
$$
U^{-d}{\cal J}'(F,E',W')=(U^{-d}{\cal J}(F,E,m_S))A',
$$
where ${\cal J}'(F,E',W')=\{D' \in \mathrm{Der}(G(W')) : D' \cdot I(E'(W'))\subseteq I(E'(W'))\}$, notations as
in Proposition \ref{defJftype}. Since $S'$ is essentially of finite type over $k(x)$, this proposition
implies that  ${\cal J}(F,E',W')={\cal J}'(F,E',W')\hat{A'}$. This concludes the proof.
\end{proof}

We now state the main theorem of this section. Recall that the function $y \mapsto \omega (y)$ and
$\kappa (y)\in \{ 1, \geq 2\}$ have been defined for given $(S,h,E)$ and $y \in {\cal X}$
(Definition \ref{defmult} and Definition \ref{defomega}). By Proposition \ref{Estable},
$(S',h',E')$ satisfies again conditions {\bf (G)} and {\bf (E)}. The values of
$\epsilon (x')$, $\iota (x')$ are computed w.r.t. the adapted structure $(S',h',E')$.

\begin{nota}\label{indcoord} {\it Choice of coordinates:}
by Proposition \ref{bupformula}(i), there exists a r.s.p. $(u'_1,\ldots ,u'_{n'})$
which is adapted to $(S',h',E')$ for some $n' \leq n$. We take $u'_1:=u$. Let
$$
u'_i:={u_{j_i} \over u}, \ 2  \leq i \leq e'_0, \ \mathrm{where} \
\{j_2, \ldots ,j_{e'_0}\}:=\{  j\in J_E : {u_j \over u} \in m_{S'}\}.
$$
Let $\{j_{e'_0+1}, \ldots ,j_{e'}\}:=(J')_E$, $\{j_{e'+1}, \ldots ,j_{n'_0}\}=:J' \backslash (J')_E $.
We take
$$
u'_i:=u_{j_i}, \ e'_0+1  \leq i \leq n'_0.
$$
Let
$$
u'_i:={u_{j_i} \over u}, \ n'_0+1  \leq i \leq n'_1, \ \mathrm{where} \
\{j_{n'_0+1}, \ldots ,j_{n'_1}\}:=\{  j\in J \backslash J_E : {u_j \over u} \in m_{S'}\}
$$
and complete $(u'_1,\ldots ,u'_{n'_1})$ to a r.s.p. $(u'_1,\ldots ,u'_{n'})$ of $S'$.
\end{nota}

\begin{nota}\label{ordbar}
Let
$$
\overline{S'}:=\hat{{\cal O}}_{\sigma^{-1}(m_S),s'}=\hat{S'}/(u,\{u_{j'}\}_{j'\in J'})=
\widehat{k(x)[\{U_j/U\}_{j\in J}]_{\overline{m'}}},
$$
where $\overline{m'}$ denotes the ideal of the restriction of $s'$ to $\sigma^{-1}(m_S)$:
$$
\overline{m'}:=(\{\overline{u'_i}\}_{i \in F}), \ F:=\{2 , \ldots , e'_0\} \cup \{n'_0+1 , \ldots , n'\}.
$$

For $I' \subseteq \hat{S'}/(u)$ an ideal, we denote by
$$
\mathrm{ord}I':=\mathrm{ord}_{m_{\hat{S'}/(u)}}I'= \min_{\varphi ' \in I'}\{\mathrm{ord}_{m_{\hat{S'}/(u)}}\varphi '\}, \
\overline{\mathrm{ord}}I' :=\mathrm{ord}_{\overline{m'}}I'\overline{S'}.
$$

For every $I' \subseteq \hat{S'}/(u)$, we have $\mathrm{ord}I' \leq \overline{\mathrm{ord}}I' \leq +\infty $.
If furthermore $d'$ is given, $d' \leq \overline{\mathrm{ord}}I'$, we write
$$
\overline{I'} \subseteq \left (\mathrm{gr}_{\overline{m'}}\overline{S'}\right )_{d'}
=k(x')[\{U'_i\}_{i \in F}]_{d'}
$$
for the initial part  of degree $d'$ of the ideal $I'\overline{S'}$.
\end{nota}

 {
We now introduce the adapted cone associated
to a permissible blowing up. Recall the definition of $B$ from (\ref{eq2612})
({\it cf.} also Definition \ref{defomega}). We have $B=\emptyset$ if $i_0(m_S)=p$, and}
$$
 {B= \{j : U_j \ \mathrm{divides} \ H^{-1}G^p\}\ \mathrm{if} \ i_0(m_S)=p-1.}
$$

\begin{defn}\label{defcone}
 {Let ${\cal Y} \subset {\cal X}$, with generic point $y$, be a permissible center at $x $. We define a subcone \index{$C(x,{\cal Y})$, Definition~\ref{defcone}}
$$
C(x,{\cal Y}) \subset \mathrm{Spec} (k(x)[\{U_j\}_{j\in J}] )
$$
as follows: if ${\cal Y}$ is of the first kind, we let:
$$
C(x,{\cal Y}):= \mathrm{Spec}\left ({k(x)[\{U_j\}_{j\in J}] \over (\mathrm{IMax}(x)\cap k(x)[\{U_j\}_{j\in J}])}\right);
$$
if ${\cal Y}$ is of the second kind, we let $B_J:=B \backslash \{j_0\}$ with notations as in
Proposition \ref{secondkind} and define:
$$
C(x,{\cal Y}):= \mathrm{Max}(\overline{J}(F_{p,Z,W},E,W))\cap \{ U_{B_J}=0\} .
$$
In both cases, we denote the associated projective cone \index{$PC(x,{\cal Y})$, Definition~\ref{defcone}}
by
$$PC(x,{\cal Y})\hookrightarrow \sigma^{-1}(m_S) \simeq \PP^{\mid J \mid -1}_{k(x)} .$$}

\end{defn}

\begin{thm}\label{bupthm}
Assume that $m(x)=p$, $\omega (x)>0$, where $\{x\}=\eta^{-1}(m_S)$. Let $\pi : {\cal X}' \rightarrow {\cal X}$
be the blowing up along a permissible center ${\cal Y}$ (of the first kind or second kind) at $x$,
$x' \in \pi^{-1}(x)$ and $\eta ': \ {\cal X}' \rightarrow \mathrm{Spec}S'$ be with notations as above,
where $s'=\eta '(x')$. Then
\begin{equation}\label{eq364}
(m(x') ,\omega (x'), \kappa (x')) \leq (m(x),\omega (x),\kappa (x)).
\end{equation}

If equality holds in (\ref{eq364}), then $s' \in PC(x,{\cal Y})$.\\

If $\epsilon (x')>\epsilon (x)$, the following holds:
\begin{itemize}
  \item [(1)] we have $i_0(m_S)=p, \ \epsilon (y)=\epsilon (x)=\omega (x), \ \delta (y) \in \N$,
$H_{j'}\in p\N$  for every $j' \in (J')_E $
and
$$
F_{p,Z}\in (k(x')[U_1, \ldots ,U_n])^p[\{U_j\}_{j \in J_E \backslash \{j_2, \ldots ,j_{e'_0}\}}];
$$
  \item [(2)] let $(u'_1,\ldots ,u'_{n'};Z')$ be well adapted coordinates at $x'$. Then
\begin{equation}\label{eq3642}
{H'}^{-1}F_{p,Z'} \nsubseteq k(x')[U'_1, \ldots , U'_{n'_1}]_{\epsilon (x')}
\oplus ( \{U'_i\}_{i \not \in F})  \cap G(m_{S'})_{\epsilon (x')}
\end{equation}
and there exists
$\Phi' \in k(x')[{U'_1}^p, \ldots ,{U'_{n'_1}}^p][U'_{n'_1+1} ,\ldots , U'_{n'}]_{p\delta (x')}$  such that
\begin{equation}\label{eq3641}
{H'}^{-1}(F_{p,Z'} - \Phi' ) \subseteq ( \{U'_i\}_{i \not \in F}) \cap G(m_{S'})_{\epsilon (x')}.
\end{equation}
\end{itemize}

\end{thm}

\begin{proof}
Since ${\cal Y}$ is permissible, ${\cal Y} $ is
Hironaka-permissible at $x$ and this implies that $m(x')\leq m(x)=p$ in any case. We are
done unless equality holds, so assume that $m(x')=p$.

The polyhedron $\Delta_{S'}(h';u'_1,\ldots ,u'_{n'};X')$ need not be minimal. We must
take $Z'=X'-\theta '$, $\theta ' \in S'$ such that the polyhedron $\Delta_{S'}(h';u'_1,\ldots ,u'_{n'};Z')$
is minimal in order to read off $\epsilon (x')$ and $\omega (x')$ from $\mathrm{in}_{m_{S'}}h'$.

By Proposition \ref{bupformula}(iii), we have $\mathrm{ord}_{(u)}H(x')=p (\delta (y)-1)$.
The initial form $H_{W'}$ of $H(x')$ in $G(W')$ is given by Proposition \ref{bupformula}(iv):
\begin{equation}\label{eq3651}
H_{W'}= <U^{p (\delta (y)-1)}\prod_{i=2}^{e'}\overline{u'_i}^{H_{j_i}}>.
\end{equation}

We have ${\theta '}^p \in H(x')$ since $f_{p,X'} \in H(x')$. Let $\Theta ' \in G(W')_{\delta (y)-1}$ be the
initial form of $\theta '$ (in particular $\Theta ' =0 $ if $\delta (y) \not \in \N$). Then
\begin{equation}\label{eq365}
\mathrm{in}_{W'}h' = {Z'}^p -G_{W'}^{p-1}Z' + F_{p,X',W'} + {\Theta '}^p - G_{W'}^{p-1} \Theta ' \in G(W')[Z']
\end{equation}
where $G_{W'}=U^{-1}G_W$, $F_{p,X',W'}=U^{-p}F_{p,Z,W}$ by Proposition \ref{bupformula}(ii).
According to our notations, we have:
$$
F_{p,Z',W'}=F_{p,X',W'} + {\Theta '}^p - G_{W'}^{p-1} \Theta '.
$$

Note that derivatives in ${\cal D}_{W'}$ decrease orders by at most one. Since $H_{W'}$ is the initial form of $H(x')$ in $G(W')$,
we have:
\begin{equation}\label{eq3652}
\epsilon (x') \leq \min\{ \mathrm{ord}_{m_{S'/(u)}}(H_{W'}^{-1}G_{W'}^p),
1+\mathrm{ord}_{m_{S'/(u)}} J(F_{p,Z',W'},E',W')\}.
\end{equation}
Inequality may be strict, since the $H(x')^{-i}f_{i,Z'}^p$, $1 \leq i \leq p$
may acquire terms of lower order not coming from $\mathrm{in}_{W}h$. Moreover, some derivatives in ${\cal D}_{W'}$
do not decrease orders and may give a sharper bound in (\ref{eq3652}).\\

Recall that if $M \subseteq \widehat{G(W)}_d$, $d \in \N$ is a submodule, and $d_0$ is given, there are  associated initial forms
$$
\overline{M} \subseteq \left (\mathrm{gr}_{\overline{m}_{S}}G(W)_d\right )_{d_0} \subset
{G(m_S)_{d+d_0} \over <(\{U_j\}_{j\in J})^{d+1}\cap G(m_S)_{d+d_0}>}
$$
under the conditions described in (\ref{eq2627}) and (\ref{eq2626}). Note that
$$
\left (\mathrm{gr}_{\overline{m}_{S}}G(W)_d \right )_{0}=   \Gamma (\sigma^{-1}(m_S), {\cal O}_{\sigma^{-1}(m_S)}(d))=
k(x)[\{U_j\}_{j\in J}]_{d}
$$
for $d_0=0$.\\

Since ${\theta '}^p \in H(x')$, we have ${\Theta '}^p\in H_{W'}$ in (\ref{eq365}). We have
$\Theta ' = 0$ or $\delta (y) \in \N$ and
$$
G_{W'}^{p-1} \Theta ' \in G_{W'}^{p-1}\left \lceil H_{W'}^{{1 \over p}}\right \rceil, \
\left \lceil H_{W'}^{{1 \over p}}\right \rceil :=
<U^{\delta (y)-1}\prod_{i=2}^{e'}{\overline{u'_{i}}^{\left \lceil {H_{j_i} \over p}\right \rceil}}>.
$$
Since $D \cdot {\Theta '}^p =0$ for every $D \in {\cal D}_{W'}$, we deduce from (\ref{eq365}) that
\begin{equation}\label{eq367}
J(F_{p,Z',W'},E',W')\equiv  J(F_{p,X',W'},E',W')
\ \mathrm{mod}H_{W'}^{-1}G_{W'}^{p-1}\left \lceil H_{W'}^{{1 \over p}}\right \rceil .
\end{equation}

Note that if $i_0 (m_S)=p$, or if $H_{j'} \not \in p\N$
for some $j' \in (J')_E$, we have
\begin{equation}\label{eq3672}
G_W=0 \ \mathrm{or} \  \mathrm{ord}_{(u_{j'})}(H_W^{-1}G_W^p)>0 \ \mathrm{for} \
\mathrm{some}  \ j' \in (J')_E
\end{equation}
by applying  Proposition \ref{deltaint}(iii) in the latter case. In this case, we obtain the
following from Proposition \ref{bupformula}(v) and (\ref{eq367}):
\begin{equation}\label{eq3671}
(H_{W'}^{-1}G_{W'}^p)\overline{S'}=0, \ J(F_{p,Z',W'},E',W')\overline{S'} =
J(F_{p,X',W'},E',W')\overline{S'} .
\end{equation}

\noindent {\it Case 1:}  $i_0(m_S)=p$ {\it and} ${\cal Y}$ {\it is of the first kind.} In order to
get an estimate of $\epsilon (x')$ from (\ref{eq3652}), we take:
$$
M=J(F_{p,Z,W},E,W),  \ d=\epsilon (y)=\epsilon (x), \ d_0=0.
$$

\begin{rem}
By Proposition \ref{firstkind}, there is an equality
$$
H^{-1}F_{p,Z} =\mathrm{cl}_{\epsilon (x)}H_W^{-1}F_{p,Z,W}\subseteq k(x) [\{U_j\}_{j\in J}]_{\epsilon (x)},
$$
but we emphasize that the induced inclusion
\begin{equation}\label{eq3675}
J(F_{p,Z},E,m_S) \subseteq \mathrm{cl}_{\epsilon (x)}J(F_{p,Z,W},E,W).
\end{equation}
is strict in general: this is because elements of the form
$$
\mathrm{cl}_{\epsilon (x)}(H_W^{-1}{\partial F_{p,Z,W} \over \partial \overline{u}_{j'}}), \ j' \in J' \backslash (J')_E
$$
may be nonzero.
\end{rem}

By Proposition \ref{indiff}(ii) and the remark, we have
$$
0\neq J(F_{p,Z},E,m_S) \subseteq \overline{M} \subseteq k(x) [\{U_j\}_{j\in J}]_{\epsilon (x)}.
$$

Let $I'=J(F_{p,X',W'},E',W')\subseteq \widehat{S'/(u)}$, $d'=\mathrm{ord} I'$. By Proposition \ref{bupformula}(v), we have
$$
\left (U^{-\epsilon (x)}J(F_{p,Z},E,m_S)\right )_{\overline{m'}}\subseteq I'\overline{S'} .
$$

Since $i_0(m_S)=p$, we obtain from (\ref{eq3671}) that:
\begin{equation}\label{eq369}
\left (U^{-\epsilon (x)}J(F_{p,Z},E,m_S)\right )_{\overline{m'}}\subseteq
I'\overline{S'} = J(F_{p,Z',W'},E_{W'},W')\overline{S'}.
\end{equation}

If $\omega (x)=\epsilon (x)$, Definition \ref{deftauprime} gives
$$
\mathrm{Max}(x)=\mathrm{Max}(J(F_{p,Z},E,m_S)).
$$
We deduce that $\overline{\mathrm{ord}}I' \leq \omega (x)$ and
\begin{equation}\label{eq3691}
s' \not \in PC(x,{\cal Y}) \Longrightarrow  \overline{\mathrm{ord}}I'< \omega (x) .
\end{equation}

If $\omega (x)=\epsilon (x)-1$, Definition \ref{deftauprime} gives
$$
\mathrm{Max}(x)=\mathrm{Max}(V(F_{p,Z},E,m_S)).
$$
Since  $U_{j_1}V(F_{p,Z},E,m_S)\subseteq J(F_{p,Z},E,m_S)$ (recall that $u=u_{j_1}$), we also deduce
that $\overline{\mathrm{ord}}I' \leq \omega (x)$ and (\ref{eq3691}) holds. We have:
$$
\epsilon (x') \leq  1 + \mathrm{ord} I' =1+d' \leq 1 + \overline{\mathrm{ord}} I',
$$
by (\ref{eq3652}). We have proved that
\begin{equation}\label{eq3692}
\epsilon (x') \leq 1 + \overline{\mathrm{ord}} I' \leq 1 +\omega (x)
\end{equation}
with strict inequality on the right hand side under the assumption of (\ref{eq3691}).
The proof is now an easy consequence of the following claim:
$$
\epsilon (x') = 1 + \overline{\mathrm{ord}} I' \Longrightarrow \omega (x')=\epsilon (x')-1.
$$

Namely, assuming the claim, we have $\omega (x') \leq \omega (x)$ and this inequality is strict
under the assumption of (\ref{eq3691}). The first part of the proof is complete since
$i_0(m_S)=p$ implies $\kappa (x) \geq 2$. To prove the claim, let
$$
\mathrm{in}_{m_{S'}}h= {Z'}^p -{G'}^{p-1}Z' + F_{p,Z'}\in G(m_{S'})[Z']
$$
be the initial form polynomial.
Since  it is assumed that $\epsilon (x')=1 + \overline{\mathrm{ord}} I'$, we have $\overline{I'} \neq 0$ and:
\begin{equation}\label{eq3693}
\overline{I'}= <\left \{{H'}^{-1}{\partial F_{p,Z'} \over \partial U'_{j}}\right \}_{j= n'_0+1}^{n'}>
\ \mathrm{mod} (\{U'_{j'}\}_{j' \not \in F})\cap G(m_{S'})_{d'}.
\end{equation}
To compute $\omega (x')$, we must introduce a truncation operator
$$
T': G(m_{S'})_{p\delta (x')}\rightarrow G(m_{S'})_{p\delta (x')}
$$
as in Definition \ref{defomega}.  By (\ref{eq3651}), we have
$$
H':=\mathrm{cl}_{p\delta (x')-\epsilon (x')}H(x')= <U^{p (\delta (y)-1)}\prod_{i=2}^{e'}{U'_i}^{H_{j_i}}>\in G(m_{S'}).
$$
Going back to Definition \ref{defT}, we have
$$
F_{p,Z'}-T' F_{p,Z'} \in <{G'}^{p-1}U^{\delta (y)-1} \prod_{i=2}^{e'}{U'_i}^{\left \lceil{ H_{j_i}\over p}\right \rceil}>.
$$
Since $i_0(m_S)=p$, (\ref{eq3671}) applies and implies that
\begin{equation}\label{eq3694}
{H'}^{-1}(F_{p,Z'}-T' F_{p,Z'}) \subseteq (\{U'_i\}_{i \not \in F})\cap G(m_{S'})_{\epsilon (x')}.
\end{equation}
Comparing with  (\ref{eq3693}), there exists $i$, $n'_0+1 \leq i \leq n'$ such that
\begin{equation}\label{eq3695}
{H'}^{-1}{\partial T 'F_{p,Z'} \over \partial U'_{i}}\neq 0 ,
\end{equation}
since $\overline{I'}\neq 0$. This proves that $\omega (x')=\epsilon (x')-1$ as claimed. \\

To conclude the proof in case 1, assume that $\epsilon (x')>\epsilon (x)$. If some
inequality is strict in (\ref{eq3691}), we have $\epsilon (x') \leq \omega (x) \leq \epsilon (x)$: a
contradiction. So $\omega (x')= \omega (x)$ and by the above claim, we get
\begin{equation}\label{eq370}
\epsilon (x)=\omega (x)=\omega (x')=\epsilon (x')-1=\mathrm{ord} I'= \overline{\mathrm{ord}} I'.
\end{equation}

We use notations as in (\ref{eq2412}). Suppose that there exists $j' \in (J')_E$ such that
$H_{j'} \not \in p\N$. By Proposition \ref{firstkind}, we have
$$
H^{-1}U_{j'}{\partial F_{p,Z} \over \partial U_{j'}}\neq 0.
$$
Going back to (\ref{eq369}), we have
$$
\phi_{j'}:=\left (U^{-\epsilon (x)}H^{-1}U_{j'}{\partial F_{p,Z} \over \partial U_{j'}}\right )_{\overline{m'}}
\subseteq  J(F_{p,Z',W'},E',W')\overline{S'}.
$$
Applying the transformation rule in Proposition \ref{bupformula}(v),
we have
$$
\phi_{j'} = (H_{W'}^{-1}\overline{u}_{j'}{\partial F_{p,Z',W'} \over \partial \overline{u}_{j'}})
\overline{S'}.
$$
Since $\overline{\mathrm{ord}}\phi_{j'} \leq \epsilon (x)$, we deduce that
$$
\epsilon (x') \leq \overline{\mathrm{ord}}(H_{W'}^{-1}F_{p,Z',W'}) \leq
\overline{\mathrm{ord}}(H_{W'}^{-1}\overline{u}_{j'}{\partial F_{p,Z',W'} \over \partial \overline{u}_{j'}})
 \leq \epsilon (x).
$$
This is a contradiction with (\ref{eq370}). Hence $H_{j'} \in p\N$ for every $j' \in (J')_E$.

Suppose that $\delta (y) \not \in \N$. Similarly, by Proposition \ref{firstkind}, we have:
$$
H^{-1}D \cdot F_{p,Z}\neq 0, \ D:=\sum_{j \in J}U_j{\partial \hfill{} \over \partial U_j}\in \mathrm{Der}(G(W)).
$$
Note that we have $\Theta' =0$ in (\ref{eq365}) since $\delta (y) \not \in \N$. By (\ref{eq3635}):
$$
\phi_D :=\left (U^{-\epsilon (x)}H^{-1}D \cdot F_{p,Z}\right )\widehat{S'/(u)}
=H_{W'}^{-1}U{\partial F_{p,Z',W'} \over \partial U}.
$$
Arguing as above, we get a contradiction from:
$$
\epsilon (x') \leq \mathrm{ord}(H_{W'}^{-1}F_{p,Z',W'})
\leq \mathrm{ord}(H_{W'}^{-1}U{\partial F_{p,Z',W'} \over \partial U}) \leq \epsilon (x).
$$

Let now $i \in \{2, \ldots ,e'_0\}$. By (\ref{eq369}), we have
$$
\phi_{i}:=\left (U^{-\epsilon (x)}H^{-1}U_{j_i}{\partial F_{p,Z} \over \partial U_{j_i}}\right )_{\overline{m'}}
\subseteq  J(F_{p,Z',W'},E_{W'},W')\overline{S'}.
$$

Applying once again (\ref{eq3635}) and since $\epsilon (x')>\epsilon (x)=\omega (x)$, we get
$$
\mathrm{cl}_{\epsilon (x)}( \{ H_{W'}^{-1}\overline{u}_{i}{\partial F_{p,Z,W'} \over \partial
\overline{u}_{i}} \}_{2\leq i\leq e'_0})\equiv \mathrm{cl}_{\epsilon (x)}(\{\phi_i\}_{2\leq i\leq e'_0})
\ \mathrm{mod} (\{U'_{i'}\}_{i' \not \in F}) \cap G(m_{S'})_{\epsilon (x)}.
$$
If $\phi_{i}\neq 0$ for some $i$, $2 \leq i \leq e'_0$, we get
$$
\epsilon (x') \leq \mathrm{ord}(H_{W'}^{-1}F_{p,Z',W'})
\leq \mathrm{ord}(H_{W'}^{-1}\overline{u}_{i}{\partial F_{p,Z,W'} \over \partial
\overline{u}_{i}}) \leq \epsilon (x),
$$
again a contradiction. Since $\epsilon (x)=\omega (x)$, we have ${\partial F_{p,Z} \over \partial U_j}=0$ for
every $j \in J \backslash J_E$.

Finally, assume that $F_{p,Z} \not \in k(x')^p[U_1, \ldots ,U_n]$.  Let $(d\lambda'_{l'})_{l' \in \Lambda'}$ be
a basis of $\Omega^1_{k(x')}$. By assumption, there exists $l \in \Lambda$ such that
${\partial F_{p,Z} \over \partial \lambda_l}\neq 0$. We may assume w.l.o.g. that $\lambda_l=\lambda'_{l'}$
for some $l'\in \Lambda'$. Arguing as above, we get
$$
\mathrm{cl}_{\epsilon (x)}(  H_{W'}^{-1}{\partial F_{p,Z',W'} \over \partial \lambda_l })
\equiv \mathrm{cl}_{\epsilon (x)}\left (U^{-\epsilon (x)}H^{-1}{\partial F_{p,Z} \over \partial \lambda_l}\right )_{\overline{m'}}
\ \mathrm{mod} (\{U'_{i'}\}_{i' \not \in F}) \cap G(m_{S'})_{\epsilon (x)},
$$
a contradiction and the proof of (1) in the theorem is complete. \\

We now proceed to prove (2). By Proposition \ref{bupformula}(i), we have
$$
H_{W'}^{-1}F_{p,X',W'}\overline{S'}=(U^{-\epsilon (x)}H_W^{-1}F_{p,Z,W})_{\overline{m'}}
=(U^{-\epsilon (x)}H^{-1}F_{p,Z})_{\overline{m'}}.
$$
By (1) in the theorem and Proposition \ref{firstkind}, there is an expansion
$$
F_{p,Z}=\left (\prod_{i=e'_0+1}^{e'}U_{j_{i}}^{H_{j_{i}}}\right )
\sum_{\mathbf{a}\in A}F_{p,Z,\mathbf{a}}(\{U_j\}_{j \in J'_1})\prod_{j \in J_1}U_j^{pa_j},
\ A \subset \N^{J_1},
$$
with $J_1:=\{j_2, \ldots , j_{e'_0}, j_{n'_0+1}, \ldots , j_{n'_1}\}$, $ J'_1:=J \backslash J_1$,
$F_{p,Z,\mathbf{a}}\in  k(x')^p[\{U_j\}_{j \in J'_1}]$. We deduce that
\begin{equation}\label{eq371}
(U^{-\epsilon (x)}H^{-1}F_{p,Z})_{\overline{m'}}= \overline{H'}^{-1}
\left (\sum_{\mathbf{a}\in A}F_{p,Z,\mathbf{a}}(\{{U_j \over U}\}_{j \in J'_1})
\prod_{j \in J_1}({U_j \over U})^{pa_j}\right ) ,
\end{equation}
with $\overline{H'}:=(\prod_{i=2}^{e'_0}\left ({U_{j_i} \over U}\right )^{H_{j_i}})\subseteq \overline{S'}$.
Since $(H_{W'}^{-1}G_{W'}^p)\overline{S'}=0$ by (\ref{eq3671}), there exists $\theta ' \in S'/(u)$ such that
\begin{equation}\label{eq3711}
H_{W'}^{-1}F_{p,Z',W'}\overline{S'}=H_{W'}^{-1}(F_{p,X',W'}+{\theta '}^p)\overline{S'}.
\end{equation}
We deduce from (\ref{eq371}) that there exists a
finite subset $A'  \subset \N^{J_1}$, $A \subseteq A'$ and elements
$$
\theta '_\mathbf{a} \in k(x)[\{{U_j \over U}\}_{j \in J'_1}] \ \mathrm{for} \
\mathrm{every} \  \mathbf{a}\in A'
$$
such that (letting  $F_{p,Z,\mathbf{a}}(\{{U_j \over U}\}_{j \in J'_1})=0$ for
$\mathbf{a}\in A ' \backslash A $) we have:
$$
H_{W'}^{-1}F_{p,Z',W'}\overline{S'}= \overline{H'}^{-1}\left (\sum_{\mathbf{a}\in A'}
(F_{p,Z,\mathbf{a}}(\{{U_j \over U}\}_{j \in J'_1})
+{\theta '_\mathbf{a}}^p) \prod_{j \in J_1}({U_j \over U})^{pa_j}\right ) .
$$
Let $d_\mathbf{a}:= \epsilon (x') + \sum_{i=2}^{e'_0}H_{j_i}-p\mid \mathbf{a}\mid $ for $\mathbf{a} \in A'$.
Since $\overline{\mathrm{ord}}(H_{W'}^{-1}F_{p,Z',W'})=\epsilon (x')$ we have
$$
\overline{\mathrm{ord}}(F_{p,Z,\mathbf{a}}(\{{U_j \over U}\}_{j \in J'_1})+{\theta '_\mathbf{a}}^p)\geq d_\mathbf{a}
$$
for every $\mathbf{a} \in A'$. Taking classes in $G(\overline{m'})$, we define:
$$
\Phi '_\mathbf{a}:= \mathrm{cl}_{d_\mathbf{a}}(F_{p,Z,\mathbf{a}}(\{{U_j \over U}\}_{j \in J'_1})
+{\theta '}_\mathbf{a}^p) \in k(x')[U'_{n'_1+1}, \ldots , U'_{n'}]_{d_\mathbf{a}}.
$$

To conclude the proof, let $I_1:=\{2, \ldots , e'_0, n'_0+1, \ldots , n'_1\}$. We take
$$
\Phi ':={U'_1}^{p (\delta (y)-1)}\left (\prod_{i=e'_0+1}^{e'}{U'_i}^{H_{j_i}}\right )
\sum_{\mathbf{a}\in A'}\Phi '_\mathbf{a} \prod_{i \in I_1}{U'_i}^{pa_{j_i}}
$$
and claim that $\Phi '$ satisfies (2) in the theorem. By the above definition and (1) in
the theorem, we have $\Phi' \in k(x')[{U'_1}^p, \ldots ,{U'_{n'_1}}^p][U'_{n'_1} ,\ldots , U'_{n'}]_{p\delta (x')}$.
Also (\ref{eq3641}) follows immediately from  (\ref{eq3711}). With notations as in the above proof of (1), we have
$$
J(F_{p,Z},E,m_S)=H^{-1}
<\{U_j{\partial F_{p,Z} \over \partial U_j}\}_{j \in J_E \backslash \{j_2, \ldots ,j_{e'_0+1}\}},
\{{\partial F_{p,Z} \over \partial \lambda_l}\}_{l\in \Lambda}>.
$$
Applying once more (\ref{eq3635}), we get
$$
\left.
  \begin{array}{c}
    \mathrm{cl}_{\epsilon (x)}( \{ H_{W'}^{-1}{\partial F_{p,Z,W'} \over \partial
\overline{u'}_{i}} \}_{n'_1 \leq i\leq n'}) \hfill{} \\
    \equiv
\mathrm{cl}_{\epsilon (x)}(U^{-\epsilon (x)}J(F_{p,Z},E,m_S) )_{\overline{m'}}
\ \mathrm{mod} (\{U'_{i'}\}_{i' \not \in F}) \cap G(m_{S'})_{\epsilon (x)}. \\
  \end{array}
\right.
$$
Since $J(F_{p,Z},E,m_S)\neq 0$, we obtain that
$$
{H'}^{-1}{\partial F_{p,Z'} \over \partial U'_i} \not \in (\{U'_{i'}\}_{i' \not \in F}) \cap G(m_{S'})_{\epsilon (x)}
$$
for some $i$, $n'_1 \leq i\leq n'$, and the conclusion follows. This concludes the proof of (2).\\

\noindent {\it Case 2:}  $i_0(m_S)=p-1$  {\it and} ${\cal Y}$ {\it is of the first kind}.
We first take $d=\epsilon (y)$ and
$$
M:=H_{W}^{-1}G_{W}^p,  \ d_0=0.
$$
By Proposition \ref{firstkind}, there is an expansion
$H^{-1}G^p=<\prod_{j\in J}U_j^{pB_j}>$. With notations as in Definition \ref{defomega}, we have
\begin{equation}\label{eq3682}
pb_j -H_j=pB_j, \ j \in J \ \mathrm{and} \ B=\{j \in J : B_j>0\}.
\end{equation}
We deduce:
$$
(0) \neq \overline{M} =(\prod_{j\in B} U_j^{pB_j})\subseteq k(x)[\{U_j\}_{j\in J}]_{\epsilon (x)}.
$$
Let $ I'_0=H_{W'}^{-1}G_{W'}^p$, $d'_0=\mathrm{ord} I'_0$. We have:
\begin{equation}\label{eq368}
I'_0 \overline{S'} = \left ( U^{-\epsilon (x)}\prod_{j\in B}U_j^{pB_{j}}\right )_{\overline{m'}} .
\end{equation}
This proves that $\epsilon (x') \leq \overline{\mathrm{ord}} I'_0 \leq \epsilon (x)$ and equality holds only if
\begin{equation}\label{eq3653}
s' \in \mathbf{Proj}\left ( {k(x)[\{U_j\}_{j\in J}] \over (U_B)}\right ) .
\end{equation}
Suppose that $\epsilon (x')<\epsilon (x)$. Then :
$$
\omega (x')\leq \epsilon (x') \leq  \epsilon (x)-1 \leq \omega (x).
$$
If $\omega (x')=\omega (x)$, then $\omega (x)=\epsilon (x)-1$, so $\kappa (x)\geq 2$. On
the other hand, we have $\omega (x')= \epsilon (x')$ and therefore $\kappa (x')=1$ by
Definition \ref{defomega}. Hence inequality is strict in (\ref{eq364}).
In other terms, it can be assumed from now on that (\ref{eq3653}) holds and that
\begin{equation}\label{eq3654}
\epsilon (x')=\epsilon (x).
\end{equation}

We now resume the argument used in case 1 by taking
$$
M=J(F_{p,X,W},E_W,W),  \ d=\epsilon (y)=\epsilon (x), \ d_0=0.
$$
To begin with, (\ref{eq369}) holds whenever (\ref{eq3671}) applies, i.e. if $H_{j'} \not \in p\N$
for some $j' \in (J')_E$ or if $\delta (y) \not \in \N$. Suppose that $\delta (y) \in \N$ and
$H_{j'} \in p\N$ for every $j' \in (J')_E$. In this case, (\ref{eq367}) reduces to
\begin{equation}\label{eq3681}
J(F_{p,Z',W'},E',W')\equiv  J(F_{p,X',W'},E',W')
\ \mathrm{mod}K'{\hat{S'} \over (u)} ,
\end{equation}
$$
K':= (\prod_{i=2}^{e'_0}{u'_{i}}^{(p-1)b_{j_i} -H_{j_i} + \left \lceil {H_{j_i} \over p}\right \rceil}) \subseteq S'
$$
with notations as in (\ref{eq3682}). We let :
$$
k':= \sum_{j\in J}\left ((p-1)b_j -H_j + \left \lceil {H_j \over p}\right \rceil \right )=\mathrm{ord}_{m_{S'}}K'.
$$
Going back to Definition \ref{defomega}, we have
$$
F_{p,Z} -TF_{p,Z} \in  (\prod_{j\in J}{U_j^{(p-1)b_{j} + \left \lceil {H_{j} \over p}\right \rceil}}G(m_S))_{p\delta (x)}
$$
and we deduce now from (\ref{eq3681}) that
\begin{equation}\label{eq3683}
J(F_{p,Z',W'},E_{W'},W')\overline{S'}
\equiv \left (U^{-\epsilon (x)}J(TF_{p,Z},E,m_S)\right )_{\overline{m'}} \ \mathrm{mod}K'\overline{S'}.
\end{equation}
Note that the previous equation remains valid when $H_{j'} \not \in p\N$
for some $j' \in (J')_E$ or when $\delta (y) \not \in \N$. The proof now goes on as in case 1
and we deduce that $\overline{\mathrm{ord}}I' \leq \omega (x)$; joining (\ref{eq3653}) and (\ref{eq3683}),
we obtain that (\ref{eq3691}) holds, i.e.
$$
s' \not \in \mathbf{Proj}\left ( {k(x)[\{U_j\}_{j\in J}] \over (\mathrm{IMax}(x)\cap k(x)[\{U_j\}_{j\in J}])} \right )
\Longrightarrow  \overline{\mathrm{ord}}I'< \omega (x) .
$$
Equation (\ref{eq3692}) now follows, while (\ref{eq3693}) gets replaced by
\begin{equation}\label{eq3684}
\overline{I'}= <\left \{{H'}^{-1}{\partial F_{p,Z'} \over \partial U'_{j}}\right \}_{j= n'_0+1}^{n'}>
\ \mathrm{mod} ((\{U'_{j'}\}_{j' \not \in F})+ (\mathrm{cl}_{k'}K'))\cap G(m_{S'})_{d'}.
\end{equation}
Finally, we obtain that
$$
{H'}^{-1}(F_{p,Z'}-T' F_{p,Z'}) \subseteq ((\{U'_i\}_{i \not \in F})+ (\mathrm{cl}_{k'}K'))\cap G(m_{S'})_{\epsilon (x')}
$$
and this concludes the proof of the claim, hence of the theorem, as in case 1.\\

\noindent {\it Case 3:} ${\cal Y}$ {\it is of the second kind.} First recall from
Proposition \ref{secondkind} that  $\epsilon (x)-1 = \omega (x)$,
so $\kappa (x) \geq 2$ in particular.
Let $ I'_0:=H_{W'}^{-1}G_{W'}^p$, $d'_1=\mathrm{ord} I'_0$.

Suppose that $i_0(m_S)=p-1$. By Proposition \ref{secondkind}, there exists an expansion
$$
H^{-1}G^p=<U_{j_1}\prod_{j\in B_J}U_j^{pB_j}>, \ j_1 \in (J')_E, \ B_j>0 \ \mathrm{for} \ j\in B_J,
$$
with notations as in Definition \ref{defcone}. By Proposition \ref{bupformula}(v), we have:
\begin{equation}\label{eq3721}
I'_0 S'/(u) = \overline{u}_{j_1}\left ( U^{-\epsilon (y)}\prod_{j\in B}U_j^{pB_{j}}\right )_{m_{S'/(u)}} .
\end{equation}
This proves that $\epsilon (x') \leq \mathrm{ord} I'_0 \leq \epsilon (x)$ and equality holds only if
\begin{equation}\label{eq372}
s' \in \mathbf{Proj}\left ( {k(x)[\{U_j\}_{j\in J}] \over (U_{B_J})}\right ) .
\end{equation}

Suppose furthermore that $\epsilon (x')<\epsilon (x)$. We have:
$$
\omega (x')\leq \epsilon (x') \leq  \epsilon (x)-1 = \omega (x).
$$
If $\omega (x')=\omega (x)$, then $\omega (x')= \epsilon (x')$ and therefore $\kappa (x')=1$ by
Definition \ref{defomega}, so inequality is strict in (\ref{eq364}).
Therefore if $i_0(m_S)=p-1$, it can be assumed that $\epsilon (x')=\epsilon (x)$ and
in particular that (\ref{eq372}) holds.\\

Going back to the general situation of case 3, we now take
$$
M=J(F_{p,X,W},E_W,W),  \ d=\epsilon (y), \ d_0=0.
$$
Note that (\ref{eq3671}) is always valid in this case 3: we either have $i_0(m_S)=p$ or
(\ref{eq3672}) holds for $j'=j_0$. Applying Proposition \ref{bupformula}(v) gives:
$$
J(F_{p,Z',W'},E_{W'},W')\overline{S'}=
\left (U^{-\epsilon (y)}\overline{J}(F_{p,Z,W},E_W,W)\right )_{\overline{m'}}.
$$
With notations as in Proposition \ref{secondkind}, we have
$$
(0)\neq \overline{J}(F_{p,Z,W},E_W,W)=<\{\Phi_{j'}(\{U_j\}_{j\in J})\}_{j'\in J' \backslash (J')_E}>.
$$
We deduce that
\begin{equation}\label{eq374}
J(F_{p,Z',W'},E_{W'},W')\overline{S'}=
<\{\left (U^{-\epsilon (y)}\Phi_{j'}(\{U_j\}_{j\in J})\right )_{\overline{m'}}\}_{j'\in J' \backslash (J')_E}>.
\end{equation}
Since Definition \ref{defcone} gives
$$
C(x,{\cal Y}):= \mathrm{Max}(\overline{J}(F_{p,Z,W},E,W))\cap \{ U_{B_J}=0\},
$$
we deduce that $\overline{\mathrm{ord}}J(F_{p,Z',W'},E_{W'},W')\leq \omega (x)$
and equality holds only if $s' \in PC(x,{\cal Y})$. We obtain:
\begin{equation}\label{eq375}
\epsilon (x')\leq 1+ \mathrm{ord}J(F_{p,Z',W'},E_{W'},W')\leq 1+\overline{\mathrm{ord}}J(F_{p,Z',W'},E_{W'},W')
\leq \epsilon (x).
\end{equation}

Suppose that $s' \not \in PC(x,{\cal Y})$ and $\omega (x')\geq \omega (x)$. Formula (\ref{eq375}) shows that
$\epsilon (x')=\omega (x')=\omega (x)$. If $i_0(m_{S'})=p-1$, we get $\kappa (x')=1$ so
inequality is strict in (\ref{eq364}).  If $i_0(m_{S'})=p$, we may
pick $j'=j_i \in J' \backslash (J')_E$, $e'+1 \leq i \leq n'_0$, such that
$$
\overline{\mathrm{ord}}\left (U^{-\epsilon (y)}\Phi_{j'}(\{U_j\}_{j\in J})\right )_{\overline{m'}}<\omega (x).
$$
By (\ref{eq374}), we have ${H'}^{-1}{\partial F_{p,Z'} \over \partial U'_i}\neq 0$. This is
a contradiction with the assumption $\epsilon (x')=\omega (x')$. Thus it can be assumed that
$s' \in PC(x,{\cal Y})$.

We get $\omega (x')\leq \epsilon (x')\leq \omega (x)$ unless all inequalities in (\ref{eq375}) are
equalities. In this case, we claim that $\omega (x')=\epsilon (x')-1$ and this will conclude the proof.
To prove the claim, we may pick $j_i \in J' \backslash (J')_E$, $e'+1 \leq i \leq n'_0$, such that
$\Phi_{j_i}(\{U_j\}_{j\in J})\neq 0$ by Proposition \ref{secondkind}. Arguing as above, we have
\begin{equation}\label{eq376}
{H'}^{-1}{\partial F_{p,Z'} \over \partial U'_i} \equiv
<\mathrm{cl}_{\omega (x)}\left (U^{-\epsilon (y)}\Phi_{j_i}(\{U_j\}_{j\in J})\right )_{\overline{m'}}>
\ \mathrm{mod}((\{U'_{j'}\}_{j' \not \in F})\cap G(m_{S'})_{\omega (x)},
\end{equation}
and this proves that ${H'}^{-1}{\partial F_{p,Z'} \over \partial U'_i}\neq 0$. If $i_0(m_{S'})=p$,
we get $\omega (x')= \omega (x)$.

If $i_0(m_{S'})=p-1$, we must introduce a truncation operator
$$
T': G(m_{S'})_{p\delta (x')}\rightarrow G(m_{S'})_{p\delta (x')}
$$
as in Definition \ref{defomega} in order to compute $\omega (x')$. In any case,
we have
\begin{equation}\label{eq377}
    {H'}^{-1}{G'}^p \subseteq (U'_{i \not \in F})\cap G(m_{S'})_{\epsilon (x')},
\end{equation}
which follows from the identity $I'_0 S'/(u)=0$ (resp. from (\ref{eq3721})) if $i_0(m_{S})=p$
(resp. if $i_0(m_{S})=p-1$), {\it cf.} beginning of the proof of case 3.

Going back to Definition \ref{defT}, we have
$$
{H'}^{-1}(F_{p,Z'}-T' F_{p,Z'}) \subseteq (\{U'_i\}_{i \not \in F})\cap G(m_{S'})_{\epsilon (x')}.
$$
It now follows from (\ref{eq376}) that
$$
{H'}^{-1}{\partial T' F_{p,Z'} \over \partial U'_i} \equiv
<\mathrm{cl}_{\omega (x)}\left (U^{-\epsilon (y)}\Phi_{j_i}(\{U_j\}_{j\in J})\right )_{\overline{m'}}>
\ \mathrm{mod}((\{U'_{j'}\}_{j' \not \in F})\cap G(m_{S'})_{\omega (x)}.
$$
This proves at last that ${H'}^{-1}{\partial T' F_{p,Z'} \over \partial U'_i}\neq 0$, so
$\omega (x')=\epsilon (x')-1$ and this concludes the proof of the claim, hence of the theorem.
\end{proof}

\subsection{Consequences of the Blowing up Theorem and constructibility.}

In this section, we prove some basic properties of our main invariant
$$
y\mapsto (m(y),\omega (y),\kappa (y))
$$
and of our notion of permissibility. The following theorem expresses the persistence of permissibility
under permissible blowing ups.

\begin{thm}\label{transfstricte}
Assume that $m(x)=p$, $\omega (x)>0$, where $\{x\}=\eta^{-1}(m_S)$. Let ${\cal Y}_0 \supset {\cal Y}_1$
with respective generic point $y_0, y_1$ be permissible centers  {(of the first or second kind)} at $x$ and $\pi : {\cal X}' \rightarrow {\cal X}$
be the blowing up along  ${\cal Y}_1$.

The strict transform ${\cal Y}'_0$ of ${\cal Y}_0$ is permissible at every $x' \in \pi^{-1}(x)$.
\end{thm}

\begin{proof}
By definition of permissibility, we have $m(y_0)=m(y_1)=p$. Let
$W_i=\eta ({\cal Y}_i)$, $i=0,1$ be with notations as in the previous theorem. There exist associated subsets
$J_0 \subset J_1 \subseteq \{1, \ldots ,n\}$ such that $I(W_i)=( \{u_j\}_{j\in J_i})$ for a
certain choice of an adapted r.s.p. $(u_1,\ldots ,u_n)$ of $S$. Let $(u_1,\ldots ,u_n;Z)$ be well adapted coordinates
at $x$. By Proposition \ref{Deltaalg}, the polyhedron
$$
\Delta_{\hat{S}}(h;\{u_j\}_{j \in J_i};Z)=
\mathrm{pr}_{J_i}(\Delta_{S}(h;u_1,\ldots,u_n;Z))\ \mathrm{is} \
\mathrm{minimal},
$$
where $\mathrm{pr}_{J_i}: \ \R^n \rightarrow \R^{J_i}$ denotes the projection on the $(u_j)_{j \in
J_i}$-space, $i=0,1$. In particular, we have ${\cal Y}_i=V(Z, \{u_j\}_{j\in J_i})$, $i=0,1$.
The strict transform $W'_0$ of $W_0$ at $s'$ has normal crossings with $E':=\sigma^{-1}(E)_{\mathrm{red}}$.
Since $m(x')\geq m(y_0)$ for every $x' \in {\cal Y}'_0$, this proves that ${\cal Y}'_0$ is Hironaka-permissible
w.r.t. $E'$. \\

 {For convenience of the reader, we include at the end of the proof a table summing up the possible kinds for the permissible centers  ${\cal Y}_0, {\cal Y}_1$ and ${\cal Y}'_0$. The current proof requires discussing the first two columns of the table in cases 1 and 2 below.}

Applying again Proposition \ref{Deltaalg}  {and the definitions of permissible centers}, we have
\begin{equation}\label{eq311}
\epsilon (y_0) \leq \epsilon (y_1)\leq \epsilon (x) \leq 1 +\epsilon (y_0), \ \epsilon (y_0)\leq \epsilon (x').
\end{equation}

On the other hand, Theorem \ref{bupthm} applied to $\pi$ gives $\epsilon (x')\leq \epsilon (x)+1$ while classifying equality
cases in (1) and (2). Thus ${\cal Y}'_0$ is permissible of the first kind except possibly in the following two cases:\\

\noindent {\it Case 1:} ${\cal Y}_1$ is of the first kind and $\epsilon (x')=\epsilon (x)+1$;

\noindent {\it Case 2:} ${\cal Y}_0$ is of the second kind and $\epsilon (x')=\epsilon (x)$.\\

Since $x' \in {\cal Y}'_0$, we have, with notations as in Theorem \ref{bupthm} ({\it cf.} Notation \ref{indcoord}):
\begin{equation}\label{eq312}
(J_0)_E \subseteq \{j_i, \ 2 \leq i \leq e'_0\},
\ J_0 \backslash (J_0)_E \subseteq \{j_i, \ n'_0+1 \leq i \leq n'_1\}.
\end{equation}
Also, letting $F_0:=\{2 , \ldots , e'_0 \}\cup  \{n'_0+1 , \ldots , n'_1\}$, we have ({\it cf.} Notation \ref{ordbar}):
\begin{equation}\label{eq315}
J_0 \subseteq F_0 \subseteq F =F_0 \cup \{n'_1+1, \ldots ,n'\}.
\end{equation}

\noindent {\it Proof in case 1:} an immediate consequence of Theorem \ref{bupthm}(1) is that :
$$
i_0(m_S)=p, \ {\partial F_{p,Z} \over \partial U_j}=0, \ j \in J_0 \ \mathrm{or} \ j \geq e+1.
$$
This is incompatible with Definition \ref{secondkind}(iii) applied to ${\cal Y}_0$, so ${\cal Y}_0$ is also of the first kind.
By Proposition \ref{firstkind} we deduce that
\begin{equation}\label{eq313}
H^{-1}G^p=0, \ H^{-1}F_{p,Z} \subseteq k(x)[\{U_{j}\}_{j\in J_0}]_{\epsilon (x)}.
\end{equation}
Since $\epsilon (y_0)=\epsilon (x')-1$, we also have
\begin{equation}\label{eq316}
{H'}^{-1}<{G'}^p,F_{p,Z'}>\subseteq (\{U'_{i}\}_{j_i \in J_0})^{\epsilon (y_0)}\cap G(m_{S'})_{\epsilon (x')}.
\end{equation}

We claim that ${\cal Y}'_0$ is permissible of the second kind at $x'$. To prove the claim, note that
(\ref{eq313}) implies that
$$
H_{W_1}^{-1}G_{W_1}^p \subseteq (\overline{u}_{j'})G(W_1)_{\epsilon (x)} \ \mathrm{for} \ \mathrm{some} \ j' \in (J'_1)_E.
$$
Since ${\cal Y}_0$ is permissible of the first kind at $x$, we actually have
$$
H_{W_1}^{-1}G_{W_1}^p \subseteq (\overline{u}_{j'})S/(\{u_j\}_{j \in J_1})[\{U_{j}\}_{j \in J_0}]_{\epsilon (x)}.
$$
Letting $j'=:j_{i'}$, $e'_0+1 \leq i' \leq e$, Proposition \ref{bupformula}(ii) then shows  that
$$
H_{W'_1}^{-1}G_{W'_1}^p \subseteq (\overline{u}_{i'})S'/(u'_1)[\{U'_{i}\}_{j_i \in J_0}]_{\epsilon (x)},
\ W'_1:=\sigma^{-1}(W_1).
$$
In other terms, we have
$$
{H'}^{-1}{G'}^p \subseteq (U'_1, U_{i'})k(x')[\{U'_{i}\}_{j_i \in J_0}],
$$
and this proves that ${\cal Y}'_0$ satisfies property (ii) of Definition \ref{defsecondkind}.
Finally, applying (\ref{eq316}) gives an expansion
$${H'}^{-1}F_{p,Z'} =<\sum_{i=1}^{n'}U'_i \Phi_i (\{U'_{i'}\}_{j_{i'} \in J_0})>.
$$
Then Definition \ref{defsecondkind}(iii) is equivalent to:
$$
\exists i \in J'_0 \cap \{e'+1 , \ldots ,n'\} : \Phi_i \neq 0.
$$
By equation (\ref{eq3642}) in Theorem \ref{bupthm}(2), there exists $i \geq n'_1+1$ (hence
$i \in J'_0$) such that $\Phi_i \neq 0$, since $j_{i'} \in J_0  \Longrightarrow i' \leq n'_1$
by (\ref{eq312}) and this completes the proof in case 1.\\

\noindent {\it Proof in case 2.} Since ${\cal Y}_0$ is permissible of the second kind,
the initial form $\mathrm{in}_{m_S}h  \in G(m_S)[Z]$ satisfies (\ref{eq2628}). The corresponding
integer $j_0$ satisfies $j_0 \not \in J'_0$ and the corresponding family $(\Phi_{j'}(\{U_j\}_{j\in J_0}))_{j' \in J'_0}$
is such that $\Phi_{j'} \neq 0$ for some $j' \in J'_0 \backslash (J'_0)_E$. In order to prove that
${\cal Y}'_0$ is of the second kind at $x'$, we consider two subcases:\\

\noindent {\it Case 2a:} ${\cal Y}_1$ is of the second kind at $x$. Then $j_0 \in J'_1$ and
$\Phi_{j'} \neq 0$ for some $j' \in J'_1 \backslash (J'_1)_E$. By assumption   $\epsilon (x')=\epsilon (x)$,
and we deduce from (\ref{eq3721}) (resp. from (\ref{eq377})) if $i_0(m_S)=p-1$ (resp. if $i_0(m_S)=p$)
that the initial form  $\mathrm{in}_{m_{S'}}h'  \in G(m_{S'})[Z']$ satisfies
\begin{equation}\label{eq317}
{H'}^{-1}{G'}^p  \subseteq  U_{j'_0}k(x')[\{U'_i\}_{j_i\in J_0}]_{\epsilon (y_0)}  \ \mathrm{for} \ \mathrm{some} \ j'_0 \in
\{1, e'_0+1, \ldots ,e'\}
\end{equation}
and Definition \ref{defsecondkind}(ii) is checked for ${\cal Y}'_0$ at $x'$.  Similarly,
Definition \ref{defsecondkind}(iii) is checked from (\ref{eq376}): we have
${H'}^{-1}{\partial F_{p,Z'} \over \partial U'_i}\neq 0$ for any $i$,  $e'+1 \leq i \leq n'_0$
such that $j_i \in J'_1 \backslash (J'_1)_E$ and $\Phi_{j_i}\neq 0$; take $j_i=j'$
with notations as above.\\

\noindent {\it Case 2b:} ${\cal Y}_1$ is of the first kind at $x$. Then $j_0 \in J_1$ and
$\Phi_{j'} = 0$ for any $j' \in J'_1 $. By Proposition \ref{secondkind} and our assumption
$\epsilon (x')=\epsilon (x)$, we have
$$
\omega (x)=\epsilon (y_0)=\epsilon (x)-1=\epsilon (x')-1 \leq \omega (x').
$$
Therefore Theorem \ref{bupthm} implies that $\omega (x')=\omega (x)$. We have
$\kappa (x), \kappa (x') \geq 2$ since $\omega (x)=\epsilon (x)-1$, $\omega (x')=\epsilon (x')-1$.
This is the equality case $(m(x'),\omega (x'),\kappa (x'))= (m(x),\omega (x),\kappa (x))$
discussed in Theorem \ref{bupthm}.

\smallskip

If $i_0(m_S)=p$, we are in the equality case of (\ref{eq3692}). Then (\ref{eq317}) holds and
there exists $i$, $n'_1+1 \leq i \leq n'$ or ($n'_0+1 \leq i \leq n'_1$ and $\Phi_{j_i} \neq 0$) such that
\begin{equation}\label{eq318}
{H'}^{-1}{\partial F_{p,Z'} \over \partial U'_{i}}\neq 0
\end{equation}
by (\ref{eq3695}). We may take here $j_i:=j' \in J'_0 \backslash (J'_0)_E$.
This checks  Definition \ref{defsecondkind}(ii) and (iii) respectively.

If $i_0(m_S)=p-1$, the initial form  $\mathrm{in}_{m_{S'}}h'  \in G(m_{S'})[Z']$ satisfies
$$
{H'}^{-1}{G'}^p  \subseteq  U'_{i_1}k(x')[\{U'_i\}_{j_i\in J_0}]_{\epsilon (y_0)},
$$
where $j_{i_1}:=j_0 \in J'_0$, $2 \leq i_1 \leq e'_0$ and  Definition \ref{defsecondkind}(ii)
is checked. Equation (\ref{eq318}) also remains valid for some $i$, $n'_0+1 \leq i \leq n'$,
in this case: this follows from (\ref{eq3695}) which is still valid (end of the proof of
case 2 of Theorem \ref{bupthm} where (\ref{eq3684}) replaces (\ref{eq3693}).
This checks  Definition \ref{defsecondkind}(iii) and the proof is complete.


\smallskip

 {The following table sums up the different cases occurring in the proof. The proof is immediate for the last  two columns: $\epsilon(x')=\epsilon(y_0)$ in these cases. }

\bigskip
{\sffamily %
\begin{tabular}{|c|c|c|c|c|}
\hline  & kinds &  kinds &  kinds  &  kinds\\
\hline   ${\cal Y}_0$ & 1 & 2 & 2 & 1  \\
\hline   ${\cal Y}_1$ & 1&1 or  2 & 1 or 2 & 1 \\
\hline   ${\cal Y}'_0$  & 2 & 2 & 1 & 1  \\
\hline   case in proof  & case 1 &case  2 & trivial & trivial  \\
\hline
\end{tabular}
}


\end{proof}




We now turn to formal arcs on ${\cal X}$ and their image. Recall that it is assumed all along
this chapter that $m(x)=p$, $\omega (x)>0$ and $\{x\}=\eta^{-1}(m_S)$.

\begin{defn}\label{defformalarc}
A {\it formal arc} \index{formal arc, Definition~\ref{defformalarc}}on $({\cal X},x)$ is a local morphism
$\varphi : \ \mathrm{Spec}{\cal O}\rightarrow (X,x)$,
where $({\cal O},N,l)$ is a complete discrete valuation ring. We denote the closed (resp. generic)
point of $\mathrm{Spec}{\cal O}$ by $O$ (resp. $\xi$) and call {\it support of} \index{support of a formal arc, Definition~\ref{defformalarc}} $\varphi$ the
subscheme $Z(\varphi):=\overline{\{\varphi (\xi)\}} \subseteq ({\cal X},x)$.

The arc $\varphi$ is said to be {\it well parametrized} \index{well parametrized formal arc, Definition~\ref{defformalarc}} if the inclusion
$$
{\cal O}_\xi :={\cal O}\cap k(\varphi (\xi)) \subseteq {\cal O}
$$
induces an isomorphism $\widehat{{\cal O}_\xi}\simeq {\cal O}$. The arc $\varphi$ is said to be
{\it nonconstant} if $\varphi (\xi)\neq x=\varphi (O)$.
\end{defn}

 {Let us note that, up to replacing $ {\cal O}$ by $\widehat{{\cal O}_\xi}$, the  arc $\varphi$ becomes  well parametrized.}

Given a nonconstant formal arc on $({\cal X},x)$, and $\pi : \ {\cal X}'\rightarrow {\cal X}$ a blowing up along a permissible
center ${\cal Y} \subset {\cal X}$ at $x$ such that ${\cal Y} \subsetneq Z(\varphi)$, there exists a unique lifting
$\varphi ': \ \mathrm{Spec}{\cal O}\rightarrow {\cal X}'$. Let
$$
x':=\varphi '(O), \ ({\cal X}_1,x_1):=({\cal X}',x') \ \mathrm{and} \ \varphi_1 : \  \mathrm{Spec}{\cal O}\rightarrow
({\cal X}_1,x_1)
$$
be the induced morphism. The arc $\varphi_1$ is again nonconstant, so the process can be iterated. Let
\begin{equation}\label{eq291}
({\cal X},x)=:({\cal X}_0,x_0) \leftarrow ({\cal X}_1,x_1)\leftarrow \cdots \leftarrow ({\cal X}_r,x_r) \leftarrow \cdots
\end{equation}
be a sequence of such local blowing ups and centers with
\begin{equation}\label{eq292}
x_r \in {\cal Y}_r \subsetneq Z_r(\varphi):=\overline{\{\varphi_r (\xi)\}} \subset {\cal X}_r.
\end{equation}
Note that the local ring ${\cal O}_{{\cal X}_r , \varphi_r (\xi)}$ is independent of $r \geq 0$. In particular,
$m(\varphi_r (\xi))$, $\epsilon (\varphi_r (\xi))$ and $\omega (\varphi_r (\xi))$ are
independent of $r \geq 0$. An important case of such sequences is when taking
${\cal Y}_r=\{x_r\}$ for every $r \geq 0$; then (\ref{eq291}) is called the {\it quadratic sequence along} $\varphi$. \index{quadratic sequence along a formal arc, equations~\eqref{eq291}\eqref{eq292}}

In any case, given a sequence (\ref{eq291}), we let
$$
d (\varphi ) :=\min_{r\geq 0}\{\mathrm{dim}{\cal O}_{{\cal X}_r,x_r}\}.
$$
If $m(x)=p$ and $\omega (x)>0$, Theorem \ref{bupthm} implies that
$$
(m(x_1),\omega (x_1),\kappa (x_1)) \leq (m(x),\omega (x),\kappa (x)).
$$
If $m(x_r)=p$ and $\omega (x_r)>0$ for every $r \geq 0$, we let
$$
m (\varphi ):=p, \ \omega (\varphi) :=\min_{r\geq 0}\{\omega (x_r)\}>0.
$$

\begin{prop}\label{permisarc}
With notations as above, let $\varphi : \ \mathrm{Spec}{\cal O}\rightarrow ({\cal X},x)$ be a
nonconstant well parametrized formal arc on $({\cal X},x)$ whose quadratic sequence is
such that $m (\varphi ) =p$ and $\omega (\varphi )>0$. Then $l|k(x_r)$ is algebraic for
$r>>0$.

Assume that $l|k(x_r)$ is algebraic with finite  {inseparability} degree for some $r\geq 0$. Then
there exists $r_0\geq 0$ such that the following holds: the support
$Z_r (\varphi)$ is Hironaka-permissible at $x_r$ and $\epsilon (x_r)=\epsilon (x_{r_0})$ for every $r\geq r_0$;
furthermore exactly one of the following conditions is satisfied:
\begin{itemize}
  \item [(1)] $Z_r (\varphi)$ is permissible of the first kind at $x_r$ for every $r\geq r_0$;
  \item [(2)] there exists a finite sequence (\ref{eq291}):
  $$
  ({\cal X}_{r_0},x_{r_0})=:({\cal X}',x') \leftarrow ({\cal X}'_1,x'_1)\leftarrow \cdots \leftarrow ({\cal X}'_{r_1},x'_{r_1})=:(\tilde{{\cal X}},\tilde{x})
  $$
  of local blowing ups with permissible centers of the first kind contained in and of codimension one in the successive strict
  transforms of $Z_{r_0} (\varphi)$, such that the quadratic sequence along $\varphi$:
  $$
  (\tilde{{\cal X}},\tilde{x})=:(\tilde{{\cal X}}_0,\tilde{x}_0) \leftarrow (\tilde{{\cal X}}_1,\tilde{x}_1)\leftarrow \cdots
  \leftarrow (\tilde{{\cal X}}_r,\tilde{x}_r) \leftarrow \cdots
  $$
  has the following properties for every $r \geq 0$:
  \begin{itemize}
    \item [(a)] $\epsilon (\tilde{x}_r)=\epsilon (x_{r_0})$;
    \item [(b)] $\mathrm{dim}{\cal O}_{\tilde{Z}_r (\varphi), \tilde{x}_r} =
    \mathrm{dim}{\cal O}_{Z_{r_0} (\varphi), x_{r_0}}\geq 2$;
    \item [(c)] $\tilde{Z}_r (\varphi)$ is permissible of the second kind at $\tilde{x}_r$
    (resp. $\omega (\tilde{x}_r)=0$) if $\epsilon (x_{r_0})\geq 2$ (resp. if $\epsilon (x_{r_0})=1$).
  \end{itemize}
  \end{itemize}
\end{prop}

\begin{proof}
It can be assumed without loss of generality that
$$
d (\varphi) =\mathrm{dim}{\cal O}_{{\cal X},x}, \ m(x)=p \ \mathrm{and} \ \omega (x)=\omega (\varphi) >0.
$$
Since $m (\varphi ) =p$ and $\omega (\varphi )>0$, we let
$\eta_r : \ ({\cal X}_r,x_r)\rightarrow \mathrm{Spec}S_r$ be the corresponding projection,
$I_r (\varphi)\subseteq S_r$ be the ideal of $W_r (\varphi):=\eta_r(Z_r (\varphi))$. We drop the
reference to $\varphi$ in what follows in order to avoid cumbersome notations.\\

For $f \in m_{S_0}$, $f \not \in I_0$ we denote by $\overline{f} \in {\cal O}$, $\overline{f} \neq 0$
its image by $\varphi^\sharp$. Let $v$ be the discrete valuation associated with ${\cal O}$ and let
$$
M_r:= \{v (\overline{f}), \ f \in S_r \backslash I_r\}
$$
be the semigroup of values of $S_r$ w.r.t. $v$. The group generated by $M_r$ is the value group of the
restriction $v_{|\overline{K}}$ to $\overline{K}=QF(S/I_0)$, hence independent of $r\geq 0$,
and is denoted by $a\Z \subseteq v(N) \Z $, $a \in \N$.

Suppose that  $M_0 \neq a\N$. Let
$\alpha \geq 2$, $\beta \in \N \backslash \alpha \N$ be defined by:
\begin{equation}\label{eq280}
a \alpha :=\min\{M_0 \backslash (0)\},  \ a \beta :=\min\{M_0 \backslash a\alpha  \N\}.
\end{equation}
We pick $u,w \in m_{S_0}$ such that $v(\overline{u})=a \alpha $, $v(\overline{w})= a\beta $. Obviously $u$ is
a regular parameter of $S$ and $wu^{-1} \in m_{S_1}$. Suppose $M_1 \neq a\N$. There are associated integers
$\alpha_1, \beta_1$ as in (\ref{eq280}) which satisfy $(\alpha_1 , \beta_1)<(\alpha ,\beta)$ for the
lexicographical ordering. This can repeat only finitely many times so we get $M_r=a\N$ for some $r \geq 0$.
W.l.o.g. it can be assumed that $M_0 =a \N$. \\

Let $(u_1, \ldots ,u_n)$ be a r.s.p. of $S=S_0$ which is adapted to $E=\mathrm{div}(u_1 \cdots u_e)$.
Without loss of generality, it can be assumed that $v(\overline{u}_e)=a$. Up to renumbering coordinates,
there exists $e (\varphi)$, $0 \leq  e(\varphi) < e$ such that
$$
(u_1, \ldots ,u_{e(\varphi)}) \subseteq I:=I_0 , \ u_j \not \in I \ \mathrm{for} \ e (\varphi) +1 \leq j \leq e.
$$
For $j$, $e (\varphi) +1 \leq j \leq e-1$, let $v(\overline{u}_j)=:a \alpha_j $, $\alpha_j \geq 1$. Note that
$u_ju_e^{-\alpha_j}$ is a unit in $S_{\alpha_j}$; in other terms, replacing $S$ by $S_{\max\{\alpha_j\}}$,
it can be assumed that $e (\varphi) =e-1$. \\

Let $f \in m_{S_0} \backslash I_0$ and write $f=u_e^{\alpha_r(f)}f_r \in S_r$, where $u_e$ does not divide $f_r$ in $S_r$
and note that
$$
f_r \in m_{S_r} \Longrightarrow v(\overline{f})> \alpha_r(f)v(\overline{u}_e)\geq ar  .
$$
Since $M_0 =a \N$, there exists $r \geq 0$ such that $f_r$ is a unit. This implies that for every ideal
$\overline{J} \subseteq S_0/I_0$, $\overline{J}S_r/I_r$ is a principal ideal for $r>>0$.
This is a well known characterization of valuation rings, i.e.
\begin{equation}\label{eq281}
{\cal O}_{v_{|\overline{K}}} = \bigcup_{r \geq 0}S_r /I_r .
\end{equation}
Let $l_0$ be the residue field of the valuation $v_{|\overline{K}}$. Then $l|l_0$ is algebraic (of
degree at most $p$) and  $l_0|k(x_r)$ is algebraic for $r>>0$ by (\ref{eq281}).
This proves the first statement in the theorem. We thus may assume from now on, again by (\ref{eq281}),
that
\begin{equation}\label{eq282}
l_0| k(x_0) \ \mathrm{is} \ \mathrm{separable} \ \mathrm{algebraic}.
\end{equation}

Let $S^{\mathrm{sh}}$ be the strict
Henselization of $S$, so $l^{\mathrm{sh}}:=S^{\mathrm{sh}}/m_{S^{\mathrm{sh}}}$ is the
separable algebraic closure of $l$. The residue action induces an isomorphism
$$
\mathrm{Gal}(S^{\mathrm{sh}}|S^\mathrm{h})\simeq \mathrm{Gal}(l^{\mathrm{sh}}|k(x))
$$
where $S^\mathrm{h}$ is the Henselization of $S$. Let $\tilde{S}$ be the fixed subring of
$S^{\mathrm{sh}}$ by the inverse image of $\mathrm{Gal}(l^{\mathrm{sh}}|l_0)$ under the previous
group morphism. Then $S \subset \tilde{S}$ is a local ind-\'etale map such that
$l_0=\tilde{S}/m_{\tilde{S}}$. In particular $S \subset \tilde{S}$ is regular \cite{ILO} Theorem I.8.1(iv).
Since ${\cal O}$ is Henselian and $l_0 \subseteq l={\cal O}/N$, the morphism $\varphi$ factors through $\tilde{S}$.

Recall Notation \ref{notageomreg1} and Notation \ref{notaprime} for the regular local
base change $S \subset \tilde{S}$.
We apply Theorem \ref{omegageomreg} with $\tilde{s}:=m_{\tilde{S}}$ and get:
$$
m(\tilde{x})=m(x)=p, \ \omega (\tilde{x})=\omega (\varphi)>0 \ \mathrm{and}
\ \epsilon (\tilde{x})=\epsilon (x)>0,
$$
the right hand side equality holding because $\tilde{n}=n$.
Applying Theorem \ref{initform}, $\tilde{{\cal X}}=\mathrm{Spec}(\tilde{S}[X]/(\tilde{h}))$ is irreducible, so
in the separable case (case (b) of assumption $\mathbf{(G)}$), the $G=\Z /p$-action extends uniquely to
$\tilde{{\cal X}}$ and $\mathbf{(G)}$ holds for $(\tilde{S},  \tilde{h},\tilde{E})$.
This proves that $(\tilde{S},  \tilde{h},\tilde{E})$ satisfies the
assumption of the Proposition, all other assumptions being trivially satisfied.

Now $W_0 \times_{k(x_0)}\mathrm{Spec}l_0$ may be reducible, but
$W_r\times_{k(x_r)}\mathrm{Spec}l_0$ is irreducible for $r>>0$. After possibly changing indices,
it can be assumed that $W:=W_0\times_{k(x_0)}\mathrm{Spec}l_0$ is irreducible.
Then $W$ has normal crossings with $E$ at $x$ if and only if
$\tilde{W}:=W \times_S\mathrm{Spec}\tilde{S}$ has normal crossings with $\tilde{E}$ at $\tilde{x}$.
Let $\tilde{Z}:=Z \times_S\mathrm{Spec}\tilde{S}$ and $\tilde{z}$ be the
generic point of a component of $\tilde{Z}$. By Theorem \ref{omegageomreg}, we have
$m(\tilde{z})=m(z)$, so $\tilde{Z}$ is Hironaka-permissible at $\tilde{x}$ w.r.t. $\tilde{E}$ if and only if
$Z$ is Hironaka-permissible at $x$ w.r.t. $E$.  In other terms, we may replace
$S$ by $\tilde{S}$ and thus assume that $l_0=k(x_0)$ in order to prove the second statement. \\

Let now
$$
e_r:=\mathrm{dim}_{k(x_r)}{I_r + m_{S_r}^2 \over m_{S_r}^2}\geq e-1, \ t_r :=e_r -(e-1)\geq 0
$$
for $r \geq 0$. It can be assumed w.l.o.g. that $(u_{e+1}, \ldots ,u_{e+t_0})\subseteq I_0$.
We have $e_{r+1}\geq e_r$ for every $r\geq 0$ and let $e_\infty :=\max_{r\geq 0}\{e_r\}$.
It can be assumed w.l.o.g. that $e_0=e_\infty$.

Since $l_0=k(x_r)$ and $M_r=a\N$ for every $r \geq 0$, the ring morphism
$S_r \rightarrow \widehat{{\cal O}_{v_{|K}}}$
factors through $\hat{S_r}$ to a {\it surjective} morphism
$$
\hat{\varphi}_r: \ \hat{S_r} \rightarrow \widehat{{\cal O}_{v_{|K}}}.
$$
Let $\hat{I}_r$ be the kernel of $\hat{\varphi}_r$, so we have
\begin{equation}\label{eq284}
I_r\hat{S_r}\subseteq \hat{I}_r \ \mathrm{and} \ I_r=\hat{I}_r \cap S_r .
\end{equation}
After possibly replacing $S_0$ by $S_r$ for some $r \geq 0$,
it can be assumed that the curve $\mathrm{Spec}(\hat{S_0}/\hat{I_0})$ is transverse to
$\hat{E}=\mathrm{div}(u_1 \cdots u_e) \subset \mathrm{Spec}\hat{S_0}$. We claim that
\begin{equation}\label{eq283}
I_0 = (u_1, \ldots ,u_{e-1}, u_{e+1}, \ldots ,u_{e+t_0}).
\end{equation}

To prove the claim, suppose that $I_0 \neq J_0:=(u_1, \ldots ,u_{e-1}, u_{e+1}, \ldots ,u_{e+t_0})$.
We let  $\hat{u}_j:=u_j$, $1 \leq j \leq e+t_0$ and pick a basis
\begin{equation}\label{eq285}
\hat{I_0}=J_0+(\hat{u}_{e+t_0+1}, \ldots ,\hat{u}_{n})
\end{equation}
of $\hat{I_0}$.
Since $S_0$ is excellent, the ring $(\hat{S}_0/I_0)_{\hat{I_0}}$ is regular, hence reduced. By assumption,
$I_0 \neq J_0$, so there exists $f \in I_0 \backslash J_0$ such that $f$ restricts to a
regular parameter $\overline{f}$ in $\overline{S}:=(\hat{S}_0 /J_0)_{\hat{I_0}}$:
\begin{equation}\label{eq288}
\mathrm{ord}_{\hat{I_0}}f=1, \ \mathrm{ord}_{m_{\overline{S}}} \overline{f}=1.
\end{equation}
Let $F \in \mathrm{gr}_{\hat{I_0}}(\hat{S}_0)\simeq \hat{S}_0/\hat{I}_0 [\{\hat{U}_j\}_{j\neq e}]$
be the initial form of $f$. There is an expansion
$$
F=\sum_{j\neq e}F_j\hat{U}_j, \ F_j \in \hat{S}_0/\hat{I}_0.
$$
By (\ref{eq288}) we have $F_j \neq 0$ for some $j$, $1 \leq j \leq e+t_0$. Suppose that
$$
\exists j_0, \ 1 \leq j_0 \leq e+t_0 \mid \
m:=\min_{j \neq e}\{\mathrm{ord}_{(\overline{u}_e)}F_j\}=\mathrm{ord}_{(\overline{u}_e)}F_{j_0}.
$$
Replacing $f$ with $f-\gamma_{j_0}u_{j_0}u_e^{m}$
for some unit $\gamma_{j_0} \in S_0$ preserves (\ref{eq288}) while increasing
$\mathrm{ord}_{(\overline{u}_e)}F_{j_0}$. Applying finitely many times this procedure, it can be assumed that
\begin{equation}\label{eq289}
m:=\min_{j \neq e}\{\mathrm{ord}_{(\overline{u}_e)}F_j\}<
\min_{j_0\leq e+t_0}\{\mathrm{ord}_{(\overline{u}_e)}F_{j_0}\}.
\end{equation}
By Lemma \ref{equimultiple} below, there exists $r \geq 1$ and a writing
$$
f_r =u_e^{m+r}g_r, \ g_r \not \in (u_e)S_r, \ \mathrm{ord}_{m_{S_r}}g_r=1.
$$
Furthermore the last statement in {\it ibid.} shows that
$\mathrm{in}_{\hat{I_r}}g_r \in (\mathrm{gr}_{\hat{I_r}}\hat{S}_r)_1$ is transverse to the
initial forms $\overline{u}_e^{-r}U_j$, $1 \leq j \leq e+t_0$, $j\neq e$ by (\ref{eq289}). Since
$g_r \in I_r$, this implies that $e_r >e_0$: a contradiction, so claim (\ref{eq283}) is proved.
Since (\ref{eq283}) is stable by further blowing ups, this proves that $W_r$ is
transverse to the reduced preimage of $\mathrm{div}(u_1 \cdots u_e)$ for every $r>>0$.\\

Let $(\hat{u}_1,\ldots ,\hat{u}_n;Z)$ be well adapted coordinates at $x$. There is an
asso\-cia\-ted expansion
$$
h=Z^p+f_{1,Z}Z^{p-1}+ \cdots +f_{p,Z}, \ f_{1,Z}, \ldots , f_{p,Z} \in \hat{S}_0.
$$
We factor out $f_{i,Z}=u_e^{m_i}g_{i,Z}$, $1 \leq i \leq p$, with $g_{i,Z}=0$ or
($u_e$ does not divide $g_{i,Z}$, $m_i \in \N$).
The {\it formal completion} $\hat{S}_1$ of the local blowing up $S_1$ has a r.s.p.
$(\hat{u}'_1, \ldots , \hat{u}'_n)$ given by
$$
\hat{u}'_e =\hat{u}_e =u_e \ \mathrm{and} \ \hat{u}'_j=\hat{u}_j /u_e, \ j\neq e.
$$
Let $Z':=Z/u_e$, $h':=u_e^{-p}h \in S_1[Z']$ define the strict transform $({\cal X}_1,x_1)$, since $m (\varphi)=p$.
We thus have
\begin{equation}\label{eq287}
f_{i,Z'}=u_e^{-i}f_{i,Z}, \ 1 \leq i \leq p.
\end{equation}

By Proposition \ref{originchart}, the polyhedron $\Delta_{\hat{S}_1} (h'; \hat{u}'_1, \ldots , \hat{u}'_n ;Z')$
is minimal. Applying again Lemma \ref{equimultiple} below, it can be assumed w.l.o.g. that
\begin{equation}\label{eq286}
\mathrm{ord}_{m_{\hat{S}_0}}g_{i,Z}=\mathrm{ord}_{\hat{I_0}}g_{i,Z}, \ 1 \leq i \leq p.
\end{equation}

Let $\hat{Z}_0:=V(Z', \hat{I}_0)\subset (\hat{{\cal X}}_0, \hat{x})$ and $\hat{z}$ be its generic point. Suppose that
$\delta (\hat{z})<1$ and let $i_0$ such that $i_0\delta (\hat{z})=\mathrm{ord}_{\hat{I_0}}f_{i_0,Z}<i_0$.
Applying (\ref{eq287}) gives
$$
\mathrm{ord}_{m_{\hat{S}_1}}f_{i_0,Z'}=m_{i_0}+ i_0(\delta (\hat{z})-1)<m_{i_0}.
$$
This can repeat only finitely many times, a contradiction with $m (\varphi ) =p$. Hence
$\delta (\hat{z}) \geq 1$, i.e. $m(\hat{z})=p$. By excellence, this implies that $m(z)=p$.
Therefore $Z_r$ is Hironaka-permissible at $x_r$ for every $r >>0$.\\

Similarly, replacing $S_0$ by $S_r$ for some $r \geq 0$ and arguing as above, it can be assumed that
$$
\epsilon (\hat{z})=\min_{1 \leq i \leq p}\left \{{\mathrm{ord}_{\hat{I_0}}(H(x)^{-i}f_{i,Z}^p) \over i} \right \}
=\epsilon (\hat{x}).
$$
This proves that $\hat{Z}_0$ is permissible of the first kind at $\hat{x}$. Note that this
furthermore implies that  $\epsilon (x_r)=\epsilon (\hat{z})$ for every $r \geq 0$ and the second
statement of the proposition is proved.\\

In order to prove that alternative (1) in the last statement holds, we may also replace
$S$ by $\tilde{S}$ as above and thus assume that $l_0=k(x_0)$. If $\epsilon (z)=\epsilon (\hat{z})$,
then $Z_r$ is permissible of the first kind at $x_r$ (Definition \ref{deffirstkind}(ii)). This proves that
alternative (1) in the proposition is fulfilled or $\epsilon (\hat{z})> \epsilon (z)$ which we may
assume from now on.\\

By Theorem \ref{omegageomreg}(2.ii), we have $\mathrm{dim}Z_r\geq 2$ (statement $\tilde{n}>n$ of
{\it ibid.} applied under the assumption $l_0=k(x_0)$) and
\begin{equation}\label{eq293}
\epsilon (\hat{z})- 1 =\omega (z)= \epsilon (z) =\epsilon (\hat{x})-1=\epsilon (x)-1, \ i_0(\hat{z})=i_0(z)=p.
\end{equation}
We pick again well adapted coordinates $(\hat{u}_1,\ldots ,\hat{u}_n;\hat{Z})$  at $\hat{x}$.
Since $\hat{Z}_0$ is permissible of the first kind at $\hat{x}$, Proposition \ref{firstkind}
(with notations as therein) gives the following property for the initial form
$\mathrm{in}_{m_{\hat{S}_0}}h \in G(m_{\hat{S}_0})[\hat{Z}]$:
$$
H_0^{-1}G_0^p \in
k(\hat{x})[\hat{U}_1,\ldots ,\hat{U}_{e-1}, \hat{U}_{e+1}, \ldots \hat{U}_n]_{\epsilon (\hat{x})}.
$$
Since $i_0(\hat{z})=p$, we have $G_0=0$, i.e. $i_0(\hat{x})=p$.
This proves that Definition \ref{defsecondkind}(ii) is satisfied in any case.

To prove that alternative (2) in the proposition is fulfilled, we first assume that
$l_0=k(x_0)$ as before, then push down the result from $\tilde{S}$ to $S$.
Let $(u_1,\ldots ,u_n;Z)$ be well adapted coordinates  at $x$ and
consider the initial form $\mathrm{in}_{W_0}h =Z^p +F_{p,Z,W_0}\in G(W_0)[Z]$. Let
$$
J:=\{1, \ldots ,e-1,e+1, \ldots ,e+t_0\}.
$$
Since $\epsilon (\hat{z})>\epsilon (z)$, we have $\delta (z)\in \N$ and
\begin{equation}\label{eq294}
G(W_0)=S_0/I_0[\{U_j\}_{j\in J}], \ F_{p,Z,W_0} \in (\hat{S}_0/\hat{I}_0[\{U_j\}_{j\in J}]_{\delta (z)})^p
\end{equation}
by Theorem \ref{omegageomreg}(2.ii). By Proposition \ref{Deltaalg}, the polyhedron
$$
\Delta_{\hat{S}_0}(h;\{u_j\}_{j\in J};Z)=\mathrm{pr}_J(\Delta_{\hat{S}}(h;u_1,\ldots,u_n;Z))
\ \mathrm{is} \ \mathrm{minimal},
$$
where $\mathrm{pr}_J: \ \R^n \rightarrow \R^J$ denotes the projection on the $(u_j)_{j \in
J}$-space. Let
\begin{equation}\label{eq295}
\Phi_j:=H^{-1}_{W_0}{\partial F_{p,Z,W_0} \over \partial \overline{u}_j}\subseteq G(W_0)_{\epsilon (z)}, \ \mathrm{cl}_0\Phi_j =0,
\ j \not \in J, j\neq e,
\end{equation}
since $\epsilon (x)=\epsilon (z)+1$. The local blowing up $S_1$ has a r.s.p. $(u'_1, \ldots , u'_n)$ given by
$$
\left\{
\begin{array}{ccccc}
    u'_j & = & u_j/u_e & \mathrm{if} & j \in J  \\
    u'_e & = & u_e &    &   \\
    u'_j & = & u_j/u_e -\delta_j & \mathrm{if} & j \not \in J, j\neq e \\
\end{array}
\right.
$$
where $\delta_j \in S_0$ is a unit or zero since we are assuming that $l_0=k(x_0)$.
Let
$$
Z':=Z/u_e -\theta ,  \ \theta \in S_1, \ h':=u_e^{-p}h \in S_1[Z']
$$
define the strict transform $({\cal X}_1,x_1)$, with
$\Delta_{S_1}(h';u'_1,\ldots ,u'_n;Z')$ minimal and consider the initial form
$$
\mathrm{in}_{W_1}h ={Z'}^p +F_{p,Z',W_1}\in G(W_1)[Z'], \ G(W_1)=S_1/I_1[\{U'_j\}_{j\in J}].
$$
It is easily derived from (\ref{eq294})(\ref{eq295}) that
$$
\Phi'_j:=H^{-1}_{W_1}{\partial F_{p,Z',W_1} \over \partial \overline{u}'_j}=\overline{u}_e^{-\epsilon (x)}\Phi_j
\subseteq G(W_1)_{\epsilon (z)}, \ j \not \in J, j\neq e .
$$
Applying again Lemma \ref{equimultiple} below, it can be assumed w.l.o.g. that
\begin{equation}\label{eq296}
(\Phi_j=\overline{u}_e^{m_j}\Psi_j, \ \mathrm{cl}_{0}\Psi_j\neq 0)  \ \mathrm{or} \ \Phi_j=0, \ j \not \in J, j\neq e .
\end{equation}
This equation is valid when $l_0=k(x_0)$ and holds for $S$ if and only if it holds for $\tilde{S}$.
We may therefore replace $S$ by $\tilde{S}$ as before.

Let $\mathbf{x}=(x_1, \ldots ,x_n)\in \N^n$ be a vertex of
$\Delta_{S_0}(h;u_1,\ldots ,u_n;Z)$ mapping to a vertex of
$\Delta_{S_0}(h;\{u_j\}_{j\in J};Z)$ with $\sum_{j \in J}x_j=\delta (y)$.
By (\ref{eq294}) we have $x_j \in \N$ for $j \in J$. Suppose that $x_j \in \N$ for
every $j \neq e$. Since $\hat{S}_0/\hat{I}_0\simeq k(x)[[\overline{u}_e]]$, (\ref{eq294}) implies that
$\mathbf{x}$ is solvable: a contradiction. Taking $j$ such that $x_j \not \in \N$, there exists
$j \not \in J$, $j \neq e$ such that $\Phi_j \neq 0$.
This proves that
$$
r_1:= \min \{m_j, j \not \in J, j\neq e : \Phi_j \neq 0\}
$$
is well defined and that we have
\begin{equation}\label{eq297}
\Phi_{p,Z,W_0}: =\overline{u}_e^{-r_1}H_{W_0}^{-1}F_{p,Z,W_0}\subseteq G(W_0)_{\epsilon (z)},
\ \mathrm{cl}_{1}\Phi_{p,Z,W_0}\not \in  (\overline{u}_e)G(W_0)_{\epsilon (z)}.
\end{equation}

If $r_1=0$, then alternative (2) is fulfilled (Definition \ref{defsecondkind}(iii)) since
$$
\overline{J}(F_{p,Z,W_0},E,W_0)=<\{\mathrm{cl}_{0}\Phi_j\}_{j \not \in J, j\neq e}>\neq 0.
$$
by (\ref{eq297}). Note that this situation does not occur if $\epsilon (x_{r_0})=1$, since
$\omega (\varphi )>0$.

Otherwise, we define
$V_0:=V(u_e, I_0)$ and ${\cal Y}_0:=\eta_0^{-1}(V_0)\subset Z_0$. Then ${\cal Y}_0$ is Hironaka-permissible at $x_0$ and
its generic point $y_0$ has $\epsilon (y_0)=\epsilon (x)$ by (\ref{eq297}). Let $\tilde{{\cal X}}_1$ be the blowing
up of ${\cal X}_0$ along ${\cal Y}_0$ and note that $\varphi$ lifts to the point $\tilde{x}_1$ on the strict transform
$\tilde{Z}_1$ of $Z_0$. Let $\tilde{h}:=u_e^{-p}h \in \tilde{S}_1[\tilde{Z}]$
define the strict transform $(\tilde{{\cal X}}_1, \tilde{x}_1)$ of $({\cal X},x)$, $\tilde{W}_1:=\tilde{\eta}_1(\tilde{Z}_1)$.
By Proposition \ref{originchart}, the initial form
$$
\mathrm{in}_{\tilde{W}_1}\tilde{h} =\tilde{Z}^p +F_{p,\tilde{Z},\tilde{W}_1}\in G(\tilde{W}_1)[\tilde{Z}],
\ G(\tilde{W}_1)=\tilde{S}_1/\tilde{I}_1[\{\tilde{U}_j\}_{j\in J}]
$$
satisfies a relation (\ref{eq297}) with associated integer $\tilde{r}_1=r_1-1$. Iterating $r_1$ times
this procedure, we get some $(\tilde{{\cal X}}_{r_1},\tilde{x}_{r_1})$ with initial form
$$
\mathrm{in}_{\tilde{W}_r}\tilde{h}_r =\tilde{Z}_r^p +F_{p,\tilde{Z}_r,\tilde{W}_r}\in G(\tilde{W}_r)[\tilde{Z}_r],
\ G(\tilde{W}_r)=\tilde{S}_r/\tilde{I}_r[\{\tilde{U}_{j,r}\}_{j\in J}]
$$
with $\tilde{U}_{j,r}=\overline{u}_e^{-r_1}U_j$, $j\in J$. We have
\begin{equation}\label{eq298}
\tilde{\Phi}_r:= H_{\tilde{W}_r}^{-1}F_{p,\tilde{Z}_r,\tilde{W}_r})\subseteq G(\tilde{W}_r)_{\epsilon (z)},
\ \mathrm{cl}_{1}\tilde{\Phi}_r \not \in  (\overline{u}_e)G(W_0)_{\epsilon (z)}.
\end{equation}
By Proposition \ref{secondkind}, we now have $\omega (\tilde{x}_{r_1})=\epsilon (z)=\epsilon (x_{r_0})-1\geq 0$.
Thus $\omega (\tilde{x}_{r_1})>0$ if $\epsilon (x_{r_0})\geq 2$ and we are done by the former case $r_1=0$.
Otherwise, $\epsilon (x_{r_0})=1$ and $\omega (\tilde{x}_{r_1})=0$ and the conclusion follows.
\end{proof}

\begin{exam}\label{exampermisarc}
Take $S=k[u_1, u_2,u_3,u_4]_{(u_1,u_2,u_3,u_4)}$ with $k$ a field of characteristic $p>0$. We let:
$$
h=Z^p +u_2^pu_4u_3^p+u_3u_1^p\in S[Z].
$$
Then $(u_1,u_2,u_3,u_4)$ are adapted to $(S,h,E)$, $E:=\mathrm{div}(u_1u_2)$ (Definition \ref {defadapted})
and $(u_1,u_2,u_3,u_4;Z)$ are well adapted coordinates at the closed point $x=(Z,u_1,u_2,u_3,u_4)$
of ${\cal X}=\mathrm{Spec}(S[Z]/(h))$ (Definition  \ref{defwelladapted}).
Indeed, it is easily seen that:
$$
\mathrm{Sing}_p{\cal X}:=\{y \in {\cal X} : m(y)=p\}=V(Z,u_1,u_2) \cup V(Z,u_1,u_3),\  \omega (x)=p.
$$

Let $\vartheta (t):=\sum_{i\geq 1}\lambda_it^i \in k[[t]]$ be a power series which is {\it transcendental}
over $k(t)$. We define a nonconstant well-parametrized $k$-linear formal arc on $({\cal X},x)$ by:
$$
\varphi (Z)=\varphi (u_1)=\varphi (u_3)=0, \ \varphi (u_2)=t, \ \varphi (u_4)=\vartheta (t)^p.
$$
Let $u_j^{(0)}:=u_j$, $1 \leq j \leq 4$. For $r\geq 1$, well adapted coordinates at $x_r$ are
$u_j^{(r)}:=u_j^{(r-1)}/u_2$, $j=1,3$, $u_2^{(r)}:=u_2$ and
$$
v_4^{(r)}:=u_2^{-r}(u_4- \sum_{ip \leq r}\lambda_i^pu_2^{ip}), \
T_r:= u_2^{-r}(Z +(u_3^{(r)})^p\sum_{ip \leq r}\lambda_i^pu_2^{ip}).
$$
Then $\varphi$ lifts through
$$({\cal X}_r,x_r)=\mathrm{Spec}(S_r[T_r]/(h_r),x_r),  \
S_r=S [u_1^{(r)}, u_{3}^{(r)},u_4^{(r)}]_{(u_1^{(r)}, \ldots , v_{4}^{(r)})},
$$
and the strict transform $h_r$ of $h$ is given by
$$
h_r:=T_r^p +(u_2^{(r)})^r\left ((u_2^{(r)})^pv_4^{(r)}(u_3^{(r)})^p+u_3^{(r)}(u_1^{(r)})^p \right ).
$$
We have $Z_r:=V(T_r,u_1^{(r)}, u_3^{(r)})$ for every $r \geq 1$. Note that $Z_r$ is not
permissible at $x_r$. Therefore $\varphi$
fulfills alternative (2) of Proposition \ref{permisarc}.
\end{exam}

\begin{rem}
We do not know if the conclusion of Proposition \ref{permisarc} is still valid for $n \geq 4$ when
removing the assumption ``$l|k(x_r)$ is algebraic with finite inseparable degree for some $r\geq 0$''.

\smallskip

When $n=3$, it can be proved that the above assumption is actually implied by ``$m(\varphi)=p$ and
$\omega (\varphi)>0$''. This is a (very) special case of the proof of Theorem \ref{projthm}. The following
elementary corollary will be used repeatedly.
\end{rem}

\begin{cor}\label{permisarcthree}
Assume that $n=3$. Let $(S,h,E)$ be as before and $x \in {\cal X}$. Let
\begin{equation}\label{eq299}
({\cal X},x)=:({\cal X}_0,x_0) \leftarrow ({\cal X}_1,x_1) \leftarrow \cdots \leftarrow ({\cal X}_r,x_r)\leftarrow \cdots
\end{equation}
be a (possibly infinite) composition of local blowing ups at closed points with
($m(x_r)=p$, $\omega (x_r)>0$ and $k(x_r)=k(x)$) for every $r\geq 0$.
With notations as in Proposition \ref{Hironakastable} and Notation \ref{notaprime},
assume that $(S_r,E_r,h_r)$ is such that
$E_r$ is irreducible for every $r\geq 0$. Then (\ref{eq299}) is finite.
\end{cor}

\begin{proof}
Let $E=\mathrm{div}(u_1)$ and $(u_1,u_2^{(0)},u_3^{(0)};Z^{(0)})$ be well adapted coordinates at $x$.
Since $k(x_r)=k(x)$ and $E_r$ is irreducible for every $r\geq 1$, $S_r$ has well adapted coordinates
$$
(u_1, u_2^{(r)}:=u_2^{(r-1)}/u_1-\gamma_2^{(r)}, u_3^{(r)}:=u_3^{(r-1)}/u_1-\gamma_3^{(r)};
Z^{(r)}:=Z^{(r-1)}/u_1-\phi^{(r)})
$$
where $\gamma_2^{(r)}, \gamma_3^{(r)}, \phi^{(r)} \in S$. Suppose that (\ref{eq299}) is infinite. We let
$$
\hat{u}_j:=u_2- \sum_{r\geq 1}\gamma_j^{(r)}u_1^{(r)}\in \hat{S}, \ j=2,3, \ \mathrm{and} \
\hat{Z}:=Z - \hat{\phi}, \ \hat{\phi}:=\sum_{r\geq 1}\phi^{(r)}u_1^{(r)}\in \hat{S}.
$$
The induced morphism
$$
\varphi : \ \mathrm{Spec}(\hat{S}[Z]/(\hat{u}_2, \hat{u}_3,\hat{Z}))\longrightarrow ({\cal X},x)
$$
is a nonconstant well parametrized formal arc on $({\cal X},x)$ with $l=k(x)$ and whose associated quadratic sequence is
(\ref{eq299}). By Proposition \ref{permisarc}, $Z_r (\varphi)$ is Hironaka-permissible for some $r\geq 0$:
a contradiction with {\bf (E)}, since $Z_r (\varphi)\nsubseteq E_r$.
\end{proof}

The following lemma is elementary and well-known.

\begin{lem}\label{equimultiple}
Let $S$ be a regular local ring of dimension $n \geq 1$
with r.s.p. $(u_1, \ldots ,u_n)$ and
$$
C:=V(u_1, \ldots ,u_{n-1})\subset ({\cal S}_0,s_0):=\mathrm{Spec}S
$$
be a regular curve. Let
$$
({\cal S}_0,s_0) \leftarrow ({\cal S}_1, s_1) \leftarrow \cdots
\leftarrow ({\cal S}_i, s_i)\leftarrow \cdots
$$
be the composition of local blowing ups such that ${\cal S}_i$ is the blowing up of ${\cal S}_{i-1}$
along $s_{i-1}$ and $s_i\in {\cal S}_i$ is the point on the strict transform $C_i$ of $C$ for $i \geq 1$.

Let $f \in S$, $f\neq 0$ and denote $d:=\mathrm{ord}_{C}f$. There exists $m,i_0 \in \N$ such that
for every $i \geq i_0$, there is a decomposition
$$
f=u_n^{m+di}g_i, \ g_i \in S_i:={\cal O}_{{\cal S}_i,s_i} \ \mathrm{and} \
\mathrm{ord}_{C_i}g_i=\mathrm{ord}_{s_i}g_i=d.
$$
Furthermore, the initial form $\mathrm{in}_{C_i}g_i\in (\mathrm{gr}_{I_{C_i}}S_i)_d$ is the strict transform of
$$
\mathrm{in}_{C}f \in (\mathrm{gr}_{I_{C}}S)_d\simeq S/(u_1, \ldots ,u_{n-1})[U_1, \ldots ,U_{n-1}]_d.
$$
\end{lem}

\begin{proof}
We have
$S_i=S_{i-1}[u_1^{(i)}, \ldots , u_{n-1}^{(i)}]_{(u_1^{(i)}, \ldots , u_{n}^{(i)})}$, where
$u_j^{(i)}:=u_j^{(i-1)}/u_n^{(i-1)}$, $1 \leq j \leq n-1$, $u_n^{(i)}:=u_n^{(i-1)}$ for
every $i \geq 1$, with $u_j^{(0)}:=u_j$, $1 \leq j \leq n$.
Then $C_i=V(u_1^{(i)}, \ldots , u_{n-1}^{(i)})$ with these notations. There is an expansion
$$
f=(u_n^{(i-1)})^{m_{i-1}}g_{i-1}, \ g_{i-1}:=\sum_{\mathbf{x}\in \mathbf{S} }
\gamma(\mathbf{x})^{(i-1)} (u_1^{(i-1)})^{x_1} \cdots (u_n^{(i-1)})^{x_n} \in S_{i-1},
$$
where $\gamma(\mathbf{x})^{(i-1)}\in S_{i-1}$ is a unit for each $\mathbf{x}\in \mathbf{S}$,
$\mathbf{S}\subset \N^n$ a finite set, $m_{i-1}\in \N$, $g_{i-1} \not \in (u_n^{(i-1)})$.
Since $\mathrm{ord}_{C}f=d$, it can be assumed without loss of generality that
$$
d=\min_{\mathbf{x}\in \mathbf{S} }\{x_1+ \cdots +x_{n-1}\}.
$$
Therefore
$$
d=\mathrm{ord}_{C_{i-1}}g_{i-1}\leq d_{i-1}:=\mathrm{ord}_{s_{i-1}}g_{i-1}
=\min_{\mathbf{x}\in \mathbf{S}}\{\mid \mathbf{x}\mid\}.
$$
Note that the initial form $\mathrm{in}_{C_{i-1}}f$ is given by
$$
\mathrm{in}_{C_{i-1}}f =\sum_{x_1+ \cdots +x_{n-1}=d}\overline{\gamma}(\mathbf{x})^{(i-1)}(\overline{u}_n^{(i-1)})^{x_n}
 (U_1^{(i-1)})^{x_1} \cdots (U_{n-1}^{(i-1)})^{x_{n-1}} ,
$$
where $\overline{\gamma}(\mathbf{x})^{(i-1)}, \overline{u}_n^{(i-1)}\in
S_{i-1}/(u_1^{(i-1)}, \ldots , u_{n-1}^{(i-1)})$ denote the classes of the corresponding elements in $S_{i-1}$.
After blowing up, we get an expansion
$$
f=(u_n^{(i)})^{m_{i-1}+d_{i-1}}g_{i}, \ g_{i}:=\sum_{\mathbf{x}\in \mathbf{S}}
\gamma(\mathbf{x})^{(i-1)} (u_1^{(i)})^{x_1} \cdots (u_{n-1}^{(i)})^{x_{n-1}}
(u_n^{(i)})^{\mid\mathbf{x}\mid -d_{i-1}} \in S_{i}.
$$
Let $A_{i-1}:=\{\mathbf{x}\in \mathbf{S}  : x_1+ \cdots +x_{n-1}< d_{i-1}\}$.
For each $\mathbf{x}\in A_{i-1}$, we have $\mid\mathbf{x}\mid -d_{i-1}<x_n$. We deduce:
$$
0 \leq \min_{\mathbf{x}\in A_i}\{x_n\} < \min_{\mathbf{x}\in A_{i-1}}\{x_n\}.
$$
This proves that there exists $i_0 \geq 0$ such that $A_i=\emptyset$ for every $i\geq i_0$. Then
$d_i=d$ for $i \geq i_0$. This proves the first statement in the lemma,
taking $m:=m_{i_0}-di_0 \geq 0$. Finally, this construction preserves
the initial form $\mathrm{in}_{C}f$, i.e.
$$
\mathrm{in}_{C_{i}}f =\overline{u}_n^{-(m+di) }(\mathrm{in}_{C}f)
\left ( \overline{u}_n^{i}U_1^{(i)}, \ldots , \overline{u}_n^i U_n^{(i)} \right ),
$$
and this concludes the proof.
\end{proof}

\begin{thm}\label{Zariskiopen}
Let ${\cal Y} \subset ({\cal X},x)$ be an integral closed subscheme with generic point $y$. The set
$$
\Omega ({\cal Y}):=\{y' \in {\cal Y} : (m(y'), \omega (y'), \kappa (y'))=(m(y), \omega (y), \kappa (y))\}\subseteq {\cal Y}
$$
contains a nonempty Zariski open subset of ${\cal Y}$.

Let furthermore ${\cal Z} \supset {\cal Y} $ be an integral closed subscheme with generic point $z$ such that ${\cal Z}$
is permissible (of the first or second kind) at $y$. The set
$$
\mathrm{Perm}({\cal Y},{\cal Z}):= \{y' \in {\cal Y} : {\cal Z} \ \mathrm{is} \ \mathrm{permissible} \ \mathrm{at} \ y'\}\subseteq {\cal Y}
$$
contains a nonempty Zariski open subset of ${\cal Y}$.
\end{thm}

\begin{proof}
Our function $(m,\omega ,\kappa)$ refines the multiplicity function $m$ on ${\cal X}$,
and our notion of permissible blowing up refines the Hironaka-permissibility.
We may thus apply the well known {\it constructibility} of multiplicity and Hironaka-permissibility.
It is therefore sufficient to prove the first statement when $m(y)=p$. For the second statement, we take a
nonempty Zariski open set ${\cal U}_1 \subseteq {\cal Y}$ such that ${\cal Z}$ is
Hironaka permissible at every $y' \in {\cal U}_1$.

Let $W:=\eta ({\cal Y})$, $s:=\eta (y)$, $W_{\cal Z}:=\eta ({\cal Z})$ for the second statement.
We pick an adapted r.s.p. $(u_1,\ldots ,u_{n_s})$ of $S_s$, where $E_{s}=\mathrm{div}(u_1 \cdots u_{e_s})$.
For every $y' \in {\cal U}_1$ there exists an adapted r.s.p.
$(u_1,\ldots ,u_{n_{y'}})$ of $S_{\eta (y')}$ (i.e. $E_{\eta (y')}=\mathrm{div}(u_1 \cdots u_{e_{y'}})$,
$e_{y'}\geq e_s$) such that $S_s$ is the localization of $S_{\eta (y')}$ at some prime
$$
I(W_{y'})=(\{u_j\}_{j \in J_{y'}}), \ J_{y'}\subseteq \{1, \ldots ,n_{y'}\}.
$$
After possibly shrinking ${\cal U}_1\subseteq {\cal Y}$, it can be assumed without loss of generality that
$e_{y'}=e_s$ for every $y' \in {\cal U}_1$.

We now {\it choose} any point $y_0 \in {\cal U}_1$. Let $(u_1,\ldots ,u_{n_0};Z)$ be
well adapted coordinates at $y_0$, $s_0:=\eta (y_0)$, $S_0:=S_{s_0}$. There is a corresponding expansion
$$
h=Z^p+f_{1,Z}Z^{p-1}+ \cdots +f_{p,Z} \in S_{0}[Z],  \ f_{1,Z}, \ldots , f_{p,Z} \in S_{0}.
$$
After possibly restricting again ${\cal U}_1$, we may assume that the rational functions
$u_1,\ldots ,u_{n_0}, f_{1,Z}, \ldots , f_{p,Z}$ are regular at $\eta (y')$ for every $y' \in {\cal U}_1$.
Moreover, we have in $S_{\eta (y')}$
$$
I(W)=( \{u_j\}_{j\in J}) \  (\mathrm{and} \  I(W_{\cal Z})=( \{u_j\}_{j\in J_{\cal Z}})
\ \mathrm{for} \ \mathrm{the} \ \mathrm{second} \ \mathrm{statement})
$$
with $J_{\cal Z} \subseteq J =\{1, \ldots ,n\}$, $n_{y'}\geq n$, subsets which do not depend on $y'$.
We fix an associated expansion at $s_0$:
$$
f_{i,Z}=\sum_{\mathbf{x}\in \mathbf{S}_i}\gamma(i,\mathbf{x})
    {u_1^{ix_1} \cdots u_{n_0}^{ix_{n_0}}} \in S_{0}, \ 1 \leq i \leq p,
$$
with $\mathbf{S}_i \subset ({1\over i}\N)^{n_0}$ finite and $\gamma(i,\mathbf{x})\in S_{0}$
a unit for each $\mathbf{x}\in \mathbf{S}_i$. After
possibly restricting again ${\cal U}_1$, it may also be assumed that each $\gamma(i,\mathbf{x})$
appearing in some $f_{i,Z}$, $1 \leq i \leq p$, is a regular function at $\eta (y')$.
By Proposition \ref{Deltaalg}, the polyhedra
\begin{equation}\label{eq320}
    \Delta_{S_0}(h;\{u_j\}_{j\in J};Z) \ (\mathrm{and} \ \Delta_{S_0}(h;\{u_j\}_{j\in J_{\cal Z}};Z))
   \ \mathrm{are} \ \mathrm{minimal}.
\end{equation}

We define $A_{i}\subset ({1\over i}\N)^{J}$ (and $A_{i,{\cal Z}}\subset ({1\over i}\N)^{J_{\cal Z}}$ for the second
statement) to be the respective images of $\mathbf{S}_i$ by the projections $\mathrm{pr}_{J}: \R^{n_0} \rightarrow \R^{J}$
and $\mathrm{pr}_{J_{\cal Z}}: \R^{n_0} \rightarrow \R^{J_{\cal Z}}$. Given $\mathbf{a}\in A_{i}$, we let:
$$
\gamma(i,\mathbf{a}):=
\sum_{\mathrm{pr}_{J}(\mathbf{x})=\mathbf{a}}\gamma(i,\mathbf{x})\prod_{j \not \in J}u_j^{ix_j}\in S_0.
$$
By definition of $\epsilon (y)$, we have:
\begin{equation}\label{eq323}
\epsilon (y)= p\min_{1\leq i \leq p}\min_{\mathbf{a} \in A_{i}}
\{\mid \mathbf{a} \mid : \gamma(i,\mathbf{a})\neq 0\} -\sum_{j=1}^{e_s}H_j.
\end{equation}

Let $B \subset \Q^{n}$ be the set of $(i,\mathbf{a})$ achieving equality on the right hand side
of (\ref{eq323}). The initial form polynomial $\mathrm{in}_{m_{S_s}}h$ is thus of the form
\begin{equation}\label{eq325}
\mathrm{in}_{m_{S_s}}h =Z^p +\sum_{(i,\mathbf{a})\in B}\overline{\gamma}(i,\mathbf{a})
\prod_{j \in J}U_j^{ia_{j}}Z^{p-i} \in G(m_{S_s})[Z] ,
\end{equation}
where $\overline{\gamma}(i,\mathbf{a})$ denotes the image in $k(y)$. Let
$$
B_0:=\{(i,\mathbf{a})\in B : \exists (i,\mathbf{a})\in B, i \neq p \ \mathrm{or} \
(i=p \ \mathrm{and} \  \mathbf{a} \not \in \N^{J})\}.
$$

\noindent {\it Case 1.} Suppose that $B_0 \neq \emptyset$. We define:
$$
{\cal U}:=\{ y'\in {\cal U}_1 : \forall (i,\mathbf{a}) \in B_0, \overline{\gamma}(i,\mathbf{a}) \ \mathrm{is} \ \mathrm{a}
\ \mathrm{unit} \ \mathrm{in} \ S_{\eta (y')}\}.
$$
Since $\gamma(i,\mathbf{a})$ is nonzero for $(i,\mathbf{a})\in B$ by (\ref{eq323}),
${\cal U}$ is a nonempty Zariski open subset of ${\cal Y}$. To $y' \in {\cal U}$, we associate
$\mathbf{x} \in \Delta_{S_{\eta (y')}}(h; u_1,\ldots ,u_{n_{y'}};Z)$ (depending on
$(i,\mathbf{a})$) by
$$
\left\{
  \begin{array}{ccccc}
    x_j & = & a_{j} & \mathrm{if} & j\in J \\
    x_j & = & 0     & \mathrm{if} & j \not \in J \\
  \end{array}
\right.
$$
Computing initial forms from Definition \ref{definh} with $\alpha_{y'}  :=(1, \ldots ,1)\in \R^{n_{y'}}$,
$\delta_{\alpha_{y'}} (h;u_1,\ldots ,u_{n_{y'}};Z)=\delta (y)$, the corresponding initial
form polynomial
\begin{equation}\label{eq321}
\mathrm{in}_{\alpha_{y'}}h =Z^p +\sum_{i=1}^pF_{i,Z,\alpha_{y'} }Z^{p-i} \in G(m_{S_{\eta (y')}})[Z]
\end{equation}
is such that $F_{i,Z,\alpha_{y'} } \neq 0$ for some $i\neq p$ or
$F_{p,Z,\alpha_{y'} }\not \in k(y')[U^p_1, \ldots ,U^p_{n_{y'}}]$. Therefore $\delta (y')=\delta (y)$
and we deduce that
\begin{equation}\label{eq324}
\epsilon (y')=\epsilon (y) \ \mathrm{for} \ \mathrm{every} \ y' \in {\cal U}.
\end{equation}

To prove the first statement, note that we are already done by (\ref{eq324}) if $\epsilon (y)=0$.
Assume now that $\epsilon (y)>0$. If $i_0(y)=p-1$,
there exists some $(p-1, \mathbf{a}_0)\in B_0$ for some $\mathbf{a}_0\in \N^J$. Let $y' \in {\cal U}$ and pick well adapted
coordinates $(u_1,\ldots ,u_{n_{y'}};Z_{y'})$ at $y'$. The corresponding initial form polynomial
$$
\mathrm{in}_{m_{S_{\eta (y')}}}h =Z_{y'}^p -G_{y'}^{p-1}Z_{y'}+F_{p,Z_{y'}} \in G(m_{S_{\eta (y')}})[Z_{y'}]
$$
is such that $<G_{y'}>=<U^{\mathbf{a}_0}>$ (resp. $G_{y'}=0$) if $i_0(y)=p-1$
(resp. if $i_0(y)=p$). We have
$$
F_{p,Z_{y'}}=\sum_{(p, \mathbf{a})\in B_0}\lambda_{y'} (p,\mathbf{a})U^{\mathbf{a}}+ \Psi_{y'}
\subseteq G(m_{S_{\eta (y')}})_{\epsilon (y)},
$$
where $\lambda_{y'} (i,\mathbf{a})\in k(y')$, $\lambda_{y'} (i,\mathbf{a})\neq 0$,
$\Psi_{y'}\in k(y')[\{U^p_j\}_{j\in J}]$ for every $(p,\mathbf{a}) \in B_0$ and every $y' \in {\cal U}$.
Comparing with Definition \ref{defomega}, we have $\omega (y')=\omega (y)$, $\kappa (y')=1$ if
$\kappa (y)=1$ for $y' \in {\cal U}$. This proves the first statement in case 1.\\

For the second statement, we are also done if $\epsilon (z)=\epsilon (y)$, i.e. if ${\cal Z}$ is of the first kind at $y$.
Suppose that ${\cal Z}$ is permissible of the second kind at $y$. In particular, we have $\epsilon (y)>0$.
There exist $j_1(y) \in J \backslash J_{\cal Z}$ and $j'(y)\in J \backslash J_{\cal Z}$, $j'(y)\geq e_s+1$,
satisfying the conclusion of Proposition \ref{secondkind}. Let $y' \in {\cal U}$ and pick well adapted
coordinates $(u_1,\ldots ,u_{n_{y'}};Z_{y'})$ at $y'$. The corresponding initial form polynomial
(\ref{eq324}) again satisfies
$$
H_{y'}^{-1}G_{y'}^p\subseteq U_{j_1(y)}k(y')[U_1,\ldots ,U_{n_{y'}}]_{\epsilon (y)}
$$
and there is an expansion
$$
H_{y'}^{-1}F_{p,Z_{y'}}=<\sum_{j' \in J'}U_{j'}\Phi_{j'}(\{U_j\}_{j\in J})+ \Psi(\{U_j\}_{j\in J})>
\subseteq G(m_{S_{\eta (y')}})_{\epsilon (y)}
$$
with $\Phi_{j'(y_0)} \neq 0$, hence ${\cal Y}$ is permissible of the second kind at $y'$ and the conclusion follows.\\

\noindent {\it Case 2.} Suppose on the contrary that $B_0=\emptyset$. By (\ref{eq325}), we have
\begin{equation}\label{eq328}
\mathrm{in}_{m_{S_s}}h =Z^p +\sum_{(p,\mathbf{a})\in B}\overline{\gamma}(p,\mathbf{a})
\prod_{j \in J}U_j^{pa_{j}} \in G(m_{S_s})[Z]
\end{equation}
and this proves that
\begin{equation}\label{eq322}
\delta (y)\in \N , \ \omega (y)=\epsilon (y) \ \mathrm{and} \ \kappa (y)\geq 2.
\end{equation}
Since $(\{u_j\}_{j\in J};Z)$ are well adapted coordinates at $y$, there exists a vertex
$\mathbf{a}_0\in \Delta_{S_s} (h;\{u_j\}_{j\in J};Z)$, $(p,\mathbf{a}_0)\in B$ which is
not solvable, i.e. $\overline{\gamma}(p,\mathbf{a}_0)\not \in k(y)^p$. Let $B_1 \subseteq B_0$ be the
nonempty subset defined by
$$
B_1:=\{ (p,\mathbf{a})\in B : \overline{\gamma}(p,\mathbf{a})\not \in k(y)^p\}.
$$
Given $(p,\mathbf{a})\in B_1$, we define a morphism:
$$
\eta_{(p,\mathbf{a})}: \ {\cal Y}_{(p,\mathbf{a})}:=
\mathrm{Spec}\left ({{\cal O}_{{\cal U}_1}[T] \over (T^p -\overline{\gamma}(p,\mathbf{a}))}\right )
\longrightarrow {\cal U}_1 .
$$
Note that ${\cal Y}_{(p,\mathbf{a})}$ is integral and $\eta_{(p,\mathbf{a})}$ is finite and purely inseparable. We define:
$$
{\cal U}:=\{ y'\in {\cal U}_1 : \forall (p,\mathbf{a}) \in B_1, \eta_{(p,\mathbf{a})}^{-1}(y')_\mathrm{red} \ \mathrm{is} \ \mathrm{a}
\ \mathrm{regular} \ \mathrm{point} \ \mathrm{of} \ {\cal Y}_{(p,\mathbf{a})}\}.
$$
Since ${\cal Y}_{(p,\mathbf{a})}$ is excellent, its regular locus is a nonempty Zariski open set.
We deduce that ${\cal U}$ is a  nonempty Zariski open subset of ${\cal Y}$.

For $y' \in {\cal U}_1$ and $(p,\mathbf{a}) \in B$, we denote by $\lambda_{y'}(p,\mathbf{a}) \in k(y')$ the
residue of $\overline{\gamma}(p,\mathbf{a})$. The property
$$
``\eta_{(p,\mathbf{a})}^{-1}(y')_\mathrm{red}\ \mathrm{is} \ \mathrm{a}
\ \mathrm{regular} \ \mathrm{point} \ \mathrm{of} \ {\cal Y}_{(p,\mathbf{a})} \ "
$$
is equivalently characterized as follows: either
(a) $\lambda_{y'}(p,\mathbf{a}) \not \in k(y')^p$, or
(b) there exists $\delta_{y'}(p,\mathbf{a}) \in {\cal O}_{{\cal Y},y'}$ such
that
$$
v_{y'}(p,\mathbf{a}):=\overline{\gamma}(p,\mathbf{a}) - \delta_{y'}(p,\mathbf{a})^p
$$
is a regular parameter at $y'$.

We now prove the first statement. Let $y' \in {\cal U}$ and pick well adapted coordinates
$(u_1,\ldots ,u_{n_{y'}};Z_{y'})$ at $y'$. Let
$$
B(y'):=\{ (p,\mathbf{a})\in B_1 :  \mathrm{(a)} \ \mathrm{is} \ \mathrm{satisfied}\}.
$$
Suppose that $B(y') \neq \emptyset$. We get
$\delta (y')=\delta (y)$, $i_0(y')=p$ and the initial form polynomial
$\mathrm{in}_{m_{S_{\eta (y')}}}h \in G(m_{S_{\eta (y')}})[Z_{y'}]$ is
$$
\mathrm{in}_{m_{S_{\eta (y')}}}h =Z_{y'}^p +\sum_{(p, \mathbf{a})\in B(y')}\lambda_{y'} (p,\mathbf{a})U^{\mathbf{a}}
+ \Psi_{y'}^p
$$
where $\lambda_{y'} (p,\mathbf{a})\not \in k(y')^p$ and $\Psi_{y'}\in k(y')[\{U^p_j\}_{j\in J}]$.
This shows that
$$
\omega (y')=\epsilon (y')=\epsilon (y)=\omega (y),
$$
the right hand side equality by (\ref{eq322}). Moreover $\kappa (y')\geq 2$, so $y' \in \Omega ({\cal Y})$.

Suppose on the contrary that $B(y') = \emptyset$. We get
$$
\delta (y')=\delta (y)+{1 \over p}, \ i_0(y')=p \ (\mathrm{since} \  \delta (y') \not \in \N)
$$
and the initial form polynomial $\mathrm{in}_{m_{S_{\eta (y')}}}h \in G(m_{S_{\eta (y')}})[Z_{y'}]$ is
$$
\mathrm{in}_{m_{S_{\eta (y')}}}h =Z_{y'}^p + \sum_{(p, \mathbf{a})\in B_1}V_{y'} (p,\mathbf{a})U^{\mathbf{a}}
+ \Psi_{y'},
$$
where $V_{y'} (p,\mathbf{a})\in <U_1,\ldots ,U_{n_{y'}}> \backslash <\{U_j\}_{j\in J}>$,
$\Psi_{y'}\in k(y')[\{U_j\}_{j\in J}]_{p\delta (y)+1}$.
This shows that $\omega (y')=\epsilon (y')-1=\epsilon (y)=\omega (y)$,
applying again (\ref{eq322}). Moreover $\kappa (y')\geq 2$, so $y' \in \Omega ({\cal Y})$.
This concludes the proof of the first statement.\\

For the second statement, note that ${\cal Z}$ is necessarily of the first kind at $y$ in case 2,
since (\ref{eq328}) is not compatible with Proposition \ref{secondkind}. With notations as
above, ${\cal Z}$ is then permissible of the first kind (resp. of the second kind) at $y'$
if $B(y')\neq \emptyset$ (resp. if $B(y')= \emptyset$).
\end{proof}

\begin{cor}\label{constructible}
With notations as above, the function
$$
\iota : {\cal X} \rightarrow \{1, \ldots ,p\}\times \N \times \{0,1,\geq 2\},
\ y \mapsto (m(y), \omega (y), \kappa (y))
$$
is a constructible function on ${\cal X}$. In particular, it takes finitely many distinct values.
\end{cor}

\begin{proof}
This follows from the previous theorem and Noetherian induction on ${\cal X}$.
\end{proof}

\begin{rem}\label{remconstructible}
The constructible sets ${\cal X}_{p,a}:=\{y \in {\cal X} : (m(y),\omega (y))\geq (p, a)\}$, $a \in \N$
are not in general Zariski closed (Example \ref{Exampleconstructible} below).
See next proposition for closedness of the set ${\cal X}_{p,1}$.

We do not know if the sets $\mathrm{Perm}({\cal Y},{\cal Z})$ as in the theorem are constructible subsets of ${\cal Y}$.
An important issue about permissibility is addressed below in Question \ref{questpermissible}.

About a possible extension of our methods to a global Resolution of Singularities statement, we
remark the following: let ${\cal S}$ be an excellent regular domain,
$$
\eta : \ {\cal X} \rightarrow {\cal S}
$$
be a finite morphism, $x \in {\cal X}$ be such that $({\cal X},x)\rightarrow {\cal S}_{\eta (x)}$
satisfies the assumption of Theorem \ref{Zariskiopen}. It is
easily seen that its conclusion extends to some affine neighbourhood ${\cal U}$ of $x$ on ${\cal X}$.
\end{rem}

\begin{exam}\label{Exampleconstructible}
Let $S=k[[u_1,u_2,u_3]]$, $k$ a (nonperfect) field of characteristic $p>0$ and
$\lambda ,\mu \in k $ be $p$-independent. We take:
$$
h=Z^p -(u_1u_2)^{p-1}Z +\lambda u_3^{p} +u_3u_1^{p-1} + \mu u_1^{p}\in S[Z], \ E=\mathrm{div}(u_1u_2).
$$
The coordinates $(u_1,u_2,u_3;Z)$ are well adapted to $(S,h,E)$. Let
$$
x:=(Z,u_1,u_2,u_3), \ y:=(Z,u_1,u_3).
$$
We have $H(x)=(1)$, $m(x)=m(y)=p$, and compute:
$$
\mathrm{in}_{m_S}h=Z^p +\lambda U_3^p +U_3U_1^{p-1} + \mu U_1^{p}, \ i_0(x)=p, \ \omega (x)=\epsilon (x)-1=p-1  .
$$
On the other hand, we have:
$$
\mathrm{in}_{m_{S_{\eta (y)}}}h=Z^p -(U_1\overline{u}_2)^{p-1}Z +\lambda U_3^p +U_3U_1^{p-1} + \mu U_1^{p},
\ i_0(y)=p-1, \ \epsilon (y)=p.
$$
In order to compute $\omega (y)$, we must introduce a truncation operator
$$
T_y: k(y)[U_1,U_3]_{p} \rightarrow k(y)[U_1,U_3]_{p}
$$
as in Definition \ref{defomega} and get $T_yF_{p,Z,y}=\lambda U_3^p$, so $\omega (y)=p >\omega (x)$. This
proves that the set ${\cal X}_{(p,p)}:=\{z \in {\cal X} : (m(z),\omega (z))\geq (p, p)\}$ is {\it not} Zariski closed.
\end{exam}

\begin{prop}\label{omegapositiveclosed}
Let $({\cal X},x)$ be as in the theorem. The set
$$
\Omega_+ ({\cal X}):=\{y \in {\cal X} : (m(y), \omega (y))>(p,0)\}\subseteq {\cal X}
$$
is Zariski closed and of dimension at most $n-2$.
\end{prop}

\begin{proof}
Let $\xi \in {\cal X}$ be the generic point of an irreducible component of
$\eta^{-1}(E)$. Then $(m(\xi), \epsilon (\xi))\leq (p,0)$, so $\xi \not \in \Omega_+ ({\cal X})$.
Therefore it is sufficient to prove that $\Omega_+ ({\cal X})$ is Zariski closed.

We will use the Nagata Criterion to prove openness of ${\cal X} \backslash \Omega_+ ({\cal X})$.
By Theorem \ref{Zariskiopen}, it is sufficient to prove that $\Omega_+ ({\cal X})$ is
stable by specialization. Let $y_0 \rightsquigarrow y_1$ be a specialization in ${\cal X} $
and assume that $y_1 \not \in \Omega_+ ({\cal X})$. Since the multiplicity does not decrease by specialization
\cite{Gi1} Theorem 3.9 p.II-30, we may assume that $m(y_1)=p$. We are done unless $m(y_0)=p$
which we assume from now on. Let ${\cal Y}_0:=\overline{\{y_0\}}$.

By localizing $\eta$ at $\eta (y_1)$, it can be furthermore assumed that $y_1=x$. Arguing by
induction on the dimension of ${\cal Y}_0$, it can be furthermore
assumed that ${\cal Y}_0$ is a curve. Let
$$
({\cal X},x)=:({\cal X}_0,x_0) \leftarrow ({\cal X}_1,x_1)\leftarrow \cdots \leftarrow ({\cal X}_r,x_r) \leftarrow \cdots
$$
be a sequence of local blowing ups  at closed points belonging to the strict transform of ${\cal Y}_0$.
We have $m(x_r)\geq m(y_0)=p$ {\it ibid.}, so $m(x_r)=p$ for every $r \geq 0$. Since $S$ is excellent,
the strict transform of ${\cal Y}_0$ in ${\cal X}_r$ is Hironaka permissible for $r>>0$. By construction,
these maps induce local isomorphisms at $y_0$.

We then have $(m(x_r), \omega (x_r)) \leq (p,0)$ by Proposition \ref{bupomegazero},
hence $\omega (x_r)=0$ since $m(x_r)=p$ for every $r \geq 0$. In other words,
after possibly replacing $({\cal X},x)$ by $({\cal X}_r,x_r)$ for some $r \geq 0$,
it can be assumed that ${\cal Y}_0$ is Hironaka permissible.
Then there exist well adapted coordinates $(u_1,\ldots ,u_n;Z)$ at $x$ such that
$$
I(W_0)=( \{u_j\}_{j\in J_0}), W_0:=\eta ({\cal Y}_0)
$$
with $J_0 =\{1, \ldots ,n\}\backslash \{j'\}$ for some $j'$ (since ${\cal Y}_0$ is a curve).
We let $s_0:=\eta (y_0)$, $S_0:=S_{s_0}$. By Proposition \ref{Deltaalg}, the polyhedron
$\Delta_{S}(h;\{u_j\}_{j\in J};Z)$ is minimal,
so we deduce that $\epsilon (y_0)\leq \epsilon (x)$.

Since $\omega (x)=0$ by assumption, we have $\omega (y_0)=0$
except possibly if $\epsilon (y_0)= \epsilon (x)=1$. Since $\omega (x)=0$,
the initial form polynomial $\mathrm{in}_{W_0}h \in G(m_S)[Z]$ then satisfies
$$
H^{-1}_{W_0}F_{p,Z,W_0}=<\sum_{j \in J_0}\gamma_jU_j > \subseteq G(W_0)_1 =S/I(W_0)[ \{U_j\}_{j\in J_0}],
$$
and there exists $j_0 \in J_0$, $e+1 \leq j_0 \leq n$ such that $\gamma_{j_0}$ is a unit in $S/I(W_0)$.
This gives $\omega (y_0)=0$ if $i_0(y)=p$. If $i_0(y)=p-1$, we must introduce a truncation operator
$$
T_0 : \ G(m_{S_0})_{p\delta (y_0)}\rightarrow  G(m_{S_0})_{p\delta (y_0)},
$$
as in Definition \ref{defomega} in order to compute $\omega (y_0)$. However, $T_0$ proceeds from
Definition \ref{defT} in the special case $p\delta (y_0)=1+ \sum_{j\in J_0}H_j$. Lemma \ref{kerT}
then implies that
$$
H^{-1}_{W_0}\mathrm{Ker}T_0 \subseteq <\{U_j\}_{j\in J_0, j \leq e}> \subset G(m_{S_0})_{p\delta (y_0)}.
$$
Since $j_0 \geq e+1$, we thus have $H_{W_0}U_{j_0} \nsubseteq \mathrm{Ker}T_0$ and this proves that
$\omega (y_0)=0$ as required.
\end{proof}

A very special case of the following question (for $\mu$ a discrete valuation with some extra
assumption) has been answered in the affirmative in  {Proposition} \ref{permisarc} above. See also
Theorem \ref{contactmaxFIN} for a related result.

\begin{quest}\label{questpermissible}
Let ${\cal Y}={\cal Y}_0$ be an integral closed subscheme with generic point $y$, $m(y)=p$, $\omega (y)>0$,
and let $\mu$ be a valuation centered at $m_S$. Does there exist a finite
sequence of permissible local blowing ups along $\mu$:
$$
({\cal X},x)=:({\cal X}_0,x_0) \leftarrow ({\cal X}_1,x_1)\leftarrow \cdots \leftarrow ({\cal X}_r,x_r)
$$
with centers ${\cal Z}_i \subset ({\cal Y}_i,x_i)$, ${\cal Y}_i$ denoting the strict transform of ${\cal Y}$ in $({\cal X}_i,x_i)$,
$0 \leq i \leq r$, such that ${\cal Y}_r$ is permissible at $x_r$?
\end{quest}

\section{Application to Resolution in dimension three.}

In this chapter, we deduce Theorem \ref{mainthm} from Theorem \ref{luthm} and prove corollaries \ref{completeresolution}
and \ref{integralmodel}. Achieving condition {\bf (E)} allows us to use all results from the previous chapters.

\smallskip

While Theorem \ref{mainthm} is global in nature for it states the existence
of a proper morphism resolving singularities of ${\cal X}$, Theorem \ref{luthm} is very local: it only deals
with valuations and the existence of a model which is regular at their center. Deducing the former from
the latter goes back to O. Zariski Fundamental Theorem \cite{Z5} p.539 on patching Local Uniformizations. Zariski proved:

\begin{prop}\label{Fundamental}\textbf{(Zariski)}
Let $K$ be a function field in three variables over an algebraically closed firld $k$
of characteristic zero. Let ${\cal N}$ be a set of valuations of $K$, trivial on $k$.
Let $\Sigma$ be a projective model of $K | k$. If there exists a resolving system of ${\cal N}$ consisting
of two projective models $V$ and $V'$, then there also exists a resolving model
for ${\cal N}$ (i.e. a projective model of $K | k$ on which every valuation of $N$ has
a regular center).
\end{prop}

Proposition \ref{redtoLU} below states the appropriate version of the Fundamental Theorem
in the category of quasi-excellent reduced and separated Noetherian schemes of dimension at most three.
Once the appropriate definitions have been set, the proof goes along the same line as
Zariski's. Zariski could not state this more general result because the main notions (schemes,
proper morphisms and quasi-excellence) were not defined at the time. Furthermore, his
proof relies on the next three propositions which were only known for varieties of characteristic
zero at the time.

\smallskip

We also remark that Zariski's Fundamental Theorem has been enhanced by the first author in
\cite{Co8} in the context of algebraic varieties over arbitrary ground fields. This enhancement
is essential  to obtain (ii) in Theorem \ref{mainthm}.

\smallskip

All results in this chapter are extensions of \cite{CoP1}. The proofs are based on the following
three characteristic free results which can be found respectively in \cite{Ab2} Theorem 3,
a special case of \cite{CoJS} Theorem 0.3 (with $B=\emptyset$) and \cite{CoP1} Proposition 4.2:

\begin{prop}\label{factbir}\textbf{(Abhyankar)}
Let $(R,m)$ and $(R',m')$ be regular two-dimen\-sional local domains
with a common quotient field and such that
$$
R \subseteq R', \ m'\cap R=m.
$$
Then $R'$ is an iterated quadratic transform of $R$.
\end{prop}

\begin{prop}\label{embedsurf}\textbf{(Cossart-Jannsen-Saito)}
Let ${\cal S}$ be a regular Noetherian irreducible scheme of dimension three which is
excellent and $X \hookrightarrow {\cal S}$  be a reduced subscheme.

\smallskip

There exists a composition of blowing ups along integral regular subschemes
$\sigma: \ {\cal S}'\rightarrow {\cal S}$
such that the strict transform $X'\hookrightarrow {\cal S}'$ of $X$ has strict normal crossings
with the reduced exceptional divisor $E$ of $\sigma$. Moreover $\sigma$
restricts to an isomorphism
$$
\pi: \ X' \backslash \sigma^{-1}(\mathrm{Sing}X)\simeq X \backslash \mathrm{Sing}X.
$$
\end{prop}

\begin{prop}\label{principsurf}\textbf{(Cossart-Piltant)}
Let ${\cal S}$ be a regular Noetherian irreducible scheme of dimension three which is
excellent and ${\cal I} \subseteq {\cal O}_{\cal S}$  be a nonzero ideal sheaf.
There exists a finite sequence
$$
    {\cal S}=:{\cal S}(0) \leftarrow {\cal S}(1)\leftarrow \cdots \leftarrow {\cal S}(r)
$$
with the following properties:
\begin{itemize}
  \item [(i)] for each $j$, $0 \leq j \leq r-1$, ${\cal S}(j+1)$ is the blowing up
  along a regular integral subscheme ${\cal Y}(j)\subset {\cal S}(j)$ with
  $$
  {\cal Y}(j)\subseteq \{s_j \in {\cal S}(j): {\cal I}{\cal O}_{{\cal S}(j),s_j} \ \mathrm{is}
  \ \mathrm{not} \ \mathrm{locally} \ \mathrm{principal}\}.
  $$
  \item [(ii)] ${\cal I}{\cal O}_{{\cal S}(r)}$ is locally principal.
\end{itemize}
\end{prop}

\begin{proof}
The assumption ``$X/k$ is quasi-projective'' is not used in the
proof of \cite{CoP1} Proposition 4.2. The equicharacteristic assumption is used only via
the power series expansions used for defining $E$ and the characteristic polygon
``$\Delta ({\cal E};u_1,u_2;y)$ prepared'' on pp.1061-1062 of {\it ibid.}.
But this is also characteristic free by \cite{CoP3} Theorem II.3.
\end{proof}

\subsection{Reduction to local uniformization and proof of the corollaries.}

We now reduce Theorem \ref{mainthm} to its local uniformization form (LU) below.
Let $(A,m,k)$  be a quasi-excellent local domain with quotient field $K$. Recall that
quasi-excellent rings are Noetherian by Definition \cite{EGA2} (7.8.2)
and Remark (7.8.4)(i). We consider the following Local Uniformization problem:\\

\noindent (LU) \index{(LU) Local Uniformization problem} for every valuation  $v$ of $K$, with valuation ring $({\cal O}_v,m_v,k_v)$ such that
$$
A \subset {\cal O}_v \subset K, \ m_v\cap A=m, \ k_v | k \ \mathrm{algebraic},
$$
there exists a finitely generated $A$-algebra $T$, $A \subseteq T \subseteq {\cal O}_v$,
such that $T_P$ is regular, where $P:=m_v \cap T$.

\smallskip

Zariski's proof of the Fundamental Theorem (quoted from \cite{Z5} on p.539) only requires two results:
\cite{Z5} Theorem 7 of section 19 and the Lemma on p. 538. In our characteristic free context,
these are respectively Lemma \ref{elimindet} below and Proposition \ref{factbir}.

\begin{lem}\label{elimindet}
Let $A$ be a reduced excellent Noetherian domain of dimension three and
$$
{\cal X} \longrightarrow \mathrm{Spec}A, \ {\cal Y} \longrightarrow \mathrm{Spec}A
$$
be projective birational morphisms. Denote by $\rho : \ {\cal Y} \cdots \longrightarrow {\cal X}$
the birational correspondence and ${\cal F}\subset {\cal Y}$ its fundamental locus \index{fundamental locus of a birational correspondence, Lemma~\ref{elimindet}}
 {(i.e. the complement of the largest open set of
definition)}. There exists a sequence
\begin{equation}\label{eq509}
{\cal Y}=:{\cal Y}_0 \leftarrow {\cal Y}_1 \leftarrow \cdots \leftarrow {\cal Y}_{r+1}={\cal Y}'
\end{equation}
of blowing ups along regular centers ${\cal Z}_i\subseteq {\cal Y}_i$ such that
\begin{itemize}
  \item [(i)] ${\cal Z}_i$ is fundamental for $\rho_i : \ {\cal Y}_i \cdots \longrightarrow {\cal X}$, $0 \leq i\leq r$;
  \item [(ii)] $\rho \circ \pi $ is a morphism on $\pi^{-1}({\cal F} \cap \mathrm{Reg}{\cal Y})$, where
  $\pi : \ {\cal Y}'\rightarrow {\cal Y}$ is the composed map.
\end{itemize}
\end{lem}
\begin{proof}
This lemma rephrases \cite{CoP1} Proposition 4.7, using the characteristic free Proposition \ref{principsurf}.
We denote by
$$
{\cal F}^\circ := {\cal F}\cap \mathrm{Reg}{\cal Y}, \ \mathrm{dim}{\cal F}^\circ \leq 1.
$$
Let $\overline{{\cal F}} \subseteq {\cal F}$ be the Zariski closure of ${\cal F}^\circ$ in ${\cal Y}$ and
${\cal G} \subseteq \overline{{\cal F}}$ be its one-dimensional component (possibly ${\cal G}=\emptyset$).
We construct $\pi$ as a composition of blowing ups along regular subschemes {\it mapping} to $\overline{{\cal F}}$.

\smallskip

\noindent {\it Step 1:} let
\begin{equation}\label{eq5091}
\pi_1: \ {\cal Y}_{i_1}\rightarrow {\cal Y}
\end{equation}
be the minimal composition of blowing ups at closed points such that
the strict transform ${\cal G}'$ of ${\cal G}$ is a disjoint union of regular curves, followed by
the blowing up along ${\cal G}'$. Let
$$
\rho_1 : \ {\cal Y}_{i_1} \cdots \longrightarrow {\cal X}
$$
denote the composed map $\rho \circ \pi_1 $, ${\cal F}_1$ its fundamental locus.
We now denote
$$
{\cal F}_1^\circ := {\cal F}_1\cap \pi_1^{-1}(\mathrm{Reg}{\cal Y})
$$
and $\overline{{\cal F}}_1 \subseteq {\cal F}_1$ its Zariski closure in ${\cal Y}_{i_1}$. Let
furthermore ${\cal G}_1 \subseteq \overline{{\cal F}}_1$ be the union of its one-dimensional irreducible components
{\it whose image in} ${\cal Y}$ {\it has dimension one}.

\smallskip

We now iterate this construction.  Applying a classical result on quadratic sequences
in regular local rings of dimension two (e.g. \cite{ZS2} appendix 5, Theorem 3 and (E) on p.391),
we construct $\pi_n: \ {\cal Y}_{i_n}\rightarrow {\cal Y}$ such that $\rho \circ \pi_n $ is a morphism away from
$$
\pi_n^{-1}(({\cal F} \cap \mathrm{Reg}{\cal Y})\backslash \{x_1, \ldots ,x_k\}),
$$
where $x_1, \ldots ,x_k$ are finitely many closed points.

\smallskip

\noindent {\it Step 2:} let ${\cal Z}$ be the closure of the graph of $\rho \circ \pi_n$. Since ${\cal X}$ is projective,
${\cal Z}$ is isomorphic to the blowing up of ${\cal Y}_n$ along a certain ideal sheaf
${\cal I}_n\subseteq {\cal O}_{{\cal Y}_{i_n}}$. Since $\pi_n^{-1}(\mathrm{Reg}{\cal Y})\subseteq \mathrm{Reg}{\cal Y}_n$,
there exists ${\cal I}\subseteq {\cal O}_{{\cal Y}_{i_n}}$ with
\begin{equation}\label{eq5092}
V({\cal I})\subseteq \pi_n^{-1}(x_1)\cup \ldots \cup \pi_n^{-1}(x_k), \ \mathrm{dim}V ({\cal I})\leq 1,
\end{equation}
such that ${\cal Z}$ is isomorphic to the blowing up of ${\cal Y}_{i_n}$ along ${\cal I}$
above $\pi_n^{-1}(\mathrm{Reg}{\cal Y})$.  Applying Proposition \ref{principsurf} to
${\cal I}\subseteq {\cal O}_{{\cal Y}_{i_n}}$ concludes the proof.
\end{proof}

\begin{prop}\label{redtoLU}
Let ${\cal X}$  be a reduced and separated Noetherian scheme which is quasi-excellent and
of dimension at most three. Let ${\cal X}_1, \ldots , {\cal X}_c$ be the irreducible components
of ${\cal X}$. Assume that (LU) holds for every local ring of the form
$A={\cal O}_{{\cal X}_i,x_i}$  which is of dimension three, $1 \leq i \leq c$.
Then Theorem \ref{mainthm} holds for ${\cal X}$.
\end{prop}

\begin{proof}
Suppose that (i) and (ii) in Theorem \ref{mainthm} have been proved. Apply Proposition \ref{embedsurf} to
$$
X:=\pi^{-1}(\mathrm{Sing}{\cal X})_{\mathrm{red}}\subseteq {\cal X}',
$$
then blow up along $X'$: we get (iii).  {It} remains to prove (i) and (ii).

\smallskip

\noindent {\it Step 1:} it can be assumed that ${\cal X}$ is
irreducible of dimension three.

\smallskip

There is a finite birational morphism
$$
f: \ \coprod_{i=1}^c{\cal X}_i \rightarrow {\cal X},
$$
isomorphic above $\mathrm{Reg}{\cal X}$. The theorem holds for ${\cal X}$
if it holds for each ${\cal X}_i$. Resolution of singularities is known if
$\mathrm{dim}{\cal X}\leq 2$ \cite{L3}, so we may assume that $\mathrm{dim}{\cal X}=3$.

\noindent {\it Step 2:} it can be assumed that ${\cal X}=\mathrm{Spec}A$ is affine.

\smallskip

This is based on Lemma \ref{elimindet}. Consider open sets
${\cal U}\subseteq {\cal X}$ satisfying (i) and (ii) in Theorem \ref{mainthm}, i.e. there exists
$\pi_{\cal U}: \ {\cal U}'\rightarrow {\cal U}$ proper and birational, such that
\begin{equation}\label{eq50}
\mathrm{Reg}{\cal U}'={\cal U}' \ \mathrm{and} \ \pi_{\cal U}^{-1}(\mathrm{Reg}{\cal U})\simeq \mathrm{Reg}{\cal U}.
\end{equation}
We assume furthermore that a finite affine covering
${\cal U}=U_1 \cup \cdots \cup U_n$  is given such that
\begin{equation}\label{eq501}
\pi^{-1}_{\cal U}(U_i)\rightarrow U_i \ \mathrm{is} \ \mathrm{projective}.
\end{equation}

\smallskip

\noindent {\it Claim:} if two open sets ${\cal U}_1$ and ${\cal U}_2$ satisfy (\ref{eq50}) and (\ref{eq501}), so does
${\cal U}_1\cup {\cal U}_2$ w.r.t. {\it the union} of their respective coverings. Since ${\cal X}$ is
Noetherian, this claim completes reduction step 2.

\smallskip

We now prove the claim. Let ${\cal V}:= {\cal U}_1\cap{\cal U}_2$.
Denote by $\pi_i: \ {\cal U}'_i \longrightarrow {\cal U}_i$
the given resolutions of singularities satisfying (\ref{eq50}) and (\ref{eq501}). Let
$$
{\cal F}_1\subseteq {\cal U}'_1 \cap \pi_1^{-1}({\cal V})
$$
be the fundamental locus  of the birational map
$$
\rho : \ {\cal U}'_1 \cap \pi_1^{-1}({\cal V}) \cdots \longrightarrow {\cal U}'_2 \cap \pi_2^{-1}({\cal V}),
$$
and $\overline{{\cal F}}_1 \subseteq {\cal U}'_1$ be its Zariski closure in ${\cal U}'_1$. By (\ref{eq50}), we have:
$$
\pi_1 (\overline{{\cal F}}_1)\subseteq \mathrm{Sing}{\cal U}_1.
$$
In particular, we may replace ${\cal U}'_1$ by any blow up along a regular center contained in $\overline{{\cal F}}_1$.
We apply Lemma \ref{elimindet} to $\pi^{-1}_i(U_{j_1j_2})\rightarrow U_{j_1j_2}$, $i=1,2$ for each
$U_{j_1j_2}:=U_{j_1}\cap U_{j_2}$ with obvious notations.

When some ${\cal Z}_i$ in Lemma \ref{elimindet} is a curve, it can be assumed that ${\cal Z}_i$ is
regular away from (the inverse image of) ${\cal V}$ by blowing up closed points beforehand.
Furthermore the sequences (\ref{eq509}) for distinct  $U_{j_1j_2}$'s glue together,
which follows from the Definitions (\ref{eq5091})-(\ref{eq5092}). We may thus assume that
\begin{equation}\label{eq502}
\rho \ \mathrm{is} \  \mathrm{a} \ \mathrm{morphism}.
\end{equation}

Let ${\cal F}_2 \subseteq {\cal U}'_2 \cap \pi_2^{-1}({\cal V})$ be the fundamental locus of $\rho^{-1}$ and
consider the associated sequence (\ref{eq509}). We will only perform step 1 in the proof of
Lemma \ref{elimindet}.

When ${\cal Z}_i$ is a closed point mapping to ${\cal V}$,
we apply Proposition \ref{principsurf} beforehand to ${\cal I}({\cal Z}_i){\cal O}_{{\cal U}'_1}$
in order to preserve (\ref{eq502}).

When ${\cal Z}_i$ is an irreducible curve with generic point $\xi_i$, whose image in ${\cal V}$
has dimension one, the ideal ${\cal I}({\cal Z}_i){\cal O}_{{\cal U}'_1}$ is invertible above $\xi_i$ by
Proposition \ref{factbir}. Applying Proposition \ref{principsurf} beforehand to
${\cal I}({\cal Z}_i){\cal O}_{{\cal U}'_1}$, we also preserve (\ref{eq502}) while ${\cal U}'_1$
is unchanged away from the inverse image of finitely many closed points of ${\cal V}$. It can be assumed that
${\cal Z}_i$ is regular away from the inverse image of ${\cal V}$ by blowing up closed points beforehand
as above.

\smallskip

Summing up, it can be assumed that (\ref{eq502}) holds and that $\rho^{-1}$ is a morphism (hence an
isomorphism by (\ref{eq502})) away from
\begin{equation}\label{eq5021}
\pi_2^{-1}(x_1), \ldots ,\pi_2^{-1}(x_k), \ x_1, \ldots ,x_k \in {\cal V}
\ \mathrm{finitely} \  \mathrm{many} \ \mathrm{closed} \  \mathrm{points}.
\end{equation}
We may then glue ${\cal U}'_1$ and ${\cal U}'_2 \backslash \{\pi_2^{-1}(x_1), \ldots ,\pi_2^{-1}(x_k)\}$
along
$$
\pi_1^{-1}({\cal V} \backslash \{x_1, \ldots ,x_k\})=\pi_2^{-1}({\cal V} \backslash \{x_1, \ldots ,x_k\})
$$
to some proper morphism $\pi_{{\cal W}} : \ {\cal W}'\rightarrow {\cal W}:={\cal U}_1 \cup {\cal U}_2$.
By construction, $\pi_{{\cal W}}$ satisfies (\ref{eq50}) and (\ref{eq501}) for each $U_{j_1}\subseteq {\cal U}_1$.
Let $U_{j_2}\subseteq {\cal U}_2$ be fixed, so  $\pi_{2}^{-1}(U_{j_2})\rightarrow U_{j_2}$ is projective. Now
$\pi_{1}^{-1}(U_{j_1j_2})\rightarrow U_{j_1j_2}$ is projective for each $U_{j_1}\subseteq {\cal U}_1$, so
$\pi_{{\cal W}}(U_{j_2})\rightarrow U_{j_2}$ projective follows from  (\ref{eq5021}). This concludes
the proof of the claim, hence of step 2.

\smallskip

\noindent {\it Step 3:} achieving (i) in Theorem \ref{mainthm} with $\pi$ projective
for ${\cal X}=\mathrm{Spec}A$ affine.

\smallskip

The Riemann-Zariski space of valuations \index{Riemann-Zariski space of valuations}
$$
\mathrm{Zar}({\cal X}):=\{ v \ \mathrm{valuation} \ \mathrm{of} \ K : A\subseteq {\cal O}_v\}
$$
is quasi-compact by \cite{ZS2} Theorem 40 on p.113 and Noetherianity of $A$. The assumption on $v$
in (LU) means that  $v$ is a closed point of $\mathrm{Zar}({\cal X})$. Regularity is a nonempty open property
for any reduced ${\cal Y}$ which is of finite type over ${\cal X}$ because $A$ is excellent. This applies in particular
to any projective closure of $\mathrm{Spec}T$, $T$ as in (LU). Hence Theorem \ref{mainthm}(i)
is reduced to the following patching problem: let
$$
{\cal X}_1 \longrightarrow \mathrm{Spec}A, \ {\cal X}_2 \longrightarrow \mathrm{Spec}A
$$
be projective birational morphisms. There exists ${\cal Y} \longrightarrow \mathrm{Spec}A$
projective birational and morphisms $\pi_i: \ {\cal Y} \longrightarrow {\cal X}_i$, $i=1,2$, such that
$$
\pi_1^{-1}(\mathrm{Reg}{\cal X}_1)\cup \pi_2^{-1}(\mathrm{Reg}{\cal X}_2)\subseteq \mathrm{Reg}{\cal Y}.
$$
As indicated in \cite{Z5} on p.539, Zariski's Patching Theorem only requires
Proposition \ref{factbir} and Lemma \ref{elimindet} (here in our characteristic free context)
in order to deduce step 3 from (LU).

\smallskip

\noindent {\it Step 4:} achieving (ii). Let $\pi : \ {\cal X}'\rightarrow{\cal X}$ be as achieved in step 3,
i.e. projective birational with $\mathrm{Reg}{\cal X}'={\cal X}'$. Let ${\cal F} \subseteq {\cal X}$
be the fundamental locus of $\pi^{-1}$. We define
$$
{\cal F}_1:= \mathrm{Zariski} \ \mathrm{closure} \ \mathrm{in} \ {\cal X}
\ \mathrm{of} \ {\cal F} \cap  \mathrm{Reg}{\cal X}.
$$
Note that ${\cal F}_1$ has dimension at most one. We only sketch the argument and refer to \cite{Co8}
(see also \cite{Pi} section 6) for the details. There exists a commutative diagram
\begin{equation}\label{eq503}
\begin{array}{ccc}
  {\cal X}' & {\buildrel e'   \over \longleftarrow} & {\cal Y}' \\
  \downarrow &   & \downarrow \\
  {\cal X} &  {\buildrel e \over \longleftarrow} &  {\cal Y}\\
\end{array}
\end{equation}
such that $e$ (resp. $e'$) is a composition of blowing ups with regular centers
mapping to $\mathrm{Sing}{\cal X}$ (resp. to $\pi^{-1}(\mathrm{Sing}{\cal X})$). Let
$\pi ': \ {\cal Y} '\rightarrow {\cal Y}$ be the resulting morphism. This
diagram has the following property: let ${\cal G} \subset {\cal Y}$ be the fundamental locus
of ${\pi '}^{-1}$,  and ${\cal F}'_1\subseteq {\cal G}$ be the
strict transform of ${\cal F}_1$. Then any connected component of ${\cal G}$ containing points of
$\mathrm{Sing}{\cal Y}$ is disjoint from ${\cal F}'_1$ (in particular ${\cal F}'_1\subset \mathrm{Reg}{\cal Y}$).
This is achieved as follows:

\smallskip

\noindent (a) by iterating finitely many blowing ups of ${\cal X}$ at intersection points of ${\cal F}_1$
and $\mathrm{Sing}{\cal X}$, then applying Proposition \ref{principsurf}, we first obtain $e,e'$
such that ${\cal F}'_1\subset \mathrm{Reg}{\cal Y}$.

\smallskip

\noindent (b) by applying the techniques of step 2 above those irreducible curves
$C\subseteq {\cal G}$ only such that
$$
C\nsubseteq {\cal F}'_1, \ C\cap {\cal F}'_1 \neq \emptyset ,
$$
then applying Proposition \ref{principsurf} to get $e'$, we disconnect ${\cal F}'_1$ from
components of ${\cal G}$ containing points of $\mathrm{Sing}{\cal Y}$.

\smallskip

By (\ref{eq503}), there exists ${\cal U}\subseteq \mathrm{Reg}{\cal Y}$ such that the fundamental locus of
${\pi '}^{-1}({\cal U})\rightarrow {\cal U}$ is a {\it projective} subscheme (of dimension at most one)
containing ${\cal F}'_1$. We define ${\cal Z}\subset {\cal Y}'\times_{{\cal X}} {\cal Y}$ by composing
the diagonal embedding
$$
\Delta_{{\cal Y}'} : \ {\cal Y}'\rightarrow {\cal Y}'\times_{{\cal X}} {\cal Y}'
$$
with the  second projection $1\times \pi '$ above ${\cal Y}'\times_{{\cal X}} {\cal U}$.
Then ${\cal Z}\rightarrow {\cal X}$  has the required properties.
\end{proof}

\noindent {\it Proof of Corollary \ref{completeresolution}:} $A$ is excellent by \cite{EGA2}(7.8.3)(iii).

\smallskip

\noindent {\it Proof of Corollary \ref{integralmodel}:} let ${\cal Y}$ be any projective ${\cal O}$-scheme
with generic fiber ${\cal Y}_F=\Sigma$, e.g. clearing denominators in $\Sigma$.
By generic flatness \cite{EGA2}(6.9.1), there exists ${\cal U}\subseteq \mathrm{Spec}{\cal O}$ such that
$s^{-1}({\cal U})$ is flat over ${\cal U}$. Apply Theorem \ref{mainthm}
to the Zariski closure of $s^{-1}({\cal U})$ in ${\cal Y}$, where
$$
s : {\cal Y} \longrightarrow \mathrm{Spec}{\cal O}
$$
is the structure morphism.

\smallskip

\begin{rem}
Corollary \ref{integralmodel} can be strengthened in the obvious way:
given any proper and flat  ${\cal O}$-scheme ${\cal Y}$ with generic fiber ${\cal Y}_F=\Sigma$
and an open set ${\cal U}\subseteq \mathrm{Spec}{\cal O}$,
there exists a proper and flat ${\cal O}$-scheme ${\cal X}$ isomorphic to ${\cal Y}$ above ${\cal U}$
and regular away from ${\cal U}$.
\end{rem}

\subsection{Reduction to cyclic coverings.}

In this section, we reduce the local uniformization form (LU) of the previous section to Theorem \ref{luthm}.
This reduction is performed in two steps: first to complete local domains, then to cyclic coverings of degree $p$
in residue characteristic $p>0$. The first step is adapted from the descent methods of \cite{CoP1} Proposition 9.1
for (LU) inside the Henselization of finitely generated algebras of dimension three. Descent from complete local
rings to Henselian local rings, i.e. algebraization  of (LU), is proved in any dimension in \cite{ILO} Proposition 6.2,
but this does not imply Proposition \ref{redtoLUcomplete} below.

\begin{prop}\label{redtoLUcomplete}
Assume that  (LU) holds for every complete local domain of dimension three. Then
Theorem \ref{mainthm} holds.
\end{prop}

\begin{proof}

By Proposition \ref{redtoLU}, it is sufficient to prove that (LU) holds for
every quasi-excellent local domain $(A,m_A,k)$ of dimension three. As an indication,
the general strategy of the proof is deducing (LU) for $A$ from Theorem \ref{mainthm} for
$\mathrm{Spec}\hat{A}$. We will choose an extension $\hat{v}$ of $v$ to
$\mathrm{Spec}\hat{A}$ and a suitable resolution of singulaties
$\hat{{\cal Y}}\rightarrow \mathrm{Spec}\hat{A}$ (Lemma \ref{Ychapeau}) such that
$\hat{{\cal Y}}$ algebraizes at the center $\hat{y}$ of $\hat{v}$.
By general facts about excellent rings and Zariski's Main Theorem,
$\hat{{\cal Y}}$ algebraizes to a regular local ring $T_P$ at $\hat{y}$ (see (\ref{eq51021}) below).

\medskip

Let $v$ be a valuation of $K$ as in (LU). Denote by
$$\
\Gamma_v:=K^\times /{\cal O}_v^\times, \ r:=\mathrm{dim}_{\Q}(\Gamma_v\otimes_{\Z}\Q)
$$
the value group and rational rank of $v$. To begin with, we may assume that $\mathrm{dim}{\cal O}_v=1$,
i.e. $\Gamma_v \subset (\R, \geq)$, applying \cite{NSp} Theorem 1.1 (valid in all dimensions)
or using the dimension three techniques in \cite{CoP1} Proposition 5.1. We may also assume that
the residue extension $k_v | k $ is algebraic: if $x \in {\cal O}_v$ has transcendental residue,
replacing $A$ by $B:=A[x]_{m_v\cap A[x]}$ gives a reduction on dimension, since $\dim B<\dim A$
by the dimension formula.

\smallskip

Since $A$ is local quasi-excellent, its formal completion $\hat{A}$
w.r.t. $m_A$ is reduced \cite{EGA2}(7.8.3)(vii) and Remark (7.8.4)(i), so
$$
\hat{K}:=\mathrm{Tot}(\hat{A})=\prod_{i=1}^c \hat{K}_i, \ \hat{K}_i=QF(\hat{A}/\hat{P}_i)
$$
and the $\hat{P}_i$'s are minimal primes. Let $\hat{v}$ be an extension of $v$ to, say $\hat{K}_1$,
after possibly renumbering. Note that $\mathrm{dim}{\cal O}_{\hat{v}}\geq 1$
and that inequality is strict in general. We may also choose $\hat{v}$ with  $k_{\hat{v}} | k_v $ algebraic
(hence $k_{\hat{v}} | k $ algebraic) by composing again if necessary (this means
that $\hat{v}$ is a closed point in the Riemann-Zariski space of valuations
$$
\mathrm{Zar}(\mathrm{Spec}{\hat{A}\over \hat{P}_1}):=\{ w \ \mathrm{valuation} \ \mathrm{of} \
\hat{K}_1 : {\hat{A}\over \hat{P}_1}\subseteq {\cal O}_w\}
$$
see \cite{ZS2} Theorem 38 p. 111). We have
$$
r\leq d:=\mathrm{dim}(\hat{A}/\hat{P}_1).
$$


Let ${\cal X}:=\mathrm{Spec}A$, $\hat{{\cal X}}:=\mathrm{Spec}\hat{A}$ and $f: \ \hat{{\cal X}}\rightarrow {\cal X}$
be the completion morphism. By assumption in this proposition and Proposition \ref{redtoLU},
Theorem \ref{mainthm} holds for $\hat{{\cal X}}$.  Let
$$
\hat{\pi}: \ \hat{{\cal Y}}\rightarrow  \hat{{\cal X}}
$$
be the corresponding resolution of singularities.  Let $\hat{y} \in \hat{{\cal Y}}$ be the center of $\hat{v}$.
Since $k_{\hat{v}} | k $ is algebraic and $\hat{A}/\hat{P}_1$ is universally catenary,
we have
$$
d=\mathrm{dim}{\cal O}_{\hat{{\cal Y}},\hat{y}}.
$$
By \cite{EGA2}(7.8.3)(v), we have
$\mathrm{Sing}\hat{{\cal X}}=f^{-1}(\mathrm{Sing}{\cal X})$. Therefore there exists $g\in A$, $g\neq 0$
such that $\hat{\pi}$ is an isomorphism above $\hat{{\cal X}}_g=\mathrm{Spec}\hat{A}_g$ by
Theorem \ref{mainthm}(ii). Let also $f_1, \ldots ,f_r \in A$ such that
$v(f_1), \ldots ,v(f_r)$ are $\Q$-linearly independent in $\Gamma_v$ and set $h:=gf_1 \cdots f_r \in A$.
We have:

\begin{lem}\label{Ychapeau}
With notations as above, it can be assumed that
\begin{equation}\label{eq510}
    \sqrt{h{\cal O}_{\hat{{\cal Y}},\hat{y}}}=\sqrt{m_{\hat{A}}{\cal O}_{\hat{{\cal Y}},\hat{y}}}=(\hat{u}_1 \cdots \hat{u}_r),
\end{equation}
where $(\hat{u}_1,  \ldots , \hat{u}_d)$ is a r.s.p. of ${\cal O}_{\hat{{\cal Y}},\hat{y}}$.
In particular
$$
\hat{v}(\hat{u}_1),\ldots ,\hat{v}(\hat{u}_r)\in \Gamma_v\otimes_{\Z}\Q
$$
and these values are $\Q$-linearly independent.
\end{lem}

\begin{proof}
This is \cite{CoP1} Proposition 6.2, taking into account Proposition \ref{embedsurf}. Note that
it is not necessary to assume here that $\mathrm{dim}{\cal O}_{\hat{v}}= 1$ because $h\in A$.
\end{proof}

We now conclude the proof which is easily adapted from \cite{CoP1} Proposition 9.1.
By elementary linear algebra, there exists an $r\times r$ matrix $M \in {\cal M}(r, \Z)$, $a=\mathrm{det}M>0$
such that
\begin{equation}\label{eq5101}
g_j:=\prod_{i=1}^rf_i^{m_{ij}}=\hat{\delta}_j\hat{u}_j^a \in {\cal O}_{\hat{{\cal Y}},\hat{y}}\cap K,
\end{equation}
where $\hat{\delta}_j\in {\cal O}_{\hat{{\cal Y}},\hat{y}}$ is a unit, $1 \leq j \leq r$. Let
$$
\hat{Q}_j:=(\hat{u}_j)\cap \hat{A}, \ r+1 \leq j \leq d.
$$
By construction (\ref{eq510}), we have ${\cal O}_{\hat{{\cal Y}},\hat{u}_j}=\hat{A}_{\hat{Q}_j}$,
so $(\hat{u}_j)$ is the strict transform of $\hat{Q}_j$ at $\hat{y}$. Since $A$ is dense in $\hat{A}$
for the $m_A$-adic topology, the right-hand side equality in (\ref{eq510}) implies: there exists
$g'_{r+1}, \ldots ,g'_d \in A$ and positive integers $m_{ij}$, $1\leq i \leq r$, $r+1\leq j \leq d$, such that:
$$
u'_j:=g'_j \prod_{i=1}^r\hat{u}_i^{-m_{ij}}\in {\cal O}_{\hat{{\cal Y}},\hat{y}}
$$
and $(\hat{u}_1,  \ldots , \hat{u}_r, u'_{r+1},  \ldots , u'_d)$ is a r.s.p. of ${\cal O}_{\hat{{\cal Y}},\hat{y}}$.
Let now
\begin{equation}\label{eq5102}
g_j:={g'_j}^a\prod_{i=1}^rg_i^{-m_{ij}}={u'_j}^a\prod_{i=1}^r\hat{\delta}_j^{-m_{ij}}
\in {\cal O}_{\hat{{\cal Y}},\hat{y}}\cap K
\end{equation}
and $T$ be the integral closure of $A[g_1, \ldots ,g_d]$ in $K$. By \cite{EGA2} Corollary 7.7.3,
$T$ is a finitely generated $A$-algebra. Furthermore, we have
\begin{equation}\label{eq51021}
A \subseteq T \subseteq {\cal O}_{\hat{{\cal Y}},\hat{y}}\cap K \subset{\cal O}_{\hat{v}}\cap K={\cal O}_v
\end{equation}
by (\ref{eq5101})-(\ref{eq5102}). To complete the proof, it must be proved that $T_P$ is regular,
where $P:=m_v \cap T$. By \cite{EGA2} Lemma 7.9.3.1, it is sufficient to prove that
$T':=T\otimes_A \hat{A}$ is regular at the center $P':=m_{\hat{y}}\cap T'$ of $\hat{v}$.
Since $T_P$ is normal, $T'_{P'}$ is also normal {\it ibid.} and  (7.8.3)(v).
There are inclusions
$$
\hat{A} \subset T'_{P'} \subseteq {\cal O}_{\hat{{\cal Y}},\hat{y}}.
$$
By (\ref{eq5101})-(\ref{eq5102}), the right-hand side inclusion satisfies
$$
\sqrt{P' {\cal O}_{\hat{{\cal Y}},\hat{y}}}=m_{\hat{y}},
$$
so ${\cal O}_{\hat{{\cal Y}},\hat{y}}=T'_{P'}$ by Zariski's Main Theorem \cite{Ray}~Theorem~1 p.~41 and the proof is complete.
\end{proof}

\begin{prop}\label{redtoLUcyclic}
Theorem \ref{luthm} implies Theorem \ref{mainthm}.
\end{prop}

\begin{proof}
By Proposition \ref{redtoLUcomplete}, it is sufficient to prove that (LU) holds
for every complete local domain $(A,m,k)$ of dimension three. Let $({\cal O}_v,m_v,k_v)$
be the given valuation ring as in (LU). We may assume here that $\mathrm{char}k_v=p>0$,
the equicharacteristic zero version of Theorem \ref{mainthm} being known. As in
Proposition \ref{redtoLUcomplete}, it is sufficient to deal with the case $\mathrm{dim}{\cal O}_v=1$.

By Noether normalization \cite{Ma} Theorem 29.4(iii), there exists a complete regular
local domain $S \subseteq A$ such that $A$ is a finite $S$-module, $\mathrm{dim}S=3$.
We will prove that the equal characteristic techniques of \cite{CoP1} extend to our situation.
Let $F$ be the quotient field of $S$, so the field extension $K | F$ is finite algebraic. By
\cite{EGA2} Corollary 7.7.3, the integral closure of $A$ in any finite extension of $F$ is
a finite $A$-module.

\smallskip

 {Let $K^{\mathrm{sep}}\subseteq K$ be the separable closure of $F$. We first reduce to the case $K^{\mathrm{sep}}=K$. If $\mathrm{char}K=0$, we already  have $K^{\mathrm{sep}}=K$ and there is nothing to prove. Assume $p:=\mathrm{char}K>0$. The extension $K| K^{\mathrm{sep}}$
 is a tower of purely inseparable extensions
of degree $p=\mathrm{char}K$:
$$ K^{\mathrm{sep}}=:K_0 \subset K_1 \subset \cdots \subset K_n:=K, \ n\geq 0.$$
Let $i\geq 1$ and assume that (LU) holds for the integral closure $\overline{S}_{i-1}$ of $S$ in $K_{i-1}$. We  have:
$$
  K_i=K_{i-1}(x_i^{1/p}), \ x_i \not \in K_{i-1}^p.
$$
By Proposition \ref{principsurf} (applied to the ideal $(f,g)$ where $x_i={f\over g}$), we may take $x_i \in S_{i-1}$, where $S_{i-1}$ is  given by (LU)
for $\overline{S}_{i-1}$. So
$$
h:=X^p -x\in S_{i-1}[X]
$$
satisfies the assumption of Theorem \ref{luthm}(i). We conclude that (LU) holds for $\overline{S}_i$ which completes the induction step.
From now on, we assume that $K | F$ is separable.}

\smallskip

Let $\overline{K}| K$ be a Galois closure and $\overline{v}$ be an extension of $v$ to $\overline{K}$.
Ramification theory of valuations \cite{ZS2} section 12 provides a diagram of fields
\begin{equation}\label{eq51}
    \begin{array}{ccccccc}
 QF(A)= K         & \subseteq    & K^i       & \subseteq    & K^r       & \subseteq   & \overline{K}\\
  \uparrow  &                   & \uparrow  &                   & \uparrow  &    \\
  QF(S)= F       &  \subseteq   & F^i     & \subseteq    & F^r     &  &\\
\end{array}
\end{equation}
as in the proof of \cite{CoP1} Theorem 8.1.
 {More precisely, $F^i $ (resp. $K^i$) is the inertia field with respect to $\overline{v}$ of the field extension $ \overline{K}\vert F$ (resp.  $ \overline{K}\vert K$);
$F^r $ (resp. $K^r$) is the ramification field with respect to $\overline{v}$ of the field extension $ \overline{K}\vert F$ (resp.  $ \overline{K}\vert K$).}

The left-hand side (resp. middle) inclusions in this diagram
are unramified (resp. totally ramified Abelian of order prime to $p$).
The extension $K^r| F^r$ is a tower of totally ramified Galois extensions of degree $p$.

\begin{rem}
Theorem \ref{luthm} is actually required only to deal with those ramified extensions of degree
$p$ which are immediate (same value group and same residue field) w.r.t. the corresponding restrictions of $\overline{v}$.
For extensions of degree $p$ which are not immediate, a much simpler proof is available,
{\it vid.} \cite{CoP1} Proposition 6.3 in the equicharacteristic case.
\end{rem}

In order to connect ramification theory of valuations and ramification theory of
$S$-algebras essentially of finite type, we restate \cite{CoP1} Theorem 7.2 in our context
as Proposition \ref{Galoisapprox} below.

\begin{defn}\label{normalmodel}
 {Let $(R,m,k)$ be an excellent regular local ring,  $L\vert QF(R)$ be a finite field extension and $w$ be a valuation of $L$, with valuation ring $({\cal O}_w,m_w,k_w)$ such that
$$
R \subset {\cal O}_w \subset L, \ m_w\cap R=m, \ k_w | k \ \mathrm{algebraic}.
$$
A normal local model of ${\cal O}_w|R$ is  the localization $B_P$ of a finitely generated $R$-algebra $B$,
$R \subseteq B \subseteq {\cal O}_w, \ QF(B)=L$ such that $B$ is normal, where $P:=m_w \cap B$.}

\smallskip
 {Let $L'| L$ be a finite field extension and $w'$ be an extension of $w$ to $L'$.
Given a normal local model $B_P$ of ${\cal O}_w|R$, we define a normal local model $B'$ of
${\cal O}_{w'}|R$ by localizing the integral closure $\overline{B}$ of $B$ in $L'$ at $P':=m_{w'} \cap \overline{B}$.}
\end{defn}

\smallskip
Note that $B'$ is actually a normal local model because $R$, hence $B$, is excellent. Assume that  $L'| L$ is Galois. Note that
if $B'$ is a normal local model of ${\cal O}_{w'}|R$ then
$B' \cap L={B'}^{\mathrm{Gal}(L'| L)}$ is a normal local model of ${\cal O}_w|R$.  {By $G^s(B'| B)$, (resp. $G^i(B'| B)$), we mean the splitting group (resp. inertia group) of $B'| B$.
For ramification theory of local rings, we refer to \cite{Ab4} (see also
\cite{CoP1} section 2 equations (2)(3) for a quick summary of the required notions and notations).}

\smallskip
 {Finally, we denote by $G^s (w'| w)$ (resp. $G^i(w'| w)$, $G^r(w'| w)$) the splitting group (resp. the inertia group, ramification group) of  ${\cal O}_{w'}|{\cal O}_w$ from classical valuation theory: we refer to \cite{ZS2} chapter~VI, section~12, see also \cite{CoP1} section 2 pp.~1056-7.}

\begin{prop}\label{Galoisapprox}\textbf{(Galois Approximation).}
 {Let $L'| L$ be a finite Galois extension and $w'$ be an extension of $w$ to $L'$.
There exists a normal local model $B_0$ of ${\cal O}_w|R$ such that for any normal local model
$B$ of ${\cal O}_w|R$ with $B_0 \subseteq B$, the following holds:
\begin{itemize}
  \item [(1)] $G^s (w'| w)=G^s(B'| B)$ and $G^i(w'| w)=G^i(B'| B)$;
  \item [(2)] the normal model $B^r:={B'}^{G^r(w'| w)}$ of ${\cal O}_{w^r}|R$ satisfies
  $$
  B^r/m_{B^r} =B^i/m_{B^i},
  $$
  where $B^i$ is the inertia ring of $B'$ over $B$, i.e. $B^i={B'}^{G^i(B'| B)}$, and $w^r$ is
  the restriction of $w'$ to $L^r:={L'}^{G^r(w'| w)}$. Moreover
  the representation
  $$
  \rho : \ G^i(w'| w) /G^r(w'| w)\rightarrow \mathrm{GL}(m_{B^r}/m_{B^r}^2),
  \ g \mapsto (\overline{x} \mapsto \overline{g.x})
  $$
  is faithful and diagonalizable.
\end{itemize}}
\end{prop}
\begin{proof}
 {Since $\mathcal{O}_w$ is the direct union of all its normal local models $B$,
its integral closure  $\overline{\mathcal{O}_w}$ in $L'$ is the direct union of
all corresponding integral closures $\overline{B}$ in $L'$. Since
the extensions of $\mathcal{O}_w$ to $L'$ are the localizations of
$\overline{\mathcal{O}_w}$ at its maximal ideals $m_1 , \ldots , m_s$, any
$B_0$ such that for $1 \leq i \leq s$, the $m_{i}\cap
\overline{B_0}$'s are pairwise distinct satisfies the statement
about splitting groups in (1) of the proposition.}\\

 {Let now $B$ be any normal local model of $\mathcal{O}_w/R$ such that $B_0 \subset
B$. There is an inclusion
$G^i(w'/w) \subseteq G^i(B'/B)$. Let $t_1, \ldots , t_{f} $
be elements of $\mathcal{O}_{w'}$ whose residues $\overline{t_1}, \ldots ,
\overline{t_{f}} $ generate
$k_{w'}$ as a $k_w$ vector space. Enlarging $B_0$, it can be assumed that
$k_w/(B_0/m_{B_0})$ is algebraic and that $t_1, \ldots , t_f
\in {B'_0}$. Then any $g \in G^i(B'/B)$ acts trivially on $k_{w'}$, hence $g \in G^i(w'/w)$.  This concludes the
proof of (1).}

\smallskip

 {We now turn to the proof of (2).
On the one hand, the residue extension ${B' \over m_{B'}}\vert {Bî \over m_{Bî}}$ is generated by purely inseparable  elements (Theorem 1.48 \cite{Ab4}). On the other hand the field extension $QF(B^r)\vert  QF(B^i)$ is Galois of degree prime to $p$ since
\begin{equation}\label{eq27}
H:=G^i(w'/w)/G^r(w'/w)\simeq
\text{Hom}(\Gamma_{w'}/\Gamma_w,k_{w'})^\times.
\end{equation}
by \cite{ZS2}~Theorems 24 and 25. Hence  ${B' \over m_{B'}}= {B^i\over m_{B^i}}$.}

 {From now on, for $x \in m_{B^r}$, we write $\overline{x}$ for its
initial form in $m_{B^r}/m_{B^r}^2$. Consider the following
representation of $H$
\begin{equation}
\label{eq63}
\rho : \ H \rightarrow \mathrm{ GL}\left (m_{B^r}/m_{B^r}^2 \right ),
\ h \mapsto \left (\overline{x} \mapsto \overline{h. x}\right ).
\end{equation}}

 {By \cite{ZS2} middle of page 78, $\text{Hom}(\Gamma_{w'}/\Gamma_w,k_{w'})^\times$ is the entire character group, so
$k_{w'}$ contains the group $\mathbf{
\mu}_\epsilon$ of $\epsilon^{th}$-roots of unity, where $\epsilon$
is the exponent of the Abelian group $H$, and $\epsilon$ is prime
to $p$.
Since $k_{w'}\vert k_{w^r}$ is purely inseparable
(Theorem 1.48 \cite{Ab4}), we also have $\mathbf{\mu}_\epsilon
\subseteq k_{w^r}$. Enlarging $B_0$, we may assume that  $\mathbf{\mu}_\epsilon \subseteq  {B^i\over m_{B^i}}$.}

 {Now, any irreducible representation of $H$ over ${B^i\over m_{B^i}}$ has
degree one, since $H$ is Abelian and $\mathbf{\mu}_\epsilon \subseteq
{B^i\over m_{B^i}}$. Therefore, $\rho$ is diagonal up to choosing a
basis $(\overline{x_1},\ldots ,\overline{x_n})$ of
$m_{B^r}/m_{B^r}^2$. We write $\rho (h).
\overline{x_j}=:\chi_j(h)\overline{x_j}$, for $1 \leq j \leq n$
and $h\in H$, where $\chi_j \in \text{Hom}(H,{B^i\over m_{B^i}})^\times$.}

 {Let $\hat{L^r}:=QF(\hat{B^r})$ and
$\hat{L^i}:=QF(\hat{B^i})$. Since $w'/w^i$ is totally ramified, we
also have $\mathrm{Gal}(\hat{L^r}/\hat{L^i})=H$ with the natural
extension of the $H$-action to formal completions. By Hensel's Lemma, the embedding $\mathbf{\mu}_\epsilon \subseteq
{B^i\over m_{B^i}}$ lifts to an embedding $\mathbf{\mu}_\epsilon \subseteq
\hat{B^i}$. Let
\begin{equation}\label{eq64}
y_j:= {1 \over \mid H \mid}\sum_{h \in
H}{\chi_j(h^{-1})(h.x_j)}\in \hat{B^r}.
\end{equation}
It is immediately checked that $\overline{y_j}=\overline{x_j}$ and
that $h.y_j=\chi_j(h)y_j$ for each $h \in H$. After replacing
$x_j$ with $y_j$, it can therefore be assumed that
\begin{equation}\label{eq65}
h.x_j=\chi_j(h)x_j
\end{equation}
for each $h \in H$ and $1\leq j \leq n$, i.e. the action is faithful and
diagonal  on $\hat{B^r}$.}\\

\end{proof}


We now complete the proof of Proposition~\ref{redtoLUcyclic}.  To emphasize the dependence
on $v$, we say that $(\mathrm{LU}v)$ holds if (LU) holds for a particular $v$. With notations as in
(\ref{eq51}), we denote by $v_0,v_0^i,v_0^r,v^i, v^r$ the respective restrictions of $\overline{v}$ to $F$,
$F^i$, $F^r$, $K^i$ and $K^r$. The strategy is to prove successively the implications
$$
(\mathrm{LU}v_0) \Longrightarrow (\mathrm{LU}v_0^i) \Longrightarrow (\mathrm{LU}v_0^r)
\Longrightarrow (\mathrm{LU}v^r) \Longrightarrow(\mathrm{LU}v^i)\Longrightarrow (\mathrm{LU}v).
$$

Note that $(\mathrm{LU}v_0)$ holds by construction since $S$ is regular.

\smallskip

  {Firstly, we apply Proposition~\ref{Galoisapprox} (1) with
 $$R=S, \ L=F,  \ L'=F^i, \ w=
 v_0^i.$$
  By Proposition~\ref{principsurf}, we  may assume that $B_0 \subset S$. By Proposition \ref{Galoisapprox} (1),  the corresponding ring $S'$ from Definition~\ref{normalmodel},  is local-\'etale over $S$, hence $S'$ is regular. So
 $(\mathrm{LU}v_0^i)$ holds  (this follows the argument in
\cite{CoP1}~Corollary 7.3). }

\smallskip

Then $(\mathrm{LU}v_0^r)$ holds because $F^r| F^i$ is a tower of ramified
Galois extensions of prime degrees $l\neq p$: the proof relies on the Perron algorithm
as in \cite{CoP1} Proposition 6.3 and this is characteristic free.

\smallskip

To prove that $(\mathrm{LU}v^r)$ holds, we may assume that $K^r| F^r$ is a single Galois extension of degree $p$.
Let $x \in {\cal O}_{v^r}$ be a primitive element with minimal polynomial
$$
h:=X^p +f_1X^{p-1} + \cdots +f_p \in {\cal O}_{v^r_0}[X].
$$
By Proposition \ref{principsurf}, we may take $f_1, \ldots ,f_p \in T^r$, where
$T^r$ is a local uniformization, since  ($\mathrm{LU}v_0^r$) holds:   {we have the assumptions of
Theorem~\ref{luthm}(ii) which states that $(\mathrm{LU}v^r)$ holds.}

\smallskip

 {To prove that $(\mathrm{LU}v^i)$ holds, we may assume that $K^r| K^i$ is a single Galois extension of prime degree $l\not=p$. By Proposition \ref{Galoisapprox}~(2), the representation $\rho$ is faithful and diagonal. Using elementary linear algebra \cite{CoP1} p.~1080, we may assume that $\rho$ has the form:
$$ (\overline{x_1},\overline{x_2},\overline{x_3}) \mapsto(\zeta \overline{x_1},\overline{x_2},\overline{x_3}),$$
where $(x_1,x_2,x_3)$ is a suitable r.s.p. of $B^r$ and $\zeta \in \mu_l$. In this situation, we have
$$B^r=B^i\oplus B^i x_1 \oplus \cdots \oplus B^i x_1^{l-1}.$$
Let $u_i:= h^l (x_i)$, $1\leq i \leq 3$ where $h$ is the generator of  $H$.
This means that
$$  {B^r \over (u_1,u_2,u_3)B^r} {\buildrel \sim \over \longrightarrow} {B^r \over m_{B^i}B^r} \cong  {{B^r \over m_{B^r}} [X_1]\over (X_1^l)}.$$
By flatness, $m_{B^i}= (u_1,u_2,u_3)$: $B^i$ is regular.
}
\smallskip

 {To prove that  $(\mathrm{LU}v)$ holds, let $K^s$ be the splitting field  with respect to  $\overline{v}$ of the field extension $ \overline{K}\vert K$. }
  {Firstly, we apply Proposition~\ref{Galoisapprox} (1) with
 $$R=A, \ L=K,  \ L'=\overline{K}, \ w=
 \overline{v}.$$
 Let $T^i$ be a given regular local model of $\mathcal{O}_{v^i} \vert A $.
  By Proposition~\ref{principsurf}, we  may assume that $B_0 \subset T^i$. Let $\overline{T}$ be the localization of the integral closure of $T^i$ in $\overline{K}$ at the center of $ \overline{v}$;
  let $T:= T^i \cap K^s= (T^i)^{\mathrm{Gal}(K^i \vert K^s)} $ be the fixed ring. }
   {By Proposition~\ref{Galoisapprox}~(1), we have
$$\mathrm{Gal}(\overline{K} \vert K^s)=G^s (\overline{v}| v^s)=G^s(\overline{T}| T)\  \mathrm{and}\ G^i(\overline{v}| v^s)=G^i(\overline{T}| T).$$
This shows that:
$$G^i({T}^i| T)={G^s(\overline{T}| T)\over G^i(\overline{T}| T)}=\mathrm{Gal}(K^i \vert K^s).$$
By \cite{Ray}~Theorem~2 p.~110, $T^i$ is local-\'etale over $T$. Since $T^i$ is regular, so is $T$. Therefore (LU$v^s$) holds. }

\smallskip
 {Let $B_0$ be as above, let $B_0^s$ be the localization of the integral closure of $B_0$ in ${K^s}$ at the center of $ {v}^s$.    By Proposition~\ref{principsurf}, we  may assume that $B_0^s \subset T$. Note that
\begin{equation}\label{eqfingalois}
\hat{B}_0 = \hat{B}_0^s.
\end{equation}
We claim that there exist  $g_1,g_2,g_3 \in T\cap K$ such that
$$
\sqrt{(g_1,g_2,g_3) T}=m_T.
$$
The construction is the same as in the proof of Proposition~\ref{redtoLUcomplete}, using \eqref{eqfingalois}, see Lemma~\ref{Ychapeau} up to the end of the proof of Proposition~\ref{redtoLUcomplete}. By Zariski's Main Theorem \cite{Ray}~Theorem~1 p.~41, $T$ is   the localization of the integral closure of $B_0[g_1,g_2,g_3]$ in ${K^s}$ at the center of $ {v}^s$. Let $T_0$ be the localization  of the normalization of  $B_0[g_1,g_2,g_3]$ at the center of $ {v}$, $\hat{T}=\hat{T}_0$: $T_0$ is regular and (LU$v$) is proved. }

\end{proof}


\subsection{Normal crossings divisors conditions.}\label{NCDconditions}

In this section, we consider a pair $(S,h)$ satisfying the assumptions of Theorem \ref{luthm},
i.e. such that {\bf (G)} holds. We construct a sequence $\pi : {\cal X}' \rightarrow {\cal X}$
of blowing ups along Hironaka-permissible centers in such a way that every $x' \in \pi^{-1}(x)$
has either $m(x')<p$, or ($m(x')=p$ and  $x'$ satisfies condition {\bf (E)}). This is proved in
Corollary \ref{EEfait} below. Assumption {\bf (G)} is not required here
and we prove a more general version for arbitrary multiplicity in Proposition \ref{normalcrossings}.

\begin{lem}\label{imagepoints}
Let $S$, $h\in S[X]$ (\ref{eq201}) and  $\eta: {\cal X} \rightarrow \mathrm{Spec}S$ be given. Assume that
$\mathrm{dim}S=3$ and that $h$ is reduced. There exists a composition of Hironaka-permissible
blowing ups (\ref{eq210}) w.r.t. $E=\emptyset$:
$$
\begin{array}{ccc}
  {\cal X} &  {\buildrel \pi \over \longleftarrow} & {\cal X}'  \\
  \downarrow &     & \downarrow \\
  \mathrm{Spec}S &  {\buildrel \sigma  \over \longleftarrow} & {\cal S}'\\
\end{array}
$$

such that $\pi (\mathrm{Sing}_m{\cal X}')\subseteq \eta^{-1}(m_S)$.
\end{lem}

\begin{proof}
This statement means that there exists a diagram
\begin{equation}\label{eq52}
\begin{array}{ccccccc}
  {\cal X}=:{\cal X}_0 &  {\buildrel \pi_0  \over \longleftarrow} & {\cal X}_1 &
  {\buildrel \pi_1   \over \longleftarrow} & \cdots & {\buildrel \pi_{n-1}  \over \longleftarrow} &  {\cal X}_n=:{\cal X}' \\
  \downarrow &    & \downarrow & &  &   & \downarrow \\
  \mathrm{Spec}S=:{\cal S}_0 &  {\buildrel \sigma_0  \over \longleftarrow} & {\cal S}_1
  & {\buildrel \sigma_1  \over \longleftarrow} & \cdots & {\buildrel \sigma_{n-1}  \over \longleftarrow}  &  {\cal S}_n=:{\cal S} '\\
\end{array}
\end{equation}
where  each morphism $\pi_i$, $0 \leq i \leq n-1$, is the blowing up along a Hironaka-permissible center
${\cal Y}_i \subset {\cal X}_i$ w.r.t. the reduced exceptional divisor $E_i$ of
$\pi^{(i)}: {\cal X}_i\rightarrow {\cal X}$. It can be assumed that $\dim (\mathrm{Sing}_m{\cal X})\geq 1$.

\smallskip

Let $y_i \in {\cal X}_i$ denote the generic point of such a Hironaka-permissible center
${\cal Y}_i \subset {\cal X}_i$ w.r.t. $E_i$. We define:
$$
\Delta_i:=\{y \in \mathrm{Sing}_m{\cal X}_i : \dim{\cal O}_{{\cal X}_i,y}=\dim{\cal O}_{{\cal X},\pi^{(i)}(y)}=1\},
$$
$$
\delta_i:=\max\{\delta (y), y\in \Delta_i\}, \ N_i:=\sharp\{y\in \Delta_i : \delta (y)=\delta_i\}.
$$

Let $i \geq 0$. We claim that
\begin{equation}\label{eq521}
\left \{
\begin{array}{ccccc}
 (\delta_{i+1}, N_{i+1}) & = & (\delta_i, N_i) & \mathrm{if} & \dim{\cal O}_{{\cal X},\pi^{(i)}(y_i)}\geq 2 \\
  & & & & \\
 (\delta_{i+1}, N_{i+1}) & < &  (\delta_i, N_i) & \mathrm{if} & \dim{\cal O}_{{\cal X},\pi^{(i)}(y_i)}=1
\end{array}
\right .
.
\end{equation}
Namely, this is an obvious consequence of the definition if $\dim{\cal O}_{{\cal X},\pi^{(i)}(y_i)}\geq 2$.
If $\dim{\cal O}_{{\cal X},\pi^{(i)}(y_i)}=1$, let $y \in {\cal X}_{i+1}$ with $\pi_i(y)=y_i$.
We have
$$
(m(y),\delta (y))\leq (m(y_i),\delta (y_i)-1)
$$
by Proposition \ref{originchart} applied for $n=1$ and the claim follows

\smallskip

Pick $y \in \Delta_i$ with $\delta (y)=\delta_i$ and denote ${\cal Y}:=\overline{\{y\}}\subset {\cal X}_i$.
By Proposition \ref{embedsurf}, there exists
a composition of blowing ups  ${\cal X}_{i'}\rightarrow {\cal X}_i$ with regular centers contained
in the successive strict transforms of ${\cal Y}$ such that
$\eta_{i'}({\cal Y}')$ has normal crossings with $E_{i'}$, where ${\cal Y}'$ denotes the strict transform
of ${\cal Y}$ in ${\cal X}_{i'}$. Then ${\cal Y}'$ itself and each blowing up center in
${\cal X}_{i'}\rightarrow {\cal X}_i$
are Hironaka-permissible w.r.t. $E_{i'}$ because $m(y)=m$.

We have $(\delta_{i'}, N_{i'}) = (\delta_i, N_i)$
by (\ref{eq521}). Taking as blowing up center ${\cal Y}_{i'}:={\cal Y}'$ also gives $(\delta_{i'+1}, N_{i'+1}) < (\delta_i, N_i)$
by (\ref{eq521}). Since $\Delta_i$ is a finite set and $\delta_i \in {1 \over m}\N$, there exists an index
$i_1 > i$ such that $\Delta_{i_1}=\emptyset$ and this is preserved by further Hironaka-permissible blowing ups
w.r.t. $E=\emptyset$.

\smallskip

Since $\Delta_{i_1}=\emptyset$, we are done unless $\pi^{(i_1)}(\mathrm{Sing}_m{\cal X}_{i_1})={\cal C}$, where
${\cal C}$ has pure dimension one. Let $C\subset \mathrm{Spec}S$ be an irreducible component of $\eta ({\cal C})$ and
$s$ be its generic point. Note that the stalk $({\cal X}_i)_s$ at $s$ of the $S$-scheme  ${\cal X}_i$
is embedded in the regular scheme of dimension three $\mathrm{Spec}S_s[X]$ for $i=0$ and in an iterated
blowing up along regular centers of the former for $i\geq 1$. By Proposition \ref{embedsurf}, there exists
a composition of Hironaka-permissible blowing ups ${\cal X}'_s \rightarrow ({\cal X}_{i_1})_s$
w.r.t. $(E_{i_1})_s$ such that $\mathrm{Sing}_m {\cal X}'_s =\emptyset$.

Let ${\cal Y}_s \subseteq ({\cal X}_{i_1})_s$ be a Hironaka-permissible center and
${\cal Y} \subseteq {\cal X}_{i_1}$ be its Zariski closure, so in particular we have
${\cal Y}\subseteq  \mathrm{Sing}_m{\cal X}_{i_1}$. Since $\Delta_{i_1}=\emptyset$,
${\cal Y}$ is either (1) a curve mapping onto $C$, or (2) a surface mapping to some irreducible
component of $E_{i_1}$.

\smallskip

In situation (1), there exists a composition of blowing ups along closed points
${\cal X}_{i'_1}\rightarrow {\cal X}_{i_1}$ such that $\eta_{i'_1}({\cal Y}')$ has normal
crossings with $E_{i'_1}$, where ${\cal Y}'$ denotes the strict transform of ${\cal Y}$ in
${\cal X}_{i'_1}$.

In situation (2), ${\cal Y}$ itself is Hironaka-permissible w.r.t. $E_{i_1}$ and we let
$i'_1:=i_1$.

In both situations, we may blow up ${\cal X}_{i'_1}$ along ${\cal Y}'$ and iterate: this produces
an index $i_2 \geq i_1$ and a  composition of Hironaka-permissible blowing ups
${\cal X}_{i_2}\rightarrow {\cal X}_{i_1}$ w.r.t. $E_{i_1}$ such that
$\eta^{-1}(s)\cap \pi^{(i_2)}(\mathrm{Sing}_m{\cal X}_{i_2})=\emptyset$. Applying this construction to the
finitely many irreducible components of $\eta ({\cal C})$ proves the lemma.
\end{proof}

\begin{prop}\label{normalcrossings}
Let ${\cal X}'$ satisfy the conclusion of Lemma \ref{imagepoints} and $E' \subset {\cal S} '$ be the
reduced exceptional divisor of $\sigma$. Let $D \subset {\cal S} '$ be a
reduced divisor.

\smallskip

There exists a composition of Hironaka-permissible blowing ups (\ref{eq210}) w.r.t. $E'$:
$$
\begin{array}{ccc}
  {\cal X}' &  {\buildrel \pi '\over \longleftarrow} & {\cal X}''  \\
  \downarrow &     & \downarrow \\
  {\cal S}' &  {\buildrel \sigma ' \over \longleftarrow} & {\cal S}''\\
\end{array}
$$
such that the strict transform $D ''$ of $D$
is disjoint from $\eta ''(\mathrm{Sing}_m{\cal X}'')$, where
$\eta '' : ({\cal X}'',x'') \rightarrow {\cal S} ''$ is the local projection
at $x''\in \mathrm{Sing}_m{\cal X}''$.
\end{prop}

\begin{proof}
We take ${\cal S}'=\mathrm{Spec}S$. The problem is to  find a sequence (\ref{eq52})
which monomializes $P:=\mathrm{I}(D)\subset S$, i.e. such that $P_n:=P{\cal O}_{{\cal S}_n}$
is a monomial with components at normal crossings with $E_n$.

\smallskip

Let us write $P_i:=H_iQ_i$ where $H_i$ is a monomial whose components are components of $E_i$.
At the beginning, $H=H_0=1$. The strategy is to get $P_n=H_n$, $Q_n=1$ at the end.

We consider the idealistic exponents  {(see \cite{H6}~p.~54)} $(h,m)$ and $(Q,b)$ living in Spec$S[Z]$,
where $b=$ord$_{m_S}(Q)$. We make a descending induction on $b$: 
the case $b=0$ means that we get the conclusion of \ref{normalcrossings}. 
Each  pair of blowing ups $\pi_i,\sigma_i$
is locally centered at some $Y_i$ and $\eta(Y_i)$ respectively, and is Hironaka-permissible for $h$
(resp. $Q_i$) w.r.t. $E_i$.

Let $P_{i+1}=:H_{i+1}Q_{i+1}$ where $Q_{i+1}$ is the strict transform of $Q_i$.
This means that $(Q_{i+1},b)$ is the transform of $(Q_i,b)$. When ord$_{x_{i+1}}(Q_{i+1})<b$,
we have strictly improved and we go on with the new idealistic exponent $(Q_{i+1},b')$,
with $b':=$ord$_{x_{i+1}}(Q_{i+1})$.
To define a sequence of $\sigma_i$ is a consequence of \cite{CoJS} {\bf Theorem 0.3}
(Canonical embedded resolution with boundary), the problem is the sequence of $\pi_i$,
i.e. to define the pair $(\sigma_i,\pi_i)$.

 {\begin{nota}\label{nota:Vdir}To avoid cumbersome notations, from now on, $x_i,S_i,{\cal X}_i,$etc.$_i$ are denoted by $x,S,{\cal X},$etc.
and $x_{i+1} ,S_{i+1}  ,  {\cal X}_{i+1} ,$etc.$_{i+1}$ by $x',S',{\cal X}',$etc.$'$.
Let us define $\mathrm{Vdir}(x,D)$ as $\mathrm{Vdir}(h)+\mathrm{Vdir}(Q)$.
This is a vector space of codimension $\tau(x,D)$ in the Zariski's tangent space Spec(gr$_{(m_S,Z)}(S[Z])$) of Spec$(S[Z])$ at $x$. Of course, $\tau(x,D)\geq 2$. We denote by $\mathrm{IDir}(x,D)\subset $gr$_{(m_S,Z)}(S[Z])=k(x)[Z,U_1,U_2,U_3]$ the ideal of $\mathrm{Vdir}(x,D)$.
\end{nota}}

\begin{lem}\label{directriceOK}
Let $\pi$ be the blowing up along $Y$ which is permissible for both $(h,m)$ and $(Q,b)$.
Let $x' \in \pi^{-1}(x)$ be such that $m(x')=m(x)=m$ and ord$_{x'}Q'=b$. Then $x'$ is on
 {$\mathbf{Proj}(k(x)[Z,U_1,U_2,U_3]/ \mathrm{IDir}(x,D))$}. In particular, $x'$ is on the strict transform of $\mathrm{div}(Z)$.
\end{lem}
\begin{proof}
By Proposition \ref{conedirectrix} and Remark \ref{ridgedimthree}, we have
$\mathrm{Dir}(F)=\mathrm{Max}(F)$ except if $p =2$ and
\begin{equation}\label{eqnormalcrossings}
F=\lambda (Z^2 + \lambda_2 U_1^2 +\lambda_1U_2^2+\lambda_1 \lambda_2U_3^2)^\alpha, \ [k^2(\lambda_1,\lambda_2):k^2]=4
\end{equation}
up to a linear change of variables, $\lambda \neq 0$, $\alpha \geq 1$.
 {Then $\pi$ is the blowing up centered at $x$.} Since $m(x')=m(x)$, we have
$$
x ':=V(U_1^2 +\lambda_1 U_3^2, U_2^2 +\lambda_2 U_3^2, Z^2 + \lambda_1\lambda_2U_3^2)
$$
on  {${\pi'}^{-1}(x)=\mathrm{Proj}(k[Z,U_1,U_2,U_3]/(F))$}.

Since ord$_{x'}Q'=b$, the initial of $Q$ cannot satisfy (\ref{eqnormalcrossings}) (only the last three variables occur).
Therefore
\begin{equation}
 \begin{array}{ccccccc}
x'\in &  {\mathbf{Proj}(k(x)[Z,U_1,U_2,U_3]/ \mathrm{IDir}(h)) \cap \mathbf{Proj}(k(x)[Z,U_1,U_2,U_3]/ \mathrm{IDir}(Q)) }\\
&  {=\mathbf{Proj}(k(x)[Z,U_1,U_2,U_3]/ \mathrm{IDir}(x,D)).}\\
\end{array}
\end{equation}
\end{proof}

Let us come back to the proof of Proposition  \ref{normalcrossings}. We discuss
according to the value of $\tau(x,D)$.

When $\tau(x,D)= 4$, the blowing-up centered at $x$ makes $b$ strictly drop.

When $\tau(x,D)= 2$ or $3$, then, if we blow up along $x$, then $\tau(x',D') \geq \tau(x,D)$. In case $\tau(x,D)= 3$,
we make only blowing ups at closed points. Either for some $n$, $(m(x_n),$ord$_{x_n}(Q_n))<_{\mathrm{lex}}(m,b)$, then we stop at this $n$;
or we have equality for $n\geq 0$. Then,  $\tau(x_n,D_n)= 3$, $n\geq 0$,
by an usual argument, the $x_n$ are all on the strict transform of a curve $\cal{C}_n$ which, for $n>>0$ is permissible for both $(h,m)$ and $(Q,b)$ and $\eta(\cal{C}_n)$ is transverse to $E_n$.
Then at step $n$ in (\ref{eq52}), we blow up along $\cal{C}_n$. By Lemma \ref{directriceOK},  $(m(x_{n+1}),$ord$_{x_{n+1}}(Q_{n+1}))<_{\mathrm{lex}}(m,b)$.

When  $\tau(x,D)= 2$, we can choose $Z,u_3$ such that
$$
\mathrm{Vdir}(Q)=<U_3>,\ \mathrm{Vdir}(h)\equiv <Z>\ \mod(U_3).
$$

\begin{rem}\label{normalcrossingsrem}
If  there is a component $Y$ of dimension $2$ in
$$
\mathrm{Sing}(h,m)\cap \mathrm{Sing}(Q,b),
$$
then we can choose the parameters so that $I(Y)=(Z,u_3)$. Then $Q\in(z,u_3)^b$, i.e. $Q=u_3^b$,
up to multiplication by an invertible. Then, if $Y$ has normal crossing with $E$, we blow up along $Y$:
$\pi$ is the blowing up along $Y$ and $\sigma$ is the identity. In fact in $S$,
we just add $\eta(Y)=\div(u_3)$ to $E$ and we get $b=0$.

We also note that $(h,m)\cap(Q,b)=(hQ,m+b)$. In other words, we have
$$
\mathrm{Sing}(h,m)\cap \mathrm{Sing}(Q,b)=\mathrm{Sing}(hQ,m+b)
$$
and permissible centers are the same for $(hQ,m+b)$ and for $(h,m)\cap(Q,b)$.

\end{rem}

Then we apply those techniques from \cite{CoJS} {\bf 10}, {\bf 11}, {\bf 12}. 
More precisely, if for some $n_0$ the number $b$ just strictly drops, we call ``old components''
the components of $E_{n_0}$ at $x_{n_0}$  which are components of $H$ and, for $n\geq n_0$,  at $x_{n}, n\geq n_0$ with  $b(x_n)=b(x_{n_0})$, the strict transforms of this old components. 
The first step is to reach the case where $x_n$ is not on the strict transform of this old components: the invariant is $(m,b,o(x))$ where $o(x)$ is the number of these old components.  In the language of idealistic exponents, we desingularize $(hQQ_O,mbo(x))$ where $Q_O$ is the equation of the reduced divisor whose components are the old ones.
Then we look at the directrix of $hQQ_O$. When its codimension denoted by $\tau(hQQ_O)$ is $3$ or $4$, we play the same game that above with  $\tau(x,D)= 3$ or $4$. We reach the case where $\tau(hQQ_O)=2$. This means that either $Q_O=1$ (no old component) or there is one old component which is tangent to $Q$.

\smallskip

Then we look at the characteristic polyhedron $\Delta( hQQ_0,z,u_3,u_1,u_2)$
as in  \cite{CoJS} {\bf Section 7}.

\smallskip

\noindent $\bullet$ Case $\Delta( hQQ_0,z,u_3,u_1,u_2)=\emptyset $.
This is equivalent to $hQQ_0 \in (z,u_3)^{mbo(x)}$, i.e. this is equivalent to dim$(\mathrm{Sing}(hQQ_O,mbo(x))=2$.
So $QQ_O=u_3^{mbo(x)}$, call $Y:=\V(z,u_3)$, in fact, at step $n_0$, as $b(x_0)=b(x)$,
$Q$ was a $b(x_0)$ power and, if at $x$ there is one old component, it is a factor of $Q$: this is impossible,
therefore $o(x)=0$.

So, at $x$, $E$ is a union of components which are exceptional divisors of the blowing ups $\sigma_n$, $n\geq n_0$.
By  \cite{CoJS}~{\bf Theorem 8.3}, they are transverse to $u_3$: $Y$ is permissible for $(hQQ_O,mbo(x))$ and
transverse to $E$. We apply the first statement of Remark \ref{normalcrossingsrem}.

\smallskip

\noindent $\bullet$ Case where dim$(\mathrm{Sing}(hQQ_O,mbo(x))\leq 1$. Then, we apply
\cite{CoJS} {\bf Theorem 5.28} which gives the result if char$k(x)\geq 3$. This hypothesis $p\not=2$
is used just to get $\mathrm{Dir}(F)=\mathrm{Max}(F)$ at each step, but we showed above in Lemma \ref{directriceOK},
that the only case where  $\mathrm{Dir}(F)\not=\mathrm{Max}(F)$ stops
after blowing up the closed point $x$.
\end{proof}

\begin{cor}\label{EEfait}
Assume that $\mathrm{char}S/m_S=p>0$ and $(S,h)$ satisfies condition {\bf (G)}. There exists
a composition of Hironaka-permissible blowing ups (\ref{eq210}) w.r.t. $E=\emptyset$:
$$
\begin{array}{ccc}
  {\cal X} &  {\buildrel \pi '' \over \longleftarrow} & {\cal X}''  \\
  \downarrow &     & \downarrow \\
  \mathrm{Spec}S &  {\buildrel \sigma '' \over \longleftarrow} & {\cal S}''\\
\end{array}
$$
such that $\eta ''(\mathrm{Sing}_p {\cal X}'')\subseteq {\sigma ''}^{-1}(m_S)$ and condition {\bf (E)} holds at every
$s' \in  \eta ''(\mathrm{Sing}_p {\cal X}'')$, where $\eta '': {\cal X}''\rightarrow {\cal S}''$ is the projection.
\end{cor}

\begin{proof}
This is a direct application of Lemma \ref{imagepoints} in the purely inseparable case
((iii) of condition {\bf (G)}). If $\eta $ is separable and $\mathrm{char}S=p$, we apply
Proposition \ref{normalcrossings} to the strict transform in ${\cal S}'$ of
$D:=\mathrm{div}(\mathrm{Disc}_X(h))$ and the conclusion follows.

\smallskip

Assume that $\mathrm{char}S=0$. Let $D'_1$ be the strict transform of $\mathrm{div}(p\mathrm{Disc}_X(h))$ in
${\cal S}'$ and $D'_2$ be the union of those components of $E'$ of characteristic zero.
We apply  Proposition \ref{normalcrossings} to $D:=D'_1\cup D'_2$. Let $E''$ be the exceptional divisor of $\sigma ''$
and $s' \in \eta ''(\mathrm{Sing}_p {\cal X}'')$. Since all blowing up centers of $\sigma '$ are Hironaka-permissible w.r.t.
$E'$, they map to $\eta(x)$ and are thus of characteristic $p=\mathrm{char}S/m_S$. We deduce from Proposition
\ref{normalcrossings} that any irreducible component of $E''$ passing through $s'$ has characteristic $p$
and that (ii) of Definition \ref{conditionE} holds.
\end{proof}

\section{Projection number $\kappa (x)\in \{1,2,3,4\}$,  Projection Theorem.}

Let $(S,h,E)$ satisfy assumptions {\bf (G)} and {\bf (E)}. In this section, we perform induction
on the dimension $\mathrm{dim}S[Z]=4$ of the ambient space of ${\cal X}$, {\it vid.} introduction.
This step is for now far out of reach in higher dimensions and little more than definitions could be stated.
We reduce Theorem \ref{luthm} to Theorem \ref{projthm} below (Corollary \ref{projthmcor})
which is proved in the next sections.

\subsection{Projection number $\kappa (x)$.}

For $y \in {\cal X}$, $s:=\eta (y)\in \mathrm{Spec}S$, the assignment $\kappa (y)\geq 2$ has  {so far} been used
to express $\kappa (y)\neq 1$; we now distinguish $\kappa (y)=2,3,4$ when ($\omega (y)>0$, $\kappa (y)\geq 2$).
This completes our definition of the complexity function (\ref{eq251}):
$$
\iota : {\cal X} \rightarrow \{1,\ldots ,p\}\times  \N \times  \{1,\ldots ,4\},
y \mapsto (m(y), \omega (y), \kappa (y)).
$$
The projection number $\kappa (y)$ expresses the transverseness of $\mathrm{Vdir}(y)$
w.r.t. $E_s$. We claim no further invariance property w.r.t. regular local base change
than that of Theorem \ref{omegageomreg} when $\kappa (y)\geq 2$.

\smallskip

Since our assumptions {\bf (G)} and {\bf (E)}
are stable when changing $(S,h,E)$ to $(S_s,h_s,E_s)$ (Notation \ref{notaprime}), we may assume
that $s=m_S$. The following definition is for codimension three, the remark afterwards for
codimension two. One has $\omega (y)=\epsilon (y)=0$ in codimension one. We denote
$E=\mathrm{div}(u_1 \cdots u_e)$ as before.

\begin{defn}\label{defkappa}\textbf{(Projection Number).} \index{$\kappa(x)= 2,3,4$, Definition~\ref{defkappa}}
Assume that $m(x)=p$, $\omega (x)>0$ and $\kappa (x)\geq 2$, where $\eta^{-1}(m_S)=\{x\}$. We let
\begin{equation}\label{eq401}
    \kappa (x):=4 \ \mathrm{if} \ \mathrm{Vdir}(x)\subseteq <U_1, \ldots ,U_e>.
\end{equation}

Assume now that $\kappa (x)\neq 4$. We let $\kappa (x):=3$ if ($\omega (x)=\epsilon (x)-1$
and one of the following conditions is satisfied):
\begin{itemize}
  \item [(1)] $E=\mathrm{div}(u_1)$ and there  {exist} well adapted coordinates $(u_1,u_2 ,u_3;Z)$ at $x$
  such that
  $$
  \mathrm{Vdir}(x)\subseteq <U_1,U_3> \ \mathrm{and} \ H^{-1}{\partial F_{p,Z} \over \partial U_2 }\subseteq <U_1^{\omega (x)}>;
  $$
  \item [(2)] $E=\mathrm{div}(u_1u_2)$.
\end{itemize}
Finally, we let $\kappa (x):=2$ if $\kappa (x)\neq 3,4$.
\end{defn}

\begin{rem}\label{defkapparem}
When $\mathrm{dim}{\cal O}_{{\cal X},y}=2$, $m(y)=p$, $\omega (y)>0$ and $\kappa (y)\geq 2$, we
define: if $E_s=\mathrm{div}(u_1u_2)$, let $\kappa (y):=4$; if $E_s=\mathrm{div}(u_1)$, let:
$$
    \kappa (y):=\left\{
\begin{array}{ccc}
  2 & \mathrm{if} &  \omega (y)=\epsilon (y) \ \mathrm{and} \ \mathrm{Vdir}(y)\nsubseteq <U_1> \\
  3 & \mathrm{if} &  \omega (y)=\epsilon (y)-1 \hfill{}\\
  4 & \mathrm{if} &  \omega (y)=\epsilon (y) \ \mathrm{and} \ \mathrm{Vdir}(y) = <U_1>
\end{array}
\right .
.
$$
\end{rem}

 {\begin{rem}\label{rem:defkappa=2*}
The emblematic cases of $\kappa(x)=2$ are:
\begin{equation}
\mathrm{in}_{m_S} h=
\left\{
  \begin{array}{ccc}
    Z^p + \lambda U_1^{pd_1}U_3^{\omega (x)}, \hfill{}& E=\mathrm{div}(u_1),&\omega (x)\equiv 0 \ \mathrm{mod}\ p  \\
     & & \\
      Z^p + \lambda U_1^{pd_1}U_2^{pd_2}U_3^{\omega (x)}, \hfill{}&  E = \mathrm{div}(u_1u_2),&\omega (x)\equiv 0 \ \mathrm{mod}\ p  \\
     & & \\
    Z^p + \lambda U_1^{pd_1}U_2U_3^{\omega (x)}, \hfill{}&  E=\mathrm{div}(u_1),& \omega (x)\equiv 0\  \mathrm{mod}\ p  \\
  \end{array}
\right.
\end{equation}
 which are three kinds of the  special case $\kappa(x)=2$(*) (Definition~\ref{*kappadeux}).
\end{rem}}

\subsection{Projection Theorem.}

We now turn to the statement of the Projection Theorem. We assume that $\omega (x)>0$, so
$({\cal X},x)$ is (analytically)  irreducible by Theorem \ref{initform}. Let $\mu$ be  {a} valuation of
$L=k({\cal X})$ centered at $x$. We will consider finite sequences of
local blowing ups along $\mu$:
\begin{equation}\label{eq402}
    ({\cal X},x)=:({\cal X}_0,x_0) \leftarrow ({\cal X}_1,x_1)\leftarrow \cdots \leftarrow ({\cal X}_r,x_r)
\end{equation}
with Hironaka-permissible centers ${\cal Y}_i \subset ({\cal X}_i,x_i)$, where $x_i$,
$0 \leq i \leq r$, denotes the center of $\mu$. We require that our
assumptions {\bf (G)} and {\bf (E)} be preserved by such blowing ups and that
$$
(m(x_i),\omega (x_i))\leq (m(x_{i-1}),\omega (x_{i-1})), \ 1 \leq i \leq r.
$$
This certainly holds when the blowing up centers are permissible of the first or second kind
by Propositions  {\ref{SingX}}, \ref{Estable} and Theorem \ref{bupthm}. Another example is blowing up along
codimension one centers of the form $V(Z,u_j)$ with $d_j\leq 1$, $1 \leq j \leq e$.
In chapter 8, we will use another kind of Hironaka-permissible blowing up
with the same property.  We recall that all permissibility conditions (Definitions
\ref{Hironakapermis}, \ref{deffirstkind} and \ref{defsecondkind}) always refer to the
reduced total transform $E_i$ of $E$ in $S_i$, where there are projections
$$
\eta_i : ({\cal X}_i,x_i) \longrightarrow \mathrm{Spec}S_i, \ 0 \leq i \leq r.
$$
Similarly, $\omega (x_i), \epsilon (x_i), \kappa (x_i)$ are always computed w.r.t. $E_i$.

\smallskip

We emphasize that we do {\it not} require any particular behavior
about the numbers $\kappa (x_i)$ {\it along} the process (\ref{eq402}). Our goal
is to {\it eventually} achieve $\kappa (x_r)<\kappa (x)$ and  we
may have $\kappa (x_i)>\kappa (x)$ for some $i$, $1 \leq i < r$.

 {Our strategy consists in looking for expansions of in$_{m_S}(h)$ in each case $\kappa(x)=2,3,4$ which are stable by permissible blowing ups. These stable expansions are denoted respectively (*) (Definition \ref{*kappadeux}), (**) and (T**) (Definition \ref{**}).  Our first goal is to reach these conditions. Achieving this first step involves sequences of blowing ups (\ref{eq402}) where   we may have $\kappa(x_i)>\kappa(x)$. See for example Proposition~\ref{redto*} which relies on Lemma~\ref{sortiemonome} and Propositions~\ref{sortiekappaegaldeux} and \ref{sortiebis}.
}

\begin{defn}\label{defgood}
Assume that $m(x)=p$ and $\omega (x)>0$. Given any finite sequence (\ref{eq402}),
we say that $x_r$ is {\it very near} $x$ if $\iota (x_r)\geq \iota (x)$. \index{nearv @ very near !  $x_r$ is {very near} $x$, Definition~\ref{defgood}}

\smallskip

Let $a\in \{1, \ldots ,4\}$. We say that $x$ is {\it resolved for} \index{resolved !  $x$ is resolved for..., Definition~\ref{defgood}} $(p, \omega (x), a)$ (resp.
{\it resolved for} $m(x)=p$) if for every valuation $\mu$ of $L=k({\cal X})$ centered at $x$,
there exists a finite and independent sequence (\ref{eq402})  {(cf. Definition \ref{indepseq})} such that
$\iota (x_r)<(p,\omega (x),a)$ (resp. $m(x_r)<p$).
We simply say that $x$ is {\it good} \index{good, $x$ is good, Definition~\ref{defgood}} if $x$ is resolved for $\iota (x)$.
\end{defn}

The following Projection Theorem is proved in the next sections: Corollary \ref{projthmkappa1},
Theorem \ref{proofkkappa2}, Theorem \ref{proofkappa34}, {\it ibid.}, for $\kappa (x)=1,2,3,4$ respectively.

\begin{thm}\label{projthm}\textbf{(Projection Theorem).}
Assume that $(S,h,E)$ satisfies assumption {\bf (G)} and {\bf (E)}, with $m(x)=p$ and  $\omega (x)>0$.

\smallskip

For every valuation $\mu$ of $L=k({\cal X})$ centered at $x$, there exists a finite and
independent composition of local Hironaka-permissible blowing ups
(\ref{eq402}) such that $\iota(x_r)<\iota (x)$, i.e. $x$ is good.
\end{thm}

\begin{cor}\label{projthmcor}
Theorems \ref{mainthm} and \ref{luthm} hold true.
\end{cor}

\begin{proof}
Theorem \ref{mainthm} has been reduced to Theorem \ref{luthm} for residually algebraic valuations,
Propositions \ref{redtoLU} and \ref{redtoLUcyclic}. By Corollary \ref{EEfait}, it can be furthermore
assumed that condition {\bf (E)} is satisfied.
Theorem \ref{luthm} is then an immediate consequence of \cite{CoP4}
Main Theorem 1.3 ($m(x)<p$), Theorem \ref{omegazero}
($(m(x),\omega (x))=(p,0)$) and  Theorem \ref{projthm}.
\end{proof}

\begin{rem}\label{quadsequence}
Let $\mu$ be a valuation of $L=k({\cal X})$ centered at $x$ and consider an independent sequence of
local blowing ups (Definition \ref{indepseq})
$$
    ({\cal X},x)=:({\cal X}_0,x_0) \leftarrow ({\cal X}_1,x_1)\leftarrow \cdots \leftarrow ({\cal X}_r,x_r)\leftarrow \cdots
$$
along $\mu$. For example, the quadratic sequence along $\mu$ is an independent sequence.

\smallskip

Then $x$ is resolved for $(p,\omega (x),a)$ if for every $\mu$, there exists some $r=r(\mu)\geq 0$ such that $x_r$ is resolved
for $(p,\omega (x),a)$ (the converse follows from Definition \ref{defgood} with $r(\mu)=0$ for every $\mu$).
This fact is used all along the next chapters, {\it vid.} chapter 7 for $a=2$ and chapter 8 for $a=3$.
\end{rem}

\begin{prop}\label{tausup2}
With assumptions as in Theorem \ref{projthm}, assume furthermore that
$\mathrm{Max}(\mathrm{in}h) \neq \mathrm{Dir}(\mathrm{in}h)$,
where $\mathrm{in}h \in k(x)[U_1,U_2,U_3,Z]_p$ is the initial form of $h$ (Proposition \ref{conedirectrix}).
Then $\kappa (x)\geq 2$ and $x$ is resolved for $(p,\omega (x),2)$.
\end{prop}

\begin{proof}
By Remark \ref{ridgedimthree}, the assumption holds only if $p=2$ and
$$
\mathrm{in}h =Z^2 +F, \ F:= \lambda_2 U_1^2 +\lambda_1U_2^2+\lambda_1 \lambda_2U_3^2
$$
with $[k(x)^2(\lambda_1,\lambda_2):k(x)^2]=4$ up to a linear change of variables. We have
$H(x)=(1)$, $\omega (x)=\epsilon (x)=2$ and  $\kappa (x)=4$ (resp. $\kappa (x)=2$) if
$E=\mathrm{div}(u_1u_2u_3)$ (resp. otherwise). Since
$$
 {J(F,E,m_S)}=<{\partial F \over \partial \lambda_2}, {\partial F \over \partial \lambda_1}>
=<U_1^2 +\lambda_1 U_3^2, U_2^2 +\lambda_2 U_3^2>,
$$
we have $\tau '(x)=3$.
Let ${\cal X}' \rightarrow ({\cal X},x)$ be the blowing up along $x$ and $x'\in \pi^{-1}(x)$.
Since $\tau '(x)=3$, we have $\iota (x')\leq (2,2,1)$ by Theorem \ref{bupthm}.
\end{proof}



\section{Maximal contact, resolution of $\kappa (x)=1$.}\label{contactmaximal}

We assume in the whole section that $(S,h,E)$ satisfies conditions {\bf (G)} and
{\bf (E)}. We consider here any refinement ${\cal C}$ of the function  $x\mapsto (m(x),\omega (x))$ on
${\cal X}$.

\smallskip

Fix an irreducible component $\mathrm{div}(u_1)\subseteq E$. Let $\mu$ be a valuation of $L=k({\cal X})$
centered at  $x$. We consider in this chapter finite sequences (\ref{eq402}) of local blowing ups along $\mu$:
\begin{equation}\label{contactmaxeq1}
    ({\cal X},x)=:({\cal X}_0,x_0) \leftarrow ({\cal X}_1,x_1) \leftarrow \cdots \leftarrow ({\cal X}_r,x_r) ,
\end{equation}
with {\it permissible centers of the first kind} ${\cal Y}_i \subset ({\cal X}_i,x_i)$, where $x_i$,
$0 \leq i \leq r$, denotes the center of $\mu$. It is furthermore assumed that

\smallskip

\noindent (1) $\eta_i({\cal Y}_i)$ belongs to the strict transform of $\mathrm{div}(u_1)$ in $\mathrm{Spec}S_i$, where
$$
\eta_i : \ ({\cal X}_i,x_i) \longrightarrow \mathrm{Spec}S_i
$$
is the projection, {\it vid.} Proposition \ref{Hironakastable}, and

\smallskip

\noindent (2) ${\cal C}$ is not increasing along (\ref{contactmaxeq1}), i.e. ${\cal C}(x_i)\leq {\cal C}(x_{i-1})$,
$1 \leq i \leq r$.

\begin{defn}\label{Maximalcontact}
We say that div$(u_1)\subseteq E \subset {\cal X}$ has ``maximal contact'' \index{maximal contact @ maximal contact, weak maximal contact,  Definition \ref{Maximalcontact}} (resp. ``weak maximal contact'')
for some refinement ${\cal C}$ if for every $\mu$, any sequence (\ref{contactmaxeq1})
(resp. the quadratic sequence (\ref{contactmaxeq1}) with ${\cal Y}_i:=\{x_i\}$) satisfies the following:
\begin{equation}\label{DefinitionC}
{\cal C} (x_r)={\cal C} (x) \Longrightarrow x_r \ \mathrm{maps} \ \mathrm{to} \  \mathrm{the} \
\mathrm{strict} \  \mathrm{transform} \  \mathrm{of}  \ \mathrm{div}(u_1).
\end{equation}
\end{defn}

\begin{rem}\label{remC}
Take ${\cal C}=\iota$, where $\kappa (x)=1$. Then $\mathrm{div}(u_1)\subseteq E$ has maximal contact
for ${\cal C}$ if $U_1$ divides $H^{-1}G^p$, with notations as in Definition \ref{defomega}. This
follows from Theorem \ref{bupthm}.
\end{rem}

\smallskip

The purpose of this section is to prove Theorem \ref{contactmaxFIN} below: the value ${\cal C}(x)$ of
any such refinement can be lowered by permissible blowing ups of the first kind. A direct application proves Theorem \ref{projthm}
for $\kappa (x)=1$. Further applications are given in chapter 8. The proof of this theorem uses a
secondary invariant $\gamma (x)\in \N$ which is defined and studied afterwards, {\it viz.} (\ref{eq6022})
and (\ref{eq6031}).

\begin{thm}\label{contactmaxFIN}
Assume that  div$(u_1)$ has  maximal contact for $\cal C$. Let $\mu$ be a valuation of $L=k({\cal X})$ centered at $x$,
where $m(x)=p$ and $\omega (x)>0$.
There exists a finite and independent composition of local permissible blowing ups of the first kind:
\begin{equation}\label{eqcontactmax2}
    ({\cal X},x)=:({\cal X}_0,x_0) \leftarrow ({\cal X}_1,x_1) \leftarrow \cdots \leftarrow ({\cal X}_r,x_r) ,
\end{equation}
where $x_i \in {\cal X}_i$ is the center of $\mu$,  such that ${\cal C} (x_r)<{\cal C} (x)$ or $x_r$ is resolved
for $m(x)=p$.
\end{thm}

\begin{proof}
By Proposition \ref{omegapositiveclosed}, the set
$$
\Omega_+ ({\cal X}):=\{y \in {\cal X} : (m(y), \omega (y))> (p,0)\}\subseteq {\cal X}
$$
is Zariski closed and of dimension at most one. By performing the quadratic sequence (\ref{contactmaxeq1}),
it can be assumed that there exist well adapted coordinates $(u_1,u_2,u_3;Z)$ at $x$
such that any one dimensional irreducible component ${\cal Y}$ of $\Omega_+ ({\cal X})$,
with $\eta ({\cal Y})$ contained in $\mathrm{div}(u_1)$ either:

\smallskip

\noindent (a) maps to an intersection of components of $E$, i.e.
$$
    \eta ({\cal Y})=V(Z,u_1,u_j), \ \mathrm{div}(u_j)\subseteq E, \ j\geq 2, \ \mathrm{or}
$$
\noindent (b) $\eta ({\cal Y})=V(Z,u_1,u_3)$, $E\subseteq \mathrm{div}(u_1u_2)$.

\smallskip

Furthermore, there exists at most one  ${\cal Y}$ satisfying (b)
and such ${\cal Y}$ is permissible of the first kind by Proposition \ref{permisarc}(1). Let
${\cal X}' \rightarrow ({\cal X},x)$ be the blowing up along such ${\cal Y}$. Replacing
$({\cal X},x)$ by $({\cal X}',x')$, where $x'$ is the center of $\mu$, we may therefore
assume that any  {one-dimensional} irreducible component ${\cal Y}$ of $\Omega_+ ({\cal X})$,
with $\eta ({\cal Y})$ contained in $\mathrm{div}(u_1)$, satisfies (a) above.

\smallskip

Consider now the quadratic sequence (\ref{contactmaxeq1}) and apply Proposition \ref{contactmaxeclatpoint}
below. If alternative (ii) of that proposition holds, the theorem follows from Proposition \ref{permisarc}(2),
since the conclusion of Proposition \ref{permisarc}(1) does not hold by the above preparation of
$\Omega_+ ({\cal X})$. Assume then that alternative (i) of Proposition \ref{contactmaxeclatpoint} holds.
Then the conclusion follows from Proposition \ref{contactmaxpetitgamma} below.
\end{proof}

\begin{cor}\label{projthmkappa1}
Projection Theorem \ref{projthm} holds when $\kappa (x)=1$.
\end{cor}

The arguments are quite similar to \cite{CoP2} chapter~4 pages 1957 and following and we sketch
the argument below. This section may serve as an introduction to the more involved material in the next chapter.

\begin{nota}\label{betagammamaxcontact}

 {We assume that div$(u_1)$ has maximal contact or weak maximal contact. Then, we may also assume that $\mathrm{div}(u_1u_2)\subseteq E$. Indeed, after the first blowing up, $E'$ will contain at least two components: the strict transform of div$(u_1)$ and the new exceptional component.}

\smallskip

\noindent {\it Cases  1 and 2:} $\epsilon(x)=\omega(x)$ and ($E=\div(u_1 u_2)$ or $E=\div(u_1 u_2 u_3)$ respectively).
Let $(u_1,u_2,u_3;Z)$ be well adapted coordinates. Consider the characteristic polyhedron
$$
\Delta_S(h; u_1,u_2,u_3;Z)\subset \R^3_{\geq 0}
$$
in the affine space with origin $\mathbf{v}_0:=(d_1+\omega(x)/p,d_2,d_3)$  {with the convention $d_3=0$ when div$(u_3)\not\subseteq E$}. Perform the stereographic projection
$\mathbf{p}'_2$ from $\mathbf{v}_0$ on the plane $x_1=0$,  followed by the homothety of center $(0,0)$ and
ratio ${p \over \omega(x)}$. Let $\mathbf{p}_2$ be the resulting map. Analytically, we have:
\begin{equation}\label{eq602}
\mathbf{p}_2 : \ (x_1,x_2,x_3) \mapsto (y_2,y_3):={1\over {\omega (x)\over p}-(x_1-d_1)}(x_2-d_2,x_3-d_3).
\end{equation}

We denote for simplicity
\begin{equation}\label{eq6021}
\Delta_2 (x)\index{$\Delta_2$ @ $\Delta_2$ when there is maximal contact, Notation \ref{betagammamaxcontact}\eqref{eq6021}}:=  {\mathbf{p}_2} (\Delta(h; u_1,u_2,u_3;Z)\cap \{0\leq x_1 -d_1 <\omega (x)/ p\}).
\end{equation}
There are associated invariants:
\begin{equation}\label{eq6023}
\left\{
  \begin{array}{ccc}
    A_j(x) \index{$A$1 @$A_1,\ A_2,\ B,\ C,\ \beta,\ \beta_2, \gamma$,   when there is maximal contact, Notation \ref{betagammamaxcontact}\eqref{eq6023}\eqref{eq6022}\eqref{eq6031}}& := & \inf \ \{ y_j \ \vert \ (y_2,y_3)\in \Delta_2 (x)\} \hfill{}  \\
    B(x) & := & \inf \ \{ y_2+y_3\ \vert \   {(y_2,y_3)}\in \Delta_2(x)\} \hfill{} \\
   C(x) & := & B(x)-A_2(x)-A_3(x)\geq 0 \hbox{ in case (2)}\\
     C(x) & := &   {B(x)-A_2(x)\geq 0 \hbox{ in case (1)}}\\
    \beta (x) & := & \inf \ \{y_3\ \vert \  (A_2(x),y_3)\in \Delta_2 (x)\} \hfill{}\\
    \beta_2 (x) & := & \sup \ \{ y_3 \ \vert \ (y_2,y_3)\in \Delta_2 (x), y_2+y_3=B(x)\} \hfill{}\\
  \end{array}
\right.
.
\end{equation}

The main secondary invariant is:
\begin{equation}\label{eq6022}
\gamma(x) := \left\{
                          \begin{array}{ccc}
                            {\max}\{1, \lceil \beta (x) \rceil \} & \mathrm{if} & E= \mathrm{div}(u_1 u_2) \\
                             &  \\
                            1+\lfloor C(x) \rfloor & \mathrm{if} & E= \mathrm{div}(u_1 u_2u_3)\\
                          \end{array}
                        \right.
. \hfill{}
\end{equation}

Note that $\Delta_2 (x)\neq \emptyset$: this follows from (\ref{eq602}) and the definition of $d_1$.
Therefore
$$
A_2 (x),A_3(x), B(x)< +\infty.
$$
It is easily seen that $\Delta_2 (x)\subseteq \R^2_{\geq 0}$ is a polygon.
Since all vertices of $\Delta_S(h; u_1,u_2,u_3; {Z)} - (d_1,d_2,d_3)$ have module at least ${\epsilon(x) \over p}$,
we have $B(x)\geq 1$.

\smallskip

\noindent {\it Case 3:} $\epsilon(x)=1+\omega(x)$, $E=\div(u_1u_2)$. The definition is the same
as in cases 1 and 2 except that $\mathbf{v}_0$ is replaced by $\mathbf{v}'_0:=(d_1+\omega(x)/p,d_2,1/p)$.
Analytically, we have:
\begin{equation}\label{eq603}
\mathbf{p}_2 : \ (x_1,x_2,x_3) \mapsto (y_2,y_3):={1\over {\omega (x)\over p}-(x_1-d_1)}(x_2-d_2,x_3-1/p).
\end{equation}
Note that the image set $\Delta_2(x)$ defined by (\ref{eq6021}) may contain points with negative third coordinate.
The invariants $A_2 (x)$, $B(x)$,  $C(x):=B(x)-A_2(x)$ and $\beta (x)$ are defined as in cases 1 and 2. We let:
\begin{equation}\label{eq6031}
\gamma (x):= \max\{1 +\lfloor \beta (x) \rfloor ,1\}.
\end{equation}

\end{nota}

These definitions depend in principle on $(u_1,u_2,u_3)$, but certainly not on $Z$ such that $(u_1,u_2,u_3;Z)$ are
well adapted coordinates. Indeed, the above  {definitions} are given in terms of $\Delta(h;u_1,u_2,u_3;Z)$.
It can be proved that the numbers $A_j (x)$, $B(x)$,  $C(x)$, $\beta (x)$ and $\gamma (x)$ are
actually independent of $(u_1,u_2,u_3;Z)$ once the numbering of the components of $E$ is fixed. We skip
this fact here and refer to the next chapter (Theorem \ref{well2prepared} and
Definition \ref{definvariants2} in particular) for similar issues.

\begin{rem}\label{calculB}
The numbers $B(x),A_j(x)$ can be computed directly from the equation $h$.

In cases 1-2, let $(a,b)$ be positive real numbers such that
$$a(d_1+{\omega(x) \over p})+b(d_2+d_3)=1$$
with the convention $d_3=0$ when div$(u_3)\not\subseteq E$. Define a monomial valuation $v_{(a,b,b)}$
on $S[Z]$ by setting weights:
$$
v_{(a,b,b)}(u_1)=a, \ v_{(a,b,b)}(u_2)=v_{(a,b,b)}(u_3)=b, \ v_{(a,b,b)}(Z)=1.
$$
Then
$$
B(x)=\sup\{ {a \over b} \vert v_{(a,b,b)}(h)=p \}.
$$
The  pair $(a,b)$ giving the sup above is said to ``define $B(x)$'' ({\it viz.} \cite{CoP2} Theorem {\bf I.4},
equation (3) page 1962). As $B(x)\geq 1$, we have
 {\begin{equation}\label{eq:contactmaxa>b}
a\geq b.
\end{equation}}
 We denote:
\begin{equation}\label{eq6032}
H_B:=\clin_ {v_{(a,b,b)}}(h)=Z^p+\sum_{1\leq i \leq p} \Phi_i Z^{p-i}  , \Phi_i \in k(x)[U_1,U_2,U_3],
\end{equation}
where  $(a,b)$  ``defines $B(x)$''. By Theorem \ref{initform}, we have $\Phi_i=0$, $1 \leq i \leq p-2$ and
$-\Phi_{p-1}=G^{p-1}$ where $G$ is a constant times a monomial in $U_1,\ldots ,U_e$. We expand
the corresponding initial form as in (\ref{eq6032}) and let
\begin{equation}\label{eqcontactmax12}
 {U_1^{-pd_1}}U_2^{-pd_2}U_3^{-pd_3}\Phi_p=\lambda U_1 ^{\omega(x)} +\sum_{i=1}^{\omega(x)} U_1 ^{\omega(x)-i} F_i(U_2,U_3),
\ \lambda \in k(x),
\end{equation}
where $F_i\in k(x)[U_2,U_3]$ is homogeneous of degree $iB(x)$.

 {Note that $b\leq 1$: indeed, $b>1$ would give $a\geq b >1$ and in (\ref{eq6032}) deg$(\Phi_{p-1})<p-1$ or
deg$(\Phi_{p})<p$, which  contradicts ord$_x(h)=p$.
Furthermore, we may assume
\begin{equation}\label{eq:contactmax1>b}
 b<1\ \mathrm{or}\ a=b=1.
\end{equation}
Indeed, $b=1,\ a>1$ gives the same contradiction as above.}
\smallskip

More generally, let $\sigma_2$ be a compact face of $\Delta_2 (x)$. The topological closure of the set
$$
\sigma:=\Delta_S(h;u_1,u_2,u_3;Z)\cap \mathbf{p}_2^{-1}(\sigma_2)
$$
is a compact face of $\Delta_S(h;u_1,u_2,u_3;Z)$ defined by a weight vector $\alpha:=\alpha_{\sigma_2}$.
The corresponding initial form polynomial is written
\begin{equation}\label{eq6034}
H_{\alpha} {:=}Z^p+\sum_{1\leq i \leq p} \Phi_{i,\alpha} Z^{p-i}  , \Phi_{i,\alpha} \in \mathrm{gr}_\alpha (S),
\end{equation}

In case 3, there exists a unique compact face $\sigma \subset \Delta_S(h;u_1,u_2,u_3;Z)$ whose image
by $\mathbf{p}_2$ is the face $ {y_2+y_3}=B(x)$, maximal for this property. For $B(x)=1$,
$$
\sigma_{\mathrm{in}}:=\{\mathbf{x}\in \R^3_{\geq 0}: x_1+x_2+x_3=\delta (x)\}
$$
obviously has this property. For $B(x)> 1$, we expand the corresponding initial form as in (\ref{eq6032})
and let
\begin{equation}\label{eqcontactmax3*}
 {U_1^{-pd_1}}U_2^{-pd_2}\Phi_p=U_1 ^{\omega(x)} ( \lambda_3 U_3+ \lambda_2 U_2)
+\sum_{i=1}^{\omega(x)} U_1 ^{\omega(x)-i} F_i(U_2,U_3),
\end{equation}
with $\lambda_2 ,\lambda_3\in k(x)$, $F_i\in k(x)[U_2,U_3]$ homogeneous of degree $1+iB(x)$.

\smallskip

In cases 1-2-3, let $(a,b)$ be positive real numbers such that
$$
a(d_1+{\omega(x) \over p})+bd_2=1.
$$
We have similarly:
$$
A_2(x)=  \sup\{ {a \over b} \vert v_{(a,b,0)}(h)=p \},
$$
this suitable pair $(a,b)$ is also said to ``define $A_2(x)$''. We denote:
\begin{equation}\label{eq6033}
H_{2} {:=}\clin_ {v_{(a,b,0)}}(h)=Z^p+\sum_{1\leq i \leq p} \phi_i Z^{p-i}  , \phi_i \in {S \over (u_1,u_2)}[U_1,U_2],
\end{equation}
where  $(a,b)$  ``defines $A_2(x)$''(\cite{CoP2} Theorem {\bf I.4}, valuation $\mu_1$ on page 1962).
We expand the $\phi_i$, $1\leq i \leq p$:
$$
\phi_i= \sum_{j=0}^{\omega (x)} U_1^j U_2^{b(i,j)} \phi_{i,j},\ b(i,j)={i\over b}-jA_2(x) ,
\ \phi_{i,j}\in {S \over (u_1,u_2)},
$$
where  {${1\over b}=A_2(x)(d_1+{\omega(x) \over p})+d_2$}.
\end{rem}

All proofs are based on the following elementary lemma:

\begin{lem}\label{lem532}
Let $(R,\frak{m},k)$ be a regular local ring of dimension two,  {$\frak{m}=(v_2,v_3)$, $\mathrm{char}k=p>0$.
Let $f\in R$ with initial form
$$
\mathrm{in}_{\frak{m}}f=V_2^{a_2}V_3^{a_3}F(V_2,V_3)\in G(\frak{m}), \ \mathrm{in}_{\frak{m}}f\not\in G(\frak{m})^p.
$$
Let furthermore $P(t)\in R[t]$ be monic of degree $d\geq 1$ with irreducible residue $\overline{P}(t)\in k[t]$,
$$
R':=R\left \lbrack{v_3 \over v_2}\right \rbrack_{(v_2,v)}, \ v:=P \left ({v_3 \over v_2}\right )
$$
and for every $\alpha\in R'$, $\tilde{\alpha}:= \alpha\ \mathrm{mod}(v_2)\in {R'\over v_2R'}$. We define:
$$
a':=\max_{g'\in R'}\{\mathrm{ord}_{v_1}(f-{g'}^p)\},
\ e':=\max_{g'\in R'}\{\mathrm{ord}_{\tilde{v}}(\widetilde{v_2^{-a'}(f-{g'}^p)}) : \mathrm{ord}_{v_2}(f-{g'}^p))=a'\}.
$$
The following hold:
\begin{itemize}
  \item [(1)] $a'=a_2+a_3+\mathrm{deg}(F)$, $e'\leq 1+\lfloor {\mathrm{deg}F \over d}\rfloor$; if equality holds, then
  ${\mathrm{deg}F\over d }\in \N$, $a'/p \in \N$, $e'/ p\not \in\N$, and
  $$
  J(\mathrm{in}_{\frak{m}}f, \mathrm{div}(v_2v_3),\frak{m})=
  <\left ( V_2^d P \left ({V_3 \over V_2}\right )\right )^{{\mathrm{deg}F \over d}}>;
  $$
  \item [(2)] if $a_3=0$, then $e'\leq \max\{\mathrm{deg}F,1\}$. Equality holds only if $\mathrm{deg}F\leq 1$ or $d=1$.
\end{itemize}}
\end{lem}

\begin{proof}
  {We suppose neither $a_2$ maximal nor $a_3$ maximal, i.e. we may have $F(0,V_3)=0$ or $F(V_2,0)=0$.} The proof is identical to \cite{CoP2} {\bf II.5.3.2} on p. 1862. Note that it is not necessary to assume $R$ excellent.
\end{proof}

Now we follow \cite{CoP2} chapter~4. Consider the blowing up $\pi: \ {\cal X}' \rightarrow ({\cal X},x)$ at $x$
and let $x'\in \pi^{-1}(x)$ be a closed point, with  $d:=[k(x'):k(x)]$.
Following  \cite{CoP2} Theorem {\bf I.4} on p.1962, we have:


\begin{prop}\label{eclatpointcas12}
With hypotheses and notations as above, assume that $x$ is in case 1-2. Let $(u_1,u_2,u_3;Z)$
be well adapted coordinates at $x$ and assume furthermore that
$$
 {x'}\in \mathrm{Spec}(S[{u_1\over u_2},{u_3\over u_2}][Z']/(h')), \ h':=u_2^{-p}h, \ Z':={Z \over u_2}.
$$
If ${\cal C}(x')={\cal C}(x)$, we have:
\begin{equation}\label{eq604}
A_2(x')=B(x)-1, \ \gamma (x')\leq \gamma (x),
\end{equation}
and there exist well adapted coordinates $(u'_1:=u_1/u_2,u_2, {v};Z')$ at $x'$ such that the following holds:

\begin{itemize}
  \item [(1)] if $x'=(Z/u_2,u'_1,u_2,u_3/u_2)$, then $x'$ is again in case 1-2
  and
  $$
  C(x')\leq C(x), \ \beta(x')\leq \beta(x);
  $$
  \item [(2)] if $x'\neq (Z/u_2,u'_1,u_2,u_3/u_2)$, then $x'$ is in case 1 or 3. We have
  \begin{equation}\label{eq6041}
  \beta (x') \leq \left\{
  \begin{array}{cc}
   1+ \lfloor {C(x)\over d}\rfloor \hfill{}& \mathrm{if} \  x' \ \mathrm{is} \ \mathrm{in} \ \mathrm{case} \ 1 \\
   & \\
   {C(x)\over d}  & \mathrm{if} \ x' \ \mathrm{is} \ \mathrm{in} \ \mathrm{case} \ 3 \\
  \end{array}
  \right.
  ,
  \end{equation}
 {and  $\Phi_{p-1}\neq 0$ implies
 \begin{equation}\label{eq6041.2}
  \begin{array}{cc}
  \Phi_{p-1}=\lambda U_1^{(p-1)(d_1+{\omega(x)\over p})} U_2^{(p-1)d_2}U_3^{(p-1)d_3},\ \lambda\in k(x)^*  \hfill{} \\
   \mathrm{or}  \\
 \beta (x')=0 \hbox{ if } x' \hbox{ is in case 1, }\beta (x')<0\hbox{ if }x' \hbox { is in case 3}. \\
  \end{array}
  \end{equation}}

\smallskip

If moreover $x$ is in case 1 and $\beta (x)>0$, we have
\begin{equation}\label{eq6042}
  \left\{
  \begin{array}{cc}
   \beta (x') \leq \beta (x) \hfill{}& \mathrm{if} \  x' \ \mathrm{is} \ \mathrm{in} \ \mathrm{case} \ 1 \\
   & \\
   \beta (x')<\beta (x)  & \mathrm{if} \ x' \ \mathrm{is} \ \mathrm{in} \ \mathrm{case} \ 3 \\
  \end{array}
  \right.
  .
  \end{equation}
Furthermore, $x'$ is in case 3 only if $k(x')$ is inseparable over $k(x)$ (in particular $p$ divides $d$).
\end{itemize}
\end{prop}

\begin{rem}\label{rem:1*2*}
 {The case where
$$ \Phi_{p-1}=\lambda U_1^{(p-1)d_1} U_2^{(p-1)d_2}U_3^{(p-1)d_3}U_1^{(p-1){\omega(x)\over p}},\ \lambda\in k(x)^* $$
 is denoted 1* when $x$ is in case~1, resp.  2* when $x$ is in case~2. The monomial $ \Phi_{p-1}$ corresponds to the vertex $v_0$ defined in Notation \ref{betagammamaxcontact}}.
\end{rem}

\begin{proof}
 {Statement (1): by  Proposition \ref{originchart}, $(u'_1,u_2,u_3/u_2;Z/u_2)$ are well adapted coordinates at $x'$. Furthermore,  $\Delta_2(x')=l_1(\Delta_2(x))+\R_{>0}^2$, where $l_1$ is the affine transformation $\R^2 \longrightarrow \R^2$, $l_1(b,c):=(b+c-1,c)$.   These transformation laws are  the classical transformations of the characteristic polyhedron of a surface singularity and give statement (1). }

 {For (2), we define
\begin{equation}\label{eq6041.1}
\begin{array}{ccc}
\bar{u}_3:={u_3\over u_2}\ \mathrm{mod} (u'_1,u_2)\in {S'\over (u'_1,u'_2)} & \hbox{when } a>b\\
 \bar{u}_1:={u_1\over u_2}\ \mathrm{mod} (u_2),\ \bar{u}_3:={u_3\over u_2}\ \mathrm{mod} (u_2)\in {S'\over (u'_2)} & \hbox{when } a=b.\\
 \end{array}
 \end{equation}
 we have $v$~mod$(u'_1,u_2)=P(\bar{u_3})\in k(x)[\bar{u_3}]$, $P$ irreducible. In the extreme case $a=b=1$ (\ref{eq:contactmax1>b}), we have in$_{m_S}(h)=i$n$_x(h)$, $\tau(x)\geq 2$. The existence of $x'$ implies $d_2=0$, $\beta(x)=1$ and $U_3$ mod $(U_1,U_2)$ is in the ideal of the directrix of in$_x(h)$. The end of the proof is left to the reader. From now on, we assume $1>b$.  Let $H'_2:= \clin_ {v_{(a',b',0)}}(h')$ with $a':={a-b \over 1-b}$ and $b':={b\over 1-b}$ with $(a,b)$ defined in (\ref{eq6032}).
 Clearly $ v_{(a',b',0)}$ does not depend on the choice of $v$.
 An easy computation gives with the notations of (\ref{eq6032}):
\begin{equation}\label{eq6041.2}
\begin{array}{ccc}
H'_2={Z'}^p+ {U_2}^{-p+1}\Phi_{p-1} (U'_1U_2,U_2,\bar{u}_3U_2)Z' +  {U_2}^{-p}\Phi_{p}(U'_1U_2,U_2,\bar{u}_3U_2)  & \hbox{when } a>b\\
H'_2={Z'}^p+ {U_2}^{-p+1}\Phi_{p-1} (\bar{u}_1U_2,U_2,\bar{u}_3U_2)Z' +  {U_2}^{-p}\Phi_{p}(\bar{u}_1U_2,U_2,\bar{u}_3U_2)  & \hbox{when } a=b\\
 \end{array}
\end{equation}
$\Phi_i (U'_1U_2,U_2,\bar{u}_3U_2)\in k(x)[U'_1,U_2,\bar{u}_3]={S' \over(u'_1,u_2)}[U'_1,U_2]$ when $a>b$, resp. $\Phi_i (\bar{u}_1U_2,U_2,\bar{u}_3U_2)\in k(x)[\bar{u}_1,U_2,\bar{u}_3]={S' \over(u_2)}[U_2]$  when $a=b$. As $H'_2\not= {Z'}^p$, $a'x_1+b'x_2+0x_3=1$ is the equation of a face of $\Delta(h';u'_1,u_2,v;Z')$. This face cannot be solved by translation on $Z'$ as $H_B$  of (\ref{eq6032}) is not a $p$-th power: an eventual translation minimizing the polyhedron $Z'\leftarrow Z'+\theta$, $\theta\in S'$, will verify $v_{(a',b',0)}(\theta)\geq 1$. Furthermore, the initial form polynomial $H_2$ in (\ref{eq6033})
at $x'$ is $H'_2$ and  has $A_2(x')=B(x)-1$: this gives the equality in (\ref{eq604}).   }

 {When $\Phi_{p-1} \not=0$,   it is a monomial in $U_1, U_2$ (case 1) or  in $U_1, U_2,U_3$ (case 2) by Theorem \ref{initform}.  Let:
\begin{equation}\label{eq6041.3}
\begin{array}{c}
\Phi_{p-1} =\lambda U_1^{(p-1)d_1+d}U_2^{(p-1)d_2+e}U_3^{(p-1)d_3+f}, \ \lambda\in k(x),\\
 \ d+(p-1)d_1, e+(p-1)d_2,f +  (p-1)d_3\in \N \ \mathrm{when} \ \lambda\not=0,\\
\ \ d,e,f \in \Q_{\geq 0},\   d\leq (p-1){\omega(x)\over p}.\\
\end{array}
\end{equation}}

 {When $\Phi_{p-1} \not=0$ and $d<(p-1){\omega(x)\over p}$, then $e+f=i_0B(x)$ with $i_0=(p-1){\omega(x)\over p}-d$. In this case, the coefficient of $Z'$ in $H'_2$ is:
\begin{equation}\label{eq6041.4}
\begin{array}{ccc}
\lambda {U'_1}^{(p-1)d_1}U_2^{i_0(B(x)-1)}{U'_1}^{d} \times\mathrm{invertible}\ \mathrm{when}\ a>b,\\
\lambda \bar{u_1}^{(p-1)d_1} \bar{u_1}^{d} \times\mathrm{invertible}\ \mathrm{when}\ a=b.\\
\end{array}
\end{equation}
As this coefficient is invariant by an eventual translation $Z'\leftarrow Z'+\theta$ with $v_{(a',b',0)}(\theta)\geq 1$.
 we get $\beta(x')=0$ when $x'$ is in case~1, and $\beta(x')=-{1\over i_0}<0$ when  $x'$ is in case~3. This gives  in this case all the equalities and inequalities in (\ref{eq604}) (\ref{eq6041})(\ref{eq6042}) and (1).}

 {From now on, we suppose
\begin{equation}\label{eq6041.4}
\Phi_{p-1} =0\ \mathrm{or}\ \Phi_{p-1} =\lambda U_1^{(p-1)(d_1+{\omega(x)\over p} ) }U_2^{(p-1)d_2}U_3^{(p-1)d_3}.
\end{equation}
}
 {Let
$$\Phi_p= U_1^{pd_1}U_2^{pd_2}U_3^{pd_3}\sum_{i=0}^{\omega(x)} U_1 ^{\omega(x)-i} F_i(U_2,U_3),$$
deg$(F_i)=iB(x)$ when $F_i\not=0$. Let $i_0:=$sup$\{i \vert F_i\not=0\}>0$ and  $w:=(d_1+{\omega(x)-i_0\over p}, d_1+d_2+d_3-1+{i_0 B(x)\over p} , w_3)$ be the vertex of  smallest abscissa of the face of equation  $a'x_1+b'x_2+0x_3=1$  of $\Delta(h';u'_1,u_2,v;Z')$. This vertex $w$ is defined by the monomial $ U_2^{p(d_1+d_2+d_3-1)+i_0 B(x)}{U'_1} ^{pd_1+\omega(x)-i_0}F_{i_0}(1,\bar{u_3})$: we have $w_3={1\over p}$ord$_v(F_{i_0}(1,\bar{u_3}))$. }

 {In the case where
\begin{equation}\label{eq:coordpasent}
pd_1+\omega(x)-i_0\not=0\hbox{ mod }p\ \mathrm{or}\ p(d_1+d_2+d_3-1)+i_0 B(x)\not=0\hbox{ mod }p,
\end{equation}
this vertex $w$ is not solvable and we get
\begin{equation}
\begin{array}{cc}
\beta(x')=  {1\over i_0}\mathrm{ord}_v(F_{i_0}(1,\bar{u_3}))\hbox{ when } x' \hbox{ is in case 1 },\\
\beta(x')=  {1\over i_0}(\mathrm{ord}_v(F_{i_0}(1,\bar{u_3}))-1) \hbox{ when } x' \hbox{ is in case 3 },\\
\end{array}
\end{equation}
which gives (2) in this case. When (\ref{eq:coordpasent}) is not true, a translation  $Z'\leftarrow Z'+\theta$, with
$$\theta=\gamma{u'_1}^{d_1+{\omega(x)-i_0\over p}} u_2^{d_1+d_2+d_3-1+{i_0 B(x)\over p}} v^{x_3},\  \gamma \in S',\ \gamma\ \mathrm{invertible},$$
may solve $w$.
}
  {By (\ref{eq6041.4}), the eventual contribution of  the coefficient of $Z'$ to the term of degree $0$ in $Z'$ of $h'$  will be divisible by ${u'}_1^{(p-1)(d_1+{\omega(x) \over p})}{u'}_1^{d_1+{\omega(x)-i_0 \over p}}$, as
 $$ (p-1)(d_1+{\omega(x) \over p})+ {d_1+{\omega(x)-i_0 \over p}}=pd_1+\omega(x)-{i_0\over p}>pd_1+\omega(x)-i_0=pw_1,$$
 the eventual translation  translation may only  spoil vertices of this face  with a bigger abscissa, when $a>b$:}

\noindent  { it will just add a  $p$-th power to  $ U_2^{p(d_1+d_2+d_3-1)+i_0 B(x)}{U'_1} ^{\omega(x)-i_0}F_{i_0}(1,\bar{u_3})$, resp. $ U_2^{p(d_1+d_2+d_3-1)+i_0 B(x)}\bar{u_1} ^{\omega(x)-i_0}F_{i_0}(1,\bar{u_3})$ when $a=b$. }

 { Let $F_{i_0}(U_2,U_3)=:U_2^{a_2}U_3^{a_3}F(U_2,U_3)$ with $a_2$ maximal, $a_3=0$ if $x$ is in case~1, $a_3$ maximal if $x$ is in case~2. We have the inequality: $\beta(x')\leq$ord$_{\bar{v}}(F_{i_0}(1,\bar{u_3})+\theta^p)/i_0$, with strict inequality when $x'$ is in case~3. When $x$ is in case~1 or 2,   $\beta(x)\geq {\mathrm{deg}(F)\over i_0}$, $C(x)\geq {\mathrm{deg}(F)\over i_0}$.  The inequalities  (\ref{eq6041}) and (\ref{eq6042}) follow from the fact that
$$\beta(x')\leq {1\over i_0}\mathrm{ord}_{\bar{v}}(F_{i_0}(1,\bar{u_3})+\theta^p)$$
and Lemma~\ref{lem532}~(2) which give ord$_{\bar{v}}(F_{i_0}(1,\bar{u_3})+\theta^p)\leq{\mathrm{deg}F\over d}+1$ and in case of equality, ord$_{\bar{v}}(F_{i_0}(1,\bar{u_3})+\theta^p)\not=0$ mod $p$.}

 {Finally, assume that $x$ is in case 1. Let
\begin{equation}\label{eq6041.6}
\begin{array}{c}
\clin_{m_S}(h)=:Z^p+\lambda U_1^{(p-1)d_1}U_2^{(p-1)d_2}U_1^{a}U_2^{b} Z +U_1^{pd_1}U_2^{pd_2}\sum_{i=0}^{\omega(x)}U_1^{b_i}G_i(U_2,U_3), \\
\lambda\in k(x), \ a+(p-1)d_1,b+(p-1)d_2\in \N
\ \mathrm{when}\ \lambda\not=0,\\
b_i\in \N,\ G_i\in k(x)[U_2,U_3].\\
\end{array}
\end{equation} }
 {It is clear that, when $\lambda\not=0$, $x'$ is in case~1. Suppose $\lambda=0$. Let  $i_0:=$sup$\{i \vert G_i\not=0\}>0$. When $i_0>0$,  the proof runs along the same lines as above: $x'$ is in case~1. When $i_0=0$, then
\begin{equation}\clin_{m_S}(h)=:Z^p+\lambda' U_1^{pd_1}U_2^{pd_2} U_1^{\omega(x)}, d_1+d_2+{\omega(x)\over p}=\delta(x),\ \lambda'\in k(x)^*.
\end{equation}
The first face of $\Delta(h;u_1,u_2,u_3;Z)$ has only one vertex: $w=(d_1+{\omega(x)\over p},d_2,0)$ which will give the vertex of smallest ordinate $w'=(d_1+{\omega(x)\over p},\delta(x)-1,0)$ of $\Delta(h';u'_1,u_2,v;Z')$. When
\begin{equation}
  \left\{
  \begin{array}{cc}
   d_1+{\omega(x)\over p}\not \in \N ,\\
   \mathrm{or}\ \delta(x)-1\not \in \N , \\
    \mathrm{or}\ \lambda'\not\in k(x')^p ,\\
  \end{array}
  \right.
  \end{equation}
$w'$ is not solvable and $\omega(x')=\epsilon(x')$, $x'$ is in case~1.
When none of the conditions above are satisfied, then the coordinates of  $w=(d_1+{\omega(x)\over p},d_2,0)$ are in $\N$ and, as $w$ is not solvable, $\lambda' \not\in k(x)^p$: so $\lambda' \in k(x')^p \setminus k(x)^p$, $k(x')$ is inseparable over $k(x)$.}

\end{proof}

\begin{cor}\label{cor:beta=2}
 {With hypotheses and notation of \ref{eclatpointcas12},
assume that $x$ is in case 1 with $\beta (x)=2$. Then}

 {\noindent $\beta (x_1)\leq 2$ ($\beta (x_1)< 2$ if $k(x_1)\neq k(x)$) if $x_1$ is again in case 1.}

 {\noindent  If $x_1$ is in case 3, and $k(x_1)\neq k(x)$, we get $\beta (x_1)<1$.}

\end{cor}

\begin{proof}
 {The only case to consider is  $k(x_1)\neq k(x)$ and (\ref{eq6041.4}). As $\beta(x')\leq {1\over i_0}\mathrm{ord}_{\bar{v}}(F_{i_0}(1,\bar{u_3})+\theta^p)$, the result is clear except if:}

\noindent  {$\bullet$ $i_0=1$, $pd_1+\omega(x)-1=0$ mod $p$,}

\noindent  {$\bullet$ $d=2=$deg$(F)$,}


 {With the the notations of Lemma \ref{lem532}, $\beta(x')=e'$. By Lemma \ref{lem532}~(2), $e'<$deg$(F)=2$.}

\end{proof}

Following now \cite{CoP2}  Theorem {\bf I.5} on page 1964:

\begin{prop}\label{eclatpointcas3}
With hypotheses and notations as above, assume that $x$ is in case 3. Let $(u_1,u_2,u_3;Z)$
be well adapted coordinates at $x$ and assume furthermore that
$$
 {x'}\in \mathrm{Spec}(S[{u_1\over u_2},{u_3\over u_2}][Z']/(h')), \ h':=u_2^{-p}h, \ Z':={Z \over u_2}.
$$
If ${\cal C}(x')={\cal C}(x)$, we have
$$
A_2(x')=B(x)-1, \ \gamma (x')\leq \gamma (x),
$$
and there exist
well adapted coordinates $(u'_1:=u_1/u_2,u_2, {v};Z')$ at $x'$ such that the following holds:
\begin{itemize}
  \item [(1)] if $x'$ is in case 1, then
  $$
  \beta (x')\leq {\gamma(x) \over d} +1;
  $$
  \item [(2)] if $x'$ is in case 3, then
  $$
  \beta (x')\leq \max\{\beta (x),0\}
  $$
  and $\beta (x')< \beta (x)$ if ($k(x')\neq k(x)$ and $\beta (x)> 0$);
\end{itemize}
\end{prop}

\begin{proof}
As in the preceeding proof, we look at the initial form polynomial  {$H_B$} (\ref{eq6032}) corresponding to the valuation $v_{a,b,b}$ with $B(x)={a\over b}$. By (\ref{eqcontactmax3*}),
the term of degree $0$ in $Z$ is:
 {
\begin{equation}
\Phi_p= U_1^{pd_1}U_2^{pd_2}[( \lambda_3 U_3+ \lambda_2 U_2)
U_1 ^{\omega(x)}+\sum_{i=1}^{\omega(x)} U_1 ^{\omega(x)-i} F_i(U_2,U_3)].
\end{equation}}

The initial form polynomial  $H'_2:= \clin_ {v_{(a',b',0)}}(h')$ with $a':={a-b \over 1-b}$ and $b':={b\over 1-b}$ with $(a,b)$ defined in (\ref{eq6032}) is in fact the form $H_2$ at $x'$ and  has $  {A_2}(x')=B(x)-1$, the term of degree $0$ in $Z$ is:
 {\begin{equation}\label{eq:cas3x'}
\begin{array}{ccc}
{U'_1}^{pd_1}{U'_2}^{\epsilon(x)-p} [( \lambda_3 u'_3+ \lambda_2 ){U'_1} ^{\omega(x)}
+\sum_{i=1}^{\omega(x)} {U'_1} ^{\omega(x)-i}{U'_2}^{i(B(x)-1 )} F_i(1,u'_3)],  & \hbox{when } a>b,\\
{u'_1}^{pd_1}{U'_2}^{\epsilon(x)-p} [( \lambda_3 u'_3+ \lambda_2 ){u'_1} ^{\omega(x)}
+\sum_{i=1}^{\omega(x)} {u'_1} ^{\omega(x)-i}{U'_2}^{i(B(x)-1 )} F_i(1,u'_3)],  & \hbox{when } a=b,\\
 \end{array}
\end{equation}}
and  {when $\Phi_{p-1}= 0$,} the upper bounds for $\beta (x')$ follow from Lemma \ref{lem532}.

By (\ref{eq6031}), note that
$$
\mathrm{deg} F_i(U_2,U_3) - {i}A_2(x)\leq i\gamma (x)
$$
in (\ref{eqcontactmax3*}) whenever $F_i(U_2,U_3)\neq 0$.   {When $\Phi_{p-1}= 0$, we apply Lemma \ref{lem532} to $U_2^{pd_2} F_i(U_2,U_3)$ with $a_2> pd_2+ i A_2(x)$ and $a_3=0$. This gives
$$i\beta(x') \leq {i\gamma (x)\over d}+1.$$
  One deduces the upper bounds of (1) or (2) and
$\gamma (x')\leq \gamma (x)$.}


  {If $\Phi_{p-1}\neq 0$,  it is a monomial in $U_1, U_2$ by Theorem \ref{initform} and,
 we are in the case (\ref{eq6041.3}) with $d_3=f=0$ and
 $$ a(d_1+{d\over p-1})+b(d_2+{e\over p-1})=a(d_1+{\omega(x)\over p})+b(d_2+{1\over p}),$$ which leads to
 $$d\not=(p-1){\omega(x)\over p},$$
 when $d<(p-1){\omega(x)\over p}$, using the same arguments as in the proof of Proposition \ref{eclatpointcas12}, we get
 \begin{equation}\label{eq:betax'=0}
 \beta (x')=0\ \hbox{(resp. }\beta (x')<0\hbox{) if }x' \hbox{ is in case 1 (resp. in case 3)}.
 \end{equation}
 It may be possible that $d>(p-1){\omega(x)\over p}$, then $e=0$ and $\omega(x)={(p-1)b\over a}\leq p-1$, in this extreme case, we conclude as above, using the index $i_0:=$sup$\{i \vert F_i\not=0\}>0$.}

\end{proof}

\begin{cor}\label{cor:beta=1cas3}
 {With hypotheses and notation of \ref{eclatpointcas3},
assume that  $\beta (x)=1$ and $x'$ is in case~1. Then, one of the following is true: }

\begin{itemize}
  \item [(1)]  $\beta(x')<2$

  \item [(2)] $x$ is in case 1 with: $\beta (x)=2$ and
  $$
  \Phi_{p,\alpha}=\sum_{i=0}^{\omega (x)}{U'_1}^{\omega (x)-i}\Phi_{p,\alpha ,i}(U_2,V)
  $$
  has $\Phi_{p,\alpha ,1}\neq 0$ with notations as in (\ref{eq6034}), where  {$\sigma_2:=\{(A_2(x),2)\}$}.

\end{itemize}
\end{cor}
\begin{proof}
 {Indeed, in the proof above, we get
$i\beta(x') \leq {i\gamma (x)\over d}+1$. So (1) is true when $i\not=1$, when $i=1$, we get (2).}
\end{proof}
Following \cite{CoP2} Lemma {\bf I.5.3} on page 1966:

\begin{prop}\label{eclatpointcas3infiny}
With hypotheses and notations as above, let $(u_1,u_2,u_3;Z)$
be well adapted coordinates at $x$ and assume furthermore that
$$
x'=(Z':=Z/u_3,u'_1:=u_1/u_3,u'_2:=u_2/u_3,u_3).
$$
If ${\cal C}(x')={\cal C}(x)$, then $x'$ is in case 2,
$(u'_1,u'_2,u_3;Z')$ are well adapted coordinates at $x'$,
$$
A_3(x')=B(x)-1, \ \beta (x')=A_2(x)+\beta (x)-1, \ \gamma (x')\leq \gamma (x),
$$
and the following holds:

\begin{itemize}
  \item [(1)] if $x$ is in case 1, then $C(x')\leq \min\{\beta (x)-C(x)-A_3(x),C(x)\}$;
  \item [(2)] if $x$ is in case 2, we have $C(x')\leq \min\{\beta (x) -C(x),C(x)\}$.
  \item [(3)] if $x$ is in case 3, we have $C(x')\leq \min\{\beta (x)-C(x), C(x)-\beta_2(x)\}$.
\end{itemize}
\end{prop}

\begin{proof}
The argument is the same as in \cite{CoP2} Lemma {\bf I.5.3} on page 1966. This relies on the characteristic free Proposition \ref{originchart}  {which asserts that no changes in $Z'$ need to be performed in order to get well adapted data.
It is easy to see that $\Delta_2(x')$ is obtained from $\Delta_2(x)$ by applying the affine transformation: $(v_2,v_3)\mapsto (v_2,v_2+v_3-1)$ and adding quadrants.
In fact we focus on two vertices (maybe equal) of $\Delta_2(x)$: $(A_2(x),\beta(x))$ and $(B(x)-\beta_2(x),\beta_2(x))$. They become two vertices of $\Delta_2(x')$:  $(A_2(x),\beta(x)+A_2(x)-1)$ and $(B(x)-\beta_2(x),B(x)-1)$ which are respectively the vertex of smallest abscissa and the vertex of smallest ordinate of $\Delta_2(x')$. So
\begin{equation}\label{eq:ptinfinity}
\begin{array}{cc}
C(x')\leq \beta(x)+A_2(x)-1-(B(x)-1 )=  & \beta(x)+A_2(x)-B(x) \\
  \hbox{ in case 2}& = \beta (x)-C(x)-A_3(x)\\
  \hbox{ in case 1,3}& =  \beta (x)-C(x) \\
C(x')\leq B(x)-\beta_2(x)-A_2(x)&  \hfill{}\\
 \hbox{ in case 2} & \leq C(x)\\
\hbox{ in cases 1,3}&= C(x)-\beta_2(x)  \\
\end{array}
\end{equation}
This
gives all statements
except ``$\gamma (x')\leq \gamma (x)$''.}

\smallskip

Finally, $\gamma (x')\leq \gamma (x)$ is a trivial consequence of the Definitions (\ref{eq6022})
and (\ref{eq6031})  except if ($x$ is in case 3,   {$\beta_2(x)<0$ and $C(x)<0$}). But then $\beta_2(x)=-1/i$ for some $i$,
$1\leq i\leq \omega (x)$ and (3) gives
$$
C(x')\leq C(x)-\beta_2(x)<1,
$$
so $\gamma (x')\leq 1$ as required.

\end{proof}

\begin{rem}\label{rem:gamma=2ptinfinity}  {With the hypotheses of Proposition \ref{eclatpointcas3infiny} above, when $\gamma(x)\geq 2$,  $x$ is in case 1 or 3 and  $\gamma(x')=\gamma(x)$, then we have $\gamma(x)=2$ and:}

\begin{itemize}
   \item [(1)]   {$x$ is in case 1 and
  $
  \beta (x)=2,\ C(x)=1$,}

  \item [(2)]  { $x$ is in case 3 and
  $
 1\leq \beta (x)<2,\ \beta(x)-C(x)\geq 1.
  $}
\end{itemize}

\end{rem}
We now go ahead to prove Theorem \ref{contactmaxFIN}. The key lemma to reach the case $\gamma(x)=1$ goes as follows:

\begin{lem}\label{lemC=1}
Assume that div$(u_1)$ has weak contact maximal for  $\cal C$. Let $\mu$ be  {a} valuation
of $L=k({\cal X})$ centered at $x$ and consider the quadratic sequence (\ref{contactmaxeq1}) along $\mu$,
i.e. with ${\cal Y}_i=\{x_i\}$ for every $i\geq 0$.

\smallskip

Assume that one of the following holds:
\begin{itemize}
  \item [(1)] $x$ is in case 1 with: $\beta (x)=2$ and
  $$
  \Phi_{p,\alpha}=\sum_{i=0}^{\omega (x)}U_1^{\omega (x)-i}\Phi_{p,\alpha ,i}(U_2,U_3)
  $$
  has $\Phi_{p,\alpha ,1}\neq 0$ with notations as in (\ref{eq6034}), where  {$\sigma_2:=\{(A_2(x),2)\}$};
  \item [(2)] $x$ is in case 3 with $\beta (x)=1$.
\end{itemize}

Assume furthermore that $x_1= (Z':=Z/u_3,u'_1:=u_1/u_3,u'_2:=u_2/u_3,u_3)$, ${\cal C}(x_1)={\cal C}(x)$
and $\gamma (x_1)=2$. Then ${\cal C}(x_2)<{\cal C}(x)$  {or $\gamma (x_2)=1$ or ($x_1$ is in case~1* and $x_2$ in case~2* of Remark~\ref{rem:1*2*} with $\beta(x_2)<2$}).
\end{lem}

\begin{proof}
Note that $x_1$ is in case 2 with $\gamma (x_1)=2$ by assumption.
By Proposition \ref{eclatpointcas3infiny}, we get $A_2(x_1)=A_2(x)$ and respectively:

\smallskip

\noindent (1) $C(x)=C(x_1)=1$, $\beta_2(x)=0$, $A_3(x_1)=B(x)-1=A_2(x)$, $\beta (x_1)=A_2(x)+1$;

\smallskip

\noindent (2) $C(x)=0$, $\beta_2(x)=-1$, $C(x_1)=1$, $A_3(x_1)=A_2(x)-1$, $\beta (x_1)=A_2(x)$.

\smallskip
  {The vertices of smallest abscissa  of smallest ordinate of $\Delta_2(x')$ are:  }

 \smallskip

\noindent { (1) $(A_2(x),1+A_2(x))$ and $(1+A_2(x),A_2(x))$,}

\smallskip

\noindent ( {2)  $(A_2(x),A_2(x))$ and  $(A_2(x),A_2(x)-1)$.}

\smallskip
 {These facts  imply that in both cases:
$$
\Delta_2(x_1)=(A_2(x_1),A_3(x_1))+\{(y_2,y_3)\in \R^2_{\geq 0} : y_2+y_3\geq 1\}.
$$}
 {When  $x_2$ is again in case 2, we get $C(x_2)=0$ by (\ref{eq:ptinfinity})}. Otherwise, we may assume
that ${\cal C}(x_2)={\cal C}(x)$ and apply Proposition \ref{eclatpointcas12} to estimate $\gamma (x_2)$. We
get $\gamma (x_2)=1$ if $k(x_2)\neq k(x)$ by  {(2)} of this proposition.

\smallskip

Assume that $k(x_2)= k(x)$. We claim that,  {when $x_1$ is not in case 2*},  the following sharper bound holds, which concludes the proof:
\begin{equation}\label{eq6043}
   \beta (x_2)\leq 1 \ (\mathrm{resp.} \ \beta (x_2)\leq 0)
\end{equation}
if $x_2$ is in case 1 (resp. in case 3).

\smallskip

There are associated $d'_1,d'_2,d'_3 \in 1/p\N$ at $x_1$ with $d'_1=d_1$, $d'_2=d_2$ and
$$
d'_3=d_1+d_2-1 +{\omega (x) \over p} \ (\mathrm{resp.} \ d'_3=d_1+d_2-1 +{1+\omega (x) \over p})
$$
if $x$ is in case 1 (resp. in case 3).

\smallskip

 The initial form $H_B$ (\ref{eqcontactmax12}) at $x_1$
\begin{equation}\label{eq:HB}
H_B={Z'}^p- Z' \Phi_{p-1}+{U'_1}^{pd'_1}{U'_2}^{pd'_2}U_3^{pd'_3}\sum_{i=0}^{\omega(x)} {U'_1}^{\omega(x)-i}F_i(U'_2,U_3)
\end{equation}
has $F_1(U'_2,U_3)\neq 0$ and is
of the form
$$
F_1(U'_2,U_3)={U'_2}^{a_2}U_3^{a_3}F(U'_2,U_3), a_2+a_3+\mathrm{deg}(F)=1+ A_2(x_1)+A_3(x_1)$$
where $a_2\geq A_2(x)$, $a_3\geq A_3(x)$, and either $F\in k(x)$ or
\begin{equation}\label{eq6044}
   a_2=A_2(x_1),\ a_3=A_3(x_1)\in \N \ \mathrm{and} \ F=\lambda_2U'_2+\lambda_3U_3, \ \lambda_3\neq 0.
\end{equation}
The point $x_2$ has for parameters
$(Y,v_1,v_2,v):=(Z'/u'_2,u'_1/u'_2,u'_2,u'_3/u'_2+\lambda)$, $\lambda\in S'$, $\lambda$ invertible.  The initial form $H_2$ at $x_2$
$$H_2=Y^p -Y {G''}^{p-1}+V_1^{pd'_1}V_2^{p(d'_1+d'_2+d'_3)+\omega(x)-p}\sum_{i=0}^{\omega(x)} {V_1}^{\omega(x)-i}V_2^{i(A_2(x_1)+A_3(x_1))}F"_i(1,\bar{v} -\bar{\lambda})$$
has $F''_1$ of the form:
$$
F_1''=F(1,\bar{v} -\bar{\lambda})(\bar{v} -\bar{\lambda} )^{pd'_3+A_3(x_1)}
$$
which leads to a point $(d'_1+(\omega (x)-1)/p, d'_1+d'_2+d'_3+ {A_2(x_1)+A_3(x_1)+\omega(x)\over p}, {a\over p})$ with $a=0$ or $1$.

 {Note that, when $\Phi_{p-1}\not=0$, by Proposition~\ref{eclatpointcas12} (\ref{eq6041.2}), we get the claim (\ref{eq6043}) except when  $x_1$ is in case 2*.
 $$\Phi_{p-1}=\lambda U_1^{(p-1)(d'_1+{\omega(x)\over p} )} U_2^{(p-1)d'_2}U_3^{(p-1)d'_3}\ \lambda\in k(x)^*.$$
 In this last case, $x_2$ is in case~2* and $d'_1+{\omega(x)\over p},d'_2,d'_3 \in \N$.
 In this case we come back to the argument of the proof of Proposition~\ref{eclatpointcas12}: let  $i_0:=$sup$\{i \vert F'_i\not=0\}>0$, we have $F'_{i_0}={U'_2}^a {U'_3}^b F(U'_2,U_3)$ with $a\geq i_0 A_2(x_1)$, $b\geq i_0 A_3(x)$ and deg$(F'_{i_0})\leq i_0$. An eventual translation on $Y$ may only add a$p$-th power to
 $${V_1}^{pd'_1+\omega(x)} {V_2}^{pd'_1+\omega(x)+pd'_2+a+pd'_3+b-p}(\bar{v} -\bar{\lambda} )^{pd'_3+b}F'_{i_0}(1,\bar{v} -\bar{\lambda} ).$$
 When $i_0\not=0$~mod~$p$ or deg$(F'_{i_0})< i_0$ we get the inequalities (\ref{eq6043}) by Lemma~\ref{lem532}(1). When  $i_0=0$~mod~$p$ and deg$(F'_{i_0})= i_0$, Lemma~\ref{lem532}(1) gives the last assertion of Lemma~\ref{lemC=1}. }

 \smallskip
 {Until the end of the proof, we assume $\Phi_{p-1}\not=0$ at $x_1$.}  We get the inequalities (\ref{eq6043}) provided
$$
d'_1 +(\omega (x)-1)/p \not\in\N  \  \mathrm{or}  \ d'_2+d'_3 + {A_2(x_1)+A_3(x_1)+1 \over p}\not \in \N.
$$
Indeed no $p$-th power may pollute
$$V_1^{pd'_1}V_2^{p(d'_1+d'_2+d'_3)+\omega(x)-p}{V_1}^{\omega(x)-1}V_2^{A_2(x_1)+A_3(x_1)}F''_1(1,\bar{v} -\bar{\lambda})$$



Under assumption (1),
when this fails to hold, we have (\ref{eq6044}) with $\lambda_2 \neq 0$ and
\begin{equation}\label{eq6045}
  d_1 +(\omega (x)-1)/p \in\N , \  2(d_2+{A_2(x)+1 \over p}) \in \N
\end{equation}
by the above calculations.  {In case $p=2$, by Lemma \ref{lem532}(1), ord$_{\bar v}F_1''+ p$-th power$\leq 1$: we have the inequalities (\ref{eq6043}). When $F\in k(x)$ (ref. line above (\ref{eq6044})),  once again, by Lemma \ref{lem532}(1), ord$_{\bar v}F''_1+ p$-th power$\leq 1$. So we have just to look at the case $p\geq 3$, deg$(F)$=1, the latter implies  $a_i=A_i(x_1)\in \N$, $i=2,3$. } We deduce that $d_2+{A_2(x)+1 \over p} \in \N$, which in turn implies that
$$
U'_2 \in J({U'_2}^{pd_2 +A_2(x_1)}U_3^{pd'_3 +A_2(x_1)}F(U'_2,U_3), \mathrm{div}(u'_2u_3),\frak{m})
$$
with notations as in Lemma \ref{lem532}(1), applying $U_3{\partial \hfill{} \over \partial U_3}$. Then
equality is strict in Lemma \ref{lem532}(1) and the conclusion follows.

\smallskip

Under assumption (2),  note that since $\beta_2(x)=-1$ we necessarily have $F_1(U'_2,U_3)\neq 0$ or
$$
{H}^{-1}G^p =<{U_1}^{\omega (x)-1}U_2^{1+A_2(x)}>.
$$
In the former case, the proof is parallel to that under assumption (1), exchanging the roles of $U'_2, U_3$.
 {In the latter case, as   $e={(1+A_2(x))(p-1)\over p}\not=0$, we conclude from Proposition \ref{eclatpointcas12} (\ref{eq:betax'=0}) with $\Phi_{p-1}\neq 0$.}
\end{proof}



\begin{prop}\label{contactmaxeclatpoint}
Assume that div$(u_1)$ has weak maximal contact for  $\cal C$. Let $\mu$ be a valuation
of $L=k({\cal X})$ centered at $x$ and consider the quadratic sequence (\ref{contactmaxeq1}) along $\mu$,
i.e. with ${\cal Y}_i=\{x_i\}$ for every $i\geq 0$.

\smallskip

If ${\cal C}(x_i)={\cal C}(x)$ for every $i\geq 0$, one of the following is true:

\smallskip

\begin{itemize}
  \item [(i)] $\gamma(x_i)=1$ for every $i>> 0$, or
  \item [(ii)] there exists a formal arc  {(Definition \ref{defformalarc})} $\varphi : \ \mathrm{Spec}{\cal O}\rightarrow ({\cal X},x)$ with
$l|k(x)$ finite algebraic, support $Z:=Z(\varphi)$ with
$$
\eta (Z)\subseteq \mathrm{div}(u_1),
$$
$\eta (Z)$ not an intersection of components of $E$, whose strict transform passes through $x_i$ for every $i\geq 0$.
\end{itemize}
\end{prop}

\begin{proof}
Note that (ii) fails to hold if and only if: for every $i\geq 0$, there exists $i'>i$ such that either
$k(x_{i'})\neq k(x_i)$ (i.e. some of Proposition \ref{eclatpointcas12}, \ref{eclatpointcas3}
applies to $x_{i'}$ with $d\geq 2$) or $x_{i'}$ is in case 2.

\smallskip

Assume therefore that (ii) does not hold. By Propositions \ref{eclatpointcas12}, \ref{eclatpointcas3}
and \ref{eclatpointcas3infiny}, we have $\gamma (x_{i+1})\leq \gamma (x_i)$ for every $i\geq 0$
and inequality is strict for $i'$ as above if $\gamma (x_{i'})\geq 3$. W.l.o.g. it can be assumed
that
 {$$\gamma (x_i)=2\hbox{ for every }i\geq 0.$$
This implies by (\ref{eq6022}) (\ref{eq6031}):}

 {$\bullet$ when $x_i$ is in case 1: $\lceil {\beta(x_i)} \rceil =2$,  thus $1< \beta(x_i)\leq 2$,}

 {$\bullet$ when $x_i$ is in case 2: $\lfloor C(x_i)\rfloor=1$, thus $1\leq C(x_i)<2$,}

 {$\bullet$ when $x_i$ is in case 3: $\lfloor {\beta(x_i)} \rfloor =2$,  thus $1\leq \beta(x_i)<2$.}

   We now derive a contradiction by studying different cases.

\smallskip

\noindent (a) if $x$ is in case 1 with $\beta (x)<2$, we are done.  {Indeed take the smallest $i'$ as above, by  the last statement of  Proposition~\ref{eclatpointcas12}, $x_{i_0}$ is in case~1 for $i_0\leq i'$,  by Proposition~\ref{eclatpointcas12}(1), $\beta(x_{i'-1})\leq \beta(x)<2$. Either $k(x_{i'})\neq k(x_i)$,  Proposition~\ref{eclatpointcas12} gives $\beta(x_{i'})\leq 1$ when $x_{i'}$ is in case~1, $\beta(x_{i'})<1$  when $x_{i'}$ is in case~3. Or $x_{i'}$ is in case~2 (point at infinity), then $C(x_{i'})<1$ by Proposition~\ref{eclatpointcas3infiny}(1).}

\smallskip

 {Assume that $x$ is in case 1 with $\beta (x)=2$. By Proposition \ref{eclatpointcas12} and Corollary \ref{cor:beta=2}, we obtain
$\beta (x_1)\leq 2$ ($\beta (x_1)< 2$ if $k(x_1)\neq k(x)$) if $x_1$ is again in case 1.
If $x_1$ is in case 3, we get $\beta (x_1)<1$.}

\smallskip

Assume that $x$ is in case 3. If Proposition \ref{eclatpointcas3} applies, we obtain $\beta (x_1)\leq \beta (x)$
(with strict equality if $k(x_1)\neq k(x)$) if $x_1$ is again in case 3.
If $x_1$ is in case 1, we get $\beta (x_1)\leq 2$; if furthermore $\beta (x)=1$,  {by corollary\ref{cor:beta=1cas3} the} inequality is
strict unless $x_1$ satisfies the assumptions of Lemma \ref{lemC=1}(1). We deduce:

\smallskip

\noindent (b) if $x$ is in case 3 with $\beta (x)=1$, we are done: this follows from Lemma \ref{lemC=1}
and the previous comments.

\smallskip

\noindent (c) if $x$ is in case 1 with $\beta (x)=2$, we are done: we may assume that
Proposition \ref{eclatpointcas3infiny} applies by the previous comments; we reach (a)(b) or the
assumptions of Lemma \ref{lemC=1}(1) at $x_2$ since it is assumed that $\gamma (x_2)=2$.

\smallskip

\noindent (d) the remaining cases: $x$ is in case 2 or $x$ is  in case 3 with $\beta (x)>1$.

\smallskip

\noindent (d-1) $x$ is in case 2. The result is trivial
if $x_i$ is in case 2 for every $i>>0$.  { Indeed,  let $(\beta_3(x),A_3(x))$ the vertex of smallest ordinate of $\Delta_2(x)$, then at $x'$ the point at infinity, the vertices of smallest abscissa (resp. smallest ordinate) $(A_2(x),\beta+A_2(x)-1)$ resp.   $(B(x)-\beta_2(x),B(x)-1)$: when $C(x)>0$, we get $(\beta(x)-A_3(x))+(\beta_3(x)-A_2(x))>(\beta(x')-A_3(x'))+(\beta_3(x')-A_2(x'))$, by symmetry there is the same inequality when $x'$ is the point of parameters $(Z/u_2,u_1/u_2,u_2,u_3/u_2)$.} Otherwise, note that:   {by Proposition~\ref{eclatpointcas12}}, $C(x_1)\leq C(x)$ if $x_1$ is in case 2; $\beta (x_1)\leq C(x)$ if $x_1$ is in case 3.

\noindent (d-2)   $x$ is  in case 3 with $\beta (x)>1$. Note that $C(x_1)<\beta (x)$ if $x_1$ is in case 2;
by Proposition~\ref{eclatpointcas3}(2), $\beta (x_1)\leq \beta (x)$ if $x_1$ is in case 3  {with strict inequality when $k(x_1)\not=k(x)$. In the case where $x_i$ is in case 3 for $i>>0$, we reach $\beta(x_i)<1$ or $k(x_i)=k(x_{i+1})$ for $i>>0$, in the latter, we are in case (ii) of the proposition.}

 The conclusion follows easily.

\end{proof}

\begin{prop}\label{contactmaxpetitgamma}
Assume that div$(u_1)$ has maximal contact for  $\cal C$ and that $\gamma (x)=1$. Let $\mu$ be valuation
of $L=k({\cal X})$ centered at $x$. There exists a finite and independent composition of local permissible blowing ups
of the first kind:
$$
    ({\cal X},x)=:({\cal X}_0,x_0) \leftarrow ({\cal X}_1,x_1) \leftarrow \cdots \leftarrow ({\cal X}_r,x_r) ,
$$
where $x_i \in {\cal X}_i$ is the center of $\mu$,  such that ${\cal C} (x_r)<{\cal C} (x)$ or $x_r$
is resolved for $m(x)=p$.
\end{prop}

\begin{proof}
We may assume that ${\cal C} (x_i)={\cal C} (x)$ for every $i\geq 1$ for the resolution process to be defined below;
we will either derive a contradiction or prove that $x_r$ is resolved for $m(x)=p$ for some $r\geq 0$.
 {By Propositions~\ref{eclatpointcas12}, \ref{eclatpointcas3} and \ref{eclatpointcas3infiny}, we have $\gamma(x_i)=1$ for all $i\geq 0$.
This implies by (\ref{eq6022}) (\ref{eq6031}):}

 {$\bullet$ when $x_i$ is in case 1: $\lceil {\beta(x_i)} \rceil \leq 1$,  thus $0\leq \beta(x_i)\leq 1$,}

 {$\bullet$ when $x_i$ is in case 2: $\lfloor C(x_i)\rfloor=0$, thus $0\leq C(x_i)<1$,}

 {$\bullet$ when $x_i$ is in case 3: $\lfloor {\beta(x_i)} \rfloor \leq 0$,  thus $C(x_i)\leq \beta(x_i)<1$.}

Suppose that $i\geq 1$ and that
\begin{equation}\label{eqcontactmax11}
A_2(x_{i-1})<1 \ \mathrm{and} \ (x_{i-1} \ \hbox{is in case 2} \Longrightarrow \beta (x_{i-1}) < 1).
\end{equation}

Then we consider the quadratic sequence (\ref{contactmaxeq1}) along $\mu$. In every case, we have
$$
A_2(x_i)\leq A_2(x_{i-1}),
$$
where inequality is strict except if either Proposition \ref{eclatpointcas3infiny} applies,
or ($x_{i-1}$ is in case 1 with $\beta(x_{i-1})=1$). If Proposition
\ref{eclatpointcas3infiny} applies, we have
$$
\beta(x_i)=A_2(x_{i-1}) + \beta(x_{i-1})-1<1.
$$
This proves in particular that (\ref{eqcontactmax11}) holds at $x_{i'}$ for every $i'\geq i$. W.l.o.g.
it can be assumed that,  {when (\ref{eqcontactmax11}) occurs, then $i=1$.}

\smallskip

If $x$ is in case 1 with $\beta (x)=1$ and $k(x_1)\neq k(x)$, then $\beta(x_1)<1$ by
 {Proposition~\ref{eclatpointcas12}}; if Proposition \ref{eclatpointcas3infiny} applies to $x$,
then $\beta (x_1)< \beta (x)$. In other terms, we have
$$
(A_2(x_1),\beta (x_1))< (A_2(x),\beta (x))
$$
for the lexicographical ordering except possibly if $x$ is in case 1 with $\beta (x)=1$
and $k(x_1)=k(x)$. So in the sequence  (\ref{contactmaxeq1}), we may
assume that $x_i$ is in case 1 with
$$
A_2(x_i)=A_2(x)<1, \ \beta (x_i)=\beta (x)=1 , \ k(x_i)=k(x)
$$
for every $i \geq 0$. Applying Proposition \ref{permisarc}, we are done if alternative (2) of this
proposition holds; if alternative (1) holds, it can be assumed that there exists a permissible curve of the first
kind ${\cal Y}=V(Z,u_1,u_3)\subseteq ({\cal X},x)$.  Then $x$ is resolved by blowing up ${\cal Y}$:
in view of Definition \ref{Maximalcontact}, we need only  {to} consider the point
$x':=(Z/u_3,u_1/u_3,u_2,u_3)$ and get $\omega (x')< \omega (x)$ from Proposition \ref{originchart}. This proves the
proposition under the extra assumption (\ref{eqcontactmax11}).\\

We now consider several cases which are proved consecutively:

\smallskip

\noindent (a)  $x$ {\it is in case 1}. We have $A_2(x)\geq 1$ if the extra assumption (\ref{eqcontactmax11}) does not hold.
Let $(u_1,u_2,u_3;Z)$ be well adapted coordinates at $x$ and note that ${\cal Y}:=V(Z,u_1,u_2)$
is a permissible curve of first kind. Blowing up along ${\cal Y}$,  {as div$(u_1)$ has maximal contact for $\cal C$, the only point with $\cal {C}(x_1)=\cal {C}(x)$ is }
$x_1=(Z/u_2,u_1/u_2,u_2,u_3)$, in which case $x_1$ is again in case 1 with
$$
(A_2(x_1),\beta (x_1))=(A_2(x)-1,\beta (x)).
$$
The proof concludes by induction on $A_2(x)$. Before going along with the proof in cases 2 and 3,
we make the following remark:

\begin{rem}\label{remmaxcont}
Assume that $x$ is in case 2 with $A_2(x)\geq 1$. Let $(u_1,u_2,u_3;Z)$ be
well adapted coordinates at $x$ and denote ${\cal Y}:=V(Z,u_1,u_2)$ with generic point $y$.
Since $\epsilon (y)=\epsilon (x)$, ${\cal Y}$ is permissible of the first kind if and only if it is
Hironaka-permissible w.r.t. $E$, i.e. if $m(y)=m(x)=p$. Thus:
\begin{equation}\label{eq606}
{\cal Y} \ \mathrm{is} \ \mathrm{permissible} \ \mathrm{of} \ \mathrm{the}
\ \mathrm{first} \ \mathrm{kind} \Leftrightarrow d_1+d_2 +{\omega (x)\over p}\geq 1.
\end{equation}

Suppose that ${\cal Y}$ is  Hironaka-permissible. Blowing up along ${\cal Y}$ and arguing as in (a), we achieve:
\begin{equation}\label{eq6061}
x_1 \ \mathrm{in} \ \mathrm{case} \ 2, \ A_2(x_1)=A_2(x)-1, \ A_3(x_1)=A_3(x).
\end{equation}
This proves that it can be assumed to begin with that
\begin{equation}\label{eq6062}
A_j(x)<1 \ \mathrm{or} \ d_1+d_j +{\omega (x)\over p}< 1
\end{equation}
for each of $j=2,3$.

\smallskip

Assume that $x$ is in case 2 with $d_1+\omega (x)/p<1$ and $x$ is blown up. If
$x':=x_1$ is in case 3, we have:
$$
F_{p,Z}\in k(x)[U_2,U_3], \ d_1=0 \ \mathrm{and} \  d_2+d_3 +{\omega (x) \over p}\in \N
$$
by Theorem \ref{bupthm}(1). Let $(u'_1:=u_1/u_2,u_2,v';Z')$ be well adapted coordinates at $x'$, so we have
$$
E'=\mathrm{div}(u'_1u_2), \epsilon (x')=1+\omega (x)<p, \ d'_1=0 \ \mathrm{and} \  d'_2\in \N .
$$
Therefore $x'$ is resolved for $m(x)=p$ by blowing up codimension one centers of the form ${\cal Y}':=V(Z',u_2)$.
\end{rem}

\smallskip

\noindent {\it Algorithm:} if $x$ is in case 2 and ${\cal Y}_j:=V(Z,u_1,u_j)$ is permissible for some of $j=2,3$,
blow up along ${\cal Y}_j$,  {in case where both  ${\cal Y}_2$ and ${\cal Y}_3$ are permissible, take $j$ such that div$(u_j)$ is ``younger'' that div$(u_{j'})$ $\{j,j'\}=\{2,3\}$, i.e. let $i_0$ the index such that div $(u_{j'})$ is the strict transform of the exceptional divisor of $({\cal X}_{i_0},x_{i_0}) \leftarrow ({\cal X}_{i_0+1},x_{i_0+1})$, the projection of div$(u_j)$ on ${\cal X}_{i_0}$ is $x_{i_0}$ or a curve}; otherwise blow up along $x$.

\smallskip

We claim that this algorithm succeeds, i.e. produces $x_r$ in case 1, {\it cf.} (a),
or $x_r$ resolved for $m(x)=p$. The proof is different for small values of $\omega (x)$:

\smallskip

\noindent (b) {\it proof when} $d_1 +\omega (x)/p<1$. Let $x$ be in case 2. We may
assume that (\ref{eq6062}) holds.

\smallskip

\noindent (b1)  if $d_1+d_j+\omega(x)/p<1$, $j=2,3$, the algorithm blows up along $x$.
By the above Remark \ref{remmaxcont}, it can be assumed that $x_1=(Z/u_2,u_1/u_2,u_2, u_3/u_2)$ up to renumbering
$u_2,u_3$. We obtain
$$
d'_1=d_1, \ d'_2=d_1+d_2+d_3 +\omega(x)/p-1< d_2, \ d'_3=d_3.
$$
Assumption (b1) is stable by blowing up and can possibly repeat only finitely many times.

\smallskip

\noindent (b2) by the above Remark \ref{remmaxcont}, the algorithm succeeds
or produces an infinite sequence of points in case 2.
By (\ref{eq6061}), any subsequence of blowing ups along curves is finite,  {so, for  every  $r>> 0$, the blowing ups are centered at closed points, by the argument of the proof of Proposition~\ref{contactmaxeclatpoint}~(d-1),  $C(x_r)=0$ for every  $r>> 0$}. Take $r=0$ to begin with and assume w.l.o.g. that $x$ is
blown up. The extra assumption (\ref{eqcontactmax11}) holds if $0\leq A_2(x),A_3(x)<1$.
Up to renumbering $u_2,u_3$, we may furthermore assume by (\ref{eq6062}) that
\begin{equation}\label{eq6064}
(d_1+d_2+{\omega(x)\over p}<1, A_2(x)\geq 1), \ (d_1+d_3+{\omega(x)\over p}\geq 1, A_3(x)< 1).
\end{equation}
Let
\begin{equation}\label{eq6063}
x'_1:=(Z/u_2,u_1/u_2,u_2, u_3/u_2) \ \mathrm{and} \ x''_1:=(Z/u_3,u_1/u_3,u_2/u_3, u_3).
\end{equation}
If $x_1=x'_1$  (resp. $x_1=x''_1$), we have $d'_1=d_1$ and
$$
d'_3=d_3, \ A_3(x_1)=A_3(x), \ A_2(x_1)=A_2(x)+A_3(x)-1<A_2(x)
$$
$$
(\mathrm{resp.}  \ d'_2=d_2, \ A_2(x_1)=A_2(x), \ A_3(x_1)<A_2(x), \ d'_3<d_3).
$$
When $x_1=x''_1$ and the algorithm blows up  along a curve ($A_3(x_1)\geq 1$), note that
$$
d'_1+d'_3+\omega (x)/p-1< d'_3
$$
since $d_1+\omega (x)/p<1$. This proves that any
further blowing up at a closed point either satisfies: some of (b1) or (\ref{eqcontactmax11}),
or satisfies again (\ref{eq6064}) with a smaller value of $(A_2(x),d_3)$ for the lexicographical ordering.
Induction on $(A_2(x),d_3)$ completes the proof for $x$ in case 2
({\it vid.} the same argument in \cite{CoP2} {\bf 1.7.4} on p. 1968).

\smallskip

Let now $x$ be in case 3. We are done unless $x_1$ is again in case 3. Then,  {as $\gamma(x)=1$ ($\Rightarrow  C(x)\leq \beta(x)<1$):}
$$
A_2(x_1)=A_2(x)+ C(x)-1<A_2(x).
$$
Therefore the algorithm reaches (\ref{eqcontactmax11}) after finitely many steps. This completes
the proof of (b).

\smallskip

\noindent (c) {\it proof when} $d_1 +\omega (x)/p\geq 1$. By the above Remark \ref{remmaxcont},
we may assume that $0\leq A_2(x),A_3(x)<1$ to begin with if $x$ is in case 2. If $x$ is in case 2 (resp. in case 3), we let
$$
c'(x):=\beta (x) \  (\mathrm{resp.} \  c'(x):=A_2(x)).
$$
We have $c'(x)\geq 1$ if the extra assumption (\ref{eqcontactmax11}) does not hold.
Applying Propositions \ref{eclatpointcas12}, \ref{eclatpointcas3} and \ref{eclatpointcas3infiny},
we obtain:

\smallskip

\noindent $\bullet$ if $x$ is in case 2 and $x_1=x'_1$ (resp. $x_1=x''_1$), notations of (\ref{eq6063}), then
$$
A_3(x'_1)=A_3(x), \ c'(x'_1)\leq A_2(x)+\beta (x)-1<c'(x)
$$
$$
(\mathrm{resp.} \ A_2(x''_1)=A_2(x),  \ c'(x''_1)=A_2(x)+\beta (x)-1<c'(x)).
$$
Note that blowing up along the curve
$$
{\cal Y}':=V(Z/u_2,u_1/u_2,u_2) \ (\mathrm{resp.} \ {\cal Y}'':=V(Z/u_3,u_1/u_3,u_3))
$$
if $A_2(x'_1)\geq 1$ (resp. if $A_3(x''_1)\geq 1$) does not change $c'(x'_1)$
(resp. does not increase again $c'(x''_1)$).  If $x$ is in case 2 and $x_1$ is in case 3, then
$$
c'(x_1)=A_2(x)+A_3(x)+C(x)-1 \leq A_2(x)+\beta (x)-1 <c'(x).
$$

\noindent $\bullet$ if $x$ is in case 3 and $x_1$ is in case 2 (resp. in case 3), then,  {as $C(x)\leq \beta(x)<1$,}
$$
c'(x_1)= A_2(x)+\beta (x)-1<c'(x) \  (\mathrm{resp.}
\ c'(x_1)=A_2(x)+C(x)-1<c'(x)).
$$

Induction on $c'(x)$ completes the proof.
\end{proof}

\section{Projection Theorem: very transverse case, resolution of $\kappa (x)=2$.}

In this chapter, we prove Theorem \ref{projthm} when $\kappa (x)=2$ (Definition \ref{defkappa}).
This is restated as Theorem \ref{proofkkappa2} at the end of this chapter.

\smallskip

Assume that a valuation $\mu$ of $L=k({\cal X})$ centered at $x$ is given. We consider finite sequences of
local blowing ups along $\mu$:
\begin{equation}\label{eq701}
    ({\cal X},x)=:({\cal X}_0,x_0) \leftarrow ({\cal X}_1,x_1)\leftarrow \cdots \leftarrow ({\cal X}_r,x_r)
\end{equation}
with Hironaka-permissible centers ${\cal Y}_i \subset ({\cal X}_i,x_i)$, where $x_i$,
$0 \leq i \leq r$, denotes the center of $\mu$, see (\ref{eq402}) and following comments. Also
recall the definition of ``resolved" and ``good" (Definition \ref{defgood}) and
Remark \ref{quadsequence} about the logical scheme of the proof of Theorem \ref{projthm}.\\

 {Unfortunetaly, in (\ref{eq701}), it may happen that, for some $x_i$,
$1 \leq i \leq r$, we have $\kappa(x_i)>2$. In the next subsection \ref{sortiekappa=2},  we study points $x_{i}$  such
that $(m(x_i),\omega (x_i))=(m(x_{i-1}),\omega (x_{i-1}))$ and $\kappa (x_{i})>\kappa (x_{i-1})=2$. In fact such $x_i$ will fit the hypotheses of one of the technical lemmas. As a consequence of subsection \ref{sortiekappa=2},  such $x_i$ is always resolved for $(p,\omega (x),2)$. Let us note that the hypotheses of the lemmas are not so restrictive and some other useful cases of points $x$ resolved for $(p,\omega (x),2)$ occur.}

 {In subsection \ref{reductionmonique}, we reduce the problem to a special case (*) (Definition~\ref{*kappadeux}). By Proposition~\ref{redto*}(ii), apart the cases of subsection~\ref{sortiekappa=2}, this case $\kappa(x)=2$ and (*) is stable when  ${\cal Y}_i = x_i$.}

 {The end of the chapter is completely devoted to the resolution of the case $\kappa(x)=2$ and (*).}\\

\noindent {\it Up to the end of this chapter, ``resolved" stands for ``resolved for $(p,\omega (x),2)$''}.\\

 {\subsection{Preliminaries.}\label{sortiekappa=2}}

In this section, we study points $x'$ obtained by performing a permissible blowing up and such
that $(m(x'),\omega (x'))=(m(x),\omega (x))$ and $\kappa (x')>\kappa (x)=2$.

\begin{lem}\label{sortiemonome}
Let $(u_1,u_2 ,u_3;Z)$ be well adapted coordinates at $x$. Assume that $\epsilon (x)=\omega (x)\geq 2$,
$\kappa (x)\geq 2$ and $\mathrm{div}(u_1)\subseteq E$.

Assume furthermore $(d_1,d_2+1/p,d_3+\omega(x)/p)$
is the only vertex $\mathbf{v}=(v_1,v_2,v_3)$ of $\Delta_S(h;u_1,u_2,u_3;Z)$ in the region $v_1=d_1$.

\smallskip

Then $x$ is resolved.
\end{lem}

\begin{proof}
Since $\kappa (x)\geq 2$, there is an expansion
$$
\mathrm{in}_{m_S}h=Z^p + F_{p,Z}, \ H^{-1}F_{p,Z}\subseteq k(x)[U_1, U_2,U_3]_{\omega (x)}.
$$
Any vertex of $\Delta_S(h;u_1,u_2,u_3;Z)\cap \{\mathbf{x} : x_1+x_2+x_3=\delta (x)\}$ lies in the region
$v_1>d_1$ by assumption and we deduce that $U_1 \in \mathrm{Vdir}(x)$. Let
$$
\mathrm{in}_\mathbf{v}h=Z^p +\sum_{i=1}^pF_{i,\mathbf{v}} Z^{p-i}\in k(x)[U_1,U_2,U_3][Z]
$$
be the initial form polynomial with respect to $\mathbf{v}$. By Theorem \ref{initform}, we have
$F_{i,\mathbf{v}}=0$, $1 \leq i \leq p-2$, and $F_{p-1,\mathbf{v}}=-G_\mathbf{v}^{p-1}$ since $\epsilon (x)>0$.
Moreover $G_\mathbf{v}^{p-1}\neq 0$ implies that
\begin{equation}\label{eq702}
    \mathbf{v}\in \N^3, \ E=\mathrm{div}(u_1u_2u_3) \ \mathrm{and}
    \ (\mathrm{Disc}_Z(h))=(u_1^{pd_1}u_2^{pd_2+1}u_3^{pd_3+\epsilon (x)})^{p-1}.
\end{equation}

Let ${\cal Y}:=V(Z,u_1,u_3)\subset {\cal X}$ and $y\in {\cal X}$ be its generic point. If ${\cal Y}$ is permissible
of the first kind, i.e. $m(y)=p$ and $\epsilon (y)=\epsilon (x)$,
we take ${\cal Y}_0:={\cal Y}$ in (\ref{eq701}). By Theorem \ref{bupthm}, we have
$\iota (x_1)\leq (p,\omega (x),1)$ unless
$$
\mathrm{Vdir}(x)=<U_1> \ \mathrm{and} \ x_1=x':=(Z':=Z/u_3,u'_1:=u_1/u_3,u_2,u_3).
$$
By Proposition \ref{originchart}, $\Delta_{S'}(h';u'_1,u_2,u_3;Z')$ is minimal, and we deduce that
$H(x')=({u'_1}^{pd_1}u_2^{pd_2+1}u_3^{p(d_1+d_3-1)+\epsilon (x)})$ and $\mathbf{v}':=(d_1,d_2+1/p, d_1+d_3-1+\epsilon (x)/p)$
is a vertex of $\Delta_{S'}(h';u'_1,u_2,u_3;Z')$. Therefore $\omega (x')\leq \epsilon (x')=1$ and the lemma holds. \\

Assume now that ${\cal Y}$ is not permissible of the first kind. We take ${\cal Y}_0:=\{x\}$ in (\ref{eq701}).
If $\iota (x_1)\geq (p,\omega (x), 2)$, $x_1$ belongs to the strict transform of $\mathrm{div}(u_1)$
by Theorem \ref{bupthm}.

If $x_1=x':=(Z':=Z/u_2,u'_1:=u_1/u_2,u_2,u'_3:=u_3/u_2)$ is the point on the strict transform of ${\cal Y}$,
then $\Delta_{S'}(h';u'_1,u_2,u'_3;Z')$ is minimal by Proposition \ref{originchart} and we deduce as above that
$H(x')=({u'_1}^{pd_1}u_2^{p(d_1+d_2 +d_3-1)+\epsilon (x)}{u'_3}^{pd_3})$ and
$\mathbf{v}':=(d_1,d_1+d_2 +d_3-1+(1+\epsilon (x))/p, d_3+\epsilon (x)/p)$ is the only vertex of $\Delta_{S'}(h';u'_1,u_2,u_3;Z')$
in the region $v'_1=d_1$. Since $\epsilon (x')\leq \epsilon (x)=\omega (x)$, we deduce that
$x_1$ satisfies again the assumptions of the lemma if $\iota (x_1)\geq (p,\omega (x),2)$.

\smallskip

The conclusion of Proposition \ref{permisarc}(2.b) is not satisfied by the formal arc $\hat{{\cal Y}}\rightarrow {\cal X}$.
Iterating, we deduce from Proposition \ref{permisarc}(1) that one of the following three properties is
satisfied for some $r \geq 1$:
\begin{itemize}
  \item [(1)] $\iota (x_r)\leq (p,\omega (x),1)$;
  \item [(2)] $x_r$ belongs to the strict transform ${\cal Y}_r$ of ${\cal Y}$
in ${\cal X}_r$ and ${\cal Y}_r$ is permissible of the first kind at $x_r$, or
  \item [(3)] $x_r$ does not belong to ${\cal Y}_r$.
\end{itemize}

The lemma holds when (1) is satisfied; it has been proved above that the lemma also holds when (2)
is satisfied. If (3) is satisfied, it can be assumed w.l.o.g. that $r=1$. We claim that $x_1$
satisfies the conclusion of the lemma if $x_1\neq (Z/u_3,u_1/u_3,u_2/u_3,u_3)$.

\smallskip

To prove the claim, first note that there exists a unitary polynomial $P(t)\in S[t]$, whose reduction $\overline{P}(t)\in k(x)[t]$
is irreducible, $\overline{P}(t) \neq t$, and $x_1=(X':=Z/u_3,u'_1:=u_1/u_3, u'_2:=u_3, u'_3:=P(u_2/u_3))$. We then
denote $S':=S[u_1/u_3,u_2/u_3]_{(u'_1,u'_2,u'_3)}$ and
$$
h'=X'^p +\sum_{i=1}^pf_{i,X'}X'^{p-i} \in S'[X'],  \ E'=\mathrm{div}(u'_1u'_2).
$$
We have $H(x')=({u'_1}^{pd_1}{u'_2}^{p(d_1+d_2 + d_3-1)+\epsilon (x)})$ by Proposition \ref{bupformula}(iv)
and
$$
\mathbf{v}':=(d_1, d'_2:=d_1+d_2 +d_3-1+(1 + \epsilon (x))/p, 0)
$$
is a vertex of $\Delta_{S'}(h';u'_1,u'_2,u'_3;X')$.

\smallskip

If $\mathbf{v}'$ is not solvable (in particular if $G_\mathbf{v}\neq 0$, see (\ref{eq702}) above),
we deduce that $\omega (x')\leq \epsilon (x')=1$ and the lemma holds.

If $\mathbf{v}'$ is solvable, we had
$$
\mathrm{in}_{\mathbf{v}}h= Z^p +\lambda U_1^{pd_1}U_2^{pd_2+1}U_3^{pd_3+\epsilon (x)},
\ \lambda \in k(x), \ \lambda \neq 0
$$
to begin with. Let $\sigma ' \subset \Delta_{S'}(h';u'_1,u'_2,u'_3;X')$ be the (noncompact) face with equations
$v'_1=d_1$, $v'_2=d'_2$. The initial form polynomial corresponding to $\sigma '$ is
$$
\mathrm{in}_{\sigma '}h'={X'}^p +\lambda \overline{\left ({u_2 \over u_3}\right )^{pd_2+1}}
{U'_1}^{pd_1}{U'_2}^{pd'_2},
$$
where $\overline{\theta}$ denotes the image in $S'/(u'_1,u'_2)$ of $\theta \in S'$.
Let $\mu \in k(x')$ be the residue of $u_2/u_3$. Since $\mathbf{v}'$ is solvable, we have:
\begin{equation}\label{eq703}
    (d_1, d'_2)\in \N^2, \ \lambda \mu^{pd_2+1}  \in k(x')^p.
\end{equation}
Take  $Z':=X'-\varphi '$, $\varphi '\in S'$ such that $\Delta_{S'} (h';u'_1,u'_2,u'_3;Z')$ is minimal.
We have $\varphi '=\gamma '{u'_1}^{d_1}{u'_2}^{d'_2}$, where $\gamma ' \in S'$ is a preimage of
$(\lambda \mu^{pd_2+1})^{1/p}\in k(x')$. By (\ref{eq703}),
$\lambda U_2^{pd_2+1}$ is not a $p^\mathrm{th}$-power, since $\mathbf{v}$ was not a solvable vertex.
We deduce that
$$
\lambda \overline{\left ({u_2 \over u_3}\right )^{pd_2+1}}+\overline{\gamma '}^p \in S'/(u'_1,u'_2)
$$
is a regular parameter. Therefore $\mathbf{v}'_1:=(d_1,d'_2 +1/p,1/p)$ is a vertex of
$\Delta_{S'} (h';u'_1,u'_2,u'_3;Z')$ and this proves that $\omega (x')\leq 1$. This concludes the proof
of the claim.\\

To conclude, take $x_1= x':=(Z':=Z/u_3,u'_1:=u_1/u_3,u'_2:=u_2/u_3,u'_3:=u_3)$.
Since $\Delta_{S'}(h';u'_1,u'_2,u'_3;Z')$ is minimal by Proposition \ref{originchart}, we deduce as before that
$H(x')=({u'_1}^{pd_1}{u'_2}^{pd_2}{u'_3}^{pd'_3})$ and
$\mathbf{v}':=(d_1,d_2+1/p,d'_3+1/p)$ is the only vertex of $\Delta_{S'}(h';u'_1,u'_2,u'_3;Z')$
in the region $v'_1=d_1$, where $d'_3:=d_1+d_2 +d_3-1+\epsilon (x)/p$. Therefore $\epsilon (x')\leq 2$
and we are done unless $\omega (x)=2$,
$\iota (x')\geq (p,2,2)$ and $E=\mathrm{div}(u_1u_2u_3)$, which we assume from now on.

\smallskip

We have $E'=\mathrm{div}(u'_1u'_2u'_3)$ and the initial form polynomial has an expansion
$$
\mathrm{in}_{m_{S'}}h'={Z'}^p +F_{p,Z'}.
$$
with ${H'}^{-1}F_{p,Z'}=U'_1 (\lambda_1U'_1+\lambda_2U'_2+\lambda_3U'_3)+\mu U'_2U'_3$. The assumptions imply that
$\mu H' U'_2U'_3 \not \in G(m_{S'})^p$ and $(\lambda_1,\lambda_2)\neq (0,0)$. Moreover, we have
$$
\lambda_j \neq 0 \Longrightarrow \lambda_jH'U'_1U'_j \not \in G(m_{S'})^p, \ 1 \leq j \leq 3.
$$
If $\tau (x_1)=3$, we take ${\cal Y}_1:=\{x_1\}$ in (\ref{eq701}) and obtain $\iota (x_2)\leq (p,2,1)$.
We conclude by analyzing the cases $\tau (x_1)\leq 2$. By \cite{CoP2} {\bf II.1.5} p.1888, this implies
that $\lambda_1=0$. Therefore $\lambda_2\neq 0$ and we get
$$
\mathrm{Vdir}(x')=<U'_2,U'_3+\mu_2U'_1>, \ \mu_2  \neq 0.
$$
By Lemma \ref{joyeux} below with $(i, \omega)=(1,2)$, we have $p\geq 3$ and
\begin{equation}\label{eq705}
    d_2+1/p \in \N, \ d_1, d'_3 \not \in \N \ \mathrm{and} \ \widehat{pd_1}+\widehat{pd'_3}+1=p.
\end{equation}

Let ${\cal Y}_1:=(Z',u'_1,u'_3)$, $y_1\in {\cal X}_1$ be its generic point
and $W_1:=(u'_1,u'_3)$. For $i$, $1 \leq i \leq p-1$, consider a finite monomial expansion (\ref{eq2036}):
$$
f_{i,Z}=\sum_{\mathbf{a}\in \mathbf{S}(f_{i,Z})}\gamma(\mathbf{a}) u_1^{ia_1}u_2^{ia_2}u_3^{ia_3}\in S,
 \   \mathbf{S}(f_{i,Z}) \subset \Delta_S (h;u_1,u_2,u_3;Z).
$$
The polyhedron assumption on $h$ gives
$$
\mathbf{a}\in \mathbf{S}(f_{i,Z}) \Longrightarrow (a_1 \geq d_1 , \  a_2+a_3 \geq d_2+d_3 +{3\over p})
$$
and that at least one of these inequalities is strict. Now
$f_{i,Z'}=u_3^{-i}f_{i,Z}$ and one deduces that
\begin{equation}\label{eq706}
    {\mathrm{ord}_{W_1}f_{i,Z'} \over i} =\min_{\mathbf{a}\in \mathbf{S}(f_{i,Z})}\{2a_1 +a_2+a_3 -1\}> d_1 +d'_3.
\end{equation}
By (\ref{eq705}), we have $i (d_1 +d'_3 +1)\in \N$, so (\ref{eq706}) actually implies that
$$
{\mathrm{ord}_{W_1}f_{i,Z'} \over i}\geq d_1 +d'_3 +1,
$$
since $1 \leq i \leq p-1$ and $\mathrm{ord}_{W_1}f_{i,Z'} \in \N$.
On the other hand, $\mathrm{ord}_{W_1}f_{p,Z'}=p(d_1+d'_3)+1 \geq p$
and we deduce that ${\cal Y}_1$ is Hironaka-permissible w.r.t. $E'$ with $\epsilon (y_1)=1$.
Arguing as before, on gets $\omega (x_2)\leq \epsilon (x_2)\leq 1$ (resp. $\iota (x_2)\leq (p,2,1)$)
if the residue $\mu ' \in k(x_2)$ of $u'_3/u'_1$  does not satisfy (resp. satisfies) $\mu '+\mu_2=0$.
This concludes the proof.
\end{proof}

The following lemma extends the previous result when $\omega (x)=1$.

\begin{lem}\label{sortieomegaun}
Lemma \ref{sortiemonome} remains valid when $\epsilon (x)=\omega (x)=1$
and $\mathrm{div}(u_1)\subseteq E \subseteq \mathrm{div}(u_1u_2)$,
all other assumptions being otherwise unchanged.
\end{lem}

\begin{proof}
Let ${\cal Y}:=V(Z,u_1,u_2)\subset {\cal X}$ and $y$ be its generic point. Arguing
as in (\ref{eq702}) above, any vertex of $\Delta_S (h;u_1,u_2,u_3;Z)$ is induced by $f_{p,Z}$. By
Proposition \ref{Deltaalg}, we have $\delta (y)=d_1+d_2+{1 \over p}=\delta (x)$, since $H^{-1}F_{p,Z}=<U_1>$.
Then Proposition \ref{deltainv}(ii) implies that
$$
(m(y'),\epsilon (y'))=(m(x'),\epsilon (x'))=(p,1).
$$
Therefore ${\cal Y}$ is permissible of the first kind and we take ${\cal Y}_0:={\cal Y}$ in (\ref{eq701}).
By Theorem \ref{bupthm}, we have $\iota (x_1)\leq (p,\omega (x),1)$ unless
$$
x_1=x':=(Z':=Z/u_2,u'_1:=u_1/u_2,u_2,u_3), \ E':=\mathrm{div}(u'_1u_2).
$$
By Proposition \ref{originchart}, $\Delta_{S'}(h';u'_1,u_2,u_3;Z')$ is minimal. We deduce that
$H(x')=({u'_1}^{pd_1}u_2^{p(d_1+d_2-1)+1})$ and $\mathbf{v}':=(d_1,d_1+d_2-1+1/p, 1/p)$
is a vertex of $\Delta_{S'}(h';u'_1,u_2,u_3;Z')$. Therefore $(m(x'),\omega (x'))\leq (p,0)$ and the lemma holds.
\end{proof}

Given an integer $\alpha \in \N$, we denote by $\widehat{\alpha}\in \{0, \ldots , p-1\}$
the remainder of  $\alpha$ modulo $p$. The following elementary lemma is useful.

\bigskip
\begin{lem}\label{joyeux}
Let $(i,\omega )\in \N^2$ satisfy $0<i<\omega $ and $F_0\in k(x)[U_1,U_2]_i$, $F_0\neq 0$.
Take  {$\mathrm{div}(u_1u_2)\subset E \subset \mathrm{div}(u_1u_2u_3)$} and let
$$
(a(1),a(2),a(3))\in \N^3, \ H:=U_1^{a(1)}U_2^{a(2)}U_3^{a(3)}\in G(m_S)=k(x)[U_1,U_2,U_3],
$$
 {with the usual convention: $a(3):=0$ when  $\mathrm{div}(u_1u_2)= E$.}
We define $F:=HU_3^{\omega -i}F_0$; assume that $F\not \in G(m_S)^p$ and that
$$
<U_3,U_j> \nsubseteq \mathrm{Vdir}(J(F,E,m_S)) \  \mathrm{for} \ j=1 \ \mathrm{and} \ j=2.
$$
Then
\begin{equation}\label{eq707}
    \mathrm{Vdir}(J(F,E,m_S))=<U_3, U_1 +\lambda U_2>, \ \lambda \neq 0,
\end{equation}
and the following holds:
\begin{itemize}
  \item [(i)] if $i\equiv 0 \ \mathrm{mod}p$, there exists $0\neq c \in k(x)$ such that
$$
F - c H U_3^{\omega -i}(U_1 + \lambda U_2)^i \in G(m_S)^p;
$$
  \item [(ii)] if $i\not \equiv 0 \ \mathrm{mod}p$, let $a_j:=\widehat{a(j)}$, $1 \leq j \leq 3$,
  and $a:=\widehat{i}\neq 0$. Then:
\begin{equation}\label{eq7073}
a_3+\omega -a \equiv 0 \ \mathrm{mod}p, \ a_1a_2 \neq 0   \ \mathrm{and} \ a_1+a_2+a=p.
\end{equation}
In particular $p\geq 3$. There exists $0\neq c\in k(x)^p$ such that
$$
F - c U_3^{a(3) +\omega -i}\Phi_i(U_1,\lambda U_2) \in G(m_S)^p,
$$
where
\begin{equation}\label{eq7071}
\Phi_i(U_1,U_2):=(-1)^{a_2}U_1^{a(1)}U_2^{a(2)}\sum_{k=0}^a
\left(
  \begin{array}{c}
   a_2+k-1 \\
    k \\
  \end{array}
\right)
U_1^{a-k}(U_1+U_2)^{i-a +k}.
\end{equation}
\end{itemize}
\end{lem}

\begin{proof}
\cite{CoP2} {\bf II.5} p.1896 for (i) and (\ref{eq7073}).
 {It} remains to prove that there exists $0\neq c\in k(x)^p$ such that
$$
H_0F_0 - c \Phi_i(U_1,\lambda U_2) \in (k(x)[U_1,U_2])^p , \ H_0:=U_1^{a(1)}U_2^{a(2)}.
$$
It is easily checked that (\ref{eq707}) holds when
$F=U_3^{a(3)+\omega (x)-i}\Phi_i(U_1,\lambda U_2)$.  Note that
\begin{equation}\label{eq7072}
H_0^{-1}\Phi_i(U_1,\lambda U_2)=(-1)^{a_2}\lambda^{a(2)}\left(
  \begin{array}{c}
   a_2+a \\
    a \\
  \end{array}
\right)
U_1^i + \cdots .
\end{equation}

Let $(\lambda_l)_{l \in \Lambda}$ be an absolute $p$-basis of $k(x)$ and let
$$
D_l:={\partial \hfill{} \over \partial \lambda_l} \  D_j:=U_j{\partial \hfill{} \over \partial U_j}, \ j=1,2.
$$
We expand
\begin{equation}\label{eq708}
    F_0:=\alpha U_1^i + \alpha_1 U_1^{i-1}U_2 + \cdots , \ \alpha , \alpha_1 \in k(x).
\end{equation}
Since $H_0^{-1}D_l \cdot (H_0F_0) \in <(U_1 +\lambda U_2)^i>$ by (\ref{eq707}), $l \in \Lambda \cup \{1,2\}$,
it is easily seen that $\alpha \neq 0$.

\medskip

Suppose that  $\alpha \in k(x)^p$. Since $H_0^{-1}D_l\cdot (H_0F_0 )\in < (U_1 +\lambda U_2)^i>$, $l \in \Lambda$,
and this polynomial is divisible by $U_2$, we have $D_l \cdot H_0F_0=0$ for $l \in \Lambda$ by (\ref{eq707}).
We deduce that  $H_0F_0 \in k(x)^p[U_1,U_2]$ and in particular that $\lambda  \in k(x)^p$. Let
$$
F':=  H_0 F_0 -c \Phi_i(U_1,\lambda U_2),
\ c:=\alpha (-1)^{a_2}\lambda^{-a(2)}\left(
  \begin{array}{c}
   a_2+a \\
    a \\
  \end{array}
\right)^{-1} \in k(x)^p.
$$
By construction, we have
$H^{-1}D_l\cdot F'=0$, $l \in \Lambda  \cup \{1,2\}$, and (ii) is proved.

\medskip

Suppose that $\alpha \not \in k(x)^p$. Without loss of generality, it can be assumed that $\alpha =\lambda_l$
for some $l\in \Lambda$. For $l'\neq l$, $U_2$ divides  $H_0^{-1}D_{l'}\cdot (H_0 F_0)$, so
$D_{l'}\cdot (H_0 F_0)=0$ by (\ref{eq707}).
This proves that $F_0 \in k(x)^p(\alpha)[U_1,U_2]$. We have
$$
\left\{
  \begin{array}{ccc}
   H_0^{-1}D_l\cdot (H_0 F_0) & = &  U_1^i + (D_l \cdot \alpha_1) U_1^{i-1}U_2 + \cdots  \hfill{} \\
    &  &  \\
   H_0^{-1}D_1\cdot (H_0 F_0) & = & (a_1 +a) \alpha U_1^i + (a_1 +a -1)\alpha_1 U_1^{i-1}U_2 + \cdots \\
  \end{array}
\right.
$$
from which we deduce the identity
\begin{equation}\label{eq7081}
\left\{
  \begin{array}{ccc}
   a \lambda  \hfill{}& = &  D_l \cdot \alpha_1  \hfill{} \\
    &  &  \\
   a (a_1+a) \alpha \lambda & = & (a_1 +a-1) \alpha_1  \\
  \end{array}
\right.
.
\end{equation}
Therefore $(a_1 +a-1)\alpha_1 =(a_1+a) (D_l \cdot \alpha_1)$. Expanding $\alpha_1 =:\sum_{j=0}^{p-1}c_j^p\alpha^j$,
we then deduce that
\begin{equation}\label{eq7082}
\alpha_1=c_j^p\alpha^j, \ \mathrm{where} \ (a_1+a)j \equiv a_1+a-1 \ \mathrm{mod}p.
\end{equation}
Since $a_1+a +a_2=p$ in this case (ii), we get  $a_2(j-1)\equiv 1\ \mathrm{mod}p$ from (\ref{eq7082}).
One deduces  from (\ref{eq7081})-(\ref{eq7082}) that $\alpha = d \lambda^{a(2)}$
for some $d \in k(x)^p$, $d \neq 0$.
The proof now concludes as in the above case $\alpha \in k(x)^p$.
\end{proof}

\bigskip
\begin{lem}\label{vraijoyeux}

 {Let  $F_0\in k(x)[U_1,U_2,U_3]_{\omega}$, $\omega\in \N \setminus\{0 \}$, $F_0\neq 0$.
Take $ E := \mathrm{div}(u_1u_2u_3)$ and let
$$
(a(1),a(2),a(3))\in \N^3, \ H:=U_1^{a(1)}U_2^{a(2)}U_3^{a(3)}\in G(m_S)=k(x)[U_1,U_2,U_3].
$$
 Let $F:=HF_0$. Assume that:}

 {\begin{equation}\label{eq707}
    \mathrm{Vdir}(J(F,E,m_S))=< U_1 +\lambda_2 U_2+ \lambda_3 U_3>, \ \lambda_2 \lambda_3 \in k(x)^*.
\end{equation}
Then $\omega=0\ \mathrm{mod}\ p$ and
 there exists $0\neq c \in k(x)$ such that
$$
F - c H (U_1 + \lambda_2 U_2+ \lambda_3 U_3)^{\omega}\in G(m_S)^p.
$$}

\end{lem}

\begin{proof} { Instead of quoting \cite{Co5} Proposition~E.5.1 page 33, we give a short argument.
Let
$$F_0:=\sum_{0\leq a,b,c \leq \omega}\lambda_{a,b,c}U_1^a U_2^b U_3^{c}, \  \lambda_{a,b,c}\in k(x).$$
It is clear that $\lambda_{\omega,0,0}\lambda_{0,\omega,0}\lambda_{0,0,\omega}\in k(x)^*$.} 
 {As
$U_1{\partial F\over \partial U_1},U_2{\partial F\over \partial U_2},U_3{\partial F\over \partial U_3}$ are proportional, so are their coefficients in $U_1^{\omega},U_2^{\omega},U_3^{\omega}$, so the following matrix has rank~$1$:}

 {$$\begin{pmatrix}
   a(1)+ \omega & a(1) & a(1) \\
a(2) & a(2)+ \omega & a(2)	 \\
a(3) & a(3) & \omega+a(3)
\end{pmatrix}.
$$}
 {This leads to $\omega=0\ \mathrm{mod}\ p$.
 Let $G:=F_0-\lambda_{\omega,0,0}(U_1 +\lambda_2 U_2+ \lambda_3 U_3)^{\omega}$, we have
\begin{equation}\label{eq707}
    \mathrm{Vdir}(J(HG,E,m_S))=< U_1 +\lambda_2 U_2+ \lambda_3 U_3>\ \mathrm{or}\ HG\in G(m_S)^p.
\end{equation}
As deg$_{U_1}(G)<\omega$, the first is impossible, the second is true. This gives the result.}

\end{proof}

\begin{lem}\label{lemsortiekappaegaldeux}
Assume that $E=\mathrm{div}(u_1)$.
If $(u_1,u_2 ,u_3;Z)$ are well adapted coordinates at $x$, then
$$
\mathrm{in}_Eh=Z^p +U_1^{pd_1}\overline{F}\in S/(u_1)[U_1][Z], \ \overline{F}\neq 0.
$$
\end{lem}

\begin{proof}
This is obvious if $\mathrm{char}S=p>0$ and $h$ is purely inseparable
(case (c) of assumption {\bf (G)}). Otherwise, {\bf (E)} implies that $\mathrm{Disc}_Z(h)=\gamma u_1^D$
for some $D\geq p(p-1)d_1$ and $\gamma \in S$ a unit. Let
$$
\mathrm{in}_Eh=Z^p +\sum_{i=1}^p U_1^{id_1}F_iZ^{p-i}, \ F_i \in S/(u_1)[U_1]_{id_1},
$$
where $F_i=0$ if $id_1 \not \in \N$. Since $\mathrm{char}S/(u_1)=p>0$, condition {\bf (G)} implies that
$\mathrm{in}_E h$ has $p$ distinct roots over an algebraic closure of $k(E)$ if $F_i\neq 0$ for some $i\neq p$.
But then $D=p(p-1)d_1$: a contradiction since $\epsilon (x)>0$. We deduce that $F_i=0$, $1 \leq i \leq p-1$.
We have $F_p\neq 0$ by Proposition \ref{Deltaalg}.
\end{proof}

\medskip

\begin{prop}\label{sortiekappaegaldeux}
Assume that $\epsilon (x)=\omega (x)$, $\kappa (x)\geq 2$ and $E=\mathrm{div}(u_1)$.
Let $(u_1,u_2 ,u_3;Z)$ be well adapted coordinates at $x$. Assume furthermore that
$S/(u_1)\simeq k(x)[\overline{u}_2,\overline{u}_3]_{(\overline{u}_2,\overline{u}_3)}$
and the following two conditions are satisfied:
\begin{itemize}
  \item [(i)] the initial form polynomial $\mathrm{in}_Eh$ of Lemma \ref{lemsortiekappaegaldeux} is of the form
$$
\mathrm{in}_E h=Z^p +U_1^{pd_1}\overline{F}, \ \overline{F}\in  k(x)[\overline{u}_2,\overline{u}_3]_{1+\omega (x)};
$$
  \item [(ii)] we have
$$
\overline{\mathrm{Vdir}(x)} +\mathrm{Vdir}\left ({\partial \overline{F} \over \partial \overline{u}_2},
{\partial \overline{F} \over \partial \overline{u}_3}\right ) =<\overline{U}_2, \overline{U}_3>,
$$
where $\overline{\mathrm{Vdir}(x)}$ denotes the image of $\mathrm{Vdir}(x)$ in $<\overline{U}_2,\overline{U}_3>$.
\end{itemize}

Then $x$ is resolved.
\end{prop}

\medskip

\begin{proof} The proof is the same as that of \cite{CoP2} {\bf II.3} p.1890 and we
only indicate the necessary changes. Since $\kappa (x)\geq 2$, we have
\begin{equation}\label{eq710}
    \mathrm{in}_{m_S}h=Z^p + F_{p,Z}, \ H^{-1}F_{p,Z}\subseteq k(x)[U_1,U_2,U_3]_{\omega (x)}
\end{equation}
and $U_1 \in \mathrm{Vdir}(x)$ as in the beginning of the proof of Lemma \ref{sortiemonome}.
We discuss according to the value of $\tau '(x)$.

\smallskip

\noindent $\bullet$ {\it Assume that $\tau '(x)=3$.}  The proposition follows from Theorem \ref{bupthm}.

\smallskip

\noindent $\bullet$ {\it Assume that $\tau '(x)=2$.}
Note that $\omega (x)\geq 2$. Since $E=\mathrm{div}(u_1)$ and
$U_1 \in \mathrm{Vdir}(x)$, we have
$\mathrm{Vdir}(x)=<U_1, \lambda_2 U_2 + \lambda_3U_3>$, $(\lambda_2, \lambda_3)\neq (0,0)$.
By symmetry, it can be assumed that $\lambda_2=1$. If $\lambda_3\neq 0$, we let $v_2:=u_2+ \gamma_3u_3$,
where $\gamma_3\in S$ is a preimage of $\lambda_3\in S/(u_1)
\simeq k(x)[\overline{u}_2,\overline{u}_3]_{(\overline{u}_2,\overline{u}_3)}$.

Let $(u_1,v_2,u_3;Z_1)$ be well adapted coordinates at $x$, $Z_1=Z - \phi$, $\phi \in S$. By Lemma
\ref{lemsortiekappaegaldeux}, we have $\mathrm{ord}_{u_1}\phi>d_1$. Therefore
$$
\mathrm{in}_E h=Z_1^p +U_1^{pd_1}(\overline{F} +\overline{\phi}^p),
$$
where $\overline{\phi} =0$ (resp. $\overline{\phi} =cl_0(u_1^{-d_1}\phi)\in S/(u_1)$) if $d_1\not \in \N$ (resp $d_1 \in \N$).
Note that ($1+ \omega (x)\equiv 0 \ \mathrm{mod}p$ and $\overline{\phi}\in k(x)[\overline{u}_2,\overline{u}_3]_{(1+ \omega (x))/p}$) if $\overline{\phi} \neq 0$.
Assumptions (i) and (ii) are then unchanged, so it can be assumed w.l.o.g. that $\mathrm{Vdir}(x)=<U_1, U_2>$.
Assumption (ii) now implies
$$
\overline{F}(\overline{u}_2,\overline{u}_3) \not \in <\overline{u}_2^{1+\omega (x)}>
\ (\mathrm{resp.} \   {\overline{F}(\overline{u}_2,\overline{u}_3) }\not \in <\overline{u}_2^{1+\omega (x)}, \overline{u}_3\overline{u}_2^{\omega (x)}>)
$$
if $\omega (x) \not\equiv 0 \ \mathrm{mod}p$ (resp. if $\omega (x) \equiv 0 \ \mathrm{mod}p$).

\smallskip

Let ${\cal X}' \longrightarrow ({\cal X},x)$ be the blowing along $x$ and $x' \in {\cal X}'$ be the
center of $\mu$. By Theorem \ref{bupthm}, we have $\iota (x') \leq (p,\omega (x),1)$
except  {possibly} if $x'=(Z':=Z/u_1,u'_1:=u_1/u_3,u'_2:=u_2/u_3,u_3)$, since  $\mathrm{Vdir}(x)=<U_1, U_2>$.
If $\iota (x')\geq (p,\omega (x),2)$, there are  two cases to be considered as in \cite{CoP2} end of p.1891:

\smallskip

{\it Case 1:} $\overline{F}(\overline{u}_2,\overline{u}_3)=\lambda_0 \overline{u}_2^{1+\omega (x)}+
\lambda_1 \overline{u}_3\overline{u}_2^{\omega (x)}$, $\lambda_1 \neq 0$. Then $({\cal X}',x')$ satisfies
the assumption of Lemma \ref{sortiemonome} (instead of {\it ibid.} {\bf II.1} on p.1885) whose conclusion
proves the proposition.

\smallskip

{\it Case 2:} $\overline{F}(\overline{u}_2,\overline{u}_3)=\lambda_0 \overline{u}_2^{1+\omega (x)}+
\lambda_1 \overline{u}_3\overline{u}_2^{\omega (x)}+\lambda_2\overline{u}_3^2\overline{u}_2^{\omega (x)-1}$,
$\lambda_2 \neq 0$ and $\omega (x) - 1 \equiv 0 \ \mathrm{mod}p$. Then $\tau (x')=3$ by the
characteristic free {\it ibid.} Lemma {\bf II.3.3} on p.1892. Blowing up again $x'$ then gives
$\iota (x'')\leq (p,\omega (x),1)$ by Theorem \ref{bupthm}, where $x''$ is the center of $\mu$.

\medskip

\noindent $\bullet$ {\it Assume that $\tau '(x)=1$.}
We have $\mathrm{Vdir}(x)=k(x)U_1$ and  $F_{p,Z}=\lambda U_1^{pd_1+\omega (x)}$ in
(\ref{eq710}). Assumption (ii) now reads
\begin{equation}\label{eq711}
    \mathrm{Vdir}\left ({\partial \overline{F} \over \partial \overline{u}_2},
{\partial \overline{F} \over \partial \overline{u}_3}\right ) =<\overline{U}_2, \overline{U}_3>.
\end{equation}

Let ${\cal X}' \longrightarrow ({\cal X},x)$ be the blowing along $x$ and $x' \in {\cal X}'$ be the
center of $\mu$. By Theorem \ref{bupthm}, we have $\iota (x') \leq (p,\omega (x),1)$
except  {possibly} if $\eta '(x')$ lies on the strict transform of $\mathrm{div}(u_1)$. By symmetry  {between}
$\overline{u}_2, \overline{u}_3$, it can be assumed that
$x'=(Z':=Z/u_3,u'_1:=u_1/u_3,u'_2=P(u_2/u_3), u_3)$, where $P(t)\in S[t]$ is
a unitary polynomial whose reduction $\overline{P}(t)\in k(x)[t]$ is irreducible.
We have $E'=\mathrm{div}(u'_1u_2)$. Let
$$
\tilde{P}(\overline{U}_2, \overline{U}_3):=
\overline{U}_2^{\mathrm{deg}\overline{P}}\overline{P}(\overline{U}_3/\overline{U}_2)
\in k(x)[\overline{U}_2, \overline{U}_3]_{\mathrm{deg}\overline{P}}.
$$
By (\ref{eq711}), we have
$$
\mathrm{ord}_{\tilde{P}}({\partial \overline{F} \over \partial \overline{u}_2},
{\partial \overline{F} \over \partial \overline{u}_3} ) \leq \omega (x)-1.
$$
with equality if $\iota (x')\geq (p,\omega (x),2)$.

 \smallskip
  {If $\omega (x)\geq 2$,  this implies that $k(x')=k(x)$. It can be assumed that $P(t)=t$, ${\tilde{P}}=\bar{U}_2$: $(Z':=Z/u_3,u'_1:=u_1/u_3,u'_2=u_2/u_3, u_3)$ are well prepared at $x'$, this leads to  the same cases 1 and  2 as above. One concludes applying Lemma \ref{sortiemonome} or  \cite{CoP2} Lemma {\bf II.3.3} on p.1892, exactly as above.} 

\smallskip

  {If $\omega (x)=1$, then $({\cal X}',x')$ satisfies the assumption of Lemma \ref{sortieomegaun} or
there is an expansion
\begin{equation}\label{eq712}
   \mathrm{in}_{m_{S'}}h'={Z'}^p + {U'_1}^{pd_1}{U'_2}^{p(d_1-1)+1}(\lambda'_1U'_1+\lambda'_2U'_2)
    \in G(m_{S'})[Z']
\end{equation}
with $\lambda'_1\lambda'_2\neq 0$, where $(u'_1,u'_2,u'_3;Z')$ are
well adapted coordinates at $x'$.  With notations as in Lemma \ref{joyeux}   {applied with $a_3=0$, $i=1$}, we let $a_1:=\widehat{pd_1}$, $a_2:=\widehat{p(d_1-1)+1}$.}

  {Let ${\cal Y}':=V(Z',u'_1,u'_2)\subset {\cal X}'$ with generic point $y'$.
By (\ref{eq712}), any vertex of $\Delta_{S'} (h';u'_1,u'_2,u'_3;Z')$ is induced by $f_{p,Z'}$ and
we have $\delta (y')=2d_1-1+2/p=\delta (x')$, so ${\cal Y}'$ is permissible of the first kind at $x'$.}

  {Either we have not the conditions ( \ref{eq7073}), blowing up ${\cal Y}'$ then gives $\iota (x'')\leq (p,\omega (x),1)$ by Theorem \ref{bupthm},
where $x''$ is the center of $\mu$.}
  {Or we blow up up consecutively ${\cal Y}'_1$, then ${\cal Y}'_2$,
and iterating, we reduce to the case $d_1 =a_1$, $d_2=a_2$, $1+a_1+a_2=p$, $a_1a_2>0$: $\tau(x)=3$.}

\end{proof}

\begin{prop}\label{sortiebis}
Assume that $\epsilon (x)=\omega (x)\geq 2$, $\kappa (x)\geq 2$ and $E=\mathrm{div}(u_1)$.
Let $(u_1,u_2 ,u_3;Z)$ be well adapted coordinates at $x$. Assume furthermore that
the initial form polynomial
$\mathrm{in}_E h=Z^p +U_1^{pd_1}\overline{F}$, $\overline{F}\in  S/(u_1)$,
of Lemma \ref{lemsortiekappaegaldeux} satisfies the following two conditions:
\begin{itemize}
  \item [(i)] $\mathrm{ord}_{(\overline{u}_2,\overline{u}_3)}\overline{F}=\omega (x)+1$;
  \item [(ii)] the form $\Phi:=\mathrm{cl}_{\omega (x)+1}\overline{F} \in k(x)[\overline{U}_2,\overline{U}_3]_{\omega (x)+1}$
  is such that
$$
{\partial \Phi \over \partial \overline{U}_3}=0 \ \mathrm{and} \
\mathrm{Vdir}({\partial \Phi \over \partial \overline{U}_2})=<\overline{U}_2, \overline{U}_3>.
$$
\end{itemize}

Then $x$ is resolved.
\end{prop}

\begin{proof}
This is a simpler variation of Proposition \ref{sortiekappaegaldeux} and we build upon
its proof. To begin with, let $(u_1,u_2,u'_3;Z')$ be another set of well adapted coordinates at $x$. There is an equality
$$
U'_3=\lambda_3 U_3 +\lambda_2U_2 +\lambda_1 U_1 \in G(m_S)_1=<U_1,U_2,U_3>, \ \lambda_3 \neq 0.
$$
The corresponding initial form polynomial $\mathrm{in}_E h={Z'}^p +U_1^{pd_1}\overline{F}'$ satisfies
$$
\Phi ':=\mathrm{cl}_{\omega (x)+1}\overline{F}'=
\overline{F}(\overline{U}_2, \lambda_3^{-1}(\overline{U}'_3-\lambda_2\overline{U}_2))
+\Theta^p \in k(x)[\overline{U}_2,\overline{U}'_3]_{\omega (x)+1},
$$
where $\Theta \in k(x)[\overline{U}_2,\overline{U}'_3]_{(\omega (x)+1)/p}$, $\Theta =0$ if $d_1 \not \in \N$ or if
$\omega (x)+ 1 \not\equiv 0 \ \mathrm{mod}p$. We deduce that
\begin{equation}\label{eq715}
    {\partial \Phi ' \over \partial \overline{U}'_3}=0 \ \mathrm{and} \
\mathrm{Vdir}({\partial \Phi \over \partial \overline{U}_2})=<\overline{U}_2, \overline{U}'_3>.
\end{equation}
In other terms, (i) and (ii) remains valid for the well adapted coordinates $(u_1,u_2,u'_3;Z')$.

\smallskip

Also note that no $\Phi$ satisfies (ii) when $\omega (x)+1 \equiv 0 \ \mathrm{mod}p$, since then
\begin{equation}\label{eq7151}
{\partial \Phi \over \partial \overline{U}_3}=0 \Longrightarrow
\Phi \in k(x)[\overline{U}_2^p, \overline{U}_3^p] \Longrightarrow {\partial \Phi \over \partial \overline{U}_2}=0.
\end{equation}

Let ${\cal X}' \longrightarrow ({\cal X},x)$ be the blowing along $x$, $x' \in {\cal X}'$ be the
center of $\mu$ and suppose that $\iota (x')\geq (p,\omega (x),2)$.
We discuss according to the values of $\tau ' (x)$ as in the proof of Proposition \ref{sortiekappaegaldeux}.

\medskip

\noindent $\bullet$ {\it Assume that $\tau '(x)=3$.}  The proposition follows from Theorem \ref{bupthm}.

\smallskip

\noindent $\bullet$ {\it Assume that $\tau '(x)=2$.}
By (\ref{eq715}), it can be assumed that $\mathrm{Vdir}(x)=<U_1,U_2>$ or
$\mathrm{Vdir}(x)=<U_1,U_3>$. The polynomial assumption  {of Proposition} \ref{sortiekappaegaldeux} (i) on $\overline{F}$
is used only in cases 1 and 2 of the corresponding proof.
Therefore under the assumptions of this proposition, it is sufficient to prove that
\begin{equation}\label{eq716}
\left\{
  \begin{array}{ccc}
    \Phi \not\in <\overline{U}_2^{1+\omega (x)}, \overline{U}_3\overline{U}_2^{\omega (x)}>
    & \mathrm{if} & \mathrm{Vdir}(x)=<U_1,U_2> \\
      &  &  \\
    \Phi \not\in <\overline{U}_3^{1+\omega (x)}, \overline{U}_2\overline{U}_3^{\omega (x)}>
    & \mathrm{if} & \mathrm{Vdir}(x)=<U_1,U_3> \\
  \end{array}
\right.
\end{equation}
and that
\begin{equation}\label{eq717}
\left\{
  \begin{array}{ccc}
    \Phi \not\in <\overline{U}_2^{1+\omega (x)}, \overline{U}_3\overline{U}_2^{\omega (x)},
    \overline{U}_3^2\overline{U}_2^{\omega (x)-1}> & \mathrm{if} & \mathrm{Vdir}(x)=<U_1,U_2> \\
      &  &  \\
    \Phi \not\in <\overline{U}_3^{1+\omega (x)}, \overline{U}_2\overline{U}_3^{\omega (x)},
    \overline{U}_2^2\overline{U}_3^{\omega (x)-1}> & \mathrm{if} & \mathrm{Vdir}(x)=<U_1,U_3> \\
  \end{array}
\right.
\end{equation}
if furthermore $\omega (x)-1 \equiv 0 \ \mathrm{mod}p$. By (ii), we have
$$
\Phi \in k(x)[\overline{U}_2,\overline{U}_3^p] \backslash k(x)[\overline{U}_2] \ \mathrm{and} \
{\partial \Phi \over \partial \overline{U}_2}\not\in <\overline{U}_3^{\omega (x)}>
$$
and (\ref{eq716}) follows easily. Furthermore, (\ref{eq717}) reduces to (\ref{eq716}) except possibly
if $p=2$; but assumption (ii) then implies that $\omega (x)\equiv 0 \ \mathrm{mod}2$ by (\ref{eq7151}).

\medskip

\noindent $\bullet$ {\it Assume that $\tau '(x)=1$.}  We have $\mathrm{Vdir}(x)=<U_1>$. The polynomial assumption
Proposition \ref{sortiekappaegaldeux} (i) on $\overline{F}$ is also used only in cases 1 and 2
of the corresponding proof.

\smallskip

If $k(x')=k(x)$, one is then reduced to proving (\ref{eq716})-(\ref{eq717})
and the proof is identical as in (b).

\smallskip

If $[k(x'):k(x)]\geq 2$, the argument in \cite{CoP2} proof of {\bf II.3} (cases 1 and 2 on p.1894)
shows that $p=2$, $\omega (x)=3$ and $[k(x'):k(x)]= 2$; but assumption (ii) then implies
$\omega (x)\equiv 0 \ \mathrm{mod}2$ by (\ref{eq7151}) and the conclusion follows.
\end{proof}

\begin{prop}\label{tauegaldeux}
Assume that $E=\mathrm{div}(u_1)$, $\epsilon (x)=\omega (x)$, $\kappa (x)= 2$ and
$$
\mathrm{Vdir}(x) + k(x)U_1= <U_1,U_2,U_3>.
$$
Then $x$ is good.
\end{prop}

\begin{proof}
This follows from Theorem \ref{bupthm} if $\mathrm{Vdir}(x)=<U_1,U_2,U_3>$, i.e.
$\tau '(x)=3$.

\smallskip

Assume that $\tau '(x)=2$. Since $\mathrm{Vdir}(x)$ and $\iota (x)$ do not depend on the choice of
well adapted coordinates, it can be assumed w.l.o.g. that $\mathrm{Vdir}(x)= <U_2,U_3>$. Since $\epsilon (x)=\omega (x)$,
there is an expansion
$$
\mathrm{in}_{m_S}h=Z^p + F_{p,Z}, \ H^{-1}F_{p,Z}\subseteq k(x)[U_2,U_3]_{\omega (x)}.
$$
Let $\mu$ be a valuation of $L=k({\cal X})$ centered at $x$,
${\cal X}_1 \longrightarrow {\cal X}$ be the blowing up along $x$ and
$x_1 \in  {\cal X}_1$ be the center of $\mu$. By Theorem \ref{bupthm}, $\iota (x_1)<\iota (x)$ except possibly
if $x_1=x':=(Z':=Z/u_1, u_1, u'_2:=u_2/u_1, u'_3:=u_3/u_1)$, so $E'=\mathrm{div}(u_1)$ and $k(x')=k(x)$.

\smallskip

By Proposition \ref{originchart}, $\Delta_{S'} (h';u_1,u'_2,u'_3;Z')$ is minimal. We deduce that
$\epsilon (x')\leq \epsilon (x)$; if $x_1$ is very near $x$, we have $\epsilon (x_1)=\epsilon (x)=\omega (x_1)$ and
$$
\mathrm{in}_{m_S}h={Z'}^p -{G'}^{p-1}Z'+ F_{p,Z'}, \ {H'}^{-1}F_{p,Z'}\subseteq k(x)[U_1,U'_2,U'_3]_{\omega (x)}.
$$
Moreover Proposition \ref{bupformula}(v) implies that
$$
J(F_{p,Z'},E',m_{S'})  \equiv  U_1^{-\epsilon (x)}J(F_{p,Z},E,m_S) \ \mathrm{mod}U_1 .
$$
We deduce that $\kappa (x_1)=1$ (so $\iota (x_1)<\iota (x)$) if $G'\neq 0$. Otherwise we have
$\mathrm{Vdir}(x_1)\equiv <U'_2,U'_3>  \ \mathrm{mod}U_1$, so $x_1$ satisfies again the assumptions
of the proposition. The proposition then follows from Corollary \ref{permisarcthree}.
\end{proof}

\begin{rem}\label{rempermkappa2}
All local blowing ups considered in this section are permissible of the first kind
except when $p\geq 3$ and $\omega (x)\leq 2$ (proof of Lemma \ref{sortiemonome} for $\omega (x)=2$, proof of
Lemma \ref{sortiekappaegaldeux} for $\omega (x)=1$).
\end{rem}

\subsection{Reduction to monic expansions.}\label{reductionmonique}

In this section, we further reduce the proof of the Projection Theorem to those points with
$\kappa (x)=2$ satisfying condition (*) below. To begin with, let $(u_1,u_2,u_3;Z)$ be well adapted
coordinates and
\begin{equation}\label{eq720}
    \mathrm{in}_{m_S}h=Z^p -G^{p-1}Z+F_{p,Z} \in G(m_S)[Z]
\end{equation}
be the corresponding initial form. If $\kappa (x)=2$, we have $\mathrm{div}(u_1)\subseteq E \subseteq
\mathrm{div}(u_1u_2)$, $E=\mathrm{div}(u_1)$ if $\omega (x)=\epsilon (x)-1$.
We recall from Definition \ref{defkappa} that $G=0$ if $\omega (x)=\epsilon (x)$.

\begin{defn}\label{*kappadeux}
Assume that $\kappa (x)=2$. We say that $x$ satisfies condition (*) \index{(*) @ (*) $\kappa(x)=2$, (*1)(*2)(*3), Definition~\ref{*kappadeux}} if there  {exist} well adapted
coordinates $(u_1,u_2,u_3;Z)$ such that one of the following properties  {is} satisfied:
\begin{itemize}
  \item [(i)] $\omega (x)=\epsilon (x)$, $U_3 \in \mathrm{Vdir}(x)$ and $J(F_{p,Z},E,m_S)\subseteq G(m_S)_{\epsilon (x)}$
  contains a  {monic} polynomial in $U_3$;
  \item [(ii)] $\omega (x)=\epsilon (x)-1$, $U_3 \in \mathrm{Vdir}(x)$ and $H^{-1}{\partial F_{p,Z} \over \partial U_2}$
  is (generated by) a  {monic} polynomial in $U_3$.
\end{itemize}
Condition (*) is labeled (*1) (resp. (*2)) if $E=\mathrm{div}(u_1)$ (resp.  if $E=\mathrm{div}(u_1u_2)$) when condition (i)
holds. Condition (*) is labeled (*3) when condition (ii) holds.
\end{defn}

\begin{prop}\label{redto*}
Assume that $\kappa (x)=2$. Let $\mu$ be a valuation of $L=k({\cal X})$ centered at $x$ and
consider the quadratic sequence
$$
    ({\cal X},x)=:({\cal X}_0,x_0) \leftarrow ({\cal X}_1,x_1)\leftarrow \cdots \leftarrow ({\cal X}_r,x_r)\leftarrow \cdots
$$
along $\mu$. The following holds:

\begin{itemize}
  \item [(i)] there exists $r \geq 0$ such that $x_r$ is resolved or ($\iota (x_r)=\iota (x)$
  and $x_r$ satisfies condition (*));
  \item [(ii)] if $x$ satisfies condition (*), then $x_1$ is resolved or  ($\iota (x_1)=\iota (x)$ and
  $x_1$ satisfies again condition (*));
  \item [(iii)] if $\omega (x) \not \equiv 0 \ \mathrm{mod}p$, then $x$ is good.
\end{itemize}
\end{prop}

\begin{proof}
We first prove together (i) and (ii) by a casuistic analysis. The
discussion goes according to the value of $\tau '(x)$ and subdivides in the different situations
$\omega (x)=\epsilon (x)$ and $\omega (x)=\epsilon (x)-1$. \\

$\bullet$ {\it Assume that $\tau '(x)=3$.} Then $\iota (x_1)< \iota (x)$ by Theorem \ref{bupthm}, so $x$
is good and there is nothing more to be proved.\\

$\bullet$ {\it Assume that $\tau '(x)=1$ and $\omega (x)=\epsilon (x)$.} We may pick well adapted
coordinates $(u_1,u_2,u_3;Z)$ such that $U_3 \in \mathrm{Vdir}(x)$, so
$$
J(F_{p,Z},E,m_S)=<U_3^{\omega (x)}>.
$$
We deduce that $\omega (x)\equiv 0 \ \mathrm{mod}p$
and $x$ satisfies condition (*1) or (*2). This proves that (i) holds with $r=0$. \\

To prove (ii), we may assume that  $\iota (x_1)\geq \iota (x)$ (in particular $\omega (x_1)=\omega (x)$).
There is an expansion (\ref{eq720}) with
\begin{equation}\label{eq7201}
G=0 \ \mathrm{and} \ U_1^{-pd_1}U_2^{-pd_2}F_{p,Z}=\lambda U_3^{\omega (x)}, \ \lambda \neq 0.
\end{equation}
By Theorem \ref{bupthm}, $x_1$ lies on the strict transform of $\mathrm{div}(u_3)$. Let
$$
x':=(Z':={Z \over u_2}, u'_1:={u_1 \over u_2}, u_2, u'_3:={u_3 \over u_2}), \ E'=\mathrm{div}(u'_1u_2).
$$

If $x_1=x'$, then $\Delta_{S'} (h';u'_1,u_2,u'_3;Z')$ is minimal by Proposition \ref{originchart}. One deduces from (\ref{eq7201}) that
$$
\epsilon (x')=\omega (x) \ \mathrm{and} \ J(F_{p,Z'},E',m_{S'}) \equiv <U_3^{\omega (x)}> \ \mathrm{mod}
(U_2)\cap G(m_{S'})_{\epsilon (x')}.
$$
This proves that ($\iota (x')=\iota (x)$ and $x'$ satisfies condition (*2)), so (ii) holds.

\smallskip

If $x_1\neq x'$, there exists a  {monic} polynomial $P(t)\in S[t]$, whose reduction
$\overline{P}(t)\in k(x)[t]$ is irreducible, such that
\begin{equation}\label{eq7206}
    x_1=(X':={Z\over u_1},u_1,v_2:=P(u'_2),u'_3:={u_3 \over u_1}), \ u'_2:={u_2 \over u_1}, \ E_1=\mathrm{div}(u_1).
\end{equation}

We have $S_1/(u_1)\simeq k(x)[\overline{u}'_2,\overline{u}'_3]_{(\overline{v}_2,\overline{u}'_3)}$.
Let $(u_1, v_2,u'_3;Z_1)$ be well adapted coordinates at $x_1$, where
$Z_1=X' - \phi_1$, $\phi_1 \in S_1$.
Let $d'_1:=d_1+ d_2 -1+ \omega (x)/p$ and $c\in k(x_1)$ be the residue of $u'_2$.
Note that we may furthermore assume that $P(t)\neq t$ if $E=\mathrm{div}(u_1u_2)$
by symmetry  {between} $u_1$ and $u_2$, i.e. $c^{pd_2}\neq 0$ (and $c^{pd_2}=1$ if $c=0$).

\medskip

\noindent {\it Case 1:} $d'_1 \not \in \N$ or $\lambda c^{pd_2}\not \in k(x_1)^p$.
By (\ref{eq7201}), it can be assumed w.l.o.g. that $\mathrm{ord}_{(u_1)}\phi_1 > d'_1$.
The initial form $\mathrm{in}_{E_1}h_1$ of Lemma \ref{lemsortiekappaegaldeux} is then of the form:
$$
\mathrm{in}_{E_1}h_1={Z_1}^p +\lambda U_1^{pd'_1}{\overline{u}'_2}^{pd_2}{\overline{u}'_3}^{\omega (x)}
\in S_1/(u_1)[U_1][Z_1].
$$
We have $\epsilon (x_1)=\omega (x)$ and
$$
J(F_{p,Z_1},E_1,m_{S_1})\equiv <{U'_3}^{\omega (x)}> \ \mathrm{mod}
(U_1)\cap G(m_{S_1})_{\epsilon (x_1)}.
$$
Therefore $\iota (x_1)=\iota (x)$ and $x_1$ satisfies condition (*1), so (ii) holds.

\medskip

\noindent {\it Case 2:} $d'_1 \in \N$ and $\lambda c^{pd_2}\in k(x_1)^p$. It can be assumed w.l.o.g. that
$$
u_1^{-d'_1}\phi_1\equiv \gamma_1 {u'_3}^{{\omega (x)\over p}} \ \mathrm{mod}(u_1),
$$
where $\gamma_1 \in S_1$ is a preimage of $-(\lambda c^{pd_2})^{1/p}\in k(x_1)$.
Since $\Delta_S (h;u_1,u_2,u_3;Z)$ is minimal, we have
$$
0\neq d(\lambda U_1^{pd_1}U_2^{pd_2})\in \Omega^1_{G(m_S)/\F_p}.
$$
We deduce that $(u_1,v'_2:=\gamma {u'_2}^{pd_2} +\gamma_1^p, u'_3)$ is a r.s.p. of $S_1$, where $\gamma \in S$ is
a preimage of $\lambda$. Let $(u_1, v'_2,u'_3;Z'_1)$ be well adapted coordinates at $x_1$, so the initial form $\mathrm{in}_{E_1}h_1$
of Lemma \ref{lemsortiekappaegaldeux} is now of the form:
$$
\mathrm{in}_{E_1}h_1={Z'_1}^p +U_1^{pd'_1}\overline{v}'_2{\overline{u}'_3}^{\omega (x)} \in S_1/(u_1)[U_1][Z'_1].
$$
If $\epsilon (x_1)=\omega (x)$, then $x_1$ satisfies the assumptions of Lemma \ref{sortiemonome}, so $x_1$
is resolved. Otherwise we have $\epsilon (x_1)=1+\omega (x)$ and
$$
{H'}^{-1}{\partial F_{p,Z'_1} \over \partial V_2}\equiv <{U'_3}^{\omega (x)}> \ \mathrm{mod}
(U_1)\cap G(m_{S_1})_{\omega (x)}.
$$
Then there exist well adapted coordinates of the form $(u_1, v'_2,v_3;Z')$
at $x_1$ satisfying  Definition \ref{*kappadeux}, so $\iota (x_1)=\iota (x)$ and $x_1$
satisfies condition (*3).\\

$\bullet$ {\it Assume that $\tau '(x)=1$ and $\omega (x)=\epsilon (x)-1$.} By Definition \ref{defkappa},
we then have $H^{-1}{\partial F_{p,Z} \over \partial U_2}\neq (0)$, therefore
\begin{equation}\label{eq7205}
H^{-1}{\partial F_{p,Z} \over \partial U_2}=<U_3^{\omega (x)}>,
\end{equation}
so $x$ satisfies condition (*3). This proves that (i) holds.

\smallskip

To prove (ii), we may assume that  $\iota (x_1)\geq \iota (x)$. By (\ref{eq7205}), there is an expansion (\ref{eq720})
with
\begin{equation}\label{eq7202}
G=0, \  U_1^{-pd_1}F_{p,Z}=\lambda U_2U_3^{\omega (x)} + \Phi_0(U_2^p,U_3^p)+U_1\Phi (U_1,U_2^p,U_3^p), \lambda \neq 0.
\end{equation}
This furthermore implies that $\omega (x)\equiv 0 \ \mathrm{mod}p$, so $\Phi_0=0$. By Theorem \ref{bupthm},
$x_1$ lies on the strict transform of $\mathrm{div}(u_3)$. Note that we may furthermore assume
that
\begin{equation}\label{eq7203}
\lambda =1 \ \mathrm{and} \ \mathrm{deg}_{U_3}\Phi (U_1,U_2^p,U_3^p)< \omega (x)
\end{equation}
in (\ref{eq7202}): this is achieved by possibly changing $u_2$ to $\gamma_0u_2 + \gamma u_1$,
$\gamma_0 \gamma\in S$ a unit, then picking again well prepared coordinates. Let
$$
x':=(Z':={Z\over u_2}, u'_1:={u_1\over u_2}, u_2, u'_3:={u_3\over u_2}), \ E'=\mathrm{div}(u'_1u_2).
$$

If $x_1=x'$, the proof is identical  {to that in the case} when $\omega (x)=\epsilon (x)$: one gets
($\iota (x')=\iota (x)$ and $x'$ satisfies condition (*2)), so (ii) holds.

\smallskip

If $x_1\neq x'$, we let $d'_1:=d_1-1+ (1+\omega (x))/p$ and use  {the same notation as
in the case} $\omega (x)=\epsilon (x)$. We have $E_1=\mathrm{div}(u_1)$ and consider three cases.

\medskip

\noindent {\it Case 1:} $d'_1 \not \in \N$. By (\ref{eq7202}),
$\mathrm{ord}_{(u_1)}\phi_1 > d'_1$.

\smallskip

If $x_1=x'_1:=(Z'_1:=Z/u_1, u_1,u'_2:=u_2/u_1, u'_3:=u_3/u_1)$, we have $\Phi \in k(x)[U_2^p,U_3^p]$ and
the initial form $\mathrm{in}_{E_1}h_1$ of Lemma \ref{lemsortiekappaegaldeux} is of the form:
$$
\mathrm{in}_{E_1}h_1={Z'_1}^p +U_1^{pd'_1}(\overline{u}'_2{\overline{u}'_3}^{\omega (x)}+
\Phi ({\overline{u}'_2}^p,{\overline{u}'_3}^p))\in S_1/(u_1)[U_1][Z_1].
$$

If $\Phi =0$, we either have $\epsilon (x'_1)=\omega (x)$, so $x'_1$ satisfies the assumptions
of Lemma \ref{sortiemonome} and $x_1$ is resolved; or  $\epsilon (x'_1)=1+\omega (x)$ and
$$
{H'}^{-1}{\partial F_{p,Z'_1} \over \partial U'_2}\equiv <{U'_3}^{\omega (x)}> \ \mathrm{mod}
(U_1)\cap G(m_{S_1})_{\omega (x)}.
$$
Then  $\iota (x'_1)=\iota (x)$ and $x'_1$ satisfies condition (*3).

\smallskip

If $\Phi \neq 0$, we have $\epsilon (x'_1)=\omega (x)$ and
$$
{H'}^{-1}F_{p,Z'_1}\equiv <\Phi ({U'_2}^p,{U'_3}^p)> \ \mathrm{mod}
(U_1)\cap G(m_{S_1})_{\omega (x)}.
$$
If $U_2U_3$ divides $\Phi$, then $x_1$ is good by Proposition \ref{tauegaldeux}; otherwise $\Phi$
is monic in $U_2$ or in $U_3$, so $\iota (x'_1)=\iota (x)$ and $x'_1$ satisfies condition (*1).

\smallskip

If $x_1\neq x'_1$, then $\epsilon (x_1)=\omega (x)$, $\iota (x_1)=\iota (x)$
and $x_1$ satisfies condition (*1).

\medskip

\noindent {\it Case 2:} $d'_1 \in \N$ and $c \not \in k(x_1)^p$ (remember that $c\in k(x_1)$ is the residue of $u'_2$).
With notations as in (\ref{eq7206}) {\it sqq.}, we get $\epsilon (x')=\omega (x)$ and
$$
{H'}^{-1}F_{p,Z_1}\equiv <c{U'_3}^{\omega (x)} +\Phi_1 ({U'_2}^p,{U'_3}^p)> \ \mathrm{mod}
(U_1)\cap G(m_{S_1})_{\omega (x)},
$$
where $\mathrm{deg}_{U'_3}\Phi_1 ({U'_2}^p,{U'_3}^p)<\omega (x)$ by (\ref{eq7203}).
Therefore $\iota (x_1)=\iota (x)$ and $x_1$ satisfies condition (*1).

\medskip

\noindent {\it Case 3:} $d'_1 \in \N$ and $c \in k(x_1)^p$. It can be assumed w.l.o.g. that
$$
u_1^{-d'_1}\phi_1 \equiv \gamma_1 {u'_3}^{{\omega (x) \over p}} +
\sum_{i=1}^{{\omega (x) \over p}}\psi_i{u'_3}^{{\omega (x)\over p} -i} \ \mathrm{mod}(u_1),
$$
where $\gamma_1 \in S_1$ is a preimage of $c^{1/p}\in k(x_1)$ and
$$
\overline{\psi}_i\in k(x)[\overline{u}'_2]_{(\overline{v}_2)}\subset S_1/(u_1), \ 1 \leq i \leq {\omega (x)\over p}.
$$
Then $(u_1,v'_2:=u'_2  +\gamma_1^p, u'_3)$ is a r.s.p. of $S_1$ ({\it viz.}
above $\omega (x)=\epsilon (x)$, case 2). Let $(u_1,v'_2,u'_3;Z'_1)$ be well adapted coordinates. We have
$$
\mathrm{in}_{E_1}h_1={Z'_1}^p +U_1^{pd'_1}(\overline{v}'_2{\overline{u}'_3}^{\omega (x)}
+\Psi (\overline{u}'_2,\overline{u}'_3)) \in S_1/(u_1)[U_1][Z'_1],
$$
where $\Psi (\overline{u}'_2,\overline{u}'_3)\in k(x)[\overline{u}'_2]_{(\overline{v}_2)}[\overline{u}'_3]$,
$\mathrm{ord}_{(\overline{v}'_2,\overline{u}'_3)}\Psi \geq \omega (x)$. Since $\omega (x')=\omega (x)$, we have
\begin{equation}\label{eq7204}
\overline{\Psi}:=\mathrm{cl}_{\omega (x)}\Psi (\overline{u}'_2,\overline{u}'_3)
\in ({\overline{V}'_2}^p k(x')[{\overline{V}'_2}^p,{\overline{U}'_3}^p])_{\omega (x)}.
\end{equation}

If ($\epsilon (x_1)=\omega (x)$ and $\overline{\Psi}=0$), then $x_1$ satisfies the assumptions
of Lemma \ref{sortiemonome}, so $x_1$ is resolved.

\smallskip

If ($\epsilon (x_1)=\omega (x)$ and $0 \neq \overline{\Psi}\in <{\overline{V}'_2}^{\omega (x)}>$), we have
$$
J(F_{p,Z'_1},E_1,m_{S_1})\equiv <{V'_2}^{\omega (x)}> \ \mathrm{mod}
(U_1)\cap G(m_{S_1})_{\omega (x)}.
$$
Therefore $\iota (x_1)=\iota (x)$ and $x_1$ satisfies condition (*1).

\smallskip

If ($\epsilon (x_1)=\omega (x)$ and $\overline{\Psi}\not \in <{\overline{V}'_2}^{\omega (x)}>$), we have
$$
\kappa (x')=2 \ \mathrm{and} \ \mathrm{Vdir}(x')+k(x')U_1 =  <U_1,V'_2, U'_3> ,
$$
so $x_1$ is good by Lemma \ref{tauegaldeux}.

\smallskip

If $\epsilon (x_1)=1+\omega (x)$, we have
$$
{H'}^{-1}{\partial F_{p,Z'_1} \over \partial V'_2}\equiv <{U'_3}^{\omega (x)}> \ \mathrm{mod}
(U_1,V'_2)\cap G(m_{S_1})_{\omega (x_1)}.
$$
Then there exist well adapted coordinates of the form $(u_1, v'_2,v_3;Z')$
at $x_1$ satisfying  Definition \ref{*kappadeux}, so $\iota (x_1)=\iota (x)$ and $x_1$
satisfies condition (*3). This concludes the proof of (ii) when $\tau '(x)=1$. \\

$\bullet$ {\it Assume that $\tau '(x)=2$.} Up to a change of well adapted coordinates, it is easily seen that
$x$ belongs to one of the following types:\\

\noindent (T0) $\omega (x)=\epsilon (x)$, $E=\mathrm{div}(u_1)$ and $\mathrm{Vdir}(x)=<U_3,U_2>$;

\smallskip

\noindent (T1) $\omega (x)=\epsilon (x)-1$ and $\mathrm{Vdir}(x)=<U_3,U_2>$;

\smallskip

\noindent (T2) $\omega (x)=\epsilon (x)$, $E=\mathrm{div}(u_1u_2)$ and  $\mathrm{Vdir}(x)=<U_3, U_1 + \lambda U_2>$
with $\lambda \neq 0$;

\smallskip

\noindent (T3) $\omega (x)=\epsilon (x)$ and  $\mathrm{Vdir}(x)=<U_3,U_1>$;

\smallskip

\noindent (T4) $\omega (x)=\epsilon (x)-1$ and  $\mathrm{Vdir}(x)=<U_3,U_1>$.

\medskip

{\it Claim:} suppose $x$ is of type (Tk)\index{(Tk)@(Tk), $0 \leq k \leq 4$}, $0 \leq k \leq 4$. Then $x_1$ is resolved or
one of the following properties hold:
\begin{itemize}
  \item [(a)] $\iota (x_1)=\iota (x)$ and $x_1$ satisfies condition (*);
  \item [(b)] $\iota (x_1)=\iota (x)$, $\tau '(x_1)=2$ and $x_1$ is of type (Tl) with $l\leq k$.
\end{itemize}

If moreover $x$ satisfies condition (*), then $x_1$ is resolved or (a) holds.\\

To prove the claim, we do a case by case analysis. If $k=0$, then $x$ is good by Proposition \ref{tauegaldeux}.
\\

{\it Assume that $k=1$}. There is an expansion (\ref{eq720}) with
$$
G=0 \ \mathrm{and} \ U_1^{-pd_1}F_{p,Z}=F_{1+\omega (x)}(U_2,U_3) +
\sum_{i=1}^{1+\omega (x)}F_{1+\omega (x)-i}(U_2,U_3)U_1^i.
$$
Since $\mathrm{Vdir}(x)=<U_2,U_3>$, we have
\begin{equation}\label{eq721}
\left\{
  \begin{array}{c}
    \mathrm{Vdir}\left ({\partial F_{1+\omega (x)} \over \partial U_2}, {\partial F_{1+\omega (x)} \over \partial U_3}
    \right )=<U_2,U_3> \hfill{}\\
      \\
    F_{1+\omega (x)-i}(U_2,U_3)\in k(x)[U_2^p,U_3^p], \ 1 \leq i \leq 1+\omega (x) \\
  \end{array}
\right.
.
\end{equation}
Assume that $\iota (x')\geq \iota (x)$. By Theorem \ref{bupthm}, $x_1=x'$, where
$$
x':=(Z':=Z/ u_1, u_1, u'_2:=u_2/u_1, u'_3:=u_3/u_1).
$$
We have
$$
E'=\mathrm{div}(u_1), S'/(u_1)\simeq k(x)[\overline{u}'_2,\overline{u}'_3]_{(\overline{u}'_2,\overline{u}'_3)}
\ \mathrm{and} \ H(x')= (u_1^{p(d_1-1)+1+\omega (x)}).
$$
Assume that $\iota (x')\geq \iota (x)$. By Proposition \ref{originchart}, $\Delta_{S'} (h';u_1,u'_2,u'_3;Z')$ is minimal.
The initial form $\mathrm{in}_{E'}h'$ of Lemma \ref{lemsortiekappaegaldeux} is of the form:
$$
\mathrm{in}_{E'}h'={Z'}^p +U_1^{p(d_1-1)+1+\omega (x)}\left (F_{1+\omega (x)}(\overline{u}'_2,\overline{u}'_3) +
\sum_{i=1}^{1 +\omega (x)}F_{1+\omega (x)-i}(\overline{u}'_2,\overline{u}'_3)\right ).
$$
This proves that $F_i(U_2,U_3)=0$, $2 \leq i \leq 1+\omega (x)$. We consider two cases:

\smallskip

\noindent {\it Case 1:} $F_{\omega (x)}(U_2,U_3)=0$. If $\epsilon (x')=\omega (x)$,
then $x'$ satisfies all assumptions
of Proposition \ref{sortiekappaegaldeux} by (\ref{eq721}), so $x$ is good.

\smallskip

If $\epsilon (x')=\epsilon (x)$, then $\iota (x')=\iota (x)$ and
$$
{H'}^{-1}{\partial F_{p,Z'} \over \partial U'_j}
\equiv <{\partial F_{1+\omega (x)} \over \partial U_j}(U'_2,U'_3)>
\ \mathrm{mod}(U_1)\cap G(m_{S'})_{\omega (x)},
$$
for $j=2,3$ again by (\ref{eq721}). We conclude that $\tau '(x')=3$ (so $x$ is good) or $x'$ is again of type (T1)
as required. If $x$ satisfies condition (*), so does $x'$.

\smallskip

\noindent {\it Case 2:} $F_{\omega (x)}(U_2,U_3)\neq 0$. We have $\epsilon (x')=\omega (x)$ and
$$
\mathrm{in}_{m_{S'}}h'={Z'}^p + U_1^{p(d_1-1)+1+\omega (x)}(F_{\omega (x)}(U'_2,U'_3)+U_1\Phi '),
\  \Phi ' \in k(x')[U_1,{U'_2}^p,{U'_3}^p].
$$
Therefore $\iota (x')= \iota (x)$.
If $F_{\omega (x)}(U_2,U_3)$ is monic in $U_2$ or in $U_3$, then $x'$ satisfies condition (*1).
Otherwise $x'$ is of type (T0) and the conclusion follows.\\

Note that if $\omega (x)=1$, $x$ is of type (T1) and satisfies condition (*3). So we may assume from this point on
that $\omega (x)\geq 2$.\\

{\it Assume that $k=2$}. There is an expansion (\ref{eq720}) with $G=0$ and
$$
F_{p,Z}=U_1^{pd_1}U_2^{pd_2}\sum_{i=0}^{\omega (x)}F_{i}(U_1,U_2)U_3^{\omega (x)-i}.
$$
Note that $F_i(U_1,U_2)=0$ whenever $\omega (x)-i \not \equiv 0 \ \mathrm{mod}p$, since $\omega (x)=\epsilon (x)$;
we have $F_i\neq 0$ for some $i$, $0\leq i \leq \omega (x)-1$ since $\kappa (x)=2$; moreover
$F_0\neq 0$ iff $x$ satisfies condition (*).

\smallskip

Assume that $\iota (x')\geq \iota (x)$. By Theorem \ref{bupthm},
we have
$$
x_1=x':=(X':=Z/u_1, u_1, u'_2:=1+ \gamma  u_2/u_1 , u'_3:=u_3/u_1),
$$
$\gamma \in S$ being a preimage of $\lambda$. We have
$$
E'=\mathrm{div}(u_1), \ k(x')=k(x) \ \mathrm{and} \ H(x')= (u_1^{p(d_1+d_2-1)+\omega (x)}).
$$
Assume that $\iota (x')\geq \iota (x)$. Since $\mathrm{Vdir}(x)=<U_3,U_1 + \lambda U_2>$,
we consider two cases deduced  from Lemma \ref{joyeux}:

\smallskip

\noindent {\it Case 1:} $\omega (x)\equiv 0 \ \mathrm{mod}p$. By Lemma \ref{joyeux}(i), it can be assumed w.l.o.g
that
\begin{equation}\label{eq723}
    F_{pi}(U_1,U_2)=c_{pi}(U_1 + \lambda U_2)^{pi}, \ c_{pi} \in k(x), 1 \leq i \leq {\omega (x) \over p}.
\end{equation}
After blowing up, there is an expansion $\mathrm{in}_{m_{S'}}h'={X'}^p + F_{p,X'}$, where
\begin{equation}\label{eq7231}
U_1^{-pd'_1}F_{p,X'}=(- \lambda)^{-pd_2}\sum_{i=0}^{\omega (x)/ p}c_{pi}{U'_2}^{pi}{U'_3}^{\omega (x)-pi}
+U_1\Phi ',
\end{equation}
for some $\Phi ' \in k(x)[U_1,{U'_2}^p,{U'_3}^p]$, $d'_1:=d_1+d_2 -1 +\omega (x)/p$.

\smallskip

If $d'_1 \not \in \N$, then $\epsilon (x')=\omega (x)$ and $\iota (x')=\iota (x)$.
Moreover
$$
k(x')U_1 +\mathrm{Vdir}(x')= <U_1,U'_2,U'_3>,
$$
so $\tau '(x')=3$  or $x'$ is of type (T0). In both cases, $x$ is good.

\smallskip

If $(d_1,d_2)  \in \N^2$, it can be assumed furthermore that $c_{pi}=0$ or $c_{pi} \not \in k(x)^p$ in (\ref{eq723}).
We have $d'_1  \in \N$ and we also get $\epsilon (x')=\omega (x)$ and $\iota (x')=\iota (x)$. Since
$$
J(F_{p,Z},E,x)=H^{-1}<{\partial F_{p,Z}\over \partial \lambda_l}_{l \in \Lambda_0}>
$$
with notations as in (\ref{eq244}), we get in any case since $k(x')=k(x)$:
$$
k(x')U_1 +\mathrm{Vdir}(x')= <U_1,U'_2,U'_3>.
$$
Therefore $\tau '(x')=3$  or $x'$ is of type (T0), so $x$ is good.

\smallskip

If $d'_1 \in \N$, $d_2 \not  \in \N$, we define
$$
I:=\{i : (-\lambda)^{-pd_2}c_{pi} \not\in k(x)^p\}.
$$

If $I\neq \emptyset$, we also get $\epsilon (x')=\omega (x)$ and $\iota (x')=\iota (x)$.
If $\omega (x)\in I$, $x'$ satisfies condition (*1); otherwise $x'$ is good.

\smallskip

If $I=\emptyset$, let $(u_1,u'_2,u'_3;Z')$ be
well adapted coordinates at $x'$. We denote by $a\in \F_p$ the residue of $pd_2$.
Since $d_2 \not  \in \N$, we have $a\neq 0$. The initial form $\mathrm{in}_{E'}h'$
of Lemma \ref{lemsortiekappaegaldeux} is of the form:
$$
\mathrm{in}_{E'}h'={Z'}^p +U_1^{pd'_1}\overline{F}'(\overline{u}'_2,\overline{u}'_3)\in S'/(u_1)[U_1,Z'],
$$
where $ S'/(u_1) \simeq k(x)[\overline{u}'_2,\overline{u}'_3]_{(\overline{u}'_2,\overline{u}'_3)}$.
The form $\Phi ':= \mathrm{cl}_{\omega (x)+1}\overline{F}'$ is given by
$$
\Phi '= -a (-\lambda )^{-pd_2}
\sum_{i=0}^{\omega (x)/p}c_{pi}{\overline{U}'_2}^{pi+1}{\overline{U}'_3}^{\omega (x)-pi}
\in k(x)[\overline{U}'_2,\overline{U}'_3]_{\omega (x)+1}.
$$
If $\epsilon (x')=\omega (x)$, $x'$ thus satisfies all assumptions of Proposition
\ref{sortiebis}, so $x$ is good. Otherwise, we have $\epsilon (x')=1+\omega (x)$ and
$$
k(x')U_1 +\mathrm{Vdir}(x')= <U_1,U'_2,U'_3>.
$$
Therefore $\iota (x')=\iota (x)$ and $x$ is good (if $\tau (x')=3$) or $x'$ is of type (T1).
If $x$ satisfies condition (*2), i.e. $c_0 \neq 0$, then $x'$ satisfies condition (*3).

\medskip

\noindent {\it Case 2:} $\omega (x)\not \equiv 0 \ \mathrm{mod}p$. Recall that $F_i(U_1,U_2)=0$ whenever
$\omega (x)-i \not \equiv 0 \ \mathrm{mod}p$. Therefore $a:=\widehat{\omega (x)}=\widehat{i}$ whenever $F_i\neq 0$.
Let $a_j:=\widehat{pd_j}$, $j=1,2$. By Lemma \ref{joyeux}(ii), we have $a_1a_2 \neq 0$, $a_1+a_2+a=p$.
Moreover, it can be assumed w.l.o.g. that
\begin{equation}\label{eq7232}
    U_1^{a(1)}U_2^{a(2)}F_{i}(U_1,U_2)=c_i\Phi_{i}(U_1,\lambda U_2), \ c_i \in k(x)^p, 1 \leq i \leq \omega (x),
\end{equation}
with notations as in (\ref{eq7071}). After blowing up, the initial form $\mathrm{in}_{E'}h'$
of Lemma \ref{lemsortiekappaegaldeux} is of the form:
$$
\mathrm{in}_{E'}h'={Z'}^p +U_1^{pd'_1}\overline{F}'(\overline{u}'_2,\overline{u}'_3)\in S'/(u_1)[U_1,Z'],
$$
where $ S'/(u_1) \simeq k(x)[\overline{u}'_2,\overline{u}'_3]_{(\overline{u}'_2,\overline{u}'_3)}$.
The form $\Phi ':= \mathrm{cl}_{\omega (x)+1}\overline{F}'$ is given explicitly by
$$
\Phi '= \left(
  \begin{array}{c}
   a_2+a \\
    a+1 \\
  \end{array}
\right)
\sum_{i=0}^{\lfloor \omega (x)/ p \rfloor }c_{pi+a}{\overline{U}'_2}^{a+pi+1}{\overline{U}'_3}^{\omega (x)-a-pi}
\in k(x)[\overline{U}'_2,\overline{U}'_3]_{\omega (x)+1}.
$$
If $\epsilon (x')=\omega (x)$, $x'$ thus satisfies all assumptions of Proposition
\ref{sortiebis}, so $x$ is good. Otherwise, we have $\epsilon (x')=1+\omega (x)$ and
$$
k(x')U_1 +\mathrm{Vdir}(x')= <U_1,U'_2,U'_3>.
$$
Therefore $\iota (x')=\iota (x)$ and $x$ is good (if $\tau (x')=3$) or is of type (T1).
Note that $x$ did not satisfy condition (*2): since $J(F_{p,Z},E,m_S)\subset k[U_1,U_2,U_3^p]_{\omega (x)}$
and $\omega (x)\not\equiv 0 \ \mathrm{mod}p$,  $J(F_{p,Z},E,m_S)$ contains no monic polynomial in $U_3$. \\

{\it Assume that $k=3$}. There is an expansion (\ref{eq720}) with $G=0$ and
$$
U_1^{-pd_1}U_2^{-pd_2}F_{p,Z}= \sum_{i=0}^{\omega (x)}\lambda_i U_3^{\omega (x)-i}U_1^i.
$$
Assume that $\iota (x')\geq \iota (x)$. By Theorem \ref{bupthm}, we have $x_1=x'$, where
$$
x':=(Z':=Z/u_2, u'_1:=u_1/u_2, u_2, u'_3:=u_3/u_2).
$$
By Proposition \ref{originchart}, $\Delta_{S'} (h';u'_1,u_2,u'_3;Z')$ is minimal and we have
$$
\mathrm{in}_{m_{S'}}h'={Z'}^p + {U'_1}^{pd_1}U_2^{pd'_1}
(\sum_{i=0}^{\omega (x)}\lambda_i {U'_3}^{\omega (x)-i}{U'_1}^i+U_2\Phi '),
$$
where $d'_1:=d_1+d_2-1+\omega (x)/p$, $\Phi ' \in k(x')[U'_1,U_2,{U'_3}^p]$. since it is assumed that $\iota (x')\geq \iota (x)$. Then
$$
\iota (x')=\iota (x) \ \mathrm{and} \ k(x')U_2 + \mathrm{Vdir}(x')=<U'_1,U_2,U'_3>.
$$
We conclude that $\tau '(x')=3$ (so $x$ is good) or $x'$ is of either type (T2) or (T3).
If moreover $x$ satisfies condition (*), i.e. $\lambda_0 \neq 0$, then $x'$ satisfies condition (*2).\\

{\it Assume that $k=4$}. We have $H^{-1}G^p \subseteq k(x)U_1^{\omega (x)+1}$ and
there is an expansion (\ref{eq720}) with
\begin{equation}\label{eq7233}
    U_1^{-pd_1}F_{p,Z}= F_{1+\omega (x)}(U_1,U_3) + \sum_{i=1}^{1+\omega (x)}F_{1+\omega (x)-i}(U_1,U_3)U_2^i.
\end{equation}
Assume that $\iota (x')\geq \iota (x)$. By Theorem \ref{bupthm}, we have $x_1=x'$, where
$$
x':=(Z':=Z/u_2, u'_1:=u_1/u_2, u_2, u'_3:=u_3/u_2), \ E'=\mathrm{div}(u'_1u_2).
$$
By Proposition \ref{originchart}, $\Delta_{S'} (h';u_1,u'_2,u'_3;Z')$ is minimal. We deduce
from (\ref{eq7233}) that
$$
F_{1+\omega (x)-i}(U_1,U_3)=0, \ 2 \leq i \leq 1 +\omega (x),
$$
since it is assumed that $\iota (x')\geq \iota (x)$. Since $\kappa (x)=2$, we deduce from
Definition \ref{defkappa} that
\begin{equation}\label{eq7235}
    F_{\omega (x)}(U_1,U_3)\not \in <U_1^{\omega (x)}>.
\end{equation}
In particular, we get from (\ref{eq7233}):
$$
\epsilon (x')=\omega (x) \ \mathrm{and} \ \mathrm{Vdir}(x') \nsubseteq <U'_1,U_2>.
$$
The initial form polynomial $\mathrm{in}_{m_{S'}}h'$ is therefore given by
\begin{equation}\label{eq7234}
    \mathrm{in}_{m_{S'}}h'={Z'}^p + {U'_1}^{pd_1}U_2^{pd'_2}(F_{\omega (x)}(U'_1,U'_3)  +U_2\Phi ')
\end{equation}
where $d'_2:= d_1+d_2-1+(1+\omega (x))/p$, $\Phi ' \in k(x')[U'_1,U_2,U'_3]$.
This proves that $\iota (x')=\iota (x)$.

\smallskip

Suppose that $x$ satisfies condition (*3), i.e. $F_{\omega (x)}(U_1,U_3)$ is unitary in $U_3$.
We deduce from (\ref{eq7234}) that $x'$ satisfies condition (*2). Otherwise, $U_1$ divides
$F_{\omega (x)}(U_1,U_3)$  and we deduce from (\ref{eq7235}) that
$$
k(x')U_2 + \mathrm{Vdir}(x')=<U'_1,U_2,U'_3>.
$$
Then $x$ is good (if $\tau '(x')=3$), or ($\tau '(x')=2$ and $x'$ is of type (T2) or (T3)).
This concludes the proof of the claim. In particular, we have proved (ii).\\

We now prove (i). Suppose on the contrary that for every $r \geq 0$, $x_r$ does not satisfy condition (*).
The above proof shows that $x_r$ is resolved for some $r \geq 0$ or there exists $r_0 \geq 0$
such that for every $r \geq r_0$, we have
$$
\tau '(x_r)=2 \ \mathrm{and} \  x_r \ \mathrm{is} \ \mathrm{of} \ \mathrm{type} \ (Tk)
$$
where $k \in \{1,3\}$ is independent of $r$. If $k=1$, we derive a contradiction
from Corollary \ref{permisarcthree}.

\medskip

If $k=3$, there exists
$$
\hat{u}_3:=u_3 -\sum_{i=1}^\infty \gamma_{i,3}u_2^i \in \hat{S};
\ \hat{\phi}:= \sum_{i=1}^\infty \gamma_{i}u_2^i\in \hat{S}
$$
with the following property: for every $i \geq 0$, we have $\iota (x_i)=\iota (x)$ and
the strict transform in $({\cal X}_i,x_i)$ of the
formal curve $\hat{{\cal Y}}=(Z-\hat{\phi},u_1,\hat{u}_3)\subset \hat{{\cal X}}$ is nonempty.

\smallskip

Note that the conclusion of Proposition \ref{permisarc}(2) applied to the formal arc
$\varphi: \ \hat{{\cal Y}}\rightarrow {\cal X}$ does not hold. To see this, note that {\it ibid.}(2.b)
implies that $Z_{r_0(\varphi)}$ is an irreducible component of $E_{r_0}$; by {\it ibid.}(2.c) we have $\epsilon (x_{r_0})=1$:
a contradiction, since it is assumed (from the beginning of this proof) that $\omega (x)\geq 2$.

\smallskip

Therefore the conclusion of Proposition \ref{permisarc}(1) holds. Let $(u'_1, u'_2,u'_3;Z')$
be well adapted coordinates at $x_{r_0}$, where ${\cal Y}:=(Z',u'_1,u'_3)\subset ({\cal X}_{r_0},x_{r_0})$
is permissible of the first kind at $x_0$. Since $\mathrm{Vdir}(x_{r_0})=<U'_1,U'_3>$, $x_{r_0}$ is good
by Theorem \ref{bupthm}, hence $x$ is good.\\

To prove (iii), it can be assumed by (i) that $x$ satisfies condition (*). Suppose that
$\epsilon (x)=\omega (x)$. Then $J(F_{p,Z},E,x)$ contains no monic polynomial in $U_3$,
since $\omega (x)\not \equiv 0 \ \mathrm{mod}p$. So $\epsilon (x)=\omega (x)+1$. It has been proved
above that
$$
\tau '(x)=1 \Longrightarrow \omega (x)\equiv 0 \ \mathrm{mod}p.
$$
We deduce that $\tau '(x)\geq 2$. Therefore $x_r$ is resolved for some $r\geq 0$ or
$$
\iota (x_i)=\iota (x), \ \epsilon (x_i)=\omega (x)+1 \ \mathrm{and} \ \tau '(x_i)=2
$$
for every $i\geq 0$. The above claim shows that $x_r$ is of type (T1) for every $r>>0$.
We get $x_r$ resolved for some $r\geq 0$ arguing as in the above proof of (i), so $x$ is good.
\end{proof}

A direct consequence of Proposition \ref{redto*}(iii) and Remark \ref{rempermkappa2} is:

\begin{cor}\label{omeganoncongp}
Projection Theorem \ref{projthm} holds when $\kappa (x)=2$ and $\omega (x)\not \equiv 0 \ \mathrm{mod}p$.
One may take all local blowing ups in (\ref{eq402}) permissible of the first kind if $p=2$
or if $\omega (x)\geq 3$.
\end{cor}

\begin{rem}\label{kappadeuxancien}
Assume that $\kappa (x)=2$, $\omega (x)\equiv 0 \ \mathrm{mod}p$ and use notations as in
Proposition \ref{redto*}.

\smallskip

Suppose that $x$ satisfies condition (*1) or (*2) and $x_1$ satisfies condition (*3).
It follows from the above proof that $x_1$ is resolved or there exist well adapted coordinates
$(u_1,u_2,u_3;Z)$ at $x_1$ such that
\begin{equation}\label{eq7236}
    {H}^{-1}{\partial F_{p,Z} \over \partial U_2}\equiv <\Phi (U_2,U_3)>
    \ \mathrm{mod}(U_1)\cap G(m_{S})_{\omega (x)},
\end{equation}
where $\Phi (U_2,U_3) \in k(x_1)[U_2,{U_3}^p]$. This is precisely the definition used by the authors
for $\kappa (x)=2$ when $\epsilon (x)=1+\omega (x)$ in \cite{CoP2} {\bf I.1}(ii) on p.1899.

\smallskip

Suppose now that $\kappa (x)=2$, $x$ satisfies condition (*3) and $(u_1,u_2,u_3;Z)$ are well adapted coordinates
satisfying the requirements in Definition \ref{*kappadeux}. It also follows from the above proof that $x$ is good or
$$
H^{-1}{\partial F_{p,Z} \over \partial U_2}=
\left\{
  \begin{array}{ccc}
    <U_3^{\omega (x)}> \hfill{}& \mathrm{if} & \tau '(x)=1 \hfill{}\\
     &  &  \\
    <F_{\omega (x)} (U_1,U_3)> & \mathrm{if} & \mathrm{Vdir}(x)=<U_1,U_3> \\
  \end{array}
\right.
.
$$
In particular, (\ref{eq7236}) holds in both cases with $<\Phi >=<U_3^{\omega (x)}> $. We deduce the following:
there exists $r\geq 0$ such that $x_r$ is resolved or for every $r>>0$, we have ($\iota (x_r)=\iota (x)$,
$x_r$ satisfies condition (*)) and
$$
x_r \ \mathrm{satisfies} \ \mathrm{condition} \ (*3) \Longrightarrow (\ref{eq7236}) \ \mathrm{holds} \
\mathrm{at} \ x_r .
$$
Namely, otherwise we would have ($\iota (x_r)=\iota (x)$, $\tau ' (x_r)=2$ and $x_r$ is of type (T1))
for every $r>>0$ by the above. But this implies that $x_r$ is resolved for some $r\geq 0$ ({\it viz.}
proof of Proposition \ref{redto*}(i) for $\tau '(x)=2$).

\smallskip

This matches the present definition of $\kappa (x)=2$ with that used in \cite{CoP2}, and reduces
the proof to the same situation (\ref{eq7236}).
\end{rem}

\subsection{Monic expansions: secondary invariants.}

Proposition \ref{redto*}(i) has reduced the proof of the Projection Theorem to those points with
$\kappa (x)=2$ satisfying condition (*). Moreover, we may assume that $\omega (x)\equiv 0 \ \mathrm{mod}p$
by Corollary \ref{omeganoncongp}. For such points, we introduce a new invariant $\gamma (x)\in \N$ in
Definition \ref{definvariants2}.

\smallskip

{\it We assume in this section and in the following one
that $\omega (x)\equiv 0 \ \mathrm{mod}p$ and $x$ satisfies condition (*). }

\smallskip

Let $(u_1,u_2,u_3;Z)$ be well adapted coordinates satisfying the condition in Definition \ref{*kappadeux}.
If $x$ satisfies condition (*1) or (*2) (resp. condition (*3)), then
\begin{equation}\label{eq730}
\mathbf{v}_0:= (\mathbf{b}_0,{\omega (x)\over p}),
\ \mathbf{b}_0:=(d_1,d_2) \ (\mathrm{resp.} \ \mathbf{b}_0:= (d_1,{1 \over p}))
\end{equation}
is a vertex of $\Delta_S (h;u_1,u_2,u_3;Z)$. Consider the projection from the point $\mathbf{v}_0$:
$$
\mathbf{p}'_2 : \R^3 \backslash \{x_3=\omega (x)/ p\}\longrightarrow \A :=\mathbf{b}_0 +\{(x_1,x_2,0),\   {(x_1,x_2)} \in \R^2\}.
$$
We view here $\A$ as an {\it affine} plane with origin $\mathbf{b}_0$ and coordinates $(x_1,x_2)$. Of course
$\A$ as a set is independent of our choice of $ \mathbf{b}_0$. Let $\mathbf{p}_2:=\tau \circ \mathbf{p}'_2$,
where
$$
\tau : \A \longrightarrow \A , \ \mathbf{b}_0+(y_1,y_2) \mapsto
\mathbf{b}_0+{1 \over {\omega (x) \over p}}(y_1,y_2).
$$
Analytically, we have:
\begin{equation}\label{eq7302}
\mathbf{p}_2 : \ (x_1,x_2,x_3) \mapsto \mathbf{b}_0 + {(x_1,x_2) -\mathbf{b}_0 \over {\omega (x)\over p}-x_3}.
\end{equation}

{\it From now on, we will use affine coordinates in $\A$, i.e. $(y_1,y_2)\in \R^2$ represents the point
$\mathbf{b}_0+(y_1,y_2)\in \A$. }

\smallskip

In explicit terms, when a monomial, say $u_1^{pd_1}u_2^{pd_2}u_3^{\omega(x)-i}u_1^{a_1} u_2^{a_2}$, $i>0$,
defines the vertex
$$
\mathbf{x}=(d_1+{a_1\over p},d_2+{a_2 \over p},{\omega(x)-i \over p}) \in \Delta_S(h;u_1,u_2,u_3;Z),
$$
we have
$$
\mathbf{p}_2(\mathbf{x})=({a_1\over i},{a_2\over i}) \ (\mathrm{resp.} \ \mathbf{p}_2(\mathbf{x})=({a_1\over i},{a_2-1\over i}))
$$
in cases (*1)(*2) (resp. in case (*3)).

\begin{defn}\label{defDelta2}\index{$\Delta_2$2@$\Delta_2$ in case $\kappa(x)=2$(*)}
With notations as above, we define a convex set:
$$
\Delta_2(h;u_1,u_2;u_3;Z):=
\mathbf{p}_2 \left ( \Delta_S (h;u_1,u_2,u_3;Z)\cap \{0\leq x_3<{\omega (x)\over p}\}\right )\subseteq \A .
$$
Let furthermore
\begin{equation}\label{eq7301}
    \left\{
      \begin{array}{ccc}
        B(h;u_1,u_2;u_3;Z) & := & \mathrm{inf}_{\mathbf{y}\in \Delta_2(h;u_1,u_2;u_3;Z)}\{y_1+y_2\} \geq 1 \hfill{}\\
          &   &   \\
        \beta_2(h;u_1,u_2;u_3;Z) & := & \sup_{\left\{
                                                \begin{array}{c}
                                                  \mathbf{y}\in \Delta_2(h;u_1,u_2;u_3;Z) \hfill{}\\
                                                  y_1 +y_2=B(h;u_1,u_2;u_3;Z) \\
                                                \end{array}
                                              \right.
        }\{y_2 \} \\
      \end{array}
    \right.
    .
\end{equation}
\end{defn}

Indeed, $\Delta_2(h;u_1,u_2;u_3;Z)$ is a convex set because  the set
$$
\Delta_S (h;u_1,u_2,u_3;Z)\cap \{0\leq x_3<{\omega (x)\over p}\}
$$
is convex. Note that $\Delta_2(h;u_1,u_2;u_3;Z)$ will have in general points with negative ordinate when (*3) holds.
We now prove some basic properties of $\Delta_2(h;u_1,u_2;u_3;Z)$.
The situation is different and somewhat simpler when (*1) or (*2) holds.

\begin{lem}\label{structDelta2}
With notations as above, the following holds:
\begin{itemize}
  \item [(1)] there exists $\mathbf{a}=(a_1,a_2,a_3)\in \Delta_S (h;u_1,u_2,u_3;Z)\cap \{0\leq x_3<{\omega (x)\over p}\}$ such that $\mathbf{p}_2(\mathbf{a})=:(\alpha_2,\beta_2)\in \Delta_2(h;u_1,u_2;u_3;Z)$ satisfies
  $$
  \beta_2=\beta_2(h;u_1,u_2;u_3;Z), \ \alpha_2 +\beta_2=B(h;u_1,u_2;u_3;Z).
  $$
  \item [(2)] if $x$ satisfies condition (*1) or (*2), then $\Delta_2(h;u_1,u_2;u_3;Z)$ is a (non\-empty)
  rational polygon.
  \item [(3)] if $x$ satisfies condition (*3), then $\Delta_2(h;u_1,u_2;u_3;Z)\cap\{y_2 \geq \beta_2\}$ is a (nonempty)
  rational polygon.
  \item [(4)] assume that $x$ satisfies condition (*1) or (*2) (resp. condition (*3)). Let
  $$
  \sigma_2 \subset \Delta_2(h;u_1,u_2;u_3;Z) \  (\mathrm{resp.}\sigma_2 \subset
  \Delta_2(h;u_1,u_2;u_3;Z)\cap\{y_2 \geq \beta_2\})
  $$
  be a compact face. The topological closure $\sigma$ of
\begin{equation}\label{eq731}
\sigma^{\circ} :=\mathbf{p}^{-1}_2 (\sigma_2)\cap \Delta_S (h;u_1,u_2,u_3;Z)
\cap \{0\leq x_3<{\omega (x)\over p}\}
\end{equation}
is a compact face of the polyhedron $\Delta_S (h;u_1,u_2,u_3;Z)$ (so $\sigma =\sigma_\alpha$
for some weight vector $\alpha \in \R^3_{>0}$, {\it viz.} Definition \ref{definh}). Moreover
$\mathbf{p}_2(\sigma^\circ)=\sigma_2$ and
\begin{equation}\label{eq7313}
    \sigma = \sigma^\circ \cup \{\mathbf{v}_0\}.
\end{equation}
  \item [(5)] assume that $x$ satisfies condition (*3) and let
  $$
  \sigma_{2,\mathrm{in}} := \Delta_2(h;u_1,u_2;u_3;Z)\cap\{y_1+y_2 = B(h;u_1,u_2;u_3;Z)\}.
  $$
  If $B(h;u_1,u_2;u_3;Z)>1$, statement (4) extends to $\sigma_2=\sigma_{2,\mathrm{in}}$, with (\ref{eq7313})
  possibly replaced by
  $$
  \sigma = \mathrm{Conv}\left (\sigma^\circ \cup \{\mathbf{v}_0\}\cup \{({1 \over p}, 0 ,{\omega (x)\over p})\}\right ).
  $$
  If $B(h;u_1,u_2;u_3;Z)=1$, then
  $$
  \sigma_\mathrm{in}:=\{\mathbf{x}\in \Delta_S (h;u_1,u_2,u_3;Z) : x_1+x_2+x_3=\delta (x)\}
  $$
  is the unique compact face $\sigma$ of $\Delta_S (h;u_1,u_2,u_3;Z)$ such that
  $$
  \mathbf{p}_2\left (\sigma \cap \{0\leq x_3<{\omega (x)\over p}\}\right )=\sigma_2.
  $$
\end{itemize}
\end{lem}

\begin{proof}
Let $\mathbf{V}$ be the set of all vertices of $\Delta_S (h;u_1,u_2,u_3;Z)$ and
$$
\mathbf{V}_{-}:=\mathbf{V}\cap \{0\leq x_3<{\omega (x)\over p}\}.
$$
We claim that $\mathbf{V}_{-}\neq \emptyset$: in other terms, $\Delta_2(h;u_1,u_2,u_3;Z)$ is not empty.
Namely, suppose that $\mathbf{V}_{-}= \emptyset$. By definition, this means that
$$
\mathrm{ord}_{(u_3)}f_{i,Z}\geq i{\omega (x) \over p}, \ 1 \leq i \leq p.
$$
Since $\omega (x)/p\geq 1$, we deduce that ${\cal Y}:=V(Z,u_3)\subset \mathrm{Sing}_p{\cal X}$ by
Proposition \ref{deltainv}: a contradiction with assumption {\bf (E)}.\\

In order to prove the lemma, we must understand the limit points
$\mathbf{p}_2(\mathbf{x})\in \Delta_2(h;u_1,u_2;u_3;Z)$ when $\mathbf{x}\in \Delta_S (h;u_1,u_2,u_3;Z)$
tends to the hyperplane $\{x_3=\omega(x)/p\}$. By convexity, we have
$$
\mathbf{x}\in \mathrm{Conv}\left ( \bigcup_{\mathbf{v}\in \mathbf{V}}\{\mathbf{v}+\R^3_{\geq 0}\}\right ).
$$

$\bullet$ {\it Assume that $x$ satisfies condition (*1) or (*2).} Let $\mathbf{v} \in \mathbf{V} \backslash \mathbf{V}_{-}$.
Since $v_j \geq d_j$, $j=1,2$,
and $v_3 \geq \omega (x)/p$, we have $\mathbf{v}=\mathbf{v}_0$. One deduces immediately that
$$
\Delta_2(h;u_1,u_2;u_3;Z)=\mathrm{Conv} \left ( \{\mathbf{p}_2(\mathbf{v})+ {\R}^2_{\geq 0},
\  \mathbf{v} \in \mathbf{V}_{-}\} \right ).
$$
All statements in the lemma follow easily. \\

$\bullet$ {\it Assume that $x$ satisfies condition (*3).} Let $\mathbf{a}=(a_1,a_2,a_3)\in \mathbf{V}_{-}$
be chosen in such a way that
\begin{equation}\label{eq7311}
(\alpha_2+\beta_2,-\beta_2):=\left ( {a_1+a_2-d_1-{1 \over p}\over {\omega (x)\over p}-a_3},
{-a_2+{1 \over p}\over {\omega (x)\over p}-a_3}\right )
\end{equation}
is minimal for the lexicographical ordering, {\it viz.} (\ref{eq7302}). We now prove (1).
Let $\mathbf{v} \in \mathbf{V}\backslash  \mathbf{V}_{-}$. Since $v_3>0$,
Theorem \ref{initform} implies that
$$
\mathrm{in}_{\mathbf{v}}h =Z^p +\lambda U^{p\mathbf{v}}, \ \lambda \neq 0.
$$
If $\mathbf{v}\neq \mathbf{v}_0$, we therefore have
\begin{equation}\label{eq7312}
v_3 \geq {1+\omega (x) \over p} \ \mathrm{or} \ \mathbf{v}=\mathbf{v}_k:=(d_1+{k \over p}, 0 , {\omega (x)\over p}) \
\mathrm{for} \ \mathrm{some} \ k\geq 1.
\end{equation}
Let
$$
\alpha:=({\omega (x)\over p},{\omega (x)\over p}, \alpha_2+\beta_2)\in \R_{>0}^3,
\ L_\alpha (x_1,x_2,x_3):=x_1+x_2 + (\alpha_2+\beta_2)x_3.
$$
By (\ref{eq7311})-(\ref{eq7312}), we have
\begin{equation}\label{eq7314}
    \left\{
  \begin{array}{cc}
    L_\alpha (\mathbf{v}_0)= L_\alpha (\mathbf{b}_0)+{\omega (x)\over p}(\alpha_2+\beta_2)=L_\alpha (\mathbf{a})  \\
    L_\alpha (\mathbf{v})\geq L_\alpha (\mathbf{v}_0) \ \mathrm{if} \  \mathbf{v} \not \in \{\mathbf{v}_0,\mathbf{a}\} \hfill{}\\
  \end{array}
\right.
.
\end{equation}
This shows that $\mathbf{v}_0,\mathbf{a} \in \sigma_{\alpha}$, where $\sigma_{\alpha}$ is the compact face of
the polyhedron $\Delta_S (h;u_1,u_2,u_3;Z)$ defined by $\alpha$. In particular we have proved that
$$
\alpha_2+\beta_2 =B(h;u_1,u_2;u_3;Z).
$$
Similarly, let
$$
\alpha ':=({\omega (x)\over p}\alpha'_1,{\omega (x)\over p}, \alpha'_1 \alpha_2+\beta_2)\in \R_{>0}^3,
$$
where $\alpha'_1>1$ is chosen in such a way that $L_{\alpha '}(\mathbf{v})>L_{\alpha '}(\mathbf{a})$
for every $\mathbf{v}\in \mathbf{V}_{-}$. Such $\alpha'_1>1$ exists thanks to the minimal property
in (\ref{eq7311}). We now have
$$
\left\{
  \begin{array}{cc}
    L_{\alpha'} (\mathbf{v}_0)= L_{\alpha'} (\mathbf{b}_0)+{\omega (x)\over p}(\alpha'_1\alpha_2+\beta_2)=L_{\alpha'} (\mathbf{a})  \\
    L_{\alpha'} (\mathbf{v})> L_{\alpha'} (\mathbf{v}_0) \ \mathrm{if} \  \mathbf{v} \not \in \{\mathbf{v}_0,\mathbf{a}\} \hfill{}\\
  \end{array}
\right.
$$
and this proves that the line $(\mathbf{v}_0\mathbf{a})$ meets $\Delta_S (h;u_1,u_2,u_3;Z)$ along an edge.
This completes the proof of (1), and of (4) when $\sigma_2=\{(\alpha_2,\beta_2)\}$.

\smallskip

Statement (4) is proved along the same lines for arbitrary
$$
\sigma_2\subseteq \Delta_2(h;u_1,u_2;u_3;Z)\cap\{y_2 \geq \beta_2\}
$$
and we omit the proof. Then (3) is a consequence of (4) because $\mathbf{V}_{-}$ is a finite set.

\smallskip

To prove (5) when $B(h;u_1,u_2;u_3;Z)>1$, note that equality possibly holds in (\ref{eq7314})
only if $\mathbf{v}=\mathbf{v}_1$ and the conclusion follows.

If $B(h;u_1,u_2;u_3;Z)=1$, we have $\alpha =(1,1,1)$ with notations as above and
$\sigma_\mathrm{in}$ is the compact face of $\Delta_S (h;u_1,u_2,u_3;Z)$
generated by $\sigma^\circ$.
\end{proof}

\begin{cor}\label{Delta2+}
With notations as above, let:
$$
\Delta_2^+(h;u_1,u_2;u_3;Z)\index{$\Delta_2^+$ @ $\Delta_2^+(h;u_1,u_2;u_3;Z)$ in case $\kappa(x)=2$(*), Corollary~\ref{Delta2+}}:=\Delta_2(h;u_1,u_2;u_3;Z)\cap \{y_2 \geq \beta_2(h;u_1,u_2;u_3;Z)\}.
$$
Then $\Delta_2^+(h;u_1,u_2;u_3;Z)=\mathrm{Conv} \left ( \{\mathbf{p}_2(\mathbf{x})+ {\R}^2_{\geq 0},
\  \mathbf{x} \in \mathbf{S}\} \right )$,
where $\mathbf{S}$ is the set of vertices $\mathbf{x} \in \Delta_S (h;u_1,u_2,u_3;Z)$ with
$$
0 \leq x_3 <{\omega (x)\over p} \ \mathrm{and} \  y_2:=(\mathbf{p}_2(\mathbf{x}))_2 \geq \beta_2(h;u_1,u_2;u_3;Z).
$$
\end{cor}

Taking $\sigma =\sigma_\alpha$ as in Lemma \ref{structDelta2}(4) or (5), we deduce from Theorem \ref{initform} that:
$$
\mathrm{in}_\alpha h =Z^p +F_{p-1,Z, \alpha}Z +F_{p,Z, \alpha}\in \mathrm{gr}_\alpha S [Z].
$$
Moreover, $F_{p-1,Z, \alpha}\neq 0$ implies that $F_{p-1,Z, \alpha}=-G_\alpha^{p-1}$ and
$$
\mathrm{cl}_{p(p-1)\delta_\alpha}(\mathrm{Disc}_Z(h))=<G^{p(p-1)}>.
$$

In order to associate relevant combinatorial data to $\Delta_2(h;u_1,u_2;u_3;Z)$, some minimizing process
on the $u_3$ coordinate is required. This process is similar to that used in Definition \ref{defsolvable} and
Proposition \ref{Deltamin}.\\

\begin{defn}\label{def2solvable}
Let $x$ satisfy condition (*), $(u_1,u_2,u_3;Z)$ be well adapted coordinates
at $x$ satisfying Definition  \ref{*kappadeux} and $\mathbf{y}=(y_1,y_2)\in \R^2$
be a vertex of $\Delta_2(h;u_1,u_2;u_3;Z)$ (of $\Delta_2^+(h;u_1,u_2;u_3;Z)$ in case (*3)).

\smallskip

With notations as in Lemma \ref{structDelta2}(4) with $\sigma_2=\{\mathbf{y}\}$,
we say that $\mathbf{y}$ is 2-solvable if $\mathbf{y}\in \N^2$ and
$$
\left\{
  \begin{array}{cc}
    \mathrm{in}_\alpha h=Z^p +\lambda U_1^{pd_1}U_2^{pd_2}(U_3 -c U_1^{y_1}U_2^{y_2})^{\omega (x)} +\Phi^p
    & \mathrm{in} \ \mathrm{cases} \ (*1) \ \mathrm{or} \ (*2) \\
      &  \\
    \mathrm{in}_\alpha h=Z^p +\lambda U_1^{pd_1}U_2(U_3 -c U_1^{y_1}U_2^{y_2})^{\omega (x)} +\Phi^p \hfill{}
    & \mathrm{in} \ \mathrm{case} \ (*3) \hfill{}\\
  \end{array}
\right.
$$
where $\Phi \in \mathrm{gr}_\alpha S$ and $\lambda ,c \in k(x)$.

\smallskip

We say that $(u_1,u_2;u_3;Z)$ are well 2-adapted if furthermore the polygon
$\Delta_2(h;u_1,u_2;u_3;Z)$  ($\Delta_2^+(h;u_1,u_2;u_3;Z)$ in case (*3)) has no 2-solvable vertex.
\end{defn}

\begin{thm}\label{well2prepared}
With notations as above, there  {exist} well 2-adapted coordinates. Furthermore, the polygon
$\Delta_2^+(h;u_1,u_2;u_3;Z)$ is independent of the well 2-adapted coordinates
$(u_1,u_2;u_3;Z)$. For such  $(u_1,u_2;u_3;Z)$, let
$$
A_1(x):=\min_{\mathbf{y}\in \Delta_2^+(h;u_1,u_2;u_3;Z)}\{y_1\}\geq 0;
$$
the curve ${\cal Y}:=V(Z,u_1,u_3)\subset {\cal X}$ satisfies the equivalence:
$$
A_1(x)\geq 1 \Leftrightarrow {\cal Y} \ \mathrm{is} \ \mathrm{permissible} \ (\mathrm{of} \ \mathrm{the}
\ \mathrm{first} \ \mathrm{or} \ \mathrm{second} \ \mathrm{kind}).
$$
\end{thm}

\begin{proof}
Let $(u_1,u_2,u_3;Z)$ be well adapted coordinates and assume on the contrary
that $(u_1,u_2;u_3;Z)$ are not well 2-adapted. Let $\mathbf{y}\in \N^2$
be a 2-solvable vertex of $\Delta_2(h;u_1,u_2;u_3;Z)$ with $y_1+y_2$ minimal
(and $y_2 \geq \beta_2(h;u_1,u_2;u_3;Z)$ if $x$
satisfies condition (*3)). Let $\gamma \in S$ be a preimage of $c\in k(x)$ given by
Definition \ref{def2solvable}. Since $\mathbf{y}$ is a vertex of $\Delta_2(h;u_1,u_2;u_3;Z)$, we have
$c\neq 0$, so $\gamma$ is a unit. We let $u'_3:=u_3 -\gamma u_1^{y_1}u_2^{y_2}$.
Let $\alpha \in \R^3_{>0}$ define the edge
$$
\sigma :=\mathbf{p}^{-1}_2 (\mathbf{y})\cap \Delta_S (h;u_1,u_2,u_3;Z)
\cap \{0\leq x_3<{\omega (x)\over p}\} \cup \{\mathbf{v}_0\}
$$
of $\Delta_S (h;u_1,u_2,u_3;Z)$. Computing now initial forms for the polyhedron $\Delta_S (h;u_1,u_2,u'_3;Z)$,
we obtain
\begin{equation}\label{eq732}
\left\{
  \begin{array}{cc}
    \mathrm{in}_\alpha h=Z^p +\lambda U_1^{pd_1}U_2^{pd_2}{U'_3}^{\omega (x)} +\Phi^p & \mathrm{in} \ \mathrm{cases} \ (*1) \
    \mathrm{or} \ (*2) \\
      &  \\
    \mathrm{in}_\alpha h=Z^p +\lambda U_1^{pd_1}U_2{U'_3}^{\omega (x)} +\Phi^p \hfill{}
    & \mathrm{in} \ \mathrm{cases} \ (*3) \hfill{}\\
  \end{array}
\right.
\end{equation}
with notations as in Definition \ref{def2solvable}.

\smallskip

Let now $\mathbf{y}'\neq \mathbf{y}$ be a vertex of $\Delta_2(h;u_1,u_2;u_3;Z)$ (of $\Delta_2^+(h;u_1,u_2;u_3;Z)$
if $x$ satisfies condition (*3)). Let $\alpha '\in \R^3_{>0}$ define the corresponding edge
$$
\sigma ':=\mathbf{p}^{-1}_2 (\mathbf{y}')\cap \Delta_S (h;u_1,u_2,u_3;Z)
\cap \{0\leq x_3<{\omega (x)\over p}\} \cup \{\mathbf{v}_0\}
$$
given by Lemma \ref{structDelta2}(4). In particular we have
$$
\mu_{\alpha '} (u_1^{y_1}u_2^{y_2})>\mu_{\alpha '} (u_3).
$$
This implies that $\mathrm{in}_{\alpha '}h$ is unchanged when computed in
$\Delta_S (h;u_1,u_2,u_3;Z)$ or in $\Delta_S (h;u_1,u_2,u'_3;Z)$, i.e. obtained by substituting the
variable $U_3$ by the variable $U'_3$. Therefore $\sigma '$ is again an edge of $\Delta_S (h;u_1,u_2,u'_3;Z)$.

\smallskip

If $x$ satisfies condition (*1) or (*2), we deduce that
$$
\mathbf{p}_2(\Delta_S (h;u_1,u_2,u'_3;Z)\cap \{0\leq x_3<{\omega (x)\over p}\})
\subseteq \Delta_2(h;u_1,u_2;u_3;Z).
$$

If $x$ satisfies condition (*3), we obtain
$$
\mathbf{p}_2(\Delta_S (h;u_1,u_2,u'_3;Z)\cap \{0\leq x_3<{\omega (x)\over p}\})\cap \{y_2\geq \beta_2\}
\subseteq \Delta_2^+(h;u_1,u_2;u_3;Z).
$$

Let $(u_1,u_2,u'_3;Z')$ be well adapted coordinates, $Z':=Z-\phi$, $\phi \in S$. We first check that
$(u_1,u_2,u'_3;Z')$ satisfies Definition \ref{*kappadeux}, i.e. that $U'_3 \in \mathrm{Vdir}(x)$.
This is obvious if $y_1+y_2>1$, since $\mathrm{in}_{m_S}h$ is then unchanged.
If $y_1+y_2=1$, then $\mathbf{y}\in \{(1,0), (0,1)\}$ because 2-solvable vertices have integer coordinates.
By Definition \ref{def2solvable} and  Definition \ref{*kappadeux}, we have
$$
U_3-cU_1 \in \mathrm{Vdir}(x)+<U_2> \  (\mathrm{resp.}  \ U_3-cU_2 \in \mathrm{Vdir}(x)+<U_1>)
$$
if $\mathbf{y}=(1,0)$ (resp. if $\mathbf{y}=(0,1)$). Therefore $\tau '(x)=3$ or
$$
\mathrm{Vdir}(x)=<U_3, U_1+dU_2>  \  (\mathrm{resp.}  \ \mathrm{Vdir}(x)=<U_3, U_2+dU_1>)
$$
for some
$d \in k(x)$. In all cases, $U'_3 \in \mathrm{Vdir}(x)$ follows from the invariance of $\mathrm{Vdir}(x)$
(Definition \ref{deftauprime}) if $\tau '(x)=3$ or if $d=0$, or if ($d\neq 0$ and $x$ satisfies condition (*2)).
Otherwise, it can be assumed w.l.o.g. that $d=0$ by substituting $u_2$ by $u'_2=u_2 +\delta u_1$, where
$\delta \in S$ is a preimage of $d \in k(x)$. Note that this substitution does not change the requirements
in Definition \ref{*kappadeux} and we thus get $U'_3\in \mathrm{Vdir}(x)$ as required.

\smallskip

By (\ref{eq732}), we now have
$$
\left\{
  \begin{array}{cc}
    \mathbf{v}_0\in \Delta_S (h;u_1,u_2,u'_3;Z')\subset \Delta_S (h;u_1,u_2,u_3;Z)\\
      \\
    \mathbf{y} \not \in \Delta_2 (h;u_1,u_2;u'_3;Z')\hfill{}\\
  \end{array}
\right.
.
$$

Iterating this construction, we deduce that there exists a sequence (finite or infinite) of 2-solvable
vertices $(\mathbf{y}^{(i)})_{i\geq 0}$,
$\mathbf{y}^{(0)}:=\mathbf{y}$ and corresponding well adapted coordinates $(u_1,u_2,u_3^{(i)};Z^{(i)})$,
$Z^{(i)}:=Z^{(i-1)}-\phi^{(i-1)}$, $\phi^{(i-1)} \in S$ such that
\begin{equation}\label{eq7337}
\left\{
  \begin{array}{cc}
    \mathbf{v}_0\in \Delta_S (h;u_1,u_2,u_3^{(i)};Z^{(i)})\subset \Delta_S (h;u_1,u_2,u_3^{(i-1)};Z^{(i-1)})\\
      \\
    \mathbf{y}^{(i)} \not \in \Delta_2 (h;u_1,u_2;u_3^{(i-1)};Z^{(i-1)})\hfill{}\\
  \end{array}
\right.
\end{equation}
for $i\geq 1$. Since $y_1^{(i)}+y_2^{(i)}$ is chosen to be minimal at each step, we have
$y_1^{(i)}+y_2^{(i)}\rightarrow +\infty$ as $i\rightarrow +\infty$ if the process is infinite. Therefore
$$
\hat{u}_3=\lim_i u_3^{(i)}\in \hat{S}, \ \hat{Z}:=Z-\hat{\phi}, \ \hat{\phi}:=\sum_i\phi^{(i-1)}\in \hat{S}
$$
exist and $(u_1,u_2;\hat{u}_3;\hat{Z})$ are well 2-adapted coordinates of
$\hat{{\cal X}}=\mathrm{Spec}(\hat{S}[X]/(h))$. This proves the existence of well 2-adapted coordinates
when $S=\hat{S}$.\\

Let now $(u_1,u_2;u_3;Z)$ and $(u'_1,u'_2;u'_3;Z')$ be two sets of well 2-adapted coordinates. We assume of course that
$\mathrm{div}(u_j)=\mathrm{div}(u'_j)$, $j=1,2$, in case (*2). To prove that
$\Delta_2^+ (h;u'_1,u'_2;u'_3;Z')=\Delta_2^+ (h;u_1,u_2;u_3;Z)$, let first $\mathbf{y}\in \Delta_2^+ (h;u_1,u_2;u_3;Z)$
and let $\alpha \in \R_{>0}^3$ be given by Lemma \ref{structDelta2}(4) w.r.t. the face $\sigma_2:=\mathbf{y}$.
Since $\mathbf{y}\in \Delta_2^+ (h;u_1,u_2;u_3;Z)$, we have
$$
\mu_\alpha (u_2)<\min\{\mu_\alpha (u_1), \mu_\alpha (u_3)\}.
$$
Therefore $\mu_\alpha (u'_2)=\mu_\alpha (u_2)$.
We deduce that $\mathrm{in}_\alpha h$ is unchanged when computed w.r.t. the coordinates $(u'_1,u'_2;u_3;Z)$. This
implies furthermore that $\mathbf{y}$ is not 2-solvable in $\Delta_2 (h;u'_1,u'_2;u'_3;Z')$ provided
$\mu_\alpha (u'_3)=\mu_\alpha (u_3)$ for every $\alpha =\alpha (\mathbf{y})$. Otherwise, there is an expansion
$$
u'_3=\delta u_3 +\sum_{\mathbf{x}\in \Sigma}\gamma (\mathbf{x}) u_1^{x_1}u_2^{x_2},
$$
with $\Sigma$ finite, $\delta , \gamma (\mathbf{x})\in S$ units and $\mu_\alpha (u_1^{x_1}u_2^{x_2})<\mu_\alpha (u_3)$
for some $\mathbf{x}=\mathbf{x}_0\in \Sigma$ and $\alpha$. One deduces that $(\mathbf{v}_0\mathbf{x}_0)$ supports an
edge of $\Delta_2 (h;u'_1,u'_2;u'_3;Z')$ and that $1 /{\omega (x) \over p}\mathbf{x}_0$ is a 2-solvable vertex of
$\Delta_2 (h;u'_1,u'_2;u'_3;Z')$. Choosing $\mathbf{x}_0$ with $x_1$ minimal gives
$$
{1 \over {\omega (x) \over p}}\mathbf{x}_0 \in \Delta_2^+ (h;u'_1,u'_2;u'_3;Z').
$$
This is a contradiction since $(u'_1,u'_2;u'_3;Z')$ are well 2-adapted coordinates, so we get
$$
\Delta_2^+ (h;u'_1,u'_2;u'_3;Z')=\Delta_2^+ (h;u_1,u_2;u_3;Z)
$$
as required.\\

Let now $(u_1,u_2,u_3;Z)$ be well adapted coordinates
at $x$ satisfying Definition  \ref{*kappadeux}. Applying finitely many times
the above algorithm and (\ref{eq7337}), as $\Delta_2(h;u_1,u_2,u_3;Z)\neq \emptyset$ by Lemma \ref{structDelta2}(3),
it can be assumed w.l.o.g. that
$$\left\{
    \begin{array}{ccc}
      \alpha_2(h;u_1,u_2;u_3;Z) & = & \alpha_2(h;u_1,u_2;\hat{u}_3;\hat{Z}) \\
      \beta_2(h;u_1,u_2;u_3;Z)  & = & \beta_2(h;u_1,u_2;\hat{u}_3;\hat{Z}) \\
    \end{array}
  \right.
,
$$
where $(u_1,u_2;\hat{u}_3;\hat{Z})$ are well 2-adapted coordinates of
$\hat{{\cal X}}=\mathrm{Spec}(\hat{S}[X]/(h))$, $\hat{Z}=Z-\hat{\phi}$. Moreover,
$$
(\alpha_2, \beta_2):=(\alpha_2(h;u_1,u_2;u_3;Z), \beta_2(h;u_1,u_2;u_3;Z))
$$
is a vertex of both $\Delta_2^+(h;u_1,u_2;u_3;Z)$ and $\Delta_2^+(h;u_1,u_2;\hat{u}_3;\hat{Z})$. Let $\hat{x}$
be the closed point of $\hat{{\cal X}}$ and assume that
\begin{equation}\label{eq733}
    A_1 (\hat{x})>A_1:=\min_{\mathbf{y}\in \Delta_2^+(h;u_1,u_2;u_3;Z)}\{y_1\}.
\end{equation}
Let $J:=\{1,3\}$ and consider the weight vector $\alpha :=({\omega (x) \over p},A_1)\in \R_{>0}^J$. We
consider the initial form polynomial
$$
 \mathrm{in}_\alpha h=Z^p+\sum_{i=1}^p F_{i,Z,\alpha}Z^{p-i} \in (\mathrm{gr}_\alpha S)[Z],
$$
where
$$
\left\{
  \begin{array}{cc}
    \mathrm{gr}_\alpha S=S/(u_1)[U_1]\subseteq \mathrm{gr}_\alpha \hat{S}=\hat{S}/(u_1)[U_1] \hfill{}
    & \mathrm{if} \ A_1=0 \\
      &  \\
    \mathrm{gr}_\alpha S=S/(u_1,u_3)[U_1,U_3]\subseteq \mathrm{gr}_\alpha \hat{S}=\hat{S}/(u_1,u_3)[U_1,U_3] \hfill{}
    & \mathrm{if} \ A_1>0 \hfill{}\\
  \end{array}
\right.
.
$$

\noindent {\it Case 1:} $A_1=0$. One deduces from the above algorithm and (\ref{eq7337}) that there exists some
$\hat{c} \in (\overline{u}_2)\hat{S}/(u_1)$ such that
\begin{equation}\label{eq7331}
    F_{i,\hat{Z},\alpha}=\hat{g}_iU_1^{id_1}\overline{u}_2^{d_{2,i}}(\overline{u}_3 -\hat{c})^{i{\omega (x)\over p}}, \ 1 \leq i \leq p-1
\end{equation}
for some $\hat{g}_i \in \hat{S}/(u_1)$ ($\hat{g}_i =0$ if $d_1 \not \in \N$), $d_{2,i}\geq id_2$, and
\begin{equation}\label{eq7332}
\left\{
  \begin{array}{cc}
    F_{p,\hat{Z},\alpha}=\hat{l} U_1^{pd_1}\overline{u}_2^{pd_2}(\overline{u}_3 -\hat{c})^{\omega (x)}
    & \mathrm{in} \ \mathrm{cases} \ (*1) \ \mathrm{or} \ (*2) \\
      &  \\
    F_{p,\hat{Z},\alpha}=\hat{l} U_1^{pd_1}\overline{u}_2(\overline{u}_3 -\hat{c} )^{\omega (x)}  \hfill{}
    & \mathrm{in} \ \mathrm{case} \ (*3) \hfill{}\\
  \end{array}
\right.
\end{equation}
for some $\hat{l}\in \hat{S}/(u_1)$ a unit.

\smallskip

The regular local ring $T:=(\mathrm{gr}_\alpha S)_{(U_1,\overline{u}_2, \overline{u}_3)}$ is excellent and the polynomial
$\mathrm{in}_\alpha h \in T[Z]$ satisfies the assumptions of Proposition \ref{Deltaalg}.
Let
$$
\Xi:=\mathrm{Spec}(T[Z]/(\mathrm{in}_\alpha h)), \ \hat{\Xi}:=\mathrm{Spec}(\hat{T}[Z]/(\mathrm{in}_\alpha h)).
$$
Since $\mathbf{v}_0$ is a nonsolvable vertex of
$\Delta_{\hat{T}}(\mathrm{in}_\alpha h; U_1,\overline{u}_2, \overline{u}_3;Z)$,
we deduce from (\ref{eq7331})-(\ref{eq7332}) that
\begin{equation}\label{eq7333}
\hat{V}:=V(\hat{Z}, \overline{u}_3-\hat{c})\subseteq \mathrm{Sing}_p\widehat{\Xi} \subseteq
V(\hat{Z}, U_1\overline{u}_2^{pd_2}(\overline{u}_3-\hat{c})).
\end{equation}
Since $T$ is excellent, one deduces that the Zariski closure $V$ of $\hat{V}$ in $\Xi$
is contained in $\mathrm{Sing}_p\Xi$. Let
$$
P : \ \Xi \longrightarrow \mathrm{Spec}T
$$
be the projection. By (\ref{eq7333}), $P (V)$ is an irreducible component of $P (\mathrm{Sing}_p\Xi)$ contained
in $\mathrm{div}(U_1\overline{u}_2^{pd_2}(\overline{u}_3-\hat{c}))$. Since each of $\mathrm{div}(U_1)$,
$\mathrm{div}(\overline{u}_2)$ is Zariski closed, there exist $\hat{\delta}' \in \hat{S}/(u_1)$ a unit such
that $\overline{u}'_3:= \hat{\delta}'(\overline{u}_3-\hat{c})\in S/(u_1)$. Let $u'_3 \in S$ be a preimage
of $\overline{u}'_3$. Applying again Proposition \ref{Deltaalg},
there exist well adapted coordinates $(u_1,u_2,u'_3;Z')$ at $x$ satisfying Definition  \ref{*kappadeux} and such that
\begin{equation}\label{eq7334}
\min_{\mathbf{y}\in \Delta_2^+(h;u_1,u_2;u'_3;Z')}\{y_1\}> A_1.
\end{equation}

\noindent {\it Case 2:} $A_1>0$. The argument runs along the same lines: we now have
some $\hat{c} \in (\overline{u}_2)\hat{S}/(u_1,u_3)$, (\ref{eq7333}) is replaced by
$$
V(\hat{Z}, U_3-\hat{c}U_1^{A_1})\subseteq \mathrm{Sing}_p\widehat{\Xi} \subseteq
V(\hat{Z}, U_1\overline{u}_2^{pd_2}(U_3-\hat{c}U_1^{A_1})),
$$
with $\Xi$ as above and (\ref{eq7334}) holds.

\smallskip

Applying this procedure and (\ref{eq7334}) finitely many times, it can be assumed w.l.o.g. that $A_1=A_1(\hat{x})$.
When $x$ satisfies condition (*1) or (*2), one introduces similarly
$$
A_2:=\min_{\mathbf{y}\in \Delta_2(h;u_1,u_2;u_3;Z)}\{y_2\}
\leq \min_{\mathbf{y}\in \Delta_2^+(h;u_1,u_2;u_3;Z)}\{y_2\}=\beta_2(h;u_1,u_2;u_3;Z).
$$

The same argument shows that there  {exist} well adapted coordinates $(u_1,u_2,u_3;Z)$
at $x$ satisfying Definition  \ref{*kappadeux} and well 2-adapted coordinates $(u_1,u_2;\hat{u}_3;\hat{Z})$
of $\hat{{\cal X}}=\mathrm{Spec}(\hat{S}[X]/(h))$, $\hat{Z}=Z-\hat{\phi}$, such that
\begin{equation}\label{eq7335}
A_j:=\min_{\mathbf{y}\in \Delta_2(h;u_1,u_2;u_3;Z)}\{y_j\}
= \min_{\mathbf{y}\in \Delta_2^+(h;u_1,u_2;\hat{u}_3;\hat{Z})}\{y_j\}, \ j=1,2.
\end{equation}

Finally, if $x$ satisfies condition (*1) or (*2) (resp. (*3)), (\ref{eq7335}) (resp. (\ref{eq7334}))
proves that the region
$$
\Delta_2(h;u_1,u_2;u_3;Z)\backslash \Delta_2(h;u_1,u_2;\hat{u}_3;\hat{Z})\subseteq \R^2_{\geq 0}
$$
(resp. $\Delta_2^+(h;u_1,u_2;u_3;Z)\backslash \Delta_2^+(h;u_1,u_2;\hat{u}_3;\hat{Z})$)
is bounded. Therefore the above algorithm and (\ref{eq7337})
can repeat only finitely many times. This proves the existence of  well 2-adapted coordinates for
arbitrary $S$.\\

Let then $(u_1,u_2;u_3;Z)$ be well 2-adapted coordinates and define the curve
${\cal Y}:=V(Z,u_1,u_3)\subset {\cal X}$. By Proposition \ref{Deltaalg}, the polyhedron
$$
\Delta_{\hat{S}}(h;u_1,u_3;Z)=\mathrm{pr}^{\{1,3\}}\Delta_{\hat{S}}(h;u_1,u_2,u_3;Z)
$$
is minimal and we have
\begin{equation}\label{eq7336}
\epsilon (y)=\omega (x) \times \min \{1, A_1(x)\}.
\end{equation}
By Definition \ref{deffirstkind}, ${\cal Y}$ is permissible of the first kind at $x$ if and only if
($x$ satisfies condition (*1) or (*2)) and $A_1(x)\geq 1$.

\smallskip

By Proposition \ref{secondkind}, ${\cal Y}$ is permissible of the second kind at $x$ only if
$x$ satisfies condition (*3) and $A_1(x)\geq 1$ by (\ref{eq7336}). Conversely,
Definition \ref{defsecondkind}(i) is satisfied because
$$
m(y)\geq \epsilon (y) =\omega (x)\geq p.
$$

By (\ref{eq7336}), we have $\epsilon (y)=\epsilon (x)-1$. Suppose that $i_0(y)=p-1$. Let
$W:=\eta ({\cal Y})$, so we have
$$
\mathrm{in}_{W}h =Z^p -G_W^{p-1}Z+F_{p,W,Z} \in G(W)[Z]
$$
with $\delta (y)\in \N$, $G_W =g_WU_1^{\delta (y)}$ and
$$
0\neq \mathrm{cl}_{p(p-1)\delta(y)}\mathrm{Disc}_Zh=<g_W^{p(p-1)}U_1^{p(p-1)\delta(y)}>\in G(W)_{p(p-1)\delta(y)}
$$
by Theorem \ref{initform}. Since $E=\mathrm{div}(u_1)$, $g_W\in S/(u_1,u_3)$ is a unit by assumption {\bf (E)}. We
then get
$$
\epsilon (x)\leq {\mathrm{ord}_{m_S}(H(x)^{-(p-1)}f_{p-1,Z}^p) \over p-1}=\epsilon (y)= \epsilon (x)-1,
$$
a contradiction. Therefore Definition \ref{defsecondkind}(ii) is satisfied because $i_0(y)=p$.
Finally it follows from Definition  \ref{*kappadeux}(ii)
that Definition \ref{defsecondkind}(iii) is satisfied.
\end{proof}

The previous theorem shows that the following invariants are actually independent of the
choice of well 2-adapted coordinates.

\begin{defn}\label{definvariants2}\index{$A$2@$A_1,\ A_2,\ B,\ C,\ \alpha_2,\ \beta_2, \gamma$ in case $\kappa(x)=2$(*), see also Definition~\ref{defDelta2} and Theorem~\ref{well2prepared}}
Let $x$ satisfy condition (*)  and $(u_1,u_2;u_3;Z)$ be well 2-adapted coordinates. We let
$$
A_j(x):=\min_{\mathbf{y}\in \Delta_2(h;u_1,u_2;u_3;Z)}\{y_j\}\geq 0 \ \ \ \mathrm{for} \ \mathrm{div}(u_j)\subseteq E;
\ \ \ \ \ \ \ \
$$
$$
B(x):=B(h;u_1,u_2;u_3;Z); \ C(x):=B(x)-\sum_{\mathrm{div}(u_j)\subseteq E}A_j(x);
$$
$$
\beta (x):=\min_{(A_1(x),y_2)\in \Delta_2^+(h;u_1,u_2;u_3;Z)}\{y_2 \}\geq 0 ;
\ \ \ \ \ \ \ \ \ \ \ \ \ \ \ \ \ \ \ \ \ \ \ \ \ \
$$
$$
(\alpha_2(x),\beta_2(x)):=(\alpha_2(h;u_1,u_2;u_3;Z),\beta_2(h;u_1,u_2;u_3;Z)). \ \ \ \
$$

Finally, we define $\gamma (x)\in \N$ by:
$$
\gamma (x):= \left\{
               \begin{array}{cc}
                 \lceil \beta (x)\rceil & \mathrm{in} \ \mathrm{case} \ (*1) \\
                 1+\lfloor C(x)\rfloor  & \mathrm{in} \ \mathrm{case} \ (*2) \\
                 1+ \lfloor \beta (x)\rfloor & \mathrm{in} \ \mathrm{case} \ (*3) \\
               \end{array}
             \right.
             .
$$
\end{defn}

\begin{lem}\label{kappa2bupcurve}
Assume that $\kappa (x)=2$ and $x$ satisfies condition (*). Let $(u_1,u_2;u_3;Z)$ be well 2-adapted
coordinates and assume furthermore that
$$
A_1(x)\geq 1 \  (\mathrm{resp.} \  A_1(x)>1 \ \mathrm{or} \  (A_1(x)=1 \ \mathrm{and} \  \beta (x)< 1- {1 \over \omega (x)}))
$$
if $x$ satisfies condition (*1) or (*2) (resp. condition (*3)).  Let $\pi : {\cal X}' \rightarrow {\cal X}$
be the blowing up along ${\cal Y}:=V(Z,u_1,u_3)\subset {\cal X}$ and $x' \in \pi^{-1}(x)$.
Then $x'$ is resolved or the following holds:
$$
x'=(Z':=Z/u_1, u_1, u_2,u'_3:=u_3/u_1)
$$
and $x'$ satisfies again condition (*1) or (*2) (resp. (*3));
the coordinates $(u_1,u_2;u'_3;Z')$ are well 2-adapted  at $x'$ and
$$
\left\{
\begin{array}{cc}
 \Delta_2(u_1,u_2;u'_3;Z')=\Delta_2(u_1,u_2;u_3;Z) - (1,0) \hfill{}& \mathrm{in} \ \mathrm{case} \ (*1) \ \mathrm{and} \ (*2) \\
 \Delta_2^+(u_1,u_2;u'_3;Z')=\Delta_2^+(u_1,u_2;u_3;Z) - (1,0) & \mathrm{in} \ \mathrm{case} \ (*3) \hfill{}\\
 \end{array}
\right.
             ;
$$
in particular $A_1(x')=A_1(x)-1$ and we have:
$$
A_2(x')=A_2(x), \ C(x')=C(x), \ \beta (x')=\beta (x) \ \mathrm{and} \ \gamma (x')=\gamma (x).
$$
\end{lem}

\begin{proof}
By Theorem \ref{well2prepared}, the curve ${\cal Y}$ is permissible since $A_1(x)\geq 1$.

\smallskip

Since $U_3\in\mathrm{Vdir}(x)$ by Definition
of well 2-prepared coordinates, $x$ is then good except possibly if
$\mathrm{Vdir}(x)=<U_3>$ by Theorem \ref{bupthm}; in this case, we have $x'=(Z/u_1, u_1,u_2,u_3/u_1)$.

Let $h':=u_1^{-p}h$. By Proposition \ref{originchart}, $\Delta_{\hat{S}'}(h';u_1,u_2,u'_3;Z')$ is again minimal.
With usual notations, we have $d'_1=d_1+{\omega (x)\over p} -1$ and $\mathbf{v}'_0:=(d'_1,0,\omega (x)/p)$
($\mathbf{v}'_0:=(d'_1,1/p,\omega (x)/p)$ in case (*3)) is a nonsolvable vertex. We may assume that $x'$ is very near $x$.

\smallskip

If $x$ satisfies condition (*1) (resp. (*2)), then $\kappa (x')=2$ and $x'$ satisfies again condition (*1) (resp. (*2)).

If $x$ satisfies condition (*3) and $\epsilon (x')=\epsilon (x)$, then $\kappa (x')=2$ and $x'$ satisfies again condition (*3).

If $x$ satisfies condition (*3) and $\epsilon (x')=\omega (x)$, then $x'$ satisfies the assumptions
of Lemma \ref{sortiemonome}, so $x$ is good if $A_1(x)>1$; if ($A_1(x)=1$ and $\beta (x)< 1 - 1/\omega (x)$),
then $\Delta_2^+(u_1,u_2;u_3;Z)$ has a vertex of the form
$$
\mathbf{y}_0=(1, (i_0-1)/i)=\beta (x) , \ 1 \leq i \leq \omega (x).
$$
Therefore $i_0\leq i< \omega (x)$ or $i_0<i=\omega (x)$, since $\beta (x)< 1 - 1/\omega (x)$.
Taking $i_0$ minimal with this property, $(d'_1,i_0/p,(\omega (x)-i)/p)$
is a vertex of $\Delta_{S'}(h';u_1,u_2,u'_3;Z')$ and therefore
$$
\epsilon (x')=\omega (x) \Longrightarrow  i_0 +\omega (x)-i = \omega (x).
$$
Therefore $i<\omega (x)$ since $(i_0,i)\neq (\omega (x),\omega (x))$;
then $x'$ is good by Proposition \ref{tauegaldeux}, so $x$ is good.

\smallskip

Let $\mathbf{y}=\mathbf{p}_2(\mathbf{v})$ be
a vertex of $\Delta_2(u_1,u_2;u_3;Z)$ (of $\Delta_2^+(u_1,u_2;u_3;Z)$ in case (*3)). With notations as in Lemma
\ref{structDelta2}(4) with $\sigma_2:=\{\mathbf{y}\}$, let
$$
\mathrm{in}_\alpha h =Z^p +U_1^{pd_1}U_2^{pd_2}(F_0(U_2)U_3^{\omega (x)} +\sum_{i=1}^{\omega (x)}F_i(U_1,U_2)U_3^{\omega (x)-i}),
$$
where $0\neq F_0(U_2)\in k(x)$ ($0\neq F_0(U_2)\in k(x)[U_2]_1$ in case (*3)). Then $\mathbf{y}':=\mathbf{y}-(1,0)$
is a vertex of $\Delta_2(u_1,u_2;u'_3;Z')$ (of $\Delta_2^+(u_1,u_2;u'_3;Z')$ in case (*3)); the corresponding
initial form in Lemma \ref{structDelta2}(4) with $\sigma_2:=\{\mathbf{y}'\}$ is of the form:
$$
\mathrm{in}_{\alpha '}h' ={Z'}^p +U_1^{pd'_1}U_2^{pd_2}(F_0(U_2){U'_3}^{\omega (x)}
+\sum_{i=1}^{\omega (x)}U_1^{-i}F_i(U_1,U_2){U'_3}^{\omega (x)-i}).
$$
It follows from Definition \ref{def2solvable} that $\mathbf{y}'$ is not 2-solvable, since $\mathbf{y}$ is not.
The lemma follows easily.
\end{proof}

\begin{prop}\label{kappa2gamma0}
Assume that $\kappa (x)=2$ and $x$ satisfies condition (*). If $\gamma (x)=0$, then
$x$ is good.
\end{prop}

\begin{proof}
By Theorem \ref{well2prepared}, there exist well 2-adapted coordinates
$(u_1,u_2;u_3;Z)$ at $x$. The assumption $\gamma (x)=0$ means that ($x$ is in case (*1) and
$\beta (x)=0$) or ($x$ is in case (*3) and $\beta (x)<0$).

\smallskip

{\it Assume that $x$ is in case (*1).} We have
$$
\Delta_2(h;u_1,u_2;u_3;Z)=(A_1(x),0)+\R^2_{\geq 0}.
$$
Since $B(x)\geq 1$ ({\it viz.} (\ref{eq7301})), we have $A_1(x)\geq 1$.

\smallskip

{\it Assume that $x$ is in case (*3).} We have
$$
\Delta_2^+(h;u_1,u_2;u_3;Z)=(A_1(x),\beta (x))+\R^2_{\geq 0}
$$
in this case. Note that we have $A_1(x)\geq 1$: namely, $\beta (x)=-1/i$ for some $i$, $1 \leq i \leq \omega (x)$ such
that
$$
\epsilon (x)=1 +\omega (x)\leq iA_1(x)+\omega (x)-i+1,
$$
so $A_1(x)\geq 1$.

\smallskip

Suppose that $1 \leq A_1(x) <2$. By Lemma \ref{kappa2bupcurve}, $x$ is good or $x'$ satisfies
again the assumption of the proposition with $A_1(x')=A_1(x)-1<1$: a contradiction with the previous remark.
Induction on $\lfloor A_1(x)\rfloor$ concludes the proof.
\end{proof}

\subsection{Monic expansions: blowing up a closed point.}

In this section, we control the behavior of the secondary invariant $\gamma (x)$ (Definition \ref{definvariants2})
by blowing up a closed point. By Proposition \ref{kappa2gamma0} we may furthermore assume
that $\gamma (x)\geq 1$. At this point, we connect the proof with
the equal characteristic proof given in \cite{CoP2} chapter 3.
Namely, this control is considered in Lemmas {\bf I.8.3} and {\bf I.8.8} (resp. Lemmas {\bf I.8.7} and {\bf I.8.9})
\cite{CoP2} chapter 3 when $x$ satisfies condition (*1) or (*2) (resp. condition (*3)). The
proof relies on the definition of the form
$$
\mathrm{in}_\alpha h =Z^p -G_\alpha^{p-1}Z+F_{p,Z,\alpha}\in (\mathrm{gr}_\alpha S)[Z]
$$
in Lemma \ref{structDelta2}(4)(5) w.r.t. the initial face $\sigma_{2,\mathrm{in}}$ of
$\Delta_2(h;u_1,u_2;u_3;Z)$, where $(u_1,u_2;u_3;Z)$ are well 2-adapted coordinates at $x$.\\

\noindent {\it notations used in \cite{CoP2}.} The corresponding notation for $F_{p,Z,\alpha}$ is
\begin{equation}\label{eq7412}
F_{p,Z,\alpha}=U_1^{a(1)}U_2^{a(2)}\left (\overline{\phi}_0 U_3^{\omega (x)}
+\sum_{j\in J_0}U_3^{\omega (x)-j}\Phi_j(U_1,U_2) \right )
\end{equation}
when $x$ satisfies condition (*1) or (*2) (Definition {\bf I.8.2.1}), with
$$
a(j)=pd_j, \ j=1,2, \ 0\neq \overline{\phi}_0 \in k(x)\ \mathrm{and}
\ \Phi_j(U_1,U_2)\in k(x)[U_1,U_2].
$$
By Definition \ref{def2solvable},
we have $\Phi_j(U_1,U_2)\neq 0$ for some $j_0\neq 0$.

\smallskip

When $x$ satisfies condition (*3), the notation is the same except that $\overline{\phi}_0$ and
$\Phi_j(U_1,U_2)$ are replaced respectively by
$U_2\overline{\phi}_0$,  $\overline{\phi}_0 \in U_2^{-1}k(x)[U_1,U_2,U_3]_1$, and by
$U_2\Phi_j(U_1,U_2)$ with $\Phi_j(U_1,U_2)\in U_2^{-1}k(x)[U_1,U_2]$ (Definition {\bf I.8.6.1}).
We have $a(2)=0$ in these formul{\ae} in cases (*1) and (*3).

\smallskip

Similarly, the corresponding notation for $G_\alpha$ is
\begin{equation}\label{eq741}
G_\alpha^{p}=U_1^{a(1)}U_2^{a(2)}\mathrm{cl}_{B(x)\omega (x)}(H(x)^{-1}g^p)
\end{equation}
when $x$ satisfies condition (*1) or (*2). When $x$ satisfies condition (*3), we
have
\begin{equation}\label{eq7411}
G_\alpha^{p}=U_1^{a(1)}\mathrm{cl}_{1+B(x)\omega (x)}(H(x)^{-1}g^p).
\end{equation}

The numerical invariants $\beta (x)$ and $B(x)$ are denoted respectively by $\beta 3(x)$ and $B3(x)$
in \cite{CoP2} when $x$ satisfies condition (*3). The statement ``$\kappa (x)\leq 1$'' in \cite{CoP2}
stands for ``$x$ is resolved'' in this article.
The vector spaces $\mathrm{cl}_{\mu_0,\omega (x)}J$ (\cite{CoP2} Definitions {\bf I.8.2.3} and
{\bf I.8.6.3}) are determined by the initial form polynomial $\mathrm{in}_\alpha h $. The proofs
of the following lemmas are almost entirely based on the numerical  Lemmas {\bf I.8.2.2} and
{\bf I.8.6.2} in \cite{CoP2} which are characteristic free. We simply refer to their counterpart
in \cite{CoP2} except when they do not immediately adapt to our characteristic free setting. \\

Assume that ($\kappa (x)=2$, $x$ satisfies condition (*) and $\gamma (x)\geq 1$).
Let $\pi : {\cal X}'\longrightarrow {\cal X}$ be the blowing up along $x$ and $x'\in \pi^{-1}(x)$.
We denote by $d:=[k(x'):k(x)]$.

\begin{lem}\label{gamma2*12}
With notations as above, assume that $x$ is in case (*1) or (*2). Let $(u_1,u_2;u_3;Z)$
be well 2-adapted coordinates at $x$ and assume furthermore that
$$
\eta '(x')\in \mathrm{Spec}(S[{u_2\over u_1},{u_3\over u_1}][Z']/(h')), \ h':=u_1^{-p}h, \ Z':={Z \over u_1}.
$$
Then $x'$ is resolved or ($\kappa (x')=2$, $x'$ satisfies again condition (*) with
$$
A_1(x')=B(x)-1, \ \gamma (x')\leq \gamma (x),
$$
and there exist well 2-adapted coordinates
$(u'_1,u'_2;u'_3;Z')$ at $x'$ such that the following holds:)

\smallskip

\begin{itemize}
  \item [(1)] if $x'=(Z/u_1,u_1,u_2/u_1,u_3/u_1)$, then $x'$ is again in case (*1) (resp. in case (*2))
  and we have $C(x')\leq C(x)$, $\beta(x')\leq \beta(x)$;
  \item [(2)] if $x'\neq (Z/u_1,u_1,u_2/u_1,u_3/u_1)$, then $x'$ satisfies condition (*1) or (*3), and either
  (3') below holds or (3)-(4) below hold;
  \item [(3')] the point $x$ satisfies condition (*2) with
  $$
  U_1^{-pd_1}U_2^{-pd_2}F_{p,Z}= \mu U_3^{p} +c_p(U_1 + \lambda U_2)^p ,
  $$
  where $d_1 , d_2 \not \in \N$, $\lambda, \mu, c_{p} \in k(x)$, $\lambda \mu c_{p}\neq 0$
  and $\mu^{-1}c_{p}\not \in k(x)^p$ up to change of well 2-adapted coordinates;
  furthermore, $x'$ satisfies condition (*1), $k(x')=k(x)$ and we have
  $$
  \mathbf{y}':=(\alpha_2(x'),\beta_2(x'))=(0,p/(p-1)) \in \Delta_2(h';u'_1,u'_2;u'_3;Z')
  $$
  and
  \begin{equation}\label{eq742}
  \mathrm{in}_{\alpha '} h'={Z'}^p +{U'_1}^{pd'_1}(\lambda '{U'_3}^p +U'_3{U'_2}^p),
  \end{equation}
  with $d'_1\in \N$, $\lambda '\not\in k(x)^p$, notations as in Lemma \ref{structDelta2}(4) with $\sigma_2=\mathbf{y}'$;
  \item [(3)] we have
  $$
  \beta (x')\leq {C(x) \over d} +{1 \over p};
  $$
  \item [(4)] we have
  $$
  \beta (x') <
  \left\{
  \begin{array}{cc}
  1+ \lfloor {C(x)\over d}\rfloor \hfill{}& \mathrm{if} \  x' \ \mathrm{is} \ \mathrm{in} \ \mathrm{case} \ (*1) \\
   & \\
  1+ \lfloor {C(x)\over d}\rfloor -{1 \over \omega (x)}  & \mathrm{if} \ x' \ \mathrm{is} \ \mathrm{in} \ \mathrm{case} \ (*3) \\
  \end{array}
  \right.
  .
  $$
\end{itemize}

\end{lem}

\begin{proof}
We already know from Proposition \ref{redto*}(ii) that $x'$ is resolved or
($\kappa (x')=2$ and $x'$ satisfies condition (*)). Note that we have
$$
B(x)>1 \Leftrightarrow  \tau '(x)=1.
$$
Namely, we have $<U_3>\subseteq \mathrm{Vdir}(x)$ by Definition \ref{*kappadeux}, so
$$
\tau '(x)=1 \Leftrightarrow H^{-1}F_{p,Z}\in <U_3^{\omega (x)}> \Leftrightarrow B(x)>1,
$$
where the left hand side equivalence is true because $\Delta_{S}(h;u_1,u_2,u_3;Z)$ is minimal.

\smallskip

If $B(x)=1$, then $x$ is of type (T0), (T2) or (T3) as defined along the proof of Proposition \ref{redto*}.
What follows has been proved  {in} the course of that proof: for type (T0), $x$ is good; for type (T3),
$x'$ is resolved by Theorem \ref{bupthm} since $\mathrm{Vdir}(x)=<U_3,U_1>$;
for type (T2), $x$ is good or ($d_1+d_2\in \N$, $d_2 \not \in \N$, $B(x)=C(x)=1$). In this situation,
we have $\kappa (x')=2$, $x'$ satisfies condition (*) and there
exist well 2-adapted coordinates $(u'_1,u'_2;u'_3;Z')$ at $x'$ such that $A_1(x')=0$ and one of the following holds:

\smallskip

\noindent $\bullet$ $x'$ is in case (*1)
\begin{equation}\label{eq7421}
\beta (x')= {i+1\over i}, \  i\equiv 0 \ \mathrm{mod}p,
\ p \leq i \leq \omega (x);
\end{equation}

\smallskip

\noindent $\bullet$ $x'$ is in case (*1) and
\begin{equation}\label{eq7422}
\beta (x')={\omega (x)\over \omega (x) -1};
\end{equation}

\noindent $\bullet$ $x'$ is in case (*3) and  $\beta (x')=1$.

\smallskip

See the discussion in the proof of Proposition \ref{redto*}: these three situations
correspond respectively to $I=\{0\}$, $I=\{\omega (x)\}$ and $I=\emptyset$ therein.
When (\ref{eq7422}) holds with $\omega (x)=p$, we have (3'); otherwise, we have (3)(4).
Note that $\gamma (x')=\gamma (x)=2$ here.\\

If $B(x)>1$, statement (1) is easily deduced from the characteristic
free Proposition \ref{originchart} as in \cite{CoP2}. The rest of the proof
relies on the characteristic free transformation formula \cite{CoP2}(4) on p.1918
and numerical Lemma {\bf I.8.2.2} and is identical to that of {\bf I.8.3}(1)(2)(ii)(iv)-(vi).
If $x'$ satisfies condition (*3), note that (4) is an equivalent formulation of \cite{CoP2} Lemma {\bf I.8.3}(1).
\end{proof}

\begin{exam}\label{boundsharp}
Let $\omega (x)=\overline{\omega}p^a$, $a \geq 2$, $\overline{\omega}/p \not \in \N$. We prove here that
the bound in Lemma \ref{gamma2*12}(3) is sharp when $x'$ satisfies either condition (*1) or (*3).

\smallskip

Let $E=\mathrm{div}(u_1u_2)$, $d_1 \in  {1 \over p}\N \backslash \N$, $d_2 \in {1 \over p}\N$, $C\in \N$. Take
$$
G_\alpha =0, \ U_1^{-pd_1}U_2^{-pd_2}F_{p,Z,\alpha}=\left (U_3^p - U_1^{p-1}U_2(U_2-U_1)^{pC}\right )^{\overline{\omega}p^{a-1}},
$$
where $C(x)=C$. Let $S':=S[u_2/u_1,u_3/u_1]_{(u_1,u'_2,u'_3)}$, where
$$
u'_2:=u_2/u_1 -1, \ u'_3:=u_3/u_1 -u_1^{C}{u'_2}^C.
$$
Letting $g':={u'_3}^p - u_1^{pC}{u'_2}^{pC+1}$, we get
  $$
  h'={Z'}^p + \sum_{i=1}^{p-1}f_{i,Z'}{Z'}^{p-i} + u_1^{pd'_1}( f'+u_1f'_1)  \in S'[Z'],
  $$
  where $d'_1=d_1+d_2 +\omega (x)/p -1$, $\mathrm{ord}_{u_1}f_{i,Z'}>id'_1$, $f'_1\in S'$ and
  $$
  \left\{
    \begin{array}{cccc}
      f':= \delta ' {g'}^{\overline{\omega}p^{a-1}},\hfill{}& Z':= Z/u_1 \hfill{}
      & \mathrm{if} &  d_1+d_2 \not \in \N\\
       &  &  &  \\
      f':= \delta 'u'_2{g'}^{\overline{\omega}p^{a-1}}, & Z':= Z/u_1 +u_1^{d'_1}{g'}^{\overline{\omega} p^{a-2}}
      & \mathrm{if} & d_1+d_2 \in \N \\
    \end{array}
  \right.
  ,
$$
with $\delta ' \in S'$ a unit. In both cases we get $\beta (x')=C+1/p$.
Note that the above argument also works for ($a=1$ and $x'$ satisfies condition (*1)).
\end{exam}

\medskip

We now turn to the (*3)-version of the previous lemma.
We point out that the situation $J_0 \subset p\N$ has {\it not} been correctly
analyzed in the proof of \cite{CoP2} Lemma {\bf I.8.7}. Namely, the bound (3') ({\it ibid.} p. 1929)
may fail (case 2 on p.1930 when $d=1$) unlike stated therein; the {\it same} mistake
occurs in {\bf I.8.7.5} case 1.

\smallskip

We review and amend the corresponding statements in Lemma \ref{gamma2*3}(2) below. Denote
$F_{p,Z,\alpha}$ as in Lemma \ref{structDelta2}(5).
Adapting notations of (\ref{eq7412}), there is an expansion
\begin{equation}\label{eq7426}
U_1^{-pd_1}F_{p,Z,\alpha}=(\mu U_3 +c U_1+U_2)U_3^{\omega (x)}
+\sum_{j\in J_0}U_3^{\omega (x)-j}U_1^{b_j}\Psi_j(U_1,U_2),
\end{equation}
where $\mu ,c \in k(x)$ ($\mu =0$ if $B(x)>1$), and bounds:
$$
b_j \geq jA_1(x), \ j\beta (x) \geq \mathrm{deg}_{U_2}\Psi_j(U_1,U_2)-1.
$$
The subset $J_0 \subseteq \{1, \ldots ,\omega (x)\}$ is defined by
\begin{equation}\label{eq7426bis}
j\in J_0 \Leftrightarrow \Psi_j(U_1,U_2)\neq 0.
\end{equation}

\begin{lem}\label{gamma2*3}
Assume that $x$ satisfies condition (*3). Let $(u_1,u_2;u_3;Z)$
be well 2-adapted coordinates at $x$ and assume furthermore that
$$
\eta '(x')\in \mathrm{Spec}(S[{u_2\over u_1},{u_3\over u_1}][Z']/(h')), \ h':=u_1^{-p}h, \ Z':={Z \over u_1}.
$$
Then $x'$ is resolved or ($\kappa (x')=2$, $x'$ satisfies condition (*1) or (*3) with
$$
A_1(x')=B(x)-1, \ \gamma (x')\leq 1+\gamma (x)
$$
and there exist well 2-adapted coordinates $(u'_1,u'_2;u'_3;Z')$ at $x'$ such that
either (1') below holds, or (1)-(3) below hold:)

\smallskip

\begin{itemize}
  \item [(1')] we have
  $$
  U_1^{-pd_1}F_{p,Z}=U_2U_3^p +c_pU_1(U_2 +\lambda U_1)^p,
  $$
  where $\lambda \neq 0$, ($d_1+1/p \not \in \N$ or $c_p \not \in k(x)^p$) up to change of
  well 2-adapted coordinates; furthermore $x'$ satisfies condition (*1) and (\ref{eq742}) holds at $x'$
  with $\lambda '\neq 0$ and ($d'_1 \not \in \N$ or $\lambda ' \not \in k(x)^p$);
  \item [(1)] we have
  $$
  \beta (x')\leq {\gamma(x) \over d} +{1 \over p}
  $$
  and inequality is strict if $x'$ satisfies condition (*3);
  \item [(2)] if $\gamma (x') > \gamma (x)$, then $k(x')=k(x)$ and $x'$ is uniquely determined;
  up to a change of well 2-adapted coordinates, $x'=(Z/u_1,u_1,u_2/u_1,u_3/u_1)$  and (\ref{eq7426}) reads
  \begin{equation}\label{eq7429}
  U_1^{-pd_1}F_{p,Z,\alpha}=(\mu U_3 +U_2)U_3^{\omega (x)} +c U_1(U_3 + \lambda U_1^{k}U_2^{\gamma (x)})^{\omega (x)}
  \end{equation}
  with $k\in \N$, $\lambda c \neq 0$, ($d_1 +1/p \not \in \N$ or $c \not\in k(x)^p$),
  and ($\mu =0$ if $B(x)=k+\gamma (x)>1$); furthermore, we have
  $$
  A_1(x)=k+{1 \over \omega (x)}, \beta (x)= \gamma (x)-{1 \over \omega (x)}
  $$
  and $x'$ satisfies condition (*1) with
  $$
  \beta (x')=\gamma (x)+{1 \over \omega (x)};
  $$
  \item [(3)] if ($\gamma (x')\leq \gamma (x)$ and $x'$ is in case (*3)), then
  $$
  \beta (x')\leq \max \{\beta (x), {1 \over p}\}
  $$
  and $\beta (x')< \beta (x)$ if ($k(x')\neq k(x)$ and $\beta (x)> 1/p$).
\end{itemize}
\end{lem}

\begin{proof}
We already know from Proposition \ref{redto*}(ii) that $x'$ is resolved or
($\kappa (x')=2$ and $x'$ satisfies condition (*)). Note that we have
$$
B(x)>1 \Leftrightarrow  H^{-1}F_{p,Z}\in <U_1U_3^{\omega (x)}, U_2U_3^{\omega (x)}, U_3^{\omega (x)+1}>.
$$
If $\tau '(x)\geq 2$, we certainly have $B(x)=1$ and $x$ is of type (T1) or (T4) as defined along
the proof of Proposition \ref{redto*}. For type (T4), $x'$ is resolved by Theorem \ref{bupthm}
since $\mathrm{Vdir}(x)=<U_3,U_1>$. For type (T1), note that we have $\beta (x)=1$, hence $\gamma (x)=2$.
The following holds: $x$ is good or $\kappa (x')=2$, $x'$ satisfies condition (*)
and there exist well 2-adapted coordinates $(u'_1,u'_2;u'_3;Z')$ at $x'$ such that $A_1(x')=0$ and either:

\smallskip

\noindent $\bullet$ $x'$ is in case (*3) and $\beta (x')=1$, or

\smallskip

\noindent $\bullet$ $x'$ is in case (*1) and
$$
\beta (x')={1+i\over i}, \ i\geq 1.
$$
See the discussion along the course of the proof of Proposition \ref{redto*}: these two situations
correspond respectively to case 1 and case 2 therein. This proves that $x'$ is resolved or
($\gamma (x')=\gamma (x)=2$ and (1)(3) hold) when $\tau '(x)=2$.

\smallskip

Assume now that ($B(x)=1$ and $\tau '(x)=1$). The argument in the proof of Proposition \ref{redto*},
{\it viz.} (\ref{eq7202})-(\ref{eq7203}), gives
$$
\mathrm{in}_{m_S}h=Z^p + U_1^{pd_1}\left ((\mu U_3 +U_2)U_3^{\omega (x)}
+U_1\sum_{i=1}^{\omega (x)/p}U_3^{\omega (x)-pi}\Phi_i (U_1^p,U_2^p)\right )
$$
where $\mu \in k(x)$ and $\Phi_i\in k(x)[T_1,T_2]_{i}$, $1 \leq i \leq \omega (x)/p$.
It is easily seen from this expression that
$$
\omega (x')\leq \omega (x) -p\min_{1\leq i \leq {\omega (x)\over p}}\left \{i-{\mathrm{deg}_{T_2}\Phi_i \over d}\right \},
$$
so $\omega (x')=\omega (x)$ implies $d=1$, and $\Phi_i$ monic in $T_2$ whenever $\Phi_i\neq 0$. Similarly,
we have
$$
\sum_{i=1}^{\omega (x)/p}\mathrm{Vdir}
\left (\{{\partial \Phi_i (U_1^p,U_2^p)\over \partial \lambda_l}\}_{l \in \Lambda_0}\right )
=<U_1,U_2> \Longrightarrow \omega (x')< \omega (x),
$$
with notations as in (\ref{eq2412}). After possibly changing $Z$ with $Z- \phi$,
$\phi \in S$, it can thus be assumed that
\begin{equation}\label{eq7423}
\mathrm{in}_{m_S}h=Z^p + U_1^{pd_1}\left ((\mu U_3 +U_2)U_3^{\omega (x)}
+U_1\sum_{i=1}^{\omega (x)/p}c_iU_3^{\omega (x)-pi}(U_2 +\lambda U_1)^{pi}\right ),
\end{equation}
where $\mu \in k(x)$, $\lambda \in k(x)$ and $c_i \in k(x)$, $1 \leq i \leq \omega (x)/p$. Furthermore,
we have $x'=(Z'/u_1, u_1, u_2/u_1 + \gamma , u_3/u_1)$, where $\gamma \in S$ is a preimage of $\lambda$.
The proof now goes on along the same lines as that of the case $B(x)=1$ in the previous lemma: $x'$
is resolved or $x'$ satisfies condition (*1), $A_1(x')=0$ and one of (\ref{eq7421})-(\ref{eq7422}) holds
(in particular $\gamma (x')=2$). When (\ref{eq7422}) holds with $\omega (x)=p$, we have (1)'; otherwise,
we have (1), (3) being pointless.

\smallskip

For (2), note that $x'$ satisfies the assumptions
of Proposition \ref{tauegaldeux} (so $x$ is good) if $c_i \neq 0$ for some $i< \omega (x)/p$.
Otherwise, we have
\begin{equation}\label{eq7424}
(\alpha_2(x),\beta_2(x))=({1 \over \omega (x)},1-{1 \over \omega (x)}).
\end{equation}
By Definition \ref{definvariants2}, we also have $\beta (x)= (i_1-1)/i$, $1 \leq i \leq \omega (x)$
and $i_1\in \N$. By assumption, $\gamma (x)=1$, so $\beta (x)<1$ and we get
$$
1-{1 \over \omega (x)} = \beta_2(x) \leq \beta (x)\leq 1-{1 \over i}.
$$
We deduce that $i_1=i=\omega (x)$. By (\ref{eq7424}), this implies that
$$
(A_1(x),\beta (x))=(\alpha_2(x),\beta_2(x))=({1 \over \omega (x)},1-{1 \over \omega (x)})
$$
and the conclusion follows. This is the special situation considered in \cite{CoP2} Lemma {\bf I.8.7}(b).\\

If $B(x)>1$,  the proof is identical
to that of \cite{CoP2} Lemma {\bf I.8.7}(b)(b')(d)(i)-(iii)(v): this
relies on the numerical Lemma {\bf I.8.6.2} and characteristic free transformation formula
for $\mathrm{cl}_{\mu_0, \omega (x)}J$ (Definition {\bf I.8.6.3}). As observed before stating this lemma,
a mistake in \cite{CoP2} {\bf I.8.7.8} (case 2, $B(x)\in \N$) has to be amended at this point. Namely,
the bounds (3)(4) on p.1929 only hold when $G=\mu_2^{-1}{\partial F \over \partial U_2}\neq 0$ with notations
as in there. The correct bounds are thus no better than those given in {\bf I.8.7.8} case 3:
\begin{equation}\label{eq7425}
\beta (x') \leq {\mathrm{deg}_{U_2}\Psi_{j_1}(U_1,U_2) \over j_1d}+{1 \over p},
\ \beta 3(x') \leq {\mathrm{deg}_{U_2}\Psi_{j_1}(U_1,U_2) \over j_1d}+{1 \over p}-{1 \over p^a},
\end{equation}
where $a:=\mathrm{ord}_p\omega (x)$: this gives (1) of the present lemma.

\smallskip

We note however that the bounds (3)(3')(4)(4') on p.1929-1930 are correct if $d\geq 2$ (this relies on
Lemma \ref{lem532}(2), statement ``$d=1$ if equality holds''). This proves that $\gamma (x')\leq \gamma (x)$
if $k(x')\neq k(x)$. There remains to prove (2) and (3)
(resp. (3)) of the present lemma  for $d=1$ (resp. for $d\geq 2$).\\

\noindent {\it First assume that} $d\geq 2$, i.e. $k(x')\neq k(x)$. The
conclusion follows trivially from (1) if $\beta (x)\geq 1$, so we may assume that $\beta (x)<1$.

\smallskip

The proof involves picking some element $G\in \mathrm{cl}_{\mu_0,\omega (x)}$,
$G\neq 0$ \cite{CoP2} middle of p. 1930 and computing the order of its transform. {\it This is done
after possibly performing the Tschirnhausen transformation described in} \cite{CoP2} {\bf I.8.3.6}.
We discuss according to the set $J_0$ in (\ref{eq7426bis}):

\smallskip

{\it Case 1: $J_0 \nsubseteq p\N$}. Arguing as in  \cite{CoP2} {\bf I.8.7.7}, we get
$$
\beta (x')\leq {\mathrm{deg}_{U_2}\Psi_{j_1}(U_1,U_2)\over j_1 d} -{1 \over j_1}< {\beta (x) \over d}.
$$

{\it Case 2: $J_0 \subseteq p\N$ and $B(x)\not \in \N$}. By \cite{CoP2} (4) on p.1930, we get
$$
\beta (x)-\beta (x') \geq \left ( 1-{1 \over d}\right )\beta (x) -{1 \over pd}>
{1 \over p} \left (1 - {2 \over d}\right )\geq 0.
$$

{\it Case 3: $J_0 \subseteq p\N$, $B(x) \in \N$ and $G=U_1^{-pd_1}{\partial F_{p,Z,\alpha} \over \partial U_2}$}.
Amending \cite{CoP2} {\bf I.8.7.8} as in (\ref{eq7425}), we obtain the bound $\beta (x')\leq \beta (x)/d$
except possibly if $j_1=p^a$; in this case, we let
\begin{equation}\label{eq7427}
a':=\max\{b : U_1^{b_{p^a}}\Psi_{p^a}(U_1,U_2)\in (k(x)[U_1,U_2])^{p^b}\}< a
\end{equation}
and obtain the bound:
\begin{equation}\label{eq7428}
\beta (x') \leq \max\{p^{a'-a},\beta (x)\} \ (\mathrm{resp.} \ \beta (x') <  \beta (x))
\end{equation}
from Lemma \ref{lem532}(2) (resp. {\it ibid.} with $\mathrm{deg}F\geq  2$ if $\beta (x)>1/p$).

\smallskip

{\it Case 4: $J_0 \subseteq p\N$, $B(x) \in \N$ and
$U_1^{-pd_1}{\partial F_{p,Z,\alpha} \over \partial U_2}=U_3^{\omega (x)}$}.
The bound is:
$$
\beta (x')\leq {\mathrm{deg}_{U_2}\Psi_{j_1}(U_1,U_2) \over  j_1 d}
$$
as in case 2 with the same conclusion.\\

\noindent {\it Assume that} $k(x')= k(x)$. By the independence statement
in Theorem \ref{well2prepared}, it can be assumed that $x'$ is the origin of the chart. We build
upon (\ref{eq7426}) and connect the proof with \cite{CoP2} {\bf I.8.7.5}.
First note that $x'$ satisfies condition (*3) if and only if $\mu =0$, since $\Delta_S (h;u_1,u_2,u_3;Z)$
is minimal. In this situation one gets easily $\beta (x')\leq \beta (x)$
from Proposition \ref{originchart} as in case 3 of \cite{CoP2} {\bf I.8.7.5}. This completes
the proof when $x'$ satisfies condition (*3).

\smallskip

Assume now that $x'$ satisfies condition (*1), so $c  \neq 0$ in (\ref{eq7426}). Note to begin with that we have
\begin{equation}\label{eq744}
{\mathrm{deg}_{U_2}\Psi_{j}(U_1,U_2)-1\over j }\leq \beta (x) \Longrightarrow {\mathrm{deg}_{U_2}\Psi_{j}(U_1,U_2)\over j }\leq \gamma (x)
\end{equation}
for each $j\in J_0$ in (\ref{eq7426}).  We again consider the same cases 1 to 4
as for $k(x')\neq k(x)$:

\smallskip

{\it Case 1: $J_0 \nsubseteq p\N$}. Arguing as in  \cite{CoP2} {\bf I.8.7.7}, we get
$$
\beta (x')\leq {\mathrm{deg}_{U_2}\Psi_{j_1}(U_1,U_2)\over j_1 } \leq \gamma (x).
$$

{\it Case 2: $J_0 \subseteq p\N$ and $B(x)\not \in \N$}. Same as in case 1 by \cite{CoP2} (3') on p.1929.

\smallskip

{\it Case 3: $J_0 \subseteq p\N$, $B(x) \in \N$ and $G=U_1^{-pd_1}{\partial F_{p,Z,\alpha} \over \partial U_2}$}.
In this situation, equality in (\ref{eq744}) implies $\mathrm{deg}_{U_2}\Psi_{j}(U_1,U_2) \in p\N$. Therefore
$$
\mathrm{deg}_{U_2}{\partial \Psi_j \over \partial U_2}\leq \mathrm{deg}_{U_2}\Psi_{j}(U_1,U_2)-2
$$
in (\ref{eq7426}) and we get the same bound as in case 1.

\smallskip

{\it Case 4: $J_0 \subseteq p\N$, $B(x) \in \N$ and $U_1^{-pd_1}{\partial F_{p,Z,\alpha} \over \partial U_2}=U_3^{\omega (x)}$}.
We now have $\Psi_j(U_1,U_2)=\Phi_j(U_1^p,U_2^p)$ for $j\in J_0$ and must take
$$
G:=U_1^{-pd_1} (D \cdot F_{p,Z,\alpha}), \ D= \lambda_l {\partial \hfill{} \over \partial \lambda_l} \ \mathrm{or} \
D=U_1{\partial \hfill{} \over \partial U_1} - (pd_1)U_2{\partial \hfill{} \over \partial U_2}.
$$
Arguing as in the case ($B(x)=1$ and $\tau '(x)=1$),
we obtain the same bound as in case 1 except possibly if
\begin{equation}\label{eq7441}
 U_1^{-pd_1}F_{p,Z,\alpha} = (c U_1 +U_2)U_3^{\omega (x)}
+U_1\sum_{i=1}^{\omega (x)/p}c_{pi}U_3^{\omega (x)-pi}U_1^{kpi}U_2^{pi\gamma (x)},
\end{equation}
where $k:=B(x)-\gamma (x)\in \N$. Define:
$$
P(t):= c t^{\omega (x)} +\sum_{i=1}^{\omega (x)/p}c_{pi} t^{\omega (x)-pi}.
$$
If $pd_1 +1 \not \in \N$ (resp. $pd_1 +1 \in \N$) and $P(t)\neq c (t + \lambda)^{\omega (x)}$
(resp. and $P(t)\neq c (t + \lambda)^{\omega (x)}+ Q(t)^p$ with $Q(t)\in k(x)[t]$) for some $\lambda \in k(x)$,
then
$$
\mathbf{y}':=(B(x)-1, \gamma (x))\in \Delta_2(h';u_1,u'_2;u'_3;Z')
$$
is a vertex which is not 2-solvable and we get $\beta (x')\leq \gamma (x)$.
Otherwise, we may assume w.l.o.g. that $Q=0$ after changing $Z$ with $Z- \phi$, $\phi \in S$,
which gives (\ref{eq7429}). One concludes as in the case ($B(x)=1$ and $\tau '(x)=1$) above.
\end{proof}

We now consider the remaining point ``at infinity'' for the blowing up
$\pi : {\cal X}'\longrightarrow {\cal X}$ along $x$.

\begin{lem}\label{gamma2*infty}
With notations as above, assume that $x$ satisfies condition (*). Let $(u_1,u_2;u_3;Z)$
be well 2-adapted coordinates at $x$ and assume furthermore that
$$
x'=(Z':=Z/u_2,u'_1:=u_1/u_2,u_2,u'_3:=u_3/u_2).
$$
Then $x'$ is resolved or ($\kappa (x')=2$, $x'$ satisfies condition (*2),
$(u'_1,u_2;u'_3;Z')$ are well 2-adapted coordinates at $x'$,
$$
A_1(x')=A_1(x), \ A_2(x')=B(x)-1, \ \beta (x')=A_1(x)+\beta (x)-1, \ \gamma (x')\leq \gamma (x),
$$
and the following holds:)

\smallskip

\begin{itemize}
  \item [(1)] if $x$ is in case (*1), then $C(x')\leq \min\{\beta (x)-C(x),C(x)\}$;
  \item [(2)] if $x$ is in case (*2), we have $C(x')\leq \min\{\beta (x) -A_2(x) -C(x),C(x)\}$.
  \item [(3)] if $x$ is in case (*3), we have $C(x')\leq \min\{\beta (x)-C(x), C(x)-\beta_2(x)\}$.
\end{itemize}
\end{lem}

\begin{proof}
This relies on the characteristic free Proposition \ref{originchart}.
The argument in \cite{CoP2} Lemmas {\bf I.8.8} and {\bf I.8.9} gives all statements
before ``$\gamma (x')\leq \gamma (x)$''. Moreover equations (2) on p.1933 and (2) on p.1934 give:
\begin{equation}\label{eq743}
C(x')\leq \min\{\beta (x)-(B(x)-A_1(x)), \alpha_2(x)-A_1(x)\}.
\end{equation}

Assume that $x$ is in case (*1) or (*3). We have
\begin{equation}\label{eq7431}
\alpha_2(x) + \beta_2 (x)=B(x), \ B(x)-A_1(x)=C(x).
\end{equation}
This proves (3); if $x$ satisfies condition (*1), then $\beta_2 (x)\geq 0$ and the conclusion
follows from (\ref{eq743}).

\smallskip

If $x$ satisfies condition (*2), we have $\beta_2 (x)\geq A_2(x)$, so (\ref{eq7431}) implies that
$\alpha_2(x)-A_1(x)\leq C(x)$ and (2) follows easily. Since $\gamma (x)\geq 1$,
$\gamma (x')\leq \gamma (x)$ is a trivial consequence of Definition \ref{definvariants2} except
if ($x$ is in case (*3) and $C(x)<0$). But then we have $\beta_2(x)=-1/i$ for some $i$, $1\leq i\leq \omega (x)$
by Lemma \ref{structDelta2} and Corollary \ref{Delta2+}. Therefore
$$
C(x')\leq C(x)-\beta_2(x)<1
$$
by (3) and we get $\gamma (x')\leq 1$.
\end{proof}

\subsection{Monic expansions: the algorithm.}

In this chapter, we prove Theorem \ref{projthm} when $\kappa (x)=2$. This is restated as
Theorem \ref{proofkkappa2} below. The strategy of the proof has much in common with
the one used for Theorem \ref{contactmaxFIN} or for Embedded Resolution of Singularities
for surfaces \cite{Co3}: roughly speaking, the invariant $\gamma (x)$ is in general
nonincreasing by blowing up a point $x$, and drops  at a nonrational exceptional point or exceptional point ``at
infinity" $x'$.  Infinite chains of rational points not ``at infinity" do not occur
by Corollary \ref{permisarcthree}. This general idea is illustrated by the proof of Proposition \ref{kappa2fin0}
below which provides the logical scheme of the proof.

\smallskip

Considering however the {\it precise} behaviour of the invariant $\gamma (x)$ under blowing up, the situation
turns out to be more complicated than expected. Two phenomena contribute: on the one hand,
the directrix vector space $\mathrm{Vdir}(x)$ is not well-behaved under blowing up;
on the other hand, $\gamma (x)$ does not necessarily drop at a nonrational exceptional point or exceptional point ``at
infinity" and may also increase in some special situations (Lemma \ref{gamma2*3}(1')(2)). These
phenomena make the proof very intricate when $\gamma (x)=2$, especially when $p=2$. One is then driven
to a step by step proof where the main difficulty is to avoid loops
(Propositions \ref{kappa2fin20} to \ref{kappa2fin3}). We also emphasize that most of these
intricacies actually occur when $S$ is equicharacteristic with algebraically closed residue field.

\smallskip

Let $\mu$ be a valuation of $L=k({\cal X})$ centered at $x$ and consider the quadratic sequence
\begin{equation}\label{eq750}
    ({\cal X},x)=:({\cal X}_0,x_0) \leftarrow ({\cal X}_1,x_1)\leftarrow \cdots \leftarrow ({\cal X}_r,x_r)\leftarrow \cdots
\end{equation}
along $\mu$. We will show that $x_r$ is resolved for some $r\geq 0$, hence $x$ is good as explained
in Remark \ref{quadsequence}.

\begin{thm}\label{proofkkappa2}
Projection Theorem \ref{projthm} holds when $\kappa (x)=2$. One may take all local blowing ups
in (\ref{eq402}) permissible (of the first kind or second kind) if $p=2$
or if $\omega (x)\geq 3$.
\end{thm}

\begin{proof}
By Proposition \ref{redto*}, it can be assumed that
$\omega (x) \equiv 0 \ \mathrm{mod}p$ and that $x_r$ satisfies condition (*) for every $r\geq 0$. Under
these assumptions,  {the} invariant $\gamma (x_r)\in \N$ is defined for $r\geq 0$ (Definition \ref{definvariants2}).

\smallskip

By Proposition \ref{kappa2fin0} below, there exists $r_0\geq 0$ such that either $x_{r_0}$ is resolved or
$\gamma (x_{r_0})\leq 2$.

\smallskip

If $\gamma (x_{r_0})=0$, then $x_{r_0}$ is resolved by Proposition \ref{kappa2gamma0}.

\smallskip

Suppose that $\gamma (x_{r_0})=1$. If $x_{r_0}$ satisfies condition (*1) (resp. (*2)), then $x_{r_0}$ is
resolved by  Proposition \ref{kappa2fin10}(1) (resp. Proposition \ref{kappa2fin11}) below.
If $x_{r_0}$ satisfies condition (*3) and $\beta (x)<1-1/\omega (x)$ (resp. and $\beta (x)=1-1/\omega (x)$,
$(p,\omega (x))\neq (2,2)$; resp. and $\beta (x)=1/2$, $(p,\omega (x))= (2,2)$), then $x_{r_0}$ is
resolved by Proposition \ref{kappa2fin10}(3) (resp. Proposition \ref{kappa2fin20}(ii);
resp. Proposition \ref{kappa2fin22}(ii)).

\smallskip

Assume finally that $\gamma (x_{r_0})=2$. If $x_{r_0}$ satisfies condition (*1) (resp. (*2); resp. (*3)),
then $x_{r_0}$ is resolved by Proposition \ref{kappa2fin20}(i) or by Proposition \ref{kappa2fin23}(i)
(resp. by Proposition \ref{kappa2fin24};
resp. by Proposition \ref{kappa2fin23}(ii) or by Proposition \ref{kappa2fin3}).
\end{proof}

\begin{lem}\label{kappa2fin00}
With notations as above, assume that $x_r$ satisfies condition (*2) for every $r\geq 0$.
Then there exists $r_0\geq 0$ such that $C(x_r)=0$ for every $r\geq r_0$.
\end{lem}

\begin{proof}
We consider the points
$$
\mathbf{y}:=(A_1(x),A_2(x)+a (x)), \ \mathbf{y}':=(A_1(x)+a'(x),A_2(x)) \in \Delta_2(u_1,u_2;u_3;Z),
$$
where $(u_1,u_2;u_3;Z)$ are well 2-adapted coordinates. By standard arguments on combinatorial blowing ups, we have
$c(x_1)< c(x)$ for the lexicographical ordering whenever $C(x)>0$, where
$$
c(x):=(C(x)=\min\{a(x), a'(x)\},\max\{a(x), a'(x)\}).
$$
Since these numbers belong to ${1 \over \omega (x)!}\N^2$, we get $C(x_r)=0$ for all $r>>0$.
\end{proof}

\begin{prop}\label{kappa2fin0}
With notations as above, there exists $r_0\geq 0$ such that $x_{r_0}$ is resolved or $\gamma (x_{r_0})\leq 2$.
\end{prop}

\begin{proof}
Let $(u_1,u_2;u_3;Z)$ be well 2-adapted coordinates at $x$. We will name
point ``at infinity" for simplicity the origin $x'$ of the second chart of the blowing up, i.e.
\begin{equation}\label{eq7504}
x':=(Z/u_2,u_1/u_2,u_2,u_3/u_2).
\end{equation}
The notion is unambiguous if $E=\mathrm{div}(u_1)$, that is if $x$ satisfies condition (*1) or (*3).
If $x$ satisfies condition (*2), the point ``at infinity"  furthermore depends on the numbering
of $u_1,u_2$, where $E=\mathrm{div}(u_1u_2)$.

\smallskip

We may assume that $\gamma (x)\geq 3$ for the whole proof. Note that the special situations
described in Lemma \ref{gamma2*12}(3') and in Lemma \ref{gamma2*3}(1') occur only when $\gamma (x)\leq 2$.
We may thus disregard them in this proof. To prove the proposition,
it is sufficient to prove that there exists $r\geq 1$ such that $x_r$ is resolved
or $\gamma (x_r)<\gamma (x)$. We first bound  $\gamma (x_1)$ in terms of $\gamma (x)$
at a nonrational point or at a point ``at infinity".

\smallskip

\noindent {\it Assume that $k(x_1)\neq k(x)$}. We apply Lemma \ref{gamma2*12}(4) and Lemma \ref{gamma2*3}(1) with
$d\geq 2$. Note that for $\alpha >1$, we have
\begin{equation}\label{eq7501}
    1+\left \lfloor {\alpha \over d}\right \rfloor \leq \lceil \alpha \rceil
\end{equation}
and equality holds if and only if $\alpha =d=2$. If $x$ is in case (*1) or (*2), we deduce
that
\begin{equation}\label{eq7502}
    x_1 \ \mathrm{is} \ \mathrm{resolved} \ \mathrm{or} \ \gamma (x_1)<\gamma (x).
\end{equation}

For $\alpha \in \N$, $\alpha \geq 3$, we have similarly
$$
    \left \lceil {\alpha \over d} +{1 \over p}\right \rceil < \alpha .
$$
If $x$ is in case (*3), we deduce  from Lemma \ref{gamma2*3}(1) that (\ref{eq7502}) also holds.

\smallskip

\noindent {\it Assume that $x_1=x'$ is the point at infinity (\ref{eq7504}).}
By Lemma \ref{gamma2*infty}, $x_1$ is resolved or satisfies condition (*2).

\smallskip

If $x$ is in case (*1), Lemma \ref{gamma2*infty}(1) gives
\begin{equation}\label{eq7503}
\gamma (x_1)\leq 1+\left \lfloor {\beta (x) \over 2} \right \rfloor < \gamma (x)
\end{equation}
by (\ref{eq7501}), since $\beta (x)>2$.

\smallskip

If $x$ is in case (*2), Lemma \ref{gamma2*infty}(2) gives $C(x_1)\leq C(x)$, so $\gamma (x_1)\leq \gamma (x)$.

\smallskip

If $x$ is in case (*3), then Lemma \ref{gamma2*infty}(3) similarly gives
$$
\gamma (x_1)\leq 1+ \left \lfloor {1+\beta (x) \over 2}\right \rfloor
< 1 + \lfloor \beta (x)\rfloor = \gamma (x)
$$
since $\beta (x) \geq 2$. The conclusion is again (\ref{eq7502}).

\smallskip

\noindent {\it Assume that $x_1\neq x'$ and $k(x_1)=k(x)$.} If $x$ satisfies condition
(*1) or (*3), the independence statement in Theorem \ref{well2prepared} shows that we
may actually assume that $x_1=(Z/u_1,u_1,u_2/u_1,u_3/u_1)$.

If $x$ is in case (*1), then $x_1$ is resolved or satisfies again condition (*1) with $\beta (x_1)\leq \beta (x)$
by Lemma \ref{gamma2*12}(1).

If $x$ is in case (*3), then $x_1$ is resolved or satisfies one of conditions  (*1) or (*3). In
the latter case, we have $\beta (x_1)\leq \beta (x)$ by Lemma \ref{gamma2*3}(3); in
the former case, we have $\gamma (x_1)\leq \gamma (x)$ except if

\begin{equation}\label{eq7505}
``x \ \mathrm{satisfies} \ \mathrm{the} \ \mathrm{assumptions} \ \mathrm{of } \ \mathrm{Lemma } \  \ref{gamma2*3}(2)".
\end{equation}

This situation occurs only when $\beta (x)=\gamma (x)-1/\omega (x)$ and gives
$$
\beta (x_1)=\gamma (x)+1/\omega (x), \ \gamma (x_1)=\gamma (x)+1.
$$

We first prove the proposition when $x$ satisfies either condition (*1) or (condition (*3) with
$\beta (x)< \gamma (x)-1/\omega (x)$). By the above considerations, we are done except
possibly if $x_1$ satisfies again condition (*1) or (*3) with ($k(x_1)=k(x)$ and $\gamma (x_1)=\gamma (x)$).
Iterating, we conclude from Corollary \ref{permisarcthree} that $x_r$ is resolved or
$\gamma (x_r)<\gamma (x)$ for some $r\geq 1$.

\smallskip

Assume now that $x$ satisfies condition (*2). By the above considerations and Lemma \ref{gamma2*12}(4),
we are done except possibly if $x_1$ satisfies again condition (*2). Iterating, we conclude
from Lemma \ref{kappa2fin00} above that $x_r$ is resolved or $\gamma (x_r)<\gamma (x)$ for some $r\geq 1$.

\smallskip

Assume finally that $x$ satisfies condition (*3) with $\beta (x)= \gamma (x)-1/\omega (x)$.
By the above considerations, we are done except possibly if $k(x_1)=k(x)$ and (1) or (2) below holds:
\begin{itemize}
  \item [(1)] $x_1$ satisfies again condition (*3) with $\beta (x_1)=\beta (x)$;
  \item [(2)] $x_1$ satisfies condition (*1) with $\beta (x_1)=\gamma (x)+1/\omega (x)$, {\it viz.} (\ref{eq7505}).
\end{itemize}

Suppose that (2) holds; we now review the above proof with this extra assumption in mind.
Since $\beta (x_1)>3$, $\beta (x_1)\neq 4$, (\ref{eq7501}) or (\ref{eq7503})  applied to
the point $x_1$ give the stronger
$$
\gamma (x_2)< \gamma (x_1)-1=\gamma (x).
$$
We conclude that either $x_2$ is resolved, either $\gamma (x_2) < \gamma (x)$, or $x_2$ satisfies
again condition (*1) with $\beta (x_2)\leq \beta (x_1)$. If the latter inequality is strict,
we have $\beta (x_2)\leq \gamma (x)$ and we are thus already done. Otherwise $x_2$ satisfies again (2).

Summing up, there exists $r_0\geq 0$ such that either $x_{r_0}$ is resolved, either $\gamma (x_{r_0})< \gamma (x)$,
or ($x_r$ satisfies one and the same property (1) or (2) above for every $r\geq r_0$).
Iterating,  we conclude again by Corollary \ref{permisarcthree}.
\end{proof}

\begin{prop}\label{kappa2fin10}
Assume that $\kappa (x)=2$  and one of the following properties holds:
\begin{itemize}
  \item [(1)] $x$ satisfies condition (*1) with $\gamma (x) = 1$;
  \item [(2)] $x$ satisfies condition (*2) with $\beta (x) <1$;
  \item [(3)] $x$ satisfies condition (*3) with $\beta (x) < 1-1/ \omega (x)$.
\end{itemize}
Then $x$ is good.
\end{prop}

\begin{proof}
Note  that $A_1(x)>0$ if $x$ satisfies (2) or (3), since
$$
1 \leq B(x)\leq A_1(x)+\beta (x)
$$
in any case. If ($x$ satisfies condition (1) with $A_1(x)=0$), then $x$ is good by Proposition \ref{tauegaldeux}.
Applying repeatedly Lemma \ref{kappa2bupcurve} if $A_1(x)\geq 1$, it can be assumed w.l.o.g. that
\begin{equation}\label{eq752}
    0< A_1(x)<1.
\end{equation}
To prove the proposition, we first claim: $x_1$ is resolved or ($x_1$ satisfies again the assumptions
of the proposition and $c(x_1)\leq c(x)$ for the lexicographical ordering), where
$$
c(x):=(A_1(x), \beta (x)).
$$

If $x_1$ belongs to the first chart, i.e. $x_1$ is distinct from the point $x'$ at infinity (\ref{eq7504}),
we apply Lemma \ref{gamma2*12} and Lemma \ref{gamma2*3}. Note that the special situations
described in Lemma \ref{gamma2*12}(3') and in Lemma \ref{gamma2*3}(1')(2) do not occur under the assumptions
of the proposition, so we may also disregard them in this proof. We obtain that $x_1$ is resolved or
$x_1$ satisfies again condition (*) with
\begin{equation}\label{eq7521}
A_1 (x_1)=B(x)-1 \leq A_1(x)+\beta (x)-1 \leq A_1(x).
\end{equation}

\noindent {\it Assume that $x_1$ belongs to the first chart and $x$ satisfies (1)}. We have $C(x)\leq \beta (x)\leq 1$.
If $k(x_1)=k(x)$, it can be assumed that $x_1$ is the origin of the chart by the independence
statement in Theorem \ref{well2prepared}. By Lemma \ref{gamma2*12}(1) we have $\beta (x_1)\leq \beta (x)$
and the claim follows. Note that we obtain $c(x_1)=c(x)$ only if $\beta (x)=1$ by (\ref{eq7521}),
in which case $x_1$ satisfies again (1). If $k(x_1)\neq k(x)$, the claim follows
from Lemma \ref{gamma2*12}(4) with strict inequality $c(x_1)<c(x)$.

\smallskip

\noindent {\it Assume that $x_1$ belongs to the first chart and $x$ satisfies (2)}.
Since $\beta (x)<1$, inequality
is strict in (\ref{eq7521}). The claim also follows from Lemma \ref{gamma2*12}(1)(4)
with strict inequality $c(x_1)<c(x)$.

\smallskip

\noindent {\it Assume that $x_1$ belongs to the first chart and $x$ satisfies (3)}.
Note that if $x_1$ satisfies condition
(*1), then $x_1$ satisfies again the assumptions of the proposition since Lemma \ref{gamma2*3}(2) does not occur
for $\beta (x)<1-1/\omega (x)$; this is also true if $x_1$ satisfies condition (*3) by Lemma \ref{gamma2*3}(3)
(note that  $p=\omega (x)=2$ does not occur: (\ref{eq752}) gives $A_1(x)=1/2$ while (3) gives $\beta (x)=0$,
a contradiction with $B(x)\geq 1$).  The claim now follows with strict inequality $c(x_1)<c(x)$ by (\ref{eq7521}).

\smallskip

\noindent {\it Assume that $x_1=x'$ (point at infinity (\ref{eq7504})).} Turning to Lemma \ref{gamma2*infty}, $x'$ is
resolved or $x'$ satisfies condition (*2) with
$$
A_1(x')=A_1(x), \ \beta (x')=A_1(x)+\beta (x)-1 < \beta (x)
$$
by (\ref{eq752}). This proves the claim with $c(x_1)<c(x)$ in this case.

\smallskip

Summing up, we have proved the claim with strict inequality $c(x_1)<c(x)$ except possibly if
both $x$ and $x_1$ are in case (*1), $k(x_1)=k(x)$  and $\beta (x_1)=\beta (x)=1$.
One concludes the proof again  by Corollary \ref{permisarcthree}.
\end{proof}

\begin{prop}\label{kappa2fin11}
Assume that $\kappa (x)=2$, $x$ satisfies condition (*2) and $\gamma (x)=1$.
Then $x$ is good.
\end{prop}

\begin{proof}
By Lemma \ref{gamma2*12}(4), $x_1$ is resolved or satisfies the assumptions
of Proposition \ref{kappa2fin10}(1) or (3) if $x_1$ is not a point at infinity. Therefore $x_1$ is resolved in this case.
If $x_1$ is the origin of a chart, then $x_1$ is resolved or satisfies again the assumptions
of this proposition by Lemma \ref{gamma2*infty}(2).

\smallskip

Applying Lemma \ref{kappa2fin00}, it can thus be assumed that $C(x)=0$. Applying repeatedly Lemma \ref{kappa2bupcurve}
if $A_1(x)\geq 1$ or if $A_2(x)\geq 1$, we then reduce to the case
$$
0 \leq A_1(x),A_2(x)< 1, \ C(x)=0.
$$
Then  $\beta (x)=A_2(x)<1$ and the conclusion follows from Proposition \ref{kappa2fin10}(2).
\end{proof}

\begin{prop}\label{kappa2fin20}
Assume that $\kappa (x)=2$ and one of the following properties holds:
\begin{itemize}
  \item [(i)] $x$ satisfies condition (*1) with $\beta (x) < 2$;
  \item [(ii)] $x$ satisfies condition (*3), $\beta (x) =1-1/\omega (x)$ and $(p,\omega (x))\neq (2,2)$.
\end{itemize}
Then $x$ is good.
\end{prop}

\begin{proof}
Note that the special situations described in Lemma \ref{gamma2*3}(1')(2) do occur here.

\smallskip

\noindent {\it Assume that $x_1$ belongs to the first chart}. Under assumption (i), $x_1$ is resolved or
$x_1$ satisfies condition (*1) or (*3); note that the latter occurs only if $k(x_1)$ is
an inseparable extension of $k(x)$ (in particular $d\geq p$) and $d_1\in \N$. Then by Lemma \ref{gamma2*12}(4):
$$
\beta (x_1) <  1+ \lfloor {C(x)\over d}\rfloor -{1 \over \omega (x)}  \leq  1+ \lfloor {\beta(x)\over p}\rfloor -{1 \over \omega (x)}=1
 -{1 \over \omega (x)},
$$
so $x_1$ is resolved by \ref{kappa2fin10}(3). When $x_1$ satisfies condition (*1), by Lemma \ref{gamma2*12}(4),
$x_1$ satisfies again assumption (i) of this proposition with $k(x_1)=k(x)$ by Lemma \ref{gamma2*12}(4),
or is resolved by Proposition \ref{kappa2fin10}(1)(3).

\smallskip

Under assumption (ii), $x_1$ is resolved or $x_1$ satisfies condition (*1) or (*3). If $x$
is as stated in Lemma \ref{gamma2*3}(1'), then $x_1$ is resolved or satisfies assumption (i)
with $\beta (x_1)=p/(p-1)<2$, since $(p,\omega (x))\neq (2,2)$.

Otherwise we may apply Lemma  \ref{gamma2*3}(1)-(3): if $x_1$ satisfies condition (*1),
we get $\beta (x_1)\leq 1+1/p$, $\beta (x_1)\leq 1$ if $k(x')\neq k(x)$, from Lemma  \ref{gamma2*3}(1);
if $x_1$ satisfies condition (*3),  we get $\beta (x_1)\leq \beta (x)$, strict inequality if $k(x')\neq k(x)$,
from Lemma  \ref{gamma2*3}(2)(3). By Proposition \ref{kappa2fin10}(1)(3), $x_1$ is resolved or satisfies again
the assumptions of the proposition with $k(x_1)=k(x)$.

\smallskip

\noindent {\it Assume that $x_1=x'$ is the point at infinity.} Then $x_1$ is resolved or $x_1$ satisfies
condition (*2) with $C(x_1)<1$ by Lemma \ref{gamma2*infty}(1)(3); therefore $x_1$ is resolved
in any case by Proposition \ref{kappa2fin11}.

\smallskip

One concludes the proof again  by Corollary \ref{permisarcthree}.
\end{proof}

\begin{lem}\label{kappa2fin21}
Assume that $\kappa (x)=2$ and one of the following properties holds:
\begin{itemize}
  \item [(i)] $x$ satisfies condition (*1) with $\beta (x) = 2$;
  \item [(ii)] $x$ satisfies condition (*3) with $\beta (x)< 2$.
\end{itemize}
Let $(u_1,u_2;u_3;Z)$ be well 2-adapted coordinates at $x$ and
$$
x':=(Z':=Z/u_2,u'_1:=u_1/u_2,u_2,u'_3:=u_3/u_2)
$$
be the point at infinity. Then $x'$ is resolved or ($x'$ satisfies condition (*2) with $C(x')=1$ and the
following respectively hold:)
\begin{itemize}
  \item [(i')] $p=2$ and $d_1\not \in \N$;
  \item [(ii')] $p\geq 3$.
\end{itemize}
\end{lem}

\begin{proof}
By Lemma \ref{gamma2*infty},  $x'$ is resolved or $x'$ satisfies condition (*2).

\smallskip

Under assumption (i),  Lemma \ref{gamma2*infty}(1) furthermore gives $C(x') \leq 1$;
if $C(x')<1$, we are done by Proposition \ref{kappa2fin11}. If $C(x')=1$,
Lemma \ref{gamma2*infty}(1) implies that $C(x)=1$; moreover
\begin{equation}\label{eq7531}
A_1(x')=A_2(x')=A_1(x), \ C(x') = \beta (x')-A_2(x')=1.
\end{equation}

We now prove that $x'$ is resolved unless ($p=2$ and $d_1\not \in \N$). To prove this,
it is sufficient to prove that any possible $x_2$ in (\ref{eq750}) is resolved when $x_1=x'$.
Note that $(u'_1,u_2;u'_3;Z')$ are well 2-adapted coordinates at $x'$. Let
$$
\mathrm{in}_{\alpha'} h ={Z'}^p -{G'}_{\alpha '}^{p-1}Z'+ F_{p,Z',\alpha '},
$$
notations as in Lemma \ref{structDelta2}(4) w.r.t. the face $\sigma_{2, \mathrm{in}}$ of
$\Delta_2(h';u'_1,u_2;u'_3;Z')$.  We expand
\begin{equation}\label{eq7532}
{U'_1}^{-pd_1}{U_2}^{-pd'_2}F_{p,Z',\alpha '}=
\mu {U'_3}^{\omega (x)}+\sum_{i=1}^{\omega (x)} \mu_i{U'_3}^{\omega (x)-i}P_i(U'_1,U_2),
\end{equation}
where $d'_2:=d_1+\omega (x)/p-1$ and
$$
P_i(U'_1,U_2)={U'_1}^{a_i}U_2^{b_i}Q_i(U'_1,U_2),
$$
with $Q_i(U'_1,U_2)$ zero or not divisible by either $U'_1$ or $U_2$. Since $C(x')=1$, we have by definition
$$
a_i \geq iA_1(x_1), \  b_i \geq iA_2(x_1), \ i\geq \mathrm{deg}Q_i(U'_1,U_2)
$$
whenever $Q_i(U'_1,U_2)\neq 0$, $1 \leq i \leq \omega (x)$. Since
$$
C(x)=C(x')=\beta (x')-A_2(x')=1,
$$
we have
\begin{equation}\label{eq7533}
\mathrm{deg}_{U'_1}Q_{i_1}=i_1 \ \mathrm{and} \ \mathrm{deg}_{U_2}Q_{i_2}=i_2
\end{equation}
for some $i_1,i_2$, $1 \leq i_1,i_2\leq \omega (x)$. Let
$$
x'_2:=(Z'/u_2,u'_1/u_2,u_2,u'_3/u_2), \ x''_2:=(Z'/u'_1,u'_1,u_2/u'_1,u'_3/u'_1)
$$
be the points ``at infinity''. If $x_2\in \{x'_2,x''_2\}$, then Lemma \ref{gamma2*12}(1) implies that $x_2$ is resolved
or $x_2$ satisfies condition (*2) with $C(x_2)=0$ by (\ref{eq7533}). So $x_2$ is resolved
in any case by Proposition \ref{kappa2fin11}.

\smallskip

If $x_2\not \in \{ x'_2, x''_2\}$ and $k(x_2)\neq k(x')$, we apply Lemma \ref{gamma2*12}(4):
then $x_2$ is resolved by Proposition \ref{kappa2fin10}(1)(3).

\smallskip

If $x_2\not \in \{ x'_2, x''_2\}$ and $k(x_2)= k(x')$, we apply Lemma \ref{gamma2*12}(3')(3)(4). Note that
the special situation in Lemma \ref{gamma2*12}(3') yields $x_2$ resolved if $(p,\omega (x))\neq (2,2)$
by Proposition \ref{kappa2fin20}(i). Therefore $x_2$ is resolved or one of the following properties holds:

\smallskip

$(A)$  $x'$ satisfies the requirements in Lemma \ref{gamma2*12}(3') for $p=\omega (x)=2$ and
$x_2$ satisfies (\ref{eq742}) (in particular $d_1\not\in \N$);

\smallskip

$(B)$  $x_2$ satisfies condition (*3) with $\beta (x_2)\leq 1+1/p$.

\smallskip

\noindent Since $(d_1, 0,\omega (x)/p)$ is a vertex of $\Delta (h;u_1,u_2,u_3;Z)$ which is not solvable, we have
$\mu\not \in k(x)^p$ in (\ref{eq7532}) if $d_1\in \N$. As $k(x_2)=k(x')$, $x_2$ satisfies condition (*3) only if
$$
(d_1,d'_2) \not \in \N^2 \ \mathrm{and} \ d_1+d'_2 \in \N.
$$
On the other hand $d'_2 - d_1 =\omega (x)/p-1 \in \N$, so the latter holds if and only if ($p=2$ and $d_1\not \in \N$)
as required.\\

Under assumption (ii), we are done by Proposition \ref{kappa2fin11} if $C(x')<1$. Assuming that $C(x')\geq 1$, we have
$$
1\leq \max\{\beta (x)-C(x),C(x)-\beta_2(x)\}<2
$$
by Lemma \ref{gamma2*infty}(3). It is easily deduced that
\begin{equation}\label{eq7534}
\beta (x')-A_2(x')=\beta (x)-C(x)<2
\end{equation}
and that
\begin{equation}\label{eq7536}
\beta (x)\geq 1, \ 0\leq C(x)\leq 1-1/\omega (x)  \ \mathrm{and} \ \beta_2(x)\leq  -1/\omega (x).
\end{equation}
The proof is now a variation of that under assumption (i) and we explain now how it is to be adapted.
To begin with, (\ref{eq7532}) holds with $d'_2:=d_1+(1+\omega (x))/p-1$.
Since $C(x)<1$, $\beta_2(x)<0$ and $C(x')\geq 1$, (\ref{eq7533}) is now replaced by
\begin{equation}\label{eq7535}
\mathrm{deg}_{U'_1}Q_{i_1}=i_1 \ \mathrm{for} \ \mathrm{some} \ i_1, \ 1 \leq i_1\leq \omega (x).
\end{equation}
Note in particular that we have $C(x')=1$.

\smallskip

If $x_2\in \{x'_2,x''_2\}$, we apply Lemma \ref{gamma2*infty}: $x'_2$ (resp. $x''_2$) is resolved or
$C(x'_2)<1$ (resp. $C(x''_2)=0$) by (\ref{eq7534}) (resp. by (\ref{eq7535})). Therefore $x_2$ is resolved
in any case by Proposition \ref{kappa2fin11}.

\smallskip

If $x_2\not \in \{ x'_2, x''_2\}$ and $k(x_2)\neq k(x_1)$, then $x_2$ is resolved by the same argument
as under assumption (i).

\smallskip

If $x_2\not \in \{ x'_2, x''_2\}$ and $k(x_2)= k(x_1)$, we first note that $x'$ is {\it not}
as specified in Lemma \ref{gamma2*12}(3'): since $C(x)<1$, we have $A_1(x')=A_1(x)>0$. Applying then
Lemma \ref{gamma2*12}(3)(4), the argument used under assumption (i) gives
$x_2$ resolved or $d_1+d'_2 \in \N$. Since $d'_2 - d_1 -1/p\in \N$,
this can possibly hold only if $p\geq 3$.
\end{proof}

\begin{lem}\label{kappa2fin22}
Assume that $\kappa (x)=2$ and $x$  {has} one of the following properties:
\begin{itemize}
  \item [(i)] $x$ satisfies condition (*1), $\beta (x) = 2$ and, given well 2-adapted coordinates
  $(u_1,u_2;u_3;Z)$, the polynomial $\mathrm{in}_\alpha h =Z^p -G_\alpha^{p-1}Z+F_{p,Z,\alpha}$, where
\begin{equation}\label{eq754}
U_1^{-pd_1}F_{p,Z,\alpha}=\mu U_3^{\omega (x)}+\sum_{i=1}^{\omega (x)}
\mu_iU_3^{\omega (x)-i}U_1^{iy_1}U_2^{iy_2},
\end{equation}
notations as in Lemma \ref{structDelta2}(4) w.r.t. the face
$$
\sigma_2=\mathbf{y} :=(A_1(x), \beta (x))\in \Delta_2(h;u_1,u_2;u_3;Z)
$$
has $\mu_i \neq 0$ for some $i$ with $1 \leq i \leq p-1$;
  \item [(ii)] $x$ satisfies condition (*3) and $\beta(x)<2-1/p$.
\end{itemize}
Then $x$ is good.
\end{lem}

\begin{proof}
We again consider three cases.

\smallskip

\noindent {\it Assume that $x_1=x'$ is the point at infinity.} We review the proof of Lemma \ref{kappa2fin21}
with our extra assumptions and claim that $x'$ is resolved.

\smallskip

Under assumption (i), we get $1\leq i_2 \leq p-1$ in (\ref{eq7533}) by (\ref{eq754}). Turning to (A) and (B)
in the proof of Lemma \ref{kappa2fin21}, note that (A) does not hold since $\mu_1\neq 0$ in (\ref{eq754}).
Finally if (B) holds, then $\beta (x_2)\leq 1-1/(p-1)$ because $1\leq i_2 \leq p-1$. Therefore
$x_2$ is resolved by Proposition \ref{kappa2fin10}(3).

\smallskip

Under assumption (ii), note that (\ref{eq7536}) is strengthened to
$$
0\leq C(x)< 1-1/p \ \mathrm{and}  \ \beta_2(x)<  -1/p
$$
since $\beta (x)<2-1/p$. We thus get $1\leq i_1 \leq p-1$ in (\ref{eq7535}).
We also get $\beta (x_2)\leq 1-1/(p-1)$ if (B) holds,
so $x_2$ is resolved by Proposition \ref{kappa2fin10}(3).\\

\noindent {\it Assume that $k(x_1)\neq k(x)$.} If $x_1$ satisfies condition (*1),
Lemma \ref{gamma2*12}(4) and Lemma \ref{gamma2*3}(1) give $\beta (x)<2$ in any case. Therefore
$x_1$ is resolved by Proposition \ref{kappa2fin20}(i).

If $x_1$ satisfies condition (*3), the same conclusion holds under assumption (i) except possibly if
$C(x)=d=2$. By (\ref{eq754}), we then get $\beta (x_1)\leq 1-1/(p-1)$ and $x_1$ is resolved by
Proposition \ref{kappa2fin10}(3). Under assumption (ii), $x_1$ satisfies again the assumption (ii)
in this lemma with $\beta (x_1)<\beta (x)$ by Lemma \ref{gamma2*3}(3). \\

\noindent {\it Assume that $x_1\neq x'$ and $k(x_1)= k(x)$.} The independence statement
in Theorem \ref{well2prepared} reduces to
$$
x_1=(Z':=Z/u_1, u_1,u'_2:=u_2/u_1,u'_3:=u_3/u_1).
$$
Note that the extra assumption  (\ref{eq754}) is unaffected by this coordinate change.

\smallskip

Under assumption (i), Lemma \ref{gamma2*12}(1) shows that $x_1$ is resolved or $x_1$ satisfies again condition (*1)
with $\beta (x_1)\leq \beta (x)=2$.
By Proposition \ref{kappa2fin20}(i), $x_1$ is resolved unless equality holds.
In this case, we have
$$
C(x)=\beta (x)=\beta (x_1)=2
$$
and $x_1$ satisfies again assumption (i) of this lemma.

\smallskip

Under assumption (ii), Lemma \ref{gamma2*3} shows that $x_1$ is resolved or satisfies condition (*1)
or (*3). If one of Lemma \ref{gamma2*3}(1')(2) applies, we have $\gamma (x)=1$
and $x_1$ satisfies condition (*1) with $\beta (x_1)\leq 2$. We are done if inequality is strict
by Proposition \ref{kappa2fin20}(i); otherwise $\omega (x)=p=2$ and $x_1$ satisfies (i) of this lemma.

\smallskip

Any other situation yields $\gamma (x_1)\leq \gamma (x)$. If $x_1$ satisfies condition (*3), then
$x_1$ satisfies again (ii) of this lemma with $\beta (x_1)\leq \beta (x)$ by Lemma \ref{gamma2*3}(3).
If $x_1$ satisfies condition (*1), we have $\beta (x_1)\leq 2$. We are done if inequality
is strict by Proposition \ref{kappa2fin20}(i).

\smallskip

Assume then that ($x_1$ satisfies condition (*1) and $\beta (x_1)=2$). We argue as in the proof of
Lemma \ref{gamma2*3}. Denote $F_{p,Z,\alpha}$ as in Lemma \ref{structDelta2}(5). We have:
\begin{equation}\label{eq7543}
U_1^{-pd_1}F_{p,Z,\alpha}=(\mu U_1+U_2)U_3^{\omega (x)}
+\sum_{j\in J_0}U_3^{\omega (x)-j}U_1^{b_j}\Psi_j(U_1,U_2),
\end{equation}
where $\mu \in k(x)$ and
$$
b_j \geq jA_1(x), \ j\beta (x) \geq \mathrm{deg}_{U_2}\Psi_j(U_1,U_2) -1 .
$$
By assumption (ii), we have
$$
j\in J_0 \Longrightarrow {\mathrm{deg}_{U_2}\Psi_j(U_1,U_2) -1 \over j}< 2-1/p.
$$
Note that for $j\in J_0$, we then have $\mathrm{deg}_{U_2}\Psi_j(U_1,U_2) \leq 2j$, and inequality is strict if $j\geq p$.
If $\min J_0\geq p$, arguing as in the proof of Lemma \ref{gamma2*3} ($B(x)>1$, cases 1 to 4),
we then get $\beta (x_1)<2$: a contradiction. This proves that
\begin{equation}\label{eq7541}
1 \leq j_0:=\min J_0 \leq p-1.
\end{equation}
Let $\mathbf{y}':=(A_1(x_1), \beta (x_1))\in \Delta_2(h';u_1,u'_2;u'_3;Z')$,
where $(u_1,u'_2;u'_3;Z')$ are well 2-adapted coordinates. With notations as in Lemma \ref{structDelta2}(4),
the initial form polynomial $\mathrm{in}_{\alpha '} h'$ w.r.t. the face $\sigma'_2=\mathbf{y}'$
satisfies an equation (\ref{eq754}), say
\begin{equation}\label{eq7542}
{U_1}^{-pd'_1}F_{p,Z',\alpha '} = \mu '{U'_3}^{\omega (x)}+\sum_{j=1}^{\omega (x)}
\mu'_j{U'_3}^{\omega (x)-j}{U_1}^{jA_1(x_1)}{U'_2}^{2j},
\end{equation}
with $d'_1:=d_1+(1+\omega (x))/p-1$, $\mu'_{j_0} \neq 0$ by (\ref{eq7541}).
Therefore  $x_1$ satisfies assumption (i) in this lemma.\\

Summing up, the following  has been proved: if $x$ satisfies (i), then $x_1$ is resolved or
($k(x_1)=k(x)$ and $x_1$ satisfies again (i)). If $x$ satisfies (ii), then $x_1$ is resolved
or $x_1$ satisfies (i) or (ii); if (ii) holds, then $\beta (x_1)\leq \beta (x)$
and inequality is strict if  $k(x_1)\neq k(x)$.

\smallskip

Consider the quadratic sequence (\ref{eq750}). By the previous considerations,
there exists $r_0\geq 0$ such that either $x_{r_0}$ is resolved, or ($x_r$ satisfies
one and the same assumption in the lemma with $k(x_r)=k(x_{r_0})$ for every $r\geq r_0$).
One concludes the proof again  by Corollary \ref{permisarcthree}.
\end{proof}

We will now conclude the proof of Theorem \ref{proofkkappa2}. Note the interesting extra twist
for $p=2$.

\begin{prop}\label{kappa2fin23}
Assume that $\kappa (x)=2$ and one of the following properties holds:
\begin{itemize}
  \item [(i)] $x$ satisfies condition (*1) with $\beta (x) = 2$;
  \item [(ii)] $x$ satisfies condition (*3) and $\beta(x)<2 -1/\omega (x)$.
\end{itemize}
Then $x$ is good.
\end{prop}

\begin{proof}
This is a variation on the two previous lemmas. Note that we may disregard
the special case stated in Lemma \ref{gamma2*3}(1') in this proof.

\smallskip

\noindent {\it Assume that $x_1=x'$ is the point at infinity.}  By Lemma \ref{kappa2fin21},
$x'$ is resolved under assumption (i) (resp. (ii)) if $p\geq 3$ (resp. if $p=2$).
Reviewing the proof of Lemma \ref{kappa2fin21}, we are done except possibly when (A) or (B)
stated therein hold. If (A) holds, then $x_2$ is resolved by Lemma \ref{kappa2fin22}(i). If
(B) holds, $x_2$ satisfies condition (*3) with $\beta (x_2)\leq 1+1/p$. If $p\geq 3$ or if
($p=2$ and $\beta (x_2)<3/2$), we have $\beta (x_2)<2-1/p$ and the conclusion follows
from Lemma \ref{kappa2fin22}(ii). Therefore $x'$ is resolved or
$$
p=2 \  \mathrm{and} \  \beta (x_2)=3/2.
$$
In the special case $p=\omega (x)=2$, an explicit computation gives $\beta (x_2)\leq 1$
if $x_2$ satisfies condition (*3) ({\it cf.} (ii) of proof of Lemma \ref{kappa2fin25} below),
so $x'$ is resolved. This proves that  $x_2$ is resolved or satisfies again
the assumptions of the proposition in any case.\\

\noindent {\it Assume that $k(x_1)\neq k(x)$.} Under assumption (i), $x_1$ is resolved or
$$
\beta (x) \leq {C(x) \over d}+{1\over p} \leq 1+{1\over p}
$$
by Lemma \ref{gamma2*12}(3). Then $x_1$ is resolved by Proposition \ref{kappa2fin20}(i) or
by Lemma \ref{kappa2fin22}(ii) except possibly if $x_1$ satisfies (ii) with ($p=2$, $\beta (x)=3/2$);
in this case, note that ($x$ satisfies condition (*1), $x_1$ satisfies condition (*3)) implies
that $d_1\in \N$.

\smallskip

Under assumption (ii), $x_1$ is resolved or
$$
\beta (x) < {2 \over d}+{1\over p} \leq 1+{1\over p}
$$
by Lemma \ref{gamma2*12}(2). Then $x_1$ is resolved in any case by Proposition \ref{kappa2fin20}(i) or
by Lemma \ref{kappa2fin22}(ii). \\

\noindent {\it Assume that $x_1\neq x'$ and $k(x_1)= k(x)$.} We may assume once again that $x_1$ is the
origin of the first chart of the blowing up.

\smallskip

Under assumption (i), $x_1$ is resolved or $x_1$ satisfies again assumption (i): same proof as in
Lemma \ref{kappa2fin22}(i).

\smallskip

Under assumption (ii), $x_1$ is resolved or satisfies again one of (i)(ii): same proof as in
Lemma \ref{kappa2fin22}(ii). If $x_1$ satisfies again (ii), we have $\beta (x_1)\leq \beta (x)$
by Lemma \ref{gamma2*3}(3).

\smallskip

Summing up, it has been proved that $x_1$ is resolved or $x_1$ satisfies again the assumptions of the
proposition. Under assumption (i), $x_1$ is resolved or one of the following properties holds:
\begin{itemize}
  \item [(1)] $k(x_1)=k(x)$ and $x_1$ satisfies again (i);
  \item [(2)] $p=2$ and $x_1$ satisfies (ii) with $\beta (x_1)=3/2$;
  \item [(3)] $p=2$ and  $x_2$ satisfies (ii) with $\beta (x_2)=3/2$.
\end{itemize}

Under assumption (ii), $x_1$ is resolved or one of the following properties holds:
\begin{itemize}
  \item [(1')] $k(x_1)=k(x)$ and $x_1$ satisfies (i);
  \item [(2')] $k(x_1)=k(x)$ and $x_1$ satisfies again (ii) with $\beta (x_1)\leq \beta (x)$.
\end{itemize}

\smallskip

Consider the quadratic sequence (\ref{eq750}) and suppose that (2) (resp. (3)) above occurs.
Suppose that event (1') occurs again at $x_r$ for $r\geq 1$ (resp. for $r\geq 2$). By (2') and Lemma
\ref{kappa2fin22}(ii), we may assume that $\beta (x_r)=3/2$, so
$x_r$ is resolved by Lemma \ref{kappa2fin25} below. Therefore there exists $r_0\geq 0$
such that either $x_{r_0}$ is resolved, or ($x_r$ satisfies
one and the same assumption (i) or (ii) with $k(x_{r})=k(x_{r_0})$ for every $r\geq r_0$).
The proof now concludes once again by Corollary \ref{permisarcthree}.
\end{proof}

\begin{lem}\label{kappa2fin25}
Assume that $p=2$, $\kappa (x)=2$ and $x$ satisfies condition (*3) with $\beta (x)=3/2$. If
$x_1$ satisfies condition (*1), then $x_1$ is resolved.
\end{lem}

\begin{proof}
We argue as in the proof of Lemma \ref{kappa2fin22} (\ref{eq7543}) and (\ref{eq7542}):
we have $\beta (x_1)=2$ and, since $\beta (x)=3/2$, there exist well 2-adapted coordinates
$(u_1,u'_2;u'_3;Z')$ at $x_1$ such that
\begin{equation}\label{eq756}
{U_1}^{-2d'_1}F_{2,Z',\alpha '} = \mu '{U'_3}^{\omega (x)}+\sum_{j=1}^{\omega (x)}
\mu'_j{U'_3}^{\omega (x)-j}{U_1}^{jA_1(x_1)}{U'_2}^{2j},
\end{equation}
with $d'_1:=d_1+(1+\omega (x))/2-1$, $\mu'_1 \neq 0$  or $\mu'_2 \neq 0$.
We conclude by Lemma \ref{kappa2fin22}(i) if $\mu'_1\neq 0$.

\smallskip

Assume then that $\mu'_1= 0$ and let $a:=\mathrm{ord}_2\omega (x)$. If ($a=1$, $A_1(x)\in \N$
and $\mu'_2 {\mu'}^{-1}=\lambda^2$ for some $\lambda \in k(x)$), we may perform the
Tschirnhausen transform $U'_3\mapsto U'_3 +\lambda U_1^{A_1(x)}{U'_2}^2$ and get $\mu'_2 = 0$
in (\ref{eq756}). Since $\beta (x_1)=2$, we nevertheless obtain $\mu'_{j_0} \neq  0$
for some $j_0\geq 3$ in (\ref{eq756}). In other terms, we may assume that
one of the following assumptions holds:
\begin{itemize}
  \item [(i)] $a\geq 2$ and $\mu'_2\neq 0$;
  \item [(ii)] $a=1$, ($A_1(x)\not \in \N$ or $\mu'_2 {\mu'}^{-1}\not \in k(x)^2$) and $\mu'_2\neq 0$;
  \item [(iii)] $a=1$, $A_1(x) \in \N$, $\mu'_2= 0$ and $\mu'_{j_0}\neq 0$ for some $j_0\geq 3$.
\end{itemize}

We consider three cases and review again the proof of Lemma \ref{kappa2fin22}:

\smallskip

\noindent {\it Assume that $x_2=x'_1$ is the point at infinity.} Situation (A) has been solved
in Lemma \ref{kappa2fin22}(i). Situation (B) does not hold by \cite{CoP2} proof of {\bf I.8.3}:
equality $\beta (x_3)=3/2$ is achieved only in the situation of {\it ibid.} {\bf I.8.3.6} case 2.
This implies ($\mu'_{j}=0$ for $1 \leq j \leq 2^a-1$, and $\mu'_{2^a}\neq 0$): a
contradiction with (i) and (iii) above. This also implies $B(x)=A_1(x)+\beta (x)\in \N$
{\it viz.} \cite{CoP2} {\bf I.8.3.4} (so $A_1(x)\in \N$ since $\beta (x)=2$), and
$$
{U'_1}^{-2d'_1}{\partial F_{2,Z',\alpha '} \over \partial \lambda_l }
\in <{U'_1}^{-2d'_1}U'_1{\partial F_{2,Z',\alpha '} \over \partial U'_1 }>, \ l \in \Lambda_0
$$
{\it viz.} \cite{CoP2} {\bf I.8.3.5} where $d'_1 \not \in \N$ here: a
contradiction with (ii). One gets $\beta (x_3)<3/2$ (actually: $\beta (x_3)\leq 1$ if $x_3$ satisfies condition (*3)),
so $x'$ is resolved by Lemma \ref{kappa2fin22}(ii).

\smallskip

\noindent {\it Assume that $k(x_2)\neq k(x_1)$.} Then $x_2$ is resolved.

\smallskip

\noindent {\it Assume that $x_2\neq x'_1$ and $k(x_2)= k(x_1)$.} Then $x_2$ is resolved or $x_2$ satisfies
again (\ref{eq756}) with $\mu'_j\neq 0$ for some $j\geq 1$, $j\leq 2$ if $a\geq 2$.

\smallskip

Iterating, the conclusion follows again from Corollary \ref{permisarcthree}.\\

\begin{prop}\label{kappa2fin24}
Assume that $\kappa (x)=2$, $x$ satisfies condition (*2) with $\gamma (x)=2$.
Then $x$ is good.
\end{prop}

\noindent {\it Proof.} By Lemma \ref{gamma2*12}, $x_1$ is resolved or satisfies again condition (*)
with $\gamma (x_1)\leq 2$.

If $x_1$ satisfies condition (*1), then $x_1$ is resolved by Proposition \ref{kappa2fin20}(i)
or by Proposition \ref{kappa2fin23}(i).

If $x_1$ satisfies condition (*3), we have $\beta (x_1)<2 -1/\omega (x)$
by Lemma \ref{gamma2*12}(4). Therefore $x_1$ is resolved by Proposition \ref{kappa2fin23}(ii).

If $x_1$ satisfies condition (*2) and $\gamma (x_1)=1$, $x_1$ is resolved by Proposition \ref{kappa2fin11}.
Therefore $x_1$ is resolved or satisfies again the assumptions of the lemma. The conclusion follows from
Lemma \ref{kappa2fin00}.
\end{proof}

\begin{prop}\label{kappa2fin3}
Assume that $\kappa (x)=2$, $x$ satisfies condition (*3) with $\beta (x)=2- 1/\omega (x)$.
Then $x$ is good.
\end{prop}

\begin{proof}
This is now a variation of Proposition \ref{kappa2fin0}. By Lemma \ref{gamma2*3},
$x_1$ is resolved or satisfies again condition (*) with $\gamma (x_1)\leq 2$ except in the special
situation specified in Lemma \ref{gamma2*3}(2). Applying the previous lemmas,
we are done except possibly if $k(x_1)=k(x)$ and (1) or (2) below holds:
\begin{itemize}
  \item [(1)] $x_1$ satisfies again condition (*3) with $\beta (x_1)=\beta (x)=2-1/\omega (x)$;
  \item [(2)] $x_1$ satisfies condition (*1) with $\beta (x_1)=2+1/\omega (x)$.
\end{itemize}

Suppose that (2) holds; by Lemma \ref{gamma2*12}(1)(4) and Lemma \ref{gamma2*infty}(2), $x_2$ is resolved
($\gamma (x_2) \leq 2$, $\beta (x_2)<2-1/\omega (x)$ if $x_2$ satisfies condition (*3)) or satisfies again
(2) with $k(x_2)=k(x_1)$. In both cases (1)(2), we conclude once more by Corollary \ref{permisarcthree}.
\end{proof}

\bigskip \bigskip

\section{Projection Theorem: transverse  and tangent cases, reduction of $\kappa (x)=3, 4$ to monic expansions.}
\medskip

In this chapter and the next one, we prove Theorem \ref{projthm} when $\kappa (x)=3,4$ (Definition \ref{defkappa}).
This is restated as Theorem \ref{proofkappa34} below.
The structure of the proof is similar to that of Theorem \ref{proofkkappa2}: first getting a stable form
for the equation of $\mathrm{in}_{m_S}h$ (i.e. monic expansions, Definition \ref{**} below), then introducing
a projected polygon with secondary  invariant $\gamma (x)$.

\smallskip

Two important differences with $\kappa (x)=2$ arise. On the one hand, no simple reduction
works for each of $\kappa (x)=3,4$ separately and we have to deal with both cases at the same time.
On the other hand, the monic case is resolved by blowing up Hironaka-permissible centers ${\cal Y}\subset {\cal X}$
which are not necessarily permissible in the sense of Definitions \ref{deffirstkind} and \ref{defsecondkind}. \\

Given a valuation $\mu$ of $L=k({\cal X})$ centered at $x$, we consider finite sequences of
local blowing ups along $\mu$:
\begin{equation}\label{eq801}
    ({\cal X},x)=:({\cal X}_0,x_0) \leftarrow ({\cal X}_1,x_1)\leftarrow \cdots \leftarrow ({\cal X}_r,x_r)
\end{equation}
with Hironaka-permissible centers ${\cal Y}_i \subset ({\cal X}_i,x_i)$,  {\it viz.} (\ref{eq402}). \\

\noindent {\it Up to the end of this chapter, ``resolved" stands for ``resolved for $(p,\omega (x),3)$"
(Remark \ref{quadsequence})}.\\

\begin{defn}\label{**} \textrm{(Monic expansion for $\kappa (x)\geq 3$).}
Assume that $\kappa(x)\geq 3$. We say that $x$ satisfies condition (**) \index{(**) $\kappa(x)\geq 3$}if  there exists well adapted
coordinates $(u_1,u_2,u_3;Z)$ at $x$ such that the following conditions are fulfilled:
\begin{itemize}
  \item [(i)] $1+\omega(x)\not=\ 0\ \mod(p)$;
  \item [(ii)] $E=\div (u_1)$ (resp. $E=\div(u_1u_2))$, and
  $\mathbf{v}:=(d_1, d_2, (1+ \omega (x))/p)$ is the only vertex (resp. is a vertex) of $\Delta_{S} (h;u_1,u_2,u_3;Z)$
  in the region $x_1=d_1$,  {with the usual convention $d_2=0$ when $\div (u_2)\not\subset E$}.
\end{itemize}

Assume $\kappa(x)=4$, we say that $x$ satisfies   condition (T**) \index{(T**) @ (T**) $\kappa(x)=4$} (for ``towards (**)'') if  there exists well adapted
coordinates $(u_1,u_2,u_3;Z)$ at $x$ such that {\it one} of the following conditions is fulfilled:

\begin{itemize}
  \item [(i)] $\epsilon (x)=\omega (x)$, $\div(u_1)\subseteq E$ and $\mathrm{Vdir}(x)=<U_1>$;
  \item [(ii)] $\epsilon (x)=\omega (x)$, $\div(u_1u_2)\subseteq E$ and $\mathbf{v}:=(d_1+ \omega (x)/p, d_2, d_3)$
  is the only vertex of $\Delta_{S} (h;u_1,u_2,u_3;Z)$ in the region $x_2=d_2$;
  \item [(iii)] $E=\div(u_1u_2)$ and $\mathbf{v}:=(d_1+ \omega (x)/p, d_2, 1/p)$ is the only vertex
  of $\Delta_{S} (h;u_1,u_2,u_3;Z)$ in the region $x_2=d_2$.
\end{itemize}

\end{defn}

When $x$ satisfies any of  (**) or (T**), we simply say  that ``$h$ has a {\it monic expansion} for $(u_1,u_2,u_3;Z)$''. \index{monic3@ monic expansion in case (**) or (T**),! $h$ has a {\it monic expansion} for $(u_1,u_2,u_3;Z)$}
In cases (**) and (T**)(iii), the nonexceptional variable $u_3$ will usually be denoted $v$.

\begin{rem}\label{remT**}
First, let us remark that, by Definitions \ref{defomega} and \ref{defkappa},
 {\begin{equation}\label{eq:kappa=3}
 \kappa (x)=3 \Rightarrow H^{-1}{\partial F_{p,Z} \over \partial U_3} \not \in k(x)[U_1,U_2],
\end{equation}}
 {Indeed, this is clear when $E=$div$(u_1u_2)$. When $E=$div$(u_1)$, then $\mathrm{Vdir}(x)\subset<U_1,U_3>$ and $H^{-1}{ \partial F_{p,Z} \over \partial U_2}\subset<U_1^{\omega(x)}>$, (\ref{eq:kappa=3}) follows.}

If $x$ satisfies (i)(ii) or ((iii) with $\epsilon (x)=\omega (x)$) above for (T**),
we have $\kappa (x)\leq 2$ or $\kappa (x)=4$. On the other hand, one may have (iii) with $\kappa (x)=3$
if $\epsilon (x)=1+\omega (x)$. We however claim that $\tau '(x)=3$ in this situation.
Namely,
W.l.o.g. it can be assumed that $U_3 \in \mathrm{Vdir}(x)$. By (iii), we then have
$$
H^{-1}{\partial TF_{p,Z} \over \partial U_3} =\lambda U_1^{\omega (x)} + U_2\Phi (U_1,U_2,U_3),
$$
with $\lambda \neq 0$ and $\Phi \not \in k(x)[U_1,U_2]$. It is then obvious that $\tau '(x)=3$,  {by Theorem~ \ref{bupthm}, $x$ is resolved.}

\smallskip

As a consequence, it is sufficient for our purpose to check (i)(ii) or (iii) in order to check (T**),
since $x$ is already resolved if $\kappa (x)\leq 3$.
\end{rem}

\subsection{Preliminaries: transverse case.}

Let $(u_1,u_2,u_3;Z)$ be well adapted coordinates at $x$, where $\kappa (x)=3$. In particular, we have
$\epsilon (x)=\omega (x)+1$. The initial form polynomial
$$
\mathrm{in}_{m_S} h=Z^p -G^{p-1}Z +F_{p,Z}\in G(m_S)[Z]
$$
has $H^{-1}G^p \subset k(x)[U_1, \ldots ,U_e]_{\omega (x)+1}$ and an expansion
\begin{equation}\label{eq802}
    U_1^{-pd_1}U_2^{-pd_2}F_{p,Z}=c U_3^{\omega (x)+1} + \sum_{i=0}^{\omega (x)}U_3^{\omega (x)-i}\Phi_{i+1}(U_1,U_2),
\end{equation}
with $U_3 \in \mathrm{Vdir}(x)$, $c \in k(x)$ and   {$\Phi_{i+1}\in k(x)[U_1,U_2]_{i+1}$}, $0 \leq i \leq \omega (x)$.
Since $\kappa (x)=3$,  {by (\ref{eq:kappa=3})}, we have
\begin{equation}\label{eq8021}
\left\{
  \begin{array}{c}
    (\omega (x)+1\not \equiv 0 \ \mathrm{mod}p \ \mathrm{and}  \ c \neq 0), \ \mathrm{or} \hfill{} \\
     \\
    \Phi_{i+1}(U_1,U_2)\neq 0 \ \mathrm{for} \ \mathrm{some} \  i\leq \omega (x)-2, \ \omega (x)-i \not \equiv 0 \ \mathrm{mod}p \\
  \end{array}
\right.
.
\end{equation}

\begin{prop}\label{sortiebiskappaegaltrois}
Assume that $\kappa (x)=3$,  $E=\mathrm{div}(u_1u_2)$ and
$$
\mathrm{Vdir}(x)=<U_3, \lambda_1U_1 + U_2>, \ \lambda_1 \neq 0.
$$
Then $x$ is resolved.
\end{prop}

\begin{proof}
Take ${\cal Y}_0:=\{x\}$ in (\ref{eq801}) and assume that $x_1$ is very near   {to} $x$. Since $U_1 \not \in \mathrm{Vdir}(x)$,
we have $G=0$. Let $u'_j:=u_j/u_1$, $j=2,3$.
By Theorem \ref{bupthm},  we have
$$
x_1=(X':=Z/u_1, u_1, v:=u'_2 +\gamma_1, u'_3), \ E'=\mathrm{div}(u_1), \ k(x_1)=k(x),
$$
where $\gamma_1  \in S$ is a preimage of $\lambda_1 $. By assumption,
\begin{equation}\label{eq:8.1}
\begin{array}{c}
F_{p,Z}=\sum_{i=0}^{\omega (x)+1}\Psi_i(U_1,U_2)U_3^{\omega (x)+1-i}\\
\Psi :=U_1^{-pd_1}U_2^{-pd_2}{\partial F_{p,Z} \over \partial U_3}=
\sum_{i=0}^{\omega (x)}c_i(\lambda_1U_1 + U_2)^{i}U_3^{\omega (x)-i}\\
 \end{array}
\end{equation}
with $\Psi_i(U_1,U_2)\in k(x)[U_1,U_2]$ $=0$ or homogeneous of degree $i$, $c_i\in k(x)$ and $c_i\neq 0$ for some $i\neq \omega (x)$. Let $(u_1,v,u'_3;Z')$ be well adapted
coordinates at $x_1$. Applying Proposition \ref{bupformula}(v) (with $W':=\mathrm{div}(u_1)\subset \mathrm{Spec}S'$),
we have
\begin{equation}\label{eq8011}
(\Psi (1, \overline{v} -\lambda_1, {\overline{u}'_3}))\subseteq J(F_{p,Z',W'},E',W')
\subseteq k(x)[\overline{u}'_2,\overline{u}'_3]_{(\overline{v},\overline{u}'_3)}.
\end{equation}

Since  {$\omega(x)=\omega(x_1)$ and} $\kappa (x_1)\geq 3$ are assumed, we have
$$
\mathrm{ord}_{(\overline{v},\overline{u}'_3)}U_1^{-pd'_1}F_{p,Z',W'}=\epsilon (x),
$$
where
\begin{equation}\label{eq8012}
  {d'_1=d_1+d_2 -1 + \epsilon (x)/p}.
\end{equation}

If $\epsilon (x_1)=\epsilon (x)$, we get $\mathrm{Vdir}(x_1)+<U_1>=<U_1, V,U'_3>$
by (\ref{eq8011}), so $\kappa (x_1)=2$ by Definition \ref{defkappa}:
a contradiction. Therefore $\epsilon (x_1)=\omega (x)$.
Let
$$
\Phi ':=\mathrm{cl}_{\epsilon (x)}U_1^{-pd'_1}F_{p,Z',W'}\in k(x)[\overline{V},\overline{U}'_3]_{\epsilon (x)}.
$$
We deduce from (\ref{eq8011}) that
\begin{equation}\label{eq8013}
{\partial \Phi' \over \partial \overline{U}'_3}=\Psi (1, \overline{V} -\lambda_1, {\overline{U}'_3}),
\ \mathrm{Vdir}\left ({\partial \Phi' \over \partial \overline{U}'_3}\right ) =<\overline{V}, \overline{U}'_3>.
\end{equation}

The proof is now a variation of that of Proposition \ref{sortiekappaegaldeux}, $\tau '(x)=1$.   {We treat first the case $d'_1\not\in \N$.}
We state the case $d'_1\in \N$ in the following lemma for further use: the assumptions are satisfied by
(\ref{eq8012})-(\ref{eq8013}) and this will conclude the proof.

\smallskip

\noindent   {Case $d'_1\not\in \N$. }

  {$$U_1^{-pd'_1}F_{p,Z',W'}=\gamma \Phi(\overline{V}, {\overline{U}'_3}), \Phi\in k(x)[\overline{V}, {\overline{U}'_3}]_{1+\omega(x)} ,\ \gamma:=(\overline{V}-\lambda_1)^{pd_2}.$$}

  {Let us blow up along $x_1$, let $x_2$ be a point very near to $x_1$. As $\kappa(x_1)>2$ is assumed, we have $<U_1>$=Vdir$(x_1)$.  By   (\ref{eq8013}), we have $\omega(x)\geq 2$, all this implies that $x_2$ is rational over $x_1$. Obviously, $x_2$ is not the point of parameters $(Z'/u'_3,u_1/u'_3,v/u'_3,u'_3)$. After an eventual translation on $u'_3$ and maybe on $Z'$ to get   well adapted
coordinates at $x_1$, after some abuse of notations, we may assume that $x_2$ is the point of parameters $(Y,u_1,u_2,w):=(Z'/v,u_1/v,v,u'_3/v)$.
With $W'':=$div$(u_1)$, we get
$$
\Phi'':=U_1^{-pd'_1}\overline{U_2}^{-pd'_1-\omega(x)+p}F_{p,Y,W''}=\overline{\gamma}\overline{U_2} \Phi(1, {\overline{W}}),
$$
with ord$_{\overline{W}} \Phi(1, {\overline{W}})\leq \omega(x)$. When ord$_{\overline{W}} \Phi(1, {\overline{W}})= \omega(x)-1$, we get $\kappa(x_2)\leq 2$, when ord$_{\overline{W}} \Phi(1, {\overline{W}})= \omega(x)$, we are done by Lemma   \ref{sortiemonome}.}

\end{proof}

\begin{lem}\label{sortiekappaegaldeuxbis}
Assume that $\epsilon (x)=\omega (x)$ and $E=\mathrm{div}(u_1)$.
Let $(u_1,u_2 ,u_3;Z)$ be well adapted coordinates at $x$. Assume furthermore that
the initial form polynomial
  $$
  \mathrm{in}_Eh=Z^p +U_1^{pd_1}\overline{F}, \ \overline{F}\in S/(u_1)
  $$
of Lemma \ref{lemsortiekappaegaldeux} has $d_1\in \N$ and
$$
\mathrm{Vdir}\left ({\partial \Phi \over \partial \overline{U}_2},
{\partial \Phi \over \partial \overline{U}_3}\right ) =<\overline{U}_2, \overline{U}_3>,
$$
where $\Phi :=\mathrm{cl}_{\omega (x)+1}\overline{F} \in k(x)[\overline{U}_2, \overline{U}_3]_{\omega (x)+1}$.
Then $x$ is resolved.
\end{lem}

\begin{proof}
It can be assumed that $\kappa (x)=4$, i.e. $\mathrm{Vdir}(x)=<U_1>$. We then
review the proof of Proposition \ref{sortiekappaegaldeux} for $\tau '(x)=1$, cases 1 and 2.
We take ${\cal Y}_0:=\{x\}$ in (\ref{eq801}).

\smallskip

Case 2 of {\it loc.cit.} gives $\iota (x_1)\leq (p, \omega (x),2)$ after blowing up $x$, hence $x_1$ is resolved.
Similarly, case 1 yields $\iota (x_1)\leq (p, \omega (x),2)$ or after possibly changing well adapted coordinates:
$$
\Phi \in <\overline{U}_3^{\omega (x)+1},\overline{U}_2\overline{U}_3^{\omega (x)}>,
$$
with $\omega (x)\not \equiv 0 \ \mathrm{mod}p$ and
\begin{equation}\label{eq8015}
    x_1=(Z':=Z/u_2, u'_1=u_1/u_2, u_2,u'_3:=u_3/u_2), \ E'=\mathrm{div}(u'_1u_2).
\end{equation}
The case $\omega (x)=1$ is dealt with as in Proposition \ref{sortiekappaegaldeux}:
  {If $\omega (x)=1$, then $({\cal X}',x')$ satisfies the assumption of Lemma \ref{sortieomegaun} or
there is an expansion \eqref{eq712}
\begin{equation}
   \mathrm{in}_{m_{S'}}h'={Z'}^p + {U'_1}^{pd_1}{U'_2}^{p(d_1-1)+1}(\lambda'_1U'_1+\lambda'_2U'_2)
    \in G(m_{S'})[Z'].
\end{equation}}
  {As in Proposition \ref{sortiekappaegaldeux}, ${\cal Y}':=V(Z',u'_1,u'_2)\subset {\cal X}'$ is
 is permissible of the first kind at $x'$ and either  blowing up ${\cal Y}'$ then gives $\iota (x'')\leq (p,\omega (x),1)$ by Theorem \ref{bupthm},
where $x''$ is the center of $\mu$, or we blow up up consecutively ${\cal Y}'_1$, then ${\cal Y}'_2$,
and iterating, we reduce to the case $d_1 =a_1$, $d_2=a_2$, $1+a_1+a_2=p$, $a_1a_2>0$: $\tau(x)=3$.}

\smallskip

Assume that  $\omega (x)\geq 2$. Let $E_1:=\mathrm{div}(u'_1)\subset \mathrm{Spec}S'$
be the strict transform of $E$. We get an expansion
$$
  \mathrm{in}_{E_1}h'={Z'}^p +{U'_1}^{pd_1}\overline{F_1}, \   {\overline{F_1}}\in S'/(u'_1),
$$
where $d'_1=d_1$, $d'_2=d_1-1+\omega (x)/p$ and
\begin{equation}\label{eq8016}
(\overline{u}_2)^{-(pd'_2+1)}\overline{F_1}\equiv ({\overline{u}'_3}^{\omega (x)})\ \mathrm{mod}\overline{u}_2.
\end{equation}

It can be furthermore assumed that $\kappa (x_1)=4$. By Lemma \ref{joyeux}(ii)   {applied with $a(3)=0$}, we have
$\mathrm{Vdir}(x_1)=<U'_1>$ or $\mathrm{Vdir}(x_1)=<U'_1,U_2>$ since $\widehat{pd_1}=0$ is assumed
in this lemma. We take ${\cal Y}_1:=\{x_1\}$ in (\ref{eq801}) and first consider the point
$$
x'':=(Z'':=Z'/u'_3, u''_1=u'_1/ {u'_3}, u''_2:=u_2/u'_3,u'_3), \ E''=\mathrm{div}(u''_1u''_2u'_3).
$$
By (\ref{eq8016}), we obtain $\omega (x'')< \omega (x)$ (resp. $\tau '(x'')=3$) if $\omega (x)\geq 3$
(resp. if $\omega (x)=2$), so $x''$ is resolved in any case. By Theorem \ref{bupthm} it can therefore
be assumed that
\begin{equation}\label{eq8017}
\mathrm{Vdir}(x_1)=<U'_1>.
\end{equation}
Applying again (\ref{eq8016}), we obtain
$$
(\overline{u}_2)^{-(pd'_2+1)}{\partial \overline{F_1} \over \partial \overline{u}'_3}\equiv
({\overline{u}'_3}^{\omega (x)-1})\ \mathrm{mod}\overline{u}_2.
$$
Once again, we obtain $\iota (x_2)\leq (p,\omega (x),2)$ or after possibly changing well adapted coordinates:
$$
x_2=(Z'/u_2, u''_1:=u'_1/u_2, u_2,u'_3/u_2), \ E''=\mathrm{div}(u''_1u_2).
$$

It is now clear that (\ref{eq8016})-(\ref{eq8017}) are stable by blowing up. Iterating, we obtain that
$x_r$ is resolved for some $r\geq 1$ in (\ref{eq801}) or there exists a formal curve
$\hat{\cal Y}=V(\hat{Z},u_1,\hat{u}_3)$ whose strict transform passes through all points $x_r$, $r\geq 1$.
By Proposition \ref{permisarc}(1), it can be assumed that ${\cal Y}=V(Z,u_1,u_3)$ is permissible of
the first kind. Then $x$ is resolved by blowing up ${\cal Y}$ and the conclusion follows.
\end{proof}

\begin{lem}\label{kappa3prelim}
Assume that $\kappa (x)=3$. Then $x$ is good, or there exist well adapted coordinates
$(u_1,u_2,u_3;Z)$ at $x$ and an expansion (\ref{eq802}) such that one of the following properties holds.
\begin{itemize}
  \item [(1)] we have
  \begin{equation}\label{eq8022}
\left\{
  \begin{array}{c}
    \Phi_{i+1} \in k(x)[U_1], \ 0 \leq i \leq \omega (x)-1, \ \mathrm{and} \hfill{} \\
     \\
    \Phi_{\omega (x)+1}=(\lambda_1 U_1 +\lambda_2U_2)U_1^{\omega (x)}, \ \lambda_1,  \lambda_2 \in k(x) \\
  \end{array}
\right.
.
\end{equation}
Furthermore ($ {\Phi_{i+1}} = 0$ for every $i\geq 0$) or ($x_1=x'$ in (\ref{eq801})), where
$$
{\cal Y}_0:=\{x\} \ \mathrm{and} \ x':=(Z':=Z/u_2,u'_1:=u_1/u_2, u_2, u'_3:=u_3/u_2);
$$
  \item [(2)] we have $E=\mathrm{div}(u_1u_2)$, $\tau '(x)=1$ and $x$ satisfies condition (**) (Definition \ref{**}).
\end{itemize}
\end{lem}

\begin{proof}
We always take ${\cal Y}_0:=\{x\}$ in (\ref{eq801}) and discuss according to $x_1$.
It can be assumed that $\iota (x_1)\geq \iota (x)$ (in particular $\omega (x_1)=\omega (x)$).

\smallskip

First suppose that $x_1=x'$.
By Proposition \ref{originchart}, $(u'_1, u_2, u'_3;Z')$ are well adapted coordinates
at $x'$. Since $\epsilon (x')\geq \omega (x)$ by assumption, we deduce that $\mathrm{deg}_{U_2}\Phi_{i+1}\leq 1$,
$0 \leq i \leq \omega (x)$. Similarly, $\Phi_{i+1}\in k(x)[U_1]$ for $\omega (x)-i\not \equiv 0 \ \mathrm{mod}p$
(resp. for $\omega (x)-i\equiv 0 \ \mathrm{mod}p$, $i\neq \omega (x)$) because $\omega (x')=\omega (x)$ (resp.
because $\kappa (x')>2$). Therefore (\ref{eq8022}) holds if $\iota (x')\geq \iota (x)$.

\smallskip

Assume now that $x_1\neq x'$. By Theorem \ref{bupthm}, $x_1$ is resolved if
$$
<U_1,U_3> \subseteq \mathrm{Vdir}(x).
$$
If ($E=\mathrm{div}(u_1u_2)$ and $\tau '(x)=2$), it can thus be assumed by symmetry on $u_1,u_2$ that
$\mathrm{Vdir}(x)=<U_3,\lambda_1U_1+U_2>$, $\lambda_1\neq 0$.
Then $x$ is resolved by Proposition \ref{sortiebiskappaegaltrois}.

Since $x_1$ is very near  {to} $x$, it can be assumed from now on that
\begin{equation}\label{eq8025}
\mathrm{Vdir}(x)=<U_3>.
\end{equation}
We get in (\ref{eq802}): $G=0$ and $ {\Phi_{i+1}}=0$ for  $\omega (x)-i\not \equiv 0 \ \mathrm{mod}p$.
By (\ref{eq8021}), we furthermore have
\begin{equation}\label{eq8023}
c\neq 0 \ \mathrm{and} \ \omega (x)+1\not \equiv 0 \ \mathrm{mod}p .
\end{equation}
If $E=\mathrm{div}(u_1u_2)$, we therefore have (2) and the proof is complete.

\smallskip

Assume now that $E=\mathrm{div}(u_1)$.  Let $I:=\{i : \Phi_{i+1}\neq 0\}$. To conclude the proof,
we will prove that
$$
I\neq \emptyset \Longrightarrow x_1 \ \mathrm{is} \ \mathrm{resolved}.
$$

Let $i\in I$. By (\ref{eq8025}) and (\ref{eq8023}), we have
\begin{equation}\label{eq8026}
\omega (x)-i\equiv 0 \ \mathrm{mod}p, \ i+1 \not \equiv 0 \ \mathrm{mod}p.
\end{equation}
There is an expansion
\begin{equation}\label{eq8024}
\Phi_{i+1}(U_1,U_2) =U_1^{a_i}\Psi_{i+1}(U_1,U_2) , \ a_i \geq 0.
\end{equation}
where $U_1$ does not divide $\Psi_{i+1}$. By (\ref{eq8025}), we have
${\partial \Phi_{i+1} \over \partial U_2}=0$, therefore $\Psi_{i+1}\in k(x)[U_1^p,U_2^p]$,
whence $a_i\geq 1$ by (\ref{eq8026}). Expand
$$
\Psi_{i+1}(U_1,U_2)=:\mu_i U_2^{pb_i} + \cdots , \ \mu_i \neq 0, \ b_i\in \N.
$$
 {As $pb_0\leq 1$ and $b_0\in \N$}, $i\not\in I$: \eqref{eq8022} holds for $i=0$. After possibly changing $Z$ with $Z-\phi$, $\phi \in S$, it can be
assumed that $pd_1+a_i \not \equiv 0 \ \mathrm{mod}p $ or $\mu_i \not \in k(x)^p$.

\smallskip

If $I=\{0\}$, $\kappa (x_1)>2$ implies that $\epsilon (x_1)=\omega (x_1)$: $x_1$ satisfies the assumptions
of Lemma \ref{sortiemonome} (or of Lemma \ref{sortieomegaun}) and the conclusion follows.

\smallskip

Suppose that $i\geq 1$ in what follows. We can take a unitary polynomial $P(t)\in S[t]$,
whose reduction $\overline{P}(t)\in k(x)[t]$ is irreducible and
$$
x_1=(X':=Z/u_1,u_1,  {v}:=P(u_2/u_1), u'_3:=u_3/u_1).
$$
Let $(u_1,  {v}, u'_3; Z')$ be well adapted coordinates at $x_1$. Given
$$
D\in \left \{ U_1{\partial \hfill{} \over \partial U_1},
\ \{{\partial \hfill{} \over \partial \lambda_l}\}_{l\in \Lambda_0}\right \},
$$
we let $\phi_{i,D}(\overline{u_ 2\over u_1}):=U_1^{-(pd_1+i+1)} (D \cdot  U_1^{pd_1}\Phi_{i+1} )\in k(x)[\overline{ {u_ 2\over u_1}}]_{\leq pb_i}$.
By Proposition \ref{bupformula}(v), we have
$$
\omega (x_1)\leq \min_{i,D}\{\omega (x)-i +\mathrm{ord}_{\overline{v}}\phi_{i,D}(\overline{ {u_ 2\over u_1}})\}\leq \omega (x),
$$
where equality holds only if $a_i=1 \ \mathrm{and} \  k(x_1)=k(x)$
by Lemma \ref{lem532}(2). In particular we have $I \subset p\N$.
Since $k(x_1)=k(x)$, it can be assumed w.l.o.g. that $P(t)=t$ and
$$
\Phi_{i+1}(U_1,U_2)=\mu_i U_1U_2^{i},  \ \mathrm{for} \ \mathrm{every} \ i \geq 0
$$
after possibly changing well adapted coordinates (including $i=0$, {\it cf.} above).
Then $(u_1, u'_2, u'_3;X')$ are well adapted coordinates at $x_1$ by Proposition \ref{originchart}.
We obtain: $\epsilon (x_1)=\omega (x)$ and
$$
{H'}^{-1}F_{p,X'}=\sum_{k=0}^{\omega (x)/p}\mu_{kp}{U'_3}^{\omega (x)-kp}{U'_2}^{kp} +U_1\Phi ',
$$
for some $\Phi ' \in k(x)[U_1,U'_2,U'_3]$. But then $\kappa (x_1)\leq 2$: a contradiction. This completes the proof
when $E=\mathrm{div}(u_1)$.
\end{proof}

\subsection{Preliminaries: tangent case.}

Let $(u_1,u_2,u_3;Z)$ be well adapted coordinates at $x$, where $\kappa (x)=4$. This splits into two
different situations:

\smallskip

$\bullet$ if $\omega (x)=\epsilon(x)$, the initial form polynomial is of the form
\begin{equation}\label{eq9012}
\mathrm{in}_{m_S} h=Z^p +F_{p,Z}\in G(m_S)[Z],
\end{equation}
where $H^{-1}F_{p,Z} \subset k(x)[U_1, \ldots ,U_e]_{\omega (x)}$, $1 \leq e \leq 3$.

\smallskip

$\bullet$ if $\omega (x)=\epsilon(x)-1$, the initial form polynomial is of the form
\begin{equation}\label{eq901}
\mathrm{in}_{m_S} h=Z^p -G^{p-1}Z +F_{p,Z}\in G(m_S)[Z]
\end{equation}
with $H^{-1}G^p \subset k(x)[U_1, \ldots ,U_e]_{\omega (x)+1}$, $1 \leq e \leq 2$. By Definition
\ref{defomega}, we have
\begin{equation}\label{eq9011}
(0)\neq V(TF_{p,Z},E,m_S)\subseteq k(x)[U_1, \ldots ,U_e]_{\omega (x)}.
\end{equation}

\begin{defn}\label{skewdir}
Assume that $\kappa (x)=4$ and $\epsilon (x)=\omega (x)$. We say that $\mathrm{Vdir}(x)$ is skew \index{skew $\kappa (x)=4$,! $\mathrm{Vdir}(x)$ is skew}
if for every subset $J\subseteq \{1, \ldots ,e\}$, we have
$$
\mathrm{Vdir}(x)\neq <\{u_j\}_{j\in J}>.
$$
\end{defn}

Assume that $\mathrm{Vdir}(x)$ is skew \and first note that $e=2$ or $e=3$.
Elementary casuistics, similar to that performed in the proof of Proposition \ref{redto*},
yield the following types up to reordering exceptional variables:\\

\noindent (T0) $E=\mathrm{div}(u_1u_2u_3)$ and
\begin{equation}\label{eq902}
\mathrm{Vdir}(x)=<\lambda_1U_1 + \lambda_2U_2 +U_3>, \ \lambda_1\lambda_2\neq 0.
\end{equation}

\noindent (T1) $E=\mathrm{div}(u_1u_2u_3)$ and
\begin{equation}\label{eq903}
\mathrm{Vdir}(x)=<\lambda_1U_1 + U_2 , \lambda_2  {U_1} +U_3>, \ \lambda_1\lambda_2 \neq 0.
\end{equation}

\noindent (T2) $E=\mathrm{div}(u_1u_2u_3)$ and
\begin{equation}\label{eq904}
\mathrm{Vdir}(x)=<\lambda_1U_1 + U_2 , U_3>, \ \lambda_1 \neq 0.
\end{equation}

\noindent (T3) $\mathrm{div}(u_1u_2)\subseteq E \subseteq \mathrm{div}(u_1u_2u_3)$ and
\begin{equation}\label{eq905}
\mathrm{Vdir}(x)=<\lambda_1 U_1 +U_2>, \ \lambda_1 \neq 0.
\end{equation}

\begin{prop}\label{skewresolved}
Assume that $\mathrm{Vdir}(x)$ is skew. Assume furthermore that
$$
J(F_{p,Z},E,m_S) \nsubseteq (U_3)\cap G(m_S)_{\epsilon (x)}
$$
if $x$ is of type (T2) above. Take (\ref{eq801}) to be the quadratic sequence along $\mu$.

\smallskip

Then there exists $r \geq 0$ such that either $x_r$ is resolved or $x_r$ satisfies condition (**).
If $\omega (x)<p$, then $x$ is resolved.
\end{prop}

\begin{proof}
We discuss according to $x_1$ in (\ref{eq801}), where $x_0=x$ is of type (T$k$) for
some $k\in \{0,1,2,3\}$.
It can be assumed w.l.o.g. that $\iota (x_1)\geq (p,\omega (x),3)$. Let
$u'_j:=u_j/u_1$, $j=2,3$.

\smallskip

\noindent $\bullet$ {\it Assume that $k=0$}. By (\ref{eq902}), we have
\begin{equation}\label{eq9021}
J(F_{p,Z},E,m_S)=<(\lambda_1U_1 + \lambda_2U_2 +U_3)^{\omega (x)}>.
\end{equation}
  {By Lemma \ref{vraijoyeux}, $\omega(x) \equiv 0 \ \mathrm{mod}p$ (in particular $\omega (x)\geq p$) and
 it can  be assumed that}





$$
U_1^{-pd_1}U_2^{-pd_2}U_3^{-pd_3}F_{p,Z}=\lambda (\lambda_1U_1 + \lambda_2U_2 +U_3)^{\omega (x)}
$$
after possibly changing $Z$ with $Z-\phi$, $\phi \in S$.
After possibly reordering exceptional variables, we have
$$
x_1=(X':=Z/u_1, u_1, v:=P(u'_2), w:= u'_3 + \gamma_2 u'_2 +\gamma_1),\  {u'_3(x_1)\not=0}
$$
where $\gamma_1 , \gamma_2 \in S$ are preimages of $\lambda_1, \lambda_2$ and $P(t)\in S[t]$ is a
unitary polynomial whose reduction $\overline{P}(t)\in k(x)[t]$ is irreducible. Applying Proposition
\ref{bupformula}(v) (with $W':=\mathrm{div}(u_1)\subset \mathrm{Spec}S'$), we have
\begin{equation}\label{eq9022}
J(F_{p,X',W'},E',W')=({\overline{w}}^{\omega(x)})
\subseteq k(x_1)[\overline{u}'_2,\overline{u}'_3]_{(\overline{v},\overline{w})}.
\end{equation}
Since $\iota (x_1)\geq (p,\omega (x),3)$ is assumed, (\ref{eq9022}) reads
$$
U_1^{-pd'_1}({\partial F_{p,X',W'} \over \partial \overline{v}},{\partial F_{p,X',W'} \over \partial \overline{w}})
=({\overline{w}}^{\omega(x)})
$$
when $E'=\mathrm{div}(u_1)$. If (\ref{eq9022}) is achieved by ${\partial \hfill{} \over \partial \overline{v}}$,
we then have $\epsilon (x_1)=\omega (x)$ and $x_1$ satisfies
the assumptions of Lemma \ref{sortiemonome}; hence $x_1$ is resolved. Otherwise (\ref{eq9022})
gives
$$
U_1^{-pd'_1} {V}^{-pd'_2}F_{p,Z',W'}=({\overline{w}}^{1+\omega(x)}),
$$
for $E'=\mathrm{div}(u_1)$  {(so $d'_2=0$)} or $E'=\mathrm{div}(u_1v)$. This proves that $x_1$ satisfies condition (**).\\

\noindent $\bullet$ {\it Assume that $k=1$}. By Theorem \ref{bupthm} and
(\ref{eq903}),  we have
$$
x_1=(X':=Z/u_1, u_1, v:=u'_2 +\gamma_1, w:= u'_3 + \gamma_2 ), \ E'=\mathrm{div}(u_1)
$$
where $\gamma_1 , \gamma_2 \in S$ are preimages of $\lambda_1, \lambda_2$.

\smallskip

Assume that $\epsilon (x_1)=\omega (x)$.  {As $\kappa(x_1)>2$, $d'_1\in \N$}, by Proposition \ref{bupformula}(v), $x_1$ satisfies the assumptions
of Lemma \ref{sortiekappaegaldeuxbis} and the conclusion follows.

\smallskip

Assume now that $\epsilon (x_1)=1+\omega (x)$. Let $(u_1,v,w;Z')$ be well adapted coordinates
at $x_1$. By Proposition \ref{bupformula}(v) and (\ref{eq903}), we have
$$
\mathrm{Vdir}(x_1)+<U_1>=<U_1,V,W>.
$$
This is a contradiction with Definition \ref{defkappa}, since $\kappa (x_1)\geq 3$ by assumption.\\

\noindent $\bullet$ {\it Assume that $k=2$}.  By Theorem \ref{bupthm} and
(\ref{eq904}),  we have
$$
x_1=(X':=Z/u_1, u_1, v:=u'_2 +\gamma_1, u'_3), \ E'=\mathrm{div}(u_1u'_3), \ k(x_1)=k(x),
$$
where $\gamma_1  \in S$ is a preimage of $\lambda_1 $. By assumption, there exists
$$
\Phi :=\sum_{i=0}^{ {\omega (x)}}\Phi_i(U_1,U_2)U_3^{\omega (x)-i} \in J(F_{p,Z},E,m_S)
$$
with $\Phi_i \in k(x)[U_1,U_2]_i$ and $\Phi_{\omega (x)}= c(\lambda_1U_1+U_2)^{\omega (x)}$, $c\neq 0$.
Applying Proposition \ref{bupformula}(v) (with $W':=\mathrm{div}(u_1)\subset \mathrm{Spec}S'$), we have
\begin{equation}\label{eq9041}
(\Phi (1, \overline{v} -\lambda_1, {\overline{u}'_3}))\subseteq J(F_{p,Z,W'},E',W')
\subseteq k(x)[\overline{u}'_2,\overline{u}'_3]_{(\overline{v},\overline{u}'_3)}.
\end{equation}

Therefore $x_1$ satisfies condition (**) since $E'=\mathrm{div}(u_1u'_3)$, $c\neq 0$.

\smallskip

Assume now that $\omega (x)<p$. By Lemma \ref{joyeux}(ii), we have
\begin{equation}\label{eq9042}
d_1,d_2 \not \in \N, \ d_3\in \N, \ \widehat{pd_1}+\widehat{pd_2}+\omega (x)=p.
\end{equation}

If $d_j\geq 1$, $j=1,2,3$, the center ${\cal Y}_j:=V(Z,u_j)$
is Hironaka-permissible w.r.t. $E$. Blowing up finitely many times, we reduce to the case
$d_3=0$, $0<d_1,d_2<1$. By (\ref{eq9042}), we thus have
$$
p\delta (x)=p(d_1+d_2) +\omega (x)=p, \ \omega (x)\leq p-2.
$$
We thus deduce that $m(x_1)\leq 1+\omega (x)<p$, hence $x_1$ is resolved.\\

\noindent $\bullet$ {\it Assume that $k=3$}. If $\omega (x)<p$, we may assume to begin with
that $\delta (x)=1$ arguing as in (\ref{eq9042}) {\it sqq.} Let
$$
x':=(Z':=X/u_3, v_1:=u_1/u_3, v_2:=u_2/u_3, u_3), \ E':=\mathrm{div}(v_1v_2u_3).
$$

First assume that $x_1\neq x'$. We have
$$
x_1=(Z/u_1, u_1, v:=u'_2 +\gamma_1, w:= P(u'_3)),
$$
where $\gamma_1 \in S$ is a preimage of $\lambda_1$ and $P(t)\in S[t]$ is a
unitary polynomial whose reduction $\overline{P}(t)\in k(x)[t]$ is irreducible. Let
$(u_1,v,w;Z'_1)$ be well adapted coordinates. Applying Proposition
\ref{bupformula}(v) (with $W':=\mathrm{div}(u_1)\subset \mathrm{Spec}S'$), we have
\begin{equation}\label{eq9043}
J(F_{p,Z'_1,W'},E',W')=({\overline{v}}^{\omega(x)})
\subseteq k(x_1)[\overline{u}'_2,\overline{u}'_3]_{(\overline{v},\overline{w})}.
\end{equation}
The conclusion follows as for type (T0): $x_1$ satisfies condition (**) or $x_1$ is
resolved by Lemma \ref{sortiemonome}. The latter holds if $\omega (x)<p$.

\smallskip

Assume now that $x_1=x'$. By Proposition \ref{originchart},
$(v_1,v_2,u_3;Z')$ are well adapted coordinates at $x'$. We
deduce that $\epsilon (x')=\omega (x)$. Furthermore, (\ref{eq905}) implies that
\begin{equation}\label{eq9051}
J(F_{p,Z'},E',m_{S'})\equiv <(\lambda_1V_1+V_2)^{\omega (x)}>
\ \mathrm{mod} (U_3)\cap G(m_{S'})_{\epsilon (x')}.
\end{equation}

Suppose that $\mathrm{Vdir}(x')$ is {\it not} skew. By (\ref{eq9051}), we have $\tau '(x')=3$,
hence $x'$ is resolved.

\smallskip

Suppose that $\mathrm{Vdir}(x')$ is skew. By (\ref{eq9051}), $x'$ is of type (Tk)
for some $k \in \{0,1,2,3\}$. Furthermore if $k=2$, then $x'$ satisfies again the extra assumption
in the proposition also by (\ref{eq9051}). We are already done if $k\leq 2$, so we may assume again
that $x'$ is of type (T3) and iterate. In particular, we have $e=3$. In case $\omega (x)<p$, we
again have $d'_j=d_j$, $1 \leq j \leq 3$.

\smallskip

By Proposition \ref{permisarc}, it can be assumed that ${\cal Y}:=V(Z,u_1,u_2)$
is permissible of the first kind. Let $\pi : {\cal X}' \rightarrow {\cal X}$
be the blowing up along ${\cal Y}$ and $x'_1 \in \pi^{-1}(x)$ satisfy $\iota (x'_1)\geq (p, \omega (x),3)$.
By Theorem \ref{bupthm}, we have
$$
x'_1=(X/u_1, u_1, v:=u_2/u_1 +\gamma_1,u_3), \ E'_1= \mathrm{div}(u_1u_3),
$$
where $\gamma_1  \in S$ is a preimage of $\lambda_1 $. Then $x'_1$ satisfies condition (**). If
$\omega (x)<p$, then $m(x'_1)<p$ and $x$ is resolved.
\end{proof}

\begin{prop}\label{skewresolvedbis}
Assume that $\kappa (x)=4$, $\epsilon (x)=\omega (x)$ and $E=\mathrm{div}(u_1u_2)$. Assume furthermore that the
following properties are satisfied:
\begin{itemize}
  \item [(i)] $\mathrm{Vdir}(x)=<U_1,U_2>$;
  \item [(ii)] the polyhedron $\Delta_{S}(h;u_1,u_2,u_3;Z)$ has a vertex of the form
  $$
  \mathbf{v}:=( {d_1+v_1}, d_2, v_3), \ v_1+v_3={1 +\omega (x) \over p}, \ v_3>{1 \over p},
  $$
  where $(u_1,u_2,u_3;Z)$ are well adapted coordinates at $x$.
\end{itemize}

Take  (\ref{eq801}) to be the quadratic sequence along $\mu$.
There exists $r \geq 0$ such that either $x_r$ is resolved or $x_r$ satisfies condition (**).
If $\omega (x)<p$, then $x$ is resolved.
\end{prop}

\begin{proof}
Suppose that $x_1$ is very near  {to} $x$. By (i) and Theorem \ref{bupthm}, we have
$$
x_1:=(Z':=Z/u_3,u'_1:=u_1/u_3,u'_2:=u_2/u_3,u_3), E':=\mathrm{div}(u'_1u'_2u_3),
$$
and the polyhedron $\Delta_{S'}(h';u'_1,u'_2,u_3;Z')$ is minimal
by Proposition \ref{originchart}. Since $v_3>0$ in (ii), $\mathbf{v}$ is induced by $f_{p,Z}$ by
Theorem \ref{initform}, and $f_{p,Z}$ has an expansion
$$
f_{p,Z}=\sum_{\mathbf{x}\in \mathbf{S}}\gamma(\mathbf{x}) \prod_{j=1}^3u_j^{px_j} ,
    \ \gamma(\mathbf{x})\in S
$$
 such that $\gamma(\mathbf{v})$ is a unit and  { $\mathbf{S}\subset \Delta_{S}(h;u_1,u_2,u_3;Z)$ is a finite set}. By (ii), $x_1$ is very near $x$ only if $v_3=2/p$, i.e.
\begin{equation}\label{eq8033}
{U'_1}^{-pd'_1}{U'_2}^{-pd'_2}{U_3}^{-pd'_3}F_{p,Z'}=
U_3(\lambda ' {U'_1}^{\omega (x)-1}+ U_3\Phi ')+\Phi (U'_1,U'_2),
\end{equation}
for some $\Phi ' \in k(x)[U'_1,U'_2,U_3]$, where $\lambda ' \neq 0$ is induced by $\mathbf{v}$, $\  {\Phi=H^{-1}F_{p,Z}}$ and
\begin{equation}\label{eq8034}
(d'_1,d'_2,d'_3)=(d_1, d_2, d_1+d_2-1+ {\omega (x) \over p}).
\end{equation}

To conclude the proof, we compute $\mathrm{Vdir}(x_1)$.  First note that
\begin{equation}\label{eq8041}
\mathrm{Vdir}(x_1)+<U_3>=<U'_1,U'_2,U_3>
\end{equation}
by (i). If $\tau '(x_1)=3$, then $x_1$ is resolved by Theorem \ref{bupthm}.

\smallskip

Suppose that $\tau '(x_1)\leq 2$. This gives
\begin{equation}\label{eq8042}
\mathrm{Vdir}(x_1)=<U'_1 +\lambda'_1U_3,U'_2+\lambda'_2U_3>, \ \lambda'_1, \lambda'_2 \in k(x).
\end{equation}

Since $\lambda '\neq 0$, we have $(\lambda'_1, \lambda'_2)\neq (0,0)$. We are done
by Proposition \ref{skewresolved} if $\lambda'_1\lambda'_2\neq 0$ (type (T1)) or if $\lambda'_2= 0$,
(type (T2)) where the extra assumption holds by \eqref{eq8042},  {after permuting $U'_2$ and $U_3$}.

\smallskip

Suppose finally that
$$
\mathrm{Vdir}(x_1)=<U'_1 ,U'_2+\lambda'_2U_3>, \ \lambda'_2 \neq 0.
$$

We now apply Lemma \ref{joyeux}(ii) to the ${U'_1}^{\omega (x)-1}$-term in (\ref{eq8033}),
i.e. for the variables $(U_3,U'_2,U'_1)$ respectively and $i=1$. We deduce from (\ref{eq7073}) that
$$
d'_1+ {\omega (x)-1 \over p}\in \N, \ d'_2,d'_3 \not \in \N \ \mathrm{and}
\ \widehat{pd'_2}+\widehat{pd'_3}+1 =p .
$$
Turning back to (\ref{eq8034}), we get
$$
\widehat{pd'_3}=\widehat{pd'_2}+1, \ 2(\widehat{pd'_2}+1) =p.
$$
This is a contradiction, since $p\geq 3$, and the proof is complete.
\end{proof}

\subsection{Reduction to monic expansions (**) and (T**).}

We can now conclude the reduction to monic expansions.

\begin{prop}\label{redto**3}
Assume that $\kappa (x)=3$. Let $\mu$ be a valuation of $L=k({\cal X})$ centered at $x$. There exists
a finite and independent sequence of local permissible blowing ups of the first kind
(\ref{eq801}) along $\mu$ such that
one of the following holds for some $r\geq0$:
\begin{itemize}
  \item [(i)] $x_r$ is resolved or satisfies condition (T**);
  \item [(ii)] $x_r$ satisfies condition (**).
\end{itemize}
Furthermore if
\begin{equation}\label{eq:omega<pkappa=3}
 { \omega (x)<p \ \mathrm{ and}\  \tau '(x)=2,}
 \end{equation}
  then (i) holds.
\end{prop}

\begin{proof}
It can be assumed that the  {conclusions of Lemma \ref{kappa3prelim}(1) or (2)  above hold.}

\smallskip

If $\Phi_{i+1}=0$ for every $i\geq 0$, then  {$\tau'(x)=1$ (so \eqref{eq:omega<pkappa=3} does not hold)} and $x_1$ satisfies condition (**) and we are done.
Otherwise, we may furthermore assume that
$$
x_1=x'=(Z':=Z/u_2,u'_1:=u_1/u_2,u_2,u'_3:=u_3/u_1), \ E'=\mathrm{div}(u'_1u_2)
$$
with $\iota (x')\geq \iota (x)$. Note that when $E=\mathrm{div}(u_1u_2)$,
(\ref{eq8022}) marks an exceptional component $\mathrm{div}(u_1)$ of $E$.

\smallskip

If in  {\eqref{eq802}} ($c\neq 0$ and $\omega (x)+1\not \equiv 0 \ \mathrm{mod}p$), then $x'$ satisfies
condition (**) and we are done for $\omega (x)\geq p$.  Otherwise (i.e. if either  {\eqref{eq:omega<pkappa=3} holds}
 {or}  $c= 0$, or $\omega (x)+1 \equiv 0 \ \mathrm{mod}p$), we have $E'=\mathrm{div}(u'_1u_2)$ and
$(u'_1, u_2, u'_3;Z')$ are well adapted coordinates at $x'$. Furthermore $\mathrm{Vdir}(x)=<U_1,U_3>$
 {either by \eqref{eq8022} or   by assumption  if {\eqref{eq:omega<pkappa=3} holds}}. Let
$$
\Phi (U_1,U_3):=U_1^{-pd_1}U_2^{-pd_2}F_{p,Z} \in k(x)[U_1,U_3]_{\omega (x)+1}
$$
and consider two cases:

\smallskip

\noindent {\it Case 1:} $\epsilon (x')=\omega (x)$. We have $\kappa (x')=4$ and
\begin{equation}\label{eq804}
   \Phi'(U'_1,U_2):= {U'_1}^{-pd'_1}U_2^{-pd'_2}F_{p,Z'}= \lambda_2 {U'_1}^{\omega (x)} + U_2 \Phi'_1(U'_1,U_2),
\end{equation}
with $\Phi'_1\in k(x)[U'_1,U_2]_{\omega (x)-1}$.

\smallskip

If $\tau '(x')=1$ (i.e.  $\Phi'_1=0$ or  $\mathrm{Vdir}(x)=<U_2+\lambda U_1>$, $\lambda \in k(x)$), then
$x'$ satisfies condition (T**) or $x'$ satisfies the assumptions of Proposition \ref{skewresolved} type (T3)
respectively, and the proof is complete. We may thus furthermore assume that
\begin{equation}\label{eq8032}
\mathrm{Vdir}(x')=<U'_1,U_2>.
\end{equation}

Since $\kappa (x)=3$, we have at this point:
\begin{equation}\label{eq803}
{\partial \Phi \over \partial U_3}(U_1,U_3)\not \in k(x)[U_1].
\end{equation}
Therefore $\Delta_{S'}(h';u'_1,u_2,u'_3;Z')$ has a vertex of the form
  $$
  \mathbf{v}':=( {d'_1+v'_1}, d'_2, v'_3), \ v'_1+v'_3={1 +\omega (x) \over p}, \ v'_3={\mathrm{deg}_{U_3}\Phi \over p},
  $$
where $(u_1,u_2,u_3;Z)$ are well adapted coordinates at $x$. The proposition follows from
Proposition \ref{skewresolvedbis} whose assumptions are satisfied by {\eqref{eq8032},\eqref{eq803}}.

\smallskip

\noindent {\it Case 2:} $\epsilon (x')=\epsilon (x)$.  {Note that this implies $\lambda_2=0$ in (\ref{eq8022})}. We again have $\kappa (x')=3$ and may iterate:
$$
 {{H'}^{-1}{ F_{p,Z'} }} \equiv \Phi (U'_1,U'_3) \ \mathrm{mod}(U_2)\cap G(m_{S'})_{\epsilon (x')}.
$$
Then $\mathrm{Vdir}(x')+<U_2>= <U'_1,U_2,U'_3>$, so $\tau '(x')\geq 2$.
We are done if $\tau '(x')=3$ and may iterate if $\tau '(x')=2$ as asserted.

\smallskip

Since the exceptional component $\mathrm{div}(u_1)$ of $E$ has been marked ({\it cf.} beginning
of the proof), the theorem holds except possibly if $x_r$ is in case 2 for every $r\geq 0$.
In this situation, we apply Proposition \ref{permisarc}(1): w.l.o.g. it can be assumed that
${\cal Y}:=V(Z,u_1,u_3)$ is permissible of the first kind. Since $\mathrm{Vdir}(x)=<U_1,U_3>$,
it follows from Theorem \ref{bupthm} that $x$ is resolved by blowing up ${\cal Y}$.
\end{proof}

\begin{lem}\label{redto**4lem}
Assume that $\kappa (x)=4$ and $\epsilon (x)=\omega (x)$.
Let $\mu$ be a valuation of $L=k({\cal X})$ centered at $x$. There exists
a finite and independent sequence of local permissible blowing ups of the first kind
(\ref{eq801}) along $\mu$ such that one of the following holds for some $r\geq 0$:
\begin{itemize}
  \item [(i)] $x_r$ is resolved or satisfies condition (T**);
  \item [(ii)] $x_r$ satisfies condition (**).
\end{itemize}
If $\omega (x)<p$, then (i) holds.
\end{lem}

\begin{proof}
By Proposition \ref{skewresolved}, it can be assumed that one
of the following conditions holds:

\smallskip

\begin{itemize}
  \item [(1)] $\mathrm{Vdir}(x)$ is skew and satisfies condition (T2);
  \item [(2)] $\mathrm{div}(u_1u_2)\subseteq E$ and $\mathrm{Vdir}(x)=<U_1,U_2>$.
\end{itemize}

Take (\ref{eq801}) to be the quadratic sequence along $\mu$. Under assumption (1),
we have $E=\mathrm{div}(u_1u_2u_3)$ and $\mathrm{Vdir}(x)=<U_1, \lambda_2 U_2 +U_3>$,
$\lambda_2\neq 0$ up to renumbering variables. By Proposition \ref{skewresolved}, it can be assumed that
$$
J(F_{p,Z},E,m_S)\subseteq (U_1) {\cap} G(m_S)_{\epsilon (x)}.
$$
By Theorem \ref{bupthm}, we have
$$
x_1=(Z/u_2, u'_1:=u_1/u_2, u_2,  {v:=u_3/u_2} +\gamma ), \ E'= \mathrm{div}(u'_1u_2),
$$
where $\gamma  \in S$ is a preimage of  {$\lambda_2 $}. Let $W':=\mathrm{div}(u_2)\subset \mathrm{Spec}S'$
and $(u'_1,u_2,v;Z')$ be well adapted coordinates at $x_1$. By Proposition \ref{bupformula}(v), we have
$$
J(F_{p,Z',W'},E',W')=U_2^{-\omega(x)}J(F_{p,Z},E,m_S)
\subseteq k(x_1)[\overline{u}'_1,\overline{u}'_3]_{(\overline{u}'_1,\overline{v})}.
$$

If $\mathrm{ord}_{(\overline{u}'_1,\overline{v})}H_{W'}^{-1}F_{p,Z',W'}= \omega (x)$, we
have $\kappa (x_1)\leq 2$ (so $x$ is resolved) or
\begin{equation}\label{eq:kappa3sortie}
\left\{
\begin{array}{c}
{H'}^{-1}F_{p,Z'}\equiv <{U'_1}^{\omega (x)}> \ \mathrm{mod}(U_2) \cap G(m_{S'})_{\omega (x)}  \\
\\
   {  \mathrm{cl}_{\omega(x)}({U'}_1^{-pd_1}U_2^{-pd'_2}{\partial F_{p,Z',W'} \over \partial \overline{v}})\not \in <{U'_1}^{\omega (x)}> }\hfill{}
\end{array}
\right. .
\end{equation}

In this last situation, the conclusion follows in each of the following possible cases:

\smallskip

$\bullet$ $x_1$ satisfies condition (T**) if $\mathrm{Vdir}(x_1)=<U'_1>$;

\smallskip

$\bullet$ $x_1$ satisfies the assumptions of Proposition \ref{skewresolved}
if $\mathrm{Vdir}(x_1)=<U'_1+\lambda 'U_2>$, $\lambda '\neq 0$;

\smallskip

$\bullet$ $x_1$ satisfies the assumptions of Proposition \ref{skewresolvedbis} by \eqref{eq:kappa3sortie}
if $\mathrm{Vdir}(x_1)=<U'_1,U_2>$.

\smallskip

If $\mathrm{ord}_{(\overline{u}'_1,\overline{v})}H_{W'}^{-1}F_{p,Z',W'}= \omega (x)+1$,
we are also done by Proposition \ref{redto**3} if $\kappa (x_1)=3$, since $\tau '(x_1)\geq 2$.
Assume finally that $\kappa (x_1)=4$, i.e.
$$
\epsilon (x_1)=\omega(x)=\mathrm{ord}_{(\overline{u}'_1,\overline{v})}
H_{W'}^{-1}{\partial F_{p,Z',W'}\over \partial \overline{v}}
<\mathrm{ord}_{(\overline{u}'_1,\overline{v})}H_{W'}^{-1}F_{p,Z',W'}= {1+\omega (x)}.
$$
Similarly, $x_1$ satisfies condition (T**) unless $\mathrm{Vdir}(x_1)=<U'_1,U_2>$.
The conclusion then follows again from Proposition \ref{skewresolvedbis}.

\smallskip

Under assumption (2), it can be
assumed that $x_1=x'$, $\iota (x')=\iota (x)$, where
$$
x':=(Z':=Z/u_3, u'_1:=u_1/u_3, u'_2:=u_2/u_3, u_3), \ E':=\mathrm{div}(u'_1u'_2u_3).
$$
By Proposition \ref{originchart}, $(u'_1,u'_2,u_3;Z')$ are well adapted coordinates at $x'$. We get
$\epsilon (x')=\omega (x)$ and
$$
\mathrm{Vdir}(x')+<U_3>=<U'_1,U'_2,U_3>.
$$
If $\tau '(x')=3$, then $x'$ is resolved by Theorem \ref{bupthm}. Otherwise, $x'$ satisfies
again the assumptions of the proposition, with (1) up to renumbering variables or (2) above.

\smallskip

Iterating, the proof concludes as in the proof of Proposition \ref{redto**3}: $x$ is resolved or
the curve ${\cal Y}:=V(Z,u_1,u_2)$ is permissible of the first kind; then $x$ is resolved by
blowing up ${\cal Y}$, since $\mathrm{Vdir}(x)=<U_1,U_2>$.
\end{proof}

\begin{prop}\label{redto**4}
Assume that $\kappa (x)=4$. Let $\mu$ be a valuation of $L=k({\cal X})$ centered at $x$. There exists
a finite and independent sequence of local permissible blowing ups of the first kind
(\ref{eq801}) along $\mu$ such that one of the following holds for some $r\geq 0$:
\begin{itemize}
  \item [(i)] $x_r$ is resolved or satisfies condition (T**);
  \item [(ii)] $x_r$ satisfies condition (**).
\end{itemize}
If $\omega (x)<p$, then (i) holds.
\end{prop}

\begin{proof}
By Lemma \ref{redto**4lem}, we are done if $\epsilon (x)=\omega(x)$. Otherwise
one of the following conditions holds up to reordering exceptional variables:

\begin{itemize}
  \item [(1)] $E=\mathrm{div}(u_1u_2)$,  $\mathrm{Vdir}(x)=<U_1,U_2>$;
  \item [(2)] $E=\mathrm{div}(u_1u_2)$,  $\mathrm{Vdir}(x)=<\lambda_1 U_1+U_2>$, $\lambda_1 \neq 0$;
  \item [(3)] $\mathrm{div}(u_1)\subseteq E \subseteq \mathrm{div}(u_1u_2)$, $\mathrm{Vdir}(x)=<U_1>$.
\end{itemize}

Take (\ref{eq801}) to be the quadratic sequence along $\mu$. We may always assume that
\begin{equation}\label{eq8051}
\iota (x_1)\geq (p,\omega (x),3) \ \mathrm{and} \ \epsilon (x_1)=1 + \omega (x)
\end{equation}
in this proof.  Let
$$
x':=(Z/u_3, u'_1:=u_1/u_3, u'_2:=u_2/u_3, u_3), \ \mathrm{div}(u'_1u_3)\subseteq E',
$$
where ${\partial TF_{p,Z} \over \partial U_3}\neq 0$. If $x_1=x'$, we have
$\epsilon (x')=\omega (x)$: a contradiction with (\ref{eq8051}). This
concludes the proof under assumption (1) by Theorem \ref{bupthm}.

\smallskip

Assume that $x_1\neq x'$. Under assumption (2), we can take a unitary polynomial $P(t)\in S[t]$,
whose reduction $\overline{P}(t)\in k(x)[t]$ is irreducible, and
$$
x_1=(X':=Z/u_1, u_1, v:=u_2/u_1 +\gamma_1, w:=P(u_3/u_1)), \ E'= \mathrm{div}(u_1),
$$
where $\gamma_1  \in S$ is a preimage of $\lambda_1 $.

\smallskip

Let $W':=\mathrm{div}(u_1)\subset \mathrm{Spec}S'$ and $(u_1,v,w;Z')$ be well adapted
coordinates at $x_1$, $Z':=X'-\phi$, $\phi \in S'$.
By Proposition \ref{bupformula}(v), we deduce that
$$
\mathrm{in}_{W'}h'={Z'}^p  + F_{p,Z',W'}\in G(W')[Z'],
$$
where $G(W')=k(x_1)[\overline{u}'_1,\overline{u}'_3]_{(\overline{v},\overline{w})}[U_1]$ and
\begin{equation}\label{eq805}
   (\overline{v}^{\omega (x)}) \subseteq J(F_{p,Z',W'},E',W').
\end{equation}

If $\omega (x)<p$, assumption (2) reads:
$$
H^{-1}F_{p,Z}=U_3(U_2+\lambda_1 U_1)^{\omega (x)} +\Phi (U_1,U_2), \ G=0.
$$
If $\Phi =0$, this is a contradiction since then $\kappa (x_1)=2$ by (\ref{eq805}).
After possibly performing a linear change of coordinates in $u_3$, then picking
again well adapted coordinates, we reduce to:
$$
\Phi (U_1,U_2)=U_1U_2\Psi(U_1,U_2).
$$
Since $\Delta (h;u_1,u_2,u_3;Z)$ is minimal, we have $U_1^{pd_1}U_2^{pd_2}\Phi \not \in G(m_S)^p$ and  obtain
$$
\mathrm{ord}_{(\overline{v},\overline{w})}J(F_{p,Z',W'},E',W')\leq \mathrm{deg}\Psi=\omega (x)-1,
$$
also a contradiction, since $\omega (x_1)=\omega (x)$ is assumed.

If $\omega (x)\geq p$, we may then furthermore assume that $\epsilon (x_1)=\epsilon (x)$ by (\ref{eq8051}),
so $\kappa (x_1)= 3$ by (\ref{eq805}). We conclude by Proposition \ref{redto**3}.

\smallskip

Under assumption (3), we define a refinement ${\cal C}$ of the function $x\mapsto (m(x),\omega (x))$, {\it cf.} chapter 6.
Let $\pi : \ {\cal X}' \rightarrow ({\cal X},x)$ be the blowing up along a permissible center of the
first kind ${\cal Y}\subseteq \mathrm{div}(u_1)$, $x_1\in \pi^{-1}(x)$. We set:
$$
{\cal C}(x_1)<{\cal C}(x) \Leftrightarrow x_1 \ \mathrm{satisfies} \ \mathrm{the} \ \mathrm{conclusion}
\ \mathrm{of} \ \mathrm{the} \ \mathrm{proposition}.
$$
By Theorem \ref{bupthm}, we have ${\cal C}(x_1)<{\cal C}(x)$ unless $x_1$ belongs to  the strict transform
$\mathrm{div}(u'_1)\subseteq E'= \mathrm{div}(u'_1u_2)$ of $\mathrm{div}(u_1)$.
Otherwise, we let  ${\cal C}(x_1)={\cal C}(x)$.

\smallskip

With notations as in chapter 6, we claim that $\mathrm{div}(u_1)$ has maximal contact
for the condition ${\cal C}$ (Definition \ref{Maximalcontact}). To see this, suppose that
${\cal C}(x_1)={\cal C}(x)$ and apply Proposition  \ref{redto**3}, Lemma \ref{redto**4lem} and (1) and (2) above.
It can be assumed that
$$
\epsilon (x_1)=\epsilon (x), \ \kappa (x_1)\geq 3 \ \mathrm{and} \ {\cal Y}=\{x\}.
$$

If $\omega (x)\geq p$, we are done unless $x_1$ satisfies again (3) and the claim is proved;
if $\omega (x)<p$, we must still check that the situation
$$
\kappa (x_1)=3, \ \tau '(x_1)=1
$$
does not occur. By assumption (3), (\ref{eq9011}) with $G=0$ gives an expansion
$$
U_1^{-pd_1}U_2^{-pd_2}F_{p,Z}= L(U_1,U_2,U_3)U_1^{\omega (x)} +
\sum_{i=1}^{\omega (x)}Q_{i+1}(U_2,U_3)U_1^{\omega (x)-i}
$$
with $L(0,0,U_3)\neq 0$, $Q_{\omega (x)+1}(U_2,U_3)\in k(x)[U_2^p,U_3^p]$.
Therefore
$$
(0)\neq V(F_{p,Z'},E',m_{S'})\subseteq k(x')[U'_1, U_2]_{\omega (x)}
$$
after blowing up, where $(u'_1,u_2,v';Z')$ are well adapted coordinates at $x'$:
a contradiction with $\kappa (x_1)=3$, $\tau '(x_1)=1$.  This concludes the proof of
the claim when $\omega (x)<p$. The proposition now follows from Theorem \ref{contactmaxFIN}.
\end{proof}

\section{Resolution of $\kappa(x)=3,4$ with monic expansions.}

In this chapter, we prove projection Theorem \ref{projthm}
in the case where $\kappa(x)\geq 3$.

\noindent  { {\it Up to the end of this chapter, ``resolved" stands for ``resolved for $(p,\omega (x),3)$"
(Remark \ref{quadsequence})}}.\\















\subsection{From (T**) to (**), resolution for $\epsilon(x)=\omega(x)<p $.}

The purpose of this section is to reduce Theorem \ref{projthm} for $\kappa (x)=3,4$ to points
satisfying condition (**) in Definition~\ref{**}.  {This reduction uses the concept of maximal contact with respect to a refinement $x\mapsto \mathcal{C}(x)$ of the function $x\mapsto (m(x),\omega(x))$, see chapter~6. Let $x$ be in the case (T**) of Definition~\ref{**}, in particular, $\kappa(x)=4$.
We consider a finite sequence  of local blowing ups along $\mu$:
\begin{equation}\label{eq:contactmaxeq1}
    ({\cal X},x)=:({\cal X}_0,x_0) \leftarrow ({\cal X}_1,x_1) \leftarrow \cdots \leftarrow ({\cal X}_r,x_r) ,
\end{equation}
with {\it permissible centers of the first kind} ${\cal Y}_i \subset ({\cal X}_i,x_i)$, where $x_i$,
$0 \leq i \leq r$, denotes the center of $\mu$.
For $1\leq i$,
we define:}
\begin{equation}\label{eq:contactmaxeqT**}
\left\{
\begin{array}{ccc}
\mathcal{C}(x_i)<\mathcal{C}(x)  &  \mathrm{if} & x_i \hbox{ is resolved or satisfies   (**)}\\
  &   &   \\
 \mathcal{C}(x_i)=\mathcal{C}(x) &   \mathrm{if} & x_i \hbox{ satisfies (T**)} \hfill{}
\end{array}
\right. .
\end{equation}

We prove the following proposition.

\begin{prop}\label{redto**casT**}
Let $x$ be in the case (T**) of Definition~\ref{**},
and $\mu$ be a valuation of $L=k({\cal X})$ centered at $x$. There exists a finite and independent
sequence of  permissible blowing ups of the first kind
$$
    ({\cal X},x)=:({\cal X}_0,x_0) \leftarrow ({\cal X}_1,x_1)\leftarrow \cdots \leftarrow ({\cal X}_r,x_r),
$$
where $x_i$ is the center of $\mu$ in ${\cal X}_i$, $0\leq i \leq r$, such that
$x_r$ is resolved or ($x_r$ satisfies condition (**) and $\omega (x)\geq p$).
\end{prop}

\begin{proof}
By Proposition \ref{T**contactmaximal} below, there is weak maximal contact (Definition \ref{Maximalcontact})
for the refinement  {$\mathcal C$ defined above}.
Furthermore  nonresolved points created by blowing up along closed points satisfy condition (**)
with $\omega (x)\geq p$ (Proposition \ref{T**contactmaximal}(i)).

\smallskip

Theorem \ref{contactmaxFIN} does not apply directly since maximal contact does not necessarily hold. We must check that its proof remains valid when
using only those blowing ups of the first kind which are well behaved w.r.t. ${\cal C}$ (Proposition \ref{T**eclatementcourbe} below).

 {By performing blowing ups at closed points, 
any  curve ${\cal Y}\subset$div$(u_1)$ of generic point $y$ with  $\omega(y)>0$}
\smallskip

\noindent (a)  {either maps to an intersection of two components of $E$, i.e.
$$
    \eta ({\cal Y})=V(Z,u_1,u_j), \ \mathrm{div}(u_j)\subseteq E, \ j\geq 2, \ \mathrm{or}
$$
\noindent (b) or $\eta ({\cal Y})=V(Z,u_1,u_3)$, $E\subseteq \mathrm{div}(u_1u_2)$, for some $u_3$.
Furthermore, this curve is unique.}

\smallskip

Suppose that there exists ${\cal Y}$ permissible satisfying (b).
Then ${\cal Y}$ satisfies assumption  {(2)} of Proposition \ref{T**eclatementcourbe}
except possibly in case (T**)(i). Let $W:=\eta ({\cal Y})$ and expand:
$$
U_1^{-pd_1}\overline{u}_2^{-pd_2}F_{p,Z,W}=\overline{\gamma}_0 U_1^{\omega (x)}+
\sum_{i=1}^{\omega (x)}\overline{\gamma}_i U_1^{\omega (x)-i}U_3^i \in G(W)_{\omega (x)},
$$
with $\overline{\gamma}_i \in S/(u_1,u_3)$, $\overline{\gamma}_0$ a unit. We are done by
Proposition \ref{T**eclatementcourbe}(1) if $\overline{\gamma}_i=0$ for $1 \leq i \leq \omega (x)$.
Otherwise, blow up along $x$. There is nothing to prove except at the point
$x':=(Z/u_2,u'_1:=u_1/u_2,u_2,u_3/u_2)$ on the strict transform ${\cal Y}'$ of ${\cal Y}$, $E'=\mathrm{div}(u'_1u_2)$.
Then $x'$ is now in case (T**)(ii) and the conclusion follows from
Proposition \ref{T**eclatementcourbe}, assumption (2).
  {From now on,   we assume that all  curves contained in div$(u_1)$ with  $\omega(y)>0$ satisfy (a). In particular, all permissible curves   satisfy (a).}

\smallskip
We use notations as in Proposition \ref{contactmaxpetitgamma}   {and Notation \ref{betagammamaxcontact}}.
 {By performing blowing ups at closed points, we reach $\gamma(x)=1$, i.e. either:}

\noindent (i)  {$\omega(x)=\epsilon(x)$, $\beta(x)\leq 1$, $E=$div$(u_1u_2)$},

\noindent (ii)  {$C(x)<1$, $E=$div$(u_1u_2u_3)$,}

\noindent (iii)  { $\omega(x)=\epsilon(x)-1$, $\beta(x)< 1$, $E=$div$(u_1u_2)$.}

\noindent    { {\it Case (i) or (iii)}. Assume ${\cal Y}=V(Z,u_1,u_2)$ permissible.}
Assumption (1) in Proposition \ref{T**eclatementcourbe} is equivalent to $A_2(x)>1$. If $A_2(x)=1$,
there is an expansion
$$
U_1^{-pd_1}U_2^{-pd_2}F_{p,Z,W}=\overline{\gamma}_0 U_1^{\omega (x)}+
\sum_{i=1}^{\omega (x)}\overline{\gamma}_i U_1^{\omega (x)-i}U_2^i, \ W:=\eta ({\cal Y}),
$$
with $\overline{\gamma}_i \in S/(u_1,u_2)$,   {$\overline{\gamma}_0$ a unit in case (i), $\overline u_3 \times$unit in case (iii)}. Then
\begin{equation}
\begin{array}{ccc}
 \min_{1 \leq i \leq \omega (x)} \left \{{\mathrm{ord}_{\overline{u}_3}\overline{\gamma}_i \over i}\right \}=\beta (x)\leq 1, \ \hbox{case(i)}, \\
 \min_{1 \leq i \leq \omega (x)} \left \{{\mathrm{ord}_{\overline{u}_3}\overline{\gamma}_i -1\over i}\right \}=\beta (x)< 1, \ \hbox{case(ii)}.
 \end{array}
\end{equation}
 We prove that Proposition  \ref{redto**casT**} holds in this situation.

If $\beta (x)>0$, we have $\mathrm{Vdir}(x)=<U_1>$ and get
$\iota (x')\leq (p,\omega (x),2)$ after blowing up, so $x$ is resolved by blowing up along ${\cal Y}$.

If $\beta (x)\leq 0$, we blow up along $x$. By
Proposition \ref{T**contactmaximal} below (proof in case (T**)(ii)), we get $x'$ resolved or
($x'$ satisfies condition (**) with $\omega (x)\geq p$) except if
$x'=(Z/u_3,u'_1:=u_1/u_3,u'_2:=u_2/u_3,u_3)$ is the point on the strict transform ${\cal Y}'$ of ${\cal Y}$,
$E'=\mathrm{div}(u'_1u'_2u_3)$. We now have $\mathrm{Vdir}(x')=<U'_1,U'_2>$ or $\mathrm{Vdir}(x')=<\lambda_1U'_1 +U'_2>$,
$\lambda_1 \neq 0$. Blowing up along  ${\cal Y}'$ gives $x''$ resolved or ($x''$ satisfies (**) with $\omega (x)\geq p$),
arguing as in the proof of Proposition \ref{T**eclatementcourbe} below,  assumption (2).

\smallskip

\noindent   {{\it Case (ii)}:} ${\cal Y}=V(Z,u_1,u_j)$, $E=\mathrm{div}(u_1u_2u_3)$,
$j=2$ or $j=3$. Assumption (1) (resp. assumption (2)) of Proposition \ref{T**eclatementcourbe} is
equivalent to $A_j(x)>1$ (resp. to: $j=3$ and $A_2(x)>0$). By symmetry, there remains to deal with the case
${\cal Y}=V(Z,u_1,u_3)$ with $A_2(x)=0$, $A_3(x)=1$. There is an expansion
$$
u_1^{-pd_1}u_2^{-pd_2}u_3^{-pd_3}f_{p,Z}=\gamma u_1^{\omega (x)}+
\sum_{i=1}^{\omega (x)}f_iu_1^{\omega (x)-i}u_3^i, \ f_i\in S
$$
with $\gamma \in S$ a unit. Let $\overline{f}_i \in S/(u_1)$ be the residue of $f_i$. Then
$$
\min_{1 \leq i \leq \omega (x)} \left \{{\mathrm{ord}_{(\overline{u}_2,\overline{u}_3)}\overline{f}_i \over i}\right \}
=C (x)< 1,
$$
since $\gamma (x)=1$ is assumed here. We consider two cases: $C(x)>0$ and $C(x)=0$,  {arguing as above in case (i)(iii): $C(x)>0$ corresponding to $\beta(x)>0$, $C(x)=0$ to case $\beta(x)=0$ (i)}.
Blowing up along ${\cal Y}$, we get respectively $x$ resolved; $x'$ resolved or
($x'$ satisfies (**) with $\omega (x)\geq p$). Proposition  \ref{redto**casT**} holds in any case.
\end{proof}

This proposition leads to:

\begin{cor}\label{omega(x)=epsilon(x)<p}
Assume that $\omega(x)<p$ and either $\kappa (x)=4$, or ($\kappa (x)=3$ and $\tau '(x)=2$). Then
$x$ is resolved.
\end{cor}

\begin{proof}
Indeed, by Propositions \ref{redto**3} and \ref{redto**4}, there exists
an independent sequence of local blowing ups (\ref{eq801}) along $\mu$ such that $x_r$ is resolved
or $x_r$ satisfies condition (T**). In the last case, apply Proposition~\ref{redto**casT**}.
\end{proof}

\begin{prop}\label{T**contactmaximal}
Let $x$ be in the case (T**) of Definition~\ref{**}.
Then   $\div(u_1)$ has weak maximal contact (Definition~\ref{Maximalcontact})
for the condition (T**) and $\kappa(x)\geq 3$. More precisely, let
$\pi : \ {\cal X}' \longrightarrow ({\cal X},x)$ be the blowing up along  $x$ and  $x' \in \pi^{-1}(x)$,
with $\iota (x')\geq (p,\omega (x),3)$:
\begin{itemize}
  \item [(i)] if $x'$ is not on the strict transform of div$(u_1)$, then $x'$ is resolved
or satisfies (**) with $\omega(x)\geq p$;
  \item [(ii)] if $x'$ is on the strict transform of div$(u_1)$, then $x'$ satisfies (T**).
\end{itemize}
\end{prop}

\begin{proof}
In the case  (T**)(i), the reader sees that $<U_1>=\mathrm{Vdir}(x)$ and, if we blow up along $x$,
any point $x'$ with $\iota (x')\geq (p,\omega (x),3)$ verifies  (T**)(ii) or  (iii).

\smallskip

In the case (T**)(ii) and not (i),  we have
$$
U_1^{-pd_1}U_2^{-pd_2}U_3^{-pd_3}F_{p,Z}= \lambda_0 U_1^{\omega(x)}+U_2P(U_1,U_2,U_3),
$$
by (\ref{eq9012}), with $0\neq \lambda_0 \in k(x)$, $0\neq P\in k(x)[U_1,\ldots ,U_e]_{\omega (x)-1}$.

Either
$$
\mathrm{Vdir}(x)=<U_1,U_2>,
$$
then,   {by Theorem \ref{bupthm}, $x'$ is very near to $x$}  only if  $x'=(Z/u_3,u_1/u_3,u_2/u_3,u_3)$. Clearly  $\iota(x')<(p,\omega(x),3)$ or $x'$ satisfies  (T**)(ii). Or we have
$$
\mathrm{Vdir}(x)=<\lambda_1U_1+U_2>, \ \lambda_1 \neq 0.
$$
This is  case (T3) of Proposition \ref{skewresolved}. Arguing as in its proof, {\it cf.} (\ref{eq9043}),
$x'$ satisfies condition (**) with $\omega (x)\geq p$ or $x'$ is resolved by Lemma \ref{sortiemonome} except possibly
if $x'=(Z/u_3,u_1/u_3,u_2/u_3,u_3)$. Then $\iota(x')<(p,\omega(x),3)$ or $x'$ satisfies again (T**)(ii).

\smallskip

In the case (T**)(iii), we apply Lemma \ref{sortiemonome} when   {$\epsilon(x)=\omega(x)\geq 2$ or Lemma \ref{sortieomegaun}  when $\epsilon(x)=\omega(x)=1$}:
$x$ is resolved for $\iota=(p,\omega(x),2)$. Assume that $\epsilon(x)=1 +\omega(x)$. By
Remark \ref{remT**}, we may assume $\kappa (x)=4$.

If $x'=(Z/u_3,u_1/u_3,u_2/u_3,u_3)$, we have $\omega(x')\leq \omega(x)$ and in case of equality,
$\epsilon(x')=\omega(x)$ and $x'$ satisfies (T**)(ii). In particular, we are done if $\mathrm{Vdir}(x)=<U_1,U_2>$
by Theorem \ref{bupthm}. There remains to deal with the case $\tau '(x)=1$.

\smallskip

\noindent {\it Case 1:} $\mathrm{Vdir}(x)=<U_1>$. Expand
\begin{equation}\label{eq911}
U_1^{-pd_1}U_2^{-pd_2}F_{p,Z}=U_3 U_1^{\omega(x)}+U_2Q, \ Q \in k(x)[U_1,U_2,U_3]_{\omega (x)}.
\end{equation}
If $Q=0$, the reader sees that $x'$ satisfies (T**)(ii) or  (T**)(iii) if $\iota (x')\geq (p,\omega (x),3)$.
The difficult case is $Q\neq 0$. By (\ref{eq9011}), we have
$$
V(TF_{p,Z},E,m_S)=H^{-1}{\partial TF_{p,Z} \over \partial U_3}\subseteq < U_1^{\omega(x)}>.
$$
This gives ${\partial Q \over \partial U_3}=0$, i.e. $Q \in k(x)[U_1,U_2,U_3^p]$ in both
cases $G= 0$ and $G\neq 0$. Expand again
\begin{equation}\label{eq912}
Q=\sum_{i=0}^{i_0} U_1^{\omega (x)-i} Q_{i}(U_2,U_3^p),
\   {Q_{i_0}(U_2,U_3^p)}\neq 0 .
\end{equation}

If $i_0 =0$, we reduce to $Q=0$ after possibly picking new well adapted coordinates $(u_1,u_2,v;Z')$
at $x$.

If $i_0\geq 1$, we apply Proposition \ref{bupformula}(v) to
those elements of $J(F_{p,Z},E,m_S)$ of the form:
$$
U_1^{-pd_1}U_2^{-pd_2}D \cdot F_{p,Z}=\lambda_D U_3 U_1^{\omega(x)}+
U_2\sum_{i=0}^{i_0} U_1^{  {\omega (x)-i}} Q_{i,D}(U_2,U_3^p),
$$
where $D\in \{U_1{\partial \hfill{} \over \partial U_1}, U_2{\partial \hfill{} \over \partial U_2},
\{{\partial \hfill{} \over \partial \lambda_l}\}_{l\in \Lambda_0}\}$.

By  Lemma \ref{lem532}(2) (applied to $F:= Q_{i_0}(U_2,U_3^p)$), we get
$\omega (x')\leq \omega (x)$ with strict equality if $k(x')\neq k(x)$. If  $k(x')= k(x)$, it can be assumed w.l.o.g.
that $x'=(Z/u_2,u_1/u_2,u_2,u_3/u_2)$. Then $\iota (x')\leq (p, \omega (x),2)$ by (\ref{eq911})-(\ref{eq912})
and the conclusion follows.

\smallskip

\noindent {\it Case 2:} $\mathrm{Vdir}(x)=<\lambda_1U_1+U_2>$, $\lambda_1 \neq 0$. We now have $G=0$ and expand
$$
U_1^{-pd_1}U_2^{-pd_2}F_{p,Z}=U_3 (\lambda_1U_1+U_2)^{\omega(x)}+U_2Q, \ Q \in k(x)[U_1,U_2,U_3^p]_{\omega (x)}.
$$
If $Q\neq 0$,  {as $H^{-1}{\partial TF_{p,Z} \over \partial U_3}\subseteq < U_1^{\omega(x)}>$}, we expand again
$$
Q=\sum_{i=0}^{i_0} U_3^{pi}Q_{\omega (x)-i}(U_1,U_2),
\   {Q_{\omega (x)-i_0}(U_1,U_2)}\neq 0.
$$
Since $(u_1,u_2,u_3;Z)$ are well adapted coordinates, we have
$$
U_1^{pd_1}U_2^{pd_2+1}  {Q_{\omega (x)-i_0}(U_1,U_2)}\not \in G(m_S)^p.
$$

If $i_0=0$, we argue as in the proof of Proposition \ref{redto**4}, {\it cf.} (\ref{eq805}) {\it sqq.}:
after possibly picking new well adapted coordinates $(u_1,u_2,v;Z')$
at $x$, it can be assumed that $U_1$ divides $Q=Q_{\omega (x)}[U_1,U_2]$. We get $\omega (x')<\omega (x)$
if $Q\neq 0$; if $Q=0$, we obtain $\iota (x')\leq (p, \omega (x),2)$ or $x'$ satisfies
the assumptions of Lemma \ref{sortiemonome} (Lemma \ref{sortieomegaun} if $\omega (x)=1$), so $x'$ resolved.
In particular, the proof is complete if $\omega (x)<p$.

If $i_0\geq 1$, arguing as in case 1, we obtain $\omega (x')<\omega (x)$ except possibly if $k(x')=k(x)$
and
$$
a(1):=pd_1, \ a(2):=pd_2+1, \ a(3):=0, \ F_0:=Q_{\omega (x)-i_0}[U_1,U_2]
$$
satisfies the assumptions of Lemma \ref{joyeux} with $\lambda =\lambda_1^{-1}$. Then it can be assumed w.l.o.g.
that $x'=(Z/u_1,u_1, \gamma_1 +u_2/u_1,u_3/u_1)$, where $\gamma_1 \in S$ is a unit with residue $\lambda_1$.
We obtain $\iota (x')\leq (p, \omega (x),2)$ or $x'$ satisfies
the assumptions of Lemma \ref{sortiekappaegaldeuxbis}. Then $x'$ is resolved and this concludes the proof.

\end{proof}

\begin{prop}\label{T**eclatementcourbe}
Let $x$ be in the case (T**) of Definition \ref{**} and $\mathcal{Y}\subset ({\cal X},x)$
be a permissible curve of the first kind, $\eta ({\cal Y})\subset \mathrm{div}(u_1)$, with generic point $y$. Let
$$
    \pi : \ {\cal X}' \longrightarrow ({\cal X},x)
$$
be the blowing up along  $\mathcal{Y}$ and  $x' \in \pi^{-1}(x)$,  $\iota (x')\geq (p,\omega (x),3)$.
Assume furthermore that one of the following extra assumptions holds:
\begin{itemize}
  \item [(1)] $\mathrm{Vdir}(y)=<U_1>$;
  \item [(2)] ${\cal Y}=V(Z,u_1,u_3)$ and $x$ satisfies (T**)(ii) or (iii),
\end{itemize}
where $(u_1,u_2,u_3;Z)$ are well adapted coordinates. Then one of the following holds:
\begin{itemize}
  \item [(i)] $x'$ is resolved, or ($x'$ satisfies (**) with $\omega(x)\geq p$);
  \item [(ii)] $x'$ maps to the strict transform of div$(u_1)$ and satisfies (T**).
\end{itemize}
\end{prop}

\begin{proof}
As $\mathcal{Y}$ has normal crossings with $E$, we can choose in any case
well adapted coordinates $(u_1,u_2,u_3;Z)$ at $x$ such that $\mathcal{Y}=V(Z,u_1,u_i)$, $i=2$ or $i=3$.

\smallskip

Let us see the case where  $\mathcal{Y}=V(Z,u_1,u_2)$,  up to
renumbering $u_2, u_3$. As $\mathcal{Y}$ is a permissible curve of the first kind, we have,  {with the usual convention $d_3=0$ when div$(u_3)\not\subset E$}:
$$
U_1^{-pd_1}U_2^{-pd_2}U_3^{-pd_3}F_{p,Z}\in k(x)[U_1,U_2]_{\epsilon (x)}
$$
by Proposition \ref{firstkind}. This implies  $\epsilon(x)=\omega(x)$.

\smallskip

If $<U_1>\subseteq \mathrm{Vdir}(x)$, we are done by Theorem \ref{bupthm} unless equality holds and
$x'=(Z':=Z/u_2,u'_1:=u_1/u_2,u_2,u_3)$. We may therefore assume that $x$ satisfies (T**)(i). Note that
$(u'_1,u_2,u_3;Z')$ are well adapted coordinates at $x'$ by Proposition \ref{originchart}. The proof is trivial
under  assumption (1) and we get $x'$ resolved or (T**)(ii).
Under assumption (2) (with  $u_2,u_3$ relabeled), we have $E=\mathrm{div}(u_1u_2u_3)$ and there is an expansion
$$
u_1^{-pd_1}u_2^{-pd_2}u_3^{-pd_3}f_{p,Z}\equiv \gamma u_1^{\omega (x)} \ \mathrm{mod} u_3(u_1,u_2)^{\omega (x)},
$$
with $\gamma \in S$ a unit. We get $x'$ resolved or (T**)(ii).

\smallskip

Finally if $\mathrm{Vdir}(x)=<\lambda_1U_1+U_i>$,  {$\lambda_1\neq 0$, $i=2$ or $3$. Assumption (2) is true. So $i=3$ and $x$ is in case (T**)(ii)
with $E=\mathrm{div}(u_1u_2u_3)$. We are done by Theorem \ref{bupthm} unless
$$
x'=(X':=Z/u_1,u_1,u_2,v:=\gamma_1 + u_3/u_1), \ E'=\mathrm{div}(u_1u_2),
$$
where $\gamma_1 \in S$ is a unit with residue $\lambda_1$. Applying Proposition \ref{bupformula}(v)
(with $W':=\mathrm{div}(u_1)\subset \mathrm{Spec}S'$), we get
\begin{equation}\label{eq913}
J(F_{p,X',W'},E',W')=({\overline{v}}^{\omega(x)})\subseteq S/(u_1,u_3)[\overline{u}'_3]_{(\overline{v},\overline{u}_2)}.
\end{equation}
If $\iota (x_1)\geq (p,\omega (x),3)$, (\ref{eq913}) thus reads
$$
U_1^{-pd'_1}\overline{u}_2^{-pd_2}{\partial F_{p,X',W'} \over \partial \overline{v}}=({\overline{v}}^{\omega(x)}),
$$
where $d'_1:=d_1+d_3 +\omega (x)/p-1$, i.e. $x'$ satisfies condition (**).} This situation occurs only if
$(d'_1,d_2)\in \N^2$; therefore $x'$ is resolved for $m(x)=p$ if $\omega (x)<p$.

\smallskip

Let us now see the case where $\mathcal{Y}=V(Z,u_1,u_3)$, $E=\div(u_1u_2)$. If $\epsilon (x)=\omega (x)$,
we thus have $\mathrm{Vdir}(x)=<U_1>$ by Proposition \ref{firstkind}, in particular $x$ satisfies
(T**)(i) or (ii). We are done by Theorem \ref{bupthm} unless $x'=(Z/u_3,u_1/u_3,u_2,u_3)$.
The reader ends the proof easily as above, under either assumption (1) or (2): we get $x'$  resolved
or (T**)(ii).

If $\epsilon (x)=1+\omega (x)$, $x$ satisfies (T**)(iii)  {by definition}. By Proposition \ref{firstkind},
we have $H^{-1}F_{p,Z}=<U_3U_1^{\omega (x)}>$. Since $\mathrm{Vdir}(x)=<U_1>$, we are done by Theorem \ref{bupthm}
unless $x'=(Z/u_3,u_1/u_3,u_2,u_3)$. The reader ends the proof easily as before.
\end{proof}

\subsection{Resolution for  (**),  the end.}

The purpose of this section is to prove the following proposition and theorem which end the proof
of Projection Theorem \ref{projthm}.

\begin{prop}\label{END}
Assume that $x$ is in case (**) (Definition \ref{**}), then $x$ is resolved for $\iota=(p,\omega(x),3)$.
\end{prop}

\begin{proof}
This follows from Corollary \ref{corEprime} and Propositions \ref{**versgammaegal1} and \ref{**gammaegal1} below.
\end{proof}

\begin{thm}\label{proofkappa34}
Assume that $\kappa(x)\geq 3$, then $x$ is resolved.
\end{thm}
\begin{proof}
By Propositions \ref{redto**3} and \ref{redto**4}, it can be assumed that
$$\kappa(x)\geq 3,\ x \ \hbox{satisfies (**) or (T**)}.$$
By Proposition \ref{redto**casT**}, the remaining case is when $x$ satisfies (**).
This case is just the assumption of Proposition \ref{END}.
\end{proof}

\subsubsection{An extra assumption on the singular locus.}

The following extra assumption {\bf (E)'} is used as a shortcut in order to ensure that certain exceptional
curves on ${\cal X}$ are Hironaka-permissible and can be blown up in order to reduce $\omega (x)$.
Such blowing up centers are not used in \cite{Co5} and the authors
do not know if such blowing ups are relevant in dimension $n \geq 4$.

\begin{defn}\label{defEprime}
We say that $(S,h,E)$ satisfies condition {\bf (E)'} \index{{\bf (E)'}, condition {\bf (E)'}, Definition \ref{defEprime}} if it satisfies condition {\bf (E)} and if
$$
\omega (x)\geq p \Longrightarrow \eta^{-1}(E) = \mathrm{Sing}_p{\cal X}.
$$
where $\eta^{-1}(m_S)=:\{x\}$.
\end{defn}

 {Proposition \ref{EEprime} below will show that we can attain condition {\bf (E)'}.}
As stated after Definition \ref{conditionE}, we have in any case $\mathrm{Sing}_p{\cal X} \subseteq \eta^{-1}(E) $
whenever $(S,h,E)$ satisfies condition {\bf (E)}.

\begin{prop}\label{Eprimestable}
Let $\pi : {\cal X}'\rightarrow {\cal X}$ be a permissible blowing up (of the first or second kind)
at $x \in \eta^{-1}(m_S)$ and $x' \in \pi^{-1}(x)$. If $(S,h,E)$ satisfies condition {\bf (E)'},
then $(S',h',E')$ satisfies again {\bf (E)'} at $x'$.
\end{prop}

\begin{proof}
This reduces to Proposition \ref{Estable} if  $\omega (x)\leq p-1$.
Assume that $\omega (x)\geq p$, so we have $d_j\geq 1$, $1 \leq j \leq e$, by assumption {\bf (E)'}.
Let ${\cal Y} \subset {\cal X}$ be permissible
with generic point $y$, $W:=\eta ({\cal Y})=V(\{u_j\}_{j\in J}) \subset E$ and $I(W)S'=:(u)$, where
$$
\eta ' : \ ({\cal X}',x') \longrightarrow \mathrm{Spec}S'
$$
is the projection. By Definition \ref{deffirstkind} or Proposition \ref{secondkind},
we have $\epsilon (y)\geq \omega (x)\geq p$. Applying Proposition \ref{bupformula}(iv), we have
$H(x')=u^{\epsilon (y)-p}H(x)S'$, therefore
$$
\mathrm{ord}_{(u)}H(x')=\epsilon (y)-p + \mathrm{ord}_{W}H(x)\geq \min\{pd_j: j\in J_E\}\geq p
$$
and the conclusion follows.
\end{proof}

\begin{cor}\label{corEprime}
It can be assumed that condition {\bf (E)'} holds in the proof of Proposition \ref{END} and Theorem \ref{proofkappa34}.
\end{cor}

\begin{proof}
All blowing ups used in the proofs of Propositions \ref{redto**3}, \ref{redto**4}
and \ref{redto**casT**} are permissible of the first kind.
\end{proof}

\begin{lem}\label{lemEEprime}
 {Assume that condition {\bf (E)'} does not hold at $x$.} Let $\mu$ be a valuation of $L=k({\cal X})$ centered at $x$.
There exists a finite and independent composition of local permissible blowing ups of the first kind:
$$
    ({\cal X},x)=:({\cal X}_0,x_0) \leftarrow ({\cal X}_1,x_1) \leftarrow \cdots \leftarrow ({\cal X}_r,x_r) ,
$$
where $x_i \in {\cal X}_i$ is the center of $\mu$, such that $x_r$ is resolved or $H(x_r)\neq (1)$.
\end{lem}

\begin{proof}
 {By definition of  condition {\bf (E)'}}, $\omega (x)\geq p$. Since resolved means ``resolved for $(p,\omega (x),3)$'' in this section,
it can be assumed that
$$
\omega (x_i)=\omega (x), \ \kappa (x_i)\geq 3
$$
for every $i\geq 0$ along the process to be defined. Note that $\mathrm{ord}_{m_{S_1}}H(x_1)>0$ is
achieved by blowing up $x$ if $\delta (x)>1$.

\smallskip

Assume now that $\delta (x)=1$, i.e. $\tau (x)\geq 2$ (Definition \ref{defmult}).
Since $\kappa (x)\geq 3$ and $\epsilon (x)=\omega (x)=p$, we actually have $\kappa (x)=4$  {and $G=0$ as $\kappa(x)>1$}, i.e.
\begin{equation}\label{eq820}
\mathrm{in}h =Z^p +F_{p,Z}, \ 0\neq F_{p,Z} \in k(x)[U_1, \ldots ,U_e]_p,
\end{equation}
where $(u_1,u_2,u_3;Z)$ are well adapted coordinates.

\smallskip

$\bullet$ if $\tau '(x)=3$, let ${\cal X}' \rightarrow ({\cal X},x)$ be the blowing up along $x$.
Then $x$ is resolved by Theorem \ref{bupthm}.

\smallskip

$\bullet$  if $\tau '(x)=2$, let also ${\cal X}' \rightarrow ({\cal X},x)$ be the blowing up along $x$.
W.l.o.g. we have
\begin{equation}\label{eq821}
\mathrm{Vdir}(x)=<U_1 + \alpha_1 U_3, U_2 +\alpha_2 U_3>, \ \alpha_1 , \alpha_2\in k(x),
\end{equation}
where $\mathrm{div}(u_1u_2)\subseteq E$, and $E=\mathrm{div}(u_1u_2u_3)$ if $(\alpha_1,\alpha_2)\neq (0,0)$.  {As $H(x)=(1)$,}
\begin{equation}\label{eq822}
 {<\{{\partial F_{p,Z} \over \partial \lambda_l}\}_{l\in \Lambda_0}>\subseteq k(x)[U_1 + \alpha_1 U_3, U_2 +\alpha_2 U_3],}
\end{equation}
 {where  $(\lambda_l)_{l \in \Lambda_0}$ is an absolute $p$-basis of $S/m_S$ (see beginning of section \ref{section:diff}).
By Theorem \ref{bupthm}, we have $k(x_1)=k(x)$ and $x_1$ has for parameters $(Z/u_3,u_1/u_3+\gamma_1,u_2/u_3+\gamma_2,u_3)$ where $\gamma_i$ has residue $\alpha_i\in S$, $i=1,2$. By Proposition \ref{bupformula}(v), }
\begin{equation}\label{eq822bis}
 {U_3^{-p}{\partial F_{p,Z} \over \partial \lambda_l} \in J(F_{p,Z',W},E_1,W), \ l \in \Lambda_0,}
\end{equation}
 {where $W=$div$(u_3)$}. %

%



If $\alpha_1\alpha_2\neq 0$, as $\kappa(x_1)>2$, we therefore have $F_{p,Z}\in k(x)^p[U_1,U_2,U_3]$. In particular,
$$
0<d:=\mathrm{deg}_{U_1}F_{p,Z}<p,\  {F_{p,Z}:=\sum_{i=0}^{i=d}U_1^i \Phi_i(U_2,U_3)}
$$
since $\Delta_S(h;u_1,u_2,u_3;Z)$ is minimal. Lemma \ref{joyeux}(ii)
applied to the term   {$U_1^d\Phi_d(U_2,U_3)$} of $F_{p,Z}$,  {after a relabeling $U_3 \leftrightarrow U_1$}, gives a contradiction with (\ref{eq821}),
since  {$d=\widehat{d}\not \equiv 0 \ \mathrm{mod}p$, $d$ corresponds to $a$ in Lemma \ref{joyeux}(ii)}.

We now assume that $\alpha_1=0$.

If $\alpha_2 \neq 0$, we derive a contradiction in a similar way: by (\ref{eq822bis}), the coefficient of
degree 0 in $U_1$ in $F_{p,Z}$ must be zero; Lemma \ref{joyeux}(ii) applied to the term of minimal
degree $d$ in $U_1$  of $F_{p,Z}$ gives again a contradiction, since $0<d<p$.
This proves that $\mathrm{Vdir}(x)=<U_1,U_2>$, $ {F_{p,Z}\in k(x)[U_1,U_2]}$ {mod}$ {G(m_S)^p}$.

By Proposition \ref{originchart}, we have $\delta (x_1)=1$ and may iterate. By Proposition \ref{permisarc},
this process is finite or the curve ${\cal Y}:=V(Z,u_1,u_2)$ is permissible of the first kind.
Since $\mathrm{Vdir}(x)=<U_1,U_2>$, blowing up along ${\cal Y}$ then completes the proof.

\smallskip

$\bullet$  if $\tau '(x)=1$, it can be assumed that (\ref{eq820}) has the form
\begin{equation}\label{eq823}
\mathrm{in}h =Z^p +\lambda (U_1 +\alpha_2U_2 +\alpha_3U_3)^p, \ \lambda \not \in k(x)^p,
\end{equation}
with $\mathrm{div}(u_1)\subseteq E$, and $\mathrm{div}(u_j)\subseteq E$ if $\alpha_j\neq 0$, $j=2,3$.

If $\alpha_2\alpha_3\neq 0$, let ${\cal X}' \rightarrow ({\cal X},x)$ be the blowing up along $x$.
We get a contradiction with $\kappa (x_1)\geq 3$ unless $\lambda \in k(x_1)^p$; but then $\delta (x_1)=1$
implies that $x_1$ satisfies the assumptions of Lemma \ref{sortiemonome} from which the conclusion follows.
We now assume that $\alpha_3=0$.

If $\alpha_2 \neq 0$, let also ${\cal X}' \rightarrow ({\cal X},x)$ be the blowing up along $x$.
The previous argument works in the same way unless $x_1=(Z/u_3,u_1/u_3,u_2/u_3,u_3)$.
Then $x_1$ satisfies again (\ref{eq823}) for some $\alpha_3\in k(x)$ and we may iterate.
By Proposition \ref{permisarc}, this process is finite or the curve
${\cal Y}:=V(Z,u_1,u_2)$ is permissible of the first kind and we blow up along ${\cal Y}$.
But then $k(x_1)=k(x)$, and this gives a contradiction with $\kappa (x_1)\geq 3$. Therefore
the Lemma is proved unless
\begin{equation}\label{eq824}
\mathrm{in}h =Z^p +\lambda U_1^p, \ \lambda \not \in k(x)^p, \ \mathrm{div}(u_1)\subseteq E.
\end{equation}

We now define a refinement ${\cal C}$ of the function $x\mapsto (m(x),\omega (x))$, {\it cf.} chapter 6.
Let $\pi : \ {\cal X}' \rightarrow ({\cal X},x)$ be the blowing up along a permissible center of the
first kind ${\cal Y}\subseteq \mathrm{div}(u_1)$, $x'\in \pi^{-1}(x)$. We set:
$$
{\cal C}(x')<{\cal C}(x) \Leftrightarrow x' \ \mathrm{satisfies} \ \mathrm{the} \ \mathrm{conclusion}
\ \mathrm{of} \ \mathrm{the} \ \mathrm{lemma}.
$$
By Theorem \ref{bupthm}, we have ${\cal C}(x')<{\cal C}(x)$ unless $x' \in \mathrm{div}(u'_1)$,
where $\mathrm{div}(u'_1)\subseteq E'$ is the strict transform of $\mathrm{div}(u_1)$. Otherwise, we
let  ${\cal C}(x')={\cal C}(x)$.

With notations as in chapter 6, we claim that $\mathrm{div}(u_1)$ has maximal contact
for the condition ${\cal C}$ (Definition \ref{Maximalcontact}). To see this, suppose that
${\cal C}(x')={\cal C}(x)$. Note that $\delta (x')>1$ or $x'$ satisfies the assumptions of Lemma \ref{sortiemonome}
if $\lambda \in k(x')^p$: a contradiction. If $\delta (x')=1$ and $\lambda \not \in  k(x')^p$,
we get an expansion
$$
\mathrm{in}h'={Z'}^p +F_{p,Z'}, \ 0\neq F_{p,Z'} \in k(x)[U'_1, \ldots ,U'_{e'}]_p,
$$
where $(u'_1,u'_2,u'_3;Z')$ are well adapted coordinates at $x'$, and the leading coefficient
of $F_{p,Z'}$ in $U'_1$ is $\lambda {U'_1}^p$. Since ${\cal C}(x')={\cal C}(x)$ is assumed,
we actually have
$$
\mathrm{in}h' ={Z'}^p +\lambda {U'_1}^p
$$
by (\ref{eq824}) and the claim is proved. The conclusion now follows from Theorem \ref{contactmaxFIN}.
\end{proof}

\begin{prop}\label{EEprime}
Let $\mu$ be a valuation of $L=k({\cal X})$ centered at $x$.
There exists a finite and independent composition of local permissible blowing ups of the first kind:
\begin{equation}\label{eq825}
    ({\cal X},x)=:({\cal X}_0,x_0) \leftarrow ({\cal X}_1,x_1) \leftarrow \cdots \leftarrow ({\cal X}_r,x_r) ,
\end{equation}
where $x_i \in {\cal X}_i$ is the center of $\mu$, such that $x_r$ is resolved or $x_r$ satisfies
condition {\bf (E)'}.
\end{prop}

\begin{proof}
It can also be assumed that $\omega (x)\geq p$ and that
$$
\omega (x_i)=\omega (x), \ \kappa (x_i)\geq 3
$$
for every $i\geq 0$ along the process to be defined. By Lemma \ref{lemEEprime}, we may assume that $H(x)\neq (1)$
to begin with. Order
$$
d_1\geq  \cdots \geq d_e \geq 0=:d_{e+1},  \ d_1>0,
$$
where $E=\mathrm{div}(u_1 \cdots u_e)$. We define $e_0$, $1 \leq e_0\leq e$, by:
$$
\min\{1,d_{e_0}\}=\min\{1,d_1\} \ \mathrm{and} \ d_{e_0+1}<\min\{1,d_1\}.
$$
The invariant is:
$$
d(x):=(d'(x):=\max\{0,1-d_1\}, d''(x):=e-e_0)_{\mathrm{lex}}.
$$
Note that $d(x)=(0,0)$ if and only if $x$ satisfies condition {\bf (E)'}.

\smallskip

Let $\pi : \ {\cal X}' \rightarrow ({\cal X},x)$ be the blowing up along a permissible
center of the first kind ${\cal Y}$ and $x'\in \pi^{-1}(x)$:  {${\cal Y}=x$ or ${\cal Y}$ is an irreducible  curve of generic point $y$ with $\epsilon(y)=\epsilon(x)$}. We
refine the function $x\mapsto (m(x),\omega (x))$, {\it cf.} chapter 6, by setting:
$$
{\cal C}(x')<{\cal C}(x)\Leftrightarrow d(x')< \min\{d(x), (d'(x),1)\}.
$$
Otherwise, we let  ${\cal C}(x')={\cal C}(x)$. To prove the proposition, it is sufficient
to prove that there exists a sequence (\ref{eq825}) such that ${\cal C}(x_r)<{\cal C}(x)$.
We claim the following: assume that
\begin{equation}\label{eq830}
\eta ({\cal Y})\subset \mathrm{div}(u_j) \ \mathrm{for} \ \mathrm{some} \ j, \ 1 \leq j \leq e_0.
\end{equation}
Then $d(x')\leq d(x)$; if $d(x')=d(x)$ (resp. if ${\cal C}(x')={\cal C}(x)$),
then $x'$ belongs to the strict transform of $\mathrm{div}(u_j)$ for every $j$ (resp. for some $j$)
such that $e_0<j\leq e$.

\smallskip

To prove this claim, let $W:=\eta ({\cal Y})$ and $I(W)S'=:(u)$, where
$$
\eta ' : \ ({\cal X}',x') \longrightarrow \mathrm{Spec}S'
$$
is the projection. By Proposition \ref{bupformula}(iv), we have
$H(x')=u^{\epsilon (y)-p}H(x)S'$. Therefore:
\begin{equation}\label{eq831}
 {\mathrm{ord}_uH(x')\over p}= {\epsilon (x) \over p}-1+ {\mathrm{ord}_{W}H(x) \over p}\geq \min\{1,d_1\}
\end{equation}
by (\ref{eq830})  {and $\epsilon(x)\geq \omega(x)\geq p$}.  {We get $d'_1\geq d_1$: either   $x'$ is on the strict transform of div$(u_1)$,  or we can take $u=u_1$, so  ${\mathrm{ord}_{W}H(x) \over p}\geq d_1$ in \eqref{eq831}.}
We get
$$
d'(x')= \max\{0,1-d'_1\}\leq \max\{0,1-d_1\}=d'(x).
$$
If equality holds,
(\ref{eq831}) implies that $\min\{1,d'_1\} = \mathrm{ord}_uH(x')/ p$, i.e.
$$
{\mathrm{ord}_uH(x')\over p}=d'_{j'} \ \mathrm{for} \ \mathrm{some} \ j', \ 1 \leq j '\leq e'_0:=e_0(x').
$$
The claim follows immediately.

\smallskip

We now define $\Omega (x)\subset ({\cal X},x)$ to be the Zariski closure of the set:
$$
\Omega^\circ (x):=\{y\in {\cal X}: m(y)=p, \ \omega (y)>0 \ \mathrm{and}
\ \forall j, \ 1 \leq j \leq e_0, \ y \not \in \mathrm{div}(u_j)\}.
$$
By Proposition \ref{omegapositiveclosed}, $\Omega (x)$ is a (possibly empty) curve.
Note that
\begin{itemize}
  \item [(1)] $\Omega (x')$ is the strict transform of $\Omega (x)$ in $({\cal X}',x')$ if
${\cal Y}$ satisfies (\ref{eq830}), and
  \item [(2)] $\Omega (x)=\emptyset$ if $e_0\geq 2$ or if $d''(x)=0$.
\end{itemize}
We consider two cases:

\smallskip

\noindent {\it Case 1:} $\Omega (x)=\emptyset$. This implies that any permissible
center of the first kind ${\cal Y}$ satisfies (\ref{eq830}).  {As we are done if $d''(x)=0$,} by the above claim,
there exists $j$, $e_0<j \leq e$ such that $\mathrm{div}(u_j)$ has maximal contact
for the condition ${\cal C}$. By Theorem \ref{contactmaxFIN}, we obtain
a sequence (\ref{eq825}) such that ${\cal C}(x_r)<{\cal C}(x)$.

\smallskip

\noindent {\it Case 2:} $\Omega (x)\neq \emptyset$. Consider the quadratic sequence
along $\mu$. By the above claim and (1), we either obtain ${\cal C}(x_r)<{\cal C}(x)$
(in particular if we reach case 1), or achieve that $\Omega (x_{r_1})$ is irreducible for some $r_1\geq 0$;
by Proposition \ref{permisarc}, it can be furthermore assumed that $\Omega (x_{r_1})$
is permissible of the first kind when the latter holds. Let then $y_1\in ({\cal X}_{r_1},x_{r_1})$
be the generic point of $\Omega (x_{r_1})$. By (2), we also have:
\begin{equation}\label{eq832}
e_0(x_r)=e_0=1 \ \mathrm{and} \ d''(x_{r_1})\geq 1.
\end{equation}
Let $\pi_1 : \ {\cal X}' \rightarrow ({\cal X}_{r_1},x_{r_1})$ be the blowing up along $\Omega (x_{r_1})$ and
$x'\in \pi_1^{-1}(x)$. Since $d'(x')\leq d'(x_{r_1})=d'(x)$, we have ${\cal C}(x')<{\cal C}(x)$ or
are done by (1) and case 1 unless
$$
d'(x')=d'(x), \ e'_0:=e_0(x')=1 \ \mathrm{and} \ d''(x')\geq 1.
$$
Then $\pi_1$ restricts to a finite morphism
\begin{equation}\label{eq833}
\Omega (x') \longrightarrow \Omega (x_{r_1}).
\end{equation}

We now iterate this construction: this constructs a sequence
$$
({\cal X}_{r_1},x_{r_1}) \leftarrow ({\cal X}_{r_2},x_{r_2}) \leftarrow \cdots \leftarrow ({\cal X}_{r_k},x_{r_k})
 \leftarrow \cdots
$$
where $x_{r_i} \in {\cal X}_{r_i}$ is the center of $\mu$. If ${\cal C}(x_{r_k})={\cal C}(x)$, there is an induced
two-dimensional quadratic sequence
$$
({\cal X}_{r_1},y_1) \leftarrow ({\cal X}_{r_2},y_2) \leftarrow \cdots \leftarrow ({\cal X}_{r_k},y_k)
\leftarrow \cdots
$$
where $y_k \in ({\cal X}_{r_k},x_{r_k})$ is the  generic point of the permissible curve $\Omega (x_{r_k})$
by (\ref{eq833}). By two-dimensional resolution, we have $(m(y_k),\omega (y_k))< (p,\omega (x))$ for $k>>0$:
a contradiction with permissibility. Therefore the above sequence achieves ${\cal C}(x_{r_k})<{\cal C}(x)$
for some $k\geq 0$ and the proof is complete.
\end{proof}

\subsubsection{Proof of  Proposition \ref{END}. }

From now on, we assume that \textbf{(E)'} is satisfied.

\begin{defn}\label{kappa3preparation} {\bf (Preparation)}.
Assume that $x$ is in case (**) (Definition \ref{**}). We define\index{$\Delta_2$3@$\Delta_2$ in case
 $\kappa(x)\geq 3$(**)}
\begin{equation}
\begin{array}{ccc}
  \mathrm{pr}:\{(x_1,x_2,x_3)\in \R_{\geq 0}^3 \vert \  {x_3} < {1+\omega(x) \over p}\} \longrightarrow \R_{\geq 0}^2,   \\
   {(x_1,x_2,x_3) \mapsto \mathbf{x}:={1\over { 1+\omega(x)\over p}-x_3}(x_1-d_1,x_2-d_2)}
\end{array}
.
\end{equation}
as the translation by the vector $(-d_1,-d_2,0)$ followed by projection  from the point $(0,0,{1+\omega(x) \over p})$
over the $(x_1,x_2)$-plane, followed by the homothety of ratio ${p \over 1+\omega(x)}$.
We will write $\Delta_2(h;u_1,u_2;v;Z)$, even $\Delta_2$ if no confusion is possible instead of pr$\Delta_S(h;u_1,u_2,v;Z)$ for short.

Let $\mathbf{x}$ be a vertex of $\Delta_2$. We say that $\mathbf{x}$ is a  {\textit{left vertex}} if its ordinate is bigger or equal  {to} the ordinate of the vertex of  {biggest} ordinate of the side of slope $-1$.

Let $\mathbf{x}$ be a vertex of $\Delta_2$. Let pr$^{-1}(\mathbf{x})$ the edge of $\Delta(h;u_1,u_2,v;Z)$ giving $\mathbf{x}$ by projection, this edge is defined by an equation $\alpha_1 x_1+ \alpha_2 x_2 + \alpha_3 x_3=1$, $\alpha_1 \alpha_2 \alpha_3>0$, as usual we define the monomial valuation $v_{\alpha_{\mathbf{x}}}$ by
$$v_{\alpha_{\mathbf{x}}}(Z)=1,\  v_{\alpha_{\mathbf{x}}}(u_1)=\alpha_1, \  v_{\alpha_{\mathbf{x}}}(u_2)=\alpha_2, \  v_{\alpha_{\mathbf{x}}}(v)=\alpha_3.$$
We say that   {$(u_1,u_2,v,Z)$ is prepared for $\mathbf{x}$} if
$$Z^p- G_{\mathbf{x}}^{p-1}Z+ F_{p,Z,\mathbf{x}}:=\clin_{\alpha_{\mathbf{x}}}(h)\in k(x)[Z,U_1,U_2,V]$$
verifies one of the following:

\noindent 1- either $G_{\mathbf{x}}\not=0$,

\noindent 2-  {or}  $H^{-1}{\partial F_{p,Z,\mathbf{x}}\over \partial V}$ is not proportional to an $\omega(x)$-power,

\noindent 3- or  $H^{-1}{\partial F_{p,Z,\mathbf{x}}\over \partial V}=\lambda V^{\omega(x)}$, $\lambda \in k(x)^*$.

We say that $(Z,u_1,u_2,v)$ is totally prepared if

(i) $\Delta_S(h;u_1,u_2,v;Z)$  {is} minimal,

(ii) when $pd_2=0$ (f.i. when $E=\div(u_1)$), all the left vertices of $\Delta_2(h;u_1,u_2;v;Z)$ are prepared,

(iii)
when $pd_1>0$ and $pd_2>0$ ($\Leftrightarrow \ E=\div(u_1u_2)$ when $\omega(x)\geq p$), all the vertices of $\Delta_2(h;u_1,u_2;v;Z)$ are prepared.
\end{defn}

\begin{prop}\label{kappa3prepatot}
Assume that $x$ is in case (**) Definition \ref{**}. There exists $v\in S$, $\phi \in S$ such that
$(Z-\phi,u_1,u_2,v)$ is totally prepared. Furthermore $x$ is resolved for $m(x)=p$ if $\Delta_2(h;u_1,u_2;v;Z-\phi)=\emptyset$.
\end{prop}

\begin{proof}
We apply a strategy similar to Hironaka's strategy of minimizing in \cite{H3}.
Let us start by  a vertex  $\mathbf{x}=(x_1,x_2)$ not prepared. With the notations as above,
we have $\clin_{\alpha_{\mathbf{x}}}(h)=Z^p+ F_{p,Z,\mathbf{x}}$, with
$$
U_1^{-pd_1}U_2^{-pd_2}F_{p,Z,\mathbf{x}}=\lambda V^{1+\omega(x)}+\sum_{1\leq j \leq 1+\omega(x)}\lambda_j V^{1+\omega(x)-j}U_1^{jx_1}U_2^{jx_2},
\ \lambda \in k(x)^*,
$$
$$
U_1^{-pd_1}U_2^{-pd_2}{\partial F_{p,Z,\mathbf{x}} \over \partial V}
=(1+\omega(x))\lambda(V+\lambda' U_1^{x_1}U_2^{x_2})^{\omega(x)}, \ \lambda' \in k(x)^* ,$$
in particular, $x_i \in \N$, $i=1,2$. We take any invertible $\gamma_{\mathbf{x}}\in S$ whose residue is $\lambda'$
and we define
$$w:=v+ \gamma_{\mathbf{x}} u_1^{x_1}u_2^{x_2}.$$
Then $(Z,u_1,u_2,w)$ is a regular system of parameters of $S$.
$$\Delta_2(h;u_1,u_2;w;Z)\subset \Delta_2(h;u_1,u_2;v;Z).$$
Furthermore, let $\mathbf{y}=(y_1,y_2)$ another vertex of $\Delta_2(h;u_1,u_2;w;Z)$, let
$$
 \alpha'_1 x_1+ \alpha'_2 x_2 + \alpha'_3 x_3=1
$$
be an equation of the edge of $\Delta_S(h;u_1,u_2,v;Z)$ defined by  $\mathbf{y}$,
of course $v_{\alpha_{\mathbf{y}}}(u_1^{x_1}u_2^{x_2})>1$, so
$\clin_{\alpha_{\mathbf{y}}}(v)=\clin_{\alpha_{\mathbf{y}}}(w)$.
In particular, $ {\mathbf{y}}$ is still a vertex of  $\Delta_2(h;u_1,u_2;w;Z)$ and, if it was prepared for $(u_1,u_2,v;Z)$,
it is still prepared for $(u_1,u_2,w;Z)$. Furthermore, if we make an eventual translation on $Z\leftarrow Z-\phi$,
$\phi \in S$ to minimize $\Delta_S(h;u_1,u_2,w;Z)$, as $\clin_{\alpha_{\mathbf{y}}}(v)=\clin_{\alpha_{\mathbf{y}}}(w)$,
in the of expansion  {of} $\clin_{\alpha_{\mathbf{y}}}(h)$, we just change  $\clin_{\alpha_{\mathbf{y}}}(v)$ by $\clin_{\alpha_{\mathbf{y}}}(w)$:
we can choose $\phi$ with  $v_{\alpha_{\mathbf{y}}}(\phi)>1$. So
$$\Delta_2(h;u_1,u_2;w;Z-\phi)\subset \Delta_2(h;u_1,u_2;v;Z),$$
any prepared vertex $\mathbf{y}=(y_1,y_2)$ of $\Delta_2(h;u_1,u_2,v;Z)$ is a prepared vertex of $\Delta_2(h;u_1,u_2;w;Z-\phi)$.

We apply this process to each  $\mathbf{x}=(x_1,x_2)$ to be prepared, starting by those of smallest modules. When this process is finite, we get the announced result.

When this process is infinite, we get $\phi, \psi \in \hat{S}$ such that $(u_1,u_2,v-\psi ; Z-\phi)$ is totally prepared.
Let us remark that $x$ is resolved if
$$\Delta_2(h;u_1,u_2;w;Z-\phi)\not=\emptyset.$$
The contrary  would mean that $\Delta(h;u_1,u_2;w;Z-\phi)$ has only one vertex $(d_1,d_2, {1+\omega(x)\over p})$:
$\V(Z-\phi, w)$ would be a component of dimension $two$ of the locus of multiplicity $\min\{p, 1+\omega (x)\}$,
$\eta( \V(Z-\phi, w))\nsubseteq E$. This contradicts {\bf (E)} if $\omega (x)\geq p$ or if $h$ is separable
(assumption (ii) in Theorem \ref{luthm}). If $\omega(x)<p$ and $h=Z^p +f_{p,Z}$, $\mathrm{char}S=p$,
$x$ is resolved for $m(x)=p$ by a combinatorial algorithm, {\it vid.} proof of Theorem \ref{omegazero}.

The remark above implies that, after a finite number of steps, we apply infinitely the process to vertices of smallest abscissa or
(smallest ordinate and $E=\div(u_1u_2)$) of $\Delta_2(h;u_1,u_2 {;}v;Z)$ and this smallest abscissa or smallest ordinate remains constant.

Let us study  the very special case where   $\mathbf{x}:=(A,\beta)$ is the vertex of smallest abscissa  of $\Delta_2$
and that the process dissolves it, creating a new vertex $(A,\beta')$, $\beta'>\beta$ infinitely times.
This implies $A,\beta,\beta' \in \N$.

Let $\alpha=(\alpha_1,0,\alpha_3)$, such that  $\alpha_1 x_1+ \alpha_3 x_3=1$ is the equation of the non compact face
of $\Delta_S(h;u_1,u_2,v;Z)$ whose image by pr is the non compact face $x_1=A$ of $\Delta_2$.
We get $\alpha_1pd_1+\alpha_3(1+\omega(x))=p$, $\alpha_1 A- \alpha_3=0$, and
$$
\clin_{\alpha}h=Z^p-  G^{p-1}_{\mathbf{x}}Z+ F_{p,\mathbf{x}}\in \mathrm{gr}_{\alpha}(S[Z])={S\over (v,u_1)}[U_1,V][Z].$$
Let ${\cal C}:=\mathrm{Spec}{S\over (v,u_1)}$. By quasi-homogeneity and the uniqueness of the solution \cite{H3} Corollary (4.1.1),
there exists $\Phi\in  \widehat{{\cal O}_{\cal C}[U_1,V]}=  \widehat{{\cal O}_{\cal C}}[[U_1,V]]$ with
$$
\Phi^p\in U_1^{pd_1}(V, U_1^A)^{1+\omega (x)}, \ \Psi \in \bar{u_2}^{{ \beta }} U_1^{{A }} \widehat{{\cal O}_{\cal C}}[[U_1,V]],
$$
such that
\begin{equation}\label{eq834}
\clin_{\alpha}h=(Z-\Phi)^p+U_1^{pd_1}\overline{\gamma}  (V-\Psi)^{1+\omega(x)}.
\end{equation}

\begin{lem}\label{nettoyageinfini}
There exists
$$
\phi \in S, \ \phi^p\in u_1^{pd_1}(v, u_1^A)^{1+\omega (x)} \ \mathrm{and} \
w\in S, \ v-w \in (v^2,u_1^A u_2^{\beta} )S
$$
such that
$$\clin_{\alpha}h=(Z-\clin_{\alpha}\phi)^p+U_1^{pd_1}\overline{\gamma}  W^{1+\omega(x)}. $$
\end{lem}

When $\omega (x)\geq p$, (\ref{eq834}) means that $\V(Z-\Phi,V-\Psi)$ is the only component in the locus of multiplicity $p$ of
$$
\Xi:=\mathrm{Spec}(\widehat{{\cal O}_{\cal C}}[[U_1,V]]/(\clin_{\alpha}h))
$$
not contained in div$(U_1)$. Since ${\cal O}_{\cal C}[U_1,V]$ is excellent and Noetherian,
by \cite{CoJS} Lemma \textbf{1.37}, this component is algebraic and the conclusion follows.

When $\omega (x)< p$,   $\V(Z-\Phi,V-\Psi)$ is the only component in the locus of multiplicity $1+\omega(x)$ of
$\Xi$  not contained in div$(U_1)$: we conclude as above. This ends the proof of Lemma \ref{nettoyageinfini}.

\smallskip

Let us remark that, if there exists another vertex ${\mathbf{x}}_1$ which is already prepared, then
$$\clin_{\alpha_{\mathbf{x}_1}}(Z)=\clin_{\alpha_{\mathbf{x}_1}}(Z-\phi),\ \clin_{\alpha_{\mathbf{x}_1}}(v)= \clin_{\alpha_{\mathbf{x}_1}}(w),$$
so  ${\mathbf{x}}_1$ is still prepared for $(u_1,u_2,w;Z-\phi)$.

By applying Lemma \ref{nettoyageinfini}, we see that there exists $\phi\in S$ and $w\in S$ such that the vertex of smallest abscissa of $\Delta_2(h;u_1,u_2;w;Z-\phi)$ is prepared.

The case where the process is infinite along points of smallest ordinates is, mutatis mutandis, the same: by applying the remark above,  we see that, when $E=\div(u_1u_2)$,  there exists $\phi\in S$ and $w\in S$ such that both the vertices of smallest abscissa and smallest ordinate of $\Delta_2(h;u_1,u_2;w;Z-\phi)$ are prepared. This ends the proof of Proposition \ref{kappa3prepatot}.
\end{proof}

\begin{defn}\label{kappa3invariants} {\bf (Invariants).} \index{$A$3@$A_1,\  A_2,\ \beta, \ B, \ C, \ \gamma$ in the case $\kappa(x)=3$(**)}
Suppose  $\kappa(x)= 3$, suppose that $(Z,u_1,u_2,v)$ is totally prepared.
In the case where $E=\div(u_1u_2)$  {and $\omega(x)<p$ (so $\epsilon(x)< p$ by Definition \ref{**} and $d_1+d_2+\epsilon(x)\geq p$)} , we choose $u_1$ so that $d_1>0$ and let

\smallskip

\noindent (i) $(A_1(Z,u_1,u_2,v),\beta(Z,u_1,u_2,v))$ is the vertex of smallest abscissa of $\Delta_2$;

\smallskip

\noindent (ii) $B(Z,u_1,u_2,v) {:=}\mathrm{inf}\{ \vert \mathbf{x} \vert \ \vert \mathbf{x}\in  {\Delta_2}\}$;

\smallskip

\noindent (iii) $A_2(Z,u_1,u_2,v)$ is the inf of the ordinates of points in  $\Delta_2$,
$$
C(Z,u_1,u_2,v) {:=}B(Z,u_1,u_2,v)-A_1(Z,u_1,u_2,v)-A_2(Z,u_1,u_2,v);
$$

\noindent (iv) $\gamma(Z,u_1,u_2,v) \in \N$ is given by:
$$
\gamma(Z,u_1,u_2,v):=
\left\{
  \begin{array}{ccc}
    \lceil \beta(Z,u_1,u_2,v)\rceil & \mathrm{if} & E=\div(u_1) \hfill{} \\
     \\
    1+\lfloor C(Z,u_1,u_2,v)\rfloor  & \mathrm{if} &  E=\div(u_1u_2) \hfill{} \\
  \end{array}
\right.
.
$$

For sake of simplicity, most of the time, we will skip $(Z,u_1,u_2,v)$ and write $A_1(x),A_2(x),B(x),C(x),\beta(x), \gamma(x)$.
\end{defn}

\begin{prop}\label{kappa3instable}
Suppose  $x$ satisfies conditions (**) and {\bf (E)'} with  $\kappa(x)=3$ and $(Z,u_1,u_2,v)$ is totally prepared.
The following holds:
\begin{itemize}
  \item [(i)] $V\in \mathrm{Vdir}(x)$ or $x$ is resolved;
  \item [(ii)] if $B(x)=1$ and  $E=\div(u_1)$,  $x$ is resolved or
  $$
  x':=(Z',u'_1,u'_2,v')=(Z/u_2,u_1/u_2,u_2,v/u_2)
  $$
  is the unique closed point $x_1\in \pi^{-1}(x)$ in the blowing up $\pi : {\cal X}'\rightarrow {\cal X}$
  along $x$ such that $\iota (x_1)\geq \iota (x)$, and $x'$ then satisfies conditions (**) and {\bf (E)'};
  \item [(iii)] if $B(x)=1$ and  $\omega(x)<p$, $x$ is resolved.
\end{itemize}

\end{prop}

\begin{proof}
When $B(x)>1$,   clearly  $V\in\mathrm{Vdir}(x)$. When $B(x)=1$, then
$$
U_1^{-pd_1}U_2^{-pd_2}F_{p,Z}=\lambda V^{1+\omega(x)}+\sum_{1\leq i \leq 1+\omega(x)} V^{1+\omega(x)-i}Q_i(U_1,U_2), \ \lambda \neq 0.
$$
Suppose $V\not\in \mathrm{Vdir}(x)$, then
$$U_1^{-pd_1}U_2^{-pd_2}{\partial F_{p,Z} \over \partial V}\not= (1+\omega(x))\lambda V^{\omega(x)},$$
so $\tau '(x)\geq 2$ by total preparedness. By  {Lemma \ref{kappa3prelim}},  $x$ is resolved except possibly if
$$\mathrm{Vdir}(x)=<V+aU_2,U_1>, \ a \in  {k(x)^*}$$
up to renumbering $u_1,u_2$ if $E=\mathrm{div}(u_1u_2)$.

\smallskip

 {As $a\neq 0$},  it would mean that $\mathbf{x}=(0,1)$ is a vertex of $\Delta_2$. This implies that
$$
U_1^{-pd_1}U_2^{-pd_2}{\partial F_{p,Z} \over \partial V}=({1+\omega(x)})\lambda (V+aU_2)^{\omega(x)}
+\sum_{1\leq i \leq \omega(x)} \lambda_i (V+aU_2)^{1+\omega(x)-i} U_1^i,
$$
so $H^{-1}{\partial F_{p,Z,\mathbf{x}}  \over \partial V}=({1+\omega(x)})\lambda (V+aU_1)^{\omega(x)}$
with notations as in Definition \ref{kappa3preparation}: a contradiction with total preparedness and
(i) is proved.

\smallskip

Assume that $E=\mathrm{div}(u_1)$, so we have
$$
\mathrm{Vdir}(x)=<V> \ \mathrm{or} \ \mathrm{Vdir}(x)=<V,U_1>
$$
by (i). Apply now Lemma \ref{kappa3prelim}(1) and note that the form
(\ref{eq8022}) is automatically achieved when $(Z,u_1,u_2,v)$ is totally prepared.
  {As $B(x)=1$, in \eqref{eq802}, there exists $i$, $0\leq i \leq \omega(x)$ with $\Phi_{i+1}\not=0$.}

 If $\mathrm{Vdir}(x)=<V>$,   {as div$(u_2)\not\subset E$, $\lambda_2=0$} in \ref{kappa3prelim}(1), we have
\begin{equation}\label{eq92}
U_1^{-pd_1}F_{p,Z}\in k(x)[U_1,V]_{\epsilon (x)}
\end{equation}
by (\ref{eq8022}); if $\mathrm{Vdir}(x)=<V,U_1>$, we have
$$
U_1^{-pd_1}F_{p,Z}\in k(x)[U_1,V]_{\epsilon (x)}\oplus <U_1^{\omega (x)}U_2>
$$
by (\ref{eq8022}). Therefore (ii) follows from Lemma \ref{kappa3prelim}(1) and Proposition \ref{Eprimestable}.

\smallskip

To prove (iii), it can be assumed that $\mathrm{Vdir}(x)=<V>$ by (i) and Corollary \ref{omega(x)=epsilon(x)<p}.
In particular, we have
$$
\mathrm{in}_{m_S} h =Z^p +F_{p,Z} , \ U_1^{-pd_1}U_2^{-pd_2}F_{p,Z}=\lambda V^{1+\omega (x)}+Q(U_1,U_2),
$$
with $\lambda \neq 0$, $Q \neq 0$, and $Q \in k(x)[U_1]$ if $E=\mathrm{div}(u_1)$. We blow up along $x$
and let $x':=(Z/u_2,u_1/u_2,u_2,v/u_2)$.

\smallskip

Assume that $E=\mathrm{div}(u_1)$. By (ii) and (\ref{eq92}), the only
point to consider is $x'$. By Corollary \ref{omega(x)=epsilon(x)<p},
we are done unless $\iota (x')=\iota (x)$, so $x'$ satisfies again  {the assumption in (iii)}
of the proposition with $E'=\mathrm{div}(u'_1u'_2)$. Note that we have $A_1(x')>0$ by (**).

\smallskip

Assume that $E=\mathrm{div}(u_1u_2)$ and let $x_1\in \pi^{-1}(x)$ with $\iota (x_1)\geq \iota (x)$.
By Corollary \ref{omega(x)=epsilon(x)<p}, we are done unless $\iota (x_1)=\iota (x)$. If $E'=\mathrm{div}(u'_1)$,
we have $B(x')=1$ except possibly if
$$
a(1):=pd_1, \ a(2):=pd_2, \ F_0:=Q(U_1,U_2)
$$
satisfies the assumptions of Lemma \ref{joyeux}(ii). This holds only if
$$
d'_1:=d_1 + d_2 + {1+\omega (x) \over p}\in \N.
$$
Then $x_1$ is resolved for $m(x)=p$ by blowing up $d'_1$ times along
codimension two centers of the form $(Z',u'_1)$. Otherwise, we have $<Q>=<U_1^{1 + \omega (x)}>$,
$x_1=x'$ up to  renumbering $u_1,u_2$, so $B(x')=1$ and $x'$ satisfies again  {the assumption in (iii)} of  the
proposition. Note that no renumbering is necessary if $A_1(x)>0$.

\smallskip

Summing up,  $x$ is resolved or we construct a sequence of infinitely near points lying on the
successive strict transforms of a formal curve
$$
\hat{{\cal Y}}=V(\hat{Z},u_1, \hat{v})\subset \hat{{\cal X}}={\cal X}\times_S\mathrm{Spec}\hat{S}.
$$
By Proposition \ref{permisarc}  we may assume that ${\cal Y}$ is permissible of the first kind,
so $x$ is resolved by blowing up along ${\cal Y}$.
\end{proof}

\begin{rem} {
The case $\kappa(x)=3$ and (**) is not at stable by permissible blowing ups of first or second kind. To avoid this problem, in the following Propositions \ref{redto3**casii}, \ref{**gamma} and \ref{**versgammaegal1}, we make some blowing ups $\pi$ which are Hironaka permissible and  keep the conditions (**) and {\bf (E)'} at
$x'\in \pi^{-1}(x)$ with $\iota (x')\geq (p,\omega (x),3)$. The transformation laws on $\Delta_2$ are the usual ones up to an horizontal  translation by $1$  in the case where $D:=V(z,u_1,u_2)$.}

\end{rem}

\begin{prop}\label{redto3**casii}
Assume that $x$ satisfies conditions (**)  and {\bf (E)'} with $\kappa(x)=4$ and let $(Z,u_1,u_2,v)$ be  totally prepared.
Let us call ${\cal Y}:=V(Z,u_1,v)$ with generic point $y$.
\begin{itemize}
  \item [(1)] if $\omega(x)<p$, $x$ is resolved;
  \item [(2)] if $\omega(x)\geq p$ and $\epsilon(y)\geq 2$, then  $(d_1,d_2,{1+\omega(x)\over p})$
is the only vertex of $\Delta_S(h;u_1,u_2,v;Z)$ in the region $x_1=d_1$. Furthermore ${\cal Y}$ is Hironaka-permissible
and $x$ is resolved.
  \item [(3)] if $\omega(x)\geq p$ and $E=\div(u_1)$, let $\pi : {\cal X}'\rightarrow {\cal X}$ be the
  blowing up along  $x$ and $x'\in \pi^{-1}(x)$ with $\iota (x')\geq (p,\omega (x),3)$. Then $x'$ is
  resolved or there is a Hironaka-permissible line
  $$
  D'=V(Z',u'_1,u'_2), \  E'=\mathrm{div}(u'_1u'_2).
  $$

  Let $\pi ': {\cal X}''\rightarrow {\cal X}'$ be the blowing up along $D'$ and $x'' \in {\pi '}^{-1}(x')$
  with $\omega (x'')\geq \omega (x')$. Then:

(i) $x''$ satisfies again {\bf (E)'} and $\omega(x'')= \omega(x)$;

(ii) $x''$ satisfies condition (**), $E''=\mathrm{div}(u''_1u''_2)$  and $\kappa(x'')=3$;

(iii)  $C(x'')<1-{1 \over 1+\omega(x)}$, $A_1(x'')<1,\ A_2(x'')<1$.
\end{itemize}
\end{prop}

\begin{proof}
Statement (1) has been proved in Corollary \ref{omega(x)=epsilon(x)<p}. From now on, we assume that $\omega(x)\geq p$.

\smallskip

Let us prove (3). As $\kappa(x)=4$,  $E=\div(u_1)$,  we have  $\mathrm{Vdir}(x)=<U_1>$. By (**):
$$
f_{p,Z}=u_1^{-pd_1}(\gamma v^{1+\omega(x)} + \gamma' u_1^{\omega(x)}+ u_1\phi),\ \gamma,\gamma'\ \mathrm{invertible},\ \phi \in m_S^{\omega(x)}.
$$

We blow up along $x$,  {let $x'$ be a point above $x$}: if  {$\omega(x')=\omega(x)$}, $x'$ is on the strict transform of $\div(u_1)$. In the chart of origin
$(Z',u'_1,u'_2,v):=(Z/v,u_1/v,u_2/v,v)$,  {so called $v$-chart}, we get, before any preparation:
$$
f_{p,Z'}={u'_1}^{pd_1}v^{pd_1+\omega(x)-p}(\gamma v +u'_1\phi'),\ \phi'\in S', \ E'=\mathrm{div}(u'_1v).
$$
As $1+\omega(x)\not=\ 0\ \mod(p)$, the monomial ${u'_1}^{pd_1}v^{pd_1+\omega(x)-p} v$ is not a $p^{th}$-power,
it cannot be spoilt by any translation on $Z'$:  $\omega(x')=1<p\leq \omega(x)$,  {for any $x'$ in this chart on the strict transform of div$(u_1)$}.
The  {last point to look at} is the point  {on the strict transform of div$(u_1)$, not in the $v$-chart}
$$
x'=(Z',u'_1,u'_2,v'):=(Z/u_2,u_1/u_2,u_2,v/u_2), \ E'=\mathrm{div}(u'_1u'_2).
$$
There is an expansion $h'={Z'}^p +\sum_{i=1}^p f_{i,Z'}{Z '}^{p-i}$, with
\begin{equation}\label{eq93}
f_{p,Z'}={u'_1}^{pd_1}{u'_2}^{pd_1+\omega(x)-p}(\gamma {v'}^{1+\omega(x)}u'_2 +\gamma' {u'_1}^{\omega(x)}+u'_1u'_2\psi'),\ \psi'\in S'.
\end{equation}
As we are at the origin of a chart, $(Z',u'_1,u'_2,v)$ are well adapted: $\epsilon(x')\leq \omega(x)$.
As $\omega(x)\geq p$, we keep condition {\bf (E)'} at $x'$ (Proposition \ref{Eprimestable}). We are done unless
$$
\iota (x')=\iota (x), \ \mathrm{ord}_{x'}(u'_1u'_2\psi')\geq \omega(x).
$$
In particular, we have $\mathrm{in}_{m_{S'}}h'={Z'}^p + F_{p,Z'}$.

\noindent $\bullet$ Case ord$_{x'}(u'_1u'_2\psi')= \omega(x)$. Since $\kappa (x')=4$, we have
$$
\mathrm{Vdir}(x')\subseteq <U'_1,U'_2>.
$$
By (\ref{eq93}), we have $<U'_1>\subsetneq \mathrm{Vdir}(x')$, so  $\mathrm{Vdir}(x')= <U'_1,U'_2>$.

Then we blow up along $x'$, the only possible $\omega$-near point (i.e. $m(x')=m(x")=p$, $\omega(x')=\omega(x")$) \index{near @ $\omega$-near} is
$$
x''=(Z'',u''_1,u''_2,v''):=(Z'/v',u'_1/v',u'_2/v',v'), \ E''=\mathrm{div}(u''_1u''_2v'').
$$
There is an expansion
$$f_{p,Z''}={u''_1}^{pd_1}{u''_2}^{pd_1+\omega(x)-p}{v''}^{2(pd_1+\omega(x)-p)}(\gamma {v''}^{2}u''_2 +\gamma' {u''_1}^{\omega(x)}+u''_1u''_2 \psi''),\ \psi''\in S'' $$
and we get $\omega(x'')\leq 3$: we are done for $\omega (x)\geq 4$.

When $\omega(x)=3$, in $J(F_{p,Z''},E'',m_{S''})$, there is an homogeneous polynomial
$$
P:={V''}^{2}U''_2 +U''_1U''_2 (\lambda U''_1+\mu U''_2+\nu V'')+ \delta {U''_1}^3, \ \lambda,\mu,\nu,\delta \in k(x)=k(x'').
$$
Applying the Hasse-Schmidt derivation $ {1\over 2}\times {\partial^2 P\over \partial {V''}^2}=U''_2$ gives  $U''_2\in \mathrm{Vdir}(x'') $.
The reader ends the computation and sees that $\tau'(x'')=3$: $x''$ is is resolved.

When $\omega(x)=2$,  $\psi''$ is invertible, we have $\mathrm{Vdir}(x'')=<U''_1, U''_2>$.
We blow up along $x''$, at the only possible $\omega$-near points, we have, with suitable variables:
$$
f_{p,Z'''}={u'''_1}^{pd_1}{u'''_2}^{pd_1+\omega(x)-p}{v'''}^{3(pd_1+\omega(x)-p)}(\gamma {v'''}u'''_2 +\gamma' {u'''_1}^{\omega(x)}+u'''_1u'''_2 \psi'''). $$
A quick computation shows that $\tau'(x''')=3$, so $x'''$ is resolved.

\smallskip

\noindent $\bullet$ Case    ord$_{x'}(u'_1u'_2\psi')> \omega(x)$. We get  $\mathrm{Vdir}(x')=<U'_1>$.
We may decompose in (\ref{eq93}):
$$
\psi'=\psi'_1+v\psi'_2,\ \psi'_1\in  (u'_1,u'_2)^{\omega(x)-1}, \ \psi'_2 \in  (u'_1,u'_2).
$$
By condition {\bf (E)'}, the line $D':=\mathrm{V}(Z',u'_1,u'_2)$ with generic point $y'$
is Hironaka-permissible, $\epsilon(y')=1$.

\smallskip

Let us blow up along $D'$. Let us begin with the point $x'_2$ at infinity, i.e.
$$
x'_2:=(Z'',u''_1,u''_2,v'')=(Z'/u'_1,u'_1,u'_2/u'_1,v'), \ E''=\mathrm{div}(u''_1u''_2).
$$
We get $H(x'_2)=({u''_1}^{2pd_1+\omega(x)+1-2p}{u''_2}^{pd_1+\omega(x)-p})$ and
$$
H(x'_2)^{-1}f''_{p,Z''}=\gamma{v''}^{1+\omega(x)}u''_2+ \gamma' {u''_1}^{\omega(x)-1}+ {u''_1}^{\omega(x)} \psi'_1+ {u''_2}u''_1v''\psi'_2.
$$
As we are at the origin of a chart, the coordinates $(Z'',u''_1,u''_2,v'')$ are well adapted,
so $\epsilon(x'_2)\leq \omega(x)-1$.

\smallskip

For $x_2\in {\pi '}^{-1}(x')$ in the chart of origin
$$
x'':=(Z'',u''_1,u''_2,v''):=(Z'/u'_2,u'_1/u'_2,u'_2,v'), \ E''=\mathrm{div}(u''_1u''_2),
$$
we get $H(x_2)=({u''_1}^{{pd_1}}{u''_2}^{2pd_1+\omega(x)+1-2p})$ (in particular {\bf (E)'} holds) and
\begin{equation}\label{eq:x_2}
H(x_2)^{-1} {f_{p,Z''}}=\gamma{v''}^{1+\omega(x)}+{u''_2}^{\omega(x)-1} \gamma' {u''_1}^{\omega(x)}+ {u''_2}^{\omega(x)} u''_1\psi'_1+ {u''_2}u''_1v''\psi'_2.
\end{equation}
As $1+\omega(x)\not=0\ \mod(p)$, the monomial $H(x_2){v''}^{1+\omega(x)}$ cannot be spoilt by any translation on $Z''$:
we have $(m(x_2),\omega(x_2))\leq (p,\omega(x))$. Because of the monomial $H(x_2){u''_2}^{\omega(x)-1} {u''_1}^{\omega(x)}$,
we must have $u''_1(x_2)=0$: therefore $x_2=x''$ is the origin of the chart. We have
$$
\min\{\mathrm{ord}_{m_{S''}}({u''_2}^{\omega(x)} u''_1\psi'_1), \mathrm{ord}_{m_{S''}}({u''_2}u''_1v''\psi'_2)\}\geq \omega(x)+1
$$
if $(m(x''),\omega(x''))= (p,\omega(x))$: $x''$ is in case (**) with $\kappa (x'')=3$. This proves (i) and (ii).

\smallskip

Let us prove assertion (iii) which is valid only for the point $x''$ of parameters
$$
(Z'',u''_1,u''_2,v''):=(Z'/v',u'_1/u'_2,u'_2,v')=(Z/u_2^2,u_1/u_2^2,u_2,v/u_2).
$$
In the expansion of $f_{p,Z}$, the monomial $(u_1^{pd_1})\times u_1^au_2^b v^c=H(x)u_1^au_2^b v^c$ becomes
$${u''_2}^{2p}{u''_1}^{pd_1}{u''_2}^{2pd_1+\omega(x)+1-2p}\times {u''_1}^a{u''_2}^{2a+b+c-(\omega(x)+1)}{v''}^{c}.$$
As $ {f_{p,Z''}}={u''_2}^{-2p}f_{p,Z}$, to the monomial $H(x)u_1^au_2^b v^c$ corresponds the monomial $H(x'') {u''_1}^a{u''_2}^{2a+b+c-(\omega(x)+1)}{v''}^{c}$ in the expansion of $ {f_{p,Z''}}$. The point
$$
({a\over 1+\omega(x)-c},{b\over 1+\omega(x)-c})\in \mathrm{pr}(\Delta(h;u_1,u_2,v;Z))
$$
gives the point $({a\over 1+\omega(x)-c},{2a+b\over 1+\omega(x)-c}-1)$ of pr$(\Delta(h;u''_1,u''_2,v'';Z''))$.
For example, the monomial $H(x)\gamma' u_1^{\omega(x)}$ becomes
$$H(x''){u''_1}^{\omega(x)}{u''_2}^{\omega(x)-1}.$$

Choose $(a_0,b_0,c_0)$ such that $({a_0\over 1+\omega(x)-c_0},{b_0\over 1+\omega(x)-c_0}-1)$ is a vertex of of pr$(\Delta(h;u_1,u_2,v;Z))$ with ${2a_0+b_0\over 1+\omega(x)-c_0}$ minimal. Then, because of the monomial $H(x)\gamma'  {u_1}^{\omega(x)}$,
\begin{equation}\label{eq931}
{2a_0+b_0\over 1+\omega(x)-c_0}-1\leq {2\omega(x) \over \omega(x)+1}-1=1-{2 \over \omega(x)+1} ,
\end{equation}
in particular
\begin{equation}\label{eq931bis}
{a_0\over 1+\omega(x)-c_0}\leq {a_0+b_0/2\over 1+\omega(x)-c_0}\leq {\omega(x) \over \omega(x)+1}<1,
\end{equation}
so the point
$$
({a_0\over 1+\omega(x)-c_0},{2a_0+b_0\over 1+\omega(x)-c_0}-1)
$$
has both coordinates $<1$, it is the vertex of $\Delta_2(h'';u''_1,u''_2;v'';Z'')$ of smallest ordinate.

Let us note that if $(a_0,b_0,c_0)\not=(\omega(x),0,0)$, then, as  {Vdir$(x)=<U_1>$}, we have $a_0+b_0\geq 1+\omega(x)-c_0$, so ${2a_0+b_0\over 1+\omega(x)-c_0}-1\geq {a_0\over 1+\omega(x)-c_0}>0$,   {\eqref{eq:x_2} gives the last inequality}. When $(a_0,b_0,c_0)=(\omega(x),0,0)$, we get
$$
{2a_0+b_0\over 1+\omega(x)-c_0}-1={2\omega(x)\over 1+\omega(x)}-1={\omega(x)-1\over 1+\omega(x)}\geq {p-1\over 1+\omega(x)}>0,
$$
\begin{equation}\label{eq932}
 {f_{p,Z''}}= H(x'')( \gamma {v''}^{1+\omega(x)}+u''_1u''_2 \vartheta),\ \vartheta \in S''.
\end{equation}

As we saw above,  $\epsilon(x'')=\omega(x'')+1$, $\kappa(x'')=3$ and we have (**).
Then $({a_0\over 1+\omega(x)-c_0},{2a_0+b_0\over 1+\omega(x)-c_0}-1)$ is  the vertex of $\Delta_2(h'';u''_1,u''_2;v'';Z'')$
of smallest ordinate, both coordinates are $<1$ and positive. As $x'$ and $x''$ are origins of chart,
$(Z'', u''_1,u''_2,v'')$ are well prepared and no translation on $v''$ can spoil this vertex. By (\ref{eq931})(\ref{eq931bis}), we get:
$$C(x'')\leq {a_0\over 1+\omega(x)-c_0}-A_1(x'')<1-{1 \over 1+\omega(x)} ,$$
$$0<A_2(x'')={2a_0+b_0\over 1+\omega(x)-c_0}-1<1,$$
$$
A_1(x'')\leq {a_0\over 1+\omega(x)-c_0}\leq {2a_0+b_0\over 1+\omega(x)-c_0}-1<1.
$$
Note that $A_1(x'')>0$ because of (\ref{eq932}). This proves (iii).

Let us prove (2). Since $\epsilon (y)>0$, we have $A_1(x)>0$ and $(d_1,d_2,{1+\omega(x)\over p})$
is the only vertex of $\Delta_S(h;u_1,u_2,v;Z)$ in the region $x_1=d_1$, $U_1\in \mathrm{Vdir}(x)$.
If $\mathrm{Vdir}(x)=<U_1,U_2>$, then, if we blow up along $x$, as $\omega(x)\geq p\geq 2$, there is no $\omega$-near point.
The only case we have to look at is $\mathrm{Vdir}(x)=<U_1>$.

As $\omega(x)\geq p$, by condition  {\bf (E)'} at $x$: $pd_1\geq p$, ${\cal Y}$ is Hironaka-permissible.
Let us denote by $d:=\epsilon(y)\geq 2$. Then $\gamma v^{1+\omega(x)} + g\in (v,u_1)^d$
with $g=\gamma' u_1^{\omega(x)}+ u_1\phi$, $\phi \in m_S^{\omega(x)}\cap (v,u_1)^{d-1}$, $\gamma'$ invertible.
Up to change $\gamma'$ modulo $m$, there is a decomposition:
$\phi=v\phi_1+u_2\phi_2$, $\phi_1\in (u_1,v)^{\omega(x)-1}$, $\phi_2\in (u_1,v)^{d-1}$.
$$f_{p,Z}={u_1}^{pd_1}{u_2}^{pd_2}(\gamma {v}^{1+\omega(x)}+\gamma' u_1^{\omega(x)}+u_1v\phi_1+u_1u_2\phi_2).$$
Let us blow up along ${\cal Y}$.
In the first chart of origin
$$
(Z',u'_1,u'_2,v'):=(Z/u_1,u_1,u_2,v/u_1),
$$
we get
$$  f_{p,Z'}={u'_1}^{pd_1+d-p}{u'_2}^{pd_2}(\gamma {v'}^{1+\omega(x)}{u'_1}^{\omega(x)+1-d}+\gamma' {u'_1}^{\omega(x)-d}+{u'_1}^{\omega(x)-d+1}\phi'_1+u'_2\phi'_2),$$
$\phi'_1,\phi'_2 \in S'$. Because of the monomial
$${u'_1}^{pd_1+d-p}{u'_2}^{pd_2}\gamma' {u'_1}^{\omega(x)-d}=H(x')\gamma' {u'_1}^{\omega(x)-d},$$
we get $\omega(x_1)\leq  \omega(x)-d < \omega(x)-1$ for any $x_1$ in this chart.

Let us see the point at infinity  $x'=(Z',u'_1,u'_2,v'):=(Z/v,u_1/v,u_2,v)$, we get
$$
f_{p,Z'}={u'_1}^{pd_1}{u'_2}^{pd_2}{v'}^{pd_1+d-p}(\gamma v^{\omega(x)+1-d}+ {u'_1}\phi'), \  {\phi' \in S'}.
$$
As we are at the origin of a chart, $(Z',u'_1,u'_2,v)$ are well adapted:
$\epsilon(x')\leq \omega(x)+1-d\leq \omega(x)-1$.
\end{proof}

\begin{prop}\label{betapetitto3*}
Assume that $x$ satisfies conditions (**) and {\bf (E)'}
with $\kappa(x)=4$, $E=\div(u_1)$ and let $(Z,u_1,u_2,v)$ be  totally prepared.
With the notations of Proposition \ref{redto3**casii}, assume furthermore that
$$
\epsilon(y)= 1\ \mathrm{ and}\  \beta(Z,u_1,u_2,v) <  1.
$$
Then  $x$ is resolved.

\end{prop}

\begin{proof}
By Proposition \ref{redto3**casii}(1), we may assume $\omega(x)\geq p$. As $A_1(x)>0$ by condition (**),
$\epsilon(y)= 1$ implies that $\Delta_S(h;u_1,u_2,v)$ has a vertex
$$
\mathbf{x}=({d_1 +1\over p}, {b \over p}, 0), \ b\in \N.
$$
This leads to
$$
A_1(x)= {1 \over 1+\omega(x)}, \ \beta (x)={ b \over 1+\omega(x)}.
$$
On the other hand, since $\kappa (x)=4$, we have $b\geq \omega (x)$, i.e. $b=\omega (x)$.

Let us come back to the proof of Proposition \ref{redto3**casii}(2). The only point to consider is the
point $x'$ at infinity, $E'=\mathrm{div}(u'_1v)$. We get an expansion
\begin{equation}\label{eq941}
f_{p,Z'}={u'_1}^{pd_1}{v'}^{pd_1+1-p}(\gamma v^{\omega(x)}+ {u'_1}\phi'), \ (\phi ')\equiv (u_2^{\omega (x)}) \ \mathrm{mod}(v').
\end{equation}
The conclusion follows from Lemma \ref{sortiemonome} applied to the well prepared coordinates
$(v',u'_1,u_2;Z')$.
\end{proof}

The following proposition produces bounds identical to those occurring for
embedded resolution of surfaces \cite{Co3}.

\begin{prop}\label{**gamma}
Assume  that $x$ satisfies conditions (**) and {\bf (E)'} with $\kappa(x)\geq 3$.
Consider Hironaka-permissible blowing ups $\pi : \ {\cal X}'\rightarrow ({\cal X},x)$
of the following kinds:

\smallskip

Case 1: $E=\div(u_1u_2)$ and $\omega(x)\geq p$; we blow-up along $D:=$V$(Z,u_1,u_2)$.

\smallskip

Case 2: $\kappa(x)=3$,  {$E=\div(u_1)$ or $\omega(x)< p$} ; we blow up along $x$.

\smallskip

\noindent Let $x'\in \pi^{-1}(x)$ with $(m(x'),\omega (x'))\geq (p,\omega (x))$. Then $\omega(x')\leq \omega(x)$ and
($x'$ is resolved or the following holds):

\smallskip

\noindent  (i) conditions (**) and  \textbf{(E)'} are satisfied at $x'$ and we have
$$
\gamma (x')\leq \max\{\gamma (x),1\}.
$$

\smallskip

\noindent (ii) if $E=\div(u_1u_2)$ and  $\eta '(x')\in \mathrm{Spec}S[u'_2]$
(resp. $\eta '(x')\in \mathrm{Spec}S[u'_2,v']$), where
$$
(u_1,u'_2:= {u_2 \over u_1},v) \ (\mathrm{resp.} \ (u_1,u'_2,v':={v \over u_1}))
$$
in case 1 (resp. case 2), then  $A_1(x')=B(x)$, (resp.  $A_1(x')=B(x)-1$) and,
$$
\beta(x')\leq A_2(x)+C(x)\leq \beta(x);
$$

\noindent if ($k(x')\neq k(x)$ and $\beta(x)\geq 1$), we have $\beta(x')<\beta(x)$;

\smallskip

\noindent if  $u'_2\in m_{S'}$, then $C(x')\leq \min\{C(x), \beta (x)-C(x)\}$, so $C(x')\leq {\beta(x)\over 2}$;

\smallskip

\noindent if  $u'_2\not \in m_{S'}$,  then  $\beta(x')< 1+ \lfloor C(x)\rfloor$;

\smallskip

\noindent (iii) if $x'$ is the origin of the second chart, i.e.
$$
x'=(Z':={Z \over u_2} ,u'_1= {u_1 \over u_2},u_2,v)  \ (\mathrm{resp.} \ (Z', u'_1,u_2,{v\over u_2}))
$$
in case 1 (resp. case 2), then  $A_1(x)=A_1(x')$, $C(x')\leq {\beta(x)\over 2}$  and
$$
\beta(x')=A_1(x) +\beta(x) \ (\mathrm{resp.} \ \beta(x')=A_1(x) +\beta(x)-1);
$$

\smallskip

\noindent (iv) if $E=\div(u_1)$, $E'=\div(u'_1)$ and $\beta(x) >0$, then
$\beta(x')\leq \beta(x)$, with strict inequality if ($k(x')\neq k(x)$ and $\beta (x)\geq 1$).

\end{prop}

\begin{proof}

We first prove the proposition in case 1.  Let $x'$ be in the  chart with origin
$(X':={Z \over u_1} ,u_1,u'_2,v)$.
In the expansion of $f_{p,Z}$ the monomial $u_1^{pd_1}u_2^{pd_2}v^{1+\omega(x)-i}u_1^au_2^b$ transforms into  $u_1^{pd_1+pd_2-p}{u_2'}^{pd_2}v^{1+\omega(x)-i}u_1^{a+b}{u_2'}^b$ in the expansion of $f_{p,Z'}$,
$0\leq i \leq 1+\omega(x)-i$.  This leads to:
$$
f_{p,Z'}=u_1^{pd'_1}{u_2'}^{pd_2}(\gamma v^{1+\omega(x)}+u_1 \phi), \ d'_1:=d_1+d_2-1.
$$
As $1+\omega(x)\not\equiv 0\ \mod \ p$, the monomial $u_1^{pd'_1}{u_2'}^{pd_2}\gamma v^{1+\omega(x)}$ will not be spoilt by any translation on $Z'$: $x'$ satisfies (**) and $(m(x'),\omega (x'))\leq (p,\omega (x))$. If $\omega (x)\geq p$, we
have  $d_1,d_2\geq 1$, so $x'$ satisfies condition \textbf{(E)'}. Statement $\gamma (x')\leq \gamma (x)$ follows from
(ii)  {that we prove in the next lines}.

\smallskip

The monomials defining $B(x)$ in the expansion of $f_{p,Z}$ are minimal for the monomial valuation $v_\alpha$ defined by
the weight vector $\alpha :=(a,a,aB(x))$:
$$v_\alpha(Z)=1,\ v_\alpha(u_1)=v_\alpha(u_2)=a,\ v_\alpha(v)=a B(x),
$$
with
$$a:={p \over pd_1+pd_2+B(x)(1+\omega(x))}. $$
Let us denote by
$$
\clin_{v_\alpha}h=Z^p-G_\alpha^{p-1}Z+F_{p,Z,\alpha}\in \mathrm{gr}_{\alpha} S[Z]$$
At $x'$, there exists $P(t)\in S[t]$, unitary of degree $d:=[k(x'):k(x)]$, whose reduction modulo $m_S$ is irreducible
and $w:=P(u'_2)$ is such that $(X',u_1,w,v)$ is a system of coordinates  at $x'$.

Of course, we take $w=u'_2$ when $x'$ is the origin of the chart. In this special case where
$x'$ is the origin,  {the argument is the same as in \cite{CoP2} Lemma {\bf I.5.3} on page 1966. This relies on the characteristic free Proposition \ref{originchart} which asserts that no changes in $Z'$ need to be performed in order to get well adapted data:  $(X',u_1,w,v)$ is totally prepared.
It is easy to see that $\Delta_2(x')$ is obtained from $\Delta_2(x)$ by applying the affine transformation: $(v_1,v_2)\mapsto (v_1,v_1+v_2)$  when $D$ is the center of the blowing up (resp. $(v_1,v_2)\mapsto (v_1,v_1+v_2-1)$ when $x$ is the center of the blowing up) and adding quadrants.
The reader verifies that all the statements of
(ii) are true, despite of the fact that in case where $D$ is the center of the blowing up, there is not the usual shift by $-1$. }

\smallskip

From now on, $E'=\mathrm{div}(u_1)$. Monomials defining $B(x)$ become the monomials defining $A_1(x')=B(x)$.
The monomials defining the vertices of smaller abscissa of $\Delta_2(h';u'_1,w, {v};X')$ are those
minimal for the valuation $v_{\alpha '}$ given by
$$      v_{\alpha '}(X')=1,\ v_{\alpha '}(u_1)=a,\ v_{\alpha '}(w)=0, \ v_{\alpha '}(v)=a B(x).
$$

 {We set}
$$
\clin_{v_{\alpha '}}h={X'}^p-{G}_{\alpha '}^{p-1}{X'}+F_{p,X',{\alpha '}}
\in \mathrm{gr}_{\alpha '}S=k(x)[\overline{u}'_2]_{(\overline{w})}[U_1,V,X'].
$$
When $G_{\alpha '}\not=0$, we have $A_1(x')=B(x)$, $\beta(x')=0$, so (ii) holds.
Assume now that $G_{\alpha '}\neq 0$.

\smallskip

\noindent {\it Subcase 1.1:} when
$$
U_1^{-pd_1}U_2^{-pd_2}{\partial F_{p,Z,\alpha}\over \partial V}\not \in <V^{\omega(x)}>.
$$
We expand
\begin{equation}\label{eq950}
    U_1^{-pd_1}U_2^{-pd_2}{\partial F_{p,Z,\alpha}\over \partial V}=
\lambda V^{\omega(x)}+\sum_{1\leq i \leq \omega(x)} V^{\omega(x)-i}U_1^{a_1(i)}U_2^{a_2(i)}Q_i(U_1,U_2),
\end{equation}
with $\lambda \neq 0$, $Q_i=0$ or $Q_i$ divisible neither by $U_1$, nor by $U_2$. For $Q_i\not=0$:
$$
a_j(i)\geq iA_j(x), \  \mathrm{deg}(Q_i)\leq iC(x).
$$

By Proposition  {\ref{bupformula}(v), $H(x')^{-1}{\partial F_{p,X',\alpha '}\over \partial V}$ is
the strict transform of $H(x)^{-1}{\partial F_{p,Z,\alpha}\over \partial V}$ (with some abuse of notation for $H(x),H(x')$).}
Then, by \cite{Co3} Lemma~6.2.3~a and page 92, the lowest abscissa of the vertices of the polygon
$$
  {\Delta(H(x')^{-1}{\partial F_{p,X',\alpha '}\over \partial V};U_1,\overline{w};V)}
$$
is $B(x)$. The non compact face of lowest abscissa is not solvable and, after a possible  translation:
$$
Z'= X'+\phi ', \ \phi '\in {U'_1}^{\lceil B(x)\rceil} k(x)[\overline{u}'_2]_{(\overline{w})}[V],
$$
the ordinate $\beta'$ of the vertex of lowest abscissa of
$$
  {\Delta(H(x')^{-1}{\partial F_{p,Z',\alpha '}\over \partial V};U_1,\overline{w};V)}
$$
satisfies
$$
\beta '<1 + \lfloor {C(x)\over d} \rfloor, \ \beta'\leq \beta_2(x),
$$
where $\beta_2(x)$ is the ordinate of the left vertex of the initial face  of the polygon
$  {\Delta(H(x')^{-1}{\partial F_{p,Z,\alpha}\over \partial V};U_1,U_2;V)}$. Then we have
\begin{equation}\label{eq951}
\beta(x')\leq \beta' <1+\lfloor {C(x)\over d} \rfloor ,\ \beta(x')\leq \beta' \leq \beta_2(x)\leq \beta(x).
\end{equation}
This implies all the assertions in subcase~1-1.

\smallskip

\noindent {\it Subcase 1.2:} when
$$
U_1^{-pd_1}U_2^{-pd_2}{\partial F_{p,Z,\alpha}\over \partial V}\in <V^{\omega(x)}>.
$$
We now have an expansion
\begin{equation}\label{eq952}
    U_1^{-pd_1}U_2^{-pd_2}F_{p,Z,\alpha}=
\lambda V^{1+\omega(x)}+\sum_{i=0}^{\lfloor {1+\omega (x)\over p}\rfloor} V^{pi}U_1^{a_1(i)}U_2^{a_2(i)}Q_i(U_1,U_2),
\end{equation}
with $\lambda \neq 0$, $Q_i=0$ or $Q_i$ divisible neither by $U_1$, nor by $U_2$. For $Q_i\not=0$:
$$
a_j(i)\geq (1+\omega (x)-pi)A_j(x), \  \mathrm{deg}(Q_i)\leq (1+\omega (x)-pi)C(x).
$$
Take $i_0$, $1\leq i_0 < (1+\omega(x))/p$ maximal such that $U_1^{pd_1+a_1(i_0)}U_2^{pd_2+a_2(i_0)}Q_{i_0}$
is not a $p^{th}$-power. This $i_0$ exists by total preparation.
By (\ref{eq952}), the transform of ${\partial F_{p,Z,\alpha}\over \partial V}$ now reads
\begin{equation}\label{eq953}
{U_1}^{-pd'_1}{\partial F_{p,X',\alpha '}\over \partial V}
= \lambda '{V}^{\omega(x)}, \ \lambda '\ \mathrm{a} \ \mathrm{unit}.
\end{equation}
Preparation along the face of abscissa $B(x)$ will thus be a translation $Z'=X'+\phi '$ on $X'$,
no translation on $v$: this will just add a $p^{th}$-power to the term
${U_1}^{pd'_1+(1+\omega (x)-pi_0)B(x)}{\overline{u}'_2}^{pd_2}Q_{i_0}(1,\overline{u}'_2)$
in (\ref{eq952}), which will become of the form
$$
\overline{\gamma} '{U_1}^{pd'_1+(1+\omega (x)-pi_0)B(x)}\overline{w}^c,
\ \overline{\gamma} ' \in k(x)[\overline{u}'_2]_{(\overline{w})}, \ \overline{\gamma} ' \ \mathrm{invertible}.
$$

By the usual computations (\cite{Co3}  page 92 or the blowing up formula applied to
$ U_1^{pd_1+a_1(i_0)}U_2^{pd_2+a_2(i_0)}Q_{i_0}(U_1,U_2)$), we have
\begin{equation}\label{eq954}
c\leq 1+{\mathrm{deg}(Q_{i_0})\over d};\ \hbox{when }d_2=0,\ c\leq a_2(i_0)+\mathrm{deg}(Q_{i_0})\leq \beta_2(x)\leq \beta(x).
\end{equation}
This implies all the assertions in subcase~1-2, $x'$ not the origin and (ii) is proved.
Permuting $u_1$ and $u_2$ gives (iii).

\smallskip

\noindent We now turn to case 2. Let $x'$ be in the chart of origin $(X':={Z \over u_1} ,u_1,u'_2,v')$.
By Proposition \ref{kappa3instable}(ii), we may assume that $B(x)>1$, i.e. $<V>= \Vdir(x)$, so $v'\in m_{S'}$.
In the expansion of $f_{p,Z}$ the monomial
$$
u_1^{pd_1}u_2^{pd_2}v^{1+\omega(x)-i}u_1^au_2^b, \ 0\leq i \leq 1+\omega(x)-i
$$
becomes  $u_1^{pd_1+pd_2+1+\omega(x)-p}{u_2'}^{pd_2}v^{1+\omega(x)-i}u_1^{a+b}{u_2'}^b$ in the expansion of $f_{p,Z'}$.  This leads to:
$$
f_{p,Z'}=u_1^{pd'_1}{u_2'}^{pd_2}(\gamma v^{1+\omega(x)}+u_1 \phi), \ d'_1:=d_1+d_2+{1+\omega (x) \over p} -1.
$$
Then $x'$ is resolved or $x'$ satisfies conditions (**) and  \textbf{(E)'} as in case 1.
Then the proof runs along the same lines as above: equations (\ref{eq951}) and (\ref{eq954}) remain true.

\smallskip
The case where $x'$ is the origin of the second chart is given by a permutation of $u_1$ and $u_2$
in the computations above and the fact that the vertices of
$\Delta_2(h';u_1/u_2,u_2;v/u_2;Z/u_2)$ are the transforms of the \textbf{left vertices} of
$\Delta_2(h;u_1,u_2;v;Z)$ by the affinity $(x_1,x_2)\mapsto (x_1,x_1+x_2-1)$: they are totally prepared.
\end{proof}

\begin{prop}\label{**versgammaegal1}
Assume that $x$ satisfies conditions (**) and {\bf (E)'}.  Let $\mu$ be a valuation of $L=k({\cal X})$ centered at $x$.
There exists a finite and independent composition of local Hironaka-permissible blowing ups w.r.t. $E$:
\begin{equation}\label{eq970}
    ({\cal X},x)=:({\cal X}_0,x_0) \leftarrow ({\cal X}_1,x_1) \leftarrow \cdots \leftarrow ({\cal X}_r,x_r) ,
\end{equation}
where $x_i \in {\cal X}_i$ is the center of $\mu$, such that $x_r$ is resolved or ($x_r$ satisfies
again conditions (**) and {\bf (E)'} together with one of the following):
\begin{itemize}
  \item [(i)] $E_r=\div(u_{1,r})$, $\beta(x_r)<1$;
  \item [(ii)] $E_r=\div(u_{1,r}u_{2,r})$,  $C(x_r)=0$.
\end{itemize}
\end{prop}

\begin{proof}
Let $(Z,u_1,u_2,v)$ be  totally prepared. Let ${\cal Y}=V(Z,u_1,v)$ with generic point $y$.
We define by induction on $i\geq 0$ a sequence of local Hironaka-permissible blowing ups w.r.t. $E$, or
composition of two such local blowing ups. Take $i=0$ w.l.o.g. in the following  definition.

\smallskip

\noindent (1) if  ($E=\div(u_1)$, $\kappa(x)=3$), blow up along $x$ (Proposition \ref{**gamma}, case 2);

\smallskip

\noindent (2) if  ($E=\div(u_1)$, $\kappa(x)=4$, $\epsilon(y)\leq 1$), blow up along $x$, then along
$D'=V(Z',u'_1,u'_2)$ (notations of Proposition \ref{redto3**casii}(3));

\smallskip

\noindent (3) if  ($E=\div(u_1)$, $\kappa(x)=4$, $\epsilon(y)\geq 2$), blow up along ${\cal Y}$
(Proposition \ref{redto3**casii}(2));

\smallskip

\noindent (4) if  ($E=\div(u_1u_2)$, $\omega(x)\geq p$), blow up along $D=$ {V}$(Z,u_1,u_2)$
(Proposition \ref{**gamma}, case 1);

\smallskip

\noindent (5) if  ($E=\div(u_1u_2)$, $\omega(x)< p$), blow up along $x$ (Proposition \ref{**gamma}, case 2).

\smallskip

We must prove that (A) this algorithm is well defined, i.e. $x_1$ is resolved or satisfies again conditions
(**) and {\bf (E)'}, so it builds up a sequence (\ref{eq970}), then (B) this sequence is finite.

\smallskip

Note that any $x$ fits into  {exactly one} of (1)-(5). To prove (A)(B), we recollect results from the previous propositions.
 {Proposition \ref{redto3**casii}(2) shows that $x$ is resolved when $x$ is in case (3) above.   In case (2), Proposition \ref{redto3**casii}(3)} produces $x_1$ satisfying again the assumptions of  {Proposition~\ref{**versgammaegal1}}
and fitting into (4) with  $\kappa (x_1)=3$, $\gamma (x_1)=1$.

\smallskip

We now turn to Proposition \ref{**gamma}  {in the cases (1)(4)(5)} above. Statement (i) shows that $x_1$ is resolved or satisfies again the
assumptions of the lemma. The proof of  (A) is thus complete and we turn to (B). Assume w.l.o.g. that $x$
neither satisfies (i) nor (ii). In particular $\gamma (x)\geq 1$. We first claim that there exists $r_0\geq 0$
such that $x_{r_0}$ is resolved or
\begin{equation}\label{eq971}
    \gamma (x_r)=1 \ \mathrm{for} \  \mathrm{all}  \ r\geq r_0.
\end{equation}
By Proposition \ref{**gamma}(i), we have $\gamma (x_1)\leq \gamma (x)$; by Proposition \ref{**gamma}(iii),
inequality is strict if:
$$
E=\mathrm{div}(u_1), \ E_1=\mathrm{div}(u_{1,1}u_{2,1})
$$
provided $\gamma (x)\geq 2$, $\beta (x)\neq 2$. In case $\beta (x)=2$, we obtain $C(x_1)\leq 1$.
Then any further occurrence of $E_r=\mathrm{div}(u_{1,r})$ along the algorithm will satisfy
$\beta (x_r)<2$ by Proposition \ref{**gamma}(ii)-(iv). Therefore it can be
assumed that $E$ and $E_i$ have {\it the same number} of irreducible components for every $i\geq 0$
in order to prove (\ref{eq971}) (note that we are done if (2) is applied).

\smallskip

If $E=\mathrm{div}(u_1)$, we reach (i) or $k(x_i)=k(x)$ for $i>>0$ by Proposition \ref{**gamma}(iv).
The claim follows from Corollary \ref{permisarcthree}.

If $E=\mathrm{div}(u_1u_2)$, we get (\ref{eq971}) by standard arguments on combinatorial blowing ups.

\smallskip

To conclude the proof, we may hence assume that ($E=\mathrm{div}(u_1)$, $\beta (x)=1$)
or ($E=\mathrm{div}(u_1u_2)$, $C (x)<1$).

\smallskip

When ($E=\mathrm{div}(u_1)$, $\beta (x)=1$), this is stable by
blowing up or yields $E_1=\mathrm{div}(u_{1,1}u_{2,1})$ (Proposition \ref{redto3**casii}(3)
and Proposition \ref{**gamma}(iii)). Stability ends after finitely many steps
by Proposition \ref{**gamma}(iv) and Corollary \ref{permisarcthree}.

\smallskip

When ($E=\mathrm{div}(u_1u_2)$, $C (x)<1$), this is stable by
blowing up or yields (i) (Proposition \ref{**gamma}(ii)). Stability
ends up in (ii) for $r>>0$  by standard arguments on combinatorial blowing ups. \end{proof}

\begin{prop}\label{**gammaegal1}
Assume that $x$ satisfies conditions (**) and {\bf (E)'} together with one of the following:
\begin{itemize}
  \item [(i)] $E=\div(u_1)$, $\beta(x)<1$;
  \item [(ii)] $E=\div(u_1u_2)$, $A_1(x)<1$, $C(x)<{1 \over 2}$, $\beta(x) <1-{1\over 1+\omega(x)}$;
  \item [(iii)] $E=\div(u_1u_2)$ and $C(x)=0$.
\end{itemize}
Then $x$ is resolved for $(p,\omega(x),3)$.
\end{prop}

\begin{proof}
We assume that $(Z,u_1,u_2,v)$ is totally prepared. Let
$$
c(x):=(A_1(x),\beta(x))
$$
with lexicographical ordering.  First suppose that
\begin{equation}\label{eq980}
A_1(x)<1 \ \mathrm{and} \ (x  \ \hbox{is in case (iii)} \Longrightarrow A_2(x) < 1).
\end{equation}

\noindent If $E=\mathrm{div}(u_1u_2)$ and $\kappa (x)=3$, we blow up along $x$.
Let $x'$ be a point $\omega$ near  {to} $x$.
When $x'$ is the origin of a chart,  by Proposition \ref{**gamma}(i)-(iii), $x'$ satisfies again
the assumptions of  the proposition  with $c(x')<c(x)$.
When $x'$ is in the first chart with $E'=\mathrm{div}(u_1)$, Proposition \ref{**gamma}(ii) gives
$$
A_1(x')=B(x)-1\leq A_1(x)+\beta(x)-1<A_1(x) \ \mathrm{and} \ \beta(x')<1.
$$
In both cases, $x'$ satisfies again the assumptions of the proposition together with (\ref{eq980})
and $c(x')<c(x)$.

\smallskip

\noindent If $E=\mathrm{div}(u_1u_2)$ and $\kappa (x)=4$, we let ${\cal Y}_j:=V(Z,v,u_j)$ with generic point $y_j$, $j=1,2$.
The condition $\epsilon (y_j)\geq 2$ is equivalent to $A_j(x)>{1\over 1+\omega(x)}$.
We apply Proposition \ref{redto3**casii}(1)(2): then $x$ is resolved except possibly if
$A_j(x)\leq {1\over 1+\omega(x)}$, $j=1,2$. Then
$$
1-{1\over 1+\omega(x)} \leq B(x)\leq A_1(x)+\beta (x)<1.
$$
We deduce that equality holds and that $\mathrm{Vdir}(x)=<U_1,U_2>$. Since $\omega (x)\geq p\geq 2$,
we obtain $\omega (x')<\omega (x)$ after blowing up along $x$, so $x$ is resolved.

\smallskip

\noindent If $E=\mathrm{div}(u_1)$ and $\kappa (x)=3$, we blow up along $x$. Note that $\beta (x)>0$
since $A_1(x)<1$. Let
$$
x':=(Z':={Z \over u_2} ,u'_1= {u_1 \over u_2},u_2,v'={v\over u_2}), \ E'=\div(u'_1u_2).
$$
If $x_1\neq x'$, Proposition \ref{**gamma}(iv) gives
$$
A_1(x')=B(x)-1\leq A_1(x)+\beta(x)-1<A_1(x) \ \mathrm{and} \ \beta(x')\leq \beta (x).
$$
Therefore $x_1$ satisfies again assumption (i) of the proposition together with (\ref{eq980})
and $c(x')<c(x)$.

If $x_1=x'$, Proposition \ref{**gamma}(iii) gives
$$
A_1(x')= A_1(x),\ C(x')<{1 \over 2}, \ \beta(x')=\beta(x)+A_1(x)-1<\beta(x).
$$
Therefore $x'$ satisfies again assumption (ii) of the proposition together with (\ref{eq980})
and $c(x')<c(x)$.

\smallskip

\noindent If $E=\mathrm{div}(u_1)$ and $\kappa (x)=4$, $x$ is resolved by Propositions \ref{redto3**casii}(1)(2)
and \ref{betapetitto3*}.

\smallskip

Therefore the proposition holds by induction on $c(x)$ under the extra assumption (\ref{eq980}).

\smallskip

Assume now that $x$ satisfies assumption (i) with $A_1(x)\geq 1$. In particular $\epsilon(x)=1+\omega(x)$
and $V\in \Vdir(x)$ by  Proposition  \ref{kappa3instable}. Furthermore,
\begin{equation}\label{eq981}
d_1+{1+\omega (x)\over p}>1.
\end{equation}
We have $m(y)=m(x)$, $\epsilon(y)=\epsilon(x)$ where ${\cal Y}=V(Z,u_1,v)$ with generic point $y$, so
${\cal Y}$ is permissible of first kind. Let us blow up along ${\cal Y}$.

We are done by Theorem \ref{bupthm} if $\mathrm{Vdir}(x)=<V,U_1>$. Otherwise we have $A_1(x)>1$
or $\beta(x)>0$.  Since $V\in \Vdir(x)$, the only point which may be $\omega$-near $x$
is the point
\begin{equation}\label{eq982}
x':=(Z',u'_1,u'_2,v')=(Z/u_1,u_1,u_2,v/u_1), \ E'=\mathrm{div}(u_1).
\end{equation}
These are well adapted coordinates. If $A_1(x)>1$, we have
$$
\beta(x')=\beta(x),\ A_1(x')=A_1(x)-1>0,\ d'_1=d_1+{1+\omega(x)\over p}-1.
$$
Then $x'$ satisfies again conditions (**) and \textbf{(E)'} by (\ref{eq981}). By induction on $A_1(x)$, we reduce to
$A_1(x)=1$, since $A_1(x)<1$ is (\ref{eq980}).

\smallskip

If $A_1(x)=1$, expand
$$
f_{p,Z}=u_1^{pd_1}(\gamma v^{1+\omega(x)}+\sum_{1\leq i \leq 1+\omega(x)} \gamma_i v^{1+\omega(x)-i}u_1^i u_2^{a_2(i)} +f_1),
$$
with $f_1\in (v,u_1)^{2+\omega(x)}$, $\gamma \in S$ invertible, $\gamma_i\in S$ invertible or zero, $\gamma_{i_0}$ invertible
for some $i_0$ with $a_2(i_0)=i_0\beta(x)<i_0$. We get
$$
f_{p,Z'}={u'_1}^{pd_1+1+\omega(x)-p}(\gamma {v'}^{1+\omega(x)}
+\sum_{1\leq i \leq 1+\omega(x)} \gamma_j {v'}^{1+\omega(x)-i} {u'_2}^{a_2(i)} +u'_1f'_1), \ f'_1\in S'.
$$
Clearly $\iota(x')\leq (p,\omega(x),2)$ and $x$ is resolved for $(p,\omega(x),3)$.

\smallskip

There remains to prove the Proposition in case (iii) with $A_i(x)\geq 1$, $i=1$ or $2$.
See \cite{CoP1}  \textbf{II.6.2} and   \textbf{II.6.3} on pp. 1950-1951. The argument is similar to
the one used in the proof of Proposition \ref{contactmaxpetitgamma}(b)(c).

\smallskip

If ($\omega (x) \geq p$ and $A_1(x)\geq 1$), then ${\cal Y}:= {V}(Z,u_1,v)$ is permissible of the
first kind. Blowing up along ${\cal Y}$, the only point which may be $\omega$-near $x$ is the point $x'$
as in (\ref{eq982}). We have
$$
A_1(x')=A_1(x)-1,\ A_2(x')=A_2(x'),\ C(x')=0, \ d'_1=d_1+{1+\omega(x) \over p}-1\geq 1.
$$
Then $x'$ satisfies again conditions (**) and \textbf{(E)'}.
A descending induction on $\max \{A_1(x),A_2(x)\}$ reduces to $A_1(x),A_2(x)<1$ which is (\ref{eq980})
and the proof is complete.

\smallskip

If  $1+\omega(x)< p$, we argue by induction on
$$
c'(x):=(\max\{A_1(x),A_2(x)\},\max\{d_1,d_2\},n)
$$
where $n:=2$ if $(A_1(x)=A_2(x),d_1=d_2)$, $n:=1$ otherwise.

Suppose that $A_1(x)\geq 1$, $d_1+{1+\omega(x) \over p}\geq 1$. Up to renumbering $u_1,u_2$,
it can be assumed that  {$c'(x)=(A_1(x),d_i,n)$, $i=1,2$ or $c'(x)=(A_2(x),d_1,1)$  with $d_2+{1+\omega(x) \over p}<1$}.
Blowing up along ${\cal Y}:=(Z,u_1,v)$, the only point which may be $\omega$-near $x$ is the point $x'$
as in (\ref{eq982}). If $(m(x'),\omega (x'))=(p,\omega (x))$, $x'$ is in case (**) and we have
$$
A_1(x')=A_1(x)-1,  \ C(x')=0, \ d'_1=d_1+{1+\omega(x) \over p}-1 < d_1,\ A_2(x')=A_2(x).
$$
It is easily seen that $c'(x')<c'(x)$.

\smallskip

The remaining case: up to renumbering $u_1,u_2$, we have
$$
A_1(x)<1 \leq A_2(x), \ d_2+{1+\omega(x) \over p}<1 \leq d_1+{1+\omega(x) \over p}.
$$
We then blow up along $x$. As case (i) is resolved, we have just to look at the origins of both charts.
Let us look at the first chart, of origin the point $x'$ as above. If $(m(x'),\omega (x'))=(p,\omega (x))$,
$x'$ is in case (**) and we have $A_2(x')=A_2(x)$, $d'_2=d_2$ and
$$
A_1(x')=A_1(x)+A_2(x)-1<A_2(x),  \ C(x')=0, \ d'_1=d_1+{1+\omega(x) \over p}-1 < d_1.
$$
Therefore $c'(x')<c'(x)$. The last point to look at is the point
$$
x''=({Z\over u_2},{u_1\over u_2},{u_2},{v\over u_2}).
$$
If $(m(x''),\omega (x''))=(p,\omega (x))$, $x''$ is in case (**), and we have $A_1(x'')=A_1(x)$ and
$$
A_2(x'')=A_1(x)+A_2(x)-1<A_2(x),  \ C(x'')=0.
$$
Therefore $c'(x'')<c'(x)$. This concludes the proof.
\end{proof}

\vfill\eject

\printindex

\end{document}